\documentclass[11pt]{amsart}
\usepackage{amsfonts}
\usepackage{geometry}                
\usepackage{graphicx}
\usepackage{amssymb}
\usepackage{epstopdf}
\usepackage{setspace}
\usepackage{leftidx}
 \usepackage{setspace}

\usepackage{xcolor}

\theoremstyle{plain}
\newtheorem{theo}{Theorem}[section]
\newtheorem{lem}[theo]{Lemma}
\newtheorem{cor}[theo]{Corollary}
\newtheorem{prop}[theo]{Proposition}

\newtheorem{cla}[theo]{Claim}

\numberwithin{equation}{section}

\theoremstyle{definition}

\newtheorem{remark}[theo]{Remark}
\newcommand{\bequ}{\begin{equation}}
\newcommand{\eequ}{\end{equation}}
\newcommand{\bali}{\begin{align}}
\newcommand{\eali}{\end{align}}

\def\vf{\varphi}

\def\P{\mathfrak P}

\def\d{\delta}

\def\D{\Delta}

\def\g{\gamma}

\def\LL{\mathcal L}
\def\pp{\partial}
\def\l{\lambda}

\def\s{\sigma}

\def\E{\mathbb E}

\def \P{\mathbb P}
\def\ov{\overline}
\def\un{\underline}

\def\om{\omega}
\def\Om{\Omega}

\def\W{\mathcal W}

\def\wt{\widetilde}
\def\({\biggl(}
\def\){\biggr)}

\def\<{\bold\langle}
\def\>{\bold\rangle}

\def\LL{{{\mathcal L}}}
\def\XX{\mathcal{X}}

\def\M{\widetilde {M}}

\def\bfv{{\bf v}}

\def\oa{\overleftarrow}
\def\lf{\lfloor}
\def\rc{\rceil}

\definecolor{Lin}{RGB}{220, 30, 70}
\definecolor{Francois}{RGB}{70,30, 220}

\usepackage{mathtools,mathabx}
\newcommand\AR[1]{\makebox[0pt][l]{$#1$}\kern0.5em\raisebox{1.5ex}{$\curvearrowleft$}} 

\usepackage{graphicx}
\usepackage{float}
\usepackage{stmaryrd}
\usepackage{caption}

\title[The regularity of the  linear drift   in negatively curved spaces]{The regularity of the  linear drift   in negatively curved spaces}
\author{Fran\c cois Ledrappier  and  Lin Shu}
\address{Fran\c cois Ledrappier,  Department of Mathematics, University of Notre Dame, IN 46556-4618, USA and  LPSM, Sorbonne Universit\'e, Bo\^{i}te Courrier 158, 4, Place Jussieu, 75252 PARIS cedex
05, France}\email{fledrapp@nd.edu}
\address{Lin Shu,  LMAM, School
of Mathematical Sciences, Peking University, Beijing 100871,
People's Republic of China} \email{lshu@math.pku.edu.cn}

\thanks{The second author was  partially  supported by  NSFC (No.11331007 and No.11422104) and Beijing Higher Education Young Elite Teacher Project (YETP0003).} 

\begin{document}
\maketitle
\begin{abstract} We show the linear drift of the Brownian motion on the universal cover of a closed connected Riemannian manifold is $C^{k-2}$ differentiable  along   any $C^{k}$ curve  in  the manifold of $C^k$ metrics  with negative sectional curvature.  We also show that the stochastic entropy of the Brownian motion is  $C^1$ differentiable  along   any $C^{3}$ curve of $C^3$ metrics with negative sectional curvature. We formulate the first derivatives of the linear drift and entropy, respectively,  and show they are critical at locally symmetric metrics. 
\tiny\begin{spacing}{1.08}
{\tableofcontents}
\end{spacing}
\end{abstract}
\section{Introduction and statement of results}

If we think of curvature as a measure of  the geometric complexity of a closed connected Riemannian manifold, the `simplest' geometric objects are those with constant sectional curvatures since their universal covers must be spheres, planes or Poincar\'{e} disks.  A little more `complicated' objects are locally symmetric spaces, whose universal covers are symmetric.  An attractive  problem in  geometry is to characterize locally symmetric spaces using other complexities, for instance,   Lichnerowicz's conjecture in 1944 (\cite{Li}) says that symmetry is equivalent to the harmonic property of the space, which means the geodesic spheres have constant mean curvature depending only on their radii.

From the point of view of dynamical systems, geometry influences dynamics and hence the geometric complexities can be read using dynamical complexities.  One example is  volume entropy,  which  is the exponential growth rate of the volume of a ball in the universal cover as a function of the radius.  It is named entropy since it is no bigger than the topological entropy of the geodesic flow in the unit bundle,  with equality if the underlying space is of nonpositive curvature (\cite{Man}) or the underlying space has no conjugate points and its Riemannian metric is H\"{o}lder $C^3$ (\cite{FM}).   In 1983, Gromov (\cite{Gro}) conjectured that among all metrics of volume equal to the volume of a locally symmetric metric $g_0$, the volume entropy is minimized at metrics isometric to $g_0$.  For negatively curved spaces,  this was shown by Katok (\cite{K}) for the 2-dimensional  case and was shown for higher dimensional cases by Besson,  Courtois and  Gallot (\cite{BCG}).  The remarkable rigidity result in \cite{BCG}   implies the Mostow rigidity (\cite{Mos}) (and its generalizations by Corlette (\cite{Cor}), Siu (\cite{Si}) and Thurston  (\cite{Th})) and also has many interesting rigidity applications in dynamics  combined with the results of  \cite{BFL}, \cite{FL}, \cite{L2}, etc. This helps us to understand the interaction between differential geometry and dynamical systems,  and  leads to many more rigidity studies on both sides. 

Since the  geodesic flow in the unit tangent bundle always preserves the Liouville measure, its  entropy  is  another natural quantity  (besides the topological entropy) for the description of the  dynamical complexity of the system. Clearly, the entropy of the Liouville measure is always less or equal to the topological entropy for the geodesic flow. It was conjectured by Katok in 1982 (\cite{K}, see also \cite{BK}) that in the negatively curved manifold case,  these two entropies coincide (if and) only if the manifold is locally symmetric. This is true  in the 2-dimensional case (\cite{K}). For the higher dimensional cases, it is a very difficult problem and it depends on our understanding of the dedicate difference between the Liouville measure and the Bowen-Margulis measure (for topological  entropy).  To approach this conjecture,  many experts tried to study  the variations of the two entropies with respect to  perturbations of the original system and to derive formulas for their infinitesimal changes (see e.g., \cite{Con, Fl, KKPW, KKW, KW, Kn}). We mention some of them briefly.  The smoothness of the topological entropy for perturbations of the Anosov flows were  considered  by Katok, Knieper, Pollicott and Weiss in \cite{KKPW} (see \cite{KKW} for  the first order derivative formula of the topological entropy of the geodesic flow under one-parameter family of $C^2$ perturbations of the original $C^2$ negative curved metric).   (As a corollary of the results of  \cite{BCG}  and \cite{KKPW}, a  locally symmetric negatively curved metric $g_0$ is a critical point of the topological entropy. Whether the reverse is true or not is an open question that was addressed in \cite{KKW}.)  Contreras (\cite{Con}) continued to analyze the  regularity of the Liouville entropy with respect to perturbations of the system. Furthermore, Flaminio (\cite{Fl}) gave a partial positive answer to Katok's conjecture by showing that along any non-trivial deformation the topological entropy and the difference between the  topological entropy and the Liouville entropy are locally strictly convex functions of the deformation parameter.   Besides its connection with the above  rigidity problems, the studies of the regularities of the entropies  have their own interest in the dynamical dimension theory (see e.g., \cite{Ano2, Katok2, Mis1, Mis2, Ne, P, Rug, Y1, Y2}). They are also  in the same flavor of the studies of the linear response problems in statistical mechanics for the understanding of the heat conduction (see  Ruelle (\cite{Ru1, Ru2, Ru3, Ru4})). The key step in the linear response theory is to justify, derive and understand  the first order derivative of the measure theoretical entropy of the SRB measure under smooth perturbations of the original system  (see  \cite{Ru5} and \cite{B} for nice introductions to  this field and hot references).  

Now, if we consider Brownian motion instead of the geodesic flow,  can we find similar connections between the stochastic dynamics  and  the geometric complexities?

 Let  $M$ be an $m$-dimensional orientable closed connected smooth manifold with fundamental group $G$.  Its universal cover space $\M$ is such that  $M=\wt{M}/G$.  For a $C^2$ Riemannian metric $g$ on $M$, let    $\wt{g}$  be its $G$-invariant extension  to $\M$. Consider the Brownian motion on $(\wt{M}, \wt{g})$ with starting point $x\in \M$. Its density function of the distribution at time $t\in \Bbb R_{+}$,  denoted by  $p(t, x,
y), y\in \wt{M}$, is the fundamental solution
to  the heat equation ${\partial u }/{\partial t}=\Delta u$,  where
$\Delta:=\mbox{Div}\nabla$ is the Laplacian of metric $\wt{g}$ on $C^2$ functions on
$\wt{M}$. Denote by $\mbox{Vol}_{\wt{g}}$ the Riemannian volume on
$(\wt{M}, \wt{g})$.  In this paper, we are mainly interested in the behaviors of the following two dynamical quantities. One is the  \emph{linear drift} 
\[\ell:=\lim_{t\rightarrow +\infty} \frac{1}{t}\int d_{\wt{g}}(x, y)p(t, x, y)\
  d\mbox{Vol}_{\wt{g}}(y), \]
  which was  introduced by
Guivarc'h (\cite{Gu}). It 
 tells the average in time of the shift of the Brownian motion from its starting point.  The other is the \emph{(stochastic) entropy} 
 \[
h:=\lim\limits_{t\to +\infty}-\frac{1}{t}\int \ln p(t, x, y) p(t, x, y)\ d{\rm Vol}_{\wt{g}}(y),
\]
which was introduced by Kaimanovich (\cite{K1}). It tells the average  decay rate of the transition probabilities of the Brownian motion.   Both $\ell$ and $h$ are independent of the choice of $x$ and are well-defined since we have a compact quotient. 

The linear drift, the stochastic entropy and the volume entropy (denoted by $\upsilon$)  are  interrelated  as follows:
\begin{equation}\label{ineq}
\ell^2\; \stackrel {(a)}{\leq} \; h \;\stackrel {(b)}{ \leq }\; \ell
\upsilon \;\stackrel{(c)}{ \leq }\upsilon^2.
\end{equation}
(For $(a)$, see \cite{K1} for  the negatively curved case and see \cite{L4} for the general case.  For $(b)$, see \cite{Gu}.  Inequality $(c)$ was derived in \cite{L4}  as a corollary of $(a)$ and $(b)$.)  All the equalities in (\ref{ineq}) turn out to be  related to the rigidity problem of locally symmetric spaces.  The equality $\ell^2=h$ (and hence $\upsilon^2=h$ and $\ell=\upsilon$) implies the space is locally symmetric in the negative curvature case by results in \cite{K1, BCG, BFL, FL, FM} and this characterization continues to hold in the non focal point case (\cite{LS1}).

 For $h=\ell \upsilon$, whether it   holds only for locally symmetric spaces is equivalent to a conjecture of Sullivan (see \cite{L2} for a discussion), which is not even known for  negatively  curved manifolds with dimensions greater than 2. Note that for Brownian motion, it is associated with a natural important probability measure in the unit tangent bundle of $M$, the so-called harmonic measure (see Section \ref{Sec-2.3}) and,  in the negatively curved case,  the quotient    $h/\ell$ is proportional to the Hausdorff dimension of the harmonic measure at the infinity boundary (\cite{L1}). Hence Sullivan's conjecture depends on the understanding of the dedicate difference between the harmonic measure and the Bowen-Margulis measure. 
This, together with the works that we mentioned above on  Katok's conjecture and the linear response theory, motivates our  study  in \cite{LS2} and the present paper to  analyze  the regularities of the linear drift and the stochastic entropy with respect to metric changes, to derive  formulas for the corresponding differentials and to understand the critical points.

We need some  notations to state our regularity results  in a precise form. 

For $k\in \Bbb N$, let  $C^k(S^2T^*)$ be the collection of $C^k$ sections of $S^2T^*$,  the bundle of symmetric $2$-forms on  the tangent space $TM$. It is a Banach space with the topology of the uniform convergence in $k$
derivatives. The set  of  all
smooth sections of  $S^2T^*$, denoted by 
$C^{\infty}(S^2T^*):=\bigcap_{k=0}^{\infty}C^k(S^2T^*)$, is a
Fr\'{e}chet space whose topology is given by all the $C^k$-norms.
Let $\mathcal{M}^k(M)$ denote the set of $C^k$ Riemannian
metrics on $M$. It  is the collection of
elements in $C^k(S^2T^*)$  which induces a positive definite inner
product on each tangent space $T_xM$, $x\in M$. 
The space of all smooth Riemannian metrics
$\mathcal{M}^{\infty}(M)=\bigcap_{k=1}^{\infty}\mathcal{M}^k(M)$ consists of
an open convex positive cone in $C^{\infty}(S^2T^*)$ and is a
Fr\'{e}chet manifold.

  Let $\Re^k(M)$ $(k\geq 3\ \mbox{or}\ k=\infty)$ be  the submanifold of $\mathcal{M}^k(M)$ made of  negatively curved $C^k$ metrics  on
$M$.   It  is open in
$\mathcal{M}^k(M)$.   For any curve $\lambda\in (-1, 1)\mapsto g^{\lambda}\in
\Re^k(M)$,  the linear drift  for each $(M, g^{\lambda})$, denoted by
$\ell_{\lambda}$, is positive  (\cite[Theorem 10]{K1}). 

 Our main result in this paper is the following.
\begin{theo}\label{main}
Let  $M$ be a closed connected smooth manifold. For any $C^{k}$
$(k\geq 3)$ curve $\lambda\in (-1, 1)\mapsto
g^{\lambda}\in \Re^{k}(M)$,  the function $\lambda\mapsto
\ell_{\lambda}$ is $C^{k-2}$ differentiable; for  any $C^{\infty}$
 curve $\lambda\in (-1, 1)\mapsto
g^{\lambda}\in \Re^{\infty}(M)$,  the function $\lambda\mapsto
\ell_{\lambda}$ is $C^{\infty}$ differentiable.
\end{theo}

A special case of Theorem \ref{main} was treated in \cite{LS2}, where we considered the case that $g^{\lambda}=e^{2\vf^{\lambda}} g$ is a $C^3$ curve of $C^3$ conformal changes of $g$ in $\Re^3(M)$ and showed the differentiability of $\ell_{\l}$ in $\l$. In that setting,  the relation between the $\wt{g}^{\l}$-Laplacian $\Delta^{\lambda}$  and the $\wt{g}$-Laplacian $\Delta$  can be  formulated as:  for $f$ a $C^2$ function on $\M$,
\[
\Delta^{\lambda} f=e^{-2\vf^{\lambda}}\left(\Delta f+(m-2)\langle \nabla\vf^{\l},\nabla f\rangle_g\right), 
\]
where we still denote $\vf^{\lambda}$ its  $G$-invariant  extension to $\M$. So we can split the difference $\ell_{\l}-\ell_0$ into two parts corresponding to the time change and drift change of the diffusion, respectively. The first part differentiability can be  handled using the results in \cite{FF, LMM}, while the second part  differentiability was shown in the process of the diffusion using the Cameron-Martin-Girsanov  formula and the Central Limit Theorem for the linear drift (\cite{L}).  There is no such simple picture  for the  $C^1$ regularity of the linear drift for general deformation of metrics or for the higher order regularities consideration.

Our strategy to prove Theorem \ref{main} is to use the expression
of the linear drift at the infinity boundary
of $\M$ and prove the $C^{k-2}$ regularity  of the ingredients in that formula.

Let $\wt{g}^{\lambda}$ be the $G$-invariant extensions of $g^{\l}$ in $\M$.  The \emph{geometric boundary}
of $(\wt{M}, \wt{g}^{\l})$, denoted $\partial\M^\l$, is the collection of  the equivalence classes of unit speed $\wt{g}^{\lambda}$-geodesics that remain a bounded distance apart.  Each $\partial\M^\l$ can be identified with  $\partial \M^0$ (or simply $\partial \M$)  since
the identity isomorphism from $G$ to itself induces a natural homeomorphism between the two boundaries.  For $x\in \M$ and $\xi\in \partial \M$,
let $X^{\l}(x, \xi)$ be the initial speed vector of the unit speed $\wt{g}^{\lambda}$-geodesic  starting from $x$  belonging to the equivalent class of $\xi$.  Let ${\mbox{Div}}^{\l}$ be the divergence operator of $(\M, \wt{g}^{\lambda})$.  It is true (see Section 2 for a more precise statement) that
\begin{equation}\label{ell-lambda}
\ell_{\lambda}=-\int_{M_0\times \partial \widetilde{M}}{\mbox{Div}^{\l}}X^{\lambda}\ d\widetilde{{\bf m}}^{\lambda},
\end{equation}
where $M_0$ is a connected  fundamental domain and $d\wt{{\bf
m}}^{\lambda}=dx^{\l}\times d\wt{{\bf m}}^{\lambda}_x$, where $dx^{\l}$ is proportional  $d{\mbox{Vol}}_{\wt{g}^{\lambda}}$ and $\wt{{\bf m}}^{\lambda}_x$ is the hitting probability at
$\pp\M$ of the  $\wt{g}^{\lambda}$-Brownian motion starting at $x$.

The term $-{\mbox{Div}^{\l}}X^{\lambda}$ in (\ref{ell-lambda}) has its
geometric feature as  being  the mean curvature  of  the strong 
stable horosphere of the geodesic flow in the metric $\wt{g}^{\l}$ (see (\ref{Div-trace})); its regularity in $\l$ can be deduced using the results from \cite{Con, KKPW, LMM} on the Morse
correspondence map between the geodesic flows of two negatively curved spaces  (Proposition \ref{regularity-Div}).

To conclude Theorem \ref{main}, we show the following  on the regularity  in $\l$ of the harmonic measure ${\bf m}^{\l}:=\widetilde{{\bf m}}^{\lambda}|_{SM}$, where $SM:=M_0\times \partial \M$ (see Section \ref{sec-regularities-linear drift} for precise definitions).  

\begin{theo}\label{regularity-harmonic measure} Let  $M$ be a closed connected smooth manifold. For any $g\in \Re^{k}(M)$, $k\geq 3$, there  exist a neighborhood $\mathcal{V}_g$ of $g$ in $\Re^{k}(M)$ and a Banach subspace  $\mathcal{H}_{{\mathtt{b}}}^0$  of continuous functions on $SM$ such that for any $C^k$ curve $\lambda\in
(-1, 1)\mapsto g^{\lambda}\in \mathcal{V}_g$  with $g^0=g$, the mapping $\l\mapsto {\bf m}^{\l}$ is $C^{k-2}$ in the weak topology of the dual space $(\mathcal{H}_{{\mathtt{b}}}^0)^*$. 
\end{theo}

The regularity problem in  Theorem \ref{regularity-harmonic measure}  was not discussed in \cite{LS1} for the conformal change case. It  is subtle  since harmonic  measures are not the dual of  linear functionals acting on the space of continuous functions on $SM$.
For each ${g}^{\l}$, it is defined  naturally  a one parameter family  of actions  ${{\rm Q}}_t^{\l}$ $(t\geq 0)$ on continuous functions $f$ on $SM$:   \begin{equation}\label{def-Q-t} {{\rm Q}}_t^{\l}(f)(x,
\xi):=\int_{\scriptscriptstyle{M_0\times \pp\M}} \wt{f}(y, \eta){\bf q}^{\l}(t,
(x, \xi), d(y, \eta)),
\end{equation}
where ${\bf q}^{\l}$ denotes the transition probability of the $g^{\l}$-Brownian motion on the stable leaves of $SM$ and $\wt{f}$ denotes  the $G$-invariant extension of $f$ to $\M\times \partial \M$. Since $(M, g^{\l})$ is negatively curved, it is known (\cite{L}) that each ${{\rm Q}}_T^{\l}$ (for $T$ large) is a contraction  on  some Banach space $\mathcal{H}^\l_{{\mathtt{b}}}$ of
 continuous  functions on $SM$ which are H\"{o}lder continuous with respect to direction changes and this makes ${\bf m}^{\l}$  a fixed point of the dual of ${{\rm Q}}_T^{\l}\big|_{\mathcal{H}^\l_{{\mathtt{b}}}}$.  
The idea to prove Theorem \ref{regularity-harmonic measure} is to use the classical  perturbation result on a linear contraction  in a Banach  space (\cite{Ka}). Hence, it suffices to  find a common Banach space $\mathcal{H}_{{\mathtt{b}}}^0$ and a $T>0$ such that 
\begin{itemize}
\item all ${{\rm Q}}_T^{\l}$, ${\l\in (-1, 1)}$,  are  contractions on $\mathcal{H}_{{\mathtt{b}}}^0$, uniformly in $\l$, and
\item $\l\mapsto {{\rm Q}}_T^{\l}$ is $C^{k-2}$ as maps from $\mathcal{H}_{{\mathtt{b}}}^0$ into itself. 
\end{itemize}
To achieve this, we not only need the regularity of the heat kernels ${\bf q}^{\l}$ in $g^{\l}$, but also need  the estimations on its differentials, which we present  with full generality  as follows.

For each  $C^k$ Riemannian metric $g=(g_{ij}(x))\in \mathcal{M}^k(M)$, set $\|g\|_{C^a}$ ($a\leq k$) for the $C^a$-norm of $g$ which involves  the bounds of $\{g_{ij}(x)\}$ and of their differentials  up to the $a$-th  order.  Each $C^k$ curve  $\lambda\in
(-1, 1)\mapsto g^{\lambda}\in \mathcal{M}^k(M)$ defines a one parameter family  of  tangent vectors $\XX^{\l}=(\XX^{\l}_{ij}(x)) \in C^k(S^2T^*)$.  Let  
\[
(\XX^{\l})^{(0)}=\XX^{\l}, \  (\XX^{\l})^{(l)}=\big((\XX^{\l})^{(l-1)}\big)^{(1)}_{\l}, \ l=1, \cdots, k-1.
\]
All $(\XX^{\l})^{(l)}$ are elements in  $C^k(S^2T^*)$. By $\|(\XX^{\l})^{(l)}\|_{C^a}$ ($a\leq k$), we mean the $C^a$-norm of $(\XX^{\l})^{(l)}$, which involves the bounds of the $(\XX^{\l})^{(l)}_{ij}(x)$ and of their differentials in $x$ up to the $a$-th  order.    

Let $C^{k, \iota}(\M)$   denote  the collection  of $C^k$ functions on $\M$ with H\"{o}lder exponent $\iota$. The set of continuous functions on $\M$ is denoted by $C(\M)$. For any one parameter family  of   real functions on $\M$ or real numbers $\l\mapsto a^{\l}$, let $(a^{\l})^{(i)}_{\l}$ denote the $i$-th differential in $\l$ whenever it exists.

\begin{theo}\label{diff-HK-estimations-gen} For  any $g\in \mathcal{M}^k(M)$, $k\geq 3$,  there exist $\iota\in (0,
1)$ and a neighborhood $\mathcal{V}_g$ of $g$ in $\mathcal{M}^{k}(M)$ such that  for any $C^{k}$ curve  $\lambda\in
(-1, 1)\mapsto g^{\lambda}\in \mathcal{V}_{g}$ with $g^0=g$:
\begin{itemize}
\item[i)] The mappings  $\l\mapsto p^{\l}(T, x, \cdot)$, $x\in \M, T\in \Bbb R_+$,  are $C^{k-2}$  in $C^{k, \iota}(\M)$.
\item[ii)] Let $T_0>0$  and $q\geq 1$. For each $i$, $1\leq i\leq k-2$, $l$, $0\leq l\leq k-2-i$, and $T>T_0$,  there exists ${c}_{\l, (l, i)}(q)$ depending on $(l, i)$, $m, q, T, T_0$, $\|g^{\l}\|_{C^{l+i+2}}$ and $\{\|(\XX^{\l})^{(j)}\|_{C^{l+i-j+1}}\}_{j\leq i-1}$  such that  
\begin{align}\label{esti-p-lam-der-k} \hspace{7mm} 
\left\|\nabla^{(l)}(\ln p^{\l})^{(i)}_{\l}(T, x, \cdot)\right\|_{L^q}\leq {c}_{\l, (l, i)}(q), \end{align}
where the $L^q$-norm is taken with respect to the distribution at $T$ of the $\wt{g}^{\l}$-Brownian motion probability.
\item[iii)]Let $T_0>0$ and $q\geq 1$.   For each $i$, $1\leq i\leq k-2$, and $T>T_0$, there exists ${c}_{\l, (i)}(q)$ depending on $i, m, q, T, T_0$, $\|g^{\l}\|_{C^{i+2}}$ and  $\{\|(\XX^{\l})^{(j)}\|_{C^{i-j+1}}\}_{j\leq i-1}$ such that  
\begin{equation}\label{p-lam-i-p-equ}
\left\|\frac{(p^{\l})^{(i)}_{\l}(T, x, \cdot)}{p^{\l}(T, x, \cdot)}\right\|_{L^q}\leq {c}_{\l, (i)}(q).
\end{equation}
\item[iv)] Let $\wt{f}\in C(\M)$  be  uniformly continuous and bounded. Then for any $T>0$ and $i$, $1\leq i\leq k-2$, the function  $\int_{\M} (p^{\l})^{(i)}_{\l}(T, x, y) \wt{f}(y)\ d{\rm{Vol}}_{\wt{g}^{\l}}(y)$ belongs to $C(\M)$. 
\end{itemize}
\end{theo}

A priori, the derivative in $\l$ of $p^{\l}(t, x, y)$, if it exists, satisfies the equation
\begin{align}\label{b-argu-1}
\left\{ \begin{array}{lll}
\frac{\partial}{\partial t}q(t,x, y)\!&=&\!\Delta^{\l}_{y} q(t, x, y)+(\Delta^{\l}_y)^{(1)}_{\l}p^{\l}(t, x, y),\\
q(0, x, y)\!&\equiv & \! 0. 
\end{array}
 \right.
\end{align}
Equation (\ref{b-argu-1}) always has a solution in the distribution sense. Our Theorem \ref{diff-HK-estimations-gen} is that this distribution is given by a function $(p^{\l})^{(1)}_{\l}(t, x, \cdot)\in C^{k, \iota}(\M)$ and that its gradients satisfy (\ref{esti-p-lam-der-k}). This does not follow directly from (\ref{b-argu-1}) since $(\Delta^{\l}_y)^{(1)}_{\l}p^{\l}(t, x, y)$ has singularities as $t$ goes to zero and $y=x$. This type of singularities was not handled in the literature and this difficulty accumulates when we  consider $\{(p^{\l})^{(i)}_{\l}(T, x, \cdot)\}_{i\geq 2}$.  Moreover the universal cover is non-compact.    We are not successful to give  a more direct proof after trying many classical analysis  methods such as parametrix, parabolic Schauder theory, Sobolev spaces,  etc. (cf. \cite{Fr, Ma, Ro}).

To get an explicit expression of the solution, we use the stochastic calculus representations of the heat kernel and the Brownian motion.  Namely, we find a $C^1$ vector field ${\rm z}_T^{\l, 1}(y)$  on $\M$ (see (\ref{z-T-lam-1})) such that, for any smooth $f$ on $\M$ with compact support, 
\begin{align*}
\left(\int_{\M}f(y)p^{\l}(T, x, y)\ d{\rm{Vol}}^{\l}(y)\right)^{(1)}_{\l} =\int_{\M}\big\langle\nabla_y^{\l} f(y),  p^{\l}(T, x, y){\rm z}_T^{\l, 1}(y)\big\rangle_{\l}\ d{\rm Vol}^{\l}(y).
\end{align*}
So, using the classical integration by parts formula, we obtain 
\begin{align}
&\left(\int_{\M}f(y)p^{\l}(T, x, y)\ d{\rm{Vol}}^{\l}(y)\right)^{(1)}_{\l}\notag\\
\label{CIBP}&\ \ \  =-\int_{\M} f(y)\left({\rm Div}^{\l}{\rm z}_T^{\l, 1}(y)+\big\langle {\rm z}_T^{\l, 1}(y), \nabla^{\l}\ln p^{\l}(T, x, y) \big\rangle_{\l}\right)p^{\l}(T, x, y) \ d{\rm Vol}^{\l}(y). 
\end{align}
In the same way, we will find $C^{1}$ vector fields $\{{\rm z}_T^{\l, j}(y)\}_{j\leq i\leq k-2}$  (see (\ref{z-T-lam-2}) and (\ref{z-T-lam-j})),  which  will enter the formulas of $(\ln p^{\l})^{(i)}_{\l}$ and the gradients $\nabla^{(l)}(\ln p^{\l})^{(i)}_{\l}$.  

 It is not hard to obtain  a stochastic expression   for ${\rm z}_T^{\l, 1}$ using the  Eells-Elworthy-Malliavin construction of the Brownian motion on a manifold. But the  associated stochastic differential equation of the Brownian motion  in the orthogonal frame bundle is degenerate.  So the main technical difficulty is that the $C^1$ regularity of  ${\rm z}_T^{\l, 1}$ does not follow directly from the stochastic pathwise integration by parts theory or the stochastic functional methods for the calculus of variations (cf. \cite{Bi1, Bi2, D1, D94, M76, M78, M78-2, W}). Similar difficulties will also arise in obtaining the stochastic expressions of $\{{\rm z}_T^{\l, j}(y)\}_{2\leq j< i\leq k-2}$ and in using these expressions to identify ${\rm z}_T^{\l, i}$.  However, since we are mainly interested in the behaviors of the projections of the various stochastic objects on the manifold,  we  can overcome these difficulties by  a constructive method using some ideas from \cite{CE, D2, Hs1, M}.  Most computations to guarantee the constructions will appear in Chapter \ref{Sec-BM-SF} for the neatness of the paper.  It is also for the introduction of  the beautiful ideas from \cite{CE, M} to treat the Brownian motion as a stochastic analogue of the geodesic flow (see Section \ref{BMM-flow} for details).   This dynamical point of view will be very helpful in understanding our  constructive proof concerning the  smoothness of all the vector fields $\{{\rm z}_T^{\l, j}(y)\}_{1\leq j\leq k-2}$. 
 
Note that the stochastic flow (for  the Brownian motion) always preserves the Liouville measure (\cite{CE}).  In analogy with  Katok's conjecture, one interesting question is when will the entropy of the Liouville measure be equal to  the topological entropy for this flow?

In showing Theorem \ref{main}, we  also obtain the formula (\ref{l-derivative})  (see the formula (\ref{Great}) for a more precise form)  for  the first order differential of the  linear drift  under one-parameter  family of  deformations of negative curved metrics, which  implies  the following two theorems. 

\begin{theo}\label{Critical}(see Corollary \ref{prop-thm1.4})  Let  $M$ be a closed connected smooth manifold. Let $g\in \Re^3(M)$ be a negatively curved locally symmetric metric.  Then for any $C^3$ curve $\lambda\in (-1, 1)\mapsto g^{\lambda}\in \Re^3(M)$ with $g^0=g$ and constant volume,
\[
(\ell_{\l})'_0:=(d\ell_{\lambda}/d\lambda)|_{\lambda=0}=0.
\]
\end{theo}

\begin{theo}\label{linear} (see Theorem \ref{Equiv-Thm1.5})
There is a linear functional $\LL$ on $C^k(S^2T^*)$ such that for all $C^3$ curve $\lambda\in (-1, 1)\mapsto g^{\lambda}\in \Re^3(M)$ with $g^0=g$ and constant volume, \[ (\ell_{\l})_0' \; = \; \LL (\XX) .\]
\end{theo}

A similar approach yields the first order differentiability in $\l$ of the stochastic entropy $h^{\l}$ of the Brownian motion on $(\M, \wt{g}^{\l})$. 

\begin{theo}\label{main-h}Let  $M$ be a closed connected smooth manifold. For any $C^{3}$ curve $\lambda\in (-1, 1)\mapsto
g^{\lambda}\in \Re^{3}(M)$,  the function $\lambda\mapsto
 h^{\lambda}$ is $C^{1}$ differentiable and is critical at $\l=0$ when $g^0$ is locally symmetric. Moreover, there is a linear functional $\mathcal{K}$ on $C^k(S^2T^*)$ such that
 \[
 (h^{\l})'_0:=(dh^{\lambda}/d\lambda)|_{\lambda=0}=\mathcal{K}(\XX). 
 \]
\end{theo}

An explicit formula of $\mathcal{K}(\XX)$ is given in Theorem  \ref{deriv-entr-for}, where the  infinitesimals of the metric changes appear in a neat way. Hence  an interesting question is to characterize the critical points of the entropies of harmonic measures.  In our approach, the higher order regularity of $\l\mapsto h^{\l}$ and the analysis on  the differentials  would depend on understanding the regularity of the Martin kernel, which is a delicate problem in the manifold setting. This will be treated in a subsequent paper (\cite{LS-n}).   

Note that  the Hausdorff dimension of the distribution of $\wt{{\bf m}}^{\l}_x$, denoted by ${\rm dim}_{{\rm H}}\wt{{\bf m}}^{\l}_x$,  is given by $h^{\l}/(\varkappa\ell^{\l})$ for a fixed  number $\varkappa$ associated with the distance function on the boundary (see (\ref{distance-bd})) (\cite{L1}). The following is a corollary of Theorem \ref{main} and Theorem \ref{main-h}. 

\begin{cor}Let  $M$ be a closed connected smooth manifold. For any $C^{3}$ curve $\lambda\in (-1, 1)\mapsto
g^{\lambda}\in \Re^{3}(M)$ and all $x\in \M$,  the function   $\lambda\mapsto {\rm dim}_{{\rm H}} \wt{{\bf m}}^{\l}_x$ is $C^1$ differentiable.  
\end{cor}

 If we switch from a negatively curved manifold  to a finitely generated hyperbolic group $G$, we do not have to control any more the subtle influence of the changes of soft geometric structures. But diffusions live on in the form of random walks, so the regularity problem of random dynamics with respect to probabilities still has its  interest.  More precisely,  let $\mathcal{N}$ be the set of probability measures with support a fixed finite subset $G_0\subset G$  which generates $G$ as a semigroup.  Then $\mathcal{N}$ is  an open finite dimensional simplex, in particular, it has a natural real analytic structure.  Each element $\mu\in \mathcal{N}$ defines a random walk  on $G$ by convolutions $\{\mu^{(n)}\}_{n\in \Bbb N}$. The  linear drift  and  the entropy of $\mu$  are defined by 
$$ \ell _\mu  := \lim\limits _{n \to +\infty } \frac{1}{n} \sum _{\gamma \in G} |\gamma| \mu^{(n)} (\gamma) , \quad h_\mu  := \lim\limits _{n \to +\infty }  - \frac{1}{n}  \sum _{\gamma \in G}  \mu^{(n)} (\gamma) \log  \mu^{(n)} (\gamma)   ,$$ where, for $\gamma \in G , |\gamma| $ denotes the word length of $\gamma$.  In this setting, much progress has  been achieved in understanding  the  regularity  of  $\ell_{\mu}, h_{\mu}$ with respect to $\mu$:  the continuity property was considered by Erschler and Kaimanovich (\cite{EK}), the Lipschitz property was shown by  one of the authors (\cite{L13}), the differentiability under one parameter family of differentiable curve of $\mu$  is due to Mathieu (\cite{Math}),  and,  more recently, the real analytic property is shown by Gou\"ezel (\cite{Go}).  (See \cite{Go} for  the whole history and other previous results in various settings.) In the same flavor of the rigidity problems in the manifold case, a basic question is what can we say about the group structure using our knowledge of the dynamical quantities $\ell_{\mu}$ and $h_{\mu}$?   We don't have an  answer to this general question, but we can mention one result which is related to $(b)$ of (\ref{ineq}) in the above group setting: in \cite{GMM}, Gou\"{e}zel, Math\'{e}us and Maucourant show that if $G$ is not virtually free, then there is $c<1$ such that for any symmetric measure $\mu\in \mathcal{N}$, $h_{\mu}\leq c\ell_{\mu}\upsilon$, where $\upsilon$ denotes the volume entropy of the group in the word metric.

We arrange the paper as follows. In  Section 2, we give some preliminaries. In Section 3, we  assume Theorem \ref{diff-HK-estimations-gen} and prove consecutively  Theorem \ref{regularity-harmonic measure}, Theorem \ref{main}, Theorem \ref{Critical} and Theorem \ref{linear}. Section 4 is for the Eells-Elworthy-Malliavin construction of the stochastic flow corresponding to the Brownian motion  and its related dynamical properties. The estimations of the growth of various stochastic tangent structures are done with some special  care since we are in the non-compact case.  The strategy for proving the first order differentiability in Theorem \ref{diff-HK-estimations-gen} and the $i=1$ case of  (\ref{esti-p-lam-der-k}) and (\ref{p-lam-i-p-equ}) is explained in Section \ref{Obs-Stra}. Section \ref{sec5} is devoted to the details of that proof:  Section \ref{flow-F-S-y} is for the  construction for  the $C^1$ regularity of ${\rm z}_T^{\l, 1}$, followed by the existence  proof and estimations in Section \ref{sec-exist-F-s}-\ref{the flow F-S}, and  the proof of Theorem \ref{diff-HK-estimations-gen} with $i=1$ is given in Section \ref{TDOp-lam}  using the regularities and estimations  of ${\rm z}_T^{\l, 1}$.  The rest of the proof of Theorem \ref{diff-HK-estimations-gen} is by induction on the order of differentiability. See Section \ref{skectch-6.1} for the description of the necessary steps and Section \ref{sec6.2} for their proofs.  Finally, in Section 7, we consider the first order regularity of the entropy.

\section{Preliminaries}
 In this section, we introduce the basic notions related to  formula (\ref{ell-lambda}).   In the rest of the paper, if it is not specified, we only consider the elements of $\mathcal{M}^k(M),$  $\Re^k(M)$ with $k\geq 3$.

\subsection{Jacobi fields and the geodesic flow} For $g\in \mathcal{M}^k(M)$, let $\nabla, R$ be the Levi-Civita connection and the curvature tensor on $(M, g)$ and $(\M, \wt{g})$. Recall that a unit speed  $\wt{g}$-geodesic $t\mapsto\gamma(t)\in \M$ is such that  $\nabla_{\dot{\g}}{\dot{\g}}=0$, where $\dot{\g}(t)=\nabla_{\frac{\partial}{\partial t}}\g (t)$.  The Jacobi fields along $\gamma$ are vector fields $t\mapsto J(t)\in T_{\gamma(t)}\M$ which describe the infinitesimal variations of the geodesics around $\gamma$.  It is well known that $J(t)$ satisfies the Jacobi equation
\[
\nabla_{\dot{\gamma}(t)}\nabla_{\dot{\gamma}(t)} J(t)+R(J(t), \dot{\gamma}(t))\dot{\gamma}(t)=0
\]
and is uniquely determined by the values of $J(0)$ and $J'(0)$.   Let $N(\gamma)$ be the normal bundle of $\gamma$, i.e., 
\[
N(\gamma):=\bigcup_{t\in \Bbb R}N_t(\gamma), \ \mbox{where}\ N_t(\gamma)=\big\{Y\in T_{\gamma(t)}\M:\ \langle Y, \dot{\gamma}(t)\rangle=0\big\}.
\]
A \emph{$(1, 1)$-tensor} along $\gamma$ is a family $V=\{V(t), \ t\in \Bbb
R\}$, where each $V(t)$ is an endomorphism of $N_t(\gamma)$ such that for
any family $Y_t$ of parallel vectors along $\gamma$, the covariant
derivative $\nabla_{\dot{\gamma}(t)} (V(t)Y_t)$ exists.  The
curvature tensor $R$ induces a symmetric $(1, 1)$-tensor along
$\gamma$ by $R(t)Y=R(Y, \dot{\gamma}(t))\dot{\gamma}(t)$. A $(1, 1)$-tensor $V(t)$ along $\gamma$ is called a \emph{Jacobi tensor} if it
satisfies
\[
\nabla_{\dot{\gamma}(t)}\nabla_{\dot{\gamma}(t)} V(t)+R(t)V(t)=0.
\]
If $V(t)$ is a Jacobi tensor along $\gamma$, then $V(t)Y_t$ is a Jacobi field for any parallel field $Y_t$ along $\g$.

 The Jacobi fields can  also be visualized using the geodesic flow map on the unit tangent bundle.   For $x\in \M$
and ${\rm v}\in T_x\M$, an element  ${\rm w} \in T_{\rm v} T\M$   is
{\it {vertical }} if its projection on $T_x\M$ vanishes. The vertical subspace $V_{\rm v} $   is  identified with $T_x\M$.
The connection defines a {\it {horizontal }} complement $H_{\rm v}$,  which also can be  identified with $T_x\M.$
This gives a  horizontal/vertical Whitney sum decomposition
\[
TT\M=T\M\oplus T\M.
\]
Define the inner product on $TT\M$ by
\[
\big\langle (Y_1, Z_1), (Y_2, Z_2)\big\rangle_{\wt{g}}:=\big\langle Y_1,
Y_2\big\rangle_{\wt{g}}+\big\langle Z_1, Z_2\big\rangle_{\wt{g}}.
\]
It induces a Riemannian metric on $T\M$, the so-called Sasaki
metric. The  unit tangent bundle $S\M$ of
the universal cover $(\M, \wt{g})$ is a subspace of $T\M$  with
tangent space
\[
T_{(x, {\rm v})}S\M=\big\{(Y, Z):\ Y, Z\in T_x\M, Z\perp  {\rm v}\big\},\  \mbox{for}\
x\in \M,  {\rm v}\in S_x\M.
\]

Assume ${\bf v}=(x, {\rm v})  \in S\M$ and let $\gamma_{\bf v}$  be the $\wt{g}$-geodesic starting at $x$ with initial velocity $\rm v$.   Horizontal vectors in $T_{{\bf v}}S\M$ correspond to pairs $(J(0),0)$.  In particular, the \emph{geodesic spray}  $ \overline X_{{\bf v}}$  at $\bf v$ is the horizontal vector associated with $({\rm v},0)$. 
 A vertical vector in $T_{{\bf v}}S\M$   is a vector tangent to $S_x\M$, the set of unit tangent vectors at $x$.  It corresponds to a pair $(0, J'(0))$, with $J'(0) $ orthogonal to  ${\rm v}$. The orthogonal space to   $\overline{X}_{{\bf v}}$ in $T_{\bfv} S\M$  corresponds to pairs $({\rm v}_1, {\rm v}_2), {\rm v}_i \in N_0(\g_\bfv)$ for $i =1,2$.  
 
  The vector field $\{\overline{X}_{{\bf v}}\}_{{\bfv}\in S\M}$ generates the geodesic flow  $\{{\bf \Phi}_t\}_{t\in \Bbb R}$ on the unit tangent bundle,  where  ${\bf \Phi}_t: S\M\to S\M, \ {\bfv}\mapsto \dot{\gamma}_{\bf v}(t)$. Any Jacobi field along a geodesic $\gamma_{\bf v}$ is of the form $D{\bf \Phi}_t ({\bf w})$, where ${\bf w}\in T_{\bf v}S\M$ is an infinitesimal change of  the initial point $\bf v$.   More explicitly,  if $(J(0), J'(0) )$ is the horizontal/vertical decomposition of  ${\bf w}\in T_{\bfv}S\M$,  then $(J(t), J'(t) )$ is the horizontal/vertical decomposition of  $D {\bf \Phi}_t({\bf w})\in T_{{\bf \Phi}_t(\bfv)}S\M$.

\subsection{Anosov flow and invariant manifolds} \label{Sec-Anosov}Assume  $g\in \Re^k(M)$.  The $\wt{g}$-geodesic flow ${\bf \Phi}_t$ on $S\M$ has some special properties  due the negative curvature nature of the space. 

  Firstly,  $(\M, \wt{g})$ has no conjugate points. Hence we can  identify $S\M$ with $\M\times
\partial\M$ since each  pair $(x, \xi)\in \M\times \partial\M$ corresponds to a unique unit speed geodesic $\gamma_{x, \xi}$,  which begins at $x$ and is asymptotic to $\xi$,  and the mapping
$\partial \M\mapsto S_x\M$ sending $\xi$ to $\dot{\gamma}_{x,
\xi}(0)$ is a bijection.  In the $(\M, \partial\M)$-coordinate, the geodesic flow map ${\bf \Phi}_t$ has the expression
\begin{equation*}\label{geodesic flow}
{\bf\Phi}_t(x, \xi)=(\gamma_{x, \xi}(t), \xi), \ \forall  (x, \xi)\in S\M.
\end{equation*}

Furthermore,  the geodesic flow on $S\M$ is \emph{Anosov}: the tangent bundle $TS\M$ decomposes into the Witney sum of three $D{\bf \Phi}_t$-invariant   subbundles ${\bf E}^{\rm c}\oplus {\bf E}^{\rm ss}\oplus {\bf E}^{\rm su}$, where ${\bf E}^{\rm c}$ is the 1-dimensional subbundle tangent to the flow and $ {\bf E}^{\rm ss}$  and $ {\bf E}^{\rm su}$ are the strongly contracting and expanding subbundles, respectively, so that there are constants $C, c>0$ such that
\begin{itemize}
\item[i)] $\|D{\bf\Phi}_t {\bf w}\|\leq Ce^{-ct}\|{\bf w}\|$ for ${\bf w}\in {\bf E}^{\rm ss}$, $t>0$.
\item[ii)] $\|D{\bf \Phi}_t ^{-1}{\bf w}\|\leq Ce^{-ct}\|{\bf w}\|$   for ${\bf w}\in {\bf E}^{\rm su}$, $t>0$.
\end{itemize}
The ${\bf E}^{\rm ss}, {\bf E}^{\rm su}$ and  ${\bf E}^{\rm c}$ are the so-called \emph{stable, unstable} and \emph{central bundles}, respectively.

The subbundles ${\bf E}^{\rm ss}$, ${\bf E}^{\rm su}$  have their characterizations  using Jacobi tensors.  Assume ${\bf v}=(x, {\rm v})\in S\M$.  For each $s>0$, let $S_{{\bfv}, s}$ be the Jacobi tensor along  $\gamma_{\bfv}$ with the boundary conditions $S_{\bfv, s}(0)={\mbox{Id}}$ and $S_{\bfv, s}(s)=0$.  Since $(\M, \wt{g})$ has no conjugate points, the limit $\lim_{s\to +\infty}S_{\bfv, s}=:S_{\bfv}$ exists  (\cite{Esc}) and is called the \emph{stable tensor} along the geodesic
$\gamma_{\bfv}$. Similarly, by reversing the time $s$, we obtain the
\emph{unstable tensor} $U_{\bfv}$ along the geodesic $\gamma_{\bfv}$.  The stable subbundle ${\bf E}^{\rm ss}$ at $\bfv$
is the graph of the mapping $S'_{\bfv}(0)$, considered as a map
from $ \overline{N_0(\gamma _\bfv )}$ to $V_{\rm v} $
sending $Y$ to
$S'_{\bfv}(0)Y$, where  $\overline{N_0(\gamma _\bfv )} := \{ {\bf {w}}, {\bf w} \in H_{\rm v}, {\bf {w}} \perp \overline X_\bfv \}$. 
Similarly, the unstable subbundle ${\bf E}^{\rm su}$
at $\bfv$ is the graph of the mapping $U'_{\bfv}(0)$
considered as  a map from $ \overline{N_0(\gamma _\bfv )}$ to $V_{\rm v}$.

 Due to the Anosov property of the geodesic flow, the distributions of ${\bf E}^{\rm ss}, {\bf E}^{\rm su}$ (and hence ${\bf E}^{\rm c}\oplus {\bf E}^{\rm ss}, {\bf E}^{\rm c}\oplus {\bf E}^{\rm su}$) are H\"{o}lder continuous (\cite{Ano},   see also \cite[Proposition 4.4]{Ba}).  Hence, the $(1,1)$-tensors  $S_{
\bfv}, S'_{\bfv}, U_{\bfv}$ and  $U'_{\bfv}$ are also H\"{o}lder
continuous with respect to $\bfv$.

Associated with the  bundle ${\bf E}^{\rm cs}:={\bf E}^{\rm c}\oplus {\bf E}^{\rm ss}$ are  the (weak) \emph{stable manifolds} of ${\bf \Phi}_t$: 
\begin{equation}\label{weak stable manifold}
W^{\rm{s}}(x, \xi):=\left\{(y, \eta)\in \M\times \partial \M:\ \limsup\limits_{t\to +\infty}\frac{1}{t}\ln {\mbox{dist}}\left({\bf \Phi}_t(y, \eta),  {\bf \Phi}_t(x, \xi)\right)\leq 0\right\}.
\end{equation}
Each  $W^{\rm{s}}(x, \xi)$ coincides with 
the collection of the initial speed vectors of the geodesics asymptotic to $\xi$  and can be identified with $\M$. Associated with ${\bf E}^{\rm ss}$ are the  \emph{strong stable manifolds}
\begin{equation}\label{stable manifold}
W^{\rm{ss}}(x, \xi):=\left\{(y, \eta)\in\M\times \partial\M:\ \limsup\limits_{t\to +\infty}\frac{1}{t}\ln {\mbox{dist}}\left({\bf \Phi}_t(y, \eta),  {\bf \Phi}_t(x, \xi)\right)<0\right\}.
\end{equation}
Each $W^{\rm{ss}}(x, \xi)$, locally, is a $C^{k-1}$ graph from ${\bf E}_{(x, \xi)}^{\rm
ss}$ to ${\bf E}_{(x, \xi)}^{\rm c}\oplus{\bf E}_{(x, \xi)}^{\rm
su}$ and is tangent to ${\bf E}^{\rm{ss}}_{(x, \xi)}$  (\cite{SFL}).   It is true that 
\[{\bf \Phi}_t \left(W^{\rm{ss}}(x, \xi)\right)= W^{\rm{ss}}\left({\bf \Phi}_t(x, \xi)\right)\] 
and the union of these images is just  the stable manifold, i.e.,
 \[W^{\rm{s}}(x, \xi)=\bigcup_{t\in \Bbb R}{\bf \Phi}_t \left(W^{\rm{ss}}(x, \xi)\right).\]
 The weak and strong unstable manifolds, denoted by  $W^{\rm{u}}(x, \xi)$ and $W^{\rm{su}}(x, \xi)$, respectively,  can be defined similarly as in (\ref{weak stable manifold}) and (\ref{stable manifold})  by reversing the time.  They have tangents  ${\bf E}^{\rm cu}:={\bf E}^{\rm c}\oplus {\bf E}^{\rm su}$ and ${\bf E}^{\rm su}$, respectively.

  The geodesic flow ${\bf{\Phi}}_t$ on $S\M$ naturally descends to the geodesic flow $\Phi_t$  on  $g$-unit tangent bundle $SM$, carrying the tangent splitting and the  corresponding submanifolds downstairs.  Indeed, 
  the action of  $G$ on the tangent bundle $\bf E$ (where $\bf E$ denotes any one of ${\bf E}^{\rm ss}, {\bf E}^{\rm su}$ and ${\bf E}^{\rm c}$) satisfies $\psi ({\bf E}(x, \xi))={\bf E}(D\psi(x, \xi))$ for all $\psi\in G$ so that it defines the $D\Phi_t$-invariant subbundles $E^{\rm ss}, E^{\rm su}$  and $E^{\rm c}$ of $TSM$,  the so-called stable, unstable and central bundles. We see that $E^{\rm c}$ is tangent to the flow direction and  $E^{\rm ss}, E^{\rm su}$ are such that 
  \begin{itemize}
\item[i)] $\|D{\Phi}_t {w}\|\leq Ce^{-ct}\|{w}\|$ for ${w}\in {E}^{\rm ss}$, $t>0$.
\item[ii)] $\|D{\Phi}_t ^{-1}{w}\|\leq Ce^{-ct}\|{w}\|$   for ${w}\in {E}^{\rm su}$, $t>0$.
\end{itemize}
Similarly, the action of $G$ on the submanifolds $W$ (where $W$ denotes any one of $W^{\rm{s}}, W^{\rm{ss}}, W^{\rm{u}}$ and $W^{\rm{su}}$)
satisfies $\psi(W(x, \xi))=W(D\psi(x, \xi))$ for all $\psi\in G$ so that it defines the stable, strong stable, unstable and 
strong unstable manifolds of the geodesic flow on $SM$, which have tangents $E^{\rm ss}\oplus E^{\rm c}$, $E^{\rm ss}$,  $E^{\rm su}\oplus E^{\rm c}$ and $E^{\rm su}$, respectively. In particular,  the collection of $W^{\rm{s}}(x, \xi)$ defines a foliation $\mathcal{W}=\{W^{\rm{s}}(v)\}_{v\in SM}$ on $SM$, the so-called \emph{stable foliation} of $SM$.   Each $W^{\rm{s}}(x, \xi)$ can be identified with $\M\times \{\xi\}$.  Hence the quotients $W^{\rm{s}}(v)$  are naturally endowed with the Riemannian metric induced from $\wt{g}$.  They are $C^{k-1}$ immersed submanifolds of $SM$ depending continuously on $v$ in the $C^{k-1}$ topology (\cite{SFL}).

\subsection{Harmonic measure for the stable foliation}\label{Sec-2.3}  We continue to assume $g\in \Re^k(M)$.  Associated with the stable foliation $\W$ is the harmonic measure which is closely related to the leafwise Brownian motion. 
Write $\Delta^{\mathcal{W}}$ for the leafwise Laplace operator of $\W$, which acts on functions that are of class $C^2$ along the  leaves of $\mathcal{W}$. A probability measure ${\bf m}$ on $SM$ is called \emph{harmonic} if it satisfies, for any $C^2$ function $f$ on $SM$,
\[
\int_{SM}\Delta^{\mathcal{W}}f\ d{\bf m}=0.
\]
Since $(M, g)$ is negatively curved,  there is a unique harmonic measure  $\bf m$ associated to the stable foliation  (\cite{Ga}). Let $\wt{\bf m}$ be the  $G$-invariant extension of $\bf m$ to $S\M$. It is closely related to the Brownian motion on the stable leaves.   For $(x, \xi)\in S\M$, let 
\[
{\bf p}(t, (x, \xi), d(y, \eta)):=p(t, x, y)\ d{\rm Vol}_{\wt{g}}(y)\delta_{\xi}(\eta),
\]
where $\delta_{\xi}(\eta)$ is the Dirac function at $\xi$.  Then $\bf p$ is  just the transition probability function of the Brownian motion on $W^{s}(x, \xi)=\M\times \{\xi\}$  starting  from $(x, \xi)$.    Let $\wt{\Omega}_{+}$ be the space of continuous paths ${\omega}:  [0, +\infty)\to S\M$ equipped with the smallest $\sigma$-algebra
 for which the projections $R_t: {\omega}\mapsto {\omega}(t)$ are measurable.  Let $\{\P_{(x, \xi)}\}$ be the corresponding
Markovian family of ${\bf p}$ on  ${\Omega}_+$. Then for every $t>0$ and every Borel set $A\subset \M\times
\partial\M$, 
\[
\wt{\P}_{(x, \xi)}\left(\{{\omega}\in {\Omega}_+:  {\omega}(t)\in
A\}\right)=\int_{A}{\bf p}(t, (x, \xi), d(y, \eta)).
\]

\begin{prop}(\cite{Ga}) 
 The following hold true. 
\begin{itemize}
\item[i)]The measure $\wt{\bf m}$ satisfies, for any $f\in C^2(\M\times \partial\M)$ with compact support,
\[
\ \ \ \ \ \int_{\M\times \partial\M}\left(\int_{\M\times \partial\M}f(y,
\eta){\bf p}(t, (x, \xi), d(y, \eta))\right)\ d\wt{\bf m}(x,
\xi)=\int_{\M\times \partial\M}f(x, \xi)\ d\wt{\bf m}(x, \xi).
\]
\item[ii)] The measure $\wt{\P}=\int\wt{\P}_{(x, \xi)}\ d\wt{\bf m}(x, \xi)$ on ${\wt{\Omega}}_+$ is invariant under every  $t$-time shift mapping $\sigma_t:\ {\wt{\Omega}}_+\to {\wt{\Omega}}_+,\ \sigma_t(\wt{\omega}(s))=\wt{\omega}(s+t)$,  for $s>0$ and $\wt{\omega}\in {\wt{\Omega}}_+$.
\item[iii)] The measure $\wt{\bf m}$ can be expressed locally at $(x, \xi)\in \M\times \partial\M$ as $d\wt{\bf m}=dx\times d\wt{\bf m}_x$, where $dx$ is proportional to the volume element and $\wt{\bf m}_x$  is the hitting probability at $\partial\M$ of the Brownian motion starting at $x$.
\end{itemize}
\end{prop}

The group $G$ acts naturally and discretely on the space $ \wt \Om
_+$ with quotient  the space $\Om _+$ of
continuous paths in $SM$, and this action commutes with the shift
$\s_t, t\geq 0$. Therefore, the measure $\wt{\P}$ is the extension of
a finite, shift invariant  measure $\P$ on $\Om _+.$ We identify $SM$ with  $M_0\times \partial\M$, where $M_0$ is a connected fundamental domain of $(\M, \wt{g})$.   Hence we can also identify $\Om_+$ with the lift of its elements  in $\wt{\Om}_+$ starting from $M_0$.  We will continue to denote elements in $\Om_+$ by $\omega$ and will clarify the notation  whenever there is an ambiguity.  In this paper, we normalize the harmonic measure $\bf m$ to be a probability measure, so that $\P$ is also a probability measure. We denote by $\E_\P$ the corresponding expectation symbol.

A nice property for the laminated Brownian motion is that the semi-group $\s_t, t\geq 0,$ of transformations of $\Om_+$ has strong ergodic properties with
respect to the probability $\P$. 
\begin{prop}\label{mixing}(\cite{Ga}, cf. \cite[Proposition 2.3]{LS2}) The shift semi-flow $\s_t, t \geq 0, $ is mixing on
$(\Om _+ , \P)$ in the sense for any bounded measurable functions
$f_1,f_2$ on $\Om_+$, 
\[ \lim\limits_{t\to +\infty} \E_{\P} (f_1\,(f_2\circ \s_t)) \; = \; \E_{\P}( f_1) \E_{\P}(f_2) .\]
\end{prop}

\subsection{Busemann function and the linear drift}\label{Sec-Busemann}
In this subsection,   we derive (\ref{ell-lambda}). 

Let $g\in \Re^k(M)$.  For ${\bf v}=(x, \xi)\in M_0\times \partial \M$, the projection on $\M$ of the law of $\P_{\bf v}$ on $W^s(x, \xi)=\M\times \{\xi\}$ is the same as that of $\P_x$ of the Brownian motion on $\M$ starting from $x$. For $\om\in \Om_+$,  we still denote by $\om$ its projection to $\M$.  By ergodicity of $\P$ with respect to the shift map $\sigma_t$ (Proposition \ref{mixing}), for $\P$-almost all path $\om\in \Om_+$, its leafwise linear drift coincides with $\ell$.

Since $\wt{g}$ is negatively curved, for $\P$-almost all path $\om$,  $\om(t)$ tends to a point in the geometric boundary $\partial \M$ (\cite{K1}). Write  $\om(\infty):=\lim_{t\to +\infty}\om(t)$. Roughly speaking, $\om$ follows $\gamma_{\om(0), \om(\infty)}$. Hence the drift of $\om(t)$ from $\om(0)$ can be measured via its shadow on  $\gamma_{\om(0), \om(\infty)}$.  A candidate function for this measurement is the Busemann function. 
Let $x_0\in \M$ be a reference point.  For $y, z\in \M$, define
\[
b_{x_0,  y}(z):=d(z, y)-d(x_0, y).
\]
The assignment of $y\mapsto b_{x_0, y}$ is continuous, one-to-one and
takes value in a relatively compact set of functions for the
topology of uniform convergence on compact subsets of $\M$. The
Busemann compactification of $\M$ is the closure of $\M$ for that
topology (\cite{BGS})  and it coincides with the geometric compactification in the negative curvature case (see \cite{Ba}). 
So
for each ${\bf v}=(x, \xi)\in \M\times \pp\M$, the function 
\[
b_{{\bf v}}(z):=\lim\limits_{y\to \xi}b_{x, y}(z), \ \mbox{for}\
z\in \M, 
\]
is well-defined and is called the \emph{Busemann function
at ${\bf v}$}.  
It is known (\cite{EO}) that,   if we consider $b_{\bfv}$ as a function defined on $W^s(x,\xi),$ then 
\begin{equation}\label{Buse-geo}
\nabla b_{{\bf v}}(z)=-\overline{X}(z, \xi). 
\end{equation}
The difference between $b_{\bfv}(y)$ and $b_{\bfv}(y')$ is preserved  when $(y, \xi)$ and $(y', \xi)$ are driven by the geodesic flow $\Phi_t$.  Hence
\[
W^{\rm{ss}}(\bfv)=\left\{(y, \xi):\  b_{\bfv}(y)=b_{\bfv}(x)\right\}.
\]
 Note that $W^{\rm{ss}} (\bf v) $
 locally is a $C^{k-1}$ graph from ${\bf E}_{\bfv}^{\rm
ss}$ to ${\bf E}_{\bfv}^{\rm c}\oplus{\bf E}_{\bfv}^{\rm
su}$ and is tangent to ${\bf E}^{\rm ss}_{\bfv}$. So, by the Jacobi tensor characterization of ${\bf E}^{\rm ss}_{\bfv}$ and (\ref{Buse-geo}), it is
true (\cite{Esc, HIH}) that
\[
\nabla_{{\bf w}}(\nabla b_{\bfv})(x)=-S'_{\bfv}(0)({ \bf w}),\ 
 \forall { \bf w}\in T_{x}\M.
\]
Thus, 
\begin{equation}\label{Div-trace}\Delta_{x}b_{\bfv}=-{\rm Div}\overline{X}=-{\mbox{Trace of}}\ S'_{\bfv }(0), \end{equation} which is the mean curvature of  the set of footpoints of $W^{\rm{ss}}(x, \xi)$. 
Note that for  each $\psi\in G$,
\[
b_{(x_0, \psi\xi)}(\psi x)=b_{(x_0 , \xi)}(x)+b_{(\psi ^{-1}x_0, \xi)}( x_0).
\]
Hence  $\Delta_x b_{(x_0, \xi)}$ satisfies $\Delta_{\psi
x}b_{(x_0, \psi\xi)}=\Delta_x b_{(x_0, \xi)}$   and defines a function $\mathtt{B}$ on the
unit tangent bundle $SM$, which is called the \emph{Laplacian of the
Busemann function}. The function $\mathtt{B}$ is a H\"{o}lder continuous function on $SM$ by the H\"{o}lder continuity of the strong stable tangent bundles (\cite{Ano}, see Section \ref{Sec-Anosov}).

Now, we can derive the integral formula of the linear drift using the geodesic spray and the harmonic measure (\cite{K1}).  For $\P$-almost all  path
${\om}\in {\Om}_{+}$, let  ${\bf v}:=\om(0)$ and  $\eta:={\om}(\infty)\in \pp\M$.  When $t$
goes to infinity, the process $b_{{\bf v}}({\om}(t))-d(x,
{\om}(t))$ converges $\P$-a.e. to the a.e. finite
number $-2(\xi|\eta)_x$, where the Gromov product $(\cdot|\cdot)_x$ is such that
\begin{equation}\label{Gromov-product-def-sec3}
(\xi|\eta)_x:=\lim\limits_{y\to \xi, z\to \eta}(y|z)_x\ \mbox{and}\  (y|z)_x:=\frac{1}{2}\left(d(x,
y)+d(x, z)-d(y, z)\right).
\end{equation}
So for $\P$-almost all $\om\in \Om_+$, we have
\[
\lim\limits_{t\to +\infty}\frac{1}{t}b_{{\bf
v}}({\om}(t))=\ell.
\]
Using the fact that the leafwise  Brownian motion has generator  $\Delta$ and is ergodic with invariant measure ${\bf m}$ on $SM$, we obtain 
\begin{eqnarray}
  \ell&=&\lim_{t\to +\infty}\frac{1}{t}\int_{0}^{t}\frac{\pp}{\pp s} b_{{\bf v}}({\om}(s))\ ds\notag\\
  &=& \lim_{t\to +\infty}\frac{1}{t}\int_{0}^{t}\Delta b_{{\bf v}}({\om}(s))\ ds \, \left(= \, \int _{\scriptscriptstyle{M_0\times \pp\M}}\Delta b_{{\bf v}}\, \  d{\bf m}\right)\notag\\
  &=& -\int_{\scriptscriptstyle{M_0\times \pp\M}} {\rm{Div}^{\W}}(\overline{X})\
  d{\bf{m}},\label{equ-1.1}
\end{eqnarray}
where ${\rm{Div}^{\W}}$ is the laminated divergence operator for the stable foliation $\W$. Since on each leaf we have ${\rm{Div}^{\W}}(\overline{X})={\rm{Div}}(X)$, (\ref{equ-1.1}) reduces to  (\ref{ell-lambda}). But that will not simplify the discussion of the regularity of the linear drift under metric changes since ${\rm{Div}}(X)$ is essentially a leafwise object. In contrast, (\ref{equ-1.1}) is more suitable for this purpose because of the natural connection between the geodesic spray $\overline{X}$ and the geodesic flow.

\section{Regularity of the  linear drift}\label{sec-regularities-linear drift}

In this section, we assume Theorem \ref{diff-HK-estimations-gen} holds true. 
We first prove  Theorem \ref{main} by showing  the regularities of  ${\rm{Div}^{\W}}\overline{X}$ and ${{\bf m}}$ under a one parameter family  of $C^k$ deformation of metrics in $\Re^k(M)$ and then prove Theorems \ref{Critical} and \ref{linear}.

\subsection{Regularity of the leafwise divergence term  ${\rm{Div}}^{\W}\overline{X}$}
Clearly, the geodesic sprays of a  metric $g\in \Re^k(M)$ form a $C^{k-1}$ vector field which varies  $C^{k-1}$ with respect to $C^k$ metric change.  But this does not imply the regularity of ${\mbox{Div}^{\W}}\overline{X}$ with respect to the metric changes  since we are considering the leafwise divergence.

Laminate $S\M=\M\times \partial \M$ into  stable leaves $\{W^{\rm{s}}(x, \xi)=\M\times\{\xi\}\}$, where each leaf can be identified with $(\M, \wt{g})$, but is only H\"{o}lder continuous in  the $\xi$-coordinate (see Section \ref{Sec-Anosov}).  Consequently,  $\overline{X}(y, \xi)\in TW^{\rm{s}}(x, \xi)$ is  $C^{k-1}$ in the $y$-coordinate, but is only H\"{o}lder continuous in the $\xi$-coordinate.  Let $g'\in \Re^k(M)$ be another metric.  Its geometric boundary $\partial \M_{g'}$ can be identified with $\partial \M$. But the $\wt{g}'$-geodesic spray  $\overline{X}_{\wt{g}'}(x, \xi)$ differs from $\overline{X}(x, \xi)$  and the  divergence operator on the   $\wt{g}'$-stable leaf $W^{\rm{s}}_{\wt{g}'}(x, \xi)$ differs from that on $W^{\rm{s}}(x, \xi)$.  Both difference contribute to the change  of $({\mbox{Div}^{\W}}\overline{X})(x, \xi)$ in  metrics. This, by (\ref{Div-trace}),  can be understood by a study of the regularity of  $\ov{X}$ and ${\bf E}^{\rm ss}$ in $\Re^k(M)$.

Assume $g\in \Re^k(M)$. The set of $\wt{g}$-oriented geodesics in $\M$  can be identified with $\partial^2\M:=(\partial\M\times \partial\M)\backslash\{(\xi, \xi):\
\xi\in \partial\M\}$. Indeed, for $(x, \xi)\in
S\M$, let $\gamma: \Bbb R\mapsto \M$ be the unique geodesic
with $\dot{\gamma}(0)=(x, \xi)$ and write $\partial^+\gamma:=\lim_{t\to
+\infty}\gamma(t)$ and $\partial^-\gamma:=\lim_{t\to
-\infty}\gamma(t)$.  The mapping $\gamma\mapsto (\partial^+\gamma,
\partial^-\gamma)$  establishes a homeomorphism between
the set of all oriented geodesics in $(\M, \wt{g})$ and
$\partial^2\M$.  Consequently, for any $g'\in \Re^k(M)$, the mapping $D_{g'}:\partial^2(\wt{M})\to \partial^2(\wt{M}_{\wt{g}'})$ induced from the identity isomorphism from $G$ to itself can be viewed as a homeomorphism between the set of oriented geodesics in $(\wt{M}, \wt{g})$ and $(\wt{M}, \wt{g}')$.
 Further realize points from $S\wt{M}_{\wt{g}'}$ by pairs $(\gamma, y)$,
where $\gamma$ is an oriented geodesic and $y\in \gamma$.  For $g'$ close to $g$, we obtain a map $\wt{F}_{g'}: S\M\to
S\M_{\wt{g}'}$ which sends $(\gamma, y)\in S\wt{M}$ to
\[
\wt{F}_{g'}(\gamma, y)=(D_{g'}(\gamma), y'),
\]
where $y'$ is the unique intersection point of $D_{g'}(\gamma)$ and
the hypersurface $\{\exp_{\wt{g}}Y:\ Y\perp {\bf v}\}$ with ${\bf
v}$ being the vector in $S_y\wt{M}$ pointing at $\partial^+\gamma$.
  The map $\wt{F}_{g'}$ is a homeomorphism between $S\M$ and $S\M_{\wt{g}'}$ which preserves the geodesics, i.e.,
  sending $\wt{g}$-geodesics  to $\wt{g}'$-geodesics,   and is referred to as a $(\wt{g}, \wt{g}')$-\emph{Morse
  correspondence map.}  The restriction of $\wt{F}_{g'}$ to geodesics asymptotic to $\xi\in \partial\M$ is a homeomorphism from  $W^{\rm{s}}_{\wt{g}}(x, \xi)$ to $W^{\rm{s}}_{\wt{g}'}(x, \xi)$.   Let  $\wt{\pi}_{g'}: S\M_{\wt{g}'}\to S\M$ be the
map sending ${\bfv}$ to $\bfv/\|\bfv\|_{\wt{g}}$ which records the direction
information points of $S\M_{\wt{g}'}$ in $S\M$.  Then $\wt{\pi}_{g'}\circ
\wt{F}_{g'}$ is a homeomorphism between $S\M$ and itself.

The map $\wt{F}_{g'}$ induces a homeomorphism $F_{g'}$ between $SM$ and $SM_{g'}$ which sends $g$-geodesics to $g'$-geodesics and is called a  \emph{$(g, g')$-Morse
  correspondence map}. For
any sufficiently small $\epsilon$, if $g'$ is sufficiently close to
$g$, then $F_{g'}$ is such that the footpoint of $F_{g'}(v)$
belongs to the hypersurface of points $\{\exp_gY:\ Y\perp v, \
\|Y\|_g<\epsilon\}$, where $v$ is the projection of ${\bf v}$ in
$SM$.  Let  $\pi_{g'}: SM_{g'}\to SM$ be the natural projection
map sending $v$ to $v/\|v\|_{g}$.  Then $\pi_{g'}\circ
F_{g'}$ is a homeomorphism between $SM$ and itself.

For $g'$ in a small neighborhood of $g$ in $\Re^k(M)$,  let ${\bf E}_{g'}$ (resp. $E_{g'}$) denote any one of ${\bf E}_{g'}^{\rm ss}, {\bf E}_{g'}^{\rm su}$ and ${\bf E}_{g'}^{\rm c}$ (resp. any one of ${E}_{g'}^{\rm ss}, {E}_{g'}^{\rm su}$ and ${E}_{g'}^{\rm c}$).   We also regard ${\bf E}_{g'}$ (resp. $E_{g'}$) as a mapping from $S\M_{\wt{g}'}$ (resp. $SM_{g'}$) to its tangent bundle.  Of our special interest, is the regularity of the mappings 
$g'\mapsto \wt{\pi}_{g'}\circ \wt{F}_{g'}$, $g'\mapsto D\wt{\pi}_{g'}\circ {\bf E}_{g'}$. Equivalently, we can consider the regularity of the downstairs mappings $g'\mapsto \pi_{g'}\circ F_{g'}$ and $g'\mapsto D{\pi}_{g'}\circ {E}_{g'}$, for which, we can take advantage of the compactness of $M$ to construct certain manifolds of maps so that the implicit function theory applies (\cite{LMM, KKPW}).

Let
$\mathcal{H}^{k-1}(SM)$ be the Banach space of $C^{k-1}$
vector fields on $SM$ endowed with the topology of uniform $C^{k-1}$
convergence on compact subsets. Let $\underline{X}_{g}$ be the vector field generating the $g$-geodesic flow. Then $\underline{X}_{g'}$, the projection  (via $D\pi_{g'}$) of the generating vector field of the $g'$-geodesic flow on $SM_{g'}$,  belongs to $\mathcal{H}^{k-1}(SM)$
and is $C^{k-1}$ close to $\underline{X}_{g}$ whenever $g'$ is $C^k$ close to
$g$.  For $\alpha\in [0, 1)$, let $C^{\alpha}(SM, N)$ denote the
Banach space of $\alpha$-H\"{o}lder (or continuous for
$\alpha=0$) maps from $SM$ to a Banach space $N$ endowed with the
topology given by the $\alpha$-H\"{o}lder norm on $SM$.  Consider\[
C_{\Phi}^{\alpha}(SM, SM):=\!\left\{F\in C^{\alpha}(SM, SM):\
D_{\Phi}F(v):=\!\left.\frac{d}{dt} F(\Phi_t(v))\right|_{t=0} \ \! \mbox{exists and
is}\ \alpha\mbox{-H\"{o}lder}\right\} 
\]
with the topology of the norm $\|F\|+\|D_{\Phi}F\|_{\alpha}$, where $\|\cdot\|_{\alpha}$ denotes the $\alpha$-H\"{o}lder norm, together with   the
mapping
\begin{align*}
&\Psi:\ \mathcal{H}^{k-1}(SM)\times C_{\Phi}^{\alpha}(SM, SM)\times
C^{\alpha}(SM, \Bbb R)\to C^{\alpha}(SM, TSM)\\
&\ \ \ \ \ \ \ \ \ \ \ \ \ \ \ \ \ \ \ \ \ \  \Psi(Y, F, f)=Y\circ
F-f\cdot D_{\Phi}F.
\end{align*}
By hyperbolicity of  the $g$-geodesic flow $\Phi_t$, the
implicit function theory applies to $\Psi$ if we further require
$F\in C^{\alpha}_{\Phi}(SM, SM)$ to be such that the footpoint of
$F(v)$ lies in $\{\exp_g(w):\ w\perp v\}$ for any $v\in SM$. The following structural stability theorem is due to de la
Llave-Marco-Moriy\'{o}n (\cite{LMM}) for continuous case and 
Katok-Knieper-Pollicott-Weiss (\cite{KKPW}) for H\"{o}lder continuous case.

\begin{prop}(\cite[Proposition 2.2]{KKPW})\label{Con-regularity-1} For $g\in \Re^{k}(M)$, there exist $\alpha\in (0, 1)$ and a neighborhood $\mathcal{U}\subset
  \mathcal{H}^{k-1}(SM)$ of $\underline{X}_{g}$ and $C^{k-2}$ maps $\mathcal{U}\to C_{\Phi}^{\alpha}(SM,
  SM):\ Y\mapsto F_{Y}$ and $\mathcal{U}\to C^{\alpha}\left(SM, [\frac{1}{2}, +\infty)\right):\ Y\mapsto
  f_{Y}$ such that $Y\circ F_{Y}=f_{Y}D_{\Phi}F$. Moreover, the maps $\mathcal{U}\to C_{\Phi}^0(SM, SM):\ Y\mapsto
  F_{Y}$ and $\mathcal{U}\to C^0\left(SM, [\frac{1}{2}, +\infty)\right):\ Y\mapsto
  f_{Y}$ are $C^{k-1}$. 
\end{prop}

Define $C^{\alpha}(S\M, N), C^{\alpha}_{\Phi}(S\M, N)$ analogously as  $C^{\alpha}(SM, N)$,
$C^{\alpha}_{\Phi}(SM, N)$.  A consequence of Proposition \ref{Con-regularity-1} is 

\begin{cor}\label{g-F_g}Assume  $g\in \Re^{k}(M)$.  There exist $\alpha\in (0, 1)$ and a
neighborhood $\mathcal{V}$ of $g$ in $\Re^k(M)$ such that the map
$g'\in \mathcal{V}\mapsto \pi_{g'}\circ F_{g'}$ is $C^{k-2}$ into
$C_{\Phi}^{\alpha}(SM, SM)$ and is $C^{k-1}$ into $C_{\Phi}^0(SM, SM)$; the map $g'\in \mathcal{V}\mapsto \wt{\pi}_{g'}\circ \wt{F}_{g'}$  is $C^{k-2}$ into $C_{\Phi}^{\alpha}(S\M, S\M)$ and is  $C^{k-1}$ into $C_{\Phi}^0(S\M, S\M)$. 
\end{cor}

The regularity of $g'\mapsto D{\pi}_{g'}\circ {E}_{g'}$ and  $g'\mapsto D\wt{\pi}_{g'}\circ {\bf E}_{g'}$ can be analyzed analogously (\cite{Con}).  Let $\mathcal{G}$ be
the Grassmann bundle of $u$-planes on $TSM$, where $u=\mbox{dim}
E^{\rm  su}_{g}$. Let $C_{\Phi}^{\alpha}(SM, \mathcal{G})$ be the
space of $\alpha$-H\"{o}lder maps $\widehat{F}: SM\to \mathcal{G},
\widehat{F}(v)=(F(v), E(v))$, where $F\in C_{\Phi}^{\alpha}(SM,
SM)$, with the topology of the $\alpha$-H\"{o}lder norm on $F,
D_{\Phi}F$ and $E$.  Then instead of $\Psi$, one can consider the
maps
\begin{eqnarray*}
&&\widehat{\Psi}_{\pm}:\ \mathcal{H}^{k-1}(SM)\times
C_{\Phi}^{\alpha}(SM, \mathcal{G})\times
C^{\alpha}(SM, \Bbb R)\to C^{\alpha}(SM, TSM\oplus \mathcal{G})\\
&&\ \  \widehat{\Psi}_{\pm}(Y, \widehat{F}, f)=\left(Y\circ F-f\cdot
D_{\Phi}F,\ D\psi_{\tau_Y(v)}\circ F(\Phi_{\pm 1}(v))E(\Phi_{\pm
1}(v))\right),
\end{eqnarray*}
where $\psi_t$ is the time $t$ map of the flow generated by $Y$ and
$\tau_{Y}$ is the time change such that
$$\psi_{\tau_Y(v)}\circ F_{Y}(\Phi_{\pm 1}(v))=F_{Y}(v),\  \forall v\in SM.$$
Again, by hyperbolicity of the flow generated by $Y$ which is close to $\underline{X}_{g}$ and the invariance of the corresponding strong stable and
unstable bundles, denoted by $E_{Y}^{\rm ss}, E_{Y}^{\rm su}$, the implicit
function theory applies for $\widehat{\Psi}_+, \widehat{\Psi}_{-}$
and gives the following.

\begin{prop}(\cite[Proposition 2.1]{Con})\label{Con-regularity-2} For $g\in \Re^{k}(M)$, there exist a neighborhood
$\mathcal{U}$ of $\underline{{X}}_{g}$ in $\mathcal{H}^{k-1}(SM)$ and $\alpha\in
(0, 1)$ such that the map $\mathcal{U}\to C^{\alpha}_{\Phi}(SM,
\mathcal{G}):\ Y\mapsto (v\mapsto E_{Y}\circ F_{Y}(v))$ is $C^{k-3}$ and
the map $\mathcal{U}\to C^{0}_{\Phi}(SM, \mathcal{G}):\ Y\mapsto
E_{Y}\circ F_{Y}$ is $C^{k-2}$, where $E_{Y}=E^{\rm ss}_{Y} \ \mbox{or}\
E^{\rm su}_{Y}$. 
\end{prop}

Let $\wt{\mathcal{G}}$ be 
the Grassmann bundle of $u$-planes on $TS\M$  (where $u=\mbox{dim}
{\bf E}_{g}^{\rm su}$) and define $C^{\alpha}_{\Phi}(S\M, \wt{\mathcal{G}})$ in analogy with   $C_{\Phi}^{\alpha}(SM, \mathcal{G})$. The following is an  application of Proposition \ref{Con-regularity-2} to the geodesic flows. 

\begin{cor}\label{g_E_g}
There exist $\alpha\in (0, 1)$ and a neighborhood $\mathcal{V}$ of
$g$ in $\Re^k(M)$ such that the map $g'\in \mathcal{V}\mapsto
D\pi_{g'}\circ E_{g'}\circ F_{g'}$ is $C^{k-3}$ into
$C^{\alpha}_{\Phi}(SM, \mathcal{G})$ and is $C^{k-2}$ into
$C^0_{\Phi}(SM, \mathcal{G})$, where $E_{g'}$ is any one of ${E}_{g'}^{\rm ss}, {E}_{g'}^{\rm su}$ and ${E}_{g'}^{\rm c}$. Similarly,  the map $g'\in \mathcal{V}\mapsto
D\wt{\pi}_{g'}\circ {\bf E}_{g'}\circ \wt{F}_{g'}$ is $C^{k-3}$ into
$C^{\alpha}_{\Phi}(S\M, \wt{\mathcal{G}})$ and is $C^{k-2}$ into
$C^0_{\Phi}(S\M, \wt{\mathcal{G}})$, where ${\bf E}_{g'}$ is any one of ${\bf E}_{g'}^{\rm ss}, {\bf E}_{g'}^{\rm su}$ and ${\bf E}_{g'}^{\rm c}$.
\end{cor}

For  $\lambda\in (-1,
1)\mapsto g^{\lambda}\in \Re^k(M)$, we write ${\ov{X}}^{\l}$ for the $\wt{g}^{\l}$-geodesic spray, 
$({\bf E}^{\rm ss})^{\lambda}$ for the $\wt{g}^{\l}$-stable bundle  and   ${\rm
Div}^{\l}$ for the divergence operator associated with the $\wt{g}^{\l}$-stable foliation. 

\begin{prop}\label{regularity-Div} Let $g\in \Re^k(M)$. There exist $\alpha\in (0,
1)$ and a neighborhood $\mathcal{V}_g$ of $g$ in $\Re^{k}(M)$ such
that for any $C^{k}$ curve  $\lambda\in
(-1, 1)\mapsto g^{\lambda}\in \mathcal{V}_{g}$ with $g^0=g$, 
\begin{itemize}
\item[i)] $\lambda\mapsto \ov{X}^{\lambda}$ is  $C^{k-3}$ into  $C^{\alpha}(\M\times \partial\M,  TT\M)$  and is $C^{k-2}$ into $C^{0}(\M\times \partial\M, TT\M)$,
\item[ii)] $\l\mapsto ({\bf E}^{\rm ss})^{\lambda}$ is $C^{k-3}$ into  $C^{\alpha}(\M\times \partial\M,  \wt{\mathcal{G}})$ and is $C^{k-2}$ into $C^{0}(\M\times \partial\M, \wt{\mathcal{G}})$, and 
\item[iii)] $\lambda\mapsto {\rm
  Div}^{\lambda}\ov{X}^{\lambda}$ is $C^{k-3}$ into $C^{\alpha}(\M\times \partial\M, \Bbb
  R^+)$ and is $C^{k-2}$ into $C^{0}(\M\times \partial\M, \Bbb
  R^+)$.
  \end{itemize}
\end{prop}

\begin{proof}Express the $(\wt{g},
\wt{g}^{\lambda})$-Morse correspondence map $\wt{F}_{g^{\lambda}}$ from $\M\times
\partial \M$ to itself as
\[
\wt{F}^{\lambda}(x, \xi)=\big(f_{\xi}^{\lambda}(x), \xi\big), \ \forall (x,
\xi)\in \M\times \partial\M,
\]
where  $f^{\l}_{\xi}$ records the change of the footpoint for the unit vector pointing at $\xi$ in the boundary.  
For $(x, \xi)\in \M\times \partial \M$, we transform  $\overline{X}_{\wt{g}}(x,
\xi)$ to  $\overline{X}_{\wt{g}^{\l}}(x,
\xi)$  in three steps: the first is to follow the footpoint of the  inverse of the $(g^{\l},g)$-Morse correspondence  from $\overline{X}_{\wt{g}}(x,
\xi)$ to $\overline{X}_{\wt{g}}((f_{\xi}^{\l})^{-1}(x),
\xi)$ with the constraint that the vector remains  within $TW^{\rm{s}} (x,\xi )$;   the second is to use the  $(g^{\l}, g)$-Morse correspondence  from $\overline{X}_{\wt{g}}((f_{\xi}^{\l})^{-1}(x),
\xi)$ to $\overline{X}_{\wt{g}^{\l}}(x,
\xi)/\|\overline{X}_{\wt{g}^{\l}}(x,
\xi)\|_{\wt{g}};$  the third is to adjust the length of  $\overline{X}_{\wt{g}^{\l}}(x,
\xi)/\|\overline{X}_{\wt{g}^{\l}}(x,
\xi)\|_{\wt{g}}$ to be 1 in the metric $\wt{g}^{\l}$. Hence, 
\begin{align*}
 & \overline{X}_{\wt{g}^{\l}}(x,
\xi)-\overline{X}_{\wt{g}}(x,
\xi)\\
&=\left(\overline{X}_{\wt{g}^{\l}}(x,
\xi)-\frac{\overline{X}_{\wt{g}^{\l}}(x,
\xi)}{\|\overline{X}_{\wt{g}^{\l}}(x,
\xi)\|_{\wt{g}}}\right)+\left(\frac{\overline{X}_{\wt{g}^{\l}}(x,
\xi)}{\|\overline{X}_{\wt{g}^{\l}}(x,
\xi)\|_{\wt{g}}}-\overline{X}_{\wt{g}}((f_{\xi}^{\l})^{-1}(x),
\xi)\right)\\
&\ \ \ \ \ \ \ \ \ \ \ \ \ \ \ \ \ \ \ \ \ \ \ \ \ \ \ \ \ \ \ \ \ \ \ \ \ \ \ \ \ \   +\left(\overline{X}_{\wt{g}}((f_{\xi}^{\l})^{-1}(x),
\xi)-\overline{X}_{\wt{g}}(x,
\xi)\right)\\
&=:(a)_{\l}+(b)_{\l}+(c)_{\l}. 
\end{align*}
Note that $(a)_0, (b)_0$ and  $(c)_0$ are all zero. So the regularity  of  $\lambda\mapsto \ov{X}^{\lambda}$ will follow from that of $(a)_{\l}, (b)_{\l}$ and  $(c)_{\l}$ by Taylor's formula.  This is true since $(a)_{\l}$ corresponds to length change and is $C^{k}$ in $\l$, $(b)_{\l}$ is $C^{k-2}$  (or $C^{k-3}$) in $\l$ depending on $\alpha=0$ (or not) by  Corollary \ref{g-F_g},  while $(c)_{\l}$ has the same regularity as $(b)_{\l}$ since $\ov{X}(x, \xi)$ is $C^{k-1}$ in the $x$-coordinate.

Similarly, we write ${\bf
v}^{\lambda}=\ov{X}^{\lambda}(x, \xi)$ and 
\begin{align*}
 ({\bf E}^{\rm ss})^{\lambda}({{\bf v}^{\lambda}})-({\bf E}^{\rm ss})^{0}({{\bf
  v}^{0}})=& \ \left(({\bf E}^{\rm ss})^{\lambda}({{\bf v}^{\lambda}})-({\bf E}^{\rm ss})^{0}((f_{\xi}^{\lambda})^{-1}(x),
  \xi)\right)\\
& \ \ +\left(({\bf E}^{\rm
ss})^{0}((f_{\xi}^{\lambda})^{-1}(x),
  \xi)-({\bf E}^{\rm ss})^{0}({{\bf
  v}^{0}})\right)\\
  =:& \  (d)_{\lambda}+(e)_{\lambda}.
\end{align*}
This means we can transport $({\bf E}^{\rm ss})^{0}({{\bf
  v}^{0}})$ to $({\bf E}^{\rm ss})^{\lambda}({{\bf v}^{\lambda}})$  in two steps: first is to transport $({\bf E}^{\rm ss})^{0}({{\bf
  v}^{0}})$ to $({\bf E}^{\rm ss})^{0}((f_{\xi}^{\lambda})^{-1}(x),
  \xi)$ along the tangent bundle of $W^{\rm{s}}(x, \xi)$ and follow the   footpoint of the  inverse of the $(\wt{g}^{\l},\wt{g})$-Morse correspondence; the second is to use the Morse correspondence for the stable bundle from $({\bf E}^{\rm ss})^{0}((f_{\xi}^{\lambda})^{-1}(x),
  \xi)$  to $ ({\bf E}^{\rm ss})^{\lambda}({{\bf v}^{\lambda}})$.  Note that $(d)_{0}, (e)_0$ are zero. The regularity  of  $\lambda\mapsto ({\bf E}^{\rm ss})^{\lambda}$ will follow from that of $(d)_{\l}, (e)_{\l}$ by Taylor's formula, which will follow by Corollary \ref{g_E_g} if we can show the  $C^{k-1}$ dependence of ${\bf E}^{\rm ss}(x,  \xi)$ on the $x$-coordinate.  This is
true because each ${\bf E}^{\rm ss}(y, \xi)$ is the tangent plane of
the strong stable manifold $W^{\rm{ss}}(y, \xi)$.  Locally,
$W^{\rm{ss}}(x, \xi)$ is a $C^{k-1}$ graph from ${\bf E}_{(x,
\xi)}^{\rm ss}$ to ${\bf E}_{(x, \xi)}^{\rm c}\oplus{\bf E}_{(x,
\xi)}^{\rm su}$. This means,  locally, $y\mapsto {\bf E}^{\rm ss}(y,
\xi)$ is $C^{k-1}$ along the leaf $W^{\rm{ss}}(x, \xi)$. On the other
hand, by invariance of the strong stable bundle with respect to the
geodesic flow, $y\mapsto {\bf E}^{\rm ss}(y, \xi)$ is smooth as $y$
varies on the geodesic passing through $x$ asymptotic to $\xi$. By
invariance of the strong stable leaf under  the  geodesic flow, $W^{\rm{ss}}(x, \xi)$ and the time direction (i.e. the direction of the geodesic spray) consist of a coordinate chart for $W^{\rm{s}}(x, \xi)$. This
shows, locally at $x$, $y\mapsto {\bf E}^{\rm ss}(y, \xi)$ is
$C^{k-1}$ along $W^{\rm{s}}(x, \xi)=\M\times \{\xi\}$.

Finally iii) is just an application  of ii) noting that for any $g\in \Re^k(M)$,  we have
\[
({\rm Div}\ov{X})(x, \xi)={\mbox{Trace of}}\ S'_{\bf v}(0),\ \forall  {\bfv}=(x, \xi)\in \M\times \partial\M, 
\] and the stable bundle ${\bf E}^{\rm ss}$ at $\bfv$
is the graph of the mapping $S'_{\bfv}(0)$, considered as a map
from $ \overline{N_0(\gamma _\bfv )}$ to $V_{\rm v}$ 
sending $Y$ to
$S'_{\bfv}(0)Y$, where  $\overline{N_0(\gamma _\bfv )} := \{ {\bf {w}}, {\bf w} \in H_{\rm v}, {\bf {w}} \perp \overline X_\bfv \}$. 
\end{proof}

\subsection{Regularity of the harmonic measure}In this subsection, we prove Theorem \ref{regularity-harmonic measure} following the sketch that  we gave in the Section 1.  

For $g^{\l}\in \Re^k(M)$, we introduce a metric on $\partial \M$ as follows.  Let $\varkappa>0$. For $x\in \M$, define
\begin{equation}\label{distance-bd}
d_x^{\varkappa, \l}(\zeta, \eta):=e^{-\varkappa(\zeta|\eta)_x^{\l}}, \
\forall \zeta, \eta\in \pp\M,
\end{equation}
where $(\cdot|\cdot)_x^{\l}$ is the Gromov product defined in (\ref{Gromov-product-def-sec3}) for $d_{\wt{g}^{\l}}$.  
 If $\varkappa$ is small,  each  $d_x^{\varkappa,\l}(\cdot, \cdot)$ defines a distance on $\pp\M$, the so-called
 $\varkappa$-Busemann distance  (\cite{K2}), which is related
to the $\wt{g}^{\l}$-Busemann functions $b^{\l}$ since
\begin{equation}\label{Bus-xi-zeta}
b_{{\bf v}}^{\l}(y)=\lim\limits_{\zeta, \eta\to
\xi}\left((\zeta|\eta)_y^{\l}-(\zeta|\eta)_x^{\l}\right),\ \mbox{for any} \
{\bf v}=(x, \xi)\in S\M,\ y\in \M.
\end{equation}
Let  ${\mathtt{b}}>0$. For continuous functions $f$ on $SM=M_0\times \partial \M$, define 
\[
\|f\|_{{\mathtt{b}}}^{\l}:=\sup\limits_{x, \xi}|\wt{f}(x, \xi)|+\sup\limits_{x, \xi_1,
\xi_2} |\wt{f}(x, \xi_1)-\wt{f}(x, \xi_2)|e^{{\mathtt{b}}(\xi_1|\xi_2)_x^{\l}}. 
\]
Let   $\mathcal{H}_{{\mathtt{b}}}^{\l}$ be the Banach space of continuous
functions $f$ on $SM$ with $\|f\|_{{\mathtt{b}}}^{\l}<+\infty$.  Elements of $\mathcal{H}_{{\mathtt{b}}}^{\l}$ are  continuous  on $SM$  and  H\"{o}lder continuous with respect to  the direction changes. 

 Recall that the transition probability of the $\wt{g}^{\l}$-Brownian motion on  the stable leaf $W^{{\rm{s}}}_{\wt{g}^{\l}}(x, \xi)=\M\times \{\xi\}$  starting  from $(x, \xi)$ is given by 
\[
{\bf p}^{\l}(t, (x, \xi), d(y, \eta)):=p^{
\l}(t, x, y)\ d{\rm Vol}^{\l}(y)\delta_{\xi}(\eta),
\]
where $\{p^{\l}(t, x, \cdot)\}_{x\in \M, t\in \Bbb R_+}$ is the transition probabilities of the  $\wt{g}^{\l}$-Brownian motion on $\M$,  $\delta_{\xi}(\eta)$ is the Dirac function at $\xi$  and ${\rm Vol}^{\l}$ is the $\wt{g}^{
\l}$ volume element.  Then ${\bf p}^{\l}$ descends to be the transition probability of $g^{\l}$-Brownian motion the stable leaves of $SM$: for   $(x, \xi), (y, \eta)\in SM=M_0\times \partial\M$, the transition  probability is 
\begin{align*}
{\bf q}^{\l}\left(t, (x, \xi), d(y, \eta)\right)=&\sum_{\beta\in G}{\bf p}^{\l}(t, (x, \xi), d(\beta y, \beta\eta))\\
=& \sum_{\beta\in G}p^{\l}(t, x, \beta y)d{\rm Vol}^{\l}(y)\delta_{\xi}(\beta \eta). 
\end{align*}
Let ${{\rm Q}}_t^{\l}$ $(t\geq 0)$ be given in (\ref{def-Q-t}). It defines  the action
of $[0, +\infty)$ on continuous functions $f$ on $SM$ which
describes the $\Delta^{\W}_{g^{\l}}$-diffusion.  It was shown in \cite{L} that for sufficiently small ${\mathtt{b}}>0$, 
there exists $T>0$ such that ${{\rm Q}}_{T}^{\l}$ is a contraction on $\mathcal{H}_{{\mathtt{b}}}^{\l}$ and hence,   as $t \to \infty $,
${{\rm Q}}_t^{\l}$ converges to the mapping $f\mapsto \int  f\ d{\bf m}^{\l}$
    exponentially in $t$ for $f\in \mathcal{H}_{{\mathtt{b}}}^{\l}$.  Thus,   each harmonic measure  ${\bf m}^{\l}$ is a fixed point of the dual operation $({{\rm Q}}_{T}^{\l})^*$ in the dual space $(\mathcal{H}^{\l}_{{\mathtt{b}}})^*$ with the weak topology, where 
   \[({{\rm Q}}_{T}^{\l})^*(\mu)(f):=\mu({{\rm Q}}_{T}^{\l}(f)),\ \mbox{for all}\ \mu\in (\mathcal{H}^{\l}_{{\mathtt{b}}})^*,\  f\in  \mathcal{H}_{{\mathtt{b}}}^{\l}. \]  
The following proposition shows  that $\mathcal{H}_{{\mathtt{b}}}^{\l}$ can be chosen to be independent of  $g^{\l}$.

\begin{prop}\label{Uni-contraction} Let $\mathcal{V}_g$ be as in Theorem \ref{diff-HK-estimations-gen}. For every ${\mathtt{b}}>0$ small enough, there exist $C>0$ and $\Bbbk<1$ such that,  for all $\l\in (-1, 1)$, $t>0$ and $f\in \mathcal{H}_{{\mathtt{b}}}^{0}$, 
\begin{equation*}
\left\|{\rm Q}_t^{\l}f-\int f\ d{\bf m}^{\l}\right\|_{{\mathtt{b}}}\leq C\Bbbk^t\|f\|_{{\mathtt{b}}}. 
\end{equation*}
\end{prop}
The proof of Proposition \ref{Uni-contraction} follows  \cite[Theorem 3]{L} for an individual metric.  The only  modification is to find a common H\"{o}lder continuous function space independent of the metrics where the contractions (of H\"{o}lder norm) happen.  Denote  $d$ and $(\xi \big| \eta)_x$ for the $\wt{g}^0$ distance and its Gromov product.  The key lemma is the following. 
 
 \begin{lem}\label{UC-step 0} Let $\mathcal{V}_g$ be as in Theorem \ref{diff-HK-estimations-gen}. There is a number ${\mathtt{b}}'>0$ such that for any $\mathtt{b}$, $0<{\mathtt{b}}<{\mathtt{b}}'$, there exists $\Bbbk_1<1$ such that for $t$ large enough, $x\in M_0$ and all $\xi, \eta$, $\xi\not=\eta$, we have for all $\l\in (-1, 1)$, 
\begin{equation*}
\E_{x, \xi}^{\l}\left(e^{-{\mathtt{b}}\left((\xi |\eta)_{\lfloor {\rm x}_t \rceil^{\l}}-(\xi |\eta)_x\right)}\right)<\Bbbk_1^t, 
\end{equation*}
where $\lfloor {\rm x}_t \rceil^{\l}$ denotes the  $\wt{g}^{\l}$-Brownian motion on $W^s(x, \xi)$ starting from  $(x, \xi)$ and $\E_{x, \xi}^{\l}$ denotes its corresponding expectation.
 \end{lem}

As a preparation for the proof of Lemma \ref{UC-step 0}, define on $M_0\times \partial \M\times\partial\M$ the  transition probabilities
\[
{\bf q}^{2, \l}(t, (x, \xi_1, \xi_2), d(y, \eta_1, \eta_2)):=\sum_{\beta\in G}p^{
\l}(t, x, \beta y)\ d{\rm Vol}^{\l}(y)\delta_{\xi_1}(\beta \eta_1)\delta_{\xi_2}(\beta\eta_2)
\]
and the corresponding operator ${\rm Q}_{t}^{2, \l}$ on continuous functions on $M_0\times \partial \M\times \partial \M$:
\[
{\rm Q}_t^{2, \l}f(x, \xi_1, \xi_2)=\int f(y, \eta_1, \eta_2) {\bf q}^{2, \l}\left((x, \xi_1, \xi_2), d(y, \eta_1, \eta_2)\right). 
\]
By analogy with the case of  ${\rm Q}_t^{\l}$, there is a unique ${\rm Q}_t^{2, \l}$-invariant probability measure on $M_0\times\partial \M\times \partial\M$ which is related to the harmonic measure ${\bf m}^{\l}$ as follows. 

\begin{lem}(\cite[Proposition 1]{L})\label{L-Prop 1} For each $g^{\l}\in \Re^k(M)$, with the above notations, there is a unique probability measure ${\bf m}^{2, \l}$ on $M_0\times\partial \M\times \partial\M$ satisfying
\[
\int {\rm Q}_t^{2, \l}f\ d{\bf m}^{2, \l}=\int f\ d{\bf m}^{2, \l}
\]
for all $f\in C(M_0\times\partial \M\times \partial\M, \Bbb R)$ and all positive $t$. The measure ${\bf m}^{2, \l}$ is characterized  by 
\[
\int f\ d{\bf m}^{2, \l}=\int_{M_0\times \partial \M} f(x, \xi, \xi)\ d{\bf m}^{\l}(x, \xi). 
\]
\end{lem}

For $g'\in \Re^k(M)$, let $\wt{g}'$ be its $G$-invariant extension to $\M$, ${\rm x}_t^{\wt{g}'}({\rm w})$ its Brownian motion on $\M$ and ${\bf m}^{\wt{g}'}$ its harmonic measure.  The following limit  exists almost surely:
\[
\lim\limits_{t\to +\infty}\frac{1}{t}b_{(x, \xi)}\left({\rm x}_t^{\wt{g}'}({\rm w})\right)=\int_{M_0\times \partial \M} \Delta^{\wt{g}'}b_{(x, \xi)}\ d{\bf m}^{\wt{g}'}=:\ell'_{g'}.
\]
As $g'\to g$, ${\bf m}^{\wt{g}'}\to {\bf m}$ and hence both $\ell'_{g'}, \ell_{g'}$ converges to $\ell$. We may assume the  neighborhood  $\mathcal{V}_g$ of $g$ in Theorem \ref{diff-HK-estimations-gen} is  such that $\underline{\ell}:=\min\limits_{g'\in \mathcal{V}_g}\{\ell_{g'}, \ell'_{g'}\}$ is positive.
Consequently,  for any curve $\l\to g^{\l}\in \mathcal{V}_g$, 
\[\min_{\l\in (-1, 1)}\{\ell_{g^\l}, \ell'_{g^\l}\}\geq \underline{\ell}>0.\]

\begin{lem}\label{L-Prop2}  Let $\mathcal{V}_g$ be as in Theorem \ref{diff-HK-estimations-gen}. For $T>0$ large enough,  for all $\l\in (-1, 1)$, $x\in M_0$ and  $\xi, \eta\in \partial\M$, $\xi\not=\eta$, 
\[
\frac{1}{T}\E_{x, \xi}^{\l}\left((\xi\big| \eta)_{{\rm x}_T^{\l}}-(\xi\big|\eta)_x\right)\geq \frac{1}{4}\underline{\ell}. 
\] 
\end{lem}
\begin{proof}We may assume $g^{\l}$ is defined for $\l\in [-1, 1]$.  Assume the conclusion is not  true. Then there exist $\l_n\in [-1, 1]$,  $T_n\in \Bbb R_+$, $T_n\to \infty$, and points $x_n, \xi_n, \eta_n, \xi_n\not=\eta_n$, such that
\begin{equation}\label{eq-UC-C-1}
\frac{1}{T_n}\E_{x_n, \xi_n}^{\l_n}\left((\xi_n\big|\eta_n)_{{\rm x}_{T_n}^{\l_n}}-(\xi_n\big|\eta_n)_{x_n}\right)<\frac{1}{4}\underline{\ell}. 
\end{equation}
By definition of the Gromov product $(\cdot |\cdot)$, for all $\xi\not=\eta\in \partial \M$, $y, z\in M_0$ and $\l\in [-1, 1]$, 
\[
\big|(\xi|\eta)_{y}-(\xi|\eta)_z\big|\leq 2d(y, z)\leq {\mbox Const.}\cdot d^{\l}(y, z),
\]
where the constant is independent of $\l$, $\xi, \eta, y$ and $z$. 
Hence by uniform continuity of $\l\mapsto p^{\l}(t, x, \cdot)$ in $x$, we can find $t_0$ small enough such that 
\begin{equation}\label{eq-UC-C-2}
\sup\limits_{\l\in [-1, 1]}\sup\limits_{0\leq t\leq t_0}\sup\limits_{x, \xi}\sup\limits_{\eta\not=\xi}\E_{x, \xi}^{\l}\left(\big|(\xi\big|\eta)_{\lfloor {\rm x}_t \rceil^{\l}}-(\xi\big|\eta)_{x}\big|\right)\leq \frac{1}{4}\underline{\ell}. 
\end{equation}
By using (\ref{eq-UC-C-1}), (\ref{eq-UC-C-2}) and suitably relabelling $\l_n, x_n, \xi_n$ and  $\eta_n$,  we can find a sequence $\l_j\in [-1, 1]$, a sequence of integers $N_j\to \infty$,  and points $x_j, \xi_j$ and  $\eta_j$ such that, for all $j$, 
\begin{equation}\label{eq-UC-C-3}
\frac{1}{N_j t_0}\E_{x_j, \xi_j}^{\l_{j}}\left((\xi\big|\eta)_{{\rm x}_{N_j t_0}^{\l_j}}-(\xi\big|\eta)_{x_j}\right)<\frac{1}{2}\underline{\ell}. 
\end{equation}
By passing to suitable subsequences, we may also assume that $\l_n$ converges to some $\l_0\in [-1, 1]$, as $n$ goes to infinity. 
For $\l\in [-1, 1]$, write $\phi^{\l}$ for the function on $M_0\times \partial \M\times \partial \M$ defined for $x\in M_0$ and  $\xi, \eta\in \partial \M$, $\xi\not=\eta$, by
\[
\phi^{\l}(x, \xi, \eta)=\frac{1}{t_0}\E_{x, \xi}^{\l}\left((\xi\big|\eta)_{{\rm x}_{t_0}^{\l}}-(\xi\big|\eta)_x\right). 
\]
Then, by (\ref{Bus-xi-zeta}),  $\phi^{\l}$ has a continuous extension to the diagonal,  still denoted $\phi^{\l}$, given by 
\[
\phi^{\l}(x, \xi, \xi)=\frac{1}{t_0}\E_{x, \xi}^{\l}\left(b_{(x, \xi)}({\rm x}_{t_0}^{\l})\right). 
\]
Write $\lfloor {\rm x}_t \rceil^{\l}=\beta_t^{\l}\underline{\rm x}_t^{\l}$, where $\beta_t^{\l}\in G$ and $\underline{\rm x}_t^{\l}\in M_0$. 
Using $\phi^{\l}$, (\ref{eq-UC-C-3}) shows that  there exist sequences $\l_j\to \l_0$, $N_j\to +\infty$, as $j\to \infty$, and points $x_j, \xi_j, \eta_j$,  such that for all $j$, 
\[
\frac{1}{N_j}\sum_{k=0}^{N_j-1}\E_{x_j, \xi_j}^{\l_j}\left(\phi(\underline{\rm x}_{kt_0}^{\l_j}, (\beta_{kt_0}^{\l_j})^{-1}\xi_{j}, (\beta_{kt_0}^{\l_j})^{-1}\eta_{j})\right)<\frac{1}{2}\underline{\ell}.
\]
This means for $\l_j, N_j, x_j, \xi_j$ and  $\eta_j$ as above, 
\begin{equation}\label{eq-UC-C-4}
\frac{1}{N_j}\sum_{k=0}^{N_j-1}{\rm Q}_{kt_0}^{2, \l_j}\phi^{\l_j}(x_j, \xi_j, \eta_j)<\frac{1}{2}\underline{\ell}. 
\end{equation}
Define a sequence of probability measures $\mu_j$ on $M_0\times \partial \M\times \partial \M$ by 
\[
\mu_j:=\frac{1}{N_j}\sum_{k=0}^{N_j-1}({\rm Q}_{kt_0}^{2, \l_j})^*\left(\delta(x_j, \xi_j, \eta_j)\right) d(\cdot, \cdot, \cdot), 
\]
where $({\rm Q}_{kt_0}^{2, \l_j})^*$ is the dual action of ${\rm Q}_{kt_0}^{2, \l_j}$  and $\delta(x_j, \xi_j, \eta_j)$ is the Dirac measure at $(x_j, \xi_j, \eta_j)$. 
Then,
\[
\big\|({\rm Q}_{t_0}^{2, \l_j})^*\mu_j-\mu_j\big\|\leq \frac{2}{N_j}. 
\]
Moreover, $({\rm Q}_{t_0}^{2, \l_j})^*$ converges to $({\rm Q}_{t_0}^{2, \l_0})^*$ in norm as $j$ goes to infinity by Theorem \ref{diff-HK-estimations-gen} since 
\begin{align*}
\big\|{\rm Q}_{t_0}^{2, \l_j}-{\rm Q}_{t_0}^{2, \l_0}\big\|\leq & \sup\limits_{x\in \M}\left|\int_{\M}p^{
\l_j}(t, x, y)\ d{\rm Vol}^{\l_j}(y)-p^{
\l_0}(t, x, y)\ d{\rm Vol}^{\l_0}(y)\right|\\
= &\sup\limits_{x\in \M}\int_{\l_0}^{\l_j}\int_{\M}\left|(\ln p^{\l})^{(1)}_{\l}(t_0, x, y)+(\ln \rho^{\l})^{(1)}_{\l}(y)\right| p^{\l}(t_0, x, y)\ d{\rm Vol}^{\l}(y)\ d\l \\
\leq & Const. \left|\l_{j}-\l_0\right|,
\end{align*}
where  ${\rho}^{\l}=d{\rm Vol}^{\l}/d{\rm Vol}^{0}$. 
Consequently,  if  $\mu$ is a weak limit of $\mu_j$, we have
\[
({\rm Q}_{t_0}^{2, \l_0})^*\mu=\mu.
\]
Let $\mu'=(1/t_0)\int_0^{t_0}({\rm Q}_s^{2, \l_0})^*\mu\ ds$. The measure $\mu'$ is ${\rm Q}_{t}^{2, \l_0}$-invariant  ($t>0$) and hence coincides with ${\bf m}^{2, \l_0}$ by Lemma \ref{L-Prop 1}.
Note that $\phi^{\l_j}$ converges to $\phi^{\l_0}$ as $j$ goes to infinity. We conclude from (\ref{eq-UC-C-4}) that $\int \phi^{\l_0}\ d\mu\leq \underline{\ell}/2$. Using (\ref{eq-UC-C-2}) again, we find that 
\[
\int \phi^{\l_0}\ d{\bf m}^{2, \l_0}\leq \frac{3}{4}\underline{\ell}. 
\]
But, by Lemma \ref{L-Prop 1}, we also have
\[
\int \phi^{\l_0}\ d{\bf m}^{2, \l_0}=\frac{1}{t_0}\int \E_{x, \xi}^{\l_0}\left(b_{x, \xi}({\rm x}^{\l_0}_{t_0})\right)\ d{\bf m}^{2, \l_0}=\lim\limits_{t\to \infty}\frac{1}{t}\int \E_{x, \xi}^{\l_0}\left(b_{x, \xi}({\rm x}_t^{\l_0})\right)\ d{\bf m}^{\l_0}\geq \underline{\ell},
\]
which is a contradiction. 
\end{proof}

\begin{proof}[Proof of Lemma \ref{UC-step 0}]For $\l\in (-1, 1)$, $x\in M_0$, $\xi, \eta\in \partial \M$ and  $t\in \Bbb R_+$, write
\[
\psi^{\l}_{{\mathtt{b}}}(x, \xi, \eta, t):=\E_{x, \xi}^{\l}\left(e^{-{\mathtt{b}}\big((\xi |\eta)_{\lfloor {\rm x}_t \rceil^{\l}}-(\xi |\eta)_x\big)}\right).
\]
For each $\l$ and ${\mathtt{b}}$, it is true by  the Markov property of the $\wt{g}^{\l}$-Brownian motion that 
\[
\sup\limits_{x, \xi, \eta}\psi^{\l}_{{\mathtt{b}}}(x, \xi, \eta, t_1+t_2)\leq \sup\limits_{x, \xi, \eta}\psi^{\l}_{{\mathtt{b}}}(x, \xi, \eta, t_1)\cdot \sup\limits_{x, \xi, \eta}\psi^{\l}_{{\mathtt{b}}}(x, \xi, \eta, t_2). 
\]
Hence  for Lemma \ref{UC-step 0}, it suffices to find,  for a fixed $T$ and ${\mathtt{b}}'$ sufficiently small,  positive numbers $C'$ and $\Bbbk'$ such that for all $\l\in (-1, 1)$ and ${\mathtt{b}}<{\mathtt{b}}'$, 
\begin{align}
 \sup\limits_{x, \xi, \eta}\sup\limits_{0\leq t<T}\psi^{\l}_{{\mathtt{b}}}(x, \xi, \eta, t)&\leq C',\label{L-CLT-a}\\
 \sup\limits_{x, \xi, \eta}\psi^{\l}_{{\mathtt{b}}}(x, \xi, \eta, T)&\leq \ \Bbbk'<1. \label{L-CLT-b}
\end{align} 
Let $T$ be as in Lemma \ref{L-Prop2}. Note that there is some constant ${C}$ such that 
\[
\big|(\xi |\eta)_{\lfloor {\rm x}_t \rceil^{\l}}-(\xi |\eta)_x\big|\leq 2d(\lfloor {\rm x}_t \rceil^{\l}, x)\leq {C} d^{\l}(\lfloor {\rm x}_t \rceil^{\l}, x).
\]
Using  Taylor's expansion of the exponential function, we obtain  
\[
e^{-{\mathtt{b}}((\xi |\eta)_{\lfloor {\rm x}_t \rceil^{\l}}-(\xi |\eta)_x)}\leq 1-{\mathtt{b}} ((\xi |\eta)_{\lfloor {\rm x}_t \rceil^{\l}}-(\xi |\eta)_x)+(C {\mathtt{b}} d^{\l}(\lfloor {\rm x}_t \rceil^{\l}, x))^2e^{C {\mathtt{b}} d^{\l}(\lfloor {\rm x}_t \rceil^{\l}, x)}. 
\]
Since  the metrics ${\wt g}^{\l}$ have negative sectional curvatures bounded uniformly away from $0$ for all  $\l$, we have the  exponential decay of  the  kernel functions, which implies  that there exists some constant $C_1$ such that for all $t$, $0\leq t\leq T$,  and all $\l$, 
\[
\E_{x, \xi}^{\l}\left(\big(Cd^{\l}(\lfloor {\rm x}_t \rceil^{\l}, x)\big)^2e^{C d^{\l}(\lfloor {\rm x}_t \rceil^{\l}, x)}\right)<C_1.
\]
So, using Lemma \ref{L-Prop2},  we obtain for ${\mathtt{b}}\leq 1$, 
\begin{eqnarray*}
\sup\limits_{0\leq t<T}\psi^{\l}_{{\mathtt{b}}}(x, \xi, \eta, t)\leq 1+{\mathtt{b}} C_1+{\mathtt{b}}^2 C_1,\\
\psi^{\l}_{{\mathtt{b}}}(x, \xi, \eta, T)\leq 1-\frac{1}{4}{\mathtt{b}}\underline{\ell}+{\mathtt{b}}^2C_1.
\end{eqnarray*}
Put ${\mathtt{b}}'=\min\{1, \underline{\ell}/(8C_1)\}$. We see that (\ref{L-CLT-a}) and (\ref{L-CLT-b}) are satisfied for  all $\l\in (-1, 1)$ and ${\mathtt{b}}<{\mathtt{b}}'$ with $C'=1+\underline{\ell}/8+\underline{\ell}^2/(64C_1)$, $\Bbbk'=1-\underline{\ell}^2/(64C_1)$. 
\end{proof}

\begin{proof}[Proof of Theorem \ref{regularity-harmonic measure}] Let $T>0$ be fixed. Assume $g\in \Re^k(M)$.  By Proposition \ref{Uni-contraction}, there exist some neighborhood $\mathcal{V}_g$ of $g$ in $\Re^k(M)$ such that for any continuous curve $\l\mapsto g^{\l}$ in $\mathcal{V}_g$, there is some positive ${\mathtt{b}}$ and $\Bbbk_0<1$ such that for all $f\in \mathcal{H}_{{\mathtt{b}}}^0, n\in \Bbb N$, 
\begin{equation}\label{RHM-1}
\left\|({\rm Q}_T^{\l})^n f-\int f\ d{\bf m}^{\l}\right\|_{{\mathtt{b}}}\leq \Bbbk_0^n\|f\|_{{\mathtt{b}}}. 
\end{equation}
(For later consideration, we choose ${\mathtt{b}}$ to be small such that $2{\mathtt{b}}$ also fulfills the requirement of Proposition \ref{Uni-contraction} and $2{\mathtt{b}}<{\mathtt{b}}'$, where ${\mathtt{b}}'$ is from Lemma \ref{UC-step 0}.)
The inequality (\ref{RHM-1}) means  each operator ${\rm Q}_T^{\l}$ is a bounded operator on $\mathcal{H}_{{\mathtt{b}}}^{0}$,  $1$ is its  isolated eigenvalue  and  ${\bf m}^{\l}$ is the eigenfunction of eigenvalue 1 of the dual operator $({\rm Q}_T^{\l})^*$.
By the classical spectrum theory on operators in Banach space (cf. \cite[Theorem 6.17]{Ka}), we can decompose $\mathcal{H}_{{\mathtt{b}}}^{0}$ into the direct sum of one-dimensional $E_{\rho}$ associated to the eigenvalue  1, and an infinite-dimensional space $E_{<1}$ on which $({\rm Q}_T^{\l})^n$ tends exponentially fast to $0$. Let $\mathcal{C}$ be any  circle around 1 with a small radius.  Then the projection of $f\in \mathcal{H}_{{\mathtt{b}}}^0$ to $E_1$ is given by
\[
\frac{1}{2i\pi}\int_{\mathcal{C}}\left(z{\rm Id}-{\rm Q}_T^{\l}\right)^{-1} f\ dz. 
\]
Using  this and (\ref{RHM-1}), we conclude that the following two functional on $\mathcal{H}_{{\mathtt{b}}}^0$ coincide:
\[
\int \cdot \ d{\bf m}^{\l}= \frac{1}{2i\pi}\int_{\mathcal{C}}\left(z{\rm Id}-{\rm Q}_T^{\l}\right)^{-1} \cdot \ dz.
\]

For the regularity of $\l\mapsto {\bf m}^{\l}$, we mean the regularity of $\l\mapsto \int\cdot\ {\bf m}^{\l}$, which is the composition of two mappings
\[
\l\mapsto {\rm Q}_T^{\l}\ \ \mbox{and}\ \ {\rm Q}_T^{\l}\mapsto  \frac{1}{2i\pi}\int_{\mathcal{C}}\left(z{\rm Id}-{\rm Q}_T^{\l}\right)^{-1} \cdot\ dz.
\]
Note that by spectral continuity results for isolated simple eigenvalues (cf. \cite[Theorem 3.11]{Ka}), for $L\in( \mathcal{H}_{{\mathtt{b}}}^{0})^*$ in a small neighborhood of  $Q_T^0$,  the mapping
\[L\mapsto  \frac{1}{2i\pi}\int_{\mathcal{C}}\left(z{\rm Id}-L\right)^{-1} \cdot\ dz\]
is analytic. We may assume $\mathcal{V}_g$ is such that all ${\rm Q}_T^{\l}$ belong to this neighborhood. Then for the regularity of $\l\mapsto \int\cdot\ {\bf m}^{\l}$, it remains to show the regularity of the mapping $\l\mapsto {\rm Q}_T^{\l}$. 

For $f\in \mathcal{H}_{{\mathtt{b}}}^{0}$, let $\wt{f}$ be its  $G$-invariant extension to $\M\times \partial\M$.  Then 
\[
{\rm Q}_T^{\l}f(x, \xi)=\int_{\wt{M}}\wt{f}(y, \xi)p^{\l}(T, x, y)\ d{\rm Vol}^{\l}(y).
\]
Put ${\rho}^{\l}:=d{\rm Vol}^{\l}/d{\rm Vol}^{0}$. Then $\l\mapsto {\rho}^{\l}$  is $C^k$ in $\l$ in $C^k(\M)$. By Theorem \ref{diff-HK-estimations-gen} i) and iii), for every  $i$, $\ 1\leq i\leq k-2,$  and every $(x, \xi)\in \M\times \partial \M$, the following  differential exists:
\[
\big({\rm Q}_T^{\l}f(x, \xi)\big)^{(i)}_{\l}=\sum_{j=0}^{i}\left(\begin{matrix}i\\ j\end{matrix}\right)\int \wt{f}(y, \xi)(p^{\l})^{(j)}_{\l}(T, x, y)({\rho}^{\l})^{(i-j)}_{\l}(y)\ d{\rm Vol}^{0}(y). 
\]
To conclude this defines the $i$-th  differential of ${\rm Q}_T^{\l}$ in $\l$ in $(\mathcal{H}_{{\mathtt{b}}}^0)^*$, we only need to show it defines a bounded operator from $\mathcal{H}_{{\mathtt{b}}}^0$ into itself. For $\mathcal{V}_g$ small,  the norms of the differentials $({\ln \rho}^{\l})^{(i)}_{\l}$, $i=1, \cdots, k-2$, and hence the norms of $({\rho}^{\l})^{(i)}_{\l}/\rho^{\l}$, $i=1, \cdots, k-2$, are all bounded. So   it suffices to consider  $S^i_{\l}$, where 
\[
\big(S^i_{\l}f\big)(x, \xi):=\int_{y\in \M}\wt{f}(y, \xi) (p^{\l})^{(i)}_{\l}(T, x, y)\ d{\rm Vol}^{\l}(y),
\]
and show it is a bounded functional of $\mathcal{H}_{{\mathtt{b}}}^0$. For each $\xi\in \M$, $\wt{f}(\cdot, \xi)$ is uniformly continuous in $x$ and bounded. Hence Theorem \ref{diff-HK-estimations-gen} iv) applies and shows that $\big(S^i_{\l}f\big)(x, \xi)$ is continuous in $x$.  Using Theorem \ref{diff-HK-estimations-gen} iii), we continue to compute that 
\[
\big|\big(S^i_{\l}f\big)(x, \xi)\big|\leq \|f\|_{\infty}\cdot\int \frac{(p^{\l})^{(i)}_{\l}(T, x, y)}{p^\l(T, x, y)}p^\l(T, x, y)\ d {\rm Vol}^\l(y)\leq {c}_{\l, (i)}(2) \|f\|_{{\mathtt{b}}},  
\]
where ${c}_{\l, (i)}(2)$ is as in (\ref{p-lam-i-p-equ}). For the  H\"{o}lder continuity  of $\xi\mapsto \big(S^i_{\l}f\big)(x, \xi)$ and the corresponding H\"{o}lder  norm estimation, it suffices to show the latter is bounded.   By H\"{o}lder's inequality, Theorem \ref{diff-HK-estimations-gen} iii) and  Lemma \ref{UC-step 0}, we obtain   \begin{align*}
&\big|\big(S_{\l}^{i}f)(x, \xi_1)-\big(S_{\l}^{i}f)(x, \xi_2)\big|e^{{\mathtt{b}}(\xi_1\big|\xi_2)_x}\\
&\leq \left(\int_{\M}\big|\wt{f}(y, \xi_1)-\wt{f}(y, \xi_2)\big|\cdot \big|(p^{\l})^{(i)}_{\l}(T, x, y)\big|\ d{\rm Vol}^{\l}(y)\right) e^{{\mathtt{b}}(\xi_1\big|\xi_2)_x}\\
&\leq  \|f\|_{{\mathtt{b}}}\int_{\M}e^{-{\mathtt{b}}\left((\xi_1\big|\xi_2)_y-(\xi_1\big|\xi_2)_x\right)}\left|\frac{(p^{\l})^{(i)}_{\l}(T, x, y)}{p^{\l}(T, x, y)}\right| p^{\l}(T, x, y)\ d{\rm Vol}^{\l}(y)\\
&= \|f\|_{{\mathtt{b}}}\cdot\left( \E_{x, \xi_1}^{\l}\left(e^{-2{\mathtt{b}}((\xi_1 |\xi_2)_{\lfloor {\rm x}_T \rceil^{\l}}-(\xi_1 |\xi_2)_x)}\right)\right)^{\frac{1}{2}}\left\|\frac{(p^{\l})^{(i)}_{\l}(T, x, y)}{p^{\l}(T, x, y)}\right\|_{L^2}\\
&\leq {c}_{\l, (i)}(2)(\Bbbk_1^T)^{\frac{1}{2}}\|f\|_{{\mathtt{b}}}. 
\end{align*}
Altogether, we have  that each $S^i_{\l}$ maps $\mathcal{H}_{\mathtt{b}}^0$ into itself and is a bounded operator since  
\begin{align*}
\|S^i_{\l}f\|_{{\mathtt{b}}}= &\; \sup\limits_{x, \xi}\big|\big(S^i_{\l}f\big)(x, \xi)\big|+\sup\limits_{x, \xi_1,
\xi_2} \big|\big(S_{\l}^{i}f)(x, \xi_1)-\big(S_{\l}^{i}f)(x, \xi_2)\big|e^{{\mathtt{b}}(\xi_1\big|\xi_2)_x}\\
\leq &\; {c}_{\l, (i)}(2)\big(1+(\Bbbk_1^T)^{\frac{1}{2}}\big)\|f\|_{{\mathtt{b}}}. 
\end{align*}
\end{proof}

\subsection{Differentials of the linear drift}\label{sec-3.1}We are in a situation to prove Theorem \ref{main}. 

\begin{proof}[Proof of Theorem \ref{main}] It suffices to show the first statement. 

Let $\mathcal{V}_g$ be such that Proposition \ref{regularity-Div} and Theorem \ref{regularity-harmonic measure} hold true. We may also assume the H\"{o}lder exponents $\alpha$ of Proposition \ref{regularity-Div}  and ${\mathtt{b}}$ of Theorem \ref{regularity-harmonic measure} coincide.  As before, for  any $C^{k}$ curve  $\lambda\in (-1,
1)\mapsto g^{\lambda}\in \mathcal{V}_{g}$ with $g^0=g$, we write ${\ov{X}}^{\l}$ for the $\wt{g}^{\l}$-geodesic spray, ${\rm
Div}^{\l}$ for the divergence operator associated with the $\wt{g}^{\l}$-stable foliation and  ${\bf m}^{\l}$ for the $g^{\l}$-harmonic measure on $SM$.   Let $\ell_{\l}$ be the linear drift of $g^{\l}$.  By (\ref{equ-1.1}), 
\begin{equation}\label{equ-1.1-lam} \ell_{\l}=-\int_{\scriptscriptstyle{M_0\times \pp\M}} \big({\rm
Div}^{\l}{\ov{X}^{\l}}\big)(x, \xi)\
  d{\bf m}^{\l}=-{L_{\l}}\big({\rm
Div}^{\l}{\ov{X}^{\l}}\big). 
\end{equation}

By Proposition \ref{regularity-Div} iii),  $\lambda\mapsto {\rm
  Div}^{\lambda}\ov{X}^{\lambda}$ is $C^{k-3}$ into $C^{{\mathtt b}}(\M\times \partial\M, \Bbb
  R^+)$ and is $C^{k-2}$ into $C^{0}(\M\times \partial\M, \Bbb
  R^+)$. Write $({\rm
  Div}^{\lambda} \ov{X}^{\lambda})^{(0)}_{\l}={\rm
  Div}^{\lambda} \ov{X}^{\lambda}$ and  $({\rm
  Div}^{\lambda} \ov{X}^{\lambda})^{(i)}_{\l}$, $i=1, \cdots, k-2$, for its $i$-th derivative in $\l$. Then   $({\rm
  Div}^{\lambda} \ov{X}^{\lambda})^{(i)}_{\l}$ belongs to  $C^{0}(\M\times \partial\M, \Bbb
  R^+)$ for $i\leq k-2$,  and belongs to $C^{{\mathtt{b}}}(\M\times \partial\M, \Bbb
  R^+)$ for $i\leq k-3$.  Regard each ${\bf m}^{\l}$ as a measure on $M_0\times \partial\M$.   The operator ${L_{\l}}:=\int_{M_0\times \partial \M}\cdot \ d{\bf m}^{\l}$ is an bounded operator on continuous functions on $M_0\times \partial\M$. Moreover,  by Theorem \ref{regularity-harmonic measure}, $\l\mapsto {L_{\l}}$ is $C^{k-2}$ differentiable as elements of $(\mathcal{H}_{{\mathtt{b}}}^0)^*$. Using these regularities and (\ref{equ-1.1-lam}),   we  conclude  that  the function $\lambda\mapsto \ell_{\lambda}$ is $C^{k-2}$ differentiable. Denote by ${L_{\l}^{(i)}}$, $i=1, \cdots, k-2$, the $i$-th differential functional of $L_{\l}$. 
Then, for every  $i$, $1\leq i\leq k-2$, the $i$-th differential of $\ell_{\l}$ in $\l$, i.e., $\ell_{\l}^{(i)}$,  is given  by 
\begin{equation}\label{l-derivative-j}
\ell_{\l}^{(i)}=-\sum_{j=0}^{i}\left(\begin{matrix}i\\ j\end{matrix}\right) {L^{(j)}_{\l}}\left(({\rm
  Div}^{\lambda} \ov{X}^{\lambda})^{(i-j)}_{\l}\right). 
\end{equation}
\end{proof}

Specifying  (\ref{l-derivative-j}) for $i= 1$, we write:
\begin{cor}\label{prop-thm1.4}Let $g\in \Re^3(M)$.  For any $C^3$ curve $\lambda\in (-1, 1)\mapsto g^{\lambda}\in \Re^3(M)$ with $g^0=g$ and constant volume, we have
\begin{equation}\label{l-derivative}
(\ell_{\l})_0'=-\int {\rm{Div}}^{0}\overline{X}^{0} d({\bf m}^{\l})'_0-\int \big({\rm{Div}}^{\l}\overline{X}^{\l}\big)'_0\ d{\bf m}^0.
\end{equation}
In particular, if  $g=g^0$ is a locally symmetric metric and the volume ${\rm{Vol}}^\l (M) $ is constant in $\l$, then we have  $(\ell_{\l})_0'=0$.
\end{cor}
\begin{proof} We apply  (\ref{l-derivative-j}) for $i= 1$ and $\l = 0$. The operator $L_0$ extends to the harmonic measure ${\bf m}^0$ and $L_0^{(1)}$ is a linear functional on the space  $\mathcal{H}_{{\mathtt{b}}}^0$ that we denote $({\bf m}^{\l})'_0.$ Formula (\ref{l-derivative}) follows.

Let $\upsilon^\l $ be the volume entropy of $(\M, \wt g^\l),$ \[\upsilon^\l := \lim\limits _{R \to \infty } \frac{1}{R} \ln {\rm {Vol}}^\l B(x,R), \] where $B(x, R) $ is the ball of radius $R$ about $x$ in $\M$. We know by \cite{K1} that for all $\l $, $\ell ^\l \leq \upsilon^\l $ and by \cite{BCG} that the volume entropy of a negatively curved locally symmetric space achieves its minimum over all metrics of  the same volume on that space. Since $\l \mapsto \ell _\l $  and $\l \mapsto \upsilon^\l $ are differentiable at $0$ (by  Theorem \ref{main} and by \cite{KKPW}), the derivative has to be $0.$
\end{proof}

We develop formula (\ref{l-derivative}). The vector $(\overline{X}^{\l})'_0$ is a  vertical vector given by \cite[Proposition 4.5]{LS2}. For ${\bf {v}} = (x, \xi)$,  it is the sum of $\left(\|\overline {X}^{\l} \|_{\wt{g}^0}\right)'_0 ({\bf {v}} ) {\overline{X}^0({\bf v})} $ and of a vector $Y({\bf v})$ orthogonal to ${\bf {v}} ,$ where $Y$ is a $C^1$ vector field along the stable manifolds. Let $u_0$ be the function  such that 
\[
\Delta u_0=-{\rm{Div}}\ov{X}-\ell, \ ({\mbox{see \cite[(5.12)]{LS2}}}). 
\]

\begin{theo}\label{Equiv-Thm1.5} Let  $M$ be a closed connected smooth manifold and let $g\in \Re^3(M)$.   For any $C^3$ curve $\lambda\in (-1, 1)\mapsto g^{\lambda}\in \Re^3(M)$ with $g^0=g$ and constant volume,
\begin{align}
(\ell_{\l})'_0 \;=\; &\int \left(-\frac{1}{2}\langle\nabla {\rm{trace}}\XX, \overline{X}\rangle + \frac{1}{2} \XX (\overline {X}, \overline {X} ) {\rm {Div}} (\ov{X})  +\frac{1}{2}  \<\nabla (\XX (\overline {X}, \overline {X})), \ov{X} \>-{\rm{Div}}Y \right)\ d{\bf m}\notag\\
&+\int \left(-\frac{1}{2}\langle \nabla {\rm{trace}}\XX, \nabla u_0 \rangle +{\rm{Div}}\big(\XX (\nabla u_0)\big)\right)\ d{\bf m}, \label{Great}
\end{align}
where we omit  the index $0$ for $\nabla^0, \ov{X}^0, \langle \cdot, \cdot\rangle_0, {\rm{Div}}^0$ and ${\bf m}^0$ at $g^0$,    and  where $\XX(\cdot)$ is considered as the $(1, 1)$-form in $\M$ such that $\langle \XX(Z), Z'\rangle=\XX(Z, Z')$.     In particular, there is a linear functional $\LL$ on $C^k(S^2T^*)$ such that $(\ell_{\l})_0' \; = \; \LL (\XX)$.
\end{theo}
\begin{proof}To obtain (\ref{Great}), 
we use the decomposition of $(\ell_{\l})_0'$ given by (\ref{l-derivative}) as above:
\[ (\ell_{\l})_0' \; = \; -\int ({\rm{Div}}^{\l}\overline{X})'_0\ d{\bf m} -\int ({\rm{Div}}\overline{X}^{\l})'_0\ d{\bf m}-\int {\rm{Div}}\overline{X}\ d({\bf m}^{\l})'_0 \]
and study the three terms successively. 

Firstly, we have $({\rm{Div}}^{\l}\overline{X})'_0 =\frac{1}{2} \< {\nabla(\rm {Trace}}\XX), \overline{X}\>$. 

Then, for ${\bf {v}} = (x, \xi)$,   $(\overline{X}^{\l})'_0$ is the sum of $\left(\|\overline {X}^{\l} \|\right)'_0 ({\bf {v}} ) \overline{X}({\bf {v}}) $ and  $Y({\bf v})$. Hence, 
\[
({\rm{Div}}\overline{X}^{\l})'_0={\rm{Div}}Y+  {\rm{Div}} \left(\left(\|\overline {X}^{\l} \|\right)'_0 ({\bf {v}} ) \overline{X}({\bf {v}})\right). 
\]
Since $\|\overline {X}^{\l} \|_{g^\l}^2 = 1$, we have 
\[ \left(\|\overline {X}^{\l} \|\right)'_0 ({\bf {v}} ) \; = \; -\frac{1}{2} \XX (\overline {X} ({\bf v}), \overline {X} ({\bf v}) ).\]
Thus, 
\[ {\rm{Div}} \left(\left(\|\overline {X}^{\l} \|\right)'_0 ({\bf {v}} ) \overline{X}({\bf {v}})\right)  =  -\frac{1}{2} \XX (\overline {X} ({\bf v}), \overline {X} ({\bf v}) ) {\rm {Div}} (\overline{X}({\bf {v}}))  -\frac{1}{2}  \<\nabla (\XX (\overline {X} ({\bf v}), \overline {X} ({\bf v}) )), \overline{X}({\bf {v}}) \> .\]

Lastly, we discuss the term $\int {\rm{Div}}\overline{X}\ d({\bf m}^{\l})'_0$.  Recall that, by Theorem \ref{regularity-harmonic measure}, $\l \mapsto {\bf m}^\l $ is differentiable at $0$, with derivative $({\bf m}^\l)'_0 \in (\mathcal{H}^{0}_{{\mathtt{b}}})^\ast$ (denoted as an integral). It follows that, for $f$ smooth on $SM$, 
\begin{equation}\label{har-diff-eq}  \int (\D^\l)'_0 f \  d{\bf m}  \; +\; \int \D f  \  d({\bf m}^\l)'_0 \; = \; 0.
\end{equation}
The  equation (\ref{har-diff-eq}) extends to functions $f$ that are of  class $C^2$ along the stable leaves with globally continuous second order derivatives. In particular, (\ref{har-diff-eq}) applies to the function $u_0$  and therefore, 
\[
\int {\rm{Div}}\overline{X}\ d({\bf m}^{\l})'_0= \int (\D^\l)'_0 u_0 \  d{\bf m}=\int \left(\frac{1}{2}\langle \nabla u_0, \nabla {\rm{trace}}\XX \rangle -{\rm{Div}}\big(\XX (\nabla u_0)\big)\right)\ d{\bf m}.
\]

To show $(\ell_{\l})_0'$ is linear in $\XX$, it remains to consider $\int {\rm {Div}} Y\, d{\bf m}$.  If we denote $k(x,y,\xi) $ the continuous version of the density $\big({dm_y}/{dm_x}\big) (\xi) $ (see e.g. \cite[Proposition 2.2]{LS2}), the integration by parts formula yields
\begin{equation}\label{ln-Martin} \int {\rm {Div}} Y\ d{\bf m} \; = \; -\int \left\langle Y, \nabla_y \ln k(x,y, \xi )|_{y=x} \right\rangle \ d{\bf {m}}.
\end{equation}
We recall from \cite[Proposition 4.5]{LS2}  the construction of the vector field $Y$.  Let ${\bf v} \in TM$. We define the vector $\Upsilon ({\bf v}) \in TTM$ as the vertical vector with vertical component given by
\[ \Upsilon ({\bf v}) \; := \;\big( \nabla ^\l_{{\bf v}}({\bf v})\big)'_0 - \big\langle \big( \nabla ^\l_{{\bf v}}({\bf v})\big)'_0, {\bf v} \big\rangle. \]
Clearly, for all ${\bf v } \in SM$, $\Upsilon ({\bf v}) $ depends linearly on $\XX$, and $\sup _{{\bf v}} \| \Upsilon ({\bf v}) \|$ is bounded by $C \| \XX \|_{C^1}$.
In order to obtain $Y({\bf v}) $, we consider the orbit ${\bf \Phi} _s ({\bf v}), s \geq 0$, under the geodesic flow. For each $s \geq 0,$ we decompose $\Upsilon ({\bf \Phi} _s ({\bf v})) $ into a sum of its unstable part $\Upsilon ({\bf \Phi}_s ({\bf v}))^{\rm{u}}$ and its stable part. The vector $Y({\bf v}) $ is the vertical part of \[ \int _0 ^\infty (D{\bf \Phi}_s)^{-1} \Upsilon ({\bf \Phi}_s ({\bf v}))^{\rm{u}} \, ds.\]
Since the geodesic flow is Anosov, there are $C, \tau >0 $ such that $(D{\bf \Phi}_s)^{-1} $ restricted to the unstable manifold has norm smaller than $C e^{-\tau s}.$ It follows that the expression $ \int {\rm {Div}} Y\, d{\bf m}$ is linear in $\XX $ and bounded by $C \| \XX \|_{C^1}.$
\end{proof}

 \begin{remark}\label{cri-remark}We can also verify that the formula (\ref{Great}) gives indeed $0$ in the case when $ g = g^0$ is locally symmetric.  
 
Assume that $g$ is a locally symmetric metric, then ${\rm{Div}}\overline{X}$ is the constant $ -\ell$ and the measure ${\bf m}$ is the normalized Liouville measure. Since the measures ${\bf m}^\l$ are normalized (and the constant functions belong to the space $\mathcal{H}_{{\mathtt{b}}}^0$), $\int {\rm{Div}}\overline{X}\ d({\bf m}^{\l})'_0 = 0 $ and formula (\ref{Great}) reduces to 
\begin{align*}(\ell_{\l})'_0\; =\; &\int \left(-\frac{1}{2}\langle\nabla {\rm{trace}}\XX, \overline{X}\rangle + \frac{1}{2} \XX (\overline {X}, \overline {X} ) {\rm {Div}} (\ov{X})  +\frac{1}{2}  \<\nabla (\XX (\overline {X}, \overline {X})), \ov{X} \>-{\rm{Div}}Y \right)\ d{\bf m}\\
\; =:\; &\int \left( {\rm{(I)}}+ {\rm{(II)}} + {\rm{(III)}}+ {\rm{(IV)}}\right)\ d{\bf m}.
\end{align*}
Since ${\rm{trace}}\XX, \XX (\overline {X}, \overline {X})$ are functions on $SM$ and we integrate with respect to the invariant Liouville measure, the integrals of $(\rm{I}), ({\rm{III}})$ vanish. Since $\wt g $ is a symmetric space, the $k(x, y, \xi)$ in formula (\ref{ln-Martin}) is given by  $-\ell _0 b_{(x, \xi )} (y) ,$ where $ b_{(x, \xi )} $ is the Busemann function (see Section \ref{Sec-Busemann}). It follows that $ \nabla_y \ln k(x,y, \xi )|_{y=x} = \ell\overline{X}({\bf {v}}).$ Since $Y({\bf v}) $ is orthogonal to $\overline{X}({\bf {v}})$, the integral $\int {\rm{(IV)}}\ d{\bf m}$ vanishes as well. Remains to consider 
\[\int {\rm{(II)}}\ d{\bf m}=-\frac{1}{2}\ell \int \XX (\overline {X}, \overline {X} ) \ d{\bf m}=-\frac{1}{2}\frac{\ell }{m} \int _{M_0}{\rm{trace}\XX}\ \frac{d{\rm{Vol}}_{\wt{g}}}{{\rm{Vol}}_{\wt{g}}(M_0)}=-\frac{\ell }{m}\frac{\big({\rm{Vol}}_{\wt{g}^{\l}}(M_0)\big)'_0}{{\rm{Vol}}_{\wt{g}}(M_0)},\]
where ${{\rm{Vol}}_{\wt{g}}}$ is the Riemannian volume. So, $\int {\rm{(II)}}\ d{\bf m}$ vanishes since the volume is constant.
\end{remark}

\section{Brownian motion and stochastic flows}\label{Sec-BM-SF}

In this section, we recall the Eells-Elworthy-Malliavin construction of the Brownian motion on a manifold through a stochastic  differential equation (SDE) on the orthogonal frame bundle and of the associated stochastic flow (see Proposition \ref{est-norm-u-t-j}). We give estimations on the growth in time of the derivatives of this stochastic flow. We will need in Sections 5 and 6  both uniform estimations and estimations in average with respect to the Brownian motion and Brownian bridge distributions in the non-compact case.

\subsection{Parallelism  and the Brownian motion}\label{BMM}
Let ${\rm N}$ be a $C^{\infty}$ $n$-dimensional Riemannian manifold.
A differential form $\vartheta$ on $\rm N$ with values in $\Bbb R^n$
is called a \emph{parallelism differential form} (\cite{M}), if it realizes for
every $u\in \rm N$ an isomorphism of $ T_{u}{\rm N}$ on $\Bbb R^n$.
A parallelism differential form $\vartheta$ is called $C^{k}$ if it
is  a $C^{k}$ section of the frame bundle space $\mathcal{F}(\rm N)$
of $\rm N$.

Let $f: [0, +\infty)\to \Bbb R^n$ be a $C^2$ curve. It defines a
one parameter family  of continuous vectors $\{(df/dt)|_{t=\tau}\}_{\tau\in
[0, +\infty)}$. Let $\vartheta$ be a $C^1$ parallelism differential
form. It, together with $f$, defines a $C^1$ vector field on ${\rm
N}\times \Bbb R^+$: \[
Z_{t, u}^f:=\vartheta_{u}^{-1}(\frac{df}{d t}), \ \forall  {u}\in {\rm N},\  t\in \Bbb R^+.
\]
By the classical theory of ordinary differential equation,  there
exists a flow $F_{f, t}$ generated by $Z_{t, {u}}^f$,
which solves Cauchy's problem
\[
\frac{d}{dt}\left(F_{f, t}({u_0})\right)=Z_{t, {u}(t)}^f,\ {\mbox{where }}\ {u}(t)=F_{f, t}({u_0})\ \mbox{and}\ F_{f, 0}({u_0})={u_0}\in {\rm N}.
\]
The orbit of each $u_0\in \rm N$ under  $F_{f,
t}$ is an analogue of the curve $f$  since the velocity at time $\tau$ is just the preimage of
$(df/dt)|_{t=\tau}$ by $\vartheta$. 
Moreover, the time $t$ map $F_{f, t}$ depends  $C^{1}$ on
the initial point $u_0$. The
variation of $F_{f, t}({u_0})$ with respect to ${u_0}$ reflects the geometric difference
between $\rm N$ and $\Bbb R^n$ and the pull back of the tangent map of $F_{f, t}$  in $\Bbb R^n$ via $\vartheta$ can be formulated using the equation of $d\vartheta$ (\cite[Proposition 3.2]{M}). In general, if $f$ is a $C^{k+1}$ curve
in $\Bbb R^n$ and $\vartheta$ is $C^{k}$, then the flow generated by
$Z_{t, {u}}^f$ depends  $C^{k}$ on the initial point.

In case $\rm N$ is the frame bundle space of $\M$, there are plenty of parallelism differential forms  using the dual form and the connection forms.    Recall that a  \emph{frame} $u$ for $T_x \M$, $x\in \M$,  is an ordered
basis $u=(u_1,\cdots, u_m)$ for $T_x \M$, which defines a linear
isomorphism form $\Bbb R^m$ to $T_x \M$ by letting $
u({y}):=\sum_{i=1}^{m}{y}^i u_i,\  \mbox{for}\  {y}=({y}^i)\in \Bbb
R^m$. 
The set of all frames $u$ for all tangent spaces $T_x \M$, denoted by
$\mathcal{F}(\M)$, is a $C^{\infty}$ manifold. The \emph{dual form} (or the canonical form) on $\mathcal{F}(\M)$ is
an $\Bbb R^m$-valued 1-form defined by
\[
\theta_{u}(Y):=u^{-1}(\pi_{*}Y), \ \forall  Y\in
T_u\mathcal{F}(\M),
\]
where $\pi_{*}$ is the tangent map of the natural  projection map from $\mathcal{F}(\M)$ to $\M$.
The  kernel of $\pi_{*}$ is  the \emph{vertical vector bundle} of $T\mathcal{F}(\M)$:
\[
VT\mathcal{F}(\M):=\big\{Y\in T\mathcal{F}(\M):\ \pi_{*}Y=0\big\}.
\]
For $A\in \mathfrak{gl}(m, \Bbb R)$,  let $A^*$ be the vector field on $\mathcal{F}(\M)$ with $A^*(u)=\dot{\gamma}_{u}(0)$, where $\gamma_u(t)=R_{\exp(tA)}u$ and $R_{a}$ denotes the right action by $a$.  A $C^{k}$ \emph{(Ehresmann) affine connection} $\varpi$ for $(\mathcal{F}(\M), \pi,\cdot)$ is a $C^{k}$ $\mathfrak{gl}(m, \Bbb R)$-valued $1$-form on $\mathcal{F}(\M)$ satisfying
\begin{align*}
\notag\varpi(A^*(u))=&\ A,\ \forall A\in \mathfrak{gl}(m, \Bbb R),\\
\varpi((R_{a})_*Y)=& \ Ad(a^{-1})\varpi(Y),\ \forall a\in GL(m, \Bbb R), \ Y\in T\mathcal{F}(\M).
\end{align*}
Each $C^k$ affine
connection form $\varpi$ of $\mathcal{F}(\M)$ assigns a unique
$C^k$-distributed complementary horizontal vector bundle
$HT\mathcal{F}(\M)$, the kernel of $\varpi$,  which is invariant under the right action of
$GL(m, \Bbb R)$.  Each  $\varpi$ induces the notion of  covariant derivative $\nabla, D$ on vector fields and forms on $\mathcal{F}(\M)$, respectively. Let $T, R$ be the corresponding torsion tensor and curvature tensor,  and ${\bf \Theta}:=D\theta$, ${\bf \Omega}:=D\varpi$ be the torsion form and curvature form. 
Then \begin{align*}
T(X, Y)=&\ u({\bf \Theta}(\breve{X},
\breve{Y})),\notag\\
R(X, Y)Z=&\ u({\bf \Omega}(\breve{X}, \breve{Y})\cdot
(u^{-1}Z)),
\end{align*}
where $\breve{X}, \breve{Y}, \breve{Z}\in T_u\mathcal{F}(\M)$ are any vectors which project to  
$X, Y, Z\in T_x\M$, respectively, and $u\in \mathcal{F}_{x}(M)$ can be chosen  arbitrarily. Any  pair $(\theta, \varpi)$  is  a parallelism
differential form for $\mathcal{F}(\M)$.
It satisfies the following structure equations (cf.
\cite[p. 327]{Sp}):
\begin{align}
\label{structure-11}
d\theta(Y_1, Y_2)=& -\{\varpi(Y_1)\cdot
\theta(Y_2)-\varpi(Y_2)\cdot\theta(Y_1)\}+{\bf \Theta}(Y_1, Y_2),\\
d\varpi(Y_1, Y_2)=& -\left[\varpi(Y_1), \varpi(Y_2)\right]+{\bf \Omega}(Y_1,
Y_2),\label{structure-22}
\end{align}
where $Y_1, Y_2\in T_u\mathcal{F}(\M)$ and $\varpi(Y_1)\cdot \theta(Y_2)$ is the action of the matrix
$\varpi(Y_1)$ on $\theta(Y_2)\in \Bbb R^m$.

 For  $g\in \mathcal{M}^k(M)$,  let   ${\mathcal{O}}^{\wt{g}}(\M)\subset \mathcal{F}(\M)$ be the collection of  $\wt{g}$-orthogonal frames, the so-called \emph{orthogonal frame bundle space} of $(\M, \wt{g})$.  Each $u\in\mathcal{O}^{\wt{g}}_x(\M)$ defines an isometry from $\Bbb R^m$ with the classical Euclidean metric to $(T_x \M, \wt{g})$.  Let $\varpi$ be the  unique  torsion free connection
form on $\mathcal{F}(\M)$ which induces the $\wt{g}$-connection $\nabla$ and curvature tensor $R$. Then  $\varpi=(\varpi_{j}^{i}), {\bf \Omega}=(\Omega_{j}^{i})$ satisfy
\[
\varpi_{j}^{i}=\sum_{k}{{\bf \Gamma}}_{kj}^{i}\theta^k, \
{\bf \Omega}_{j}^{i}=\frac{1}{2}\sum_{k, l}{{\bf
R}}_{jkl}^{i}\theta^k\wedge \theta^l,
\]
where ${\bf \Gamma}$ and ${\bf R}$ are $\nabla$ and $R$ read in the
frame $u$. The structural equations (\ref{structure-11}) and (\ref{structure-22}) of $(\theta, \varpi)$ are reduced to 
\begin{align}
  d\theta^i(Y_1, Y_2)&=-\left(\varpi_{j}^{i}(Y_1)\theta^j(Y_2)-\varpi_{j}^{i}(Y_2)\theta^j(Y_1)\right),   \label{structure-1}\\
  d\varpi_{j}^{i}(Y_1,
  Y_2)&=-\left(\varpi_{q}^{i}(Y_1)\varpi_{j}^{q}(Y_2)-\varpi_{q}^{i}(Y_2)\varpi_{j}^{q}(Y_1)\right)+{\bf
  R}_{jkl}^{i}\theta^k(Y_1)\theta^{l}(Y_2), \label{structure-2}
\end{align}
where  $Y_1, Y_2\in T_u\mathcal{F}(\M)$ and $u\in \mathcal{F}(\M)$.  The restriction of  $(\theta,
\varpi)$ to  ${\mathcal{O}}^{\wt{g}}(\M)$ also defines  a parallelism
differential form.  For instance, we can use this parallelism to  recover  the geodesic flow on $S\M$.   Let $f: [0, +\infty)\to\mathcal{O}(\Bbb R^m)$ be a half line 
with $df/dt\equiv(\vec{e}, 0)$ for some unit vector ${e}\in
\Bbb R^m$.  It defines a $C^{k-1}$ vector
field on $\mathcal{O}^{\wt{g}}(\M)\times \Bbb R^+$ by letting \[
Z_{t,u}^f:=(\theta, \varpi)_{u}^{-1}(\frac{df}{dt}), \ \forall  u\in \mathcal{O}^{\wt{g}}(\M), \] where each $Z_{t,u}^f$ is 
just the lift of $u{e}$ to $HT\mathcal{F}(\M)$. Let $F_{e, t}$ denote the flow
generated by $Z_{t,u}^f$ with $df/dt\equiv({e}, 0)$. It  projects to  the $\wt{g}$-geodesic flow on $S\M$ and the 
orbit of $u\in \mathcal{O}^{\wt{g}}(\M)$ under it  is the parallel transportation of $u$ along the
unit speed geodesic $\gamma_{u{e}}$.

The key point of the Eells-Elworthy-Malliavin construction  of the Brownian motion on a Riemannian manifold is to realize it as a transportation of the  $\Bbb R^m$-Brownian motion using the parallelism differential form of the orthogonal frame bundle.

Let $\Theta_+$ be the space of continuous paths ${w}:[0,
+\infty)\to \Bbb R^m$, equipped with the smallest $\sigma$-algebra
$\mathcal{F}$ for which the projections $R_t:{ w}\mapsto {
w}(t)$ are measurable. The sub $\sigma$-algebras
$\{\mathcal{F}_t\}_{t\in \Bbb R^+}$ of $\mathcal{F}$ is an
increasing sequence  such that  $\{R_s\}_{s\leq t}$ are measurable
in  $\mathcal{F}_t$. An  $\Bbb R^m$-Brownian motion is
a continuous time random process  $\{B_t:\ B_t({w})={
w}(t)\}_{t\in \Bbb R^+}$ on $\Theta_+$ with distribution ${\rm Q}$ so
that the induced actions ${\rm Q}_t: 
({\rm Q}_t\varphi)(x)=\E_x(\varphi(B_t({ w})))$ on smooth functions $\varphi$ form a semigroup with Euclidean Laplacian $\Delta_{\rm
Eu}$ as being the infinitesimal generator ($\lim_{t\to 0}({\rm Q}_t
\varphi-\varphi)/t=\Delta_{\rm Eu}\varphi$ whenever $\varphi\in
C_c^2(\Bbb R^m)$, the collection of $C^2$ functions on $\Bbb R^m$ with compact support). In other words,
\begin{equation}\label{Euclidean BM}
B_t=(B_t^1, \cdots, B_t^m),
\end{equation}
where all $B_t^i$ are independent 1-dimensional Brownian motions on
$\Bbb R$ with  time $t$  transition probability ${(4\pi
t)}^{-\frac{1}{2}}e^{-\frac{(x_i-y_i)^2}{4t}}$ between points $x_i$ and $y_i$
in $\Bbb R$.  In the language of Stratonovich stochastic differential
equation (SDE), (\ref{Euclidean BM}) is
\begin{equation*}
dB_t=\sum_{i=1}^{m}e_i(B_t)\circ dB_t^i,
\end{equation*}
where $\{e_i=\partial/\partial x_i\}$ is an orthogonal chart
of $\Bbb R^m$, which means for all $\varphi\in C_c^{\infty}(\Bbb R^m)$, the collection of $C^\infty$ functions on $\Bbb R^m$ with compact support,  and for all $t\in \Bbb R^+$,
\[\varphi(B_t)=\varphi(B_0)+\int_{0}^t\sum_{i=1}^{m}e_i\varphi(B_s)\circ dB_s^i. \]
Fix a $C^{\infty}$ function $\rm c$, with support contained in the
unit interval $[0, 1]$ with integral $1$.  For each $\epsilon>0$, let
${\rm c}_{\epsilon}(\tau):=\epsilon^{-1}{\rm c}(\epsilon^{-1}\tau)$
be an  approximate  unit function. For any sample path
$t\mapsto {{w}}(t)=({{w}}^1(t),\cdots, {{w}}^m(t))$ of  $B$,  we can smooth it using 
$c_{\epsilon}$ by letting
\[
{{w}}^i_{\epsilon}(t):=\int_{0}^{\epsilon}{{w}}^i(t+s){\rm
c}_{\epsilon}(s)\ ds,\ \  i=1, \cdots, m.
\]
Let ${{w}}_{\epsilon}(t)=({{w}}_{\epsilon}^{1}(t), \cdots, {
w}_{\epsilon}^{m}(t))$. We see that $t\mapsto {{w}}_{\epsilon}(t)$
is smooth and satisfies
\[
\lim_{\epsilon\to 0}\sup_{t\in \Bbb R^+}\big\|{{w}}_{\epsilon}(t)-{{w}}(t)\big\|=0.
\]
As ${{w}}$ varies, $B_t^{\epsilon}:\ {{w}}\mapsto {
w}_{\epsilon}(t)$ defines an $\mathcal{F}_{t+\epsilon}$-measurable
process on $\Theta_+$. Each $B_t^{\epsilon}$ solves
\begin{equation*}
\frac{d}{dt}(\wt{B}_t)=\sum_{i=1}^{m}e_i(\wt{B}_t)\cdot
\frac{d}{dt}({{w}}_{\epsilon}^{i}(t)) 
\end{equation*}
and,  almost surely,  the limit of $B_t^{\epsilon}$ (as
$\epsilon\to 0$) gives the  Brownian motion $B_t$ (\cite{M}).

 Given a  sample path ${{w}}$
of $B_t$ starting from the origin,  the smoothed
curve ${{w}}_{\epsilon}$ has its lift in
$\mathcal{O}(\Bbb R^m)$ with tangent vectors $(d{
w}_{\epsilon}/dt, 0)$.  Let  $g\in \mathcal{M}^k(M)$ and let $\theta, \varpi$ and $H$  be the associated dual form, $\wt{g}$-connection form and horizontal lift map, respectively.   Consider the 
$C^{k-1}$ vector field on $\mathcal{O}^{\wt{g}}(\M)\times \Bbb R^+$: 
\[
Z_{t,u}^{f, \epsilon}:=(\theta, \varpi)_{u}^{-1}(\frac{d{
w}_{\epsilon}}{dt}, 0), \ \forall u\in \mathcal{O}^{\wt{g}}(\M). 
\]
We see that 
\[
Z_{t,u}^{f, \epsilon}=\sum_{i=1}^{m}H(u, e_i)\cdot \frac{d{
w}^{i}_{\epsilon}}{dt},
\]
where  $H(u, e_i)$ is horizontal lift of $ue_i$ to $HT\mathcal{F}(\M)$.  Let $\Phi_{f, t}^{\epsilon}$ be the flow generated by $Z_{t,u}^{f,
\epsilon}$. For $u\in \mathcal{O}^{\wt{g}}(\M)$, its orbit
$u^{\epsilon}(t)$  under $\Phi_{f,
t}^{\epsilon}$ solves the differential equation
\begin{equation}\label{approx-BM-SDE}
\frac{du^{\epsilon}(t)}{dt}=\sum_{i=1}^{m}H(u^{\epsilon}(t), e_i)\cdot
\frac{d{w}^{i}_{\epsilon}}{dt}.
\end{equation}
The projection of the orbit $t\mapsto u^{\epsilon}(t)$ to $\M$ has
tangent $u^{\epsilon}(t)(d{w}_{\epsilon}/dt)$ at time $t$ and is
an analog of the curve ${w}_{\epsilon}$. As
${w}$ varies, the distribution of the projection of $u^{\epsilon}(t)$ on $\M$ simulates  the distribution of the $\Bbb R^m$ Brownian
motion.  As $\epsilon$ tends to $0$, almost surely,  the differential system (\ref{approx-BM-SDE}) tends to 
\begin{equation}\label{OM-BM-SDE}
d{\rm u}_t=\sum_{i=1}^{m}H({\rm u}_t, e_i)\circ dB_t^i({w}), 
\end{equation}
which means for all smooth function  $\varphi$ on $\mathcal{O}^{\wt{g}}(\M)$,
\[
\varphi({\rm u}_t)=\varphi({\rm
u}_0)+\int_{0}^{t}\sum_{i=1}^{m}(H({\rm u}_s, e_i)\varphi)({\rm u}_s)\circ dB_s^i, \
0\leq t<\infty.
\]
Since the vector fields $H(\cdot, e_i)$ are $C^{k-1}$, for any initial ${\rm u}_0$, there exists a unique solution ${\rm
u}=({\rm  u}_t)_{t\in \Bbb R_+}$ to (\ref{OM-BM-SDE}), which is 
continuous in $(t, {\rm u}_0)$ for all $t\in \Bbb R_+$  (see  Proposition \ref{El-Ku}).

Recall that the generator ${\rm A}$ of  ${\rm u}_t$ is such that
\begin{equation*}
\varphi({\rm u}_t)-\varphi({\rm u}_0)-\int_{0}^{t}{\rm
A}\varphi({\rm u}_s)\ ds
\end{equation*}
is a local martingale for all smooth $\varphi$. By It\^{o}'s
formula, we see that
\[
{\rm A}=\sum_{i=1}^{m}H(\cdot, e_i)^2, 
\]
which is the Bochner horizontal Laplacian
$\Delta_{\mathcal{O}^{\wt{g}}(\M)}$. It is a lift of the Laplacian
$\Delta$ in the sense that for any smooth function
$\underline{\varphi}$ on $\M$ and its lift $\varphi$ to
$\mathcal{O}^{\wt{g}}(\M)$,
\begin{equation}\label{equi-laplacian}
\Delta_{\mathcal{O}^{\wt{g}}(\M)}\varphi(u)=\Delta\underline{\varphi}(\pi
u).
\end{equation}
Let ${\rm x}=({\rm x}_t)_{t\in \Bbb R_+}$ be the projection on $\M$ of the
solution ${\rm u}=({\rm u}_t)_{t\in \Bbb R_+}$ of (\ref{OM-BM-SDE})
with initial value ${\rm u}_0\in \mathcal{O}^{\wt{g}}_{{\rm x}_0}(\M)$.  It defines a measurable map from orbits
in $\Theta_+$ starting from the origin  to $C_{{\rm x}_0}(\Bbb R^+, \M)$,
the space of continuous paths on $\M$ starting from ${\rm x}_0$.
As ${\rm x}_0$ varies, ${\rm Q} (\rm x^{-1})$ gives a
distribution in the space of continuous paths on $\M$. For
$\tau\in \Bbb R_+$, let $C_{{\rm x}_0}([0, \tau], \M)$ be the collection 
of continuous paths $\rho: [0, \tau]\to \M$ with
$\rho(0)={{\rm x}}_0$. Then ${\rm x}$ also induces a measurable map
${\rm x}_{[0, \tau]}:\ \Theta_+ \to C_{{\rm x}_0}([0, \tau], \M)$ sending  $w$ to $({\rm x}_t({w}))_{t\in [0, \tau]}$. So, 
\[
\P_{\tau}:={\rm Q} ({\rm x}_{[0, \tau]}^{-1})
\]
gives the distribution probability of  paths ${\rm x}({w})$ on
$\M$ up to time $\tau$ and this distribution  is independent of the choice of the initial 
orthogonal frame ${\rm u}_0$ that projects to ${\rm x}_0$. Since ${\rm x}$ has generator
$\Delta$ by  (\ref{equi-laplacian}), it 
visualizes the Brownian motion on $\M$. This is the 
Eells-Elworthy-Malliavin's approach to obtain the Brownian motion on a manifold  (cf. \cite{El}).

\subsection{A stochastic analogue of the geodesic flow}\label{BMM-flow}
The regularity of the Brownian companion process ${\rm u}_t$ with respect to its initials ${\rm u}_0$ can be understood by  general  theory on stochastic flows associated to SDEs.

Let
$X_1, \cdots, X_d$ be bounded vector fields on a smooth finite dimensional Riemannian
manifold $({\rm N}, \langle\cdot, \cdot\rangle)$. 
Let $(z_t)_{t\in \Bbb R_+}=(z_t^1, \cdots, z_t^d)$ be a continuous stochastic process on $\Bbb R^d$.  An ${\rm N}$-valued semimartingale $({x}_t)_{t\in \Bbb R_+}$ defined up to a stopping time $\tau$ is said
to be a solution of the following Stratonovich SDE
\begin{equation}\label{SDE-x}
d{x}_t=\sum_{i=1}^{d}X_i({x}_t)\circ dz_{t}^{i}, 
\end{equation}
 if for all $\psi\in C^{\infty}({\rm N})$, 
\[
\psi({x}_t)=\psi({x}_0)+\int_{0}^{t} \sum_{i=1}^{d} X_i \psi({x}_s)\circ dz_{s}^{i}, \ 0\leq t<\tau.
\]
The solution to (\ref{SDE-x}) always exists when all $X_i$  are $C^1$ bounded (\cite{El}). Note that  ${\bf X}=(X_1, \cdots, X_d)$ is  a linear isomorphism from $\Bbb R^d$ to $T{\rm N}$. So, $x_t$ is a parallel transportation of $z_t$ to the manifold ${\rm N}$ via ${\bf X}$.  The pair $({\bf X}, (z_t)_{t\in \Bbb R_+})$ is called  a \emph{stochastic dynamical system}  (SDS) on ${\rm N}$ (\cite{El}) and it is said to be   $C^{j}$ if all $X_i$ are $C^j$ bounded. Using ${\bf X}$, we also write (\ref{SDE-x}) as 
\begin{equation*}
d{x}_t={\bf X}(x_t)\circ dz_t.  
\end{equation*}
The mapping
\[
F_t(\cdot, {w}):\ x_0({w})\mapsto x_t({w})
\]
 has the following regularity with respect to the starting point $x_0({w})$. 

\begin{prop}\label{El-Ku}(\cite[Theorem 3, Chapter VIII]{El}) Let  $({\bf X}, (z_t)_{t\in \Bbb R_+})$ be a $C^j$ SDS  on ${\rm N}$. There is a version of the explosion time map $x\mapsto \tau^x$, defined for $x\in {\rm N}$, and a version of $F_t(x, {w})$, defined when $t\in [0, \tau^x({w}))$, such that if ${\rm N}(t, {w})=\{x\in {\rm N}:\ t<\tau^x\}$, then the following are true  for each $(t, {w})\in \Bbb R_+\times \Theta_+$. 
\begin{itemize}
\item[i)] The set ${\rm N}(t, {w})$ is open in ${\rm N}$. 
\item[ii)] The map $F_t(x, {w}):\ {\rm N}(t, {w})\to {\rm N}$ is $C^{j-1}$ and is  a diffeomorphism onto an open subset of ${\rm N}$. Moreover, the map $\tau\mapsto F_{\tau}(\cdot, {w})$ of $[0, t]$ into $C^{j-1}$ mappings of ${\rm N}(t, {w})$  is continuous.
\end{itemize}
\end{prop}

 \begin{cor}\label{El-Ku-Cor-1}
Let $g\in \mathcal{M}^k(M)$ ($k\geq 3$). There is a version of the solution flow 
\[
F_t(\cdot, {w}):\ {\rm u}_0({w})\to {\rm u}_t({w}), \ t\in \Bbb R_+, 
\]
to (\ref{OM-BM-SDE}) in $\mathcal{F}(\M)$, 
which is a $C^{k-2}$ diffeomorphism  into $ \mathcal{F}(\M)$ and is continuous in $t$.
\end{cor}
\begin{proof}
Each $x\in \M$ has infinite distance to the boundary. Hence each solution process  ${\rm u}$ to (\ref{OM-BM-SDE}) with ${\rm u}_0\in \mathcal{F}_x(\M)$ projects to be a diffusion process on $\M$ starting from $x$ and has infinity explosion time. Since $g\in \mathcal{M}^k(M)$, the vector fields $H(\cdot, e_i)$, $i=1, 2, \cdots, m$,  on $\mathcal{O}^{\wt{g}}(\wt{M})$ are all $C^{k-1}$ bounded with respect to the $\wt{g}$ metric. So  $F_t(\cdot, {w}): {\rm u}_0(w)\mapsto {\rm u}_t(w)$  is $C^{k-2}$ with respect to the initial points ${\rm u}_0$ and is continuous in $t$ by Proposition  \ref{El-Ku}. 
\end{proof}

For $l\leq j-1$, the $l$-th tangent map of $F_t$ in Proposition  \ref{El-Ku}, denoted by $D^{(l)}F_t(\cdot, {w})$, can  be formulated and its norm can be estimated if ${\rm N}$ is equipped with a reference connection. 
 
  \begin{prop}(\cite{El})\label{SDE-flow-regularity} Let  $({\bf X}, (z_t)_{t\in \Bbb R_+})$ be a $C^j$ SDS  on ${\rm N}$. Assume there is a Levi-Civita connection $\bf \nabla$ induced by some metric such that the covariant derivatives ${\bf \nabla}^{\iota}X_i$, $\iota=0, 1, \cdots, j$, $i=1, \cdots, d$,  are bounded and the curvature tensor $R$ of $\bf \nabla$ and its first $j-1$ derivatives are bounded. The following hold true. 
\begin{itemize}
\item[i)] There is a version of $\{F_t(\cdot, {w})\}$ such that almost surely, for $l\leq j-1$,  $t\in \Bbb R_+$ and ${\bf\mathsf v}_0({w})\in T^{(l)}{\rm N}$, ${\bf\mathsf v}_t({w}):=[D^{(l)}F_t(\cdot, {w})]{\bf\mathsf v}_0({w})$ satisfies the Stratonovich SDE 
\[d{\bf\mathsf v}_t=\sum_{i=1}^{d}[D^{(l)}X_i({x}_t)] {\bf\mathsf v}_t\circ dz_{t}^{i},\]
where, if we denote by $F_{t}^i$ the deterministic flow map generated by the vector field $X_i$ and $D^{(l)}F_{t}^i$ its $l$-th  differential map, then for $\mathsf v\in T^{(l)}N$ with footpoint $x\in \rm N$, 
\[[D^{(l)}X_i(x)]\mathsf v:= \frac{D}{dt}([D^{(l)}F_{t}^i]\mathsf v).\]
\item[ii)] For any $q\in [1, \infty)$, there is a bounded function $c_l(t, q)$, which depends on $t$, $m, q$, and the bounds of ${\bf \nabla}^{\iota}X_i$ and ${\bf \nabla}^{\iota-1}R$, $\iota\leq l+1$, such that 
$\|[D^{(l)}F_t(\cdot, {w})]\|_{L^q}<c_l(t, q)$.
\end{itemize}
\end{prop}

Proposition \ref{SDE-flow-regularity} applies to  the flow map corresponding to  (\ref{OM-BM-SDE}). So we can formulate the SDEs of  $\{[D^{(l)}F_t(\cdot, {w})]\}$. We will use them to  specify $c_l(t, q)$ and give  a more detailed  study of their norm growths in time for later use.

 Let $F_t(\cdot, {w})$ be as in Corollary \ref{El-Ku-Cor-1}.  The first order tangent map $D^{(1)}F_t({\rm u}_0, {w})$ records the first order infinitesimal response of $F_t({\rm u}_0, {w})$ to the change of initial point ${\rm u}_0$. Let $\mathcal{C}: (-1, 1)\mapsto \mathcal{F}(\M)$ be a differential curve with $\mathcal{C}(0)={\rm u}_0, \mathcal{C}'(0)={\mathsf v}$. Then 
\[
{\mathsf v}_t:=\big[D^{(1)}F_t({\rm u}_0, {w})\big]{\mathsf v}=\left.\frac{D}{\partial s}F_{t}(\mathcal{C}(s), w)\right|_{s=0}. 
\]
The SDEs of ${\mathsf v}_t$ can be formulated using the parallelism form $(\theta, \varpi)$ as follows.

\begin{lem}\label{El-Ku-Cor}(\cite[Theorem 5.1]{M}) Let $F_t(\cdot, {w})$ be as in Corollary \ref{El-Ku-Cor-1}.  
\begin{itemize}
\item[i)] For any ${\mathsf v}\in T_{{\rm u}_0}{\mathcal{F}}(\M)$, ${\mathsf v}_t$ satisfies the Statonovich SDE 
\[
d{\mathsf v}_t({w})=\nabla({\mathsf v}_t({w}))H({\rm u}_t, \circ dB_t). 
\]
\item[ii)] Consider the  map 
\[
[\wt{D^{(1)}{F_t}}({\rm u}_0, {w})]:=(\theta, \varpi)_{{\rm u}_t}\circ [D^{(1)}{F_t}({\rm u}_0, {w})]\circ (\theta, \varpi)_{{\rm u}_0}^{-1}. 
\]
For  $(z(0), {\bf z}(0)):=(z^i(0), {\bf z}^l_j(0))\in T\mathcal{F}(\Bbb R^m)$,   $(z(t), {\bf z}(t)):=[\wt{D^{(1)}{F_t}}({\rm u}_0, {w})](z(0), {\bf z}(0))$  satisfies the Stratonovich SDE
\begin{align}\label{Mar-diff-flow-map}
\left\{ \begin{array}{ll}
dz(t)= {\bf z}(t)\circ d B_{t}({w}),\\
d{\bf z}(t)= {\rm u}_{t}^{-1}R\left({\rm u}_{t}\circ dB_t({w}), {\rm u}_{t}z(t)\right){\rm u}_{t}. 
\end{array}
 \right.
\end{align}
\item[iii)] The It\^{o} form of (\ref{Mar-diff-flow-map})  is 
\begin{align}\label{Mar-diff-flow-map-ito}\left\{ \begin{array}{ll}
dz(t)= {\bf z}(t)\ d B_{t}({w})+{\rm Ric}({\rm u}_{t}z(t))\ dt,\\
d{\bf z}(t)= {\rm u}_{t}^{-1}R\left({\rm u}_{t} dB_t({w}), {\rm u}_{t}z(t)\right){\rm u}_{t}+ {\rm u}_{t}^{-1}R\left({\rm u}_{t}e_i, {\rm u}_{t}{\bf z}(t)e_i\right)  {\rm u}_{t}\ dt\\
\ \ \ \ \ \ \ \  \ \ \ \ +{\rm u}_{t}^{-1}(\nabla ({\rm u}_{t}e_i)R)\left({\rm u}_{t} e_i, {\rm u}_{t}z(t)\right){\rm u}_{t}\ dt,
\end{array}
 \right.
\end{align}
where  the summation  $\Sigma_{i=1}^{m}$  is omitted in (\ref{Mar-diff-flow-map-ito}) for simplicity and 
\begin{align}\label{Ric-def}
{\rm Ric}({\rm u}z):= \sum_{i=1}^{m}{\rm u}^{-1}R({\rm u}e_i, {\rm u}z){\rm u}e_i,\  \forall {\rm u}\in \mathcal{F}(\M),\   z\in \Bbb R^m. 
\end{align}
\end{itemize}
\end{lem}

For $\underline{t}, \overline{t}$, $0\leq \underline{t}<\overline{t}\leq T$, let $F_{\underline{t}, \overline{t}}(\cdot, {w})$ be the flow map of (\ref{OM-BM-SDE}) sending ${\rm{u}}_{\underline{t}}$ to ${\rm{u}}_{\overline{t}}$. Then 
\[
F_{\overline{t}}({\rm{u}}_0, {w})\equiv F_{0, \overline{t}}({\rm{u}}_0, {w})=F_{\underline{t}, \overline{t}}({\rm{u}}_{\underline{t}}, {w})\circ F_{0, \underline{t}}({\rm{u}}_0, {w}). 
\]
Let $[D^{(l)}F_{\underline{t}, \overline{t}}(\cdot, {w})]$ ($l\leq k-2$) be the $l$-th  tangent map of $F_{\underline{t}, \overline{t}}$. When $l=1$,  let
\[
[\wt{D^{(1)}F_{\underline{t}, \overline{t}}}({\rm u}_0, {w})]:=(\theta, \varpi)_{{\rm u}_{\overline{t}}}\circ [D^{(1)}F_{\underline{t}, \overline{t}}({\rm u}_0, {w})]\circ (\theta, \varpi)_{{\rm u}_{\underline{t}}}^{-1}. 
\]
Then  $[D^{(1)}F_{\underline{t}, \overline{t}}]$ (resp. $[\wt{D^{(1)}F_{\underline{t}, \overline{t}}}]$)   satisfies the same SDE as $[D^{(1)}F_{0, t}]$ (resp. $[\wt{D^{(1)}F_{0, t}}]$).

To describe  $\big[D^{(2)}F_{t}({\rm u}_0, w)\big]$, we can follow \cite{El} to use the  horizontal/vertical Whitney sum decomposition of $T_{(u, {\mathsf v})}T{\mathcal{F}(\M)}=T_{u}{\mathcal{F}(\M)}\times T_{u}{\mathcal{F}(\M)}$ with respect to the Levi-Civita connection. The second order tangent vector
 \[
 (u, {\mathsf v}; {\mathbb V}_{0}, {\mathbb V}_{1})\in T_{(u, {\mathsf v})}T{\rm N}\]
 is in one-to-one correspondence with the Jacobi field $Y(s)$ along the geodesic $s\mapsto \mathcal{C}(s):=\exp(s{\mathsf v})$ with $Y(0)={\mathbb{V}}_0, \nabla Y(0)={\mathbb V}_1$, where $Y(0)$ tells the infinitesimal change of $\mathcal{C}(0)$ (i.e.,  the horizontal part change of $(\mathcal{C}(0), \mathcal{C}'(0))$) and $\nabla Y(0)$ tells the the infinitesimal change of $\mathcal{C}'(0)$ along the geodesic  from $u_0$ with initial velocity ${\mathbb{V}}_0$ (i.e.,  the vertical part change of $(\mathcal{C}(0), \mathcal{C}'(0))$.  For the geodesic  $\tau\mapsto \mathcal{C}_1(\tau):=\exp(\tau {\mathbb{V}_1})$,  let ${{\mathsf v}}_{\sslash}(\tau)$ be the parallel transportation of ${\mathsf v}$ along $\mathcal{C}_1$ to the point $ \mathcal{C}_1(\tau)$ and define 
\begin{align}\label{cov-der-F-e-t-ran}
\nabla_{\mathbb V_0} \big[D^{(1)}F_t({\rm u}_0, {w})\big]({{\mathsf v}})=\left.\frac{D}{\partial \tau}\big[D^{(1)}F_t(\mathcal{C}_1(\tau), {w})\big]({{\mathsf v}}_{\sslash}(\tau))\right|_{\tau=0}.
\end{align}
Then for  almost all ${w}$, 
 \begin{align*}
\big[D^{(2)}F_t({\rm u}_0, {w})\big] ({\rm u}_0, {{\mathsf v}}; {\mathbb V}_{0}, {\mathbb V}_{1})=\left(\big[D^{(1)}F_t({\rm u}_0, {w})\big]({\rm u}_0; {{\mathsf v}}); \big[D^{(2)}F_t({\rm u}_0, {w})\big]\big({\mathbb V}_{0}, {\mathbb V}_{1}\big) \right), \end{align*}
 where 
 \begin{align*}
&\big[D^{(2)}F_t({\rm u}_0, {w})\big]\big({\mathbb V}_{0}, {\mathbb V}_{1}\big)\\
&\ \ \ \ \ \ \ \ \ \ \ \  =\left(\big[D^{(1)}F_t({\rm u}_0, {w})\big]({\mathbb{V}}_0),  \nabla_{\mathbb V_0} \big[D^{(1)}F_t({\rm u}_0, {w})\big]({{\mathsf v}})+\big[D^{(1)}F_t({\rm u}_0, {w})\big]({\mathbb V}_{1})\right).
\end{align*}
 By Lemma \ref{El-Ku-Cor}, to describe $\left[D^{(2)}F_t(\cdot, {w})\right]({\mathbb{V}}_0, {\mathbb V}_{1})$, it remains  to identify 
 \[
 \mathcal{V}_t({\mathsf v}, {\mathbb V}_0, {w}):= \nabla_{\mathbb V_0} \big[D^{(1)}F_t({\rm u}_0, {w})\big]({{\mathsf v}}).
 \]

\begin{lem}\label{mathcal-V-t}(\cite[Lemma 5B, Chapter VIII]{El}) Let $g\in \mathcal{M}^k(M)$, $k\geq 4$. For  ${\mathsf v}\in T_{{\rm u}_0}\mathcal{F}(\M)$,  $(\mathbb{V}_0, 0)\in T_{({\rm u}_0, {\mathsf v})}T\mathcal{F}(\M)$, let ${\mathsf v}_t:=[D^{(1)}F_t({\rm u}_0, {w})]{\mathsf v}$,   $\mathbb{V}_t:=[D^{(1)}F_t({\rm u}_0, {w})]\mathbb{V}_0$. 
\begin{itemize}
\item[i)]On $T\mathcal{F}(\M)$, the process $\mathcal{V}_t:= \mathcal{V}_t({\mathsf v}, {\mathbb V}_0, {w})$ satisfies the Stratonovich SDE 
\[
d\mathcal{V}_{t}=\nabla(\mathcal{V}_{t})H({\rm u}_{t}, \circ dB_t)+\nabla^{(2)}({\mathsf v}_{t},  {\mathbb V}_{t})H({\rm u}_{t}, \circ dB_t)+R(H({
\rm u}_{t}, \circ dB_t),  {\mathbb V}_{t}){\mathsf v}_{t}.\]
\item[ii)]On $T\mathcal{F}(\Bbb R^m)$, the  process $(\theta, \varpi)_{{\rm u}_t}(\mathcal{V}_t)$ satisfies the Stratonovich SDE 
\begin{align}
 \ \ \ \ \ d\left((\theta, \varpi)_{{\rm u}_t}(\mathcal{V}_t)\right)=&\left(\varpi(\mathcal{V}_{t})\circ dB_t({w}), {\rm u}_{t}^{-1}R({\rm u}_{t}\circ dB_t, \theta(\mathcal{V}_{t})){\rm u}_{t}\right)\notag\\
&+(\theta, \varpi)_{{\rm u}_{t}}\left(\nabla^{(2)}({\mathsf v}_{t},  {\mathbb V}_{t})H({\rm u}_{t}, \circ dB_t({w}))+R(H({
\rm u}_{t}, \circ dB_t({w})),  {\mathbb V}_{t}){\mathsf v}_{t}\right).\label{2-theta-varpi-mathcal-V}
\end{align}
\item[iii)] The  It\^{o} form of (\ref{2-theta-varpi-mathcal-V}) is 
\begin{align}\label{ito-2-theta-varpi-mathcal-V}
d\theta(\mathcal{V}_{t})=&\ \varpi(\mathcal{V}_t) d B_{t}({w})+{\rm Ric}({\rm u}_{t}\theta(\mathcal{V}_{t})) dt+\Phi_{\theta}({\mathsf v}_t,{\Bbb V}_t, dB_t, dt),\\
\label{ito-2-varpi-mathcal-V}
d\varpi(\mathcal{V}_{t})=&\  {\rm u}_{t}^{-1}R\left({\rm u}_{t} dB_t({w}), {\rm u}_{t}\theta(\mathcal{V}_{t})\right){\rm u}_{t}+{\rm u}_{t}^{-1}R\left({\rm u}_{t}e_i, {\rm u}_{t}\varpi(\mathcal{V}_{t})e_i\right)  {\rm u}_{t}\ dt\\
&\  +{\rm u}_{t}^{-1}\big(\nabla({\rm u}_{t}e_i) R\big)\left({\rm u}_{t} e_i, {\rm u}_{t}\theta(\mathcal{V}_{t})\right){\rm u}_{t}\ dt+ \Phi_{\varpi}({\mathsf v}_t,{\Bbb V}_t, dB_t, dt),\notag
\end{align}
where  the summation  $\Sigma_{i=1}^{m}$  is  omitted  and 
\begin{align*}
\Phi_{\theta}({\mathsf v}_t,{\Bbb V}_t, dB_t, dt):=
&\theta\left(\nabla^{(2)}({\mathsf v}_{t},  {\mathbb V}_{t})H({\rm u}_{t},  dB_t({w}))+R(H({
\rm u}_{t},  dB_t({w})),  {\mathbb V}_{t}){\mathsf v}_{t}\right)\\
&+2\varpi\left(\nabla^{(2)}({\mathsf v}_{t},  {\mathbb V}_{t})H({\rm u}_{t}, e_i)+R(H({
\rm u}_{t}, e_i),  {\mathbb V}_{t}){\mathsf v}_{t}\right) e_i \ dt\\
&+\theta\left(\left[H({\rm u}_{t}, e_i), \nabla^{(2)}({\mathsf v}_{t},  {\mathbb V}_{t})H({\rm u}_{t}, e_i)+R(H({
\rm u}_{t}, e_i),  {\mathbb V}_{t}){\mathsf v}_{t}\right]\right)\ dt,\\ 
\Phi_{\varpi}({\mathsf v}_t,{\Bbb V}_t, dB_t, dt):=
&\varpi\left(\nabla^{(2)}({\mathsf v}_{t},  {\mathbb V}_{t})H({\rm u}_{t},  dB_t({w}))+R(H({
\rm u}_{t},  dB_t({w})),  {\mathbb V}_{t}){\mathsf v}_{t}\right)\\
&+2{\rm u}_{t}^{-1} R\left({\rm u}_{t}e_i, {\rm u}_{t}\theta\big(\nabla^{(2)}({\mathsf v}_{t},  {\mathbb V}_{t})H({\rm u}_{t}, e_i)+R(H({
\rm u}_{t}, e_i),  {\mathbb V}_{t}){\mathsf v}_{t}\big)\right){\rm u}_{t} dt \\
&+\varpi\left(\left[H({\rm u}_{t}, e_i), \nabla^{(2)}({\mathsf v}_{t},  {\mathbb V}_{t})H({\rm u}_{t}, e_i)+R(H({
\rm u}_{t}, e_i),  {\mathbb V}_{t}){\mathsf v}_{t}\right]\right)\ dt.
\end{align*}
\end{itemize}
\end{lem}

  A corollary of Lemma \ref{mathcal-V-t} is that we can  describe $\mathcal{V}_t$ (resp.  $(\theta, \varpi)_{{\rm u}_t}(\mathcal{V}_t)$)  using the tangent maps $\big[D^{(1)}F_{\underline{t}, \overline{t}}({\rm u}_0, {w})\big]$ (resp.  $\big[\wt{D^{(1)}F_{\underline{t}, \overline{t}}}({\rm u}_0, {w})\big]$)  by a stochastic version of the variation of constant method, i.e., a stochastic Duhamel principle.

\begin{cor}\label{Duha-mathcal-V-t}Let $g\in \mathcal{M}^k(M)$ ($k\geq 4$) and let ${\mathsf v}_t$, $\mathbb{V}_t$ and  $\mathcal{V}_t$ be as in Lemma \ref{mathcal-V-t}. 
\begin{align}
{\rm i)}\ &\mathcal{V}_t=\int_{0}^{t}\big[D^{(1)}F_{\tau, t}({\rm u}_{\tau}, {w})\big]\left(\nabla^{(2)}({\mathsf v}_{\tau},  {\mathbb V}_{\tau})H({\rm u}_{\tau}, e_i)+R(H({
\rm u}_{\tau}, e_i),  {\mathbb V}_{\tau}){\mathsf v}_{\tau}\right)\circ dB_{\tau}^i.\notag\\
{\rm ii)}\ &(\theta, \varpi)_{{\rm u}_t}(\mathcal{V}_t)=\notag\\
&\int_{0}^{t}\big[\wt{D^{(1)}F_{\tau, t}}({\rm u}_{\tau}, {w})\big](\theta, \varpi)_{{\rm u}_{\tau}}\!\!\left(\nabla^{(2)}({\mathsf v}_{\tau},  {\mathbb V}_{\tau})H({\rm u}_{\tau}, e_i)+R(H({
\rm u}_{\tau}, e_i),  {\mathbb V}_{\tau}){\mathsf v}_{\tau}\right)\circ dB_{\tau}^i.\label{VC-mathcal-V-t}\end{align}
\begin{itemize}
\item[iii)]The It\^{o} form of (\ref{VC-mathcal-V-t}) is 
\begin{align}
&(\theta, \varpi)_{{\rm u}_t}(\mathcal{V}_t)=\int_{0}^{t}\big[\wt{D^{(1)}F_{\tau, t}}({\rm u}_{\tau}, {w})\big]\left(\wt{\Phi}_{\theta}({\mathsf v}_\tau,{\Bbb V}_\tau, dB_\tau, d\tau), \wt{\Phi}_{\varpi}({\mathsf v}_\tau,{\Bbb V}_\tau, dB_\tau, d\tau)\right),\label{VC-mathcal-V-t-Ito}\end{align}
where
\begin{align*}
&\wt{\Phi}_{\theta}({\mathsf v}_\tau,{\Bbb V}_\tau, dB_\tau, d\tau)=\Phi_{\theta}({\mathsf v}_\tau,{\Bbb V}_\tau, dB_\tau, d\tau)\\
&\ \ \ \ \ \ \ \ \ \ \ \ \ \ \ \ \ \ \ \ \ \ \ \ \ \  -\varpi\left(\nabla^{(2)}({\mathsf v}_{\tau},  {\mathbb V}_{\tau})H({\rm u}_{\tau}, e_i)+R(H({
\rm u}_{\tau}, e_i),  {\mathbb V}_{\tau}){\mathsf v}_{\tau}\right) e_i d\tau,\\
&\wt{\Phi}_{\varpi}({\mathsf v}_\tau,{\Bbb V}_\tau, dB_\tau, d\tau)=\Phi_{\varpi}({\mathsf v}_\tau,{\Bbb V}_\tau, dB_\tau, d\tau)\\
&\ \ \ \ \ \ \ \ \ \ \ \ \ \ \ \ \ \ \ \ \ \ \ \ \ \   -{\rm u}_{\tau}^{-1} R\left({\rm u}_{\tau}e_i, {\rm u}_{\tau}\theta\big(\nabla^{(2)}\!({\mathsf v}_{\tau},  {\mathbb V}_{\tau})H({\rm u}_{\tau}, e_i)\!+\! R(H({
\rm u}_{\tau}, e_i),  {\mathbb V}_{\tau}){\mathsf v}_{\tau}\big)\right){\rm u}_{\tau} d\tau.
\end{align*}
\end{itemize}
\end{cor}
\begin{proof} For i) and ii), it suffices to show i) since it implies ii) by applying the $(\theta, \varpi)$ map. 
Regard the tangent map $\big[D^{(1)}F_t({\rm u}_0, {w})\big]$ as a random matrix solution ${\bf y}_t({w})$ to 
\[
d{\bf y}_t({w})=\nabla({\bf y}_t({w}))H({\rm u}_t, \circ dB_t),\ \  {\bf y}_0={\mbox{Id}}.
\]
Put
\begin{align*}
\upsilon_t:= &\mathcal{V}_0+\int_{0}^{t}\big[D^{(1)}F_\tau({\rm u}_0, {w})\big]^{-1}\left(\nabla^{(2)}({\mathsf v}_{\tau},  {\mathbb V}_{\tau})H({\rm u}_{\tau}, \circ dB_{\tau})+R(H({
\rm u}_{\tau}, \circ dB_{\tau}),  {\mathbb V}_{\tau}){\mathsf v}_{\tau}\right).
\end{align*}
Then the differentiation rule of Stratonovich integral shows that 
\begin{align*}
d({\bf y}_t\upsilon_t)&=(\circ d{\bf y}_t)\upsilon_t+ {\bf y}_t\circ d\upsilon_t\\
&=  \nabla({\bf y}_t({w})\upsilon_t)H({\rm u}_t, \circ dB_t)  +\nabla^{(2)}({\mathsf v}_{t},  {\mathbb V}_{t})H({\rm u}_{t}, \circ dB_t)+R(H({
\rm u}_{t}, \circ dB_t),  {\mathbb V}_{t}){\mathsf v}_{t},
\end{align*}
where $d$ should be understood as the covariant derivative. Since ${\bf y}_0\upsilon_0=\mathcal{V}_0=0$, we obtain 
 \begin{align*}
\mathcal{V}_t&={\bf y}_t\upsilon_t=\int_{0}^{t}\big[D^{(1)}F_{\tau, t}({\rm u}_{\tau}, {w})\big]\left(\nabla^{(2)}({\mathsf v}_{\tau},  {\mathbb V}_{\tau})H({\rm u}_{\tau}, e_i)+R(H({
\rm u}_{\tau}, e_i),  {\mathbb V}_{\tau}){\mathsf v}_{\tau}\right)\circ dB_{\tau}^i.
\end{align*}

Regard $\big[\wt{D^{(1)}F_t}(\cdot, {w})\big]$ as a matrix solution ${\bf y}_t({w})$ to (\ref{Mar-diff-flow-map-ito}) with ${\bf y}_0={\mbox{Id}}$. 
Put
\begin{align*}
\wt{\upsilon}_t&:=\mathcal{V}_0+\int_{0}^{t}\big[\wt{D^{(1)}F_\tau}({\rm u}_0, {w})\big]^{-1}\left(\wt{\Phi}_{\theta}({\mathsf v}_\tau,{\Bbb V}_\tau, dB_\tau, d\tau), \wt{\Phi}_{\varpi}({\mathsf v}_\tau,{\Bbb V}_\tau, dB_\tau, d\tau)\right).
\end{align*}
Write ${\bf y}_t\upsilon_t:=(({\bf y}_t\upsilon_t)_{\theta}, ({\bf y}_t\upsilon_t)_{\varpi})$, where $({\bf y}_t\upsilon_t)_{\theta}\in \Bbb R^m$ and $({\bf y}_t\upsilon_t)_{\varpi}\in \mathcal{F}(\Bbb R^m)$. 
Then the It\^{o} form infinitesimal differentiation rule  shows that 
\begin{align*}
d({\bf y}_t\wt{\upsilon}_t)&=(d{\bf y}_t)\wt{\upsilon}_t+ {\bf y}_t d\wt{\upsilon}_t+d{\bf y}_t\cdot d\wt{\upsilon}_t\\
&= \big(({\bf y}_t\wt{\upsilon}_t)_{\varpi} d B_{t}({w})+{\rm Ric}({\rm u}_{t}({\bf y}_t\wt{\upsilon}_t)_{\theta}) dt+\Phi_{\theta}({\mathsf v}_t,{\Bbb V}_t, dB_t, dt),\\
&\ \ \ \ \ {\rm u}_{t}^{-1}R\left({\rm u}_{t} dB_t({w}), {\rm u}_{t}({\bf y}_t\wt{\upsilon}_t)_{\theta}\right){\rm u}_{t}+{\rm u}_{t}^{-1}R\left({\rm u}_{t}e_i, {\rm u}_{t}({\bf y}_t\wt{\upsilon}_t)_{\varpi}e_i\right)  {\rm u}_{t}\ dt\\
&\ \ \ \ \left.+{\rm u}_{t}^{-1}\nabla R({\rm u}_{t}e_i)\left({\rm u}_{t} e_i, {\rm u}_{t}({\bf y}_t\wt{\upsilon}_t)_{\theta}\right){\rm u}_{t}\ dt+ \Phi_{\varpi}({\mathsf v}_t,{\Bbb V}_t, dB_t, dt)\right).
\end{align*}
This means  ${\bf y}_t\wt{\upsilon}_t$ with $\wt{\upsilon}_0=(0, 0)$ solves (\ref{ito-2-theta-varpi-mathcal-V}) and  (\ref{ito-2-varpi-mathcal-V}). Thus  (\ref{VC-mathcal-V-t-Ito}) holds true. 
\end{proof}

For $\big[D^{(2)}F_t({\rm u}_0, {w})\big]$ on $T_{({\rm u}, {\rm v})}T_{{\rm u}}\mathcal{F}(\M)=T_{{\rm u}}\mathcal{F}(\M)\times T_{{\rm u}}\mathcal{F}(\M)$, we can define its Euclidean companion map $\big[\wt{D^{(2)}F_t}({\rm u}_0, {w})\big]$ on $T\mathcal{F}(\Bbb R^m)\times T\mathcal{F}(\Bbb R^m)$ as follows. For $\big(\underline{\Bbb {V}}_0, \underline{\Bbb {V}}_1\big)\in  T\mathcal{F}(\Bbb R^m)\times T\mathcal{F}(\Bbb R^m)$, let $(\Bbb V_0, \Bbb V_1) :=(((\theta, \varpi)_{{\rm u}_0})^{-1}(\underline{\Bbb {V}}_0), ((\theta, \varpi)_{{\rm u}_0})^{-1}(\underline{\Bbb {V}}_1))$. Let $\Bbb{V}_{i, t}:=[D^{(1)}F_t({\rm u}_0, {w})]\Bbb{V}_i$ for $i=0, 1$ and let  ${\rm v}_t, \mathcal{V}_t$ be defined as in Lemma \ref{mathcal-V-t}. Then 
\begin{align*}
\big[\wt{D^{(2)}F_t}({\rm u}_0, {w})\big]\big(\underline{\Bbb {V}}_0, \underline{\Bbb {V}}_1\big):=\big((\theta, \varpi)(\Bbb{V}_{0, t}), (\theta, \varpi)({\Bbb {V}}_{1, t}+\mathcal{V}_t)\big).
\end{align*}

We can continue the above discussion to formulate  $\big[D^{(l)}F_t(\cdot, {w})\big]$, $3\leq l\leq k-2$.  Put 
\begin{align*}
\left\{\begin{array}{l}
\big({\rm{u}}^{(2)}; {{\mathsf v}}^{(2)}\big)=: ({\rm{u}}, {{\mathsf v}}; {\mathbb V}_{0}, {\mathbb V}_{1}),\\
\big({\rm{u}}^{(l)}; {{\mathsf v}}^{(l)}\big)\ =:\big({\rm{u}}^{(l-1)}, {{\mathsf v}}^{(l-1)};  {\mathbb V}_{0}^{(l-1)}, {\mathbb V}_{1}^{(l-1)}\big), \ \forall({\mathbb V}_{0}^{(l-1)}, {\mathbb V}_{1}^{(l-1)})\in T_{u^{(l-1)}}T^{l-1}\mathcal{F}(\M).
\end{array}
\right. 
\end{align*}
Then,  \begin{align}\notag\left[D^{(l)}F_{t}({\rm{u}}_0, {w})\right]\!({\rm{u}}^{(l)}; {{\mathsf v}}^{(l)})=&\!\left(\left[D^{(l-1)}F_{t}({\rm{u}}_0, {w})\right](u^{(l-1)}, {{\mathsf v}}^{(l-1)});\  \big[D^{(l-1)}F_{t}({\rm{u}}_0, {w})\big]({\mathbb{V}}_0^{(l-1)}), \right.\\
\label{ind-der-F-e-t-ran}&\left.\ \!\nabla_{\mathbb V_0^{(l-1)}}\big[D^{(l-1)}F_{t}({\rm{u}}_0, {w})\big]({{\mathsf v}}^{(l-1)})+\big[D^{(l-1)}F_{t}({\rm{u}}_0, {w})\big]({\mathbb V}_{1}^{(l-1)})\right)\end{align}
and the covariant derivative term $\nabla_{\mathbb V_0^{(l-1)}} \big[D^{(l-1)}F_{t}({\rm{u}}_0, {w})\big]({{\mathsf v}}^{(l-1)})$  involves  a combination of  the $l'$-th ($l'\leq l-1$) covariant derivatives
\begin{align}\label{all-cov-der-F-e-t-ran}
\nabla_{\mathbb V_{0,l'}}\nabla_{\mathbb V_{0,l'-1}}\cdots \nabla_{\mathbb V_{0,0}}\left[D^{(1)}F_{t}({\rm u}, {w})\right](\mathsf v),  \ \forall {\mathsf v}, \mathbb V_{0,0}, \cdots, \mathbb V_{0, l'}\in T_{{\rm u}}\mathcal{F}(\M),
\end{align}
where for $l'=1$, (\ref{all-cov-der-F-e-t-ran}) was given in (\ref{cov-der-F-e-t-ran}), and for $l'>1$, let $\tau\mapsto \mathcal{C}_{l'}(\tau):=\exp(\tau {\mathbb{V}_{0, l'}})$ be the geodesic passing through ${\rm u}$ and let  ${{\mathsf v}}_{\sslash}(\tau), {{\mathbb V}_{0, 0}}_{\sslash}(\tau), \cdots, {{\mathbb V}_{0, l'-1}}_{\sslash}(\tau)$ be the parallel transportations of ${\mathsf v}, {{\mathbb V}_{0, 0}}, \cdots, {{\mathbb V}_{0, l'-1}}$ along $\mathcal{C}_{l'}$ to the point $ \mathcal{C}_{l'}(\tau)$, then 
\begin{align*}
&\nabla_{\mathbb V_{0,l'}}\nabla_{\mathbb V_{0,l'-1}}\cdots \nabla_{\mathbb V_{0,0}}\left[D^{(1)}F_{t}({\rm u}, {w})\right](\mathsf v)\\
&\ \ \ \ \ \ \ \ =\left.\frac{D}{\partial \tau}\left( \nabla_{{\mathbb V_{0,l'-1}}_{\sslash}(\tau)}\cdots \nabla_{{\mathbb V_{0,0}}_{\sslash}(\tau)}\left[D^{(1)}F_{t}(\mathcal{C}_{l'}(\tau), w)\right]({\mathsf v}_{\sslash}(\tau))\right)\right|_{\tau=0}.
\end{align*}
The Stratonovich SDE of (\ref{all-cov-der-F-e-t-ran}) involves 
$\{{\bf \nabla}^{\iota}H\}_{\iota\leq l'}$,  $\{{\bf \nabla}^{\iota}R\}_{\iota\leq l'-1}$. But the It\^{o} SDE of (\ref{all-cov-der-F-e-t-ran})  involves  $\{{\bf \nabla}^{\iota}H\}_{\iota\leq l'+1}$, $\{{\bf \nabla}^{\iota}R\}_{\iota\leq l'}$.

By Corollary \ref{El-Ku-Cor-1}, all  the tangent maps $[D^{(l)}F_t(\cdot, {w})]$ are invertible.  The inverse maps  $[D^{(l)}F_t(\cdot, {w})]^{-1}$ can be formulated by the same equation as $[D^{(l)}F_t(\cdot, {w})]$, but using the backward infinitesimals $d\overrightarrow{B}_{\tau}, -d\tau$ instead of $d{B}_{\tau}, d\tau$. We skip the details.

\subsection{Growth of the stochastic tangent maps in time}
 We use the above SDEs to estimate   the $L^q$-norm  ($q\geq 1$)  of   $\sup_{0\leq \underline{t}<\overline{t}\leq T}\|[D^{(l)}F_{\underline{t}, \overline{t}}({\rm u}, {w})]\|$. 

A  useful tool to the $L^q$-norm estimations of  stochastic integrals is Burkholder's inequality which can be  obtained using It\^{o}'s formula for $|\cdot|^q$ and Doob's inequality of martingales.

 \begin{lem}(cf. \cite[Theorem 2.3.12]{Ku2})\label{Ku-lem} For an $\mathcal{F}_{\tau}$-adapted $\Bbb R^m$ or $\mathcal{O}(\Bbb R^m)$ process $f_{\tau}$,  \begin{equation}\label{equ-Ku-lem}
 \E\left(\left|\int_{t}^{t'}\langle f_{\tau}, d B_{\tau}\rangle\right|^q\right)\leq {\mathtt C}_1(q)\cdot\E\left|\int_{t}^{t'}|f_{\tau}|^2\ d\tau\right|^{\frac{q}{2}}, \ \forall q\geq 2,
 \end{equation}
 where ${\mathtt C}_1(q)=(\frac{1}{2}q(q-1)({q}/{(q-1)})^{q-2})^{\frac{q}{2}}$.
 \end{lem}
 (When $q=2$, the inequality in (\ref{equ-Ku-lem})  becomes  an equality  and is  referred to as the  isometry property of Brownian motion.)

We would  like to  list a  simple fact  that will be used  from time to time  for  computations   in the remaining paper:  for any $q\geq 1$ and $a_1, \cdots, a_{i_0}\in \Bbb R_+\cup\{0\}$, $i_0\in \Bbb N$, 
\begin{equation}\label{abcq}
(\sum_{i=1}^{i_0} a_i)^q\leq (i_0)^{q-1}\sum_{i=1}^{i_0} a_i^q. 
\end{equation}

  Recall the Dambis-Dubins-Schwarz Theorem which relates local martingales with Brownian motion using L\'{e}vy's characterization (see Section \ref{BBCE}).

\begin{lem}(cf. \cite[Theorems 1.6 \&1.7, p. 181]{RY})\label{D-D-S} If ${\mathsf M}$ is a $(\mathtt{\Omega}, \mathtt {F}, \mathtt{P})$-continuous local martingale vanishing at $0$. Let  $T_t=\inf\{s: \langle{\mathsf M}, {\mathsf M}\rangle_s>t\}$. 
\begin{itemize}
\item[i)] If $\langle {\mathsf M}, {\mathsf M}\rangle_{\infty}=\infty$, then $\mathsf B_{t}={\mathsf M}_{T_t}$ is a $(\mathtt{F}_{T_{t}})$-Brownian motion and ${\mathsf M}_t=B_{\langle{\mathsf M}, {\mathsf M}\rangle_t}$. 
\item[ii)] If $\langle {\mathsf M}, {\mathsf M}\rangle_{\infty}<\infty$, then there exist an enlargement $(\wt{\mathtt{\Omega}}, \wt{\mathtt{F}}, \wt{\mathtt{P}})$ of $(\mathtt{\Omega}, \mathtt {F}, \mathtt{P})$ and a Brownian motion $\wt{\mathsf B}$ on $\wt{\mathtt{\Omega}}$ independent of $\mathsf M$ such that the process
\begin{align*}
\mathsf B_t=\left\{ \begin{array}{ll}
{\mathsf M}_{T_t}, &\mbox{if}\ \  t<\langle\mathsf{M}, \mathsf{M}\rangle_{\infty},\\ 
{\mathsf M}_{\infty}+\wt{\mathsf B}_{t-\langle\mathsf{M}, \mathsf{M}\rangle_{\infty}}, &\mbox{if}\ \  t\geq \langle\mathsf{M}, \mathsf{M}\rangle_{\infty}
\end{array}\right.
\end{align*}
is a standard linear Brownian motion. The process $\mathsf W$ given by
\begin{align*}
\mathsf W_t=\left\{ \begin{array}{ll}
{\mathsf M}_{T_t}, &\mbox{if}\ \  t<\langle\mathsf{M}, \mathsf{M}\rangle_{\infty},\\ 
{\mathsf M}_{\infty}, &\mbox{if}\ \  t\geq \langle\mathsf{M}, \mathsf{M}\rangle_{\infty}
\end{array}\right.
\end{align*}
is a $(\wt{\mathtt F}_t)$-Brownian motion stopped at $\langle\mathsf{M}, \mathsf{M}\rangle_{\infty}$. 
\end{itemize}
\end{lem}

Given an  $(\mathtt{\Omega}, \mathtt {F}, \mathtt{P})$-Brownian motion ${\mathsf B}$,  we know that for almost all $w$, $t\mapsto \mathsf B_t(w)$ is not differentiable, but is $\mathtt a$-H\"{o}lder continuous for every $\mathtt a\in (0, 1/2)$. Let $\intercal>0$ be fixed. Define 
\begin{equation}\label{B_0_T_a}
\|{\mathsf B}_{[0, \intercal]}({w})\|_{\mathtt a}:=\sup\limits_{0\leq t\leq \intercal}|{\mathsf B}_t({w})|+\sup\limits_{0\leq t<t'\leq \intercal}\frac{|{\mathsf B}_{t'}({w})-{\mathsf B}_{t}({w})|}{|t'-t|^{\mathtt a}}.
\end{equation}

\begin{lem}(cf. \cite{Sk})\label{lem-Sk} Let ${\mathsf B}$ be an  $(\mathtt{\Omega}, \mathtt {F}, \mathtt{P})$-Brownian motion. For any $\mathtt a\in (0, 1/2)$, there exists $\epsilon>0$ such that $\E\left(e^{\epsilon\|\mathsf B_{[0, \intercal]}\|_{\mathtt a}}\right)<\infty$. 
\end{lem}

\begin{remark}\label{rem-lem-sk} It is true for $\intercal=1$ by following Skorokhod (\cite{Sk}) to consider $\|B_{[0, 1]}\|_{\mathtt a}$ instead of $\|B_{1}\|$. Note that 
for any $t>0$ and  $a>0$, $\mathsf B_{t}$ has the same distribution as $\sqrt{a}\mathsf B_{t/a}$. In particular, this holds for $a=\intercal$. A simple calculation shows  that  $\|\mathsf B_{[0, \intercal]}\|_{\mathtt a}\leq (\sqrt{\intercal}+\sqrt{\intercal}/\intercal^{\mathtt a})\|\mathsf B_{[0, 1]}\|_{\mathtt a}$.  Hence for every $\intercal$, we can choose  $\epsilon(\intercal)=\min\{\epsilon(1)/2,\epsilon(1)/(\sqrt{\intercal}+\sqrt{\intercal}/\intercal^{\mathtt a})\}$.
\end{remark}

The following estimations are  similar to the estimation for the first order tangent map with   $\underline{t}=0, \overline{t}=T$ fixed  (see  \cite[Proposition 5A, Chapter VIII]{El}).

\begin{prop}\label{est-norm-D-j-F-t}Let $g\in \mathcal{M}^k(M)$ with $k\geq 3$.   For $x\in \M$ and  $T\in \Bbb R_+$,  let $\{{\rm{u}_t}\}_{t\in[0, T]}$ be the solution to (\ref{OM-BM-SDE}) in $\mathcal{O}^{\wt{g}}(\M)$ with ${\rm{u}}_0\in \mathcal{O}^{\wt{g}}_x(\M)$.  Then for every $l$, $1\leq l\leq k-2$,  and $q\geq 1$, there exist $\underline{c}_{l}(q)>0$, which depends on $l, m, q$ and the norm bounds of   $\{{\bf \nabla}^{l'}H\}_{l'\leq l+1}, \{{\bf \nabla}^{l'}R\}_{l'\leq l}$,  and $c_l(q)>0$,  which depends on $l, m, q$ and the norm bounds of   $\{{\bf \nabla}^{l'}H\}_{l'\leq 2}, {\bf \nabla} R$, such that
\begin{equation}\label{wt-DF-j}
\E\sup\limits_{0\leq \underline{t}<\overline{t}\leq T}\left\|\big[ D^{(l)} F_{\underline{t}, \overline{t}}( {\rm{u}}_{\underline{t}}, {w})\big]^{\pm 1} \right\|^q<\underline{c}_{l}(q)e^{c_l(q)T}.
\end{equation}
\end{prop}
\begin{proof} By using the cocycle property of the tangent maps, it suffices to show (\ref{wt-DF-j}) with $\underline{t}=0$.  We show it  by induction.   At each step, we only check the bound  for the tangent map since the estimation on the inverse map  can be obtained analogously  using its SDE.

We begin with the $l=1$ case.  It suffices to consider $\big[\wt{D^{(1)} F_{\overline{t}}}( {\rm{u}}_{0}, {w})\big]$. Following  \cite[Theorem 5.1]{M} (see Lemma \ref{El-Ku-Cor}), the solutions to (\ref{Mar-diff-flow-map-ito}) can be understood using multiplicative stochastic integral in Ito's form.  For each  $j\in \{1, 2, \cdots, m\}$ and  ${\rm u}\in \mathcal{F}(\M)$, define  a $m(m+1)\times m(m+1)$  matrix ${\bf M}_j({\rm u})$, which is an endomorphism from  $T_{o}\mathcal{F}(\Bbb R^m)$ to itself,
 such that  for $(z, {\bf z})\in T_{o}\mathcal{F}(\Bbb R^m)$, 
\begin{equation}\label{matrix-M-j-def}
{\bf M}_j({\rm u})\big((z, {\bf z})\big)=\left(({\bf z}^\iota_j)_{\iota=1}^m, ({{\bf
R}}_{q,j, l}^{\iota}({\rm u}){z}^{l})_{\iota, q=1}^{m}\right).
\end{equation}
Define another $m(m+1)\times m(m+1)$ matrix ${\bf N}({\rm u})$  (or an endomorphism  from  $T_{o}\mathcal{F}(\Bbb R^m)$ to itself)   such that  for $(z, {\bf z})\in T_{o}\mathcal{F}(\Bbb R^m)$, 
\begin{align}\label{matrix-N-def}
&{\bf N}({\rm u})\big((z, {\bf z})\big)=\left(0, {\bf N}^{\iota}_{q, l}({\rm u}){z}^l\right), 
\end{align}
where 
\begin{align*}
&{\bf N}^{\iota}_{q, l}({\rm u}):=\sum_{j=1}^m\big\langle (\nabla ({{\rm u} e_j})R\left({\rm u} e_j, {\rm u} e_l\right){\rm u} e_{q}, {\rm u} e_\iota\big\rangle.\notag
\end{align*}
Using ${\bf M}, {\bf N}$,  we conclude from  (\ref{Mar-diff-flow-map-ito}) that the  It\^{o} form of the  SDE of ${\bar{\bf z}}_t:=(z_t, {\bf z}_t)$ is\footnote{In  terms of the multiplicative stochastic integral,  (\ref{DF-multi-sto-int-0}) shows 
\begin{equation*}\label{DF-multi-sto-int}
 \big[\wt{D^{(1)} F_{\overline{t}}}( {\rm{u}}_{0}, {w})\big]=e^{\left\{\int_{0}^{\overline{t}}{\bf M}_j({\rm u}_{\tau}) \ dB^j_{\tau}({w})+{\bf N}({\rm{u}}_{\tau})\ d\tau\right\}}.\end{equation*}} 
\begin{equation}\label{DF-multi-sto-int-0}
d{\bar{\bf z}}_t({w})=\sum_{j=1}^m\left({\bf M}_j({\rm u}_t){\bar{\bf z}}_t({w})\ dB_t^j({w})+[{\bf M}_j({\rm u}_t)]^2{\bar{\bf z}}_t({w})\ dt\right) +{\bf N}({\rm u}_t){\bar{\bf z}}_t({w})\ dt.
\end{equation}
(The coefficient of ${\bf N}$ in (\ref{DF-multi-sto-int-0}) is  different  from that in \cite[Theorem 5.1]{M} since we are considering Brownian motion with generator $\Delta$ instead of $\Delta/2$.) By It\^{o}'s formula, 
\begin{align*}
d|{\bar{\bf z}}_t({w})|^{2q}=&2q|{\bar{\bf z}}_t({w})|^{2(q-1)}\langle {\bar{\bf z}}_t({w}), d{\bar{\bf z}}_t({w})\rangle+q|{\bar{\bf z}}_t({w})|^{2(q-1)}\langle d {\bar{\bf z}}_t({w}), d {\bar{\bf z}}_t({w})\rangle\\
&+2q(q-1)| {\bar{\bf z}}_t({w})|^{2(q-2)}\langle  {\bar{\bf z}}_t({w}), d {\bar{\bf z}}_t({w}) \rangle^2;\\
d\ln |{\bar{\bf z}}_t({w})|^{2q}=& \frac{1}{ |{\bar{\bf z}}_t({w})|^{2q}}d |{\bar{\bf z}}_t({w})|^{2q}-\frac{1}{2 |{\bar{\bf z}}_t({w})|^{4q}}\langle d |{\bar{\bf z}}_t({w})|^{2q}, d |{\bar{\bf z}}_t({w})|^{2q}\rangle. 
\end{align*}
Note that $\{{\bf M}_j\}, {\bf N}$ all have norms bounded by some constant depending on $R, \nabla R$.  Hence, 
\begin{equation}\label{DF-z-bound}
|{\bar{\bf z}}_{\overline{t}}({w})|^{2q}=e^{\int_{0}^{\overline{t}}d\ln |{\bar{\bf z}}_t({w})|^{2q} } |{\bar{\bf z}}_{0}|^{2q}\leq C(q)e^{C(q)\overline{t}}e^{\wt{q}\int_{0}^{\overline{t}}{\rm M}_j({\rm{u}}_{\tau})\ dB_{\tau}^j}|{\bar{\bf z}}_{0}|^{2q},
\end{equation}
where $C(q)$ depends on the norm bound of $R, \nabla R$ and $m, q$,  the number $\wt{q}$ depends on $q$,  and $\{{\rm M}_j\}_{j=1}^{m}$ are continuous real valued processes with bounds depending  on the norm bound of $R$.  Consider the  process 
\[
{\mathsf M}_t({w}):=\int_{0}^{t}{\rm M}_j({\rm{u}}_{\tau})\ dB_{\tau}^j. 
\]
It  is a continuous martingale with ${\mathsf M}_0=0$ and  with the quadratic variation $\langle \mathsf M, \mathsf M\rangle_t\leq C_1t$ for some constant $C_1$ which depends on the norm bound of  $R$. By Lemma \ref{D-D-S},  there exist a continuous martingale $\wt{\mathsf M}$ and a Brownian motion $\mathsf B$ on an enlargement $(\wt{\Theta}_+, \wt{\mathcal{F}}, \wt{\rm{Q}})$ of $(\Theta_+, \mathcal{F}, {\rm{Q}})$ so that $\wt{\mathsf M}$ has the same law as ${\mathsf M}$ and 
\[
\wt{\mathsf M}_t=\mathsf B_{\langle \wt{\mathsf M}, \wt{\mathsf M}\rangle_t}. 
\]
Fix $\mathtt a\in (0, 1/2)$ and consider  $\|\mathsf B_{[0, C_1]}\|_{\mathtt a}$. Let $\epsilon(1)$ be as in Lemma \ref{lem-Sk} and put  $\epsilon=\min\{\epsilon(1)/2, \epsilon(1)/(\sqrt{C_1}+\sqrt{C_1}/(C_1)^{\mathtt a})\}$. By Remark \ref{rem-lem-sk}, \[
\E_{\wt{\rm{Q}}}\big(e^{\epsilon \big\|\mathsf B_{[0, C_1]}\big\|_{\mathtt a}}\big)<\wt{C_1}<\infty,
\]
where $\wt{C_1}$ depends on $C_1$ and $\mathtt a$.   Let $t_1= \min\{C_1^{-1}(\epsilon \wt{q}^{-1})^{\frac{1}{\mathtt a}}, T\}$. By the definition of  $\|\cdot\|_{\mathtt a}$, 
\[
\big|{\mathsf M}_{t}-{\mathsf M}_{0}\big|\leq (C_1t)^{\mathtt a}\big\|\mathsf B_{[0, C_1]}\big\|_{\mathtt a}\leq \epsilon \wt{q}^{-1} \big\|\mathsf B_{[0, C_1]}\big\|_{\mathtt a}.
\]
Using this and (\ref{DF-z-bound}), we obtain \[
\E\sup\limits_{0\leq t\leq t_1}|{\bar{\bf z}}_{\overline{t}}({w})|^{2q}\leq C(q)e^{C(q)t_1}|{\bar{\bf z}}_{0}|^{2q}\cdot \E_{\wt{\rm{Q}}}\big(e^{\epsilon \big\|\mathsf B_{[0,  C_1]}\big\|_{\mathtt a}}\big)\leq  \wt{C_1}C(q)e^{C(q)t_1}|{\bar{\bf z}}_{0}|^{2q}. 
\]
This implies
\[
\E\sup\limits_{0\leq t\leq t_1}\left\|\big[ DF_{t}( {\rm{u}}_{0}, {w})\big] \right\|^{2q}\leq \wt{C}(q)e^{C(q)t_1},
\]
where $\wt{C}(q)$ depends on $\wt{C_1}, C(q), m, q$.  In the same way, we obtain 
\begin{align*}
& \E\sup\limits_{(i-1)t_1\leq t\leq it_1}\left\|\big[ DF_{(i-1)t_1, t}( {\rm{u}}_{(i-1)t_1}, {w})\big] \right\|^{2q}, \ \E\sup\limits_{i_1(T)t_1\leq t\leq T}\left\|\big[ DF_{i_1(T)t_1, t}( {\rm{u}}_{i_1(T)t_1}, {w})\big] \right\|^{2q}\\
&\ \ \ \ \ \leq \wt{C}(q)e^{C(q)t_1},\ \forall i, 1\leq i\leq i_1(T)=\max\{i\in \Bbb N:\ it_1<T\}. 
\end{align*}
Hence by using the  cocycle property and Markov property, we conclude that there are some $\underline{c}_{1}(q), c_1(q)$ of the prescribed type in the statement of the proposition such that 
\[
\E\sup\limits_{0\leq \overline{t}\leq T}\left\|\big[ D F_{0, \overline{t}}( {\rm{u}}_{0}, {w})\big] \right\|^{2q}<\big(\wt{C}(q)e^{C(q)t_1}\big)^{i_1(T)+1}<\underline{c}_{1}(q)e^{c_1(q)T}.
\]

We proceed  to show (\ref{wt-DF-j}) with $l=2$ and $\underline{t}=0$.  By the above conclusion in the $l=0$ case and the definition of $\big[D^{(2)}F_{{t}}({\rm u}_0, {w})\big]$,  it remains to analyze 
\[
\E\sup_{0\leq t\leq T}\left\|(\theta, \varpi)_{{\rm u}_t}(\mathcal{V}_t)\right\|^{2q},
\]
where $\mathcal{V}_t:=\nabla_{\mathbb V_0} \big[D^{(1)}F_t({\rm u}_0, {w})\big]({{\mathsf v}})$ and ${{\mathsf v}}, {\mathbb{V}}_0$ have norm 1.  
Put ${\mathsf v}_\tau:=[D^{(1)}F_{\tau}({\rm u}_0, {w})]{\mathsf v}$,   $\mathbb{V}_\tau:=[D^{(1)}F_{\tau}({\rm u}_0, {w})]\mathbb{V}_0$. Let 
\begin{align}
\label{DF2-thevarA} {\rm{A}}_\tau({w})&:=(\theta, \varpi)_{{\rm{u}}_{\tau}}\left(\nabla^{(2)}({\mathsf v}_{\tau},  {\mathbb V}_{\tau})H({\rm u}_{\tau}, \cdot)+R(H({
\rm u}_{\tau},  \cdot),  {\mathbb V}_{\tau}){\mathsf v}_{\tau}\right),\\
\notag{\rm{B}}_\tau({w})&:=(\theta, \varpi)_{{\rm{u}}_{\tau}} \left(\left[H({\rm u}_{\tau}, e_i), \nabla^{(2)}({\mathsf v}_{\tau},  {\mathbb V}_{\tau})H({\rm u}_{\tau}, e_i)+R(H({
\rm u}_{\tau}, e_i),  {\mathbb V}_{\tau}){\mathsf v}_{\tau}\right]\right),\\
\notag {\rm{C}}_\tau({w})&:=\left(\varpi\left(\nabla^{(2)}({\mathsf v}_{\tau},  {\mathbb V}_{\tau})H({\rm u}_{\tau}, e_i)+R(H({
\rm u}_{\tau}, e_i),  {\mathbb V}_{\tau}){\mathsf v}_{\tau}\right) e_i, \right. \\
&\ \ \ \ \ \   \left.  {\rm u}_{\tau}^{-1} R\left({\rm u}_{\tau}e_i, {\rm u}_{\tau}\theta\big(\nabla^{(2)}({\mathsf v}_{\tau},  {\mathbb V}_{\tau})H({\rm u}_{\tau}, e_i)+R(H({\rm u}_{\tau}, e_i),  {\mathbb V}_{\tau}){\mathsf v}_{\tau}\big)\right){\rm u}_{\tau}\right) \notag
\end{align}
and let 
\begin{align*}
\wt{\rm{A}}_t({w}) &:= \int_{0}^{t}\big[\wt{D^{(1)}F_{\tau, t}}({\rm u}_{\tau}, {w})\big] {\rm{A}}_\tau({w})\ dB_{\tau}({w}), \\
\wt {\rm{B}}_t({w}) &:=  \int_{0}^{t}\big[\wt{D^{(1)}F_{\tau, t}}({\rm u}_{\tau}, {w})\big] {\rm{B}}_{\tau}({w})\ d\tau,\  \wt {\rm{C}}_t({w}) :=  \int_{0}^{t}\big[\wt{D^{(1)}F_{\tau, t}}({\rm u}_{\tau}, {w})\big] {\rm{C}}_{\tau}({w})\ d\tau. 
\end{align*}
By Corollary \ref{Duha-mathcal-V-t}, 
\begin{align*}
\!\!\!\!(\theta, \varpi)_{{\rm u}_t}(\mathcal{V}_t)=\wt {\rm{A}}_t({w})+\wt {\rm{B}}_t({w})+\wt {\rm{C}}_t({w}). 
\end{align*}
Hence, by using (\ref{abcq}), we obtain 
\begin{align*}
3^{1-2q}\E\sup_{0\leq t\leq T}\left\|(\theta, \varpi)_{{\rm u}_t}(\mathcal{V}_t)\right\|^{2q}\leq\ &  \E\sup_{0\leq t\leq T}\left\|\wt {\rm{A}}_t({w})\right\|^{2q}+\E\sup_{0\leq t\leq T}\left\|\wt {\rm{B}}_t({w})\right\|^{2q}+\E\sup_{0\leq t\leq T}\left\|\wt {\rm{C}}_t({w})\right\|^{2q}\\
=:\ & (\wt{\rm{A}})+(\wt{\rm{B}})+(\wt{\rm{C}}). 
\end{align*}
For $(\wt{\rm{A}})$, it is true  by Doob's inequality of sub-martingales and  Burkholder's inequality that
\begin{align*}
(\wt{\rm{A}})\leq& {\mathtt C}(2q) \E\left\|\int_{0}^{T}\big[\wt{D^{(1)}F_{\tau, T}}({\rm u}_{\tau}, {w})\big]{\rm{A}}_\tau({w})\ dB_{\tau}({w})\right\|^{2q}\\
\leq& {\mathtt C}(2q){\mathtt C}_1(2q)  \E\left|\int_{0}^{T}\left\|\big[\wt{D^{(1)}F_{\tau, T}}({\rm u}_{\tau}, {w})\big]{\rm{A}}_\tau({w})\right\|^2\ d\tau\right|^{q}\\
\leq&  {\mathtt C}(2q){\mathtt C}_1(2q)T^{q}\left(\E \sup\limits_{0\leq \underline{t}<\overline{t}\leq T}\left\|\big[ \wt{D^{(1)} F_{\underline{t}, \overline{t}}}( {\rm{u}}_{\underline{t}}, {w})\big] \right\|^{4q}\right)^{\frac{1}{2}}\cdot \left(\E\sup\limits_{0\leq \tau\leq T}\left\|{\rm{A}}_\tau({w})\right\|^{4q}\right)^{\frac{1}{2}}, 
\end{align*}
where ${\mathtt C}(q):=(q/q-1)^q$ and ${\mathtt C}_1(q)$ is given in Lemma \ref{Ku-lem}.  
Using (\ref{DF2-thevarA}), we compute that \begin{align*}
\E\sup\limits_{0\leq \tau\leq T}\left\|{\rm{A}}_\tau({w})\right\|^{4q}\leq& C_{{\rm A}}^{4q}\left(\E\sup\limits_{0\leq \tau\leq T}\left\|{{\mathsf v}}_\tau({w})\right\|^{8q}\right)^{\frac{1}{2}}\cdot \left(\E\sup\limits_{0\leq \tau\leq T}\left\|\mathbb{V}_\tau({w})\right\|^{8q}\right)^{\frac{1}{2}}\\
\leq& (C_{{\rm A}}')^{4q}\E \sup\limits_{0\leq \underline{t}<\overline{t}\leq T}\left\|\big[ \wt{D^{(1)} F_{\underline{t}, \overline{t}}}( {\rm{u}}_{\underline{t}}, {w})\big] \right\|^{8q},
\end{align*}
where $C_{{\rm A}}, C_{{\rm A}}'$  depend on the  norm bounds  of  $\{{\bf \nabla}^{l}H\}_{l\leq 2}$ and  $\{{\bf \nabla}^{l}R\}_{l\leq 1}$.  With (\ref{wt-DF-j}) for $l=1$, we conclude that 
\begin{align*}
(\wt{\rm{A}})\leq C'(q)(C'_{{\rm A}}\sqrt{T})^{2q}\sqrt{c_1(4q) c_1(8q)}e^{\frac{1}{2}(c_1(4q)+c_1(8q))T}.
\end{align*}
Using  (\ref{wt-DF-j}) with $l=1$ and  H\"{o}lder's inequality, we have 
\begin{align*}
(\wt{\rm{B}})
\leq & T^{2q}\left(\E \sup\limits_{0\leq \underline{t}<\overline{t}\leq T}\left\|\big[ \wt{D^{(1)} F_{\underline{t}, \overline{t}}}( {\rm{u}}_{\underline{t}}, {w})\big] \right\|^{4q}\right)^{\frac{1}{2}}\cdot \left(\E\sup\limits_{0\leq \tau\leq T}\left\|{\rm{B}}_\tau({w})\right\|^{4q}\right)^{\frac{1}{2}}\\
\leq & (CT)^{2q}\left(\E \sup\limits_{0\leq \underline{t}<\overline{t}\leq T}\left\|\big[ \wt{D^{(1)} F_{\underline{t}, \overline{t}}}( {\rm{u}}_{\underline{t}}, {w})\big] \right\|^{4q}\right)^{\frac{1}{2}}\cdot \left(\E \sup\limits_{0\leq \underline{t}<\overline{t}\leq T}\left\|\big[ \wt{D^{(1)} F_{\underline{t}, \overline{t}}}( {\rm{u}}_{\underline{t}}, {w})\big] \right\|^{8q}\right)^{\frac{1}{2}}\\
\leq & (CT)^{2q}\sqrt{c_1(4q) c_1(8q)} e^{\frac{1}{2}(c_1(4q)+c_1(8q))T}
\end{align*}
and the same inequality holds true  for $({\rm{C}})$, 
where $C$ depends on the  norm bounds  of   $\{{\bf \nabla}^{l}H\}_{l\leq 3}$, $\{{\bf \nabla}^{l}R\}_{l\leq 2}$.   Similarly, we can obtain the estimation on  $\big[D^{(2)}F_{{t}}({\rm u}_0, {w})\big]^{-1}$. This finishes  the proof of  (\ref{wt-DF-j}) for the $l=2$ case.

Let $l\geq 3$.  Assume (\ref{wt-DF-j}) holds  true for tangents up to the  $(l-1)$-th order.   For the estimation on  $l$-th tangent map, by the inductive definition of $\big[ D^{(l)} F_{t}( {\rm{u}}_{0}, {w})\big]$    (see  (\ref{ind-der-F-e-t-ran})),  it remains to show 
\begin{align}\label{Dj-Ft-u}
 \E\sup\limits_{0\leq t\leq T}\left\|\nabla_{(\cdot)}\big[D^{(l-1)}F_{t}({\rm{u}}_0, {w})\big](\cdot) \right\|^q<\underline{c}_{l}(q)e^{c_{l}(q)T}. 
\end{align} 
This can be done as in the $l=2$ case  by formulating  $\mathcal{V}_t^{(l-1)}$ in terms of  $\big[D^{(l-1)}F_{\tau, t}({\rm{u}}_\tau, {w})\big]$  by Duhamel's principle and using the inductive assumption on (\ref{wt-DF-j}). 
\end{proof}

\subsection{Brownian bridge and conditional estimations}\label{BBCE}
We want to  further estimate the growth of (\ref{wt-DF-j}) with respect to Brownian bridge distributions using their SDEs,   which can be derived from the classical Cameron-Martin-Girsanov formula. 

We begin with some classical estimations on heat kernels in the non-compact case. 
 \begin{lem}(\cite[Theorem 6.1]{Sa})\label{Sa-Thm-6.1} Let $g\in \mathcal{M}^k(M)$ and let $p(t, x, y)$ be the heat kernel functions of the $\wt{g}$-Brownian motion on $\M$.  There exist constants $b_1, c_1, c_2,
\kappa_1$ (depending on $m$ and the curvature bound) such that for any  $t>0$ and $x, y\in \M$, we
have
\begin{equation}\label{S-p-t-upper-b}
p(t, x, y)\leq\frac{1}{\sqrt{{\rm{Vol}}_{\wt{g}}B(x, \sqrt{t}){{\rm{Vol}}_{\wt{g}}}B(y,
\sqrt{t})}}e^{c_1(1+b_1 t+\sqrt{\kappa_1 t})-\frac{(d_{\wt{g}}(x, y))^2}{c_2
t}}.
\end{equation}
\end{lem}

For later use, we would  like to  state a simplified rough  version of (\ref{S-p-t-upper-b}):  there are constants  $\underline{c}_0$ (which depends on $\|g^{\l}\|_{C^0}$) and $c_0$ (which depends on $\|g^{\l}\|_{C^2}$) such that 
\begin{equation}\label{p-t-rough}
p(T, x, y)\leq\  \underline{c}_0 T^{-m}e^{c_0(1+T)}.
\end{equation}

\begin{lem}(\cite[Theorem 1.5]{Li})  \label{Hs2-ST}Let $g\in \mathcal{M}^k(M)$ and let $p(t, x, y)$ be the heat kernel functions of the $\wt{g}$-Brownian motion on $\M$. Let $T>0$. There are constants $c(i, T)$,  $i\leq k-2$, which depend on $i, T$ and the curvature and its derivatives up to $i$-th order, such that,
for all $(t, x, y)\in (0, T]\times \M\times \M$,  the $i$-th covariant derivative of $\ln p$ satisfies
\begin{equation}\label{equ-Hs2-ST}
\|\nabla^{(i)}\ln p(t, x, y)\|\leq c(i, T)\left(\frac{1}{t}d_{\wt{g}}(x, y)+\frac{1}{\sqrt{t}}\right)^i. 
\end{equation}
\end{lem}

Let $T>0$. For $x, y\in \M$,   the distribution of the \emph{Brownian bridge} from $x$ to $y$ in time $T$,  i.e., the Brownian motion starting from $x$ conditioning on paths that are at $y$ at time $T$,  is 
\[
\P_{x, y,  T}(\cdot):=\E_{\P_x}\left(\cdot\big| {\rm x}_T=y\right).
\]
It is a probability on the bridge space 
\[
C_{x, y}([0, T], \M):=\left\{w\in C_x([0, T], M):\ w_0=x, w_T=y\right\}.\]

\begin{prop}\label{cond-nabla-ln-p} Write $\P^*_{x, y, T}:=p(T, x, y)\P_{x, y, T}$. Fix $T_0>0$.  For any $q\in \Bbb R_+$  and $T>T_0$,   there exists   $c$ depending on $m, q, T, T_0$ and  $\|g\|_{C^2}$ such that for all $x, y\in \M$,  
\begin{align}\label{exp-grad-lnp}
\E_{\P^*_{x, y, T}} e^{q\int_{0}^{T}\|\nabla\ln p(T-t, {\rm x}_{t}, y)\|\ dt}\leq & e^{c(1+d(x, y))}.
\end{align}
\end{prop}

\begin{proof} By (\ref{equ-Hs2-ST}), there is some  $\bar c$ which depends on $\|g\|_{C^2}$  and $T$ such that 
\begin{equation*}
\int_{0}^{T}\|\nabla\ln p(T-\tau, {\rm x}_{\tau}, y)\|\ d\tau\leq \bar c\sqrt{T} + \bar c\int_{0}^{T}\frac{1}{T-\tau}d({\rm x}_{\tau}, y)\ d\tau.\end{equation*}
Hence it is true by H\"{o}lder's inequality that  for $t_0\in (0, \min\{1, T_0/2\})$ small,
\begin{align*}
\!\left(\E_{\P^*_{x, y, T}}e^{q\int_{0}^{T}\|\nabla\ln p(T-t, {\rm x}_{t}, y)\|\ dt}\right)^2\!\! \leq & e^{2\bar cq\sqrt{T}}\E_{\P^*_{x, y, T}} e^{2\bar cq\int_{0}^{t_0}\frac{1}{\tau}d({\rm y}_{\tau}, y)\ d\tau} \cdot \E_{\P^*_{x, y, T}} e^{2\bar cq\int_{0}^{T-t_0}\!\frac{1}{t_0}d({\rm x}_{\tau}, y)\ d\tau}\\
=:  & e^{2\bar cq\sqrt{T}}({\mathtt E})(t_0)({\mathtt F})(t_0),
\end{align*}
where $({\rm y}_t)_{t\in [0, T]}$ denotes the Brownian motion starting from $y\in \M$. 
Let $t_0<T_0/2$. Then for $({\mathtt E})(t_0)$, by (\ref{p-t-rough}), we have 
\begin{align*}
({\mathtt E})(t_0)&= \E_{\P^*_{y, x, T}} e^{2\bar cq\int_{0}^{t_0}\frac{1}{\tau}d({\rm y}_{\tau}, y)\ d\tau}\\
&= \E_{\P_{y}} \left(e^{2\bar cq\int_{0}^{t_0}\frac{1}{\tau}d({\rm y}_{\tau}, y)\ d\tau}\cdot p(T-t_0, {\rm y}_{t_0}, x)\right)\\
&\leq \underline{c}_0 2^mT_0^{-m}e^{c_0(1+T)} \E_{\P_{y}} e^{2\bar cq\int_{0}^{t_0}\frac{1}{\tau}d({\rm y}_{\tau}, y)\ d\tau}. 
\end{align*}
To show there is some small $t_0>0$ (depending on $m, q, T, T_0$ and  $\|g\|_{C^2}$)  such that $({\mathtt E})(t_0)$ is bounded,  we can  use a trick from  Driver (\cite[Lemma 3.8]{D94}) to  compare  it  with Euclidean Brownian motions. Find finite many  smooth functions $\{f_i\}_{i=1}^{l}$ on $\M$ with $f_i(y)=0$ and  $
d(z, y)\leq \sum_{i=1}^{l}|f_i(z)|$ for all $z\in \M$, where all $f_i$ have bounded first  and second order differentials on $\M$. So for an upper bound estimation of $({\mathtt E})(t_0)$,  it suffices to consider  
\[
\E_{\P_y} e^{2\overline{c}q\int_{0}^{t_0}\frac{1}{\tau}|f({\rm y}_\tau)|\ d\tau} =:({\mathtt E}')(t_0)\]
for  any $C^2$  function $f$ on $\M$ with bounded differentials up to second order. Let $({\rm y}, {\mho},  B)$ be the triple which defines the Brownian motion on $\M$ starting from $y$.  By It\^{o}'s formula, 
\[
|f({\rm y}_t)|\leq \left|\int_{0}^{t}{\mho}_\tau^{-1}\nabla f({\rm y}_\tau)\ dB_\tau+\int_{0}^{t}\Delta f({\rm y}_\tau)\ d\tau\right|\leq \left|\int_{0}^{t}{\mho}_\tau^{-1}\nabla f({\rm y}_\tau)\ dB_\tau\right|+C(f)t, 
\]
where $C(f)$ is some constant which bounds $|\Delta f|$.  Hence, 
\[
({\mathtt E}')(t_0)\leq e^{2\overline{c}q t_0 C(f)}\cdot \E_{\P_x}\big(e^{2\overline{c}q\int_{0}^{t_0}\frac{1}{\tau}\|{\mathsf M}_{\tau}'\|\ d\tau}\big), \ \mbox{where}\ {\mathsf M}_t':=\int_{0}^{t}{\mho}_\tau^{-1}\nabla f({\rm y}_\tau)\ dB_\tau. \]
The process ${\mathsf M}_t'$ is a continuous martingale with ${\mathsf M}_0'$ being the zero vector and  has quadratic variation $\langle \mathsf M', \mathsf M'\rangle_t\leq C't$ for some constant $C'$ which depends on the bound of  $|\nabla f|$. So,  to show $({\mathtt E}')(t_0)$ is finite for small $t_0$, it suffices to show $\E_{\P_x}\big(e^{2\overline{c}q\int_{0}^{t_0}\frac{1}{\tau}\|{\mathsf M}_{\tau}'\|\ d\tau}\big)$ is for ${\mathsf M}_t'$ being in the  one-dimensional process case.   By Lemma \ref{D-D-S},  there exist a continuous martingale $\wt{\mathsf M}'$ and a Brownian motion $\mathsf B'$ on an enlargement $(\wt{\Omega}', \wt{\mathcal{F}}', \wt{\P}')$ of $(\Omega, \mathcal{F}, \P)$ such that $\wt{\mathsf M}'$ has the same law as $\mathsf M'$ and 
\[
\wt{\mathsf M}_t'=\mathsf B'_{\langle \wt{\mathsf M}', \wt{\mathsf M}'\rangle_t}. 
\]
Let $\mathtt a\in (0, 1/2)$.  By Lemma \ref{lem-Sk}, there is some $\epsilon'>0$ which depends on $\|g\|_{C^2}$ such that $
\E_{\wt{\P}'}\big(e^{\epsilon'\big\|\mathsf B'_{[0, \frac{1}{2}C']}\big\|_{\mathtt a}}\big)$ is finite. 
By the definition of  the H\"{o}lder norm $\|\cdot\|_{\mathtt a}$, 
\[
\int_{0}^{t_0}\frac{1}{\tau}\|{\mathsf M}_{\tau}'\|\ d\tau\leq \big\|\mathsf B'_{[0, \frac{1}{2}C']}\big\|_{\mathtt a}\cdot( C')^{\mathtt a} \int_{0}^{t_0} \tau^{\mathtt a-1}\ d\tau\leq \frac{1}{\mathtt a}(C't_0)^{\mathtt a}\big\|\mathsf B'_{[0, \frac{1}{2}C']}\big\|_{\mathtt a}.
\]
Hence,  for $t_0=\min\{1, T_0/2,  (\mathtt a (2\overline{c}q)^{-1}\epsilon')^{\frac{1}{\mathtt a}}/C'\}$, we have 
\[
 \E_{\P_x}\big(e^{2\overline{c}q\int_{0}^{t_0}\frac{1}{\tau}\|{\mathsf M}_{\tau}'\|\ d\tau}\big)\leq \E_{\wt{\P}'}\big(e^{\epsilon'\big\|\mathsf B'_{[0, \frac{1}{2}C']}\big\|_{\mathtt a}}\big)<\infty. 
\]
For $({\mathtt F})(t_0)$,  by symmetry of the bridge distribution, 
\begin{align*}
({\mathtt F})(t_0)=&\ \E_{\P^*_{x, y, T}} e^{2\bar cq\int_{0}^{\frac{1}{2}T}\frac{1}{t_0}d({\rm x}_{\tau}, y)\ d\tau+2\bar cq\int_{\frac{1}{2}T}^{T-t_0}\frac{1}{t_0}d({\rm y}_{T-\tau}, y)\ d\tau}\\
\ \ \ \ \ \ \ \ \ \ \ \ \ \leq&\ e^{\bar c qt_0^{-1}Td(x, y)}\left(\E_{\P^*_{x, y, T}}\! e^{4\bar cq\int_{0}^{\frac{1}{2}T}\frac{1}{t_0}d({\rm x}_{\tau}, x)\ d\tau}\right)^{\frac{1}{2}}\left(\E_{\P^*_{y, x,  T}}\! e^{4\bar cq\int_{0}^{\frac{1}{2}T}\frac{1}{t_0}d({\rm y}_{\tau}, y)\ d\tau}\right)^{\frac{1}{2}}.
\end{align*}
By (\ref{S-p-t-upper-b}) and Markov property of $p$ (see (\ref{bridge-margin})), 
\begin{align*}
\E_{\P^*_{y, x, T}} e^{4\bar cq\int_{0}^{\frac{1}{2}T}\frac{1}{t_0}d({\rm y}_{\tau}, y)\ d\tau}&=\ \E_{\P_y}\left(e^{4\bar cqt_0^{-1}\int_{0}^{\frac{1}{2}T}d({\rm y}_\tau, y)\ d\tau}\cdot p({\frac{1}{2}T}, {\rm y}_{\frac{1}{2}T}, x)\right)\\
&\leq\ \underline{c}_0 2^mT_0^{-m}e^{c_0(1+T)} \E_{\P_y} e^{4\bar cqt_0^{-1}\int_{0}^{\frac{1}{2}T}d({\rm y}_\tau, y)\ d\tau}. 
\end{align*}
Let $t'_0<\min\{1, T_0/2\}$ be small.  Partition $[0, T/2]$ into $0=\tau^0<\tau^1<\cdots<\tau^{N}<T/2$, where $\tau^i:=it'_0$ and $N:=\max\{i, \tau^i<T/2\}$, and chop the integral $\int_{0}^{\frac{1}{2}T}d({\rm y}_\tau, y)\ d\tau$ into pieces accordingly.  Using the triangle inequality, we obtain 
\begin{align*}
\int_{0}^{\frac{1}{2}T}\!\! d({\rm y}_\tau, y)\ d\tau\leq &\sum_{i=1}^{N}\left(\int_{\tau^{i-1}}^{\tau^i}\! d({\rm y}_{t}, {\rm y}_{\tau^{i-1}})\ dt+t_0'd({\rm y}_{\tau^{i-1}}, y)\!\right)\! + \int_{\tau^{N}}^{\frac{1}{2}T}\! d({\rm y}_{t}, {\rm y}_{\tau^{N}})\ dt+t_0'd({\rm y}_{\tau^{N}}, y)\\
\leq &\sum_{i=1}^{N}\left(\int_{\tau^{i-1}}^{\tau^i}\!\! d({\rm y}_{t}, {\rm y}_{\tau^{i-1}})\ dt+(N+1-i)t_0'd({\rm y}_{\tau^i}, {\rm y}_{\tau^{i-1}})\!\right)\! + \int_{\tau^{N}}^{\frac{1}{2}T}\! d({\rm y}_{t}, {\rm y}_{\tau^{N}})\ dt. 
\end{align*}
Using  the Markov property of the Brownian motion and H\"{o}lder's inequality,  we see that 
\begin{align*}
\left(\E_{\P_y} e^{4\bar cqt_0^{-1}\int_{0}^{\frac{1}{2}T}d({\rm y}_\tau, y)\ d\tau}\right)^{2}&\!\leq \left(\sup_{y'\in \M}\E_{\P_{y'}}e^{\bar q \int_{0}^{t'_0}d({\rm y}'_t, y')\ dt}\right)^{N+1}\!\!\!\cdot \prod_{i=1}^{N} \sup_{y'\in \M}\E_{\P_{y'}}e^{\bar qt_0'(N+1-i)d({\rm y}'_{t'_0}, y')}\\
&=: \big(({\mathtt F})_0\big)^{N+1}\prod_{i=1}^{N}({\mathtt F})_i,
\end{align*} 
where $\bar q:=8\bar cqt_0^{-1}$, ${\rm y}'_t$ is the Brownian motion on $\M$ which starts from $y'$ and $\P_{y'}$ is its distribution probability.  For $({\mathtt F})_0$, we estimate as in the first part.  Note that $(\M, \wt{g})$ is the universal cover of the compact space $(M, g)$, although the choice of the $f_i$ may differ from point to point, we can ensure their differentials up to second order are uniformly bounded. So we can choose $t'_0$ (for instance, $t'_0=\min\{1, T_0,  \epsilon'/\bar q\}$) such that  $({\mathtt F})_0$ is bounded.  Fix such a $t'_0$ and estimate $({\mathtt F})_i$ using Lemma \ref{Sa-Thm-6.1}. We obtain some constant $c(t'_0)$ such that 
\[
({\mathtt F})_i\leq c(t'_0)e^{c_2\bar q^2(t'_0)^3(N+1-i)^2},
\]
where $c_2$ is as in (\ref{S-p-t-upper-b}). Hence, there is some constant $c(\bar q)$ depending on $m, q,T, T_0$ and  $\|g\|_{C^2}$ such that 
\begin{equation*}
\E_{\P_y}\big(e^{4\bar cq't_0^{-1}\int_{0}^{\frac{1}{2}T}d({\rm y}_\tau, y)\ d\tau}\big)\leq e^{c(\bar q)}. 
\end{equation*}
So, 
\begin{equation*}
({\mathtt F})(t_0)\leq \underline{c}_0 2^mT_0^{-m}e^{c(\bar q)+c_0(1+T)+\bar c qt_0^{-1}Td(x, y)}. 
\end{equation*}
Putting the estimations on $({\mathtt E})(t_0), ({\mathtt F})(t_0)$ together, we obtain (\ref{exp-grad-lnp}).\end{proof}

 Consider the Wiener space 
$C_0([0, T], \Bbb R^m)$ with the standard filtration $(\mathcal{F}_t)_{t\in [0, T]}$ and  let $(B_t)_{t\in [0, T]}$ be an $(\mathcal{F}_t)$-Brownian motion starting from $0$ with respect to a probability measure   ${{\rm Q}}$ on $\mathcal{F}_{T}$.  Let $f:[0, T]\mapsto \Bbb R^m$ be square integrable with respect to Lebesgue measure. Define a random process $({\rm M}_t)_{t\in [0, T]}$ on $[0, T]$ satisfying ${\rm
M}_{0}=1$ and It\^{o}'s SDE
\[
d{\rm M}_{t}= \frac{1}{2}{\rm M}_{t}\langle f_t, 
dB_t\rangle.
\]
Then 
\[{\rm M}_{t}=e^{\{\frac{1}{2}\int_{0}^{t}\langle f_\tau,   dB_\tau({w})\rangle-\frac{1}{4}\int_{0}^{t}|f_\tau|^2\ d\tau\}}.\]
Since $\E_{\rm Q}(e^{\frac{1}{4}\int_{0}^{t}|f_\tau|^2\ d\tau})$, $t\leq T$, are all finite, we have by Novikov
(\cite{N}),    that $({\rm M}_{t})_{t\in [0, T]}$ is a continuous
$(\mathcal{F}_t)_{t\in [0, T]}$-martingale, i.e.,
\[
\E_{\rm Q}\left({\rm M}_t\right)=1,\ \ \forall  t\in [0, T]. 
\]
For $t\in [0, T]$, let $\wt{\rm Q}_{t}$ be the probability
on $C_0([0, T], \Bbb R^m)$, which is absolutely continuous with respect to $\rm{Q}$
with
\[
\frac{d{\wt{{\rm Q}}_{t}}}{d{\rm Q}}({w})={{\rm M}}_{t}(w).
\]
Since  ${\rm M}_{t}$ is a martingale, the projection
of $\wt{{\rm Q}}_{t}$ on $\mathcal{F}_{\tau}$,  $\tau<t$,  is given by
the same formula.  The classical Cameron-Martin-Girsanov Theorem (\cite{CM1, CM2, Gi}) says that  the process $(B_{t}- \int_0^{t}f_\tau\ d\tau)_{t\in [0, T]}$ is a Brownian motion with respect to $\wt{\rm Q}_{T}$.  In other words,   we have that the probability $\rm{Q}$ on Wiener space is quasi-invariant under the transformation ${\bf T}:  C_0([0, T], \Bbb R^m) \to C_0([0, T], \Bbb R^m):\ {\rm w}\ \mapsto {\rm w}+ \int_0^{\cdot}f_\tau\ d\tau$
with 
\begin{equation}\label{CM-RN-derivative}
\frac{d{{\rm Q}}\circ {\bf T}^{-1}}{d{{{\rm Q}}}}({ w})=e^{\left\{\frac{1}{2}\int_{0}^{T}\langle f_\tau({ w}),  \ dB_\tau({ w})\rangle-\frac{1}{4}\int_{0}^{T}|f_\tau({ w})|^2\ d\tau\right\}}. 
\end{equation}

As in the compact case (see \cite[Theorem 5.4.4]{Hs}), we can deduce the SDE of the Brownian bridge on $\M$ from  the Cameron-Martin-Girsanov Theorem. Let $({\rm x}_t, {\rm u}_t)_{t\in [0, T]}$ be the stochastic pair which defines the Brownian motion starting from $x$ up to time $T$.  By the Markov property of $p$, 
\begin{equation}\label{bridge-margin}
\left.\frac{d\P_{x, y,  T}}{d\P_x}\right|_{\mathcal{F}_t}=\frac{p(T-t, {\rm x}_t, y)}{p(T, x, y)}=\frac{\wt{p}(T-t, {\rm u}_t, y)}{p(T, x, y)}=: \varXi_t, \ \forall t\in [0, T),  
\end{equation}
where $\wt{p}(t, {\rm u}, y):=p(t, \pi({\rm u}), y)$.  Using (\ref{OM-BM-SDE}) and the heat equation, one can calculate using It\^{o}'s formula to obtain 
\[
d\ln\varXi_t=\langle {\rm u}_t^{-1}\nabla^{H}\ln \wt{p}(T-t, {\rm u}_t, y), dB_t \rangle-\|\nabla^{H}\ln \wt{p}(T-t, {\rm u}_t, y)\|^2\ dt. 
\]
Hence 
\begin{equation*}
\left.\frac{d\P_{x, y,  T}}{d\P_x}\right|_{\mathcal{F}_t}=e^{\left\{\int_{0}^{t}\langle {\rm u}_t^{-1}\nabla^{H}\ln \wt{p}(T-\tau,  {\rm u}_\tau, y), dB_\tau \rangle-\int_{0}^{t}\|\nabla^{H}\ln \wt{p}(T-\tau,  {\rm u}_\tau, y)\|^2 d\tau\right\}}. 
\end{equation*}
Comparing this with (\ref{CM-RN-derivative}),  it implies 
\[b_t:=B_t-2\int_{0}^{t}{\rm u}_\tau^{-1}\nabla^{H}\ln \wt{p}(T-\tau, {\rm u}_\tau, y)\ d\tau\]
is a Brownian motion with respect to $\P_{x, y,  T}$ and hence  Proposition \ref{Hsu-thm-5.4.4} holds for $t\in [0, T)$. One can  conclude that  $\{ {\mathtt U}_t\}_{t\in [0, T]}$  is a   semi-martingale on $[0, T]$  since  \begin{equation*}
\E_{\P_{y, x,  T}}\left(\int_{0}^{T}\|\nabla \ln p(T-\tau,  {\rm x}_\tau, y)\| d\tau\right)<\infty
\end{equation*}
is also true on the non-compact universal cover space $(\M, \wt{g})$ by (\ref{exp-grad-lnp}). In summary,

\begin{lem}\label{Hsu-thm-5.4.4}There is a Brownian motion $(b_t)_{t\in [0, T)}$ such that the horizontal lift $\mathtt U$ of the Brownian bridge $\mathtt x$  is a semi-martingale on $[0, T]$ which satisfies the SDE 
\begin{equation}\label{B-bridge}
d{\mathtt U}_t=H({\mathtt U}_t, e_i)\circ \left(db_t^i+2H({\mathtt U}_t, e_i)\ln \wt{p}_{M}(T-t, {\mathtt U}_t, x)dt\right).
\end{equation}
In other words, the anti-development of the Brownian bridge $\mathtt x$ (i.e., the pre-image via  parallelism, see Section \ref{flow-F-S-y} for more precise definition) is 
\[
W_t=b_t+2\int_{0}^{t}{\mathtt U}_{\tau}^{-1}\nabla \ln p(T-\tau, {\mathtt x}_{\tau}, y)\ d\tau. 
\]
\end{lem}

Now, we can use Proposition \ref{cond-nabla-ln-p} and Lemma \ref{Hsu-thm-5.4.4} to derive a bridge version  of (\ref{wt-DF-j}).

\begin{prop}\label{est-norm-D-j-F-t-cond} Let $g\in \mathcal{M}^k(M)$, $k\geq 3$.  For $x\in \M$, let $({\rm{u}_t})_{t\in[0, T]}$ be the solution to (\ref{OM-BM-SDE}) in $\mathcal{O}^{\wt{g}}(\M)$ with ${\rm{u}}_0\in \mathcal{O}^{\wt{g}}_x(\M)$.  For every $T_0>0$, $l$, $1\leq l\leq k-2$, $q\geq 1$ and $T>T_0$,  there exist $\underline{c}'_{l}(q)>0$, which depends on $l, m, q$ and the norm bounds of   $\{{\bf \nabla}^{l'}H\}_{l'\leq l}, \{{\bf \nabla}^{l'}R\}_{l'\leq l}$,  and $c_l'(q)>0$,  which depends on $l, m, q, T, T_0$ and the norm bounds of   $\{{\bf \nabla}^{l'} R\}_{l'\leq 1}$, such that
\begin{equation}\label{wt-DF-j-cond}
\E_{\P^*_{x, y, T}}\sup\limits_{0\leq \underline{t}<\overline{t}\leq T}\left\|\big[ D^{(l)} F_{\underline{t}, \overline{t}}( {\rm{u}}_{\underline{t}}, {w})\big]^{\pm 1} \right\|^q<\underline{c}'_{l}(q)e^{c'_l(q)(1+d(x, y))}, \ \forall x, y\in \M. 
\end{equation}
\end{prop}
\begin{proof}Using the cocycle property of the tangent map, it suffices to show   (\ref{wt-DF-j-cond})  for $\underline{t}=0$. We show this by induction and in each step, we only verify it for  the forward tangent map. 

When $l=1$, it suffices to consider $\big[\wt{D^{(1)}F_{\overline{t}}}(\cdot, {w})\big]$, whose SDE is as in (\ref{DF-multi-sto-int-0}) with 
\begin{align}\label{B-b-relation-x}
dB_{\tau}=db_{\tau}+2{\rm u}_{\tau}^{-1}\nabla \ln p(T-\tau, {\rm x}_{\tau}, y)\ d\tau,
\end{align}
where $(b_\tau)_{\tau\in [0, T)}$ is a Brownian motion  for  $\P_{x, y, T}$ by Lemma \ref{Hsu-thm-5.4.4}.  Hence the conditional norm of $\big\|\big[\wt{D^{(1)} F_{\overline{t}}}( {\rm{u}}_{0}, {w})\big]\big\|^q$ differs in distribution with the nonconditional case by a multiple $e^{\wt{c}(q)\int_{0}^{T}\|\nabla\ln p(T-t, {\rm x}_{t}, y)\|\ dt}$ for some constant $\wt{c}(q)$ which depends on the norm bound of $R, \nabla R$ and $m, q$.  Hence by 
H\"{o}lder's inequality and Proposition \ref{est-norm-D-j-F-t}, we have 
\[
\left(\E_{\P^*_{x, y, T}}\!\!\sup\limits_{0\leq \overline{t}\leq T}\left\|\big[\wt{D^{(1)} F_{\overline{t}}}( {\rm{u}}_{0}, {w})\big]\right\|^q\right)^2\!\leq \underline{c}_{l}(2q)e^{c_l(2q)T}\E_{\P^*_{y, x, T}}e^{2\wt{c}(q)\int_{0}^{T}\|\nabla\ln p(T-t, {\rm x}_{t}, y)\|\ dt}. 
\]
This shows  (\ref{wt-DF-j-cond})  for the $l=1$ case by Proposition \ref{cond-nabla-ln-p}.

Using  the decomposition of $[D^{(2)}F_t]$ and the first step conclusion, the $l=2$ case of (\ref{wt-DF-j-cond}) can be reduced to the estimation of 
\[({\bf V}):=\E_{\P^*_{x, y, T}}\sup_{0\leq t\leq T}\left\|(\theta, \varpi)_{{\rm u}_t}(\mathcal{V}_t)\right\|^{2q}.\]
Let ${\rm{A}}, {\rm{B}}, {\rm{C}}$ and  $\wt{\rm{A}}, \wt{\rm{B}}, \wt{\rm{C}}$ be given in the proof of Proposition \ref{est-norm-D-j-F-t}. Then 
\begin{align*}
3^{1-2q}({\bf V}) \leq\ &  \E_{\P^*_{x, y, T}}\sup_{0\leq t\leq T}\left\|\wt {\rm{A}}_t({w})\right\|^{2q}+\E_{\P^*_{x, y, T}}\sup_{0\leq t\leq T}\left\|\wt {\rm{B}}_t({w})\right\|^{2q}+\E_{\P^*_{x, y, T}}\sup_{0\leq t\leq T}\left\|\wt {\rm{C}}_t({w})\right\|^{2q}\\
=:\ & ({\bf{A}})+({\bf{B}})+({\bf{C}}). 
\end{align*}
For $({\bf{B}})$,  following its   non-conditional estimation in the proof of Proposition  \ref{est-norm-D-j-F-t}, we obtain 
\begin{align*}
({\bf{B}})\leq & (CT)^{2q}\left(\E_{\P^*_{x, y, T}} \sup\limits_{0\leq \underline{t}<\overline{t}\leq T}\left\|\big[ \wt{D^{(1)} F_{\underline{t}, \overline{t}}}( {\rm{u}}_{\underline{t}}, {w})\big] \right\|^{4q} \E_{\P^*_{x, y, T}} \sup\limits_{0\leq \underline{t}<\overline{t}\leq T}\left\|\big[ \wt{D^{(1)} F_{\underline{t}, \overline{t}}}( {\rm{u}}_{\underline{t}}, {w})\big] \right\|^{8q} \right)^{\frac{1}{2}}\\
\leq & (CT)^{2q}\sqrt{\underline{c}'_{1}(4q)\underline{c}'_{1}(8q)}e^{\frac{1}{2}(c'_1(4q)+c'_1(8q))(1+d(x, y))},
\end{align*}
where $C$ depends on the norm bounds of   $\{{\bf \nabla}^{l}H\}_{l\leq 3}, \{{\bf \nabla}^{l}R\}_{l\leq 2}$ and $\underline{c}'_{1}, c'_1$ are from (\ref{wt-DF-j-cond}) for $l=1$, and the constants can be chosen such that the same bound is  valid for $({\bf{C}})$.  For $({\bf{A}})$, we use  (\ref{B-b-relation-x}). Let 
\begin{align*}
{\ov{\rm{A}}}_{{t}}^1(w)&:= \int_{0}^{t}\big[\wt{D^{(1)}F_{\tau, t}}({\rm u}_{\tau}, {w})\big] {\rm{A}}_\tau({w})\ db_{\tau}({w}),\\
{\ov{\rm{A}}}_{{t}}^2(w)&:= \int_{0}^{t}\big[\wt{D^{(1)}F_{\tau, t}}({\rm u}_{\tau}, {w})\big] {\rm{A}}_\tau({w})\  2(\lf {\rm{u}_{\tau}}\rc^{\l})^{-1}\nabla^{\l}\ln p(T-\tau, \lf {\rm{x}}_{\tau}\rc^{\l}, y)\ d\tau. 
\end{align*}
Then, 
\[
2^{1-2q}({\bf{A}})\leq {\E}_{\P^*_{x, y, T}}\sup\limits_{0\leq {t}\leq T}\left\|{\ov{\rm{A}}}_{{t}}^1(w)\right\|^{2q}+{\E}_{\P^*_{x, y, T}}\sup\limits_{0\leq {t}\leq T}\left\|{\ov{\rm{A}}}_{{t}}^2(w)\right\|^{2q}=: ({\bf{A}})_1+({\bf{A}})_2.
\]
Using the Brownian character of $b_\tau$ with respect to $\P_{x, y, T}$,  we can estimate $({\bf{A}})_1$ as in the non-conditional case  using  Doob's inequality of sub-martingales and Burkholder's inequality. This gives 
\begin{align*}
({\bf{A}})_1& \leq {\mathtt C}(2q){\mathtt C}_1(2q)T^{q}\left({\E}_{\P^*_{x, y, T}}\sup\limits_{0\leq \tau<{t} \leq T}\left\|\big[\wt{D^{(1)} F_{\tau, {t}}}( {\rm u}_{\tau}, {w})\big]\right\|^{4q}\right)^{\frac{1}{2}}\cdot \left({\E}_{\P^*_{x, y, T}}\sup\limits_{0\leq \tau \leq T}\left\|{\rm{A}}_{\tau}\right\|^{4q}\right)^{\frac{1}{2}},\end{align*}
where ${\mathtt C}, {\mathtt C}_1$ are as in the proof of Proposition \ref{est-norm-D-j-F-t}. 
Using (\ref{DF2-thevarA}), we compute that 
\begin{align}\label{P-xyt-A}
{\E}_{\P^*_{x, y, T}}\sup\limits_{0\leq \tau \leq T}\left\|{\rm{A}}_{\tau}\right\|^{jq} \leq &(C'_{{\rm A}})^{jq}  {\E}_{\P^*_{x, y, T}}\sup\limits_{0\leq \tau \leq T}\left\|\big[\wt{D^{(1)} F_{0, \tau}}( {\rm u}_0, {w})\big]\right\|^{2jq},\ \forall j\in \Bbb N, 
\end{align}
where $C'_{{\rm A}}$ depends on the norm bounds of   $\{{\bf \nabla}^{l}H\}_{l\leq 2}, \{{\bf \nabla}^{l}R\}_{l\leq 1}$.  Hence, by (\ref{wt-DF-j-cond}) for $l=1$, 
\begin{align*}
({\bf{A}})_1\leq &  C'(q)(C'_{{\rm A}}\sqrt{T})^{2q}\sqrt{c'_1(4q) c'_1(8q)}e^{\frac{1}{2}(c'_1(4q)+c'_1(8q))(1+d_{\wt{g}}(x, y))}. 
\end{align*}
For $({\bf{A}})_2$, we have 
\begin{align*}
\big(({\bf{A}})_2\big)^3 \leq&{\E}_{\P^*_{x, y, T}}\sup\limits_{0\leq \tau<{t} \leq T}\left\|\big[\wt{D^{(1)} F_{\tau, {t}}}( {\rm u}_{\tau}, {w})\big]\right\|^{6q} \cdot {\E}_{\P^*_{x, y, T}}\sup\limits_{0\leq \tau \leq T}\left\|{\rm{A}}_{\tau}\right\|^{6q}\\
& \cdot {\E}_{\P^*_{x, y, T}}\left|\int_{0}^{T}\left\|\nabla \ln p(T-\tau, {\rm{x}}_{\tau}, y)\right\|\ d\tau\right|^{6q}. 
\end{align*}
Note that 
\[
{\E}_{\P^*_{x, y, T}}\left|\int_{0}^{T}\left\|\nabla \ln p(T-\tau, {\rm{x}}_{\tau}, y)\right\|\ d\tau\right|^{6q}\leq {\E}_{\P^*_{x, y, T}}e^{6q\int_{0}^{T}\|\nabla\ln p(T-t, {\rm x}_{t}, y)\|\ dt}. 
\]
So by Proposition \ref{cond-nabla-ln-p}, (\ref{P-xyt-A}) and (\ref{wt-DF-j-cond}) for $l=1$, we compute that 
\begin{align*}
({\bf{A}})_2\leq (C'_{{\rm A}})^{2q}\sqrt[3]{\underline{c}'_1(6q)\underline{c}'_1(12q)\underline{c}(6q)}e^{\frac{1}{3}({c}'_1(6q)+{c}'_1(12q)+{c}(6q))(1+d(x, y))}. 
\end{align*}
Hence  $({\bf V})$ has the same type of  bound  as in (\ref{wt-DF-j-cond}) for  $l=2$ as claimed.

Assume we have shown (\ref{wt-DF-j-cond})  for $l\leq l_0-1\leq k-3$. Using the induction assumption and  (\ref{ind-der-F-e-t-ran}), we can reduce the estimation of (\ref{wt-DF-j}) at $l=l_0$ to the conditional estimation of (\ref{Dj-Ft-u}), which can be  done exactly as in the $l=2$ step.  
\end{proof}

\subsection{Regularity of the stochastic analogue of the geodesic flow}\label{sec-4-5}
 Finally, we  employ the  SDE theory in the previous subsections of this section   to  discuss   the regularity of the Brownian companion process ${\rm u}$ with respect to metric changes.

Let $\lambda\in
(-1, 1)\mapsto g^{\lambda}\in \mathcal{M}^k(M)$ be a $C^k$ curve.  Each lifted metric $\wt{g}^{\l}$  in $\M$ determines a horizontal space $H^{\lambda}T\mathcal{F}(\M)$ of the frame bundle space.  For any $u\in \mathcal{F}(\M)$, let $H^{\l}(u, e_i)$, $i=1, \cdots, m$, be the vector in $H^{\lambda}_uT\mathcal{F}(\M)$ which projects to  $ue_i$. Since $g^{\lambda}\in \mathcal{M}^k(M)$, the map $u\mapsto H^{\l}(u, e_i), u\in \mathcal{F}(\M), $ is $C^{k-1}$ bounded. Hence the SDE
\begin{equation}\label{FM-BM-SDE}
d\lfloor {\rm u}_t\rceil^{\l}=\sum_{i=1}^{m}H^{\l}(\lfloor {\rm u}_t\rceil^{\l}, e_i)\circ dB_t^i({w})
\end{equation} 
is solvable in $\mathcal{F}(\M)$ for any initial  $\lfloor {\rm u}_0\rceil^{\l}\in \mathcal{F}_x(\M)$, $x\in \M$. In particular, if $\lfloor {\rm u}_0\rceil^{\l}\in \mathcal{O}^{\wt{g}^{\l}}_{x}(\M)$, $\lfloor {\rm u}_t\rceil^{\l}$ remains in $\mathcal{O}^{\wt{g}^{\l}}(\M)$ and its projection to $\M$ gives the stochastic process of the $\wt{g}^{\l}$-Brownian motion starting from $x$. Let $\lf F_t \rc^{\l}:\ \lfloor {\rm u}_0\rceil^{\l}\mapsto \lfloor {\rm u}_t\rceil^{\l} $ denote the flow map associated to (\ref{FM-BM-SDE}). Let $[D^{(l)}\lf F_t\rc^{\l}(\cdot, {w})]$, $1\leq l\leq k-2$,  be the $l$-th tangent map of $\lf F_t \rc^{\l}$ and denote by $[\wt{D^{(l)}\lf F_t\rc^{\l}}(\cdot, {w})]$ its pull back map in $T^l\mathcal{F}(\Bbb R^m)$ via the map $(\theta, \varpi)$. They have the following regularity in $\l$ by applying Proposition  \ref{El-Ku}.

\begin{lem}\label{u-lambda-differential-1}
Let $\lambda\in
(-1, 1)\mapsto g^{\lambda}\in \mathcal{M}^k(M)$ $(k\geq 3)$ be a $C^k$ curve. Assume  $H^{\l}(\cdot, e_i)$ has bounded norms (independent of $\l$) for the  covariant derivatives up to the $(k-1)$-th order with respect to the reference metric $\widetilde{g}^0$. 
\begin{itemize}
\item[i)]  Let $\l\mapsto \lfloor {\rm{u}}_0\rceil^{\l}$ be a $C^{k-2}$ curve in $\mathcal{F}(\M)$ and let $\{\lfloor {\rm{u}}_t\rceil^{\l}\}_{t\in \Bbb R_+}$ be the solution to (\ref{FM-BM-SDE}) with initial value $\lfloor {\rm{u}}_0\rceil^{\l}$.  Then there is a version of the solution to (\ref{FM-BM-SDE}) such that almost surely, $\lfloor {\rm{u}}_t\rceil^{\l}({w})$ is $C^{k-2}$  in $\lambda$ for any $t\in \Bbb R_+$.
\item[ii)]For each $l$, $1\leq l\leq k-2$, the tangent map $[D^{(l)}\lf F_t\rc^{\l}(\cdot, {w})]$ is $C^{k-2-l}$ in $\l$. In particular, for any ${\rm v}\in T^{l}\mathcal{F}(\M)$, the map $
\l\mapsto \left[D^{(l)}\lf F_t\rc^{\l}(\cdot, {w})\right]{\rm v}$  is $C^{k-2-l}$. 
\end{itemize}
\end{lem}
\begin{proof}Consider the stochastic process  $\{\wt{{\rm{u}}}_t\}_{t\in \Bbb R_+}$ on $\mathcal{F}(\M)\times (-1, 1)$ with 
\begin{equation}\label{u-t-lambda}d\wt{{\rm{u}}}_t=\sum_{i=1}^{m}{\wt{H}_i}(\wt{{\rm{u}}}_t)\circ dB_t^i({w}), \ \mbox{where}\  {\wt{H}}_i=(H^{\l}(\cdot, e_i), 0).
\end{equation}
It has the solution   $\wt{{\rm{u}}}_t=(\lfloor {\rm{u}}_t\rceil^{\l}, \l)$ for $\wt{{\rm{u}}}_0=(\lfloor {\rm{u}}_0\rceil^{\l}, \l)$, where $\lfloor {\rm{u}}_t\rceil^{\l}$ is the solution of (\ref{FM-BM-SDE}) with initial value $\lfloor {\rm{u}}_0\rceil^{\l}$.  Since  (\ref{u-t-lambda}) is a $C^{k-1}$ SDS on $\mathcal{F}(\M)\times (-1, 1)$, we have by   Proposition  \ref{El-Ku} ii) that for almost all ${w}$, the mapping $\wt{{\rm{u}}}_0({w})\to \wt{{\rm{u}}}_t(w)$ is $C^{k-2}$. Consequently, for any $C^{k-2}$ curve $\l\mapsto \lfloor {\rm{u}}_0\rceil^{\l}$, $\lfloor {\rm{u}}_t\rceil^{\l}({w})$ is $C^{k-2}$  in $\lambda$ for almost all ${w}$. 

For each $l$, $1\leq l\leq k-2$,  the SDE of $[D^{(l)}\lf F_t\rc^{\l}(\cdot, {w})]$ was given in  Section \ref{BMM} and it forms a  $C^{k-1-l}$ SDS  on $T\mathcal{F}(\M)$. As in Lemma \ref{u-lambda-differential-1}, we can treat the one parameter family  SDEs of 
$[D^{(l)}\lf F_t\rc^{\l}(\cdot, {w})]$ as a $C^{k-1-l}$ SDS on $T\mathcal{F}(\M)\times (-1, 1)$ when $\l\mapsto g^{\l}$ is $C^k$ in $\mathcal{M}^k(M)$. So Proposition  \ref{El-Ku} applies and shows ii).
\end{proof}

For  $0\leq \underline{t}<\overline{t}<\infty$,  let  $\lf F_{\underline{t}, \overline{t}} \rc^{\l}:\ \lfloor {\rm u}_{\underline{t}}\rceil^{\l}\mapsto \lfloor {\rm u}_{\overline{t}}\rceil^{\l} $ denote the flow map associated to (\ref{FM-BM-SDE}) and let $[D^{(l)}\lf F_{\underline{t}, \overline{t}}\rc^{\l}(\cdot, {w})]$, $l\leq k-2$,  be its  $l$-th tangent map. As a corollary of  the cocycle property of $\lf F_{\underline{t}, \overline{t}} \rc^{\l}$ and Lemma \ref{u-lambda-differential-1}, $\big[D^{(l)}\lf F_{\underline{t}, \overline{t}}\rc^{\l}(\cdot, {w})\big]$ is $C^{k-2-l}$ differentiable in $\l$ and we denote its $j$-th differential by $\big([D^{(l)}\lf F_{\underline{t}, \overline{t}}\rc^{\l}(\cdot, {w})]\big)^{(j)}_{\l}$ for $j\leq k-2-l$. Let  $\lfloor {\rm{u}}_t\rceil^{\l}$ be as in Lemma \ref{u-lambda-differential-1} and let $(\lfloor {\rm u}_t\rceil^{\l})^{(j)}_\l$, $j\leq k-2$,  be  its $j$-th differential  in $\l$.  We identify
\[
(\lfloor {\rm u}_{\overline{t}}\rceil^{\l})^{(j)}_\l=\big([D^{(0)}\lf F_{\underline{t}, \overline{t}}\rc^{\l}(\lf {\rm{u}}_{\underline{t}}\rc^{\l}, {w})]\big)^{(j)}_{\l}. 
\]
In the following, we show the  $L^q$-norm  bounds  in Propositions  \ref{est-norm-D-j-F-t}  and \ref{est-norm-D-j-F-t-cond} are also valid for  $\big([D^{(l)}\lf F_{\underline{t}, \overline{t}}\rc^{\l}(\cdot, {w})]\big)^{(j)}_{\l}$ by a detailed  analysis of their SDEs.

Endow $\mathcal{F}(\M)\times (-1, 1)$ with the product metric $d_{\wt{g}^0}\times d_{(-1, 1)}$, where $d_{\wt{g}^{0}}$ is the induced metric of $d_{\wt{g}^0}$ in $\mathcal{F}(\M)$ and $d_{(-1, 1)}$ is canonical. Let $\nabla$  be the $\wt{g}^{0}$ Levi-Civita connection and $\theta, \varpi$ be the associated canonical form and curvature form.  Let $(H^{\l})^{(j)}_{\l}(u, \cdot)$,  $j\leq k-2$,  be the $j$-th differential in $\l$ of the maps $H^{\l}(u, \cdot)$. The SDEs of $\big([D^{(l)}\lf F_{\underline{t}, \overline{t}}\rc^{\l}(\cdot, {w})]\big)^{(j)}_{\l}$ can be formulated by  using Proposition \ref{SDE-flow-regularity}. We state them as follows.

\begin{lem}\label{u-t-lam-tangent-map-1}Let  $\lfloor {\rm{u}}_t\rceil^{\l}$ be as in Lemma \ref{u-lambda-differential-1}.
\begin{itemize}
\item[i)]The Stratonovich SDE of $(\lfloor {\rm u}_t\rceil^{\l})^{(1)}_{\l}$ in $T\mathcal{F}(\M)$ is
\[d(\lfloor {\rm u}_t\rceil^{\l})^{(1)}_{\l}=\nabla((\lfloor {\rm u}_t\rceil^{\l})^{(1)}_{\l})H^{\l}(\lfloor {\rm u}_t\rceil^{\l}, \circ dB_t)+ (H^{\l})^{(1)}_{\l}(\lfloor {\rm u}_t\rceil^{\l}, \circ dB_t).\] 
\item[ii)] The Stratonovich SDE of $(\theta, \varpi)_{\lfloor{\rm u}_t\rceil^{\l}}\big((\lfloor {\rm u}_t\rceil^{\l})^{(1)}_{\l}\big)$ in $T\mathcal{F}(\Bbb R^m)$ is
\begin{align*}
\begin{array}{ll}
&d\big(\theta(\lfloor {\rm u}_t\rceil^{\l})^{(1)}_{\l}\big)\ =d\theta\left(H^{\l}(\lfloor {\rm u}_t\rceil^{\l}, \circ dB_t), (\lfloor {\rm u}_t\rceil^{\l})^{(1)}_{\l}\right),\\
&d\big(\varpi(\lfloor {\rm u}_t\rceil^{\l})^{(1)}_{\l}\big)=d\varpi\left(H^{\l}(\lfloor {\rm u}_t\rceil^{\l}, \circ dB_t), (\lfloor {\rm u}_t\rceil^{\l})^{(1)}_{\l}\right)+\nabla((\lfloor {\rm u}_t\rceil^{\l})^{(1)}_{\l})\left(\varpi\big(H^{\l}(\lf {\rm u}_t\rc^{\l}, \circ dB_t)\big)\right)\\
&\ \ \ \ \ \ \ \ \ \ \ \ \ \ \ \ \ \ \ \ \ +\varpi\left((H^{\l})^{(1)}_{\l}(\lf {\rm u}_t\rc^{\l}, \circ dB_t)\right).
\end{array}
\end{align*}
\item[iii)] The It\^{o} SDE of $(\theta, \varpi)_{\lfloor{\rm u}_t\rceil^{\l}}\big((\lfloor {\rm u}_t\rceil^{\l})^{(1)}_{\l}\big)$ in $T\mathcal{F}(\Bbb R^m)$ is
\begin{align*}
\begin{array}{ll}
d\big(\theta(\lfloor {\rm u}_t\rceil^{\l})^{(1)}_{\l}\big)=&d\theta\left(H^{\l}(\lfloor {\rm u}_t\rceil^{\l},  dB_t), (\lfloor {\rm u}_t\rceil^{\l})^{(1)}_{\l}\right)\\
&+\big(\nabla(H^{\l}(\lfloor {\rm u}_t\rceil^{\l},  e_i))d\theta\big)\left(H^{\l}(\lfloor {\rm u}_t\rceil^{\l}, e_i), (\lfloor {\rm u}_t\rceil^{\l})^{(1)}_{\l}\right)\ dt\\
&+d\theta\left(\nabla(H^{\l}(\lfloor {\rm u}_t\rceil^{\l},  e_i))H^{\l}(\lfloor {\rm u}_t\rceil^{\l},  e_i), (\lfloor {\rm u}_t\rceil^{\l})^{(1)}_{\l}\right) \ dt\\
&+d\theta\left(H^{\l}(\lfloor {\rm u}_t\rceil^{\l},  e_i), \nabla((\lfloor {\rm u}_t\rceil^{\l})^{(1)}_{\l})H^{\l}(\lfloor {\rm u}_t\rceil^{\l}, e_i)+ (H^{\l})^{(1)}_{\l}(\lfloor {\rm u}_t\rceil^{\l}, e_i)\right) \ dt,
\end{array}
\end{align*}
\vspace*{-0.3cm}
\begin{align*}
\begin{array}{ll}
d\big(\varpi(\lfloor {\rm u}_t\rceil^{\l})^{(1)}_{\l}\big)=&d\varpi\left(H^{\l}(\lfloor {\rm u}_t\rceil^{\l},  dB_t), (\lfloor {\rm u}_t\rceil^{\l})^{(1)}_{\l}\right)\\
&+\big(\nabla(H^{\l}(\lfloor {\rm u}_t\rceil^{\l},  e_i))d\varpi\big)\left(H^{\l}(\lfloor {\rm u}_t\rceil^{\l}, e_i), (\lfloor {\rm u}_t\rceil^{\l})^{(1)}_{\l}\right)\ dt\\
&+d\varpi\left(\nabla(H^{\l}(\lfloor {\rm u}_t\rceil^{\l},  e_i))H^{\l}(\lfloor {\rm u}_t\rceil^{\l},  e_i), (\lfloor {\rm u}_t\rceil^{\l})^{(1)}_{\l}\right) \ dt\\
&+d\varpi\left(H^{\l}(\lfloor {\rm u}_t\rceil^{\l},  e_i), \nabla((\lfloor {\rm u}_t\rceil^{\l})^{(1)}_{\l})H^{\l}(\lfloor {\rm u}_t\rceil^{\l}, e_i)+ (H^{\l})^{(1)}_{\l}(\lfloor {\rm u}_t\rceil^{\l}, e_i)\right) \ dt\\
&+\nabla((\lfloor {\rm u}_t\rceil^{\l})^{(1)}_{\l})\left(\varpi\big(H^{\l}(\lf {\rm u}_t\rc^{\l}, dB_t)\big)\right)+\varpi\left((H^{\l})^{(1)}_{\l}(\lf {\rm u}_t\rc^{\l}, dB_t)\right)\\
&+\nabla(H^{\l}(\lfloor {\rm u}_t\rceil^{\l},  e_i))\left(\nabla((\lfloor {\rm u}_t\rceil^{\l})^{(1)}_{\l})\left(\varpi\big(H^{\l}(\lf {\rm u}_t\rc^{\l}, e_i)\big)\right)\right)\ dt\\
&+ \nabla(H^{\l}(\lfloor {\rm u}_t\rceil^{\l},  e_i))\left(\varpi\left((H^{\l})^{(1)}_{\l}(\lf {\rm u}_t\rc^{\l}, e_i)\right)\right)\ dt. 
\end{array}
\end{align*}
\end{itemize}
\end{lem}

Note that  ${\rm Ker}(\theta)=VT\mathcal{F}(\M),{\rm Ker}(\varpi)=HT\mathcal{F}(\M)$ and for any ${\rm v}^1, {\rm v}^2\in HT\mathcal{F}(\M), {\rm v}^3\in VT\mathcal{F}(\M)$, the bracket $[\cdot, \cdot]$ satisfies the property  (cf. \cite[Lemma 5.5.1]{Hs})
\begin{align*}
[{\rm v}^1, {\rm v}^2]\in VT\mathcal{F}(\M), \ [{\rm v}^1, {\rm v}^3]\in HT\mathcal{F}(\M). 
\end{align*}
Using  these facts,  we can simplify the SDEs of  $(\theta, \varpi)\big((\lfloor {\rm u}_t\rceil^{\l})^{(1)}_{\l}\big)$ at $\l=0$.  

\begin{cor}\label{u-t-lam-0-tangent-map-1}Let  $\lfloor {\rm{u}}_t\rceil^{\l}$ be as in Lemma \ref{u-lambda-differential-1}.
\begin{itemize}
\item[i)]  The Stratonovich SDE of $(\theta, \varpi)_{\lfloor{\rm u}_t\rceil^0}\big((\lfloor {\rm u}_t\rceil^{\l})^{(1)}_0\big)$ on $T\mathcal{F}(\Bbb R^m)$ is
\begin{align*}
d\big(\theta(\lfloor {\rm u}_t\rceil^{\l})^{(1)}_0\big)& =\varpi(\lfloor {\rm u}_t\rceil^{\l})^{(1)}_0)\circ dB_t, \\
d\big(\varpi(\lfloor {\rm u}_t\rceil^{\l})^{(1)}_0\big)& ={(\lfloor{\rm u}_t\rceil^0)}^{-1}R\big(\lfloor{\rm u}_t\rceil^0\circ dB_t, \theta(\lfloor {\rm u}_t\rceil^{\l})^{(1)}_0)\big)\lfloor{\rm u}_t\rceil^0+\varpi\big((H^{\l})^{(1)}_0(\lfloor {\rm u}_t\rceil^{0}, \circ dB_t)\big). 
\end{align*}
\item[ii)]  The It\^{o} SDE of $(\theta, \varpi)_{\lfloor{\rm u}_t\rceil^0}\big((\lfloor {\rm u}_t\rceil^{\l})^{(1)}_0\big)$ on $T\mathcal{F}(\Bbb R^m)$ is
\begin{align*}
d\big(\theta(\lfloor {\rm u}_t\rceil^{\l})^{(1)}_0\big) & = \varpi(\lfloor {\rm u}_t\rceil^{\l})^{(1)}_0)\ dB_t+{\rm Ric}({\rm u}_{t}\theta(\lfloor {\rm u}_t\rceil^{\l})^{(1)}_0)\ dt+\varpi\big((H^{\l})^{(1)}_0(\lfloor {\rm u}_t\rceil^{0}, e_i)\big)e_i\ dt,  \\
d\big(\varpi(\lfloor {\rm u}_t\rceil^{\l})^{(1)}_0\big)& ={(\lfloor{\rm u}_t\rceil^0)}^{-1}R\big(\lfloor{\rm u}_t\rceil^0 dB_t, \lfloor {\rm u}_{t}\rceil^{0}\theta(\lfloor {\rm u}_t\rceil^{\l})^{(1)}_0)\big)\lfloor{\rm u}_t\rceil^0+\varpi\big((H^{\l})^{(1)}_0(\lfloor {\rm u}_t\rceil^{0},  dB_t)\big)\\
&\ \ \ \ +(\lfloor{\rm u}_{t}\rceil^0)^{-1}R\left(\lfloor {\rm u}_{t}\rceil^{0}e_i, \lfloor{\rm u}_{t}\rceil^0\varpi(\lfloor {\rm u}_t\rceil^{\l})^{(1)}_0)e_i\right)  \lfloor{\rm u}_{t}\rceil^0\ dt\\
&\ \ \ \ +(\lfloor{\rm u}_{t}\rceil^0)^{-1}\big(\nabla (\lfloor{\rm u}_{t}\rceil^0e_i) R\big)\left(\lfloor{\rm u}_{t}\rceil^0 e_i, \lfloor{\rm u}_{t}\rceil^0\theta(\lfloor {\rm u}_t\rceil^{\l})^{(1)}_0)\right)\lfloor{\rm u}_{t}\rceil^0\ dt.
\end{align*}
\end{itemize}
\end{cor}

Using  Corollary \ref{u-t-lam-0-tangent-map-1} and It\^{o}'s formula, we can express  $(\lfloor {\rm u}_t\rceil^{\l})^{(1)}_0$  using $D^{(1)}\lf F_{\tau, t}\rc^{0}(\cdot, {w})$ by Duhamel's principle. This can be verified as in Corollary \ref{Duha-mathcal-V-t}. We omit the proof. 

\begin{cor}\label{cor-2--u-t-lam-tangent-map-1}Let  $\lfloor {\rm{u}}_t\rceil^{\l}$ be as in Lemma \ref{u-lambda-differential-1}.
\begin{itemize}
\item[i)] On $T\mathcal{F}(\M)$, 
\begin{align*}(\lfloor {\rm u}_t\rceil^{\l})^{(1)}_0=&\left[D^{(1)}\lfloor{F}_t\rceil^{0}(\lfloor {\rm u}_0\rceil^{0}, {w})\right](\lfloor {\rm u}_0\rceil^{\l})^{(1)}_{0}+{\rm V}_c((\lfloor {\rm u}_t\rceil^{\l})^{(1)}_0),
\end{align*}
where 
\begin{align*}
{\rm V}_c((\lfloor {\rm u}_t\rceil^{\l})^{(1)}_0):=\int_{0}^{t} \left[D^{(1)}\lfloor {F}_{\tau, t}\rceil^{0}(\lfloor {\rm u}_\tau \rceil^{0}, {w})\right](H^{\l})^{(1)}_{0}(\lfloor {\rm u}_\tau \rceil^{0}, \circ dB_{\tau}({w})). 
\end{align*}
\item[ii)] On $T\mathcal{F}(\Bbb R^m)$, the It\^{o} form of $\wt{(\lfloor {\rm u}_t\rceil^{\l})^{(1)}_0}:=(\theta, \varpi)_{\lfloor{\rm u}_t\rceil^0}\left((\lfloor {\rm u}_t\rceil^{\l})^{(1)}_0\right)$ is given by 
\begin{align*}
\wt{(\lfloor {\rm u}_t\rceil^{\l})^{(1)}_0}
&=\left[\wt{D^{(1)}\lfloor{F}_t\rceil^{0}}(\lfloor {\rm u}_0\rceil^{\l}, {w})\right]\wt{(\lfloor {\rm u}_0\rceil^{\l})^{(1)}_0}+\wt{{\rm V}_c}((\lfloor {\rm u}_t\rceil^{\l})^{(1)}_0),
\end{align*}
where 
\begin{align*}
&\wt{{\rm V}_c}((\lfloor {\rm u}_t\rceil^{\l})^{(1)}_0):=\int_{0}^{t} \left[\wt{D^{(1)}\lfloor {F}_{\tau, t}\rceil^{0}}(\lfloor {\rm u}_\tau \rceil^{0}, {w})\right]\\
&\ \ \ \ \ \ \ \ \ \ \ \ \ \ \ \ \ \ \ \ \ \ \ \ \ \    \left(\varpi\big((H^{\l})^{(1)}_0(\lfloor {\rm u}_{\tau}\rceil^{0}, e_i)\big)e_i\ d\tau, \varpi\big((H^{\l})^{(1)}_0(\lfloor {\rm u}_{\tau}\rceil^{0},  dB_{\tau})\big)\right).
\end{align*}
\end{itemize}
\end{cor}
To  describe the second order differential of $\lfloor {\rm u}_t\rceil^{\l}$ in $\l$, we use the  horizontal/vertical sum decomposition of $TT{\mathcal{F}(\M)}$ of $\wt{g}^0$. By Lemma \ref{u-t-lam-tangent-map-1}, it remains to find the SDEs of 
\[
(\lfloor {\rm u}_t\rceil^{\l})^{(2)}_{\l}:=\frac{D}{d\l}\left((\lfloor {\rm u}_t\rceil^{\l})^{(1)}_{\l}\right)=\nabla((\lfloor {\rm u}_t\rceil^{\l})^{(1)}_0)(\lfloor {\rm u}_t\rceil^{\l})^{(1)}_0. 
\]
\begin{lem}\label{u-t-lam-tangent-map-2}Let $\lambda\in
(-1, 1)\mapsto g^{\lambda}\in \mathcal{M}^k(M)$  be a $C^k$ curve with $k\geq 4$. Let $\l\mapsto \lfloor {\rm u}_0\rceil^{\l}$ be $C^{k-2}$ and let $\lfloor {\rm u}_t\rceil^{\l}$ be as in  Lemma \ref{u-lambda-differential-1} with $(\lfloor {\rm u}_t\rceil^{\l})^{(1)}_\l$, $(\lfloor {\rm u}_t\rceil^{\l})^{(2)}_{\l}$ defined as above.
\begin{itemize}
\item[i)] The Stratonovich SDE of $(\lfloor {\rm u}_t\rceil^{\l})^{(2)}_{\l}$ on $T\mathcal{F}(\M)$ is
\begin{align*}
d\big((\lfloor {\rm u}_t\rceil^{\l})^{(2)}_{\l}\big)=&\nabla\big((\lfloor {\rm u}_t\rceil^{\l})^{(2)}_{\l}\big)H^{\l}(\lf {\rm u}_{t}\rc^{\l}, \circ dB_t)+\nabla^{(2)}\big((\lfloor {\rm u}_t\rceil^{\l})^{(1)}_{\l}, (\lfloor {\rm u}_t\rceil^{\l})^{(1)}_{\l}\big)H^{\l}(\lf {\rm u}_{t}\rc^{\l}, \circ dB_t)\\
&+R\big(H^{\l}(\lf {\rm u}_{t}\rc^{\l}, \circ dB_t), (\lfloor {\rm u}_t\rceil^{\l})^{(1)}_{\l}\big)(\lfloor {\rm u}_t\rceil^{\l})^{(1)}_{\l}\\
&+2\nabla\big((\lfloor {\rm u}_t\rceil^{\l})^{(1)}_{\l}\big)\big(H^{\l}\big)^{(1)}_{\l}(\lf {\rm u}_{t}\rc^{\l}, \circ dB_t)+ \big(H^{\l}\big)^{(2)}_{\l}(\lf {\rm u}_{t}\rc^{\l}, \circ dB_t).
\end{align*}
\item[ii)] The Stratonovich SDE of $(\theta, \varpi)_{\lfloor{\rm u}_t\rceil^{\l}}\big((\lfloor {\rm u}_t\rceil^{\l})^{(2)}_{\l}\big)$ on $T\mathcal{F}(\Bbb R^m)$ is
\begin{align*}
\begin{array}{ll}
&d\left(\big(\theta, \varpi\big)\big((\lfloor {\rm u}_t\rceil^{\l})^{(2)}_{\l}\big)\right)\\
&=d\big(\theta, \varpi\big)\left(H^{\l}(\lfloor {\rm u}_t\rceil^{\l}, \circ dB_t), (\lfloor {\rm u}_t\rceil^{\l})^{(2)}_{\l}\right)+\nabla((\lfloor {\rm u}_t\rceil^{\l})^{(2)}_{\l})\left(\big(\theta, \varpi\big)\big(H^{\l}(\lf {\rm u}_t\rc^{\l}, \circ dB_t)\big)\right)\\
&\ \  +\big(\theta, \varpi\big)\left(\nabla^{(2)}\big((\lfloor {\rm u}_t\rceil^{\l})^{(1)}_{\l}, (\lfloor {\rm u}_t\rceil^{\l})^{(1)}_{\l}\big)H^{\l}(\lf {\rm u}_{t}\rc^{\l}, \circ dB_t)\right)\\
&\ \  +\big(\theta, \varpi\big)\left(R\big(H^{\l}(\lf {\rm u}_{t}\rc^{\l}, \circ dB_t), (\lfloor {\rm u}_t\rceil^{\l})^{(1)}_{\l}\big)(\lfloor {\rm u}_t\rceil^{\l})^{(1)}_{\l}\right)\\
&\ \  +\big(\theta, \varpi\big)\left(2\nabla\big((\lfloor {\rm u}_t\rceil^{\l})^{(1)}_{\l}\big)\big(H^{\l}\big)^{(1)}_{\l}(\lf {\rm u}_{t}\rc^{\l}, \circ dB_t)+ \big(H^{\l}\big)^{(2)}_{\l}(\lf {\rm u}_{t}\rc^{\l}, \circ dB_t)\right).
\end{array}
\end{align*}
\item[iii)] The It\^{o} SDE of $(\theta, \varpi)_{\lfloor{\rm u}_t\rceil^{\l}}\big((\lfloor {\rm u}_t\rceil^{\l})^{(2)}_{\l}\big)$ on $T\mathcal{F}(\Bbb R^m)$ is
\begin{align*}
\begin{array}{ll}
&d\left(\big(\theta, \varpi\big)\big((\lfloor {\rm u}_t\rceil^{\l})^{(2)}_{\l}\big)\right)\\
&=d\big(\theta, \varpi\big)\left(H^{\l}(\lfloor {\rm u}_t\rceil^{\l}, dB_t), (\lfloor {\rm u}_t\rceil^{\l})^{(2)}_{\l}\right)+\nabla((\lfloor {\rm u}_t\rceil^{\l})^{(2)}_{\l})\left(\big(\theta, \varpi\big)\big(H^{\l}(\lf {\rm u}_t\rc^{\l}, dB_t)\big)\right)\\
&\ \  +\big(\theta, \varpi\big)\left(\nabla^{(2)}\big((\lfloor {\rm u}_t\rceil^{\l})^{(1)}_{\l}, (\lfloor {\rm u}_t\rceil^{\l})^{(1)}_{\l}\big)H^{\l}(\lf {\rm u}_{t}\rc^{\l}, dB_t)\right)\\
&\ \  +\big(\theta, \varpi\big)\left(R\big(H^{\l}(\lf {\rm u}_{t}\rc^{\l}, dB_t), (\lfloor {\rm u}_t\rceil^{\l})^{(1)}_{\l}\big)(\lfloor {\rm u}_t\rceil^{\l})^{(1)}_{\l}\right)\\
&\ \  +\big(\theta, \varpi\big)\left(2\nabla\big((\lfloor {\rm u}_t\rceil^{\l})^{(1)}_{\l}\big)\big(H^{\l}\big)^{(1)}_{\l}(\lf {\rm u}_{t}\rc^{\l}, dB_t)+ \big(H^{\l}\big)^{(2)}_{\l}(\lf {\rm u}_{t}\rc^{\l},  dB_t)\right)\\
&\ \  +\nabla(H^{\l}(\lfloor {\rm u}_t\rceil^{\l},  e_i))\left\{d\big(\theta, \varpi\big)\left(H^{\l}(\lfloor {\rm u}_t\rceil^{\l}, e_i), (\lfloor {\rm u}_t\rceil^{\l})^{(2)}_{\l}\right)\right.\\
&\ \ \ \  \ \ \ \ \ \ \ \ \  \ \ \ \ \ \ \ \ \ \ \   +\nabla((\lfloor {\rm u}_t\rceil^{\l})^{(2)}_{\l})\left(\big(\theta, \varpi\big)\big(H^{\l}(\lf {\rm u}_t\rc^{\l}, e_i)\big)\right)\\
&\ \ \ \  \ \ \ \ \ \ \ \ \  \ \ \ \ \ \ \ \ \ \ \   +\big(\theta, \varpi\big)\left(\nabla^{(2)}\big((\lfloor {\rm u}_t\rceil^{\l})^{(1)}_{\l}, (\lfloor {\rm u}_t\rceil^{\l})^{(1)}_{\l}\big)H^{\l}(\lf {\rm u}_{t}\rc^{\l}, e_i)\right)\\
&\ \ \ \  \ \ \ \ \ \ \ \ \  \ \ \ \ \ \ \ \ \ \ \   +\big(\theta, \varpi\big)\left(R\big(H^{\l}(\lf {\rm u}_{t}\rc^{\l}, e_i), (\lfloor {\rm u}_t\rceil^{\l})^{(1)}_{\l}\big)(\lfloor {\rm u}_t\rceil^{\l})^{(1)}_{\l}\right)\\
&\ \ \ \  \ \ \ \ \ \ \ \ \  \ \ \ \ \ \ \ \ \ \ \   \left.+\big(\theta, \varpi\big)\left(2\nabla\big((\lfloor {\rm u}_t\rceil^{\l})^{(1)}_{\l}\big)\big(H^{\l}\big)^{(1)}_{\l}(\lf {\rm u}_{t}\rc^{\l}, e_i)+ \big(H^{\l}\big)^{(2)}_{\l}(\lf {\rm u}_{t}\rc^{\l},  e_i)\right)\right\} dt.
\end{array}
\end{align*}
\end{itemize}
\end{lem}

Again, we can simplify the SDEs in Lemma \ref{u-t-lam-tangent-map-2} at $\l=0$.  
\begin{cor}\label{cor-u-t-lam-tangent-map-2}We retain all the notations in Lemma \ref{u-t-lam-tangent-map-2}. 
\begin{itemize}
\item[i)]The Stratonovich SDE of $(\theta, \varpi)_{\lfloor{\rm u}_t\rceil^0}\big((\lfloor {\rm u}_t\rceil^{\l})^{(2)}_0\big)$ on $T\mathcal{F}(\Bbb R^m)$ is\begin{align*}
d\left(\big(\theta, \varpi\big)\big((\lfloor {\rm u}_t\rceil^{\l})^{(2)}_{0}\big)\right)
&=\left(\varpi(\lfloor {\rm u}_t\rceil^{\l})^{(2)}_0)\circ dB_t, {(\lfloor{\rm u}_t\rceil^0)}^{-1}R\big(\lfloor{\rm u}_t\rceil^0\circ dB_t, \theta(\lfloor {\rm u}_t\rceil^{\l})^{(2)}_0)\big)\lfloor{\rm u}_t\rceil^0\right)\\
&\  +\big(\theta, \varpi\big)\left(\nabla^{(2)}\big((\lfloor {\rm u}_t\rceil^{\l})^{(1)}_{0}, (\lfloor {\rm u}_t\rceil^{\l})^{(1)}_{0}\big)H^{0}(\lf {\rm u}_{t}\rc^{0}, \circ dB_t)\right)\\
&\  +\big(\theta, \varpi\big)\left(R\big(H^{0}(\lf {\rm u}_{t}\rc^{0}, \circ dB_t), (\lfloor {\rm u}_t\rceil^{\l})^{(1)}_{0}\big)(\lfloor {\rm u}_t\rceil^{\l})^{(1)}_{0}\right)\\
&\  +\big(\theta, \varpi\big)\left(2\nabla\big((\lfloor {\rm u}_t\rceil^{\l})^{(1)}_{0}\big)\big(H^{\l}\big)^{(1)}_0(\lf {\rm u}_{t}\rc^{0}, \circ dB_t)+\!\big(H^{\l}\big)^{(2)}_0(\lf {\rm u}_{t}\rc^{0}, \circ dB_t)\right).
\end{align*}
\item[ii)] The  It\^{o} SDE of $\theta_{\lfloor{\rm u}_t\rceil^0}\big((\lfloor {\rm u}_t\rceil^{\l})^{(2)}_0\big)$ on $T\mathcal{F}(\Bbb R^m)$ is
\begin{align*}
d\big(\theta((\lfloor {\rm u}_t\rceil^{\l})^{(2)}_0))=&\varpi((\lfloor {\rm u}_t\rceil^{\l})^{(2)}_0) d B_{t}+{\rm Ric}(\lf {\rm u}_{t}\rc^{0}\theta((\lfloor {\rm u}_t\rceil^{\l})^{(2)}_0)) dt\notag\\
&+\Phi_{\theta}((\lfloor {\rm u}_t\rceil^{\l})^{(1)}_0,(\lfloor {\rm u}_t\rceil^{\l})^{(1)}_0, dB_t, dt)+\Phi_{\theta}^{0, 2}((\lfloor {\rm u}_t\rceil^{\l})^{(1)}_0, dB_t, dt),
\end{align*}
where $\Phi_{\theta}(\cdot,\cdot, dB_t, dt)$ is given in (\ref{ito-2-theta-varpi-mathcal-V}) for $\lf {\rm u}_{t}\rc^{0}$ and \footnote{The upper script $^{0, 2}$ is to indicate  that   $\Phi_{\theta}^{0, 2}, \Phi_{\varpi}^{0, 2}$ are associated with $\big([D^{(0)}\lf F_t(\cdot, {w})]\big)^{(2)}_{\l}$.} 
\begin{align*}
&\Phi_{\theta}^{0, 2}((\lfloor {\rm u}_t\rceil^{\l})^{(1)}_0, dB_t, dt)\\
&\ \ \ \ \ \ \ \ \   :=\varpi\left(2\nabla\big((\lfloor {\rm u}_t\rceil^{\l})^{(1)}_{0}\big)\big(H^{\l}\big)^{(1)}_0(\lf {\rm u}_{t}\rc^{0}, e_i)+ \big(H^{\l}\big)^{(2)}_0(\lf {\rm u}_{t}\rc^{0}, e_i)\right)e_i\ dt\\
&\ \ \ \ \ \ \ \ \ \ \ \ \ +\theta\left(\left[H(\lf {\rm u}_{t}\rc^{0}, e_i), 2\nabla\big((\lfloor {\rm u}_t\rceil^{\l})^{(1)}_{0}\big)\big(H^{\l}\big)^{(1)}_0(\lf {\rm u}_{t}\rc^{0}, e_i)+ \big(H^{\l}\big)^{(2)}_0(\lf {\rm u}_{t}\rc^{0}, e_i)\right]\right)\ dt,  
\end{align*}
and the  It\^{o} SDE of $\varpi_{\lfloor{\rm u}_t\rceil^0}\big((\lfloor {\rm u}_t\rceil^{\l})^{(2)}_0\big)$ on $T\mathcal{F}(\Bbb R^m)$ is
 \begin{align*}
d\big(\varpi((\lfloor {\rm u}_t\rceil^{\l})^{(2)}_0))
&=(\lf {\rm u}_{t}\rc^{0})^{-1}R\left(\lf {\rm u}_{t}\rc^{0} dB_t, \lf {\rm u}_{t}\rc^{0}\theta((\lfloor {\rm u}_t\rceil^{\l})^{(2)}_0)\right)\lf {\rm u}_{t}\rc^{0}\notag\\
&\ \ \ +(\lf {\rm u}_{t}\rc^{0})^{-1}R\left(\lf {\rm u}_{t}\rc^{0} e_i, \lf {\rm u}_{t}\rc^{0}\varpi((\lfloor {\rm u}_t\rceil^{\l})^{(2)}_0)e_i\right)  \lf {\rm u}_{t}\rc^{0}\ dt\notag\\
&\ \ \ +(\lf {\rm u}_{t}\rc^{0})^{-1}\big(\nabla(\lf {\rm u}_{t}\rc^{0}e_i) R\big)\left(\lf {\rm u}_{t}\rc^{0} e_i, \lf {\rm u}_{t}\rc^{0}\theta((\lfloor {\rm u}_t\rceil^{\l})^{(2)}_0)\right)\lf {\rm u}_{t}\rc^{0}\ dt\notag\\
&\ \ \  + \Phi_{\varpi}((\lfloor {\rm u}_t\rceil^{\l})^{(1)}_0,(\lfloor {\rm u}_t\rceil^{\l})^{(1)}_0, dB_t, dt)+\Phi_{\varpi}^{0, 2}((\lfloor {\rm u}_t\rceil^{\l})^{(1)}_0, dB_t, dt),
\end{align*}
where $\Phi_{\theta}(\cdot,\cdot, dB_t, dt)$ is given in (\ref{ito-2-theta-varpi-mathcal-V}) for $\lf {\rm u}_{t}\rc^{0}$ and 
\begin{align*}
&\Phi_{\varpi}^{0, 2}((\lfloor {\rm u}_t\rceil^{\l})^{(1)}_0, dB_t, dt):=\varpi\left(2\nabla\big((\lfloor {\rm u}_t\rceil^{\l})^{(1)}_{0}\big)\big(H^{\l}\big)^{(1)}_0(\lf {\rm u}_{t}\rc^{0}, dB_t)+ \big(H^{\l}\big)^{(2)}_0(\lf {\rm u}_{t}\rc^{0},  dB_t)\right).\end{align*}
\end{itemize}
\end{cor}

By  Corollary \ref{cor-u-t-lam-tangent-map-2} with  Lemma \ref{mathcal-V-t},  we can formulate $(\lfloor {\rm u}_t\rceil^{\l})^{(2)}_0$ and  $(\theta, \varpi)(\lfloor {\rm u}_t\rceil^{\l})^{(2)}_0$ using $D^{(1)}\lf F_{\tau, t}\rc^{0}(\cdot, {w})$ by stochastic Duhamel principle.  We only state the conclusion.

\begin{cor}\label{cor-u-t-lam-tangent-map-3} We retain all the notations in Lemma \ref{u-t-lam-tangent-map-2}.  
\begin{itemize}
\item[i)] On $T\mathcal{F}(\M)$, 
\begin{align*}(\lfloor {\rm u}_t\rceil^{\l})^{(2)}_0=&\left[D^{(1)}\lfloor{F}_t\rceil^{0}(\lfloor {\rm u}_0\rceil^{0}, {w})\right](\lfloor {\rm u}_0\rceil^{\l})^{(2)}_{0}+\nabla_{(\lfloor {\rm u}_0\rceil^{\l})^{(1)}_0}\left[D^{(1)}\lf F_t\rc^0(\lfloor {\rm u}_0\rceil^{0}, {w})\right]\big((\lfloor {\rm u}_0\rceil^{\l})^{(1)}_0\big)\\
&+{\rm V}_c\big((\lfloor {\rm u}_t\rceil^{\l})^{(2)}_0\big),
\end{align*}
where 
\begin{align*}
{\rm V}_c\big((\lfloor {\rm u}_t\rceil^{\l})^{(2)}_0\big)=&\int_{0}^{t} \left[D^{(1)}\lfloor {F}_{\tau, t}\rceil^{0}(\lfloor {\rm u}_\tau \rceil^{0}, {w})\right]\left\{\nabla^{(2)}\big((\lfloor {\rm u}_{\tau}\rceil^{\l})^{(1)}_{0}, (\lfloor {\rm u}_{\tau}\rceil^{\l})^{(1)}_{0}\big)H^{0}(\lf {\rm u}_{\tau}\rc^{0}, \circ dB_\tau)\right.\\
&\ \ \ \   -\nabla^{(2)}\big([D^{(1)}\lf F_{\tau}\rc^{0}](\lfloor {\rm u}_0\rceil^{\l})^{(1)}_{0}, [D^{(1)}\lf F_{\tau}\rc^{0}](\lfloor {\rm u}_0\rceil^{\l})^{(1)}_{0}\big)H^{0}(\lf {\rm u}_{\tau}\rc^{0}, \circ dB_\tau)\\
&\ \ \ \    +R\big(H^{0}(\lf {\rm u}_{\tau}\rc^{0}, \circ dB_\tau), (\lfloor {\rm u}_\tau\rceil^{\l})^{(1)}_{0}\big)(\lfloor {\rm u}_{\tau}\rceil^{\l})^{(1)}_{0}\\
&\ \ \ \      -R\big(H^{0}(\lf {\rm u}_{\tau}\rc^{0}, \circ dB_\tau),[D^{(1)}\lf F_{\tau}\rc^{0}](\lfloor {\rm u}_0\rceil^{\l})^{(1)}_{0}\big)[D^{(1)}\lf F_{\tau}\rc^{0}](\lfloor {\rm u}_0\rceil^{\l})^{(1)}_{0}\\
&\ \ \ \       \left.+\ 2\nabla\big((\lfloor {\rm u}_\tau\rceil^{\l})^{(1)}_{0}\big)\big(H^{\l}\big)^{(1)}_{0}(\lf {\rm u}_{\tau}\rc^{0}, \circ dB_t)+ \big(H^{\l}\big)^{(2)}_{0}(\lf {\rm u}_{\tau}\rc^{0}, \circ dB_\tau)\right\}.
\end{align*}
\item[ii)] On $T\mathcal{F}(\Bbb R^m)$, 
 the It\^{o} form of $ \wt{(\lfloor {\rm u}_t\rceil^{\l})^{(2)}_0}:=(\theta, \varpi)_{\lfloor{\rm u}_t\rceil^0}\big((\lfloor {\rm u}_t\rceil^{\l})^{(2)}_0\big)$ is
\begin{align*}
 \wt{(\lfloor {\rm u}_t\rceil^{\l})^{(2)}_0}=&\left[\wt{D^{(1)}\lfloor{F}_t\rceil^{0}}(\lfloor {\rm u}_0\rceil^{\l}, {w})\right] \wt{(\lfloor {\rm u}_0\rceil^{\l})^{(2)}_0}\\
&\ +(\theta, \varpi)\!\left(\nabla_{(\lfloor {\rm u}_0\rceil^{\l})^{(1)}_0}\left[D^{(1)}\lf F_t\rc^0(\lfloor {\rm u}_0\rceil^{\l}, {w})\right]\!\big((\lfloor {\rm u}_0\rceil^{\l})^{(1)}_0\big)\!\right)+\wt{{\rm V}_c}((\lfloor {\rm u}_t\rceil^{\l})^{(2)}_0), 
\end{align*}
where
\begin{align*}
\wt{{\rm V}_c}((\lfloor {\rm u}_t\rceil^{\l})^{(2)}_0) :=&\int_{0}^{t}\left[\wt{D^{(1)}\lfloor {F}_{\tau, t}\rceil^{0}}(\lfloor {\rm u}_\tau \rceil^{0}, {w})\right]\left\{\big(\Phi_{\theta}, \Phi_{\varpi}\big)\big((\lfloor {\rm u}_{\tau}\rceil^{\l})^{(1)}_0,(\lfloor {\rm u}_{\tau}\rceil^{\l})^{(1)}_0, dB_{\tau}, d\tau\big)\right.\\
\ \ \ \ \ \ \ \ \ \ \ \ \ \  &-\big(\Phi_{\theta}, \Phi_{\varpi}\big)\big([D^{1}\lf F_{\tau}\rc^0](\lfloor {\rm u}_0\rceil^{\l})^{(1)}_0,[D^{1}\lf F_{\tau}\rc^0](\lfloor {\rm u}_0\rceil^{\l})^{(1)}_0, dB_\tau, d\tau\big)\\
\ \ \ \ \ \ \ \ \ \ \ \ \ \  &+\big(\Phi_{\theta}^{0, 2}, \Phi_{\varpi}^{0, 2}\big)((\lfloor {\rm u}_\tau\rceil^{\l})^{(1)}_0, dB_\tau, d\tau)-\left(\varpi\big(\Phi_{e_i}^{0, 2}\big((\lfloor {\rm u}_{\tau}\rceil^{\l})^{(1)}_0,(\lfloor {\rm u}_{\tau}\rceil^{\l})^{(1)}_0\big)\big)e_i,\right.\\
\ \ \ \ \ \ \ \ \ \ \ \ \ \  &\left.\left.+\ (\lf {\rm u}_{\tau}\rc^{0})^{-1}R\big(\lf {\rm u}_{\tau}\rc^{0} e_i, \lf {\rm u}_{\tau}\rc^{0}\theta\big(\Phi_{e_i}^{0, 2}\big((\lfloor {\rm u}_\tau\rceil^{\l})^{(1)}_0,(\lfloor {\rm u}_\tau\rceil^{\l})^{(1)}_0\big)\big)\lf {\rm u}_{\tau}\rc^{0}\right)\ d\tau\right\},
\end{align*}
\vspace*{-0.7cm}
\begin{align*}
\Phi_{e_i}^{0, 2}\big((\lfloor {\rm u}_{\tau}\rceil^{\l})^{(1)}_0\!,(\lfloor {\rm u}_{\tau}\rceil^{\l})^{(1)}_0\big)\!:=&
\nabla^{(2)}\big((\lfloor {\rm u}_{\tau}\rceil^{\l})^{(1)}_{0}, (\lfloor {\rm u}_{\tau}\rceil^{\l})^{(1)}_{0}\big)H^{0}(\lf {\rm u}_{\tau}\rc^{0}, e_i)\\
&\!\!\!\!\!\!-\nabla^{(2)}\big([D^{(1)}\lf F_{\tau}\rc^{0}](\lfloor {\rm u}_0\rceil^{\l})^{(1)}_{0}, [D^{(1)}\lf F_{\tau}\rc^{0}](\lfloor {\rm u}_0\rceil^{\l})^{(1)}_{0}\big)H^{0}(\lf {\rm u}_{\tau}\rc^{0}, e_i)\\
&\!\!\!\!\!\! +R\big(H^{0}(\lf {\rm u}_{\tau}\rc^{0}, e_i), (\lfloor {\rm u}_\tau\rceil^{\l})^{(1)}_{0}\big)(\lfloor {\rm u}_{\tau}\rceil^{\l})^{(1)}_{0}\\
&\!\!\!\!\!\!-R\big(H^{0}(\lf {\rm u}_{\tau}\rc^{0}, e_i),[D^{(1)}\lf F_{\tau}\rc^{0}](\lfloor {\rm u}_0\rceil^{\l})^{(1)}_{0}\big)[D^{(1)}\lf F_{\tau}\rc^{0}](\lfloor {\rm u}_0\rceil^{\l})^{(1)}_{0}\\
&\!\!\!\!\!\!+2\nabla\big((\lfloor {\rm u}_\tau\rceil^{\l})^{(1)}_{0}\big)\big(H^{\l}\big)^{(1)}_{0}(\lf {\rm u}_{\tau}\rc^{0}, e_i)+ \big(H^{\l}\big)^{(2)}_{0}(\lf {\rm u}_{\tau}\rc^{0}, e_i).
\end{align*}
\end{itemize}
\end{cor}

Let $[D^{(2)}\lf F_t\rc^0(\cdot, {w})]$ be the restriction of the second order tangent map of  $\lf F_t\rc^0$ on the space $T_{(\lf{\rm u}\rc^0, (\lf{\rm u}\rc^{\l})^{(1)}_{\l})}T_{\lf{\rm u}\rc^0}\mathcal{F}(\M)$. We can deduce  from  Corollary \ref{cor-2--u-t-lam-tangent-map-1} and Corollary  \ref{cor-u-t-lam-tangent-map-3}  that 
\begin{align}\notag\left((\lfloor {\rm u}_t\rceil^{\l})^{(1)}_0,  (\lfloor {\rm u}_t\rceil^{\l})^{(2)}_0\right)=&\ [D^{(2)}\lf F_t\rc^0(\cdot, {w})]\left( (\lfloor {\rm u}_0\rceil^{\l})^{(1)}_0,  (\lfloor {\rm u}_0\rceil^{\l})^{(2)}_0\right)\\
\notag&+\left({\rm V}_c((\lfloor {\rm u}_t\rceil^{\l})^{(1)}_0), {\rm V}_c((\lfloor {\rm u}_t\rceil^{\l})^{(2)}_0)\right), \\
\big(\wt{(\lfloor {\rm u}_t\rceil^{\l})^{(1)}_0}\!,  \wt{(\lfloor {\rm u}_t\rceil^{\l})^{(2)}_0}\big)\label{ut2-wt-DF2}=&\  [\wt{D^{(2)}\lf F_t\rc^0}(\cdot, {w})]\big( \wt{(\lfloor {\rm u}_0\rceil^{\l})^{(1)}_0}\!,  \wt{(\lfloor {\rm u}_0\rceil^{\l})^{(2)}_0}\big)\\
\notag&\ +\left(\wt{{\rm V}_c}((\lfloor {\rm u}_t\rceil^{\l})^{(1)}_0), \wt{{\rm V}_c}((\lfloor {\rm u}_t\rceil^{\l})^{(2)}_0)\right).
\end{align}

Continuing the discussions in Lemma \ref{u-t-lam-tangent-map-1} and Lemma \ref{u-t-lam-tangent-map-2}, we can  derive the SDEs for the differentials  $(\lfloor {\rm{u}}_t\rceil^{\l})^{(j)}_\l$,  $3\leq j\leq k-2$, and their pull back  $\wt{(\lfloor {\rm{u}}_t\rceil^{\l})^{(j)}_\l}$ via  the $(\theta, \varpi)$-map, whose It\^{o} forms involve 
$\{{\bf \nabla}^{(l')}(H^{\l})^{(j')}_{\l}\}_{j'\leq j, l'+j'\leq j}, \{{\bf \nabla}^{(l')} R^{\l}\}_{l'\leq j}$.  We omit the details.

The SDEs of $\left([D^{(l)}\lf F_t\rc^{\l}(\cdot, {w})]\right)^{(j)}_{\l}$ can  be formulated as in Section \ref{BMM-flow} by analogy with the deterministic case.  We only  state the SDEs  for the  $(l, j)=(1, 1)$ case  using the reference connection of $\wt{g}^0$, whose calculations  can be done  as in Lemma \ref{mathcal-V-t}.

\begin{lem}\label{DF-t-la-diff-la}Let $\lambda\in
(-1, 1)\mapsto g^{\lambda}\in \mathcal{M}^k(M)$  be a $C^k$ curve with $k\geq 4$. Let $\l\mapsto (\lf{\rm u}_0\rc^{\l}, {\rm v}_0^{\l})\in T\mathcal{F}(\M)$ be $C^{1}$ and write 
\[
\big(\lf{\rm u}_t\rc^{\l}, {\rm v}_t^{\l}\big):=\left(\lf F_t\rc^{\l}(\lf{\rm u}_0\rc^{\l}), \big[D^{(1)}\lf F_t\rc^{\l}(\lf{\rm u}_0\rc^{\l}, {w})\big]{\rm v}_0^{\l}\right),\ \  ({\rm v}_t^{\l})^{(1)}_{\l}:=\nabla\big((\lf{\rm u}_t\rc^{\l})^{(1)}_{\l}\big){\rm v}_t^{\l}.
\]
\begin{itemize}
\item[i)] The process  $({\rm v}_t^{\l})^{(1)}_{\l}$ satisfies the Stratonovich SDE
\begin{align*}
d({\rm v}_t^{\l})^{(1)}_{\l}=& \nabla\big(({\rm v}_t^{\l})^{(1)}_{\l}\big)H^{\l}(\lf{\rm u}_t\rc^{\l}, \circ dB_t)+\nabla^{(2)}\big({\rm v}_t^{\l}, (\lf{\rm u}_t\rc^{\l})^{(1)}_{\l}\big)H^{\l}(\lf{\rm u}_t\rc^{\l}, \circ dB_t)\\
&+R\big(H^{\l}(\lf {\rm u}_t\rc^{\l}, e), (\lf{\rm u}_t\rc^{\l})^{(1)}_{\l}\big){\rm v}_t^{\l}+\nabla({\rm v}_t^{\l})(H^{\l})^{(1)}_{\l}(\lf{\rm u}_t\rc^{\l}, \circ dB_t).
\end{align*}
\item[ii)] The process  $(\theta, \varpi)_{\lf{\rm u}_t\rc^{\l}}\big(({\rm v}_t^{\l})^{(1)}_{\l}\big)$ satisfies the Stratonovich SDE
\begin{align*}
&d\big((\theta, \varpi)({\rm v}_t^{\l})^{(1)}_{\l}\big)\\
&=d(\theta, \varpi)\left(H^{\l}(\lf{\rm u}_t\rc^{\l}, \circ dB_t), ({\rm v}_t^{\l})^{(1)}_{\l}\right)+\nabla(({\rm v}_t^{\l})^{(1)}_{\l})\left((\theta, \varpi)\big(H^{\l}(\lf{\rm u}_t\rc^{\l}, \circ dB_t)\big)\right)\\
&\ \  +(\theta, \varpi)\left(\nabla^{(2)}\big({\rm v}_t^{\l}, (\lf{\rm u}_t\rc^{\l})^{(1)}_{\l}\big)H^{\l}(\lf{\rm u}_t\rc^{\l}, \circ dB_t)+R\big(H^{\l}(\lf {\rm u}_t\rc^{\l}, \circ dB_t), (\lf{\rm u}_t\rc^{\l})^{(1)}_{\l}\big){\rm v}_t^{\l}\right)\\
&\ \  +(\theta, \varpi)\left(\nabla({\rm v}_t^{\l})(H^{\l})^{(1)}_{\l}(\lf{\rm u}_t\rc^{\l}, \circ dB_t)\right).
\end{align*}
\item[iii)] The It\^{o} SDE of the process  $(\theta, \varpi)_{\lf{\rm u}_t\rc^{\l}}\big(({\rm v}_t^{\l})^{(1)}_{\l}\big)$ is
\begin{align*}
&\ \!d\big((\theta, \varpi)({\rm v}_t^{\l})^{(1)}_{\l}\big)\\
&=d(\theta, \varpi)\left(H^{\l}(\lf{\rm u}_t\rc^{\l},  dB_t), ({\rm v}_t^{\l})^{(1)}_{\l}\right)+\nabla(({\rm v}_t^{\l})^{(1)}_{\l})\big((\theta, \varpi)\big(H^{\l}(\lf{\rm u}_t\rc^{\l},  dB_t)\big)\big)\\
&\ \ \  +(\theta, \varpi)\left(\nabla^{(2)}\big({\rm v}_t^{\l}, (\lf{\rm u}_t\rc^{\l})^{(1)}_{\l}\big)H^{\l}(\lf{\rm u}_t\rc^{\l}, dB_t)+R\big(H^{\l}(\lf {\rm u}_t\rc^{\l}, e), (\lf{\rm u}_t\rc^{\l})^{(1)}_{\l}\big){\rm v}_t^{\l}\right)\\
&\ \ \   +(\theta, \varpi)\left(\nabla({\rm v}_t^{\l})(H^{\l})^{(1)}_{\l}(\lf{\rm u}_t\rc^{\l}, dB_t)\right)+\nabla(H^{\l}(\lfloor {\rm u}_t\rceil^{\l},  e_i))\left\{d(\theta, \varpi)\big(H^{\l}(\lf{\rm u}_t\rc^{\l},  e_i), ({\rm v}_t^{\l})^{(1)}_{\l}\big)\right.\\
&\ \ \  +\nabla(\big({\rm v}_t^{\l}\big)^{(1)}_{\l})\big((\theta, \varpi)\big(H^{\l}(\lf{\rm u}_t\rc^{\l}, e_i)\big)\big)+(\theta, \varpi)\big(\nabla^{(2)}\big({\rm v}_t^{\l}, (\lf{\rm u}_t\rc^{\l})^{(1)}_{\l}\big)H^{\l}(\lf{\rm u}_t\rc^{\l}, e_i)\big)\\
&\ \ \  +R\big(H^{\l}(\lf {\rm u}_{t}\rc^{\l}, e_i), (\lfloor {\rm u}_t\rceil^{\l})^{(1)}_{\l}\big){\rm v}_t^{\l}\left.+\big(\theta, \varpi\big)\big(\nabla({\rm v}_t^{\l})\big(H^{\l}\big)^{(1)}_{\l}(\lf {\rm u}_{t}\rc^{\l}, e_i)\big)\right\}\ dt.
\end{align*}
\end{itemize}
\end{lem}
As before,  the formulas in  Lemma \ref{DF-t-la-diff-la}  can be simplified at $\l=0$. 

\begin{cor}\label{Cor-DF-t-la-diff-la}Let $\lambda\in
(-1, 1)\mapsto g^{\lambda}\in \mathcal{M}^k(M)$  be a $C^k$ curve with $k\geq 4$  and let $\l\mapsto (\lf{\rm u}_0\rc^{\l}, {\rm v}_0^{\l})\in T\mathcal{F}(\M)$ be $C^{1}$.  We retain all the notations in Lemma \ref{DF-t-la-diff-la}. 
\begin{itemize}
\item[i)]The process  $(\theta, \varpi)_{\lf{\rm u}_t\rc^{\l}}\big(({\rm v}_t^{\l})^{(1)}_{0}\big)$ satisfies the Stratonovich SDE
\begin{align*}
d\big((\theta, \varpi)(({\rm v}_t^{\l})^{(1)}_{0})\big)&=\left(\varpi(({\rm v}_t^{\l})^{(1)}_{0})\circ dB_t, {(\lfloor{\rm u}_t\rceil^0)}^{-1}R\big(\lfloor{\rm u}_t\rceil^0\circ dB_t, \theta(({\rm v}_t^{\l})^{(1)}_{0})\big)\lfloor{\rm u}_t\rceil^0\right)\\
&\ +(\theta, \varpi)\left(\nabla^{(2)}\big({\rm v}_t^{\l}, (\lf{\rm u}_t\rc^{\l})^{(1)}_{\l}\big)H^{\l}(\lf{\rm u}_t\rc^{\l}, \circ dB_t)\right)\\
&\ +(\theta, \varpi)\left(\! R\big(H^{\l}(\lf {\rm u}_t\rc^{\l}, \circ dB_t), (\lf{\rm u}_t\rc^{\l})^{(1)}_{\l}\big){\rm v}_t^{\l}+\nabla({\rm v}_t^{\l})(H^{\l})^{(1)}_{\l}(\lf{\rm u}_t\rc^{\l}, \circ dB_t)\right). 
\end{align*}
\item[ii)] The It\^{o} SDE of the process $\theta_{\lf{\rm u}_t\rc^{\l}}\big(({\rm v}_t^{\l})^{(1)}_{0}\big)$ in $T\mathcal{F}(\Bbb R^n)$ is 
\begin{align*}
d\big(\theta(({\rm v}_t^{\l})^{(1)}_{0}))=&\ \!\varpi(({\rm v}_t^{\l})^{(1)}_{0}) d B_{t}+{\rm Ric}(\lf {\rm u}_{t}\rc^{0}\theta(({\rm v}_t^{\l})^{(1)}_{0})) dt\\
&+\Phi_{\theta}({\rm v}_t^{\l}, (\lfloor {\rm u}_t\rceil^{\l})^{(1)}_0, dB_t, dt)+\Phi_{\theta}^{1, 1}({\rm v}_t^{\l},(\lfloor {\rm u}_t\rceil^{\l})^{(1)}_0, dB_t, dt),
\end{align*}
where $\Phi_{\theta}(\cdot,\cdot, dB_t, dt)$ is given in (\ref{ito-2-theta-varpi-mathcal-V}) associated to  $\lf {\rm u}_{t}\rc^{0}$ and \footnote{We use the  upper script $^{1, 1}$ to indicate the functions  $\Phi_{\theta}^{1, 1}, \Phi_{\varpi}^{1, 1}$ are associated with  $\big([D^{(1)}\lf F_t(\cdot, {w})]\big)^{(1)}_{\l}$.}
\begin{align*}
\Phi_{\theta}^{1, 1}({\rm v}_t^{0},(\lfloor {\rm u}_t\rceil^{\l})^{(1)}_0\!, dB_t, dt):=&\ \!2\varpi\big(\nabla({\rm v}_t^{0})(H^{\l})^{(1)}_{0}(\lf{\rm u}_t\rc^{0}, e_i)\big)e_i dt\\
&+\theta\left(\left[H^0(\lf{\rm u}_t\rc^{0}, e_i), \nabla({\rm v}_t^{0})(H^{\l})^{(1)}_{0}(\lf{\rm u}_t\rc^{0}, e_i)\right]\right)dt. 
\end{align*}
The It\^{o} SDE of the process $\varpi_{\lf{\rm u}_t\rc^{\l}}\big(({\rm v}_t^{\l})^{(1)}_{0}\big)$ in $T\mathcal{F}(\Bbb R^n)$ is
\begin{align*}
d\big(\varpi(({\rm v}_t^{\l})^{(1)}_{0}))
&=(\lf {\rm u}_{t}\rc^{0})^{-1}R\left(\lf {\rm u}_{t}\rc^{0} dB_t, \lf {\rm u}_{t}\rc^{0}\theta(({\rm v}_t^{\l})^{(1)}_{0})\right)\lf {\rm u}_{t}\rc^{0}\\
&\ \ \ +(\lf {\rm u}_{t}\rc^{0})^{-1}R\left(\lf {\rm u}_{t}\rc^{0} e_i, \lf {\rm u}_{t}\rc^{0}\varpi(({\rm v}_t^{\l})^{(1)}_{0})e_i\right)  \lf {\rm u}_{t}\rc^{0}\ dt\\
&\ \ \ +(\lf {\rm u}_{t}\rc^{0})^{-1}\big(\nabla(\lf {\rm u}_{t}\rc^{0}e_i) R\big)\left(\lf {\rm u}_{t}\rc^{0} e_i, \lf {\rm u}_{t}\rc^{0}\theta(({\rm v}_t^{\l})^{(1)}_{0})\right)\lf {\rm u}_{t}\rc^{0}\ dt\\
&\ \ \ +\Phi_{\varpi}({\rm v}_t^{\l}, (\lfloor {\rm u}_t\rceil^{\l})^{(1)}_0, dB_t, dt)+\Phi_{\varpi}^{1, 1}({\rm v}_t^{\l},(\lfloor {\rm u}_t\rceil^{\l})^{(1)}_0, dB_t, dt),
\end{align*}
where $\Phi_{\varpi}(\cdot,\cdot, dB_t, dt)$ is given in (\ref{ito-2-varpi-mathcal-V}) associated to  $\lf {\rm u}_{t}\rc^{0}$ and
\begin{align*}
\Phi_{\varpi}^{1, 1}({\rm v}_t^{0},(\lfloor {\rm u}_t\rceil^{\l})^{(1)}_0, dB_t, dt):=&\ \!\varpi\left(\nabla({\rm v}_t^{0})(H^{\l})^{(1)}_{0}(\lf{\rm u}_t\rc^{0}, dB_t)\right).
\end{align*}
\end{itemize}
\end{cor} 

We can formulate $({\rm v}_t^{\l})^{(1)}_{0}$ and  $(\theta, \varpi)({\rm v}_t^{\l})^{(1)}_{0}$ by stochastic Duhamel principle.

\begin{cor}
 We retain  all the notations  in Corollary \ref{Cor-DF-t-la-diff-la}.  
\begin{itemize}
\item[i)]  The process  $({\rm v}_t^{\l})^{(1)}_{0}$ has the expression 
\begin{align*}
({\rm v}_t^{\l})^{(1)}_{0}\!=\big[D^{(1)}\lf F_t\rc^0(\lf {\rm u}_0\rc^{0}, {w})\big](({\rm v}_0^{\l})^{(1)}_{0})+\! \nabla({(\lf {\rm u}_0\rc^{\l})^{(1)}_0})\big[D^{(1)}\lf F_t\rc^0(\lf {\rm u}_0\rc^{0}, {w})\big]({{\mathsf v}^0_0})+\! {\rm V}_c(({\rm v}_t^{\l})^{(1)}_{0}), 
\end{align*}
where 
\begin{align*}
{\rm V}_c(({\rm v}_t^{\l})^{(1)}_{0}):=\int_{0}^{t} \big[D^{(1)}\lfloor {F}_{\tau, t}\rceil^{0}(\lfloor {\rm u}_\tau \rceil^{0}, {w})\big]\left(\nabla({\rm v}_\tau^{0})(H^{\l})^{(1)}_{0}(\lf{\rm u}_\tau\rc^{0}, \circ dB_\tau)\right).
\end{align*}
\item[ii)]On $T\mathcal{F}(\Bbb R^n)$, the process  $\wt{( {\rm v}_t^{\l})^{(1)}_0}:=(\theta, \varpi)_{\lf{\rm u}_t\rc^{\l}}\big(({\rm v}_t^{\l})^{(1)}_{0}\big)$ has the expression
\begin{align*}
\wt{( {\rm v}_t^{\l})^{(1)}_0}=&[\wt{D^{(1)}\lf F_t\rc^0}(\lf {\rm u}_0\rc^{0}, {w})]\wt{({\rm v}_0^{\l})^{(1)}_{0}}+(\theta, \varpi)\left(\nabla({(\lf {\rm u}_0\rc^{\l})^{(1)}_0}) \big[D^{(1)}\lf F_t\rc^0(\lf {\rm u}_0\rc^{0}, {w})\big]({{\mathsf v}^0_0})\right)\\
&+\wt{{\rm V}_c}(({\rm v}_t^{\l})^{(1)}_{0}), 
\end{align*}
where 
\begin{align*}
\wt{{\rm V}_c}(({\rm v}_t^{\l})^{(1)}_{0})=&\int_{0}^{t}\left[\wt{D^{(1)}\lfloor {F}_{\tau, t}\rceil^{0}}(\lfloor {\rm u}_\tau \rceil^{0}, {w})\right]\left\{\big(\Phi_{\theta}^{1, 1}, \Phi_{\varpi}^{1, 1}\big)({\rm v}_\tau^{0},(\lfloor {\rm u}_\tau\rceil^{\l})^{(1)}_0\!, dB_\tau, d\tau)\right.\\
&\ \ \ \ \ \ \ \ \ \ \ \ \ \ \ \ \ \ \ \ \ \ \ \ \ \ \ \ \ \ \  \left.-\left(\varpi\big(\nabla({\rm v}_\tau^{0})(H^{\l})^{(1)}_{0}(\lf{\rm u}_\tau\rc^{0}, e_i)\big)e_i,\ 0\right)\  d\tau\right\}.
\end{align*}
\end{itemize}
\end{cor}
\begin{proof}By a comparison of  the SDEs in Lemma \ref{DF-t-la-diff-la} and Corollary \ref{Cor-DF-t-la-diff-la} with those in Lemma \ref{mathcal-V-t}, we can  compute  as in Corollary \ref{Duha-mathcal-V-t} to derive  i) and ii).  We note that for ii), $\wt{{\rm V}_c}(({\rm v}_t^{\l})^{(1)}_{0})$ has an extra term  
\begin{align*}
&-\left(0, \ (\lf{\rm u}_{\tau}\rc^0)^{-1} R\left(\lf{\rm u}_{\tau}\rc^0 e_i, \lf{\rm u}_{\tau}\rc^0\theta\big(\nabla({\rm v}_\tau^{0})(H^{\l})^{(1)}_{0}(\lf{\rm u}_\tau\rc^{0}, e_i)\big)\right)\lf{\rm u}_{\tau}\rc^0 \right)\ d\tau, 
\end{align*}
which turns out to be  zero since $\theta\big(\nabla({\rm v}_\tau^{0})(H^{\l})^{(1)}_{0}(\lf{\rm u}_\tau\rc^{0}, e_i)\big)$ is zero.
\end{proof}

We are in a situation to state the norm estimations on the differential processes.

\begin{prop}\label{est-norm-u-t-j} Let $\l\mapsto g^{\l}\in \mathcal{M}^k(M)$  be a $C^k$  curve  with $k\geq 3$.  Let $x\in \M$  and  $\l\mapsto \lfloor {\rm{u}}_0\rceil^{\l}\in \mathcal{O}^{\wt{g}^{\l}}_x(\M)$ be a $C^{k-2}$ curve in $\mathcal{F}(\M)$ and let $\{\lfloor {\rm{u}}_t\rceil^{\l}\}_{t\in [0, T]}$ be the solution to (\ref{FM-BM-SDE}) with initial $\lfloor {\rm{u}}_0\rceil^{\l}$. 
\begin{itemize}
\item[i)] For every $q\geq 1$ and  $(l, j)$ with $1\leq l+j\leq k-2$, 
  there exist $\underline{c}_{(l,j)}(q)$ depending  on $m, q$ and the norm bounds of  $\{{\bf \nabla}^{(l')}(H^{\l})^{(j')}_{\l}\}_{j'\leq j, l'+j'\leq l+j}, \{{\bf \nabla}^{(l')}R^{\l}\}_{l'\leq l}$,  and ${c}_{(l,j)}(q)$ depending  on $(l, j), m, q$ and the norm bounds of   $\{{\bf \nabla}^{(l')}R^{\l}\}_{l'\leq 1}$, such that\begin{align}
\label{DF-j-lambda} &\E\sup\limits_{0\leq \underline{t}<\overline{t}\leq T}\left\|\left(\big[D^{(l)}\lf F_{\underline{t}, \overline{t}}\rc^{\l}({\rm{u}}_{\underline{t}}, {w})\big]\right)^{(j)}_{\l}\right\|^q\leq   \underline{c}_{(l,j)}(q)e^{c_{(l,j)}(q)T}, \ \forall  T\in \Bbb R_+. 
 \end{align}
\item[ii)]Let $T_0>0$.  For each $q\geq 1$, $(l, j)$ with $1\leq l+j\leq k-2$ and $T>T_0$, 
  there exist $\underline{c}'_{(l,j)}(q)$, which depends   on $m, q$ and the norm bounds of  $\{{\bf \nabla}^{(l')}(H^{\l})^{(j')}_{\l}\}_{j'\leq j, l'+j'\leq l+j}$, $\{{\bf \nabla}^{(l')} R^{\l}\}_{l'\leq l}$,  and ${c}'_{(l,j)}(q)$,  which depends on $(l, j), m, q, T, T_0$ and the norm bounds of   $\{{\bf \nabla}^{(l')}R^{\l}\}_{l'\leq 1}$, such that,   for any $x, y\in \M$, 
\begin{align}
\label{DF-j-lambda-cond}&\E_{\P^{\lambda, *}_{x, y, T}}\sup\limits_{0\leq \underline{t}<\overline{t}\leq T}\left\|\left(\big[D^{(l)}\lf F_{\underline{t}, \overline{t}}\rc^{\l}({\rm{u}}_{\underline{t}}, {w})\big]\right)^{(j)}_{\l}\right\|^q\leq   \underline{c}'_{(l,j)}(q)e^{c'_{(l,j)}(q)(1+d_{\wt{g}^{\l}}(x, y))}.
 \end{align}
\end{itemize}
\end{prop}

By using the SDEs  formulated in Section \ref{sec-4-5}, the estimation in (\ref{DF-j-lambda}) can be obtained using (\ref{wt-DF-j})  and the estimation in (\ref{DF-j-lambda-cond}) can be obtained using Lemma \ref{Hsu-thm-5.4.4} (\ref{exp-grad-lnp}) and (\ref{wt-DF-j-cond}).  The proofs are similar to the second steps in the proofs of Proposition \ref{est-norm-D-j-F-t} and Proposition \ref{est-norm-D-j-F-t-cond}, respectively. We omit them.

\section{The first differential of the heat kernels in metrics}\label{sec5}

Our main result in this section is a first step of the proof of Theorem \ref{diff-HK-estimations-gen}. 

\begin{theo}\label{regu-p-1st} For  any $g^0=g\in \mathcal{M}^k(M)$ ($k\geq 3$), there exist $\iota\in (0,
1)$ and a neighborhood $\mathcal{V}_g$ of $g$ in $\mathcal{M}^{k}(M)$ such that  the following hold true for any $C^{k}$ curve  $\lambda\in
(-1, 1)\mapsto g^{\lambda}\in \mathcal{V}_{g}$. 
\begin{itemize}
\item[i)]For any $x\in \M$ and $T\in \Bbb R_+$, $\l\mapsto p^{\l}(T, x, \cdot)$ is $C^1$ in $C^{k, \iota}(\M)$ with \begin{align}\label{Thm1.3-step 1}
(\ln p^{\l})^{(1)}_{\l}(T, x, y)+ (\ln \rho^{\l})^{(1)}_{\l}(y)= \phi_{\l}^{1}(T, x, y),
\end{align}
where $\rho^{\l}(y)=(d{\rm Vol}^{\l}/d{\rm Vol}^0)(y)$ and $\phi_{\l}^{1}$ is as in (\ref{phi-1-candidate}).
\item[ii)]Let $T_0>0$. For $q\geq 1$ and $l$, $0\leq l\leq k-3$, there are constants ${c}_{\l, (l, 1)}(q)$ which depend  on $m, q$, $T$,  $T_0$, $\|g^{\l}\|_{C^{l+3}}$ and  $\|\XX^{\l}\|_{C^{l+2}}$ such that  for all $x\in \M$ and  $T>T_0$,
\begin{equation}\label{grad-lnp-lam-1}
\hspace{6.5mm} \left\|\nabla^{(l)}(\ln p^{\l})^{(1)}_{\l}(T, x, \cdot)\right\|_{L^q}\leq {c}_{\l, (l, 1)}(q).
\end{equation}
\item[iii)] The function   $x\mapsto \int_{\M} (p^{\l})^{(1)}_{\l}(T, x, y) \wt{f}(y)\ d{\rm{Vol}}_{\wt{g}^{\l}}(y)$ is continuous  for any uniformly continuous and bounded $\wt{f}\in C(\M)$. 
\end{itemize}
\end{theo}

\subsection{Strategy}\label{Obs-Stra}

We   show Theorem \ref{regu-p-1st} by  describing the $C^1$ vector field  ${\rm z}_{T}^{\l, 1}$ such that (\ref{CIBP}) holds true. Before that,  let us recall some classical results for parabolic equations.

Let $\mathcal{D}\subset \mathcal{D}_1\times \mathcal{D}_2$ with $\mathcal{D}_1$ being a bounded interval of $\Bbb R_+$ and $D_2$ being a bounded connected open  domain of $\M$.  For $g\in \mathcal{M}^k(M)$,  consider the parabolic equation  
\begin{equation}\label{parabolic-equ}
{\rm L} q:=(\frac{\partial}{\partial t}-\Delta)q=r, 
\end{equation}
where  $\Delta$ is  the $\wt{g}$-Laplacian on  $C^2$ functions on
$\wt{M}$ and $r$ is a continuous function on $\mathcal{D}$.  By a solution $q$ to (\ref{parabolic-equ}), we mean a function $q$ on $\mathcal{D}$  which satisfies (\ref{parabolic-equ}) and all the derivatives of which appear in ${\rm L}q$ are continuous functions  on $\mathcal{D}$.  Such a $q$ can be smoother, depending on the regularities of  both ${\rm{L}}$ and $r$.  For instance, $q$ is $C^{\infty}$ if both both ${\rm L}$ and $r$ are  $C^{\infty}$.  In our case,  ${\rm L}$ varies  $C^{k-2}$ H\"{o}lder with respect to base points and  $q$  is mostly $C^k$ H\"{o}lder in general even in case $r$ is smooth. 

\begin{lem}\label{Friedman-lem}(\cite[Theorem 11, p.74]{Fr}) Let ${\rm L}$ be given in (\ref{parabolic-equ})  which is $C^{k-2}$ and H\"{o}lder continuous with exponent $\iota$. Assume $r$ in (\ref{parabolic-equ}) is such that 
\[
D_x^{n}D_t^{l}r, \ 0\leq n+2l\leq k-2, l\leq l',
\]
are H\"{o}lder continuous with exponent $\iota$, where $D_{a}^{b}$ means the $b$-th differential form  with respect to the $a$-coordinate.  If $q$ is a solution to (\ref{Friedman-lem}), then 
\[
D_x^{n}D_t^{l}q, \ 0\leq n+2l\leq k, l\leq l'+1, 
\]
exist and are H\"{o}lder continuous with exponent $\iota$. 
\end{lem}

In particular, if $r$ is $\iota$-H\"{o}lder and $q$ solves (\ref{parabolic-equ}), Lemma \ref{Friedman-lem} shows  that all the differentials of $q$ up to the second order (where $\partial /\partial t$ is considered as second order differential) exist and are $\iota$-H\"{o}lder. The next lemma from \cite{Fr} further shows these differentials have bounds  completely determined by the bounds  of $q$ and $r$. For ${\it P}=(\tau, x)\in \mathcal{D}$, define 
\[
d_{\it P}=\sup_{Q\in \mathcal{D}(\tau)}d({\it P}, {\it Q}), 
\]
where $\mathcal{D}({\tau})$ is the intersection of the boundary of $\mathcal{D}$ with the half-space $t\leq \tau$. 
For a function $f$ on  $\mathcal{D}$ and any non-negative integers $n, j$ and for $\iota\in (0, 1)$, define
\begin{eqnarray*}
|f|_{n, j}= \sum_{l=0}^{j}N_{n,l}[f], \  \ 
|f|_{n, j+\iota}= |f|_{n, j}+\sum_{0}^{j}N_{n, l+\iota}[f],
\end{eqnarray*}
where 
\begin{eqnarray*}
N_{n,l}[f]&=& \sum \sup_{\it{P}\in \mathcal{D}}\left\{d^{n+l}_{\it P}| D^{l}_xf(\it{P})|\right\},\\
N_{n, l+\iota}[f]&=& \sum \sup_{\it{P, Q}\in \mathcal{D}}\left\{\min\{d^{n+l+\iota}_{\it P}, d^{n+l+\iota}_{\it Q}\}\cdot \frac{|D_x^l f({\it P})-D_x^l f({\it Q})|}{d({\it P}, {\it Q})^{\iota}}\right\}, 
\end{eqnarray*}
and the summation is over all the differentials of order $l$.

\begin{lem}\label{Friedman-lem-2}(\cite[Theorem 1, p.92]{Fr}) Let ${\rm L}$ be as in Lemma \ref{Friedman-lem}. There exists some geometric constant $\kappa$ (which depends on $\iota$, $\|g\|_{C^1}$) such that if $|r|_{2, \iota}<+\infty$ and $q$ is a bounded solution to (\ref{parabolic-equ}) and all its derivatives appearing in ${\rm L}q$ are $\iota$ H\"{o}lder, then \begin{equation*}
 |q|_{0, 2+\iota}<\kappa(|q|_{0, 0}+|r|_{2, \iota}). \end{equation*}\end{lem}

(Both Lemma \ref{Friedman-lem} and Lemma \ref{Friedman-lem-2} were stated in \cite{Fr} for domains in the Euclidean case. They apply to the manifold case since (\ref{parabolic-equ}) can be treated locally in coordinate charts.)

A companion notion of a solution to a parabolic equation is a solution in the  distribution sense.   Recall that the distributions on the domain $\mathcal{D}$ are the linear continuous functionals on the test function space $C_c^{\infty}(\mathcal{D})$ of compactly supported smooth functions on $\mathcal{D}$. Given a distribution $q$ on $\mathcal{D}$, one can define its weak derivative of any order $\alpha$, denoted by $D^{\alpha, {\rm{w}}}q$,  as a distribution on $\mathcal{D}$ by letting
\[
(D^{\alpha, {\rm{w}}}q)(f):=(-1)^{|\alpha|}q(D^{\alpha}f),\  \forall f\in C_c^{\infty}(\mathcal{D}).
\]
Any locally integrable function $q\in L^1_{\rm loc}(\mathcal{D})$ can be identified with a distribution by letting 
\[
q(f):=\int_{\mathcal{D}} qfdt\times d{\rm Vol},\ \forall f\in C_c^{\infty}(\mathcal{D}), 
\]
and hence its weak derivatives of any order always exist. Let ${\rm L}$ be as in (\ref{parabolic-equ}). The ${\rm L}$ distributional derivative of a distribution $q$ on $\mathcal{D}$ will be denoted by ${\rm L}^{\rm{w}}q$. Using mollifier and Lemma \ref{Friedman-lem-2}, we have the following classical result.  

\begin{lem}\label{weak-strong-same}Let $g\in C^k(M)$ and let ${\rm L}$ be as in (\ref{parabolic-equ}). Assume $r\in C^{0, \iota}(\mathcal{D})$ for some $\iota>0$ with $|r|_{2, \iota}<\infty$. Then for any $q\in C(\mathcal{D})$, 
\[
{\rm L}^{\rm{w}}q=r\   \Longrightarrow \ {\rm L}q=r.
\]
\end{lem}

As a corollary of the above lemmas, we have the following. 
\begin{lem}\label{weak-reg-p-1}  Assume  there are locally  $L^1$ integrable functions $\{\phi_{\l}^{1}(T, x, y)\}_{x\in \M, T\in \Bbb R_+}$ on $\M$ which are continuous in $\l$-parameter and are continuous in  $(T, y)$-parameter, locally uniformly in $\l$,  such that,  for any $f\in C_{c}^{\infty} (\M)$, 
\begin{equation}\label{p-diff-induction}
\left(\int_{\M}f(y)p^{\l}(T, x, y)\ d{\rm{Vol}}^{\l}(y)\right)^{(1)}_{\l}=\int_{\M}f(y) \phi_{\l}^{1}(T, x, y)p^{\l}(T, x, y)\ d{\rm{Vol}}^{\l}(y).
\end{equation}
Then,  Theorem \ref{regu-p-1st} {\rm i)}  holds  true. 
 \end{lem}
\begin{proof}Let $T\in \Bbb R_+$ and $x\in \M$.  
If  (\ref{p-diff-induction}) is true, then for any $f\in C_c^{\infty}(\M)$,
\begin{align}
\notag&\int_{\M}f(y)\left(p^{\l}(T, x, y)\rho^{\l}(y)-p^0(T, x, y)\rho^0(y)\right) d{\rm Vol}^0(y)\\
\notag&\  =\int_{0}^{\l}\int_{\M}f(y)\phi_{\wt{\l}}^1(T, x, y)p^{\wt{\l}}(T, x, y)\rho^{\wt{\l}}(y)\ d{\rm Vol}^0(y)d\wt{\l}\\
&\ =\int_{\M}f(y)\left(\int_{0}^{\l}\phi_{\wt{\l}}^1(T, x, y)p^{\wt{\l}}(T, x, y)\rho^{\wt{\l}}(y)d\wt{\l} \right)\ d{\rm Vol}^0(y),\label{inte-equal-cont}
\end{align}
where $\rho^{\wt{\l}}=d{{\rm{Vol}}^{\wt{\l}}}/d{\rm{Vol}^0}$ and  the second equality holds by Fubini theorem.  Note that if a continuous function $\phi$ is such that  $\int_{\M}\phi(y) f(y)\ d{\rm Vol}(y)=0$ for all $f\in C_c^{\infty}(\M)$ and a volume element ${\rm Vol}$ of a $C^2$ Riemannian metric, then $\phi$ is zero.
Hence we can conclude from (\ref{inte-equal-cont}) that 
\begin{equation}\label{p-diff-1'}
p^{\l}(T, x, y)\rho^{\l}(y)-p^0(T, x, y)\rho^0(y)=\int_{0}^{\l}\phi_{\wt{\l}}^1(T, x, y)p^{\wt{\l}}(T, x, y)\rho^{\wt{\l}}(y)d\wt{\l}
\end{equation}
since the functions appearing on both sides  are all continuous in $y$-variable and  $\l$-variable.  Since $\l\mapsto \rho^{\l}$ is $C^k$,  (\ref{p-diff-1'}) implies the existence of  $(p^{\l})^{(1)}_{\l}(T, x, y)$  for every $y$ and 
\begin{equation}\label{p-diff-induction-2}
(p^{\l})^{(1)}_{\l}(T, x, y)\cdot \rho^{\l}(y)+p^{\l}(T, x, y)\cdot (\rho^{\l})^{(1)}_{\l}(y)=\phi_{\l}^{1}(T, x, y)p^{{\l}}(T, x, y)\rho^{\l}(y). 
\end{equation}  
Then (\ref{p-diff-induction-2}) implies  that   $(p^{\l})^{(1)}_{\l}(\cdot, x, \cdot)$ is  a continuous function on $\Bbb R_+\times \M$ since we have the continuity in the $(T, y)$-coordinate of  both $p^{\l}(T, x, y)$ and $\phi_{\l}^1(T, x, y)$ by assumption.  

 Shrinking the neighborhood $\mathcal{V}_g$ of $g$ if necessary, we may assume  there is $\iota>0$ such that $p^{\l}(T, x, \cdot)\in C^{k, \iota}(\M)$ for all $\l$.  Since it  is a local problem, for $(T, y)\in \Bbb R_+\times \M$, we can also restrict ourselves to a bounded domain $\mathcal{D}$ containing $(T, y)$. Note that  $
{\rm L}^{\l}p^{\l}=0$. 
Lemma \ref{Friedman-lem-2} implies  $|p^{\l}(T, x, \cdot)|_{0, 2+\iota}<\infty$ on $\mathcal{D}$.  For each $x\in \M$, since $(p^{\l})^{(1)}_{\l}(T, x, y)$ is continuous  in $(T, y)$,  its weak derivatives in $(T, y)$ of any order are well-defined. So 
\begin{equation}\label{induction-p-1}
{\rm L}^{\l,{\rm{w}}}(p^{\l})^{(1)}_{\l}(T, x, \cdot)=({\rm L}^{\l})^{(1), {\rm{w}}}p^{\l}(T, x, \cdot)=({\rm L}^{\l})^{(1)}_{\l}p^{\l}(T, x, \cdot). 
\end{equation}
We can handle the equation locally. 
Shrinking the domain $\mathcal{D}$ to $\mathcal{D}_1$ if necessary, we deduce from $|p^{\l}(T, x, \cdot)|_{0, 2+\iota}<\infty$ on $\mathcal{D}$  that $|({\rm L}^{\l})^{(1)}_{\l}p^{\l}(T, x, \cdot)|_{2, \iota}<\infty$ on $\mathcal{D}_1$.  
Since $(p^{\l})^{(1)}_{\l}(T, x, \cdot)$ is continuous,  Lemma \ref{weak-strong-same} implies that (\ref{induction-p-1}) holds true in the usual sense, i.e., 
\begin{equation}\label{p-lambda-1}
{\rm L}^{\l}(p^{\l})^{(1)}_{\l}(T, x, \cdot)=-({\rm L}^{\l})^{(1)}_{\l}p^{\l}(T, x, \cdot). 
\end{equation}
Then we can apply Lemma \ref{Friedman-lem-2} to conclude that  $|(p^{\l})^{(1)}_{\l}(T, x, \cdot)|_{0, 2+\iota}<\infty$ on $\mathcal{D}_1$ and apply Lemma \ref{Friedman-lem} to conclude that  $(p^{\l})^{(1)}_{\l}(T, x, \cdot)\in C^{k, \iota}(\mathcal{D}_1)$. The norms of $(p^{\l})^{(1)}_{\l}(T, x, \cdot)$ in $C^{k, \iota}(\mathcal{D}_1)$ are locally uniformly bounded in $\l$ by using (\ref{p-lambda-1}), Lemma \ref{Friedman-lem} and Lemma \ref{Friedman-lem-2}. So the continuity of $\l\mapsto (p^{\l})^{(1)}_{\l}(T, x, \cdot)$ in $C(\M)$ is improved to the continuity in $C^{k, \iota}(\M)$. 
\end{proof}

For Theorem \ref{regu-p-1st}, it remains to find a  candidate  $\phi^1_{\l}(T, x, y)$ for Lemma \ref{weak-reg-p-1}. Let   $x\in\M$ and let $\lfloor {\rm{u}}_0\rceil^{\l}\in \mathcal{O}^{\wt{g}^{\l}}_x(\M)$. Recall that the  solution  to the SDE
\begin{equation*}
d\lfloor {\rm{u}}_t\rceil^{\l}=\sum_{i=1}^{m}H^{\l}(\lfloor {\rm{u}}_t\rceil^{\l}, e_i)\circ dB_t^i(w)
\end{equation*}
with initial value $\lfloor {\rm{u}}_0\rceil^{\l}$ projects to be the Brownian motion  $\lfloor {\rm x}_t\rceil^{\l}$ on $\M$ starting from $x$ and the heat kernel function $p^{\l}(T, x, \cdot)$ is just the density of the distribution of $w\mapsto \lfloor {\rm x}_T\rceil^{\l}(w)$ under ${\rm Q}$. Hence for any $f\in C_c^{\infty}(\M)$, we have \[
\int_{\M} f(y)p^{\l}(T, x, y)\ d{\rm Vol}^{\l}(y)=\E\big(f(\lfloor {\rm x}_T \rceil^{\l}(w))\big)
\]
and the equality continues to hold if we  differentiate both sides in $\l$. 
Choose $\l\mapsto \lfloor {\rm{u}}_0\rceil^{\l}$ to be a $C^{k-2}$ curve. By Lemma \ref{u-lambda-differential-1},   for almost all $w$  and all  $t\in \Bbb R_+$,   $\l\mapsto \lfloor {\rm{u}}_t\rceil^{\l}(w)$  is $C^{k-2}$. By Proposition  \ref{est-norm-u-t-j}, the differentials $(\lfloor {\rm{u}}_t\rceil^{\l})^{(j)}_{\l}(w)$, $j\leq k-2$,  are  $L^1$ integrable, uniformly  in $\l$. Hence, 
\begin{align}
& \left(\E\big(f(\lfloor {\rm x}_T \rceil^{\wt{\l}}(w))\big)\right)^{(1)}_{\l}\notag\\
&\ \ =\  \E\left(\left\langle \nabla^{\l}_{\lfloor {\rm x}_T \rceil^{\l}(w)} (f\circ \pi)(\lfloor {\rm{u}}_T\rceil^{\l}(w)), (\lfloor {\rm{u}}_T\rceil^{\l})^{(1)}_{\l}(w) \right\rangle_{\l}\right)\label{indentity-diff-expect}\\
&\ \ =\  \int_{\M} \E\left(\left.\big\langle \nabla^{\l}_y f(y), D\pi(\lfloor {\rm{u}}_T\rceil^{\l})^{(1)}_{\l}(w)\big\rangle_{\l}\right| \lfloor {\rm x}_T \rceil^{\l}(w)=y\right)\cdot p^{\l}(T, x, y)\ d{\rm Vol}^{\l}(y). \label{put-random}
\end{align}
Note that  (\ref{put-random}) holds for every choice of $\lfloor {\rm{u}}_0\rceil^{\l}\in \mathcal{O}^{\wt{g}^{\l}}(\M)$ at $\l$. For some technical considerations which we will mention later, we  choose $\lfloor {\rm{u}}_0\rceil^{\l}$ in $\mathcal{O}^{\wt{g}^{\l}}(\M)$ at random with a uniform distribution normalized to be probability 1 and then choose   
\[
\lfloor {\rm{u}}_0\rceil^{\wt{\l}}=\lfloor \overline{{\rm{u}}}_0\rceil^{\wt{\l}}(\lfloor \overline{{\rm{u}}}_0\rceil^{\l})^{-1}\lfloor {\rm{u}}_0\rceil^{\l},\ \wt{\l}\in (-1, 1), 
\]
where $\lfloor \overline{{\rm{u}}}_0\rceil^{\wt{\l}}$ is some fixed $C^k$ curve in $\mathcal{F}(\M)$ with $\lfloor \overline{{\rm{u}}}_0\rceil^{\wt{\l}}\in \mathcal{O}^{\wt{g}^{\wt{\l}}}(\M)$. Write $\overline{\E}$ for the new expectation when the random choices of $\lfloor {\rm{u}}_0\rceil^{\l}$ are taken into account. Then 
 \begin{align}\label{put-random-1}
\left(\E f(\lfloor {\rm x}_T \rceil^{\wt{\l}})\right)^{(1)}_{\l}\!=\int_{\M} \overline{\E}\left(\!\left.\left\langle \nabla^{\l}_y f(y), D\pi(\lfloor {\rm{u}}_T\rceil^{\l})^{(1)}_{\l}(w)\right\rangle_{\l}\right| \lfloor {\rm x}_T \rceil^{\l}(w)=y\right)\cdot p^{\l}(T, x, y)\ d{\rm Vol}^{\l}(y).
\end{align}
For any $C^{k}$ bounded vector field $Y$ on $\M$, let 
\begin{equation}\label{Phi-lam-1-Y}
\overline{\Phi}_{\l}^1(Y)(y):=\overline{\E}\left(\left.\Phi_{\l}^1(Y, w)\right| \lfloor {\rm x}_T \rceil^{\l}(w)=y\right),
\end{equation}
where 
\begin{equation}\label{def-phi-1}
\Phi_{\l}^1(Y, w):=\big\langle Y(\lfloor {\rm x}_T \rceil^{\l}(w)), D\pi(\lfloor {\rm{u}}_T\rceil^{\l})^{(1)}_{\l}(w)\big\rangle_{\l}. 
\end{equation}
We will show the linear functional $\overline{\Phi}_{\l}^1$ is such that $\overline{\Phi}_{\l}^1(Y)$ is $C^1$ in $y$ variable, from which we can  deduce that \begin{equation}\label{z-T-lam-1}
{\rm z}_T^{\l, 1}(y):=\overline{\E}\left(\left.D\pi(\lfloor {\rm{u}}_T\rceil^{\l})^{(1)}_{\l}(w)\right| \lfloor {\rm x}_T \rceil^{\l}(w)=y\right)
\end{equation}
is a  $C^1$ vector field on $\M$.  Hence,  we can apply the the classical  integration by parts  formula   to   (\ref{put-random-1}) and compute that 
\begin{align*}
&\left(\E(f(\lfloor {\rm x}_T \rceil^{\wt{\l}}(w)))\right)^{(1)}_{\l}\\
&\ \ \  =-\int_{\M} f(y)\left(({\rm Div}^{\l}{\rm z}_T^{\l, 1}(y))+\big\langle {\rm z}_T^{\l, 1}(y), \nabla^{\l}\ln p^{\l}(T, x, y) \big\rangle_{\l}\right)p^{\l}(T, x, y) \ d{\rm Vol}^{\l}(y). 
\end{align*}
This gives a candidate of $\phi_{\l}^{1}$ for Lemma \ref{weak-reg-p-1} as 
\begin{equation}\label{phi-1-candidate}
 \phi_{\l}^{1}(T, x, y):=-\left(({\rm Div}^{\l}{\rm z}_T^{\l, 1}(y))+\big\langle {\rm z}_T^{\l, 1}(y), \nabla^{\l}\ln p^{\l}(T, x, y) \big\rangle_{\l}\right).
\end{equation}

To justify that (\ref{phi-1-candidate}) is well-defined, we  need to  show the $C^1$ dependence of $\overline{\Phi}_{\l}^1(Y)(y)$ in $y$-variable.  Let ${\rm V}$ be a  smooth  bounded vector field on $\M$ and let $\{F^s\}_{s\in \Bbb R}$ be the flow it generates.
To  compare $\overline{\Phi}_{\l}^1(Y)(F^s(y))$ with $\overline{\Phi}_{\l}^1(Y)(y)$, our strategy is to  extend every map $F^s$ on $\lfloor {\rm x}_T \rceil^{\l}(w)$, the endpoint of Brownian motion paths at time $T$,  to be a map ${\bf F}^s$ on Brownian paths up to time $T$. Let $\overline{\P}_x^{\l}$ denote the product of the probability $\P_x^{\l}$ with the uniform probability on $\mathcal{O}^{\wt{g}^{\l}}(\M)$ for the choice of $\lfloor {\rm u}_0\rceil^{\l}$. We will ensure the maps ${\bf F}^s$ are such that $\overline{\P}_x^{\l}\circ {\bf F}^s$ are absolutely continuous with respect to $\overline{\P}_x^{\l}$. Clearly,   for any bounded measurable function $f$  on $\M$, 
\begin{align}\label{L-R-compare}
\overline{\E}\left(\Phi_{\l}^1 (Y, w)f( \lfloor {\rm x}_T\rceil^{\l}(w))\right)=& \overline{\E}\left(\overline{\Phi}_{\l}^1(Y)(y) f(y)\right)\\
\label{L-R-compare-L} =& \int \overline{\Phi}_{\l}^1(Y)(F^s(y)) f(F^s(y)) p^{\l}(T, x, F^s y)\ d{\rm Vol}^{\l}(F^s(y)), 
\end{align}
where the first equality holds  by the definition of conditional measures and the second equality holds by changing the variable  to $F^s(y)$. 
The left hand side of (\ref{L-R-compare}), after a change of variable under ${\bf F}^s$, is equal to 
\begin{align}
\notag&\overline{\E}\left(\Phi_{\l}^1 \circ {\bf F}^s \cdot f\circ {\bf F}^s\cdot \frac{d\overline{\P}_x^{\l}\circ {\bf F}^s}{d\overline{\P}_x^{\l}}\right)\\
&=\int \overline{\E}\left(\left.\Phi_{\l}^1 \circ {\bf F}^s \cdot \frac{d\overline{\P}_x^{\l}\circ {\bf F}^s}{d\overline{\P}_x^{\l}}\right|\lfloor {\rm x}_T \rceil^{\l}(w)=y\right)\cdot f(F^s(y)) p^{\l}(T, x, y)\ d{\rm Vol}^{\l}(y). \label{L-R-compare-R}
\end{align}
Since $f$ is arbitrary, a comparison of (\ref{L-R-compare-L}) with (\ref{L-R-compare-R}) implies that 
\begin{align}
\overline{\Phi}_{\l}^1(Y)(F^s(y))=& \overline{\E}\left(\left.\Phi_{\l}^1\right| \lfloor {\rm x}_T \rceil^{\l}(w)=F^s(y)\right)\notag\\
=&\overline{\E}\left(\left.\Phi_{\l}^1\circ {\bf F}^s \cdot \frac{d\overline{\P}_x^{\l}\circ {\bf F}^s}{d\overline{\P}_x^{\l}}\right|\lfloor {\rm x}_T \rceil^{\l}(w)=y\right)\frac{p^{\l}(T, x, y)}{p^{\l}(T, x, F^s y)}\frac{d{\rm Vol}^{\l}}{d{\rm Vol}^{\l}\circ F^s}(y).\label{Ch-Var-compare}
\end{align}
Note that $p^{\l}(T, x, y)$ and  the volume element ${\rm Vol}^{\l}$  are $C^k$ in  the $y$ variable. So the differentiability in the $s$ parameter of $\overline{\Phi}_{\l}^1(Y)(F^s(y))$ will follow from  the differentiability in the $s$ parameter of  
\begin{equation}\label{L-R-compare-F}\overline{\E}\left(\left.\Phi_{\l}^1 \circ {\bf F}^s \cdot \frac{d\overline{\P}_x^{\l}\circ {\bf F}^s}{d\overline{\P}_x^{\l}}\right|\lfloor {\rm x}_T \rceil^{\l}(w)=y\right).
\end{equation}
In order to show this differentiability in  $s$, we will show that  our one-parameter family of maps  ${\bf F}^s$ satisfy  the following properties (see Proposition \ref{Quasi-IP-2}, Proposition \ref{Diff-int-U-s} and Proposition \ref{Diff-int-U-s-gen}),  where  all the integrals  are taken with respect to $\overline{\P}_x^{\l}$ conditioned on $\lfloor {\rm x}_T \rceil^{\l}(w)=y$.

\begin{itemize}
\item[i)] $\overline{\P}_x^{\l}\circ {\bf F}^s$ is absolutely continuous with respect to $\overline{\P}_x^{\l}$ and the Radon-Nikodyn derivative $d\overline{\P}_x^{\l}\circ {\bf F}^s/d\overline{\P}_x^{\l}$ is  $L^q$ integrable for $q\geq 1$,  locally uniformly in the $s$ parameter,
\item[ii)]the differential of $d\overline{\P}_x^{\l}\circ {\bf F}^s/d\overline{\P}_x^{\l}$ in $s$ is $\overline{\mathcal{E}^s_T}\cdot (d\overline{\P}_x^{\l}\circ {\bf F}^s/d\overline{\P}_x^{\l})$, where $\overline{\mathcal{E}^s_T}$ is  $L^q$ integrable for $q\geq 1$, locally uniformly in the $s$ parameter, and   
\item[iii)] $(\lfloor {\rm{u}}_T\rceil^{\l})^{(1)}_{\l}\circ {\bf F}^s$ is differentiable in $s$  with the differential stochastic process $L^q$ integrable for $q\geq 1$, locally uniformly  in the $s$ parameter.
\end{itemize}
With these three properties, we will obtain 
 \begin{equation*}
\left(\Phi_{\l}^1 \circ {\bf F}^s \cdot \frac{d\overline{\P}_x^{\l}\circ {\bf F}^s}{d\overline{\P}_x^{\l}}\right)'_s=\Phi_{\l}^1 \circ {\bf F}^s \cdot \overline{\mathcal{E}^s_T}\cdot \frac{d\overline{\P}_x^{\l}\circ {\bf F}^s}{d\overline{\P}_x^{\l}}+\left(\Phi_{\l}^1 \circ {\bf F}^s\right)'_s\frac{d\overline{\P}_x^{\l}\circ {\bf F}^s}{d\overline{\P}_x^{\l}}
\end{equation*}
and this differential is absolutely integrable, locally uniformly  in  the $s$ parameter.  Hence (\ref{L-R-compare-F}) is differentiable in the $s$ parameter and we are allowed to take the differential inside the expectation sign.  The uniform continuity of ${\rm z}_T^{\l, 1}(y) $ and $ {\rm Div}^{\l}{\rm z}_T^{\l, 1}(y)$ in $T$ and $y$  will follow from (see the proof of Theorem \ref{regu-p-1st} with $k=3$)
\begin{itemize}
\item[iv)]  the uniform continuity in $T$ and $y$ of 
\[
\overline{\E}\big(\left.\Phi_{\l}^1\right| \lfloor {\rm x}_T \rceil^{\l}(w)=y\big)\ \ \mbox{and}\  \ \overline{\E}\left(\left.\big(\Phi_{\l}^1 \circ {\bf F}^s \cdot \frac{d\overline{\P}_x^{\l}\circ {\bf F}^s}{d\overline{\P}_x^{\l}}\big)'_0\right|\lfloor {\rm x}_T \rceil^{\l}(w)=y\right). 
\]
\end{itemize}

The major part of the remaining subsections  is devoted to the construction of ${\bf F}^s$ and the verification of its properties  i)-iv) mentioned above, which will conclude i) of Theorem \ref{regu-p-1st}. We will discuss  Theorem \ref{regu-p-1st} ii) and iii)   in the last subsection. 

Fix $T>0$.  For each $y\in \M$, we will construct a one parameter family  of maps ${\bf F}^s_y$ on Brownian motion paths starting from $y$ up to time $T$ with ${\bf F}^{s}_{y, x}$ being  its conditional map on paths that will arrive at $x$ in time $T$. We will achieve this in two steps: one for  the   SDE description of ${\bf F}^s_y$ and the other  for its  existence by Picard's iteration argument.  The desired map ${\bf F}^{s}$ will be the collection of all ${\bf F}^{s}_{{\rm x}_T, x}$. But, we need to justify the meaning of  $\Phi_{\l}^1 \circ {\bf F}^s$ and ${d\overline{\P}_x^{\l}\circ {\bf F}^s}/{d\overline{\P}_x^{\l}}$ since  $\Phi_{\l}^1$  and $\overline{\P}_x^{\l}$ are associated with the diffusion paths from $x$. 
This and the verification of i)-iv) will be done in Sections \ref{QIPF-y-s} and \ref{the flow F-S}.  Finally, in Section \ref{TDOp-lam}, we will show the assumption of Lemma \ref{weak-reg-p-1} is satisfied and will give the estimations in  (\ref{grad-lnp-lam-1})  by an  analysis of ${\rm z}_T^{\l, 1}(y)$  and ${\rm Div}^{\l}{\rm z}_T^{\l, 1}(y)$ using the SDE theory.

\subsection{A description of  ${\bf F}_y^s$}\label{flow-F-S-y} In this part,  we fix  $T\in \Bbb R^+$. 
Let $y\in \M$ and $\overline{\beta}_0\in \mathcal{O}^{\wt{g}}_y(\M)$.  For a smooth segment $t\mapsto \alpha_t=(\alpha_{t, 1}, \cdots, \alpha_{t, m})\in \Bbb R^m, t\in [0, T]$,  with $\alpha_0=o$,   let 
$\overline{\beta}=(\overline{\beta}_t)_{t\in [0, T]}$ in  $\mathcal{O}^{\wt{g}}(\M)$ be the unique smooth segment with initial  $\overline{\beta}_0$ satisfying the differential equation 
\[
\nabla_{\frac{D}{\partial t}}\overline{\beta}_t=\sum_{i=1}^{m}H(\overline{\beta}_t, e_i)\cdot \frac{d\alpha_{t, i}}{dt}.
\]
In the language  of Section \ref{Sec-BM-SF},   this means that $\overline{\beta}$ is the transportation (or development) of $\alpha$ in  $\mathcal{O}^{\wt{g}}(\M)$  with starting point  $\overline{\beta}_0$ using  the parallelism differential form $(\theta, \varpi)$. The It\^{o} map $\mathcal{I}_{\overline{\beta}_0}:\ C_{o}^{\infty}([0, T], \Bbb R^m)\to C_{y}^{\infty}([0, T], \M)$ is given by 
\begin{equation}\label{ITO-MAP}\mathcal{I}_{\overline{\beta}_0}(\alpha):=\pi(\overline{\beta})=\beta, 
\end{equation}
where $\pi$ is the projection map from $\mathcal{O}^{\wt{g}}_y(\M)$ to $\M$. 
It is invertible  since  $\alpha$ for (\ref{ITO-MAP})   can be  uniquely determined by the equation \[
\frac{d\alpha_t}{dt}=(\overline{\beta}_t)^{-1}\nabla_{\frac{D}{\partial t}}{\beta}_t,\]
where $\overline{\beta}\subset \mathcal{O}^{\wt{g}}(\M)$  is the horizontal lift of $\beta$ with initial value $\overline{\beta}_0$, i.e., 
\[
\nabla_{\frac{D}{\partial t}}\overline{\beta}_t=H(\overline{\beta}_t, {\overline{\beta}}_t^{-1}\nabla_{\frac{D}{\partial t}}{\beta}_t). 
\]
For $\beta\in C_{y}^{\infty}([0, T], \M)$, its $\mathcal{I}$-preimage $\mathcal{I}_{\overline{\beta}_0}^{-1}(\beta)$ is called the \emph{anti-development} of $\beta$ in $\Bbb R^m$.

For a smooth segment (or curve) $\beta=(\beta_t)_{t\in [0, T]}$ on $\M$,  the classical \emph{parallel transportation map} $\sslash^{\beta}_{\smaller{t_1, t_2}}$ of tangent vectors along the  segments $({\beta}_t)_{t\in [t_1, t_2]}$ ($0\leq t_1\leq t_2\leq T$)  is given by 
\[
\sslash^{\beta}_{\smaller{t_1, t_2}}({\bf v})={\overline{\beta}}_{t_2}\circ \overline{\beta}_{t_1}^{-1}({\bf v}), \ \forall {\bf v}\in T_{\beta_{t_1}}\M, 
\]
where $\overline{\beta}$  is a horizontal lift of $\beta$. 
This definition is independent of the horizontal lift chosen since if $\overline{\beta}'$ is another horizontal lift of $\beta$, then $\overline{\beta}'_t={\overline{\beta}}_t{\overline{\beta}_0}^{-1}\overline{\beta}'_0$ for $t\in [0, T]$ and hence 
\[
\overline{\beta}'_{t_2}\circ (\overline{\beta}'_{t_1})^{-1}={\overline{\beta}}_{t_2}{\overline{\beta}_0}^{-1}\overline{\beta}'_0\circ \big((\overline{\beta}'_0)^{-1}\overline{\beta}_0 {\overline{\beta}}_{t_1}^{-1}\big)= {\overline{\beta}}_{t_2}\circ \overline{\beta}_{t_1}^{-1}.
\]

The  It\^{o} map and the parallel transportation map can also be defined in the stochastic case.  Call an $\mathcal{O}^{\wt{g}}(\M)$-valued  continuous stochastic process $\overline{\beta}=(\overline{\beta}_t)_{t\in [0, T]}$   \emph{horizontal} if there exists  a  $\Bbb R^m$-valued continuous stochastic process $\alpha=(\alpha_{t, 1}, \cdots, \alpha_{t, m})_{t\in [0, T]}$ with $\alpha_0=o$ such that $\overline{\beta}$ solves the Stratonovich SDE 
\begin{equation}\label{Parallel-general}
d\overline{\beta}_t=\sum_{i=1}^{m}H(\overline{\beta}_t, e_i)\circ d\alpha_{t, i}. 
\end{equation}
For a continuous stochastic process $({\beta}_t)_{t\in [0, T]}$  on $\M$,  its \emph{horizontal lifts} are those horizontal processes $\overline{\beta}$ in $\mathcal{O}^{\wt{g}}(\M)$ projecting to it and its \emph{anti-developments} in $\Bbb R^m$ are those $\alpha$ satisfying (\ref{Parallel-general}) (cf. \cite{Hs}).  For a fixed $y\in \M$ and $\overline{\beta}_0\in \mathcal{O}^{\wt{g}}_y(\M)$, (\ref{Parallel-general}) is uniquely solvable for every semi-martingale $\alpha$ and the \emph{It\^{o} map} 
\[
\mathcal{I}_{\overline{\beta}_0}(\alpha):=\pi(\overline{\beta})=\beta
\]
is well-defined.  In the sprit  of Section \ref{Sec-BM-SF},   $\mathcal{I}_{\overline{\beta}_0}(\alpha)$ is the projection process of a transportation (or development) of $\alpha$ in  $\mathcal{O}^{\wt{g}}(\M)$ using  the parallelism differential form $(\theta, \varpi)$.
The one-to-one correspondence between $\alpha, \beta, $ and $ \overline{\beta}$ for semi-martingales is discussed in \cite{Hs}.

For a  semi-martingale $\beta=(\beta_t)_{t\in [0, T]}$  on $\M$, its horizontal lifts $\overline{\beta}$ are  uniquely determined by the distribution of $\overline{\beta}_0$ (cf. \cite[Theorem 2.3.5]{Hs}). Hence, for almost all ${\rm w}\in \Theta_+$, we can define  a stochastic  `parallel transportation map'  $\sslash^{\beta}_{\smaller{t_1, t_2}}$ of tangent vectors along the path segments $({\beta}_t({\rm w}))_{t\in [t_1, t_2]}$ ($0\leq t_1\leq t_2\leq T$) by letting 
\[
\sslash^{\beta}_{\smaller{t_1, t_2}}({\bf v}):={\overline{\beta}}_{t_2}\circ \overline{\beta}_{t_1}^{-1}({\bf v}), \ \forall {\bf v}\in T_{\beta_{t_1}({\rm w})}\M. 
\]
As in the deterministic case, this definition is independent of the horizontal lift $\overline{\beta}$ chosen. 
Each $\sslash^{\beta}_{\smaller{t_1, t_2}}$  is an isometry between $T_{{\beta}_{t_1}({\rm w})}\M$ and $T_{{\beta}_{t_2}({\rm w})}\M$ with the inverse map 
\[
(\sslash_{\smaller{t_1, t_2}}^{\beta})^{-1}({\bf v}'):= {\overline{\beta}}_{t_1}\circ {\overline{\beta}}_{t_2}^{-1}({\bf v}'), \ \forall  {\bf v}'\in T_{{\beta}_{t_2}({\rm w})}\M. 
\]
Moreover, the  parallel transportation maps $\sslash_{\smaller{t_1, t_2}}^{\beta}$ also satisfy the cocycle property
\[
\sslash_{\smaller{t_1, t_3}}^{\beta}=\sslash_{\smaller{t_3, t_2}}^{\beta}\circ \sslash_{\smaller{t_1, t_2}}^{\beta}, \ \forall 0\leq t_1\leq t_2\leq t_3\leq T. 
\]

Let ${\rm V}$ be a smooth bounded vector field on $\M$.  For each $y\in \M$, we obtain a smooth curve $s\mapsto F^s(y), s\in \Bbb R$,  with 
 \[
 \frac{dF^s(y)}{ds}={\rm V}(F^s(y)). \]
Let $T>0$ be fixed and let $\{({\rm y}_t, \mho_t)\}_{t\in [0, T]}$ be a stochastic pair which defines the $\wt{g}$-Brownian motion starting from $y$. The mapping ${\rm w}\mapsto ({\rm y}_t({\rm w}))_{t\in [0, T]}$ gives the distribution of Brownian paths up to time $T$ in $C_{y}([0, T], \M)$, where we  use ${\rm w}$ to differ it from $w$ for $\lfloor {\rm x}_t\rceil^{\l}$.  
 We want to construct a one parameter family  of mappings ${\bf F}^s_y$ on Brownian distributions $({\rm y}_t)_{t\in [0, T]}$ so that 
 \[{\rm y}^s_t({\rm w}):=\left({\bf F}^s_y({\rm y}_{[0, T]}({\rm w}))\right)(t),\  \forall  t\in [0, T], \]  is differentiable in the $s$ `direction'  for almost all ${\rm w}$ with  initial restriction $
{d{\rm y}_0^s}/{ds}={\rm V}(F^s(y)).$ 
 
Choose a $C^1$ function $\mathtt{s}:\ [0, T]\rightarrow \Bbb R_+$  with 
\begin{align}\label{mathtts-cond}
\mathtt{s}(0)=1, \mathtt{s}(T)=0 \ \mbox{and}  \ \varlimsup_{t\to T}\frac{1}{T-t}\mathtt{s}(t)<\infty. 
\end{align}
For almost all ${\rm w}$,  we obtain a vector field along  the paths ${\rm y}_{[0, T]}({\rm w})$ with 
\[
\Upsilon_{{\rm V}, {\rm y}}(t):=\mathtt{s}(t)\cdot \sslash_{0, t}({\rm V}(y)),\  t\in [0, T],\]
where \[
\sslash_{\smaller{t_1, t_2}}({\bf v}):={\mho}_{t_2}\circ \mho_{t_1}^{-1}({\bf v}), \ \forall {\bf v}\in T_{{\rm y}_{t_1}({\rm w})}\M, \ 0\leq t_1\leq t_2\leq T. 
\]
Our desired maps ${\bf F}^s_y$ on  $({\rm y}_t({\rm w}))_{t\in [0, T]}$ are such that $({\rm y}^s_t({\rm w}))_{t\in [0, T]}$ satisfy the equation 
\begin{equation}\label{flow F_s-on M}
\frac{d {\rm y}^s_t({\rm w})}{ds}=\Upsilon_{{\rm V}, {\rm y}^s}(t),
\end{equation}
where 
\[
\Upsilon_{{\rm V}, {\rm y}^s}(t):=\mathtt{s}(t)\cdot \sslash_{0, t}^s({\rm V}({\rm y}_0^s)),\  t\in [0, T], \]
and $\sslash^s$ denotes the parallel transportation map for the process ${\rm y}^s$. The length of tangent vectors remain unchanged under parallel transportations.
Hence 
\[
\|\Upsilon_{{\rm V}, {\rm y}^s}(t)\|=\mathtt{s}(t)\cdot\|{\rm V}({\rm y}_0^s)\|\leq  \mathtt{s}(t)\sup\{\|V(y)\|\}, 
\]
which tends to zero of order $(T-t)$ as $t\to T$ by our choice of $\mathtt{s}$. So,  if the processes ${\rm y}^s$ exist, the ending points ${\rm y}^s_T$ remain in    ${\rm y}_T$.

\begin{remark}
In \cite{Hs}, Hsu introduced a  class of maps  for the  Brownian motion starting from some  point $y$ on a compact manifold:  in our notation, 
\[
\Upsilon_{{\rm V}, {\rm y}^s}(t)=\mho_t(\dot{h}(t)), 
\]
where $h$ is a fixed $\Bbb R^m$ valued curve  from the Euclidean Cameron-Martin space, i.e.,  the completion of the space of  smooth paths $h: [0, T]\mapsto \Bbb R^m$ starting from the origin $o$ with the Hilbert norm $|h|=(\int_0^1 |\dot{h}(t)|^2)^{\frac{1}{2}}$. In his construction,   the initial point ${\rm y}^s_0$ remain unchanged since $h$ starts from $o$ and hence the equations of all ${\rm y}^s$ can be transferred back to $\Bbb R^m$ using a single It\^{o} map at $y$.  In contrast,  in our construction, our manifold is non-compact and we use a vector field $\rm V$ on the manifold instead of a Euclidean Cameron-Martin space element $h$ to generate the random vector field $\Upsilon_{{\rm V}, {\rm y}^s}$. Our ends ${\rm y}^s_T$ remain unchanged for almost all paths  since  ${\mathtt s}(t)$ tends to zero as $t$ goes to $T$;  while  the  initials ${\rm y}^s_0$  changes  as $s$ varies so the It\^{o} transfer map of ${\rm y}^s$ to $\Bbb R^m$ also changes with $s$.   The $C^1$ requirement of $\mathtt s(t)$  is stronger than the $L^2$ integrability of the differentials of $h(t)$. This is to guarantee that we can obtain a continuous version of the resulting process ${\rm y}_t^s$ (and all other related processes) in the parameter $(t, s)$ (see Theorem \ref{Main-alpha-x-v-Q}), which is not true for general  $h$.
\end{remark}

We will solve the SDE (\ref{flow F_s-on M}) by identifying the  anti-developments  $\alpha^s_t=\mathcal{I}^{-1}_{\mho^s_0}({\rm y}^s_t)$ using Picard's iteration method, where   $\mho_0^s$ is the  parallel transportation of $\mho_0$ along the curve $(F^s(y))_{s\in \Bbb R}$.  In many places, the transferred equations using $\rm V$ only differ  in notations from  that for the case of   $h$ in  \cite{Hs}.  But, technically, we have to write every steps in details since the construction is different, the footpoints of the It\^{o} maps are shifting,  and we need more regularity  of ${\rm y}_t^s$ and also  more information of the associated random structures.

We first consider (\ref{flow F_s-on M})  for smooth paths.  Let $y\in \M$ and  $\overline{\beta}_0\in \mathcal{O}^{\wt{g}}_y(\M)$ be fixed.  For  $\beta=(\beta_t)\in C_{y}^{\infty}([0, T], \M)$, the  equation 
\begin{equation}\label{flow-beta}
\frac{\partial\beta^s_t}{\partial s}=\Upsilon_{{\rm V}, \beta^s}(t) :=\mathtt{s}(t)\cdot \sslash_{0, t}^s({\rm V}(F^sy)), \ \beta^0=\beta, 
\end{equation}
where $\sslash_{0, t}^s:=\sslash_{0, t}^{\beta^s}$, is always solvable.   Consider   $\alpha^s_t=\mathcal{I}^{-1}_{\overline{\beta}^s_0}(\beta^s_t)$, where  $\overline{\beta}_0^s$ is the  parallel transportation of $\overline{\beta}_0$ along the curve $s\mapsto F^s(y)$.   Then  $\left.(\partial\alpha^s_t/\partial s)\right|_{s=0}$ differs from $\overline{\beta}_t^{-1}(\Upsilon_{{\rm V}, \beta^0}(t))$ by an integral of curvature term, which can be determined  by a standard calculation exactly as in   \cite[Theorem 2.1]{Hs1}. We give the proof for completeness.

\begin{lem}\label{HS-Th2.1} Let ${\rm V}$ be a smooth bounded vector field on $\M$.  For  $\beta\in C_{y}^{\infty}([0, T], \M)$, let $\beta^{s}$ be the solution to (\ref{flow-beta})  and let   $\overline{\beta}^s$  be its horizontal lift in $\mathcal{O}^{\wt{g}}(\M)$ with initial point $\overline{\beta}_0^s$. 
\begin{itemize}
\item[i)]The differential $(\alpha_t^s)'_s:={\partial \alpha^s_t}/{\partial s}$ is given by 
\[
(\alpha_t^s)'_s=\Upsilon_{{\rm V}, \alpha^s}(t):=\int_{0}^{t}\mathtt{s}'(\tau)(\overline{\beta}_0^s)^{-1}\big({\rm V}(\beta_0^s)\big)\ d\tau-\int_{0}^{t}K_{{\rm V}, \alpha^s}(\tau)\ d\alpha_{\tau}^s,
\]
where 
 \begin{align*}
  K_{{\rm V}, \alpha^s}(\tau) &= \int_{0}^{\tau} (\overline{\beta}_{\varsigma}^s)^{-1}R\big(\overline{\beta}_{\varsigma}^s(\frac{\partial \alpha_\varsigma^s}{\partial \varsigma}), \Upsilon_{{\rm V}, \beta^s}(\varsigma)\big)\overline{\beta}_\varsigma^s\ d\varsigma.
  \end{align*}
 \item[ii)]The differential  $(\overline{\beta}^s_t)'_s:=\nabla_{\frac{D}{\partial s}}\overline{\beta}^s_t$ satisfies the equation \begin{align}\label{bar-beta-s-diff-1}
\left\{ \begin{array}{ll}
\nabla_{\frac{D}{\partial t}}\big(\theta\big((\overline{\beta}^s_t)'_s\big)\big)=\mathtt{s}'(t)(\overline{\beta}_0^s)^{-1}V(F^s y),\\
\nabla_{\frac{D}{\partial t}}\big(\varpi\big((\overline{\beta}^s_t)'_s\big)\big)=(\overline{\beta}_t^s)^{-1}R\left(\overline{\beta}_t^s (\frac{\partial \alpha^s_t}{\partial t}),  \overline{\beta}_t^s(\overline{\beta}_0^s)^{-1}\mathtt{s}(t){\rm{V}}(F^s y)\right)\overline{\beta}_t^s. 
\end{array}\right.
\end{align}
 \end{itemize}
\end{lem}
\begin{proof}For ${\rm i)}$, we have 
\begin{align*}
\frac{\partial}{\partial t}\left(\frac{\partial \alpha^s_t}{\partial s}\right)=\frac{\partial}{\partial s}\left(\frac{\partial \alpha_t^s}{\partial t}\right)=\nabla_{\frac{D}{\partial s}}\left(\theta\big(\nabla_{\frac{D}{\partial t}} \overline{\beta}_t^s\big)\right). 
\end{align*}
Using the exterior differentiation  formula in  covariant derivative (cf. \cite{GHL}) and the structure equation (\ref{structure-1})  for $\theta$,  we obtain 
\begin{align*}
\nabla_{\frac{D}{\partial s}}\left(\theta\big(\nabla_{\frac{D}{\partial t}} \overline{\beta}_t^s\big)\right)
&=\nabla_{\frac{D}{\partial t}}\left(\theta \big(\nabla_{\frac{D}{\partial s}} \overline{\beta}_t^s\big)\right)+d\theta\left(\nabla_{\frac{D}{\partial s}} \overline{\beta}_t^s,  \nabla_{\frac{D}{\partial t}} \overline{\beta}_t^s\right)\\
&=\mathtt{s}'(t) (\overline{\beta}_0^s)^{-1}({\rm V}(\beta_0^s))-\varpi\big(\nabla_{\frac{D}{\partial s}} \overline{\beta}_t^s\big)(\frac{\partial\alpha_t^s}{\partial t}).
\end{align*}
We continue to compute that 
\[
\varpi\big(\nabla_{\frac{D}{\partial s}} \overline{\beta}_t^s\big)=\int_{0}^{t}\nabla_{\frac{D}{\partial \tau}}\left(\varpi\big(\nabla_{\frac{D}{\partial s}} \overline{\beta}_\tau^s\big)\right)\ d\tau,
\]
where, by using the exterior differentiation  formula, ${\rm{Ker}}(\varpi)=HT\mathcal{F}(\M)$ and (\ref{structure-2}),  
\begin{align}
\notag\nabla_{\frac{D}{\partial \tau}}\left(\varpi\big(\nabla_{\frac{D}{\partial s}} \overline{\beta}_\tau^s\big)\right)&=\nabla_{\frac{D}{\partial s}}\left(\varpi \big(\nabla_{\frac{D}{\partial \tau}} \overline{\beta}_\tau^s\big)\right)+d\varpi\left(\nabla_{\frac{D}{\partial s}} \overline{\beta}_\tau^s,  \nabla_{\frac{D}{\partial \tau}} \overline{\beta}_\tau^s\right)\\
\notag&= \Omega\left(H({\overline{\beta}_{\tau}^s}, (\frac{\partial\alpha_\tau^s}{\partial \tau})), H\left({\overline{\beta}_{\tau}^s}, (\overline{\beta}_{0}^s)^{-1}[\mathtt{s}(\tau){\rm V}(\beta_0^s)]\right)\right)\\
\label{bar-beta-s-diff-1-pf}&= (\overline{\beta}_\tau^s)^{-1} R\left(\overline{\beta}_{\tau}^s(\frac{\partial\alpha_\tau^s}{\partial \tau}),\Upsilon_{{\rm V}, \beta^s}(\tau)\right) \overline{\beta}_\tau^s.
\end{align}
For (\ref{bar-beta-s-diff-1}), the  first equation is true by the construction. The second equation holds by (\ref{bar-beta-s-diff-1-pf}) since 
$\Upsilon_{{\rm V}, \beta^s}(\tau)=\overline{\beta}_{\tau}^s(\overline{\beta}_0^s)^{-1}\mathtt{s}(\tau){\rm{V}}(F^s y)$.  
\end{proof}

For every  smooth segment  $\alpha=(\alpha_t)_{t\in [0, T]}$ in  $\Bbb R^m$,  consider the associated flow maps $\{F_{t_1, t_2}^{\alpha}\}_{0\leq t_1<t_2\leq T}$ for the  transportation of $\alpha$ to $\M$ using the parallelism differential form $(\theta, \varpi)$, where  
$F_{t_1, t_2}^{\alpha}: \mathcal{F}(\M)\to \mathcal{F}(\M); \  \overline{\beta}_{t_1}^{\alpha}\mapsto \overline{\beta}_{t_2}^{\alpha}$ with  $(\overline{\beta}_{t}^{\alpha})_{t\in [t_1, t_2]}$ solving  the equation 
\begin{equation*}\label{parallel-flow}
\nabla_{\frac{D}{\partial t}}\overline{\beta}^{\alpha}_t=H(\overline{\beta}^{\alpha}_t, \frac{d\alpha_t}{dt}). 
\end{equation*}
Each  $F_{t_1, t_2}^{\alpha}$ is a $C^{k-1}$ diffeomorphism since $H$ is $C^{k-1}$ and $\alpha$ is smooth.  Let $DF_{t_1, t_2}^{\alpha}$  be the  tangent map of $F_{t_1, t_2}^{\alpha}$. It can be read in the $(\theta, \varpi)$-coordinate as follows. 
\begin{lem}(\cite[Proposition 3.2]{M})
Let  $\alpha=(\alpha_t)_{t\in [0, T]}\subset \Bbb R^m$ be a smooth segment. For any $t_1\in [0, T]$ and ${\bf\mathsf v}_{t_1}^\alpha\in T_{\overline{\beta}_{t_1}^\alpha}\mathcal{F}(\M)$,  let  $
{\bf\mathsf v}_t^\alpha:=\left[DF_{t_1, t}^\alpha\big(\overline{\beta}_{\tau}^\alpha, {\rm w}\big)\right]{\bf\mathsf v}_{t_1}^\alpha$ for $t\in [t_1, T]$. Then ${\bf\mathsf v}_t^\alpha$  satisfies the equation 
\begin{align}\label{DF-alpha-eq}
\nabla_{\frac{D}{\partial t}}\big({\bf\mathsf v}_t^\alpha\big)=\big(\nabla ({\bf\mathsf v}_t^\alpha)H(\overline{\beta}^\alpha_t, \cdot)\big)\frac{d\alpha_t}{d t}. 
\end{align}
In the $(\theta, \varpi)$-coordinate,  we have 
  \begin{align*}
\left\{ \begin{array}{ll}
\frac{d}{dt}\left(\theta\big({\bf\mathsf v}_t^\alpha\big)\right)= \varpi\big({\bf\mathsf v}_t^\alpha\big)\frac{d\alpha_t}{dt},\\ 
\frac{d}{dt}\left(\varpi\big({\bf\mathsf v}_t^\alpha\big)\right)=(\overline{\beta}_t^{\alpha})^{-1}R\left(\overline{\beta}_t^{\alpha}\frac{d\alpha_t}{dt}, \overline{\beta}_t^{\alpha}\left(\theta\big({\bf\mathsf v}_t^\alpha\big)\right)\right) \overline{\beta}_t^{\alpha}.
\end{array}\right.
\end{align*}
\end{lem}

Let $\alpha^s=(\alpha_t^s)_{t\in [0, T]}$ be a one parameter family  of smooth segments of curves in $\Bbb R^m$.  For any $t_1, t_2$ with $0\leq t_1<t_2\leq T$, 
$s\mapsto DF_{t_1, t_2}^{s}:=DF_{t_1, t_2}^{\alpha^s}$ is said to be  $C^1$ in $s$ if the image curve $s\mapsto [DF_{t_1, t_2}^{s}]{\bf\mathsf v}_{t_1}^s $ is  $C^1$ 
for any  $C^1$ curve $s\mapsto {\bf\mathsf v}_{t_1}^s \in T_{\overline{\beta}^s_{t_1}}\mathcal{F}(\M)$.

\begin{lem}\label{v-t-s-diff-determine}Let   $\alpha^s_t=\mathcal{I}^{-1}_{\overline{\beta}^s_0}(\beta^s_t)$, where  $\beta^s$ are given in Lemma \ref{HS-Th2.1}.  The tangent maps $(DF_{t_1, t_2}^{s})_{0\leq t_1<t_2\leq T}$ are $C^1$ in $s$.  Let $s\mapsto {\bf\mathsf v}_{t_1}^s\in T_{\overline{\beta}_{t_1}^s}\mathcal{F}(\M)$ be $C^1$.  Then the differential $({\bf\mathsf v}_t^s)'_s:=\nabla_{D/\partial s}{\bf\mathsf v}_t^s$, where $
{\bf\mathsf v}_t^s:=\big[DF_{t_1, t}^s\big(\overline{\beta}_{\tau}^s, {\rm w}\big)\big]{\bf\mathsf v}_{t_1}^s$ for $t\in [t_1, t_2]$,  solves the  equation
\begin{align}\label{mathsfv-t-s-d}
\nabla_{\frac{D}{\partial t}}({\bf\mathsf v}_t^s)'_s=&\big(\nabla(({\bf\mathsf v}_t^s)'_s)H(\overline{\beta}^s_t, \cdot)\big)\frac{\partial \alpha^s_t}{\partial t}+\circledast\big({\bf\mathsf v}_t^s, (\overline{\beta}_t^s)'_s\big),
\end{align}
where \begin{align*}
\circledast\big({\bf\mathsf v}_t^s, (\overline{\beta}_t^s)'_s\big)=\!
\nabla^{(2)}\big({\bf\mathsf v}_t^s, (\overline{\beta}_t^s)'_s\big)H(\overline{\beta}^s_t, \frac{\partial \alpha^s_t}{\partial t})\!+\!\nabla({\bf\mathsf v}_t^s)H\big({\overline{\beta}_t^s}, \Upsilon_{{\rm V}, \alpha^s}(t)\big)\!+\!R\left(H(\overline{\beta}^s_t, \frac{\partial \alpha^s_t}{\partial t}), (\overline{\beta}^s_t)'_s\right){\bf\mathsf v}_t^s.
\end{align*}
In the $(\theta, \varpi)$-coordinate, we have
 \begin{align*}
 \left\{ \begin{array}{ll}
\nabla_{\frac{D}{\partial t}}\big(\theta\big(({\bf\mathsf v}_t^s)'_s\big)\big)=\varpi\big(({\bf\mathsf v}_t^s)'_s\big)\frac{\partial \alpha^s_t}{\partial t}+\theta\left(\circledast\big({\bf\mathsf v}_t^s, (\overline{\beta}_t^s)'_s\big)\right),\\
\nabla_{\frac{D}{\partial t}}\big(\varpi\big(({\bf\mathsf v}_t^s)'_s\big)\big)=(\overline{\beta}_t^s)^{-1}R\left(\overline{\beta}_t^s \frac{\partial \alpha^s_t}{\partial t},  \overline{\beta}_t^s\theta\big(({\bf\mathsf v}_t^s)'_s\big)\right)\overline{\beta}_t^s+\varpi\left(\circledast\big({\bf\mathsf v}_t^s, (\overline{\beta}_t^s)'_s\big)\right). 
\end{array}\right.
\end{align*}
\end{lem}
\begin{proof}Using (\ref{DF-alpha-eq}), we obtain 
\begin{align*}
\nabla_{\frac{D}{\partial t}}\nabla_{\frac{D}{\partial s}}{\bf\mathsf v}_t^s&=\nabla_{\frac{D}{\partial s}}\nabla_{\frac{D}{\partial t}}{\bf\mathsf v}_t^s+R\left(H(\overline{\beta}^s_t, \frac{\partial \alpha^s_t}{\partial t}), (\overline{\beta}^s_t)'_s\right){\bf\mathsf v}_t^s\\
&=\big(\nabla(({\bf\mathsf v}_t^s)'_s)H(\overline{\beta}^s_t, \cdot)\big)\frac{\partial \alpha^s_t}{\partial t}+\circledast\big({\bf\mathsf v}_t^s, (\overline{\beta}_t^s)'_s\big). 
\end{align*}
Using  (\ref{mathsfv-t-s-d}) and the structure equation (\ref{structure-1})  for $\theta$, we continue to compute that 
\begin{align*}
\nabla_{\frac{D}{\partial t}}\big(\theta\big(({\bf\mathsf v}_t^s)'_s\big)\big)&=\big(\nabla_{\frac{D}{\partial t}}\theta\big)\big(({\bf\mathsf v}_t^s)'_s\big)+\theta\left(\big(\nabla(({\bf\mathsf v}_t^s)'_s)H(\overline{\beta}^s_t, \cdot)\big)\frac{\partial \alpha^s_t}{\partial t}\right)+\theta\left(\circledast\big({\bf\mathsf v}_t^s, (\overline{\beta}_t^s)'_s\big)\right)\\
&=d\theta\left(H(\overline{\beta}^s_t, \frac{\partial \alpha^s_t}{\partial t}), ({\bf\mathsf v}_t^s)'_s\right)+\theta\left(\circledast\big({\bf\mathsf v}_t^s, (\overline{\beta}_t^s)'_s\big)\right)\\
&=\varpi\big(({\bf\mathsf v}_t^s)'_s\big)\frac{\partial \alpha^s_t}{\partial t}+\theta\left(\circledast\big({\bf\mathsf v}_t^s, (\overline{\beta}_t^s)'_s\big)\right).
\end{align*}
Similarly, using  (\ref{mathsfv-t-s-d}) and the structure equation (\ref{structure-2})  for  $\varpi$, we obtain
\begin{align*}
\nabla_{\frac{D}{\partial t}}\big(\varpi\big(({\bf\mathsf v}_t^s)'_s\big)\big)
&=d\varpi\left(H(\overline{\beta}^s_t, \frac{\partial \alpha^s_t}{\partial t}), ({\bf\mathsf v}_t^s)'_s\right)+\varpi\left(\circledast\big({\bf\mathsf v}_t^s, (\overline{\beta}_t^s)'_s\big)\right)\\
&=(\overline{\beta}_t^s)^{-1}R\left(\overline{\beta}_t^s \frac{\partial \alpha^s_t}{\partial t},  \overline{\beta}_t^s\theta\big(({\bf\mathsf v}_t^s)'_s\big)\right)\overline{\beta}_t^s+\varpi\left(\circledast\big({\bf\mathsf v}_t^s, (\overline{\beta}_t^s)'_s\big)\right).
\end{align*}
\end{proof}

We will solve  (\ref{flow F_s-on M})  by  identifying  the anti-development of  ${\alpha}^s$ of ${\rm y}^s$ in the  set 
  \[
 \mathcal{A}:=\left\{ \alpha_t=\int_0^{t} O_\tau dB_\tau+\int_0^{t}{\mathtt g}_\tau \ d\tau, \ \ t\in[0, T]\right\},\]
where $O_\tau$ is an $\mathcal{O}(\Bbb R^m)$ valued ${\mathcal{F}}_\tau$-adapted process, ${\mathtt g}_\tau$ is a $\Bbb R^m$ valued  ${\mathcal{F}}_\tau$-adapted process with 
$|{\mathtt g}|\leq {\rm Const.} \sup|{\rm V}|$ and  $\{\mathcal{F}_t\}_{t\in \Bbb R^+}$ is the filtration of the Brownian motion in $\Bbb R^m$. We see that $\mathcal{A}$ is a complete infinite dimensional  Banach space under the  norm  \[
\| \alpha\|_{\infty, T}:=\sqrt{\|\mathtt g\|_{\infty, T}^2+\| O\|_{\infty, T}^2},
\]
where 
\[
\|\mathtt g\|^2_{\infty, T}=\E \sup\limits_{t\in [0, T]} |\mathtt g_{t}|^2, \ \|O\|_{\infty, T}^2=\E \sup\limits_{t\in[0, T]}|O_{t}|^2.
\]

Let ${\rm V}$  and $\mathtt{s}$ be as above.  
For $\alpha\in \mathcal{A}$,  let $\overline{\beta}$ be a horizontal process in $\mathcal{O}^{\wt{g}}(\M)$ with projection $\beta=\mathcal{I}_{\overline{\beta}_0}(\alpha)$ on $\M$.  For $t\in [0, T]$, put 
\[\Upsilon_{{\rm V}, \beta}(t):=\mathtt{s}(t)\cdot \sslash_{0, t}^{\beta}({\rm V}(\beta_0))=\overline{\beta}_{t}\overline{\beta}_{0}^{-1}[\mathtt{s}(t){\rm V}(\beta_0)].\]
We define 
 \begin{equation*}
\underline{\Upsilon}_{{\rm V}, \alpha}(t):=\int_{0}^{t}\mathtt{s}'(\tau)\overline{\beta}_0^{-1}({\rm V}(\beta_0))\ d\tau-\int_{0}^{t}K_{{\rm V}, \alpha}(\tau)\circ d\alpha_{\tau}, 
\end{equation*}
where $\circ$ denotes the Stratonovich stochastic integral, and 
 \begin{align}
  K_{{\rm V}, \alpha}(\tau) := &  \int_{0}^{\tau} \overline{\beta}_{\varsigma}^{-1}R\left(\overline{\beta}_{\varsigma}(\circ d\alpha_\varsigma), \Upsilon_{{\rm V}, \beta}(\varsigma)\right)\overline{\beta}_\varsigma.\label{K-Stratonovich-O}
  \end{align}
\begin{lem}\label{gamma-V-apha}
For $\alpha\in \mathcal{A}$,   the It\^{o} forms of  $\underline{\Upsilon}_{{\rm V}, \alpha}$, $K_{{\rm V}, \alpha}$  are
\begin{align} \underline{\Upsilon}_{{\rm V}, \alpha}(t)&=\int_{0}^{t}\left\{\overline{\beta}_0^{-1}[\mathtt{s}'(\tau){\rm V}(\beta_0)]-{\rm Ric}\left(\Upsilon_{{\rm V}, \beta}(\tau)\right)\right\}\  d\tau-\int_{0}^{t}\langle K_{{\rm V}, \alpha}(\tau), d\alpha_{\tau}\rangle \notag\\
&=: {\rm R}_{{\rm V}, \alpha}(t)-\int_{0}^{t}\langle K_{{\rm V}, \alpha}(\tau), d\alpha_{\tau}\rangle, \label{Ito-Upsilon}
\end{align}
where ${\rm Ric}$ was defined in (\ref{Ric-def}) and 
\begin{equation}
\label{Ito-K-V}
K_{{\rm V}, \alpha}(t)=\int_{0}^{t}{\overline{\beta}_{\tau}}^{-1}R\left({\overline{\beta}_{\tau}}d\alpha_{\tau},\Upsilon_{{\rm V}, \beta}(\tau)\right){\overline{\beta}_{\tau}}+\int_{0}^{t}{\overline{\beta}_{\tau}}^{-1}\left(\nabla (\overline{\beta}_{\tau} e_i)R\right)\left({\overline{\beta}_{\tau}}e_i, \Upsilon_{{\rm V}, \beta}(\tau)\right){\overline{\beta}_{\tau}}\ d\tau.
  \end{equation} 
 \end{lem}
\begin{proof}Using It\^{o}'s formula, we can identify the It\^{o} integral expression of  $\underline{\Upsilon}_{{\rm V}, \alpha}$ as
  \[
   \underline{\Upsilon}_{{\rm V}, \alpha}(t)=\int_{0}^{t}\mathtt{s}'(\tau)\overline{\beta}_0^{-1}({\rm V}(\beta_0))\ d\tau-\int_{0}^{t}\langle K_{{\rm V}, \alpha}(\tau), d\alpha_{\tau}\rangle-\frac{1}{2}\int_{0}^{t}\langle d K_{{\rm V}, \alpha}(\tau), \circ d\alpha_{\tau}\rangle. 
  \]
 Let $\alpha_t=\alpha_{t, 1}e_1+\cdots +\alpha_{t, m}e_m$, where $\{e_1, \cdots, e_m\}$ is  the standard orthogonal base of $\Bbb R^m$.  
Since $\alpha\in \mathcal{A}$, we see that  $\langle \circ d\alpha_{t, i}, \circ d\alpha_{t, i}\rangle=2dt$.  So, using (\ref{K-Stratonovich-O}), we obtain 
  \begin{eqnarray*}
\frac{1}{2} \int_{0}^{t}\langle d K_{{\rm V}, \alpha}(\tau), \circ d\alpha_{\tau}\rangle&=& \frac{1}{2}\int_{0}^{t}\langle \overline{\beta}_{\tau}^{-1}R\left(\overline{\beta}_{\tau}(\circ d\alpha_{\tau}), \Upsilon_{{\rm V}, \beta}(\tau)\right)\overline{\beta}_\tau, \circ d\alpha_{\tau}\rangle\\
 &=& \int_{0}^{t}\overline{\beta}_{\tau}^{-1}R\left(\overline{\beta}_{\tau} e_i, \Upsilon_{{\rm V}, \beta}(\tau)\right)\overline{\beta}_\tau e_i\ d\tau\\
 &=&\int_{0}^{t} {\rm Ric}\left(\Upsilon_{{\rm V}, \beta}(\tau)\right)\ d\tau.
  \end{eqnarray*}
  The  It\^{o} integral expression  for  $K_{{\rm V}, \alpha}(t)$  can be obtained similarly using It\^{o}'s formula. 
\end{proof}

We want  to solve (\ref{flow F_s-on M}) with ${\rm y}^0$ being the Brownian motion on $\M$ starting from $y$. 

\begin{lem}\label{Cor-O-g-formula} Let ${\rm V}$ be a smooth bounded vector field on $\M$ with the associated flow  $\{F^s\}_{s\in \Bbb R}$. For $y\in \M$,  let  $({\mho}^s_0\in \mathcal{O}^{\wt{g}}_{F^s y}(\M))_{s\in \Bbb R}$ be a solution to $d{\mho}^s_0/ds=H(\mho_{0}^s, (\mho_0^{s})^{-1}{\rm V}(F^sy))$.  \begin{itemize}
\item[i)] Let $\alpha^s=\int_0^{\cdot} O_\tau^s dB_\tau+\int_0^{\cdot}{\mathtt g}_\tau^s \ d\tau\in \mathcal{A}$ be a one parameter family  of  stochastic processes  with $\alpha_t^0=B_t$.   Then $\alpha^s$  solves ${d {\alpha}^s_t({\rm w})}/{ds}=\underline{\Upsilon}_{{\rm V}, {\alpha}^s}(t)$ iff 
 \begin{align}\label{iteration-O}
\ \  O^s_{\tau}&= {\rm Id}-\int_{0}^{s}K_{{\rm V}, \alpha^{\jmath}}(\tau)O^{\jmath}_{\tau}\ d\jmath, \\
\ \   {\mathtt g}_{\tau}^{s}&=O^{s}_{\tau}\int_{0}^{s}[O_{\tau}^{\jmath}]^{-1}\left\{({\mho}_{0}^{\jmath})^{-1}\left[\mathtt{s}'(\tau){\rm V}(F^{\jmath}y)\right]-{\rm Ric}\left({\mho}_{\tau}^{\jmath}({\mho}_0^{\jmath})^{-1}[\mathtt{s}(\tau){\rm V}(F^{\jmath}y)]\right)\right\} \ d\jmath.\label{iteration-g}
 \end{align}
 \item[ii)]Let  $\alpha^s$ be as in ${\rm i)}$ and let  ${\mho}^s$ be its horizontal lift in $\mathcal{O}^{\wt{g}}(\M)$ with initial $\mho_0^s$.  Then $\mho^s$ is differentiable in $s$ iff  the following SDE is uniquely solvable with initial $(\mho_0^s)'_s$:
\begin{align}\label{diff-mho-s-R-m}
\left\{ \begin{array}{l}
d\theta(Y^s_t)=\varpi\big(Y^s_t\big)\circ d\alpha^s_t+\circ d\underline{\Upsilon}_{{\rm V}, \alpha^{s}},\\
d\varpi(Y^s_t)=(\mho_t^s)^{-1}R\big(\mho_t^s \circ d\alpha^s_t,  \mho_t^s\theta(Y^s_t)\big)\mho_t^s.
\end{array}\right.
\end{align}
 \item[iii)]Let $\alpha^s, \mho^s$ be as in {\rm i)}, {\rm ii)}.  Then  $s\mapsto {\rm y}^s=\mathcal{I}_{\mho_0^s}(\alpha^s)$ has  the differential  process $\Upsilon_{{\rm V}, {\rm y}^{s}}$.  
  \end{itemize}
 \end{lem}
 \begin{proof} By analogy with the deterministic case (Lemma \ref{HS-Th2.1}), we have  ${\rm y}^s$ solves  (\ref{flow F_s-on M}) iff $\alpha^s$  solves ${d {\alpha}^s_t({\rm w})}/{ds}=\underline{\Upsilon}_{{\rm V}, {\alpha}^s}(t)$,  which means 
 \begin{align*}
 \alpha_t^s-\alpha_t^0 =&\int_{0}^{s}\int_0^t ({\mho}_0^{\jmath})^{-1}[\mathtt{s}'(\tau){\rm V}(F^{\jmath}y)]\ d\tau\ d\jmath-\int_{0}^{s}\int_{0}^{t}\langle K_{{\rm V}, {\alpha}^{\jmath}}(\tau), d{\alpha}_{\tau}^{\jmath}\rangle\ d\jmath\\
 & -\int_{0}^{s}\int_{0}^{t} {\rm Ric}\left({\mho}_{\tau}^{\jmath}({\mho}_{0}^{\jmath})^{-1}[\mathtt{s}(\tau){\rm V}(F^{\jmath}y)]\right)\ d\tau\ d\jmath\\
 =&  \int_{0}^{s}\int_0^t ({\rm u}_T^{\jmath})^{-1}[\mathtt{s}'(\tau){\rm V}(F^{\jmath}y)]\ d\tau\ d\jmath- \int_{0}^{s}\int_{0}^{t}{\rm Ric}\left({\mho}_{\tau}^{\jmath}({\mho}_{0}^{\jmath})^{-1}[\mathtt{s}(\tau){\rm V}(F^{\jmath}y)]\right)\ d\tau\ d\jmath\\
 &-\int_{0}^{s}\int_{0}^{t} K_{{\rm V}, {\alpha}^{\jmath}}(\tau){\mathtt g}^{\jmath}_{\tau}\ d\tau\ d\jmath -\int_{0}^{s}\int_{0}^{t} K_{{\rm V}, {\alpha}^{\jmath}}(\tau)O^{\jmath}_{\tau}\ d{B}_\tau \ d\jmath.
 \end{align*}
 Note that $\alpha^0_t=B_t$ and hence $O^0={\rm{Id}}, {\mathtt g}^0=0$. So a  comparison of the above expression with the the assumption that $\alpha^s_t=\int_0^{t} O_\tau^s dB_\tau+\int_0^{t}{\mathtt g}_\tau^s \ d\tau$  gives  (\ref{iteration-O}) and 
 \begin{align*}
{\mathtt g}_{\tau}^{s}=\int_{0}^{s}({\mho}_{0}^{\jmath})^{-1}[\mathtt{s}'(\tau){\rm V}(F^{\jmath} y)]\ d\jmath-\int_{0}^{s}{\rm Ric}\left({\mho}_{\tau}^{\jmath}({\mho}_0^{\jmath})^{-1}[\mathtt{s}(\tau){\rm V}(F^{\jmath} y)]\right)\ d\jmath-\int_{0}^{s} K_{{\rm V}, \alpha^{\jmath}}(\tau){\mathtt g}_{\tau}^{\jmath}\ d\jmath.
\end{align*}
 Hence by the variation of constants method (i.e., Duhamel's principle), we obtain (\ref{iteration-g}).

 Let $\alpha^s$ be as in {\rm i)} which solves  ${d {\alpha}^s_t({\rm w})}/{ds}=\underline{\Upsilon}_{{\rm V}, {\alpha}^s}(t)$.  Then  $\mho^s$ is differentiable in $s$ iff  the following SDE is solvable with initial $(\mho_0^s)'_s$:
 \begin{eqnarray}\label{diff-u-equation-Y}d Y^s_t=\big(\nabla (Y^s_t)H\big)({\mho}_t^s, \circ d\alpha_{t}^s)+H({\mho}^s_t, (\circ d\alpha^{s}_t)'_s). \end{eqnarray}
 Writing (\ref{diff-u-equation-Y}) in the $(\theta, \varpi)$-coordinate,  we have 
 \begin{align*}
 d(\theta(Y_t^s))&=d\theta(\circ d\mho_{t}^{s}, Y_t^s)+\theta\big(H({\mho}^s_t, (\circ d\alpha^{s}_t)'_s)\big)=\omega(Y_t^s)\circ d{\alpha}_t^s+\circ d\underline{\Upsilon}_{{\rm V}, \alpha^{\jmath}},\\
 d(\varpi(Y_t^s))&=\Omega\left(H( {\mho}_{t}^{s},  \circ d\alpha_{t}^{s}), H({\mho}_{t}^{s}, \theta(Y_t^s))\right)=(\mho_{t}^{s})^{-1}R\big(\mho_{t}^{s} \circ d\alpha_{t}^{s}, \mho_{t}^{s}\theta(Y_t^s)\big)\mho_{t}^{s}.
 \end{align*}

Let $\alpha^s, \mho^s$ be such that  {\rm i)}, {\rm ii)} hold true.   For {\rm iii)},   it suffices to  check the equality  $\nabla_{D/\partial s}(\pi({\mho}^s))=\Upsilon_{{\rm V}, {\rm y}^{s}}$.  Let $Z_t^{\jmath}:=\theta\left(({\mho}_t^s)'_\jmath\right)$.  By (\ref{diff-mho-s-R-m}), 
 \[
 Z_t^{\jmath}=\underline{\Upsilon}_{{\rm V}, \alpha^{\jmath}}(t)-\underline{\Upsilon}_{{\rm V}, \alpha^{\jmath}}(0)+\int_{0}^{t}  \left(\int_{0}^{\tau}(\mho_{\tau'}^{\jmath})^{-1}R\big(\mho_{\tau'}^{\jmath} \circ d\alpha_{\tau'}^{\jmath}, \mho_{\tau'}^{\jmath}Z_{\tau'}^{\jmath}\big)\mho_{\tau'}^{\jmath}\right)\circ d\alpha_{\tau}^s. 
 \]
Write $\mathcal{Z}_t^{\jmath}:=Z_{t}^{\jmath}-\int_{0}^{t}({\mho}_0^{\jmath})^{-1}[\mathtt{s}'(\tau){\rm V}(F^{\jmath}y)]\ d\tau$. Then we have
\begin{equation*}
\mathcal{Z}_t^{\jmath}=\int_{0}^{t}  \left(\int_{0}^{\tau}(\mho_{\tau'}^{\jmath})^{-1}R\big(\mho_{\tau'}^{\jmath} \circ d\alpha_{\tau'}^{\jmath}, \mho_{\tau'}^{\jmath}\mathcal{Z}_{\tau'}^{\jmath}\big)\mho_{\tau'}^{\jmath}\right)\circ d\alpha_{\tau}^s. 
\end{equation*}
Using It\^{o}'s formula for $|\cdot|^2=\langle\cdot, \cdot\rangle$ or  the isometry property of Brownian motion,   we can find some constant $C(s_0, T)$ depending on $R, s_0, T$ such that 
\[
\E((\mathcal{Z}_t^{\jmath})^2)\leq C(s_0, T)\int_{0}^{t} \E((\mathcal{Z}_\tau^{\jmath})^2)\ d\tau.
\]
This gives $\mathcal{Z}_t^{\jmath}=0$  by  Gronwall's Lemma (see Lemma \ref{Gronwall}). Thus $({\rm y}^s)'_s=\Upsilon_{{\rm V}, {\rm y}^{s}}$.  
 \end{proof}
 
 \begin{cor}\label{cor-mho-t-s-d}Let $\mho^s$ be as in {\rm ii)} of Lemma \ref{Cor-O-g-formula}.  Then $Y^s=(\mho^s)'_s$  is given by 
  \begin{align}\label{Y-t-s-diff-Str}
\left\{ \begin{array}{l}
d\theta(Y^s_t)=\mathtt{s}'(t)({\mho}_0^s)^{-1}V(F^s y)\ dt,\\
d\varpi(Y^s_t)=(\mho_t^s)^{-1}R\big(\mho_t^s \circ d\alpha^s_t,  \mathtt{s}(t)\mho_t^s({\mho}_0^s)^{-1}V(F^s y)\big)\mho_t^s,
\end{array}\right.
\end{align}
whose  It\^{o} form is 
 \begin{align}\label{Y-t-s-diff-ITO}
\left\{ \begin{array}{l}
d\theta(Y^s_t)=\mathtt{s}'(t)({\mho}_0^s)^{-1}V(F^s y)\ dt,\\
d\varpi(Y^s_t)=(\mho_t^s)^{-1}R\big(\mho_t^s d\alpha^s_t,  \mathtt{s}(t)\mho_t^s({\mho}_0^s)^{-1}V(F^s y)\big)\mho_t^s\\
\ \ \ \ \ \ \ \ \ \ \ \ \   +(\mho_t^s)^{-1}(\nabla ({\mho}_{t}^se_i)R)\left({\mho}_{t}^s e_i,  \mathtt{s}(t)\mho_t^s({\mho}_0^s)^{-1}V(F^s y)\right){\mho}_{t}^s\ dt.
\end{array}\right.
\end{align}
 \end{cor}
 \begin{proof}Note that $\mho^s$ is a horizontal lift of ${\rm y}^s$. Reporting this and  $({\rm y}^s)'_s=\Upsilon_{{\rm V}, {\rm y}^{s}}$ in (\ref{diff-mho-s-R-m}) shows  (\ref{Y-t-s-diff-Str}). Then  (\ref{Y-t-s-diff-ITO})  follows  by applying the  It\^{o} formula. 
 \end{proof}
 
For $\alpha=(\alpha_{t, 1}, \cdots, \alpha_{t, m})\in \mathcal{A}$,  consider the associated flow maps $\{F_{t_1, t_2}^{\alpha}\}_{0\leq t_1<t_2\leq T}$, 
where  
$F_{t_1, t_2}^{\alpha}: \mathcal{F}(\M)\to \mathcal{F}(\M); \  \overline{\beta}_{t_1}^{\alpha}\mapsto \overline{\beta}_{t_2}^{\alpha}$,  with  $(\overline{\beta}_{t}^{\alpha})_{t\in [t_1, t_2]}$ solving  the Stratonovich SDE
\begin{equation}\label{Horizontal-u}
d\overline{\beta}^{\alpha}_t=H(\overline{\beta}^{\alpha}_t, \circ d\alpha_t). 
\end{equation}
By Proposition \ref{SDE-flow-regularity}, $F_{t_1, t_2}^{\alpha}$ are  $C^{k-2}$ diffeomorphisms for almost all ${\rm w}$ and the first order tangent map  $DF_{t_1, t_2}^{\alpha}$ satisfies the following (see also Lemma \ref{El-Ku-Cor}).

\begin{lem}\label{Tangent map-SDE}Let $\alpha\in \mathcal{A}$.   For almost all ${\rm w}$, any $t_1\in [0, T]$ and ${\bf\mathsf v}_{t_1}^\alpha\in T_{\overline{\beta}_{t_1}^\alpha}\mathcal{F}(\M)$,   $
{\bf\mathsf v}_t^\alpha:=\big[DF_{t_1, t}^\alpha\big(\overline{\beta}_{t_1}^\alpha, {\rm w}\big)\big]{\bf\mathsf v}_{t_1}^\alpha, t\in [t_1, T]$ satisfies the Stratonovich SDE 
\[
d{\bf\mathsf v}_t^\alpha =\big(\nabla ({\bf\mathsf v}_t^\alpha)H\big)(\overline{\beta}^\alpha_t, \circ d\alpha_t). 
\]
In the $(\theta, \varpi)$-coordinate,   we have 
  \begin{align*}
\left\{ \begin{array}{ll}
d\left(\theta\big({\bf\mathsf v}_t^\alpha\big)\right)= \varpi\big({\bf\mathsf v}_t^\alpha\big)\circ d\alpha_t,\\ 
d\left(\varpi\big({\bf\mathsf v}_t^\alpha\big)\right)=(\overline{\beta}_t^{\alpha})^{-1}R\left(\overline{\beta}_t^{\alpha}\circ d\alpha_t, \overline{\beta}_t^{\alpha}\left(\theta\big({\bf\mathsf v}_t^\alpha\big)\right)\right) \overline{\beta}_t^{\alpha}
\end{array}\right.
\end{align*}
and its  It\^{o} form  is
\begin{align*}
\left\{ \begin{array}{l}
d\theta({\bf\mathsf v}_t^\alpha)=\varpi\big({\bf\mathsf v}_t^\alpha\big) d\alpha_t+{\rm Ric}\big({\ov{\beta}}_{t}^{\alpha}\theta({\bf\mathsf v}_t^\alpha)\big)\ dt,\\
d\varpi({\bf\mathsf v}_t^\alpha)=(\overline{\beta}_t^{\alpha})^{-1}R\left(\overline{\beta}_t^{\alpha} d\alpha_t,  \overline{\beta}_t^{\alpha}\theta({\bf\mathsf v}_t^\alpha)\right)\overline{\beta}_t^{\alpha}+(\overline{\beta}_t^{\alpha})^{-1}R\left(\overline{\beta}_t^{\alpha} e_i, \varpi({\bf\mathsf v}_t^\alpha)e_i\right)\overline{\beta}_t^{\alpha}\ dt\\
\ \ \ \ \ \ \ \ \ \ \ \ \ +(\overline{\beta}_t^{\alpha})^{-1}(\nabla (\overline{\beta}_t^{\alpha} e_i)R)\left(\overline{\beta}_t^{\alpha} e_i, \overline{\beta}_t^{\alpha}\theta({\bf\mathsf v}_t^\alpha)\right)\overline{\beta}_t^{\alpha}\ dt.
\end{array}\right.
\end{align*}
\end{lem}

Let $\alpha^s\in \mathcal{A}$ be a one parameter family  of random processes. We abbreviate 
\[
\overline{\beta}_{t}^s:=\overline{\beta}_t^{\alpha^s},\  F_{t_1, t_2}^{s}:=F^{\alpha^s}_{t_1, t_2} \  \mbox{and}\  DF_{t_1, t_2}^{s}:=DF_{t_1, t_2}^{\alpha^s}. 
\]
The maps $\{DF_{t_1, t_2}^{s}\}_{0\leq t_1<t_2\leq T}$  are said to be  $C^1$ in $s$ if, for almost all ${\rm w}$  and any  $(\upsilon_{t_1}^{s}, Q_{t_1}^{s})\in T_{\overline{\beta}^s}\mathcal{F}(\M)$ which is $C^1$ in $s$, $[DF_{t_1, t_2}^{s}](\upsilon_{t_1}^{s}, Q_{t_1}^{s})$ is also $C^1$ in $s$. The following can be formulated  using Lemma \ref{Tangent map-SDE} and It\^{o}'s formula by analogy with  Lemma \ref{v-t-s-diff-determine}.

\begin{lem}\label{diff-D-F-s-crit}
Let  $\alpha^s, {\rm y}^s$  and  $\mho^s$ be as in  Lemma \ref{Cor-O-g-formula}.  Then $\{DF_{t_1, t_2}^{s}\}_{0\leq t_1<t_2\leq T}$  are  $C^1$ in $s$ iff for any ${\bf\mathsf v}_{t_1}^s\in T_{\mho_{t_1}^s}\mathcal{F}(\M)$ $C^1$ in $s$,  there is a unique  $(\upsilon_t^s)_{t\in [t_1, t_2]}$,   continuous in $(t, s)$ with $\upsilon_{t_1}^s=\nabla_{D/\partial s}{\bf\mathsf v}_{t_1}^s$,   that solves the SDE
\begin{align}\label{mathsfv-t-s-d-stoch}
d\upsilon_t^s=&\big(\nabla(\upsilon_t^s)H\big)(\mho^s_t, \circ d\alpha^s_t)+\circledast\big({\bf\mathsf v}_t^s, (\mho_t^s)'_s\big), 
\end{align}
  where\begin{align*}
\circledast\big({\bf\mathsf v}_t^s, (\mho_t^s)'_s\big)=&
\nabla^{(2)}\big({\bf\mathsf v}_t^s, (\mho_t^s)'_s\big)H(\mho^s_t, \circ d \alpha^s_t)+\nabla({\bf\mathsf v}_t^s)H\big({\mho_t^s}, \circ d\Upsilon_{{\rm V}, \alpha^s}(t)\big)\\
&+R\left(H(\mho^s_t, \circ d\alpha^s_t), (\mho_t^s)'_s\right){\bf\mathsf v}_t^s. 
\end{align*}
In the $(\theta, \varpi)$-coordinate,  (\ref{mathsfv-t-s-d-stoch}) is 
 \begin{align}\label{diff-bfv-t-s-t-v}
\left\{ \begin{array}{ll}
d\big(\theta\big(\upsilon_t^s\big)\big)\ =\varpi\big(\upsilon_t^s\big)\circ d \alpha^s_t+\theta\left(\circledast\big({\bf\mathsf v}_t^s, (\mho_t^s)'_s\big)\right),\\
d\big(\varpi\big(\upsilon_t^s\big)\big)=(\mho_t^s)^{-1}R\left(\mho_t^s \circ d \alpha^s_t,  \mho_t^s\theta\big(\upsilon_t^s\big)\right)\mho_t^s+\varpi\left(\circledast\big({\bf\mathsf v}_t^s, (\mho_t^s)'_s\big)\right). 
\end{array}\right.
\end{align}
The  It\^{o} form of (\ref{diff-bfv-t-s-t-v}) is 
\begin{align}\label{mathsfv-t-s-d-Ito}
\left\{ \begin{array}{ll}
d\big(\theta\big(\upsilon_t^s\big)\big)\ =\varpi\big(\upsilon_t^s\big) d \alpha^s_t+{\rm Ric}\big({\mho}_{t}^s\theta(\upsilon_t^s)\big)\ dt+\theta\left(\circledast_{{\rm I}}\big({\bf\mathsf v}_t^s, (\mho_t^s)'_s\big)\right)+\circledast_{{\rm A}}^{\theta}\big({\bf\mathsf v}_t^s, (\mho_t^s)'_s\big),\\
d\big(\varpi\big(\upsilon_t^s\big)\big)=(\mho_t^s)^{-1}R\left(\mho_t^s d \alpha^s_t,  \mho_t^s\theta\big(\upsilon_t^s\big)\right)\mho_t^s+(\mho_t^s)^{-1}R\left(\mho_t^s e_i, \mho_t^s\varpi(\upsilon_t^s)e_i\right)\mho_t^s\ dt\\
\ \ \ \ \ \ \ \ \ \ \ \ \ \ \ \ +(\mho_t^s)^{-1}(\nabla ({\mho}_{t}^se_i)R)\left({\mho}_{t}^s e_i, {\mho}_{t}^s\theta(\upsilon_t^s)\right){\mho}_{t}^s\ dt\\
\ \ \ \ \ \ \ \ \ \ \ \ \ \ \ \  +\varpi\left(\circledast_{{\rm I}}\big({\bf\mathsf v}_t^s, (\mho_t^s)'_s\big)\right)+\circledast_{{\rm A}}^{\varpi}\big({\bf\mathsf v}_t^s, (\mho_t^s)'_s\big),
\end{array}\right.
\end{align}
where $\circledast_{{\rm I}}\big({\bf\mathsf v}_t^s, (\mho_t^s)'_s\big)$ is $\circledast\big({\bf\mathsf v}_t^s, (\mho_t^s)'_s\big)$ with $\circ d\alpha^s_t$ replaced by the It\^{o} infinitesimal  $d\alpha^s_t$, 
\begin{align*}
\circledast_{{\rm A}}^{\theta}\big({\bf\mathsf v}_t^s, (\mho_t^s)'_s\big)=& 2\varpi\left(\circledast\big({\bf\mathsf v}_t^s, (\mho_t^s)'_s, e_i\big)\right)e_i\ dt+\theta\left(\left[H(\mho_t^s, e_i), \circledast\big({\bf\mathsf v}_t^s, (\mho_t^s)'_s, e_i\big)\right]\right)\ dt,\\
\circledast_{{\rm A}}^{\varpi}\big({\bf\mathsf v}_t^s, (\mho_t^s)'_s\big)=& 2(\mho_{t}^s)^{-1}\!R\!\left(\mho_{t}^se_i, \mho_{t}^s\theta\big(\circledast({\bf\mathsf v}_t^s, (\mho_t^s)'_s, e_i)\big)\right)\!\mho_{t}^s\ dt\\
&+\varpi\left(\left[H(\mho_t^s, e_i), \circledast\big({\bf\mathsf v}_t^s, (\mho_t^s)'_s, e_i\big)\right]\right)\ dt, \\
\circledast\big({\bf\mathsf v}_t^s, (\mho_t^s)'_s, e_i\big)=&\nabla^{(2)}\!\big({\bf\mathsf v}_t^s, (\mho_{t}^s)'_0\big)H(\mho_{t}^s, e_i)\!+\!\!R\big(H(\mho^s_t, e_i), (\mho_t^s)'_s\big){\bf\mathsf v}_t^s\!+\!\!\nabla({\bf\mathsf v}_t^s)H\big({\mho_t^s}, K_{{\rm V}, \alpha^s}(t)e_i\big).
\end{align*}
\end{lem}

\subsection{The existence of  ${\bf F}^s_y$} \label{sec-exist-F-s}

In this part, we prove the existence of the mapping ${\rm y}\mapsto {\bf F}^s_y({\rm y})$.   By Lemma \ref{Cor-O-g-formula}, it suffices to solve ${d {\alpha}^s_t({\rm w})}/{ds}=\underline{\Upsilon}_{{\rm V}, {\alpha}^s}(t)$ in $\mathcal{A}$  with $\alpha^0=B$. We will do this using  the classical Picard method as in \cite[Theorem 3.1]{Hs1}.  In the meanwhile, we will also show the existence of the differential processes of  $\mho^s$ and $DF^{s}_{t_1, t_2}$ in $s$.  The tool we will use to obtain a continuous version of a two-parameter process is  Kolmogorov's criterion.

\begin{lem}(cf. \cite[Theorem 1.4.1]{Ku2})\label{Kol-criterion}
Let $\{\mathcal{Y}^{s}_{t}({\rm w})\}_{t\in [0, T], s\in [-s_0, s_0]}$ be a one parameter family  of random processes on a complete manifold. Suppose there are positive constants $\flat, \flat_1, \flat_2$,  with $\flat_1, \flat_2>2$,  and ${\mathtt C}_0(\flat)$ such that for  all $t, t'\in [0, T]$ and $s, s'\in [-s_0, s_0]$, 
\[
\E\left[\left|\mathcal{Y}^s_t-\mathcal{Y}^{s'}_{t'}\right|^{\flat}\right]\leq  {\mathtt C}_0(\flat)\left(|t-t'|^{\flat_1}+|s-s'|^{\flat_2}\right),\]
then  $\mathcal{Y}^{s}_t$ has a continuous modification with respect to the parameter $(t, s)$. 
\end{lem}

Besides Burkholder's inequality (Lemma \ref{Ku-lem}), another useful tool to estimate the $L^q$-norm  of stochastic integrals   is Gronwall's lemma:

\begin{lem}(cf. \cite[p. 13]{El}) \label{Gronwall} Let $\phi, \phi_1$ be real valued Lebesgue integrable functions on the interval $[0, s]$ such that for  some ${\mathtt C}>0$,  
\[
\phi(\jmath)\leq \phi_1(\jmath)+{\mathtt C}\int_{0}^{\jmath}\phi(\jmath')\ d\jmath', \ \forall \jmath\in [0, s].
\] 
Then 
\[
\phi(\jmath)\leq \phi_1(\jmath)+{\mathtt C}\int_{0}^{\jmath}e^{{\mathtt C}(\jmath-\jmath')}\phi_1(\jmath')\ d\jmath', \ \mbox{for almost all}\ \jmath\in [0, s]. 
\]
\end{lem}

We are in a situation to state the existence theorem of the maps ${\bf F}^s_y$.

\begin{theo}\label{Main-alpha-x-v-Q}Let ${\rm V}$ be a bounded  smooth vector field on $\M$ and let $\{F^s\}_{s\in \Bbb R}$ be the flow it generates. For $y\in \M$,   let $({\mho}^s_0\in \mathcal{O}^{\wt{g}}_{F^s y}(\M))_{s\in \Bbb R}$ with $d{\mho}^s_0/ds=H(\mho_{0}^s, (\mho_0^{s})^{-1}{\rm V}(F^sy))$ be a fixed horizontal lift of the smooth curve $(F^sy)_{s\in \Bbb R}$. 
\begin{itemize}
\item[i)] There exists a unique family of stochastic processes $\alpha^s\in \mathcal{A}$ such that  for almost all ${\rm w}$,  $s\mapsto \alpha^s({\rm w})$ is  differentiable with
\begin{equation}\label{unique-alpha-s}
\alpha^s_t({\rm w})={\rm w}+\int_0^s  \underline{\Upsilon}_{{\rm V}, {\alpha}^\jmath}(t, {\rm w})\ d\jmath, \ \forall t\in [0, T]. 
\end{equation}
The process  $\underline{\Upsilon}_{{\rm V}, {\alpha}^s}(t)$ has a continuous  modification in the parameter $(t, s)$. 
\item[ii)] Let  ${\mho}^s\in \mathcal{O}^{\wt{g}}(\M)$ be a horizontal  lift of  ${\rm y}^s$ with initial ${\mho}^s_0$.  There exists a one parameter family  of $\mathcal{F}_t$-adapted stochastic processes $(Y^s_t)_{t\in [0, T]}$ with $Y^s_t({\rm w})\in T_{{\mho}_t^s({\rm w})}(\mathcal{O}^{\wt{g}}(\M))$ for almost all ${\rm w}$, which satisfies
\begin{equation*}
\nabla_{\frac{D}{\partial s}}{\mho}^s_t({\rm w})= Y_t^{s}({\rm w}), \ \forall t\in [0, T].
\end{equation*}
The process $Y^s_t$  has a continuous modification in the parameter $(t, s)$. 
\item[iii)] Let ${\rm y}^s=\mathcal{I}_{\mho_0^s}(\alpha^s)$. Then $s\mapsto {\rm y}^s({\rm w})$ is differentiable  for  almost all ${\rm w}$ with 
\begin{equation}\label{unique-y-s}
\nabla_{\frac{D}{\partial s}}{\rm y}^s_t({\rm w})=\Upsilon_{{\rm V}, {\rm y}^{s}}(t, {\rm w}), \ \forall t\in [0, T].\end{equation}
The process  $\Upsilon_{{\rm V}, {\rm y}^{s}}(t)$ has a continuous modification in the parameter $(t, s)$. 
\item[iv)] For almost all ${\rm w}$,  $(DF_{t_1, t}^{s})_{0\leq t_1<t\leq T}$  are  $C^1$ in $s$.  For  ${\bf\mathsf v}_{t_1}^s\in T_{\mho_{t_1}^s}\mathcal{F}(\M)$ $C^1$ in $s$, the process ${\bf\mathsf v}_{t}^s=[DF_{t_1, t}^{s}]{\bf\mathsf v}_{t_1}^s$  is differentiable in $s$ and the differential process $\upsilon_t^s$ has a continuous modification in the parameter $(t, s)$. 
\end{itemize}
\end{theo}
\begin{proof}For simplicity, we will use $C$ to denote a constant depending on $\|g\|_{C^3}$ and the norm bound of  ${\rm V}$ and  use $C(\cdot)$ to indicate the extra coefficients it depends on, for instance, $C(s_0, T)$ means $C$ also depends on $s_0, T$.  These constants  $C$ may vary from line to line.  

We first show i). 
  For any  $s_0\in \Bbb R^+$, we  use (\ref{iteration-O}), (\ref{iteration-g})  for Picard's  iteration and show the iteration converges to a one-parameter family of  processes $\alpha^s$ ($s\leq s_0$) in norm $\|\cdot \|_{\infty, T}$.   Let  ${\mathtt g}^{s, 0}=0, {O}^{s, 0}={\mbox Id}$, $\alpha^{s, 0}=B$ and let $\mho^{s, 0}$ be the horizontal lift of $\mathcal{I}_{\mho_0^s}(B)$ in $\mathcal{O}^{\wt{g}}(\M)$ with ${\mho}_0^{s, 0}=\mho_0^s$.  Assume ${\mathtt g}^{s, n-1}, {O}^{s, n-1}$ and  $\alpha^{s, n-1}$ are obtained for some $n\in \Bbb N$.  We  write  ${\mho}^{s, n-1}$ for the horizontal development of  $\alpha^{s, n-1}$ in $\mathcal{O}^{\wt{g}}(\M)$ with ${\mho}_0^{s, n-1}=\mho_0^s$ and put
\begin{align*}
K_{{\rm V}, \alpha^{s, n-1}}(t)=&\int_{0}^{t}({\mho_{\tau}^{s, n-1}})^{-1}R\left({\mho_{\tau}^{s, n-1}}d\alpha_{\tau}^{s, n-1}, {\mho_{\tau}^{s, n-1}}(\mho_{0}^s)^{-1}[\mathtt{s}(\tau){\rm V}(F^s y)]\right){\mho_{\tau}^{s, n-1}}\\
&\!+\!\int_{0}^{t}\!({\mho_{\tau}^{s, n-1}})^{-1}\!(\nabla ({\mho_{\tau}^{s, n-1}} e_i)R)\!\left({\mho_{\tau}^{s, n-1}}e_i, {\mho_{\tau}^{s, n-1}}\!(\mho_0^s)^{-1}\![\mathtt{s}(\tau){\rm V}(F^sy)]\right){\mho_{\tau}^{s, n-1}}d\tau\notag, \\
 {\rm R}_{{\rm V}, \alpha^{s, n-1}}(t)=&\ ({\mho}_{0}^{s})^{-1}[\mathtt{s}'(t){\rm V}(F^s y)]-{\rm Ric}\left({\mho}_{t}^{s, n-1}({\mho}_0^{s})^{-1}[\mathtt{s}(t){\rm V}(F^{s}y)]\right).
  \end{align*} 
Then define ${\mathtt g}^{s, n}, {O}^{s, n}, \alpha^{s, n}$ as the processes determined by the following SDEs:
\begin{align*}
\left\{ \begin{array}{l}
 O^{s, n}_{t}= {\rm{Id}}-\int_{0}^{s}K_{{\rm V}, \alpha^{\jmath, n-1}}(t)O^{\jmath, n}_{t}\ d\jmath, \\ \\
 {\mathtt g}_{t}^{s, n}= O^{s, n}_{t}\int_{0}^{s}[O_{t}^{\jmath, n}]^{-1}{\rm R}_{{\rm V}, \alpha^{\jmath, n-1}}(t) \ d\jmath,\\ \\
 \alpha^{s, n}_{t}=\int_{0}^{t}O_{\tau}^{s, n}\ dB_{\tau}+\int_{0}^{t}{\mathtt g}_{\tau}^{s, n}\ d\tau. 
 \end{array}
 \right.
\end{align*}
When $n=1$, the  definitions of ${\mathtt g}^{s, 0}, {O}^{s, 0}, {\mathtt g}^{s, 1}$ and ${O}^{s, 1}$ show that 
\begin{align*}
O^{s, 1}_{t}-O^{s, 0}_{t}= -\int_{0}^s K_{{\rm V}, \alpha^{\jmath, 0}}(t)O^{\jmath, 1}_{t}\ d\jmath,\ \ \ {\mathtt g}_{t}^{s, 1}- {\mathtt g}_{t}^{s, 0}=O^{s, 1}_{t}\int_{0}^{s}[O_{t}^{\jmath, 1}]^{-1}{\rm R}_{{\rm V}, \alpha^{\jmath, 0}}(t) \ d\jmath.
 \end{align*} 
 Abbreviate  $\|\cdot\|_{\infty, T}$ as $\|\cdot\|$. There is some $C$ such that 
\begin{align*}
 \|K_{{\rm V}, \alpha^{\jmath, 0}}\|^2&\leq 2\E\sup_{t\in [0, T]}\left|\int_{0}^{t}({\mho_{\tau}^{s, 0}})^{-1}R\left({\mho_{\tau}^{s, 0}}dB_{\tau}, {\mho_{\tau}^{s, 0}}(\mho_{0}^s)^{-1}[\mathtt{s}(\tau){\rm V}(F^s y)]\right){\mho_{\tau}^{s, 0}}\right|^2+CT^2,\\
 &\leq 4\E\left|\int_{0}^{T}({\mho_{\tau}^{s, 0}})^{-1}R\left({\mho_{\tau}^{s, 0}}dB_{\tau}, {\mho_{\tau}^{s, 0}}(\mho_{0}^s)^{-1}[\mathtt{s}(\tau){\rm V}(F^s y)]\right){\mho_{\tau}^{s, 0}}\right|^2+CT^2,\\
 &\leq 4\E\int_{0}^{T}\left|({\mho_{\tau}^{s, 0}})^{-1}R\left({\mho_{\tau}^{s, 0}}e_i, {\mho_{\tau}^{s, 0}}(\mho_{0}^s)^{-1}[\mathtt{s}(\tau){\rm V}(F^s y)]\right){\mho_{\tau}^{s, 0}}\right|^2\ d\tau+CT^2\\
 &\leq C(T+T^2),
 \end{align*}
 where the second inequality holds by  Doob's inequality of sub-martingales and the third inequality holds by  Lemma \ref{Ku-lem}.  Hence there is some $C(T)$ such that \[\| O^{s, 1}-O^{s, 0}\|\leq \int_{0}^{s} \|K_{{\rm V}, \alpha^{\jmath, 0}}\|\ ds\leq C(T)s.\]
There  also exists some $C$ such that $\| \mathtt{g}^{s, 1}-\mathtt{g}^{s, 0}\|\leq Cs$ since ${\rm R}_{{\rm V}, \alpha^{\jmath, 0}}$ is bounded. 
So we  obtain some  $C_0(T)$ such that 
\[
\|\alpha^{s, 1}-\alpha^{s, 0}\| \leq C_0(T) s. 
\]
If we  can further find some constant $C_1(T)$ such that  
\begin{eqnarray}\label{iteration-alpha}
\| \alpha^{s, n}-\alpha^{s, n-1} \|\leq C_1(T)\int_{0}^{s}\| \alpha^{\jmath, n-1}-\alpha^{\jmath, n-2}\|\ d\jmath, 
\end{eqnarray}
 we will obtain
\[
\|\alpha^{s, n}-\alpha^{s, n-1}\|\leq \frac{1}{n!}(C_0(T)+C_1(T))^n s^n, 
\]
which will imply the existence of the limits 
\[
\mathtt g^{s}=\lim\limits_{n\rightarrow +\infty}\mathtt g^{s, n}, \ O^{s}=\lim\limits_{n\rightarrow +\infty} O^{s, n}. 
\]
Then   $\alpha^s_{t}=\int_{0}^{t} O^{s}_{\tau}\ dB_\tau+\int_{0}^{t}{\mathtt g}^{s}_{\tau}\ d\tau$ will be our desired process for i)  by Lemma  \ref{Cor-O-g-formula}.

   For (\ref{iteration-alpha}), let us  analyze $\|O^{s, n}-O^{s, n-1}\|$ and $\|{\mathtt g}^{s, n}-{\mathtt g}^{s, n-1}\|$.  Since each $O^{s, n}$ is $\mathcal{O}(\Bbb R^m)$ valued and is invertible, we have 
\begin{align*}
O^{s, n}-O^{s, n-1}&= O^{s, n}\left({\rm Id}-(O^{s, n})^{-1}O^{s, n-1}\right)\\
&= -O^{s, n}\int_0^s \left.\frac{d}{ds}\left[(O^{s, n})^{-1}O^{s, n-1}\right]\right|_{s=\jmath}\ d\jmath\\
&=  -O^{s, n}\int_0^s \left(\left.\frac{d}{ds}\left[(O^{s, n})^{-1}\right]\right|_{s=\jmath}O^{\jmath, n-1}+ (O^{\jmath, n})^{-1}\left.\frac{d}{ds}\left[O^{s, n-1}\right]\right|_{s=\jmath} \right) d\jmath. 
\end{align*}
By the inductive defining  equations of  $O^{s, n-1}$, $O^{s, n}$,  we obtain 
\begin{align*}
\left.\frac{d}{ds}\left[O^{s, n-1}\right]\right|_{s=\jmath}&=-K_{{\rm V}, \alpha^{\jmath, n-2}}O^{\jmath, n-1}, \\
\left.\frac{d}{ds}\left[(O^{s, n})^{-1}\right]\right|_{s=\jmath}&=-(O^{\jmath, n})^{-1}\left.\frac{d}{ds}\left[O^{s, n}\right]\right|_{s=\jmath}(O^{\jmath, n})^{-1}=(O^{\jmath, n})^{-1}K_{{\rm V}, \alpha^{\jmath, n-1}}.
\end{align*}
Hence
\[
O^{s, n}-O^{s, n-1}=- O^{s, n}\int_{0}^{s}(O^{\jmath, n})^{-1}\left(K_{{\rm V}, \alpha^{\jmath, n-1}}-K_{{\rm V}, \alpha^{\jmath, n-2}}\right)O^{\jmath, n-1} d\jmath. 
\]
Using  (\ref{Ito-K-V}), we conclude  that there are  some constants $C, C'$ such that 
\begin{align}
\|O^{s, n}-O^{s, n-1}\|\leq C \int_{0}^{s} \left\|K_{{\rm V}, \alpha^{\jmath, n-1}}-K_{{\rm V}, \alpha^{\jmath, n-2}}\right\|\ d\jmath\leq C' \int_{0}^{s}\| \alpha^{\jmath, n-1}-\alpha^{\jmath, n-2}\|\ d\jmath.\label{iteration-O-norm}
\end{align}
For $\|{\mathtt g}^{s, n}-{\mathtt g}^{s, n-1}\|$, we can use the inductive definitions of $\mathtt g^{s, n}$ and $\mathtt g^{s, n-1}$ to compute that \begin{align*}
{\mathtt g}^{s, n}_t-{\mathtt g}^{s, n-1}_t&=\left(O^{s, n}_{t}-O^{s, n-1}_{t}\right)\int_{0}^{s}[O_{t}^{\jmath, n}]^{-1} {\rm R}_{{\rm V}, \alpha^{\jmath, n-1}}(t)\ d\jmath\\
&\ \ \ \ + O^{s, n-1}_{t}\int_{0}^{s}[O_{t}^{\jmath, n}]^{-1}\left(O_{t}^{\jmath, n-1}-O_{t}^{\jmath, n}\right)[O_{t}^{\jmath, n-1}]^{-1} {\rm R}_{{\rm V}, \alpha^{\jmath, n-1}}(t)\ d\jmath\\
&\ \ \  \ + O^{s, n-1}_{t}\int_{0}^{s}[O_{t}^{\jmath, n-1}]^{-1}\left({\rm R}_{{\rm V}, \alpha^{\jmath, n-1}}(t)- {\rm R}_{{\rm V}, \alpha^{\jmath, n-2}}(t)\right)\ d\jmath\\
&=: {\rm (a)}_t+{\rm (b)}_t+{\rm (c)}_t. 
\end{align*}
Hence 
\[
\|{\mathtt g}^{s, n}-{\mathtt g}^{s, n-1}\|\leq \|{\rm (a)}\|+  \|{\rm (b)}\| +\|{\rm (c)}\|. 
\]
Since ${\rm V}$ is bounded on $\M$ and $\mathtt{s}$ is $C^1$ on $[0, T]$, ${\rm R}_{{\rm V}, \alpha^{\jmath, n-1}}(t)$ is also bounded. 
So there are some  constants $C, C'$ such that 
\begin{align*}
{\rm \|(a)\|}&\leq C s_0\big\|O^{s, n}-O^{s, n-1}\big\|\leq C's_0\int_{0}^{s}\| \alpha^{\jmath, n-1}-\alpha^{\jmath, n-2}\|\ d\jmath,\\
{\rm \|(b)\|}&\leq C\int_{0}^{s}\big\|O^{\jmath, n-1}-O^{\jmath, n}\big\|\ d\jmath\leq C' \int_{0}^{s}\| \alpha^{\jmath, n-1}-\alpha^{\jmath, n-2}\|\ d\jmath. 
\end{align*}
For ${\rm (c)}$,  we also have 
\[
{\rm \|(c)\|}\leq C\int_{0}^{s}\left\|{\rm R}_{{\rm V}, \alpha^{\jmath, n-1}}- {\rm R}_{{\rm V}, \alpha^{\jmath, n-2}}\right\|\ d\jmath\leq C'\int_0^s\left\|{\mho}^{\jmath, n-1}-{\mho}^{\jmath, n-2}\right\|\ d\jmath,
\]
where the last norm is measured using the distance function on $\mathcal{O}^{\wt{g}}(\M)$.  Recall that 
\begin{eqnarray}\label{u-s-int-expression-1}
{\mho}^{\jmath, n}_t={\mho}^{\jmath, n}_0+\int_{0}^{t}\sum_{i=1}^{m}H({\mho}^{\jmath, n}_\tau, e_i)\circ d\alpha_\tau^{\jmath, n; i}, \ 0\leq t\leq T,
\end{eqnarray}
where $\alpha^{\jmath, n; i}$, $i=1, \cdots, m$, denotes the $i$-th component of $\alpha^{\jmath, n}$. 
Embedding  $\mathcal{O}^{\wt{g}}(\M)$ into some higher dimensional Euclidean space $\Bbb R^l$ and extending all $H(\cdot, e_i)$ to a small tube neighborhood of $\mathcal{O}^{\wt{g}}(\M)$, we can consider (\ref{u-s-int-expression-1}) as a Euclidean SDE and compare $\left\|{\mho}^{\jmath, n-1}-{\mho}^{\jmath, n-2}\right\|$  with  $\|\alpha^{\jmath, n-1}-\alpha^{\jmath, n-2}\|$ using the Euclidean norm. 
By (\ref{u-s-int-expression-1}), 
\begin{align*}
{\mho}^{\jmath, n-1}_t-{\mho}^{\jmath, n-2}_t=&\int_{0}^{t}\big(H({\mho}^{\jmath, n-1}_\tau, e_i)-H({\mho}^{\jmath, n-2}_\tau, e_i)\big)\circ (O_\tau^{\jmath, n-1}dB_{\tau})^i\\
&+\int_{0}^{t}H({\mho}^{\jmath, n-2}_\tau, e_i)\circ ((O_\tau^{\jmath, n-1}-O_\tau^{\jmath, n-2})dB_{\tau})^i\\
&+\int_{0}^{t}\big(H({\mho}^{\jmath, n-1}_\tau, e_i)-H({\mho}^{\jmath, n-2}_\tau, e_i)\big)\  {\mathtt g}_\tau^{\jmath, n-1; i}d\tau\\
&+\int_{0}^{t}H({\mho}^{\jmath, n-2}_\tau, e_i) ({\mathtt g}_\tau^{\jmath, n-1; i}-{\mathtt g}_\tau^{\jmath, n-2; i})d\tau\\
=:&\  ({\rm I})_t+ ({\rm II})_t+({\rm III})_t+({\rm IV})_t.
\end{align*}
Using (\ref{abcq}),  we obtain,  for $\wt{t}\in [0, T]$,  
\begin{align*}
\E\sup_{t\in [0, \wt{t}]}\big|{\mho}^{\jmath, n-1}_t-{\mho}^{\jmath, n-2}_t\big|^2\leq 4\E\left(\sup_{t\in [0, \wt{t}]}\left|({\rm I})_t\right|^2+\sup_{t\in [0, \wt{t}]} \left|({\rm II})_t\right|^2+ \sup_{t\in [0, \wt{t}]}\left| ({\rm III})_t\right|^2+\sup_{t\in [0, \wt{t}]} \left|({\rm IV})_t\right|^2\right).
\end{align*}
For $({\rm I})_t$, we can consider its It\^{o} form and then apply Doob's inequality of sub-martingales and Lemma \ref{Ku-lem}, which gives 
\begin{align*}
\E\sup_{t\in [0, \wt{t}]}\left|({\rm I})_t\right|^2\leq C\E\int_{0}^{\wt{t}}\left| {\mho}^{\jmath, n-1}_\tau-{\mho}^{\jmath, n-2}_\tau\right|^2\ d\tau\leq C\int_{0}^{\wt{t}}\E\sup_{t\in [0, \tau]}\big|{\mho}^{\jmath, n-1}_t-{\mho}^{\jmath, n-2}_t\big|^2\ d\tau.
\end{align*}
The same argument shows there is some $C$ such that 
\begin{align*}
\E\sup_{t\in [0, \wt{t}]}\left|({\rm II})_t\right|^2\leq C\E\int_{0}^{\wt{t}}\left| {O}^{\jmath, n-1}_\tau-{O}^{\jmath, n-2}_\tau\right|^2\ d\tau\leq CT\left\| \alpha^{\jmath, n-1}-\alpha^{\jmath, n-2}\right\|^2.\end{align*}
Note that  $\|H(\cdot, e_i)\|$  and $|\mathtt g^{s, n-1}|$ (for all $s$ and $n$) are bounded.  Hence 
\begin{align*}
\E\sup_{t\in [0, \wt{t}]}\left|({\rm III})_t\right|^2\leq C\wt{t} \E\int_{0}^{\wt{t}}\left| {\mho}^{\jmath, n-1}_\tau-{\mho}^{\jmath, n-2}_\tau\right|^2\ d\tau\leq CT\int_{0}^{\wt{t}}\E\sup_{t\in [0, \tau]}\left|{\mho}^{\jmath, n-1}_t-{\mho}^{\jmath, n-2}_t\right|^2\ d\tau.
\end{align*}
For  $({\rm IV})_t$, we have
\begin{align*}
\E\sup_{t\in [0, \wt{t}]}\left|({\rm IV})_t\right|^2\leq C\wt{t}^2\left\|\mathtt g^{\jmath, n-1}-\mathtt g^{\jmath, n-2}\right\|\leq  CT^2\left\| \alpha^{\jmath, n-1}-\alpha^{\jmath, n-2}\right\|^2. 
\end{align*}
Altogether, there are  some constants $C_2(T), C_3(T)$ such that 
\[
\E\sup_{t\in [0, \wt{t}]}\big|{\mho}^{\jmath, n-1}_t\!-{\mho}^{\jmath, n-2}_t\big|^2\leq C_2(T)\left\| \alpha^{\jmath, n-1}\!-\alpha^{\jmath, n-2}\right\|^2+C_3(T)\int_{0}^{\wt{t}}\E\sup_{t\in [0, \tau]}\big|{\mho}^{\jmath, n-1}_t\!-{\mho}^{\jmath, n-2}_t\big|^2\ d\tau.
\]
Applying Lemma \ref{Gronwall}, we obtain some  constant $C(T)$ independent of $\jmath$ such that 
\[
\left\|{\mho}^{\jmath, n-1}-{\mho}^{\jmath, n-2}\right\|\leq C(T)\left\| \alpha^{\jmath, n-1}-\alpha^{\jmath, n-2}\right\|. 
\]
So, 
\[
{\rm \|(c)\|}\leq C\int_0^s\left\|{\mho}^{\jmath, n-1}-{\mho}^{\jmath, n-2}\right\|\ d\jmath\leq C(T)\int_{0}^{s}\left\| \alpha^{\jmath, n-1}-\alpha^{\jmath, n-2}\right\|\ d\jmath. 
\]
Putting {together the estimations of ${\rm \|(a)\|}, {\rm \|(b)\|}$ and ${\rm \|(c)\|}$,}  we conclude that 
\[
\left\|{\mathtt g}^{s, n}-{\mathtt g}^{s, n-1}\right\|\leq C(T) \int_{0}^{s}\left\|\alpha^{\jmath, n-1}-\alpha^{\jmath, n-2}\right\|\ d\jmath. 
\]
This and (\ref{iteration-O-norm}) imply  (\ref{iteration-alpha}). Hence the limit 
\[
\lim_{n\rightarrow\infty}\alpha^{s, n}=\int_{0}^{\cdot} O^{s}_{\tau}\ d\tau+\int_{0}^{\cdot}{\mathtt g}^{s}_{\tau}\ d\tau=:\alpha^s
\]
exists  and $\alpha^s$ satisfies the equation (\ref{unique-alpha-s}) by i) of  Lemma \ref{Cor-O-g-formula}.   

The  $\{\alpha^s\}$ obtained by the above iteration is  the the unique parameter of processes in $\mathcal{A}$ satisfying (\ref{unique-alpha-s}). Assume  $\{\widetilde{\alpha}^s\}\subset \mathcal{A}$ is another parameter of processes solving (\ref{unique-alpha-s}). Then, by  
using (\ref{iteration-O}), (\ref{iteration-g}) for $\alpha^s$ and $\widetilde{\alpha}^s$, respectively, the above argument shows that  (\ref{iteration-alpha}) holds true by replacing $\alpha^{s, n}$, $\alpha^{s, n-1}$ by $\alpha^{s}$, $\wt{\alpha}^{s}$, respectively, for all $s$, i.e.,  \[
\|\alpha^s-\widetilde{\alpha}^s\|\leq C(T)\int_{0}^{s}\| \alpha^{\jmath}-\widetilde{\alpha}^{\jmath}\|\ d\jmath.
\]
This implies $\alpha^s=\widetilde{\alpha}^s$ by Gronwall's lemma.   

We  proceed to show   \begin{align}\label{E-Upsilon-aim}\E\big|\underline{\Upsilon}_{{\rm V}, \alpha^{s'}}(t')-\underline{\Upsilon}_{{\rm V}, \alpha^s}(t)\big|^\flat\leq C(\flat, s_0, T)\left(|t'-t|^{\frac{1}{2}\flat}+|s'-s|^{\flat}\right)
\end{align}
for any $\flat>4$,  $t, t'\in [0, T]$ and  $s, s'\in [-s_0, s_0]$. 
This,  by  applying Lemma \ref{Kol-criterion},   will imply that $\underline{\Upsilon}_{{\rm V}, {\alpha}^s}(t)$  has a continuous modification  in the parameter $(t, s)$.  Without loss of generality, we  assume $t<t'$.  
Using (\ref{Ito-Upsilon}) and (\ref{abcq}),  we  compute that 
\begin{align*}
&\E\big|\underline{\Upsilon}_{{\rm V}, \alpha^{s'}}(t')-\underline{\Upsilon}_{{\rm V}, \alpha^s}(t)\big|^\flat\\
&\leq 5^{\flat-1}\left(\ \E\big|\int_{t}^{t'}{\rm R}_{{\rm V}, \alpha^{s'}}(\tau)\ d\tau\big|^{\flat}+\E\big|\int_{t}^{t'}\langle K_{{\rm V}, \alpha^{s'}}(\tau), d\alpha_{\tau}^{s'}\rangle\big|^{\flat}+\E\big|\int_{0}^{t}({\rm R}_{{\rm V}, \alpha^{s'}}-{\rm R}_{{\rm V}, \alpha^s})(\tau)\ d\tau\big|^{\flat}
\right.\\
&\ \ \ \ \ \ \ \ \ \ \ \ \  -\E\big|\int_{0}^{t}\langle (K_{{\rm V}, \alpha^{s'}}-K_{{\rm V}, \alpha^{s}})(\tau), d\alpha_{\tau}^{s'}\rangle\big|^{\flat}+\left.\E\big|\int_{0}^{t}\langle K_{{\rm V}, \alpha^{s}}(\tau), d(\alpha_{\tau}^{s'}-\alpha_{\tau}^s)\rangle\big|^{\flat}\right)\\
&\ \ \ =: 5^{\flat-1}\big((\rm d)+(\rm e)+(\rm f)+(\rm g)+(\rm h)\big). 
\end{align*}
To conclude (\ref{E-Upsilon-aim}), we will show 
\begin{align}\label{d-e-f-g-h}
({\rm d}), ({\rm e})\leq C(\flat, s_0, T)|t'-t|^{\frac{1}{2}\flat}\ \mbox{and}\ ({\rm f}), ({\rm g}), ({\rm h})\leq C(\flat, s_0, T)|s'-s|^{\flat}. 
\end{align}
Clearly,  \[
({\rm d})\leq C^\flat\left|t'-t\right|^{\flat}\leq (CT)^{\flat}|t'-t|^{\frac{1}{2}\flat} 
\]
for some $C$ which bounds $|{\rm R}_{{\rm V}, \alpha^{\jmath}}|$.  
For $(\rm e)$, we have 
\begin{eqnarray*}
2^{1-\flat}({\rm e})\leq \E\big|\int_{t}^{t'}\langle K_{{\rm V}, \alpha^{s'}}(\tau), O_{\tau}^{s'}dB_{\tau}\rangle\big|^{\flat}+\E\big|\int_{t}^{t'}\langle K_{{\rm V}, \alpha^{s'}}(\tau), {\mathtt g}_{\tau}^{s'}\rangle\ d\tau\big|^{\flat}=:({\rm e})_1+({\rm e})_2. 
\end{eqnarray*}
By Lemma \ref{Ku-lem}, with the constant ${\mathtt{C}}_1(\flat)$ there, 
\begin{align*}
(\rm e)_1\leq {\mathtt{C}}_1(\flat)\E\big|\int_{t}^{t'} |K_{{\rm V}, \alpha^{s'}}(\tau)|^2\ d\tau\big|^{\frac{1}{2}\flat}\leq{\mathtt{C}}_1(\flat) \big(\int_{t}^{t'}\E (|K_{{\rm V}, \alpha^{s'}}(\tau)|^\flat)\ d\tau\big)\cdot |t'-t|^{\frac{1}{2}\flat-1}.
\end{align*}
Using (\ref{abcq}) and Lemma \ref{Ku-lem}, it is easy to deduce that 
\begin{align}
\E (|K_{{\rm V}, \alpha^{s'}}(\tau)|^\flat)\leq&\ 3^{\flat-1}\left(\E\left|\int_{0}^{\tau}({\mho_{\tau'}^{s'}})^{-1}R\left({\mho_{\tau'}^{s'}}O_{\tau'}^{s'}dB_{\tau'}, {\mho_{\tau'}^{s'}}(\mho_{0}^{s'})^{-1}[\mathtt{s}(\tau'){\rm V}(F^{s'} y)]\right){\mho_{\tau'}^{s'}}\right|^{\flat} \right.\notag\\
&  +\E\left|\int_{0}^{\tau}({\mho_{\tau'}^{s'}})^{-1}R\left({\mho_{\tau'}^{s'}}{\mathtt g}_{\tau'}^{s'}d\tau', {\mho_{\tau'}^{s'}}(\mho_{0}^{s'})^{-1}[\mathtt{s}(\tau'){\rm V}(F^{s'} y)]\right){\mho_{\tau'}^{s'}}\right|^{\flat}\notag\\
&  +\left.\E\left|\int_{0}^{\tau} ({\mho_{\tau'}^{s'}})^{-1}\!(\nabla ({\mho_{\tau'}^{s'}} e_i)R)\big({\mho_{\tau'}^{s'}}e_i, {\mho_{\tau'}^{s'}}\!(\mho_0^{s'})^{-1}\![\mathtt{s}(\tau'){\rm V}(F^{s'}y)]\big){\mho_{\tau'}^{s'}}\ d\tau' \right|^{\flat}\right)\notag\\
\leq &\ 3^{\flat-1}C(\tau^{\frac{1}{2}\flat}+\tau^{\flat}). \label{K-in-t-est}
\end{align}
Hence there is a constant $C(\flat, T)$  such that 
\[
({\rm e})_1\leq C(\flat, T)|t'-t|^{\frac{1}{2}\flat}. 
\]
Since $|{\mathtt g}_{\tau}^{s}|$ is bounded by some constant depending on $s_0$ and $\sup|{\rm V}|$, using H\"{o}lder's inequality and the estimation in (\ref{K-in-t-est}), we obtain \[
({\rm e})_2\leq \big(\int_{t}^{t'}\E (|K_{{\rm V}, \alpha^{s'}}(\tau)|^\flat)\ d\tau\big)\cdot |t'-t|^{\flat-1}\leq C(\flat, T)|t'-t|^{\flat}.
\]
Thus, 
\[
({\rm e})\leq 2^{\flat-1}\left(({\rm e})_1+({\rm e})_2\right)\leq 2^{\flat-1}(T^{\frac{1}{2}\flat}+1)C(\flat, T)|t'-t|^{\frac{1}{2}\flat}=C(\flat, T)|t'-t|^{\frac{1}{2}\flat}.
\]
Using H\"{o}lder's inequality and  Burkholder's inequality,  the conclusion  in (\ref{d-e-f-g-h}) for $({\rm f}), ({\rm g})$ and  $ ({\rm h})$  holds if 
\[
\E\big|({\rm R}_{{\rm V}, \alpha^{s'}}-{\rm R}_{{\rm V}, \alpha^s})(\tau)\big|^{\flat},\  \E\big|(K_{{\rm V}, \alpha^{s'}}-K_{{\rm V}, \alpha^{s}})(\tau)\big|^{\flat}, \  \E\big|\alpha_{\tau}^{s'}-\alpha_{\tau}^s\big|^{\flat}\leq C(\flat, s_0, T)|s'-s|^{\flat},
\]
which can be further reduced to  verifying 
\[
 \E\big|{O}^{s'}_\tau-{O}^{s}_\tau\big|^{\flat}, \ \E\big|{\mathtt g}_{\tau}^{s'}-{\mathtt g}_{\tau}^{s}\big|^\flat,\  \E\big|{\mho}^{s'}_\tau-{\mho}^{s}_\tau\big|^{\flat}\leq C(\flat, s_0, T)|s'-s|^{\flat}. 
\]
By  (\ref{iteration-O}) and (\ref{K-in-t-est}), there is some constant $C(\flat, T)$ such that 
\begin{align}
 \E\big|{O}^{s'}_\tau-{O}^{s}_\tau\big|^{\flat} =\E\left|\int_{s}^{s'}K_{{\rm V}, \alpha^{\jmath}}(\tau)O^{\jmath}_{\tau}\ d\jmath\right|^{\flat}&\leq\left|\int_{s}^{s'}\E (|K_{{\rm V}, \alpha^{s'}}(\tau)|^\flat)
  \ d\tau\right|\cdot |s'-s|^{\flat-1}\notag\\
 &\leq C(\flat, T)|s'-s|^{\flat}. \label{c-kom-cri-O}
\end{align}
Using (\ref{iteration-g}),  (\ref{c-kom-cri-O}) and H\"{o}lder's inequality,  we obtain  some constant $C(\flat, s_0, T)$ such that 
\begin{align}
\E\big|{\mathtt g}_{\tau}^{s'}-{\mathtt g}_{\tau}^{s}\big|^\flat&\leq 2^{\flat-1}\!\left(\E\big(|O_{\tau}^{s'}-O_{\tau}^s|^{\flat}\cdot \big|\int_{0}^{s'}[O_{\tau}^{\jmath}]^{-1}{\rm R}_{{\rm V}, \alpha^{\jmath}}(\tau) \ d\jmath\big|^{\flat}\big)+\E\left|\int_{s}^{s'}[O_{\tau}^{\jmath}]^{-1}{\rm R}_{{\rm V}, \alpha^{\jmath}}(\tau) \ d\jmath\right|^{\flat}\right)\notag\\
&\leq C(\flat, s_0,  T)|s'-s|^{\flat}. \label{c-kom-cri-g}
\end{align}
Recall that each $\mho^s$ satisfies the SDE
\begin{align*}
{\mho}^{s}_\tau({\rm w})={\mho}^{s}_0+\int_{0}^{\tau}\sum_{i=1}^{m}H({\mho}^{s}_{\tau'}({\rm w}), e_i)\circ d\alpha_{\tau'}^{s; i}({\rm w}), \ \forall \tau\in [0, T].
\end{align*}
As before, we can treat  it as a Euclidean SDE.  Hence,  
\begin{align*}
\E\big|{\mho}^{s'}_\tau-{\mho}^{s}_\tau\big|^{\flat}&\leq 3^{\flat-1}\left(\big|{\mho}^{s'}_0-{\mho}^{s}_0\big|^{\flat} +\E\big|\int_{0}^{\tau}\sum_{i=1}^{m}H({\mho}^{s}_{\tau'}, e_i)\circ d\big(\alpha_{\tau'}^{s', i}-\alpha_{\tau'}^{s, i}\big)\big|^{\flat}\right.\\
&\ \ \ \ \ \ \ \ \ \ \ \ \ \ \  \ \ \ \ \ \ \ \ \ \ \ \ \left. +\E\big|\int_{0}^{\tau}\sum_{i=1}^{m}\big(H({\mho}^{s'}_{\tau'}, e_i)-H({\mho}^{s}_{\tau'}, e_i)\big)\circ d\alpha_{\tau'}^{s', i}\big|^{
\flat}\right)\\
&=:3^{\flat-1}\big(({\rm i})+({\rm j})+({\rm k})\big). 
\end{align*}
Clearly, $({\rm i})\leq C|s'-s|^{\flat}$ for some $C$ depending on $\sup|{\rm V}|$ and $\wt{g}$. 
 For ${\rm (j)}$,  we have 
\begin{align*}
({\rm j})&\leq \E\big|\int_{0}^{\tau}\sum_{i=1}^{m}H({\mho}^{s}_{\tau'}, e_i)\circ ({O}^{s'}_{\tau'}-{O}^{s}_{\tau'})dB_{
\tau'})^i\big|^{\flat}+ \E\big|\int_{0}^{\tau}\sum_{i=1}^{m}H({\mho}^{s}_{\tau'}, e_i) ({\mathtt g}^{s'}_{\tau'}-{\mathtt g}^{s}_{\tau'})^id{
\tau'}\big|^{\flat}\\
&=:({\rm j})_1+({\rm j})_2. 
\end{align*}
where we use the superscript $i$ to denote the $i$-th component of a vector.   For $({\rm j})_1$, we can transfer the integral into It\^{o}'s form.  Note that all $H(\cdot, e_i)$ are $C^1$ vector fields on $\mathcal{O}^{\wt{g}}(\M)$ with bounded first order differentials. Hence, using Lemma \ref{Ku-lem}, H\"{o}lder's inequality and (\ref{c-kom-cri-O}), we can  conclude that there is some constant $C(\flat, T)$ such that   
\begin{align*}
({\rm j})_1&\leq {C}(\flat, T) \big(\E\big(\int_{0}^{\tau}|{O}^{s'}_{\tau'}-{O}^{s}_{\tau'}|^2\ d\tau'\big)^{\frac{\flat}{2}}+\E\big(\int_{0}^{\tau}|{O}^{s'}_{\tau'}-{O}^{s}_{\tau'}|^2\ d\tau'\big)^{\flat}\big)\\
&\leq {C}(\flat, T)(T^{\frac{1}{2}\flat-1}+T^{\flat-1})\int_{0}^{\tau}\big(\E|{O}^{s'}_{\tau'}-{O}^{s}_{\tau'}|^{\flat}+\E|{O}^{s'}_{\tau'}-{O}^{s}_{\tau'}|^{2\flat}\big)\ \! d\tau'\\
&\leq  {C}(\flat, T)(2s_0)^{\flat}|s'-s|^{\flat}.
\end{align*}
For $({\rm j})_2$, we can use H\"{o}lder's inequality and  (\ref{c-kom-cri-g}) to conclude that 
\[
({\rm j})_2\leq  C(\flat, s_0,  T)T^{\flat}|s'-s|^{\flat}. 
\]
For $({\rm k})$, the same argument as for $({\rm j})$ gives some constant $C(\flat, s_0, T)$ such that
\[
({\rm k})\leq  C(\flat, s_0,  T)\int_{0}^{\tau}\E\big|{\mho}^{s'}_{\tau'}-{\mho}^{s}_{\tau'}\big|^{\flat}\ d\tau'.
\]
Altogether, there is some constant $C(\flat, s_0, T)$ such that 
\[
\E\big|{\mho}^{s'}_\tau-{\mho}^{s}_\tau\big|^{\flat}\leq C(\flat, s_0, T)\left(|s'-s|^{\flat}+\int_{0}^{\tau}\E\big|{\mho}^{s'}_{\tau'}-{\mho}^{s}_{\tau'}\big|^{\flat}\ d\tau'\right).
\]
The inequality also holds for $\sup_{\wt{\tau}\in [0, \tau]}\E\big|{\mho}^{s'}_{\wt{\tau}}-{\mho}^{s}_{\wt{\tau}}\big|^{\flat}$. Hence we can apply Lemma \ref{Gronwall} to conclude that there is  some constant $C(\flat, s_0, T)$ such that 
\begin{equation}\E\big|{\mho}^{s'}_\tau-{\mho}^{s}_\tau\big|^{\flat}\leq C(\flat, s_0, T)|s'-s|^{\flat}.\label{c-kom-cri-mho}
\end{equation}
This finishes  the proof of (\ref{d-e-f-g-h}) and hence (\ref{E-Upsilon-aim}) holds true.  By Lemma \ref{Kol-criterion},  we can obtain a continuous  modification of  $\underline{\Upsilon}_{{\rm V}, {\alpha}^s}(t)$ in the parameter $(t, s)$.

Let $\alpha^s, \mho^s$ be as above.  By (\ref{c-kom-cri-mho}) and Lemma \ref{Kol-criterion}, $\mho^s$ has  a version such that $s\mapsto {\mho}^s(\rm w)$ is continuous.  By Lemma \ref{Cor-O-g-formula}, to  show $\mho^s$ is differentiable in $s$,  it suffices to show (\ref{diff-mho-s-R-m}) is uniquely solvable with $Y^s_0=(\mho_0^s)'_s$.  Let $Y^{s, 0}\equiv 0$ and  let  $Y^{s, n}$ ($n\geq 1$) be such that 
\begin{align}\label{diff-mho-s-R-m-iteration}
\left\{ \begin{array}{l}
d\theta(Y^{s, n}_t)\  =\ \varpi\big(Y^{s, n-1}_t\big)\circ d\alpha^s_t+\circ d\underline{\Upsilon}_{{\rm V}, \alpha^{s}},\\
d\varpi(Y^{s, n}_t) = (\mho_t^{s})^{-1}R\big(\mho_t^s \circ d\alpha^s_t,  \mho_t^s\theta(Y^{s, n-1}_t)\big)\mho_t^s.
\end{array}\right.
\end{align}
For a $\Bbb R^m\times \mathcal{F}(\Bbb R^m)$ valued process  $(\mathfrak{v},  \mathfrak{Q})_{t\in [0, T]}$, let 
\[
\|(\mathfrak{v},  \mathfrak{Q})\|:=\sqrt{\|\mathfrak{v}\|^2+\|\mathfrak{Q}\|^2}, \ \mbox{where}\ \|\mathfrak{v}\|^2=\E\sup_{t\in [0, T]}|\mathfrak{v}_t|^2, \|\mathfrak{Q}\|^2=\E\sup_{t\in [0, T]}|\mathfrak{Q}_t|^2.
\]
We show the sequence $(\theta, \varpi)(Y^{s, n})$ converges  in norm $\|\cdot\|$.  Clearly, 
\begin{equation}\label{Y-iteration-1st}
\left\|(\theta, \varpi)\big(Y^{s, 1}\big)-(\theta, \varpi)\big(Y^{s, 0}\big)\right\|\leq CT\|\alpha^s\|.  
\end{equation}
We continue to estimate $\|(\theta, \varpi)\big(Y^{s, n}\big)-(\theta, \varpi)\big(Y^{s, n-1}\big)\|$, $n\geq 2$. By (\ref{diff-mho-s-R-m-iteration}), 
\begin{align*}
\left\{ \begin{array}{l}
d\big(\theta(Y^{s, n}_t)-\theta(Y^{s, n-1}_t)\big)\ =\big(\varpi(Y^{s, n}_t)-\varpi(Y^{s, n-1}_t)\big)\circ d\alpha^s_t,\\
d\big(\varpi(Y^{s, n}_t)-\theta(Y^{s, n-1}_t)\big)=(\mho_t^{s})^{-1}R\big(\mho_t^s \circ d\alpha^s_t,  \mho_t^s\big(\theta(Y^{s, n}_t)-\theta(Y^{s, n-1}_t)\big)\big)\mho_t^s.
\end{array}\right.
\end{align*}
Following the above discussion on $\left\|{\mho}^{\jmath, n-1}-{\mho}^{\jmath, n-2}\right\|$, we can use Doob's inequality of submartingale and Lemma \ref{Ku-lem} to conclude that 
\begin{align} \begin{array}{l}
\E\sup_{t\in [0, \wt{t}]}\big|(\theta, \varpi)\big(Y^{s, n}\big)-(\theta, \varpi)\big(Y^{s, n-1}\big)\big|^2\\
\ \ \ \ \leq C(s_0, T)\int_{0}^{\wt{t}}\E\sup_{t\in [0, \tau]}\big|(\theta, \varpi)\big(Y^{s, n-1}\big)-(\theta, \varpi)\big(Y^{s, n-2}\big)\big|^2\ d\tau.\label{uni-Y-t-s-iter}
\end{array}
\end{align}
Iterating this inequality for $n$ steps,  which, together with (\ref{Y-iteration-1st}),  imply  
\[
\E\sup_{t\in [0, \wt{t}]}\big|(\theta, \varpi)\big(Y^{s, n}\big)-(\theta, \varpi)\big(Y^{s, n-1}\big)\big|^2\leq \frac{1}{n!}(CT+C(s_0, T))^n \wt{t}^n.
\]
In particular, when  $\wt{t}=T$, this is 
\[
\big\|(\theta, \varpi)\big(Y^{s, n}\big)-(\theta, \varpi)\big(Y^{s, n-1}\big)\big\|\leq \frac{1}{n!}(CT+C(s_0, T))^n T^n. 
\]
Hence   $(\theta, \varpi)(Y^{s, n})$ converges in $\|\cdot\|$ with some limit $(\theta, \varpi)(Y^s)$ which solves  (\ref{diff-mho-s-R-m}). 

Such a solution $Y^s$ is unique.  Assume $\mathcal{Y}^s$ is another solution to (\ref{diff-mho-s-R-m}) with $\mathcal{Y}^s_0=(\mho_0^s)'_s$. Then the same argument as for (\ref{uni-Y-t-s-iter}) shows that 
\[
\E\sup_{t\in [0, \wt{t}]}\big|(\theta, \varpi)(Y^{s}_t)-(\theta, \varpi)(\mathcal{Y}^{s}_t)\big|^2\leq C(s_0, T)\int_{0}^{\wt{t}}\E\sup_{t\in [0, \tau]}\big|(\theta, \varpi)(Y^{s}_t)-(\theta, \varpi)(\mathcal{Y}^{s}_t)\big|^2\ d\tau,
\]
from which we can conclude $Y^s=\mathcal{Y}^s$ by Gronwall's Lemma.  

By Corollary \ref{cor-mho-t-s-d},  the solution $Y^s$ to (\ref{diff-mho-s-R-m}) is actually given by (\ref{Y-t-s-diff-ITO}).  Hence,  to show  the process $Y^s_t$  has a continuous modification in the parameter $(t, s)$,  it suffices to show both $\mho_{t}^s$  and $\varpi(Y_t^s)$ have  a $(t, s)$-continuous version.  Let  $\flat>4$, $t, t'\in [0, T]$ with $t<t'$ and $s, s'\in [-s_0, s_0]$.  Using  (\ref{c-kom-cri-mho}) and applying Burkholder's inequality and  H\"{o}lder's inequality to the difference ${\mho}^{s}_{t'}-{\mho}^{s}_t$,  we obtain 
\[
2^{1-\flat}\E\big|{\mho}^{s'}_{t'}-{\mho}^{s}_t\big|^{\flat}\leq \E\big|{\mho}^{s'}_{t'}-{\mho}^{s}_{t'}\big|^{\flat}+\E\big|{\mho}^{s}_{t'}-{\mho}^{s}_t\big|^{\flat}\leq C(\flat, s_0, T)\big(|s'-s|^{\flat} +|t'-t|^{\frac{1}{2}\flat}\big). 
\]
So Lemma \ref{Kol-criterion} applies and shows that there is a version of $\mho_{t}^s$ which is  continuous in the parameter $(t, s)$. 
Since $\varpi(Y_0^s)=0$, by  (\ref{Y-t-s-diff-ITO}), 
\begin{align*}
\varpi(Y_t^s)=& \int_{0}^{t}(\mho_\tau^s)^{-1}R\big(\mho_\tau^s d\alpha^s_\tau,  \mathtt{s}(\tau)\mho_\tau^s({\mho}_0^s)^{-1}V(F^s y)\big)\mho_\tau^s\\
&+\int_{0}^{t}(\mho_\tau^s)^{-1}(\nabla ({\mho}_{\tau}^se_i)R)\left({\mho}_{\tau}^s e_i,  \mathtt{s}(\tau)\mho_\tau^s({\mho}_0^s)^{-1}V(F^s y)\right){\mho}_{\tau}^s\ d\tau.
\end{align*}
Again, by Burkholder's inequality and  H\"{o}lder's inequality,  it is easy to deduce that 
\begin{align*}
&\E\big|\varpi(Y_{t'}^{s'})-\varpi(Y_{t'}^{s})\big|^{\flat}\\
 &\leq C(s_0, b, T)\big(\E\big|\int_{0}^{t'}|\alpha_\tau^{s'}-\alpha_{\tau}^s|^{2}\ d\tau\big|^{\frac{\flat}{2}}+\E\big|\int_{0}^{t'}|\mho_\tau^{s'}-\mho_{\tau}^s|^{2}\ d\tau\big|^{\frac{\flat}{2}}+\E\int_{0}^{t'}|\mho_\tau^{s'}-\mho_{\tau}^s|^{{\flat}}\ d\tau \big)\\
 &\leq C(s_0, b, T)\int_0^{t'}\big(\E|\alpha_\tau^{s'}-\alpha_{\tau}^s|^{\flat}+\E|\mho_\tau^{s'}-\mho_{\tau}^s|^{\flat}\big)\ d\tau\\
 &\leq C(\flat, s_0, T)|s'-s|^{\flat}
\end{align*}
and 
\begin{align*}
\E\big|\varpi(Y_{t'}^{s})-\varpi(Y_{t}^{s})\big|^{\flat}\leq C(s_0, b, T)\left(|t'-t|^{\frac{\flat}{2}}+|t'-t|^{\flat}\right)\leq C(s_0, b, T)|t'-t|^{\frac{\flat}{2}}.
\end{align*}
Hence 
\begin{align*}
2^{1-\flat}\E\big|\varpi(Y_{t'}^{s'})-\varpi(Y_{t}^{s})\big|^{\flat}&\leq \E\big|\varpi(Y_{t'}^{s'})-\varpi(Y_{t'}^{s})\big|^{\flat}+\E\big|\varpi(Y_{t'}^{s})-\varpi(Y_{t}^{s})\big|^{\flat}\\
&\leq C(\flat, s_0, T)\big(|s'-s|^{\flat} +|t'-t|^{\frac{1}{2}\flat}\big),
\end{align*}
which implies that  $\varpi(Y_t^s)$ has a $(t, s)$-continuous modification  by Lemma \ref{Kol-criterion}. 

Now we have shown {\rm i)} and {\rm ii)}.  Hence we can use  Lemma \ref{Cor-O-g-formula} to conclude  that ${\rm y}^s=\mathcal{I}_{\mho_0^s}(\alpha^s)$  is differentiable in $s$ and satisfies  (\ref{unique-y-s}).  The differential  process 
\[
\Upsilon_{{\rm V}, {\rm y}^{s}}(t, {\rm w})={\mathtt s}(t)\mho_{t}^s(\mho_0^s)^{-1}V(F^sy)
\]
has a $(t, s)$-continuous version since  $\mho_{t}^s$ does. 

Finally,  by Lemma \ref{diff-D-F-s-crit}, for {\rm iv)}, it suffices to show  for ${\bf\mathsf v}_{t_1}^s\in T_{\mho_{t_1}^s}\mathcal{F}(\M)$ $C^1$ in $s$, (\ref{mathsfv-t-s-d-Ito}) is uniquely  solvable with initial $({\bf\mathsf v}_{t_1}^s)'_s$.  Again, this can be done by Picard's iteration method. Let $\upsilon_t^{s, 0}=(\theta, \varpi)_{\mho_t^s}^{-1}(\theta, \varpi)_{\mho_{t_1}^s}({\bf\mathsf v}_{t_1}^s)'_s$. For $n\geq 1$, let $\upsilon_t^{s, n}$ with initial $({\bf\mathsf v}_{t_1}^s)'_s$ be such that 
\begin{align*}
\left\{ \begin{array}{ll}
d\big(\theta(\upsilon_t^{s, n})\big)\ =\varpi\big(\upsilon_t^{s, n-1}\big) d \alpha^s_t+{\rm Ric}({\mho}_{t}^s\theta(\upsilon_t^{s,n-1}))\ dt+\theta\left(\circledast_{{\rm I}}\big({\bf\mathsf v}_t^s, (\mho_t^s)'_s\big)\right)+\circledast_{{\rm A}}^{\theta}\big({\bf\mathsf v}_t^s, (\mho_t^s)'_s\big),\\
d\big(\varpi(\upsilon_t^{s, n})\big)=(\mho_t^s)^{-1}R\big(\mho_t^s d \alpha^s_t,  \mho_t^s\theta(\upsilon_t^{s, n-1})\big)\mho_t^s+(\mho_t^s)^{-1}R\big(\mho_t^s e_i, \mho_t^s\varpi(\upsilon_t^{s, n-1})e_i\big)\mho_t^s\ dt\\
\ \ \ \ \ \ \ \ \ \ \ \ \ \ \ \ \ \ \ +(\mho_t^s)^{-1}(\nabla ({\mho}_{t}^se_i)R)\big({\mho}_{t}^s e_i, {\mho}_{t}^s\theta(\upsilon_t^{s, n-1})\big){\mho}_{t}^s\ dt\\
\ \ \ \ \ \ \ \ \ \ \ \ \ \ \ \ \ \ \  +\varpi\left(\circledast_{{\rm I}}\big({\bf\mathsf v}_t^s, (\mho_t^s)'_s\big)\right)+\circledast_{{\rm A}}^{\varpi}\big({\bf\mathsf v}_t^s, (\mho_t^s)'_s\big),
\end{array}\right.
\end{align*}
where ${\bf\mathsf v}_t^s$, $\circledast_{{\rm I}}\big({\bf\mathsf v}_t^s, (\mho_t^s)'_s\big)$, $\circledast_{{\rm A}}^{\theta}\big({\bf\mathsf v}_t^s, (\mho_t^s)'_s\big)$, $\circledast_{{\rm A}}^{\varpi}\big({\bf\mathsf v}_t^s, (\mho_t^s)'_s\big)$ are as in Lemma \ref{diff-D-F-s-crit}. We will show $(\theta, \varpi)(\upsilon_t^{s, n})$ converges in norm $\|\cdot\|$, where, for any $\Bbb R^m\times \mathcal{F}(\Bbb R^m)$ valued process  $(\mathfrak{v},  \mathfrak{Q})_{t\in [t_1, t_2]}$, 
\[
\|(\mathfrak{v},  \mathfrak{Q})\|:=\sqrt{\|\mathfrak{v}\|^2+\|\mathfrak{Q}\|^2}, \ \|\mathfrak{v}\|^2=\E\sup_{t\in [t_1, t_2]}|\mathfrak{v}_t|^2\ \mbox{and}\ \|\mathfrak{Q}\|^2=\E\sup_{t\in [t_1, t_2]}|\mathfrak{Q}_t|^2.
\]
Clearly, we have
\begin{align*}
\big\|(\theta, \varpi)(\upsilon^{s, 1})-(\theta, \varpi)(\upsilon^{s, 0})\big\|\; <\; &C(s_0, T)+\left\|\!\int_{t_1}^{t}\theta\left(\circledast_{{\rm I}}\big({\bf\mathsf v}_\tau^s, (\mho_\tau^s)'_s\big)\right)\!\right\|+\left\|\!\int_{t_1}^{t}\varpi\left(\circledast_{{\rm I}}\big({\bf\mathsf v}_\tau^s, (\mho_\tau^s)'_s\big)\right)\!\right\|\\
&\ \ \ \ \ \ \ \ \ \ \ +\left\|\!\int_{t_1}^{t}\circledast_{{\rm A}}^{\theta}\big({\bf\mathsf v}_\tau^s, (\mho_\tau^s)'_s\big)\right\|+\left\|\!\int_{t_1}^{t}\circledast_{{\rm A}}^{\varpi}\big({\bf\mathsf v}_\tau^s, (\mho_\tau^s)'_s\big)\right\|\\
\; =: &\; C(s_0, T)+{(\rm A)_1}+{(\rm A)_2}+{(\rm A)_3}+{(\rm A)_4}. 
\end{align*}
Using Doob's inequality of submartingale, Lemma \ref{Ku-lem} and H\"{o}lder's inequality, we see from the expressions of $\circledast_{{\rm I}}\big({\bf\mathsf v}_t^s, (\mho_t^s)'_s\big)$, $\circledast_{{\rm A}}^{\theta}\big({\bf\mathsf v}_t^s, (\mho_t^s)'_s\big)$ and $\circledast_{{\rm A}}^{\varpi}\big({\bf\mathsf v}_t^s, (\mho_t^s)'_s\big)$ that  
\[
{({\rm A})_i}\leq C(s_0, T)\big(\left\|(\theta, \varpi)({\bf\mathsf v}_t^s)\right\|+\left\||(\theta, \varpi)({\bf\mathsf v}_t^s)|^2\right\|\cdot\left\||(\mho_t^s)'_s|^2\right\|\big), \ i=1,\; 2,\; 3,\; 4. 
\]
Note that  the process ${\bf\mathsf v}_t^s$  satisfies the SDE
\begin{align}\label{mathsf-v-t-s-SDE}
\left\{ \begin{array}{ll}
d\big(\theta({\bf\mathsf v}_t^s)\big)\ =\varpi\big({\bf\mathsf v}_t^s\big) d \alpha^s_t+{\rm Ric}({\mho}_{t}^s\theta({\bf\mathsf v}_t^s))\ dt,\\
d\big(\varpi({\bf\mathsf v}_t^s)\big)=(\mho_t^s)^{-1}R\big(\mho_t^s d \alpha^s_t,  \mho_t^s\theta({\bf\mathsf v}_t^s)\big)\mho_t^s+(\mho_t^s)^{-1}R\big(\mho_t^s e_i, \mho_t^s\varpi({\bf\mathsf v}_t^s)e_i\big)\mho_t^s\ dt\\
\ \ \ \ \ \ \ \ \ \ \ \ \ \ \ \  +(\mho_t^s)^{-1}(\nabla ({\mho}_{t}^se_i)R)\left({\mho}_{t}^s e_i, {\mho}_{t}^s\theta({\bf\mathsf v}_t^s)\right){\mho}_{t}^s\ dt.
\end{array}\right.
\end{align}
So, using Doob's inequality of sub-martingales and Lemma \ref{Ku-lem}, we compute that 
\begin{align*}
\E\sup_{t\in [t_1, \wt{t}]}\big|(\theta, \varpi)({\bf\mathsf v}_t^s)\big|^\flat\leq C(s_0, T)\int_{t_1}^{\wt{t}}\E\sup_{t\in [t_1, \tau]}\big|(\theta, \varpi)({\bf\mathsf v}_t^s)\big|^\flat\ d\tau, \ \flat=2,\; 4, 
\end{align*}
which, by Gronwall's lemma,  implies 
\begin{equation}\label{bfmathsf-v-bd}
\|(\theta, \varpi)({\bf\mathsf v}_t^s)\|, \left\||(\theta, \varpi)({\bf\mathsf v}_t^s)|^2\right\|\leq C(s_0, T). 
\end{equation}
With a similar computation,  we  conclude from (\ref{Y-t-s-diff-ITO}) that $\||(\mho_t^s)'_s|^2\|$ is also bounded by constant $C(s_0, T)$. So, 
\begin{equation}\label{upsilon-s-1-0}
\big\|(\theta, \varpi)(\upsilon^{s, 1})-(\theta, \varpi)(\upsilon^{s, 0})\big\|\leq C(s_0, T). 
\end{equation}
For $n\geq 2$, the difference $(\theta, \varpi)\big(\upsilon^{s, n}\big)-(\theta, \varpi)\big(\upsilon^{s, n-1}\big)$ satisfies the SDE
\begin{align*}
\left\{\!\!\begin{array}{ll}
d\big(\theta({\upsilon}_t^{s, n})-\theta({\upsilon}_t^{s, n-1})\big) =\big(\varpi({\upsilon}_t^{s, n-1}\!)\!-\!\varpi({\upsilon}_t^{s, n-2}\!)\big) d \alpha^s_t+{\rm Ric}\big({\mho}_{t}^s\big(\theta({\upsilon}_t^{s, n-1}\!)\!-\!\theta({\upsilon}_t^{s, n-2}\!)\big)\big) dt,\\
d\big(\varpi({\upsilon}_t^{s, n})-\varpi({\upsilon}_t^{s, n-1})\big) =(\mho_t^s)^{-1}R\big(\mho_t^s d \alpha^s_t,  \mho_t^s\big(\theta({\upsilon}_t^{s, n-1}\!)\!-\!\theta({\upsilon}_t^{s, n-2}\!)\big)\big)\mho_t^s\\
\ \ \ \ \ \ \ \ \ \ \ \ \ \ \ \  \ \ \ \ \ \ \ \ \  \ \ \ \  \ \ \ \ \ +(\mho_t^s)^{-1}R\big(\mho_t^s e_i, \mho_t^s\big(\varpi({\upsilon}_t^{s, n-1}\!)\!-\!\varpi({\upsilon}_t^{s, n-2}\!)\big)e_i\big)\mho_t^s\ dt\\
\ \ \ \ \ \ \ \ \ \ \ \ \ \ \ \ \ \ \ \ \ \ \ \ \  \ \ \ \ \ \ \ \  \  +(\mho_t^s)^{-1}(\nabla ({\mho}_{t}^se_i)R)\left({\mho}_{t}^s e_i, {\mho}_{t}^s\big(\theta({\upsilon}_t^{s, n-1}\!)\!-\!\theta({\upsilon}_t^{s, n-2}\!)\big)\right){\mho}_{t}^s\ dt.
\end{array}\right.
\end{align*}
As before, we can use Doob's inequality of sub-martingales and Lemma \ref{Ku-lem} to  obtain 
\begin{align}\label{upsilon-s-t-iter-uni} \begin{array}{l}
\E\sup_{t\in [t_1, \wt{t}]}\big|(\theta, \varpi)\big(\upsilon^{s, n}_t\big)-(\theta, \varpi)\big(\upsilon^{s, n-1}_t\big)\big|^2\\
\ \ \ \ \ \ \ \leq C(s_0, T)\int_{t_1}^{\wt{t}}\E\sup_{t\in [0, \tau]}\big|(\theta, \varpi)\big(\upsilon^{s, n-1}_t\big)-(\theta, \varpi)\big(\upsilon^{s, n-2}_t\big)\big|^2\ d\tau.
\end{array}
\end{align}
Iterate this inequality for $n$ steps and then let $\wt{t}=t_2$.  This,  together with (\ref{upsilon-s-1-0}),  implies
\[
\big\|(\theta, \varpi)\big(\upsilon^{s, n}\big)-(\theta, \varpi)\big(\upsilon^{s, n-1}\big)\big\|\leq \frac{1}{n!}C(s_0, T)^n T^n.  
\]
Hence  $(\theta, \varpi)(\upsilon^{s, n})$ converges in $\|\cdot\|$ with some limit $(\theta, \varpi)(\upsilon^s)$ which solves   (\ref{mathsfv-t-s-d-Ito}).  We can also use  (\ref{upsilon-s-t-iter-uni}) and Gronwall's Lemma to conclude the uniqueness of such   $\upsilon_t^s$. 

For the existence of a continuous version of $\upsilon_t^s$ in the $(t, s)$ parameter, we use Lemma \ref{Kol-criterion}.  Let  $\flat>4$, $t, t'\in [t_1, T]$ with $t<t'$ and $s, s'\in [-s_0, s_0]$.  Using (\ref{mathsf-v-t-s-SDE}) and Lemma \ref{Ku-lem}, we deduce that
\begin{align*}
\E\big|(\theta, \varpi)({\upsilon}^{s'}_{t'})-(\theta, \varpi)({\upsilon}_{t'}^s)\big|^\flat\leq C(\flat, s_0, T)\big(|s'-s|^{\flat}+\int_{t_1}^{t'}\E\big|(\theta, \varpi)({\upsilon}^{s'}_{\tau})-(\theta, \varpi)({\upsilon}_{\tau}^s)\big|^{\flat}\ d\tau\big), 
\end{align*}
which, by Gronwall's lemma, implies 
\[
\E\big|(\theta, \varpi)({\upsilon}^{s'}_{t'})-(\theta, \varpi)({\upsilon}_{t'}^s)\big|^\flat\leq C(\flat, s_0, T)|s'-s|^{\flat}. 
\]
Similarly, it is true that 
\begin{align*}
&\E\big|(\theta, \varpi)({\upsilon}^{s}_{t'})-(\theta, \varpi)({\upsilon}_{t}^s)\big|^\flat\\
&\leq C(\flat, s_0, T)(\big\||(\theta, \varpi)\upsilon^s|^{\frac{1}{2}\flat}\big\|+\big\||(\theta, \varpi){\bf\mathsf v}|^{\frac{1}{2}\flat}\big\|+\big\||(\theta, \varpi){\bf\mathsf v}|^{\frac{1}{2}\flat}\big\|\cdot \big\||({\mho}^s)'_s|^{\frac{1}{2}\flat}\big\|)\big(|t'-t|^{\frac{1}{2}\flat}\!+|t'-t|^{\flat}\big)\\
&\leq C(\flat, s_0, T)|t'-t|^{\frac{1}{2}\flat},
\end{align*}
where,  to obtain the last inequality, we first show $\||({\mho}^s)'_s|^{\frac{1}{2}\flat}\|<C(\flat, s_0, T)$ by (\ref{Y-t-s-diff-ITO}) and then argue as for (\ref{bfmathsf-v-bd}) to show  $\||{\bf\mathsf v}|^{\frac{1}{2}\flat}\|,$ $\big\||\upsilon^s|^{\frac{1}{2}\flat}\big\|$ is also bounded by some $C(\flat, s_0, T)$. Thus, \begin{align*}
2^{1-\flat}\E\big|(\theta, \varpi)({\upsilon}^{s'}_{t'})-(\theta, \varpi)({\upsilon}_{t}^s)\big|^\flat\leq&\E\big|(\theta, \varpi)({\upsilon}^{s'}_{t'})-(\theta, \varpi)({\upsilon}_{t'}^s)\big|+\E\big|(\theta, \varpi)({\upsilon}^{s'}_{t'})-(\theta, \varpi)({\upsilon}_{t'}^s)\big|^\flat\\
\leq&  C(\flat, s_0, T)\left(|s'-s|^{\flat}+|t'-t|^{\frac{1}{2}\flat}\right). 
\end{align*}
By Lemma \ref{Kol-criterion}, there is a continuous modification of $(\theta, \varpi)(\upsilon_t^s)$ in the $(t, s)$ parameter. Note that $(\theta, \varpi)_{\mho_t^s}$ varies continuously with respect to $\mho_{t}^s$, which is also continuous in the $(t, s)$ parameter.  So, we can also obtain a $(t, s)$-continuous version of  the process $\upsilon_t^s$. 
\end{proof}

\begin{remark} As we will see  in the proof of Proposition \ref{absolute-Q-0-s}, for almost all ${\rm w}$, 
\[
\alpha^{s_1}\circ \alpha^{s_2}({\rm w})=\alpha^{s_1+s_2}({\rm w}), \ \mbox{for all}\ s_1, s_2\in \Bbb R. \]
Hence, intuitively, $\{\alpha^s\}_{s\in \Bbb R}$ introduced a one parameter family  of `flow' maps on Brownian paths starting from the $o\in \Bbb R^m$. Consequently, $\{{\bf F}^s\}_{s\in \Bbb R}$ also behaves like a one parameter family  of `flow' maps  which satisfy the cocycle property ${\bf F}^{s_1}\circ {\bf F}^{s_2}={\bf F}^{s_1+s_2}$ for any $s_1, s_2\in \Bbb R$. 
\end{remark}

\subsection{Quasi-invariance property of ${\bf F}_{y}^s$}\label{QIPF-y-s} Let  ${\rm y}^s={\bf F}_y^s {\rm y}$ be as in Theorem \ref{Main-alpha-x-v-Q}. We continue to study its distribution using the classical  Cameron-Martin-Girsanov formula.

Let $({\rm y}_t, {\mho}_t)$ be the stochastic process pair which defines the Brownian motion on $(\M, \wt{g})$ starting from ${\rm y}$ up to time $T$,   i.e., ${\rm y}_t=\pi({\mho}_t)$ and ${\mho}_t\in \mathcal{O}^{\wt{g}}(\M)$ solves the Stratonovich SDE
\begin{equation*}
d{\mho}_t=\sum_{i=1}^{m}H({\mho}_t, e_i)\circ dB_t^i({\rm w}), \ \forall t\in [0, T]. 
\end{equation*}
By an abuse of notation, we continue to use $\P_{y}$  to denote the Brownian distribution  in  $C_{y}([0, T], \M)$ (i.e., the distribution of $({\rm y}_t)_{t\in [0, T]}$)  and use ${\rm Q}$ to denote the distribution of $(B_t)_{t\in [0, T]}$ in $C_{o}([0, T], \Bbb R^m)$. Using  the It\^{o} map, we have the relation 
\[
B=(\mathcal{I}_{\mho_0})^{-1}({\rm y}) \ {\rm and}\  \P_y={\rm Q}\circ ({\mathcal{I}_{\mho_0}})^{-1}. 
\]
Similarly, let $\P_{F^sy}$ denote the  Brownian motion distribution on $C_{F^sy}([0, T], \M)$. Then 
\[
\P_{F^sy}={\rm Q}\circ ({\mathcal{I}_{\mho_0^s}})^{-1}. 
\]
Let ${\rm y}^s$ and $\alpha^s$ be the one parameter family  of stochastic processes on $\M$ and in $\Bbb R^m$, respectively,  that we obtained in Theorem \ref{Main-alpha-x-v-Q}. They are related by the identity 
\[
\alpha^s=({\mathcal{I}_{\mho_0^s}})^{-1}({\rm y}^s). 
\]
Let $\P^s, {\rm Q}^s$ be the distributions of ${\rm y}^s, \alpha^s$, respectively, where ${\rm Q}^0\equiv {\rm Q}$. Then 
\begin{align}\label{Iden-P-s-Q-s}
\P^s={\rm Q}^s\circ ({\mathcal{I}_{\mho_0^s}})^{-1}. 
\end{align}
To compare $\P^s$ with $\P_{F^s y}$, it suffices to compare ${\rm Q}^s$ with ${\rm Q}^0$, which can be understood by a simple application of the  Cameron-Martin-Girsanov formula. 

\begin{prop}\label{absolute-Q-0-s}  The distribution  ${\rm Q}^s$ is equivalent to ${\rm Q}^0$  with 
\begin{eqnarray}\label{Q-0-s-s}
\frac{d{\rm Q}^s}{d{\rm  Q^0}}({\rm w})=e^{\left\{\frac{1}{2}\int_{0}^{T}\langle {\mathtt g}^s_t(\alpha^{-s}({\rm w})),  \ dB_t({\rm w})\rangle-\frac{1}{4}\int_{0}^{T}|{\mathtt g}^s_t(\alpha^{-s}({\rm w}))|^2\ dt\right\}}. \end{eqnarray}
Consequently, the distribution $\P^s$ is equivalent to the Brownian distribution $\P_{F^s y}$ with 
\begin{equation}\label{Q-0-s-1}
\frac{d\P^s}{d\P_{F^s y}}(\beta)=\frac{d{\rm Q}^s}{d{\rm  Q}^0}\left(({\mathcal{I}_{\mho_0^s}})^{-1}(\beta)\right), \ \beta\in C_{F^sy}([0, T], \M). 
\end{equation}
\end{prop}
\begin{proof} We follow the proof of  \cite[Theorem 3.5]{Hs}.  Clearly, (\ref{Q-0-s-1}) follows from (\ref{Q-0-s-s}) by using  the identity (\ref{Iden-P-s-Q-s}).  For (\ref{Q-0-s-s}), recall that  ${\rm Q}^s$ is the distribution of $(\alpha_t^s)_{t\in [0, T]}$, where 
 \[{\alpha}^s_t({\rm w})=\int_{0}^{t}O^s_{\tau}({\rm w})\ dB_{\tau}({\rm w})+\int_{0}^{t}{\mathtt g}^s_{\tau}({\rm w})\ d\tau.\]
The process $\int_{0}^{t}O^s_{\tau}\ dB_{\tau}$ has the same Brownian distribution as $B_t$  since $O^s$ are orthogonal frames and  the distribution of a Euclidean Brownian motion is invariant under orthogonal transfers.  So $\alpha^s$ only differs from a Brownian motion by a drift term $\int_{0}^{t}{\mathtt g}^s_{\tau}\ d\tau$.  Let \[
{\rm M}^s_t({\rm w}):=e^{\left\{-\frac{1}{2}\int_{0}^{t}\left\langle  {\mathtt g}^s_{\tau}({\rm w}),  O^{s}_{\tau}({\rm w}) dB_{\tau}({\rm w})\right\rangle-\frac{1}{4}\int_{0}^{t}\left|{\mathtt g}^s_{\tau}({\rm w})\right|^2\ d\tau\right\} }
 \]
and consider a new distribution $\wt{\rm Q}^s$ on $C_0([0, T], \Bbb R^m)$ which is given  by 
\[
\frac{d{\wt{{\rm Q}}^s}}{d{\rm Q}}({\rm w})={{\rm M}}^{s}_T(\rm w).
\]
Since $|\mathtt g^s|$ is bounded from above by a multiple of $s\cdot\sup|{\rm V}|$, the Novikov's condition is satisfied. Hence the classical  Carmeron-Martin-Girsanov Theorem says that  the distribution of $\alpha^s$ under $\wt{{\rm Q}}^s$ is the same as ${\rm Q}$, i.e.,  for any measurable subset $A$ of $C_0([0, T], \Bbb R^m)$,  
\[
{\rm Q}\left(\{{\rm w}\in A\}\right)=\wt{{\rm Q}}^s\left(\{\alpha^s({\rm w})\in A\}\right),
\]
which, by a change of  variable, gives 
\[
{\rm Q}\left(\{{\rm w}\in A\}\right)={\rm Q}^s\left(\{{\rm M}^s(\alpha^{-s}({\rm w})):\ {\rm w}\in A\}\right). 
\]
Since $A$ is arbitrary, this means ${\rm Q}$ and ${\rm Q}^s$ are equivalent and 
\begin{equation}\label{equi-Q-Q-s-1}
\frac{d{\rm Q}^s}{d{\rm Q}}({\rm w})=\frac{1}{{{\rm M}}^s_T(\alpha^{-s}({\rm w}))}. 
\end{equation}
Note that the process ${\rm M}^s_t$ satisfies the equation 
\[
d{\rm M}^s_t=-\frac{1}{2}{\rm M}^s_t\left\langle  {\mathtt g}^s_t({\rm w}),  O^{s}_{t}({\rm w}) dB_t({\rm w})\right\rangle. 
\]
So, by It\^{o}'s formula, 
\begin{equation}\label{lnM-formula}
-d\ln {\rm M}^s_t(\alpha^{-s}({\rm w}))=\frac{1}{2}\langle  {\mathtt g}^s_t(\alpha^{-s}({\rm w})),  O^{s}_{t}(\alpha^{-s}({\rm w}))\  d\alpha^{-s}_t({\rm w})\rangle+\frac{1}{4}| {\mathtt g}^s_t(\alpha^{-s}({\rm w}))|^2\  dt,
\end{equation}
where the second term of the right hand side of (\ref{lnM-formula}) has coefficient $1/4$  since $\alpha^{-s}_t$ has variance $2t$. On the other hand,  we have 
\begin{equation}\label{alpha-s--s}\alpha^s\circ \alpha^{-s}({\rm w})={\rm w}=B({\rm w}), \ \mbox{for almost all}\ {\rm w}. 
\end{equation}
(Because of (\ref{equi-Q-Q-s-1}), the composition $\alpha^{s_1}\circ \alpha^{s_2}$, $s_1, s_2\in \Bbb R$, is well-defined and has a continuous version in the parameter $(s_1, s_2)$ using  Kolmogorov's criterion  as in Theorem \ref{Main-alpha-x-v-Q}. So, by the uniqueness of the $\alpha^s$ family and its continuous in $s$, we must have $\alpha^{s_1}\circ \alpha^{s_2}=\alpha^{s_1+s_2}$. In particular, (\ref{alpha-s--s}) holds true.) Now, from (\ref{alpha-s--s}), we deduce 
\[
O^{s}_{t}(\alpha^{-s}({\rm w}))\  d\alpha^{-s}({\rm w})+ {\mathtt g}^{s}_{t}(\alpha^{-s}({\rm w}))\ dt=dB_t({\rm w}). 
\]
So, (\ref{lnM-formula}) is also of the  form 
\[
-d\ln {\rm M}^s_t(\alpha^{-s}({\rm w}))=\frac{1}{2}\langle  {\mathtt g}^s_t(\alpha^{-s}({\rm w})),  dB_t({\rm w})\rangle-\frac{1}{4}| {\mathtt g}^s_t(\alpha^{-s}({\rm w}))|^2\  dt
\]
and hence 
\[
\frac{1}{{{\rm M}}^s_T(\alpha^{-s}({\rm w}))}=e^{\left\{\frac{1}{2}\int_{0}^{T}\langle {\mathtt g}^s_t(\alpha^{-s}({\rm w})),  \ dB_t({\rm w})\rangle-\frac{1}{4}\int_{0}^{T}|{\mathtt g}^s_t(\alpha^{-s}({\rm w}))|^2\ dt\right\}}. 
\]
 \end{proof}

\begin{prop}\label{Quasi-IP-1} The probability  $\P_{F^s y}\circ {\bf F}^s_y$ is absolutely continuous with respect to $\P_y$ and the Radon-Nikodyn derivative $d\P_{F^s y}\circ {\bf F}^s_y/d\P_y$ conditioned on ${\rm y}_T=x$ is $L^q$ integrable for every  $q\geq 1$, locally uniformly in the $s$ parameter.  Moreover,  $d\P_{F^s y}\circ {\bf F}^s_y/d\P_y$ is  differentiable in  $s$ with  differential $\overline{\mathcal{E}^s_T} ({d\P_{F^s y}\circ {\bf F}^s_y}/{d\P_y})$, where $\overline{\mathcal{E}^s_T}$ conditioned on ${\rm y}_T=x$ is also $L^q$ integrable for every $q\geq 1$,  locally uniformly in the  $s$ parameter. 
\end{prop}

\begin{proof}For $\P_{y}$ almost all path $\beta$, let ${\rm w}=\mathcal{I}_{\mho_0}^{-1}(\beta)$. Then $
\mathcal{I}_{\mho_0^s}^{-1}\left({\bf F}^s_y(\beta)\right)=\alpha^s({\rm w})$. 
As a corollary of Proposition \ref{absolute-Q-0-s}, we have $\P_{F^sy}\circ {\bf F}^s_y$ is  equivalent to $\P_{y}$   with 
\begin{equation*}
\frac{d\P_{F^sy}\circ {\bf F}^s_y}{d\P_{y}}(\beta)=\frac{d\P_{F^sy}}{d\P_{y}\circ ({\bf F}^{s}_y)^{-1}}\circ {\bf F}^s_y(\beta)= \frac{d\P_{F^sy}}{d\P^s}\circ {\bf F}^s_y(\beta)=\frac{d{\rm Q}^0}{d{\rm  Q}^s}\big(\alpha^s({\rm w})\big).  \ 
\end{equation*}
Note that $d\alpha^s({\rm w})=O^s({\rm w})dB({\rm w})+g^s({\rm w})dt$.  So, by (\ref{Q-0-s-s}) and (\ref{iteration-g}), we have 
\begin{align*}
\frac{d\P_{F^sy}\circ {\bf F}^s_y}{d\P_{y}}(\beta)=e^{\left\{-\frac{1}{2}\int_{0}^{T}\langle \ov{{\mathtt g}}^s_\tau({\rm w}),  \  dB_\tau({\rm w})\rangle+\frac{1}{4}\int_{0}^{T}|\ov{{\mathtt g}}^s_\tau({\rm w})|^2\ d\tau\right\}},
\end{align*}
where 
\begin{align*}
&\ov{{\mathtt g}}^s_t({\rm w}):  
=\int_{0}^{s}[O_{\tau}^{\jmath}]^{-1}\left\{({\mho}_{0}^{\jmath})^{-1}\left[\mathtt{s}'(\tau){\rm V}(F^{\jmath}y)\right]-{\rm Ric}\left({\mho}_{\tau}^{\jmath}({\mho}_0^{\jmath})^{-1}[\mathtt{s}(\tau){\rm V}(F^{\jmath} y)]\right)\right\} \ d\jmath. 
\end{align*}
Put 
\[
\mathcal{E}^{s}_t({\rm w}):=e^{\left\{-\frac{1}{2}\int_{0}^{t}\langle \ov{{\mathtt g}}^s_\tau({\rm w}),  \  dB_\tau({\rm w})\rangle+\frac{1}{4}\int_{0}^{t}|\ov{{\mathtt g}}^s_\tau({\rm w})|^2\ d\tau\right\}}, \ \forall t\in [0, T]. 
\]
For $q\geq 1$, we estimate $\E_{\P^*_{y, x, T}}|\mathcal{E}^{s}_T({\rm w})|^q$.
Let $b_t$  be the Brownian motion with respect to the bridge distribution (from $y$ to $x$ in time $T$) as  in Lemma \ref{Hsu-thm-5.4.4} such that 
\[dB_{\tau}({\rm w})=db_{\tau}({\rm w})+2{\mho}_{\tau}^{-1}\nabla \ln p(T-\tau, {\rm y}_{\tau}, x)\ d\tau.
\]
Then conditioned on ${\rm y}_T=x$, $|\mathcal{E}^{s}_T({\rm w})|^q$ has the same distribution as 
\[
e^{\{-\frac{1}{2}q\int_{0}^{T}\langle \ov{{\mathtt g}}^s_\tau({\rm w}),  \ db_\tau({\rm w})\rangle+\frac{1}{4}q\int_{0}^{T}|\ov{{\mathtt g}}^s_\tau({\rm w})|^2\ d\tau-q\int_{0}^{T}\langle \ov{{\mathtt g}}^s_\tau({\rm w}),  \mho_{\tau}^{-1}\nabla \ln p(T-\tau, {\rm y}_{\tau}, x)\rangle\ d\tau\}}.
\]
So, by  H\"{o}lder's inequality and the Cameron-Martin-Girsanov Theorem, 
\begin{align}
\notag\E_{\P^*_{y, x, T}}|\mathcal{E}^{s}_T({\rm w})|^q
&\leq\!\!\sqrt{p(T, x, y)} \left[\E_{\P_{y, x, T}}e^{\{-\int_{0}^{T}q\langle \ov{{\mathtt g}}^s_\tau({\rm w}),  \ db_\tau({\rm w})\rangle-q^2\int_{0}^{T}|\ov{{\mathtt g}}^s_\tau({\rm w})|^2\ d\tau\}}\right]^{\frac{1}{2}}\\
\notag&\ \ \ \cdot\left[\E_{\P^*_{y, x, T}}e^{\{-2q\int_{0}^{T}\langle \ov{{\mathtt g}}^s_\tau,  \mho_{\tau}^{-1}\nabla \ln p(T-\tau, {\rm y}_{\tau}, x)\rangle\ d\tau+(\frac{1}{2}q+q^2)\int_{0}^{T}|\ov{{\mathtt g}}^s_\tau|^2\ d\tau\}}\right]^{\frac{1}{2}}\\
& \leq\!\!\sqrt{p(T, x, y)}\!\left[\E_{\P^*_{y, x, T}}\! e^{\{-2q\int_{0}^{T}\langle \ov{{\mathtt g}}^s_\tau,  \mho_{\tau}^{-1}\nabla \ln p(T-\tau, {\rm y}_{\tau}, x)\rangle\ d\tau+(\frac{1}{2}q+q^2)\int_{0}^{T}|\ov{{\mathtt g}}^s_\tau|^2\ d\tau\}}\right]^{\frac{1}{2}}. \label{E-s-T-pre-est}
\end{align}
Let us  continue to use $C$ to denote  a constant depending on $\|g\|_{C^3}$, $m$ and the norm bound of ${\rm V}$ and  use $C(\cdot)$ to indicate the extra coefficients it depends on.  By our choice of $\mathtt s$ (see (\ref{mathtts-cond})), for $s\in [-s_0, s_0]$, $|\ov{{\mathtt g}}^s_\tau({\rm w})|\leq C(s_0, T)$ for some $C(s_0, T)$.  Apply  this in (\ref{E-s-T-pre-est}) and then  use (\ref{p-t-rough}) and  (\ref{equ-Hs2-ST}).  We obtain some  $C(q, s_0, T), \wt{C}(q, s_0, T)$ such that 
\[
 \E_{\P^*_{y, x, T}}|\mathcal{E}^{s}_T({\rm w})|^q\leq C(q, s_0, T)\left[\E_{\P^*_{y, x, T}} e^{\{\wt{C}(q, s_0, T)\int_{0}^{T}\|\nabla \ln p(T-\tau, {\rm y}_{\tau}, x)\|\ d\tau}\right]^{\frac{1}{2}},\]
which, by  (\ref{exp-grad-lnp}), shows  that  $d\P_{F^s y}\circ {\bf F}^s/d\P_y$ conditioned on ${\rm y}_T=x$ is  $L^q$ integrable for $q\geq 1$, locally uniformly  in the $s$ parameter. 
 
 Note that $\mathcal{E}^{s}_t({\rm w})$ satisfies the SDE 
  \[
d\mathcal{E}^{s}_t({\rm w})=\mathcal{E}^{s}_t({\rm w})\left(-\frac{1}{2}\langle \ov{{\mathtt g}}^s_t({\rm w}),  \ dB_t({\rm w})\rangle+\frac{1}{2}|\ov{{\mathtt g}}^s_t({\rm w})|^2\ dt\right). 
\]
The  differential process  $(\mathcal{E}^{s}_t({\rm w}))'_s=(d\mathcal{E}^{\jmath}_t({\rm w})/d\jmath)|_{\jmath=s}$ exists and satisfies the It\^{o}  SDE
\begin{align*}
d(\mathcal{E}^{s}_t)'_s({\rm w})=& (\mathcal{E}^{s}_t)'_s({\rm w})\left(-\frac{1}{2}\langle \ov{{\mathtt g}}^s_t({\rm w}),  \ dB_t({\rm w})\rangle+\frac{1}{2}|\ov{{\mathtt g}}^s_t({\rm w})|^2\ dt\right)\\
&\  + \mathcal{E}^{s}_t({\rm w})\left(-\frac{1}{2}\langle (\ov{{\mathtt g}}^s_t)'_s({\rm w}),  dB_\tau({\rm w})\rangle+ \langle (\ov{{\mathtt g}}^s_t)'_s({\rm w}), \ov{{\mathtt g}}^s_t({\rm w})\rangle\ d\tau \right).
\end{align*}
Hence the  Radon-Nikodyn derivative $d\P_{F^s y}\circ {\bf F}^s_y/d\P_y$ is  differentiable in  $s$ with differential $(\mathcal{E}^{s}_T)'_s({\rm w})$, which, by using  stochastic Duhamel principle (or It\^{o}'s formula), is 
\begin{align*}
(\mathcal{E}^{s}_T)'_s &=\ \mathcal{E}^{s}_T\cdot \left((\mathcal{E}^{s}_0)'_s-\frac{1}{2}\int_{0}^{T}\langle (\ov{{\mathtt g}}^s_t)'_s({\rm w}),  dB_\tau({\rm w})\rangle+\frac{1}{2}\int_0^T \langle (\ov{{\mathtt g}}^s_t)'_s({\rm w}), \ov{{\mathtt g}}^s_t({\rm w})\rangle\ d\tau \right)\\
&=\ \mathcal{E}^{s}_T\cdot \left(-\frac{1}{2}\int_{0}^{T}\langle (\ov{{\mathtt g}}^s_t)'_s({\rm w}),  dB_t({\rm w})\rangle+\frac{1}{2}\int_0^T \langle (\ov{{\mathtt g}}^s_t)'_s({\rm w}), \ov{{\mathtt g}}^s_t({\rm w})\rangle\ dt \right)\\
&=:\  \mathcal{E}^{s}_T\cdot \overline{\mathcal{E}}^{s}_T.
\end{align*}
Conditioned on ${\rm y}_T=x$, $\overline{\mathcal{E}}^{s}_T$ has the same distribution as 
\[-\frac{1}{2}\int_{0}^{T}\langle (\ov{{\mathtt g}}^s_\tau)'_s,  db_\tau\rangle+\frac{1}{2}\int_0^T \langle (\ov{{\mathtt g}}^s_\tau)'_s, \ov{{\mathtt g}}^s_\tau\rangle\ d\tau-\int_{0}^{T}\big\langle (\ov{{\mathtt g}}^s_\tau)'_s, \mho_{\tau}^{-1}\nabla \ln p(T-\tau, {\rm y}_{\tau}, x) \big\rangle\ d\tau,\]
where both $|\ov{{\mathtt g}}^s_t|$ and  $|(\ov{{\mathtt g}}^s_t)'_s|$ are bounded by some constant  $C(s_0, T)$. Hence, by H\"{o}lder's inequality and   (\ref{abcq}), we compute that 
\begin{align*}
\left(\E_{\P_{y, x, T}}|\overline{\mathcal{E}}^{s}_T|^{q}\right)^2
&\leq 3^{2q-1}\left(\E_{\P_{y, x, T}}\left|\int_{0}^{T}\langle (\ov{{\mathtt g}}^s_\tau)'_s,  db_\tau\rangle\right|^{2q}+ (C(s_0, T))^{2q}\right.\\
&\ \ \ \ \ \ \ \ \ \ \ \ \ \left.+\E_{\P_{y, x, T}}\left|\int_{0}^{T}\big\langle (\ov{{\mathtt g}}^s_\tau)'_s, \mho_{\tau}^{-1}\nabla \ln p(T-\tau, {\rm y}_{\tau}, x) \big\rangle\ d\tau\right|^{2q}\right)\\
&=:  3^{2q-1}\left(\ov{{\rm{(I)}}}+\ov{\rm{(II)}}+\ov{\rm{(III)}}\right). 
\end{align*}
Since $b$ is a Brownian motion with respect to $\P_{y, x, T}$,  by Lemma \ref{Ku-lem}, 
\[
\ov{{\rm{(I)}}}\leq {\mathtt{C}}_1(2q) \int_{0}^{T}\E_{\P_{y, x, T}}|(\ov{{\mathtt g}}^s_\tau)'_s|^{2q}\ d\tau\leq C(q, s_0, T),\]
where ${\mathtt{C}}_1$ is from (\ref{equ-Ku-lem}). 
Using  $|(\ov{{\mathtt g}}^s_\tau)'_s{\rm w})|\leq C(s_0, T)$ and (\ref{exp-grad-lnp}), we obtain 
\begin{align*}
\ov{\rm{(III)}} \leq & C(s_0, T)^{2q} \E_{\P_{y, x, T}}\big(\int_{0}^{T}\|\nabla \ln p(T-\tau, {\rm y}_{\tau}, x)\|\ d\tau\big)^{2q}\\
\leq & C(s_0, T)^{2q} \E_{\P_{y, x, T}} e^{\left\{2q\int_{0}^{T}\|\nabla \ln p(T-\tau, {\rm y}_{\tau}, x)\|\ d\tau\right\}}\\
\leq & C(s_0, T)^{2q} (p(T, x, y))^{-1} e^{c'(1+d(x, y))}.
\end{align*}
Putting all the estimations on  $\ov{{\rm{(I)}}}$, $\ov{\rm{(II)}}$ and  $\ov{\rm{(III)}}$ together,  we conclude that  $\overline{\mathcal{E}^s_T}$ conditioned on ${\rm y}_T=x$ is $L^q$ integrable for $q\geq 1$, locally uniformly in the  $s$ parameter.
\end{proof}

 Consider the distribution of $\overline{\P}_x$ on $C_x([0, T], \M)$. 
Let $({\rm x},  {\rm u})$ be the stochastic pair which defines the Brownian motion on $(\M, \wt{g})$ which starts from $x$. The distribution of $({\rm x}_t)_{t\in [0, T]}$ is independent of the choice of ${\rm u}_0$.  Hence ${\P}_{(x, {\rm u}_0)}$, which is the distribution of $({\rm x}_t)_{t\in [0, T]}$ with a  initial  frame ${\rm u}_0$, coincides with $\P_x$ on $C_x([0, T], \M)$ and  ${\P}_{(x, {\rm u}_0), y, T}:=\E_{{\P}_{(x, {\rm u}_0)}}\left(\cdot\big| {\rm x}_T=y \right)$ coincides with $\P_{x, y, T}$ on $C_{x, y}([0, T], \M)$.  This means 
\begin{align*}
\overline{\P}_x&=\int\int {\P}_{(x, {\rm u}_0), y, T}\cdot p(T, x, y)\ d{\rm Vol}(y)\ d{\rm Vol}({{\rm u}_0})\left(\!=\int \int \P_{x, y, T}\cdot p(T, x, y)\ d{\rm Vol}({{\rm u}_0})\ d{\rm Vol}(y)\right)\\
&=\int\int {\P}_{(x, {\rm u}_0), y, T}\cdot p(T, x, y)\ d{\rm Vol}({{\rm u}_0})\ d{\rm Vol}(y), 
\end{align*}
where $d{\rm Vol}({{\rm u}_0})$ is the uniform distribution  on $\mathcal{O}^{\wt{g}}_x(\M)$.  For any $y\in \M$, the Brownian bridge process connecting $x$ and $y$ in time $T$ has the following   symmetric property. 

\begin{lem}\label{Brown-bridge-symmetry}Let $(X_t, U_t)_{t\in [0, T]}$ be the pair of stochastic processes for Brownian bridge from $x$ to $y$ in time $T$.  
\begin{itemize}
\item[i)] Under $\P_{x, y, T}$, the process $(X_{T-t})_{t\in [0, T]}$ has the law $\P_{y, x, T}$.
\item[ii)]  If $U_0$ is chosen randomly with the uniform distribution in $\mathcal{O}^{\wt{g}}_x(\M)$, then $U_{T}$ is also uniformly distributed in $\mathcal{O}^{\wt{g}}_y(\M)$. 
\end{itemize}
\end{lem}
\begin{proof}i) is \cite[Proposition 5.4.3]{Hs}. (It is true since by (\ref{bridge-margin}), the finite margin of $X$, or the joint density function of $X_{t_1}, \cdots, X_{t_n}$, $0=t_0<t_1<\cdots<t_n<t_{n+1}=T$, is given by 
\[
\frac{1}{p(T, x, y)}\prod_{i=0}^{n}p(t_{i+1}-t_i, x_i, x_{i+1}), \ \mbox{where} \ x_0=x, x_{n+1}=y, 
\]
which is the same as the joint density function of $\wt{X}_{T-t_n}, \cdots, \wt{X}_{T-t_1}$ of  the bridge $\wt{X}$ from $y$ to $x$ in time $T$.) For ii), we consider (\ref{B-bridge}). 
Note  that  the distribution of the  $\Bbb R^m$-Brownian motion $b_t$ is invariant under rotations. So if  $(U_t)_{t\in [0, T]}$ solves (\ref{B-bridge}) with initial frame $U_0$, then for $\wt{U}_0=U_0\upsilon$ with $\upsilon\in\mathcal{O}(\Bbb R^m)$,  $(\wt{U}_t=U_t\upsilon)_{t\in [0, T]}$ solves (\ref{B-bridge}). This implies ii). 
\end{proof} 

 Let ${\bf F}^s_y$ be as in Theorem \ref{Main-alpha-x-v-Q}. It induces a map from $C_{y, x}([0, T], \M)$ to $C_{F^sy, x}([0, T], \M)$. We define ${\bf F}^s$ on $C_{x}([0, T], \M)$ conditioned on the value of $\beta_T$ by letting 
\[
{\bf F}^s(\beta):={\bf F}^{s}_{\beta_T}(\beta). 
\]
By  Lemma \ref{Brown-bridge-symmetry}, 
a uniform random choice of ${\rm u}_0$ at $x$ will result in a  uniform distribution of ${\rm u}_T$ at $y$  for the Brownian bridge connecting $x$ and $y$ in time $T$. Therefore, to analyze  $\overline{\P}_x\circ {\bf F}^s$, we can choose the initial  $\mho_0\in \mathcal{O}^{\wt{g}}_{\beta_T}(\M)$ with a uniform distribution to define ${\bf F}^s_{\beta_T}$. 

\begin{lem}\label{diff-P-F-S}For $\overline{\P}_x$ almost all $\beta\in C_x([0, T], \M)$, 
\begin{equation}\label{equ-P-F-S}
\frac{d\overline{\P}_x\circ{\bf F}^s}{d\overline{\P}_x}(\beta)=\frac{d\P_{F^s \beta_T}\circ {\bf F}^s}{d\P_{\beta_T}}(\beta)\cdot \frac{d{\rm Vol}(F^s \beta_T)}{d{\rm Vol}(\beta_T)}. 
\end{equation}
\end{lem}
\begin{proof}
Lemma \ref{Brown-bridge-symmetry} implies that the distribution of ${\rm u}_T$ is uniform if ${\rm u}_0$ is. 
So if we disintegrate $\overline{\P}_x$ according to the value of $({\rm x}_T, {\rm u}_T)$, we obtain 
\begin{equation*}
\overline{\P}_x=\int\int {\P}_{(x, {\rm u}_0), y, T}\ \cdot p(T, x, y)\ \! d{\rm Vol}({\rm u}_0)\ \! d{\rm Vol}(y)=\int\int \overline{\P}_{x, (y, {\mho_0}),  T}\cdot p(T, x, y)\ \! d{\rm Vol}({\mho_0})\ \! d{\rm Vol}(y),
\end{equation*}
where $d{\rm Vol}({\mho_0})$ is the uniform probability on $\mathcal{O}^{\wt{g}}_y(\M)$.   For  any measurable subset $A\subset C_x([0, T], \M)$, by the change of variable formula, 
\begin{eqnarray*}
\overline{\P}_x\left({\bf F}^s(A)\right)\!&=&\!\int\int \overline{\P}_{x, (F^s y, F^s {\mho_0}), T} \left({\bf F}^s(A)\right)\cdot p(T, x, F^sy)\ d{\rm Vol}(F^s{\mho_0})\  d{\rm Vol}(F^s y)\notag\\
&=&\!\int\int \overline{\P}_{ x, (F^s y, F^s {\mho_0}), T} \left({\bf F}^s(A)\right)\cdot p(T, x, F^sy)\frac{d{\rm Vol}\circ F^s}{d{\rm Vol}}(y)\ d{\rm Vol}({\mho_0})\  d{\rm Vol}(y).
\end{eqnarray*}
By Lemma \ref{Brown-bridge-symmetry}, the distribution of  $\overline{\P}_{x, (F^sy, F^s {\mho_0}),  T}$ on $C_{x, F^s y}([0, T], \M)\equiv C_{F^sy, x}([0, T], \M)$ can be identified with that of ${\P}_{(F^s y, F^s{\mho_0}), x,  T}$,  the Brownian bridge from $F^s y$ to $x$ in time $T$ with the initial frame $F^s{\mho_0}\in \mathcal{O}^{\wt{g}}_{F^sy}(\M)$. Hence
\begin{equation}\label{abs-disintegration}
\overline{\P}_x({\bf F}^s(A))=\int\int {\P}_{(F^s y, F^s {\mho_0}),x,  T}({\bf F}^s(A)) \cdot p(T, F^sy, x)\frac{d{\rm Vol}\circ F^s}{d{\rm Vol}}(y)\ d{\rm Vol}({\mho_0})\  d{\rm Vol}(y).
\end{equation}
The absolute continuity of $\overline{\P}_x\circ{\bf F}^s$ with respect to $\overline{\P}_x$ will follow if  ${\P}_{(F^s y, F^s {\mho_0}),x,  T}\circ {\bf F}^s$ is absolutely continuous with respect to ${\P}_{(y,{\mho_0}),x,  T}$ and the Radon-Nikodym derivative  $d{\P}_{(F^s y, F^s {\mho_0}),x,  T}\circ {\bf F}^s/d{\P}_{(y,{\mho_0}),x,  T}$ is integrable.  Since the bridge process from $y$ to $x$ in time $T$ is  just  the conditional process of ${\rm y}$ on ${\rm y}_T=x$, 
Lemma \ref{absolute-Q-0-s}  implies that ${\P}_{(F^s y, F^s {\mho_0}),x,  T}\circ {\bf F}^s$ is absolutely continuous with respect to ${\P}_{(y,{\mho_0}),x,  T}$. 

 As to (\ref{equ-P-F-S}), we see that for any  measurable set $\wt{A}\subset C_y([0, T], \M)$,  \begin{align}
\P_{F^sy}\circ {\bf F}^s (\wt{A})=\P_{y}\left(\chi_{\wt{A}}\cdot \frac{d\P_{F^s y}\circ {\bf F}^s}{d\P_{y}}\right)= \int \P_{y, z, T}\left(\chi_{\wt{A}_z}\cdot \frac{d\P_{F^s y}\circ {\bf F}^s}{d\P_{y}} \right)p(T, y, z)\ d{\rm Vol}(z),\label{pf-prop520-1}
\end{align}
where $\wt{A}_z$ is the collection of elements $w\in \wt{A}$ with $w_T=z$.  On the other hand, 
 \begin{eqnarray}
\notag\P_{F^sy}\circ {\bf F}^s (\wt{A})&=& \int \P_{F^sy, z, T}\circ {\bf F}^s(\chi_{\wt{A}_z})\cdot p(F^s y, z, T)\ d{\rm Vol}(z)\\
&=&  \int \P_{y, z, T}\left(\chi_{\wt{A}_z}\cdot \frac{d\P_{F^sy, z, T}\circ {\bf F}^s}{d\P_{y, z, T}}\right) p(T, F^s y, z)\ d{\rm Vol}(z).\label{pf-prop520-2}
\end{eqnarray}
Since $\wt{A}$ is arbitrary, we conclude from (\ref{pf-prop520-1}) and (\ref{pf-prop520-2}) that 
\[
\frac{d\P_{F^sy, z, T}\circ {\bf F}^s}{d\P_{y, z, T}}=\frac{d\P_{F^s y}\circ {\bf F}^s}{d\P_{y}}\cdot\frac{p(T, y, z)}{p(T, F^sy, z)}. 
\]
Reporting  this in (\ref{pf-prop520-2}) and (\ref{abs-disintegration}) shows  (\ref{equ-P-F-S})  for  $\overline{\P}_x$ almost all $\beta$ with ${\beta}_T=y\in \M$. \end{proof}
An immediate corollary of  Proposition \ref{Quasi-IP-1} and Lemma \ref{diff-P-F-S}  is

\begin{prop}\label{Quasi-IP-2} The probability  $\overline{\P}_x \circ {\bf F}^s$ is absolutely continuous with respect to $\overline{\P}_x$ and the Radon-Nikodyn derivative $d\overline{\P}_x \circ {\bf F}^s/d\overline{\P}_x $ conditioned on ${\rm x}_T=y$ is  $L^q$ integrable for every  $q\geq 1$, locally uniformly in the $s$ parameter.  The differential of $d\overline{\P}_x\circ {\bf F}^s/d\overline{\P}_x$ in $s$ exists and  is of the form  $\oa{\mathcal{E}^s_T}\cdot (d\overline{\P}_x\circ {\bf F}^s/d\overline{\P}_x)$,  where $\oa{\mathcal{E}^s_T}$ conditioned on ${\rm x}_T=y$ is  square integrable, locally uniformly  in the $s$ parameter. 
\end{prop}

Using   (\ref{equ-P-F-S}) and the proof of Proposition \ref{Quasi-IP-1}, we can  deduce  that $\oa{\mathcal{E}^s_T}$ differs from $\ov{\mathcal{E}^s_T}$ by the differential of ${d{\rm Vol}(F^s y)}/{d{\rm Vol}(y)}$ in the $s$ parameter, where $\ov{\mathcal{E}^s_T}$ can be understood as a backward stochastic integral on the bridge paths from $x$ to $y$ in time $T$. 

\subsection{The extended map  ${\bf F}^s$ 
}\label{the flow F-S}

 In order to show the  properties iii), iv) of ${\bf F}^s$ in Section  \ref{Obs-Stra}, we need to clarify  $(D\pi(\lfloor {\rm u}_T\rceil^{\l})^{(1)}_\l)\circ {\bf F}^s$ for ${\Phi}_{\l}^1\circ {\bf F}^s$, where ${\Phi}_{\l}^1$ is as in  (\ref{def-phi-1}).  We  will achieve this   by extending ${\bf F}^s$  to the process $(\lfloor {\rm u}_T\rceil^{\l})^{(1)}_\l$ and letting 
\[\left(D\pi(\lfloor {\rm u}_T\rceil^{\l})^{(1)}_\l\right)\circ {\bf F}^s:=D\pi\left((\lfloor {\rm u}_T\rceil^{\l})^{(1)}_\l\circ {\bf F}^s\right).\]
The rough idea is that  the maps ${\bf F}^s$ on orbits extend naturally to their tangent maps for the parallel transportations and hence can be defined for the objects they make.

We first deal with $(\lfloor {\rm u}_T\rceil^{\l})^{(1)}_0\circ {\bf F}^s$.  Let  $\l\mapsto \lfloor {\rm u}_0\rceil^{\l}\in \mathcal{O}^{\wt{g}^{\l}}(\M)$ be $C^{k-2}$ in $\mathcal{F}(\M)$ and let  $(\lfloor {\rm u}_t\rceil^{\l}\in \mathcal{O}^{\wt{g}^{\l}}(\M))_{t\in [0, T]}$ with initials $ \lfloor {\rm u}_0\rceil^{\l}$ be the unique  solution to\begin{equation}\label{FM-BM-SDE-11}
d\lfloor {\rm u}_t\rceil^{\l}=\sum_{i=1}^{m}H^{\l}(\lfloor {\rm u}_t\rceil^{\l}, e_i)\circ dB_t^i({w}), \ \forall t\in [0, T]. 
\end{equation} 
By Lemma \ref{u-lambda-differential-1}, there is a version of $\{\lfloor {\rm u}_t\rceil^{\l}\}$ such that  $\l\mapsto \lfloor {\rm u}_t\rceil^{\l}({w})$ is $C^{k-2}$  in $\lambda$ for almost all ${w}$.  By Lemma \ref{cor-2--u-t-lam-tangent-map-1},  the differential process  $(\lfloor {\rm u}_t\rceil^{\l})^{(1)}_0$  is given by  
\begin{equation}\label{diff-in-lam-1}
(\lfloor {\rm u}_T\rceil^{\l})^{(1)}_0=\left[D\overrightarrow{F}_{0, T}({\rm u}_0, {w})\right](\lfloor {\rm u}_0\rceil^{\l})^{(1)}_{0}+ \int_{0}^{T} \left[D \overrightarrow{F}_{t, T}({\rm u}_t, {w})\right](H^{\l})^{(1)}_{0}({\rm u}_t, e_i)\circ dB_t^i({w}),
\end{equation}
where  ${\rm u}=\lfloor {\rm u}\rceil^{0}$  and $\{D\overrightarrow{F}_{\underline{t}, \overline{t}}\}_{0\leq \underline{t}<\overline{t}\leq T}$ are  the tangent maps of the flow maps $\{\overrightarrow{F}_{\underline{t}, \overline{t}}\}_{0\leq \underline{t}<\overline{t}\leq T}$ associated to  (\ref{FM-BM-SDE-11}) at $\l=0$ (the  arrow  is to indicate the time is recorded starting from  $x$).  By Lemma \ref{El-Ku-Cor} (see also Lemma \ref{Tangent map-SDE}),  the $\{D \overrightarrow{F}_{t, T}\}$ are  determined by the paths $({\rm x}_\tau({w})=\pi({\rm u}_t(w)))_{\tau\in [0, T]}$ (or its anti-development in $\Bbb R^m$). Hence (\ref{diff-in-lam-1}) shows that $(\lfloor {\rm u}_T\rceil^{\l})^{(1)}_0({w})$ are objects completely determined by   $({\rm x}_\tau({w}))_{\tau\in [0, T]}$, $(\lfloor {\rm u}_0\rceil^{\l})^{(1)}_0$  and $(H^{\l})^{(1)}_{0}$.

 By symmetry of the Brownian motion,  we can  describe the distribution of $(\lfloor {\rm u}_T\rceil^{\l})^{(1)}_0$ conditioned on  ${\rm x}_T=y$  using  $({\rm y}_t, \mho_t)_{t\in [0, T]}$, which  is the stochastic pair defining  the  Brownian motion on $(\M, \wt{g})$ starting from $y$.  The two path spaces $C_{y, x}([0, T], \M)$  and $C_{x, y}([0, T], \M)$  can be identified.  Moreover,   the distribution of ${\rm y}$  conditioned on ${\rm y}_T=x$ coincides with ${\rm x}$  conditioned on ${\rm x}_T=y.$  This means for almost all such path $({\rm y}_\tau)_{\tau\in [0, T]}({\rm w})=:\beta$, it is associated with a path $({\rm x}_t)_{t\in [0, T]}({w})=(\beta_{T-\tau})_{\tau\in [0, T]}=:\overrightarrow{\beta}$. So the stochastic parallel transportation of ${\rm u}_t$ along $\overrightarrow{\beta}$ is well-defined and is given by 
$${\rm u}_t=\mho_{T-t}(\mho_{T})^{-1}{\rm u}_0.$$
For any element ${\rm X}\in T_{{\rm u}_t}\mathcal{F}(\M)$, let 
\[
(\theta,  \varpi)_{{\rm u}_t}^{-1}{\rm X}=:({\rm X}^1, {\rm X}^2). 
\]
Note that the orthonormal frames  ${\rm u}_t$ and $\mho_{T-t}$ have the same footpoint ${\rm x}_t(\omega)={\rm y}_{T-t}({\rm w})$. Hence ${\rm X}$ also naturally corresponds to an element ${\rm Y}({\rm X})={\rm Y}$ in $T_{{\mho}_{T-t}}\mathcal{F}(\M)$ with 
\begin{equation*}
(\theta, \varpi)_{\mho_{T-t}}^{-1}{\rm Y}:=\left(\mho_{T-t}^{-1}{\rm u}_t({\rm X}^1),  Ad(\mho_{T-t}^{-1}{\rm u}_t)({\rm X}^2)\right).
\end{equation*}
We see that ${\rm X}$ and ${\rm Y}({\rm X})$ are just the same vector expressed in different frame charts. Denote by ${\rm Y}$ this map which sends tangents ${\rm X}\in T_{{\rm u}_\tau}\mathcal{F}(\M)$ to ${\rm Y}({\rm X})\in T_{{\mho}_{T-\tau}}\mathcal{F}(\M)$ for any $\tau\in [0, T]$.  Let $(F_{t_1, t_2})_{0\leq t_1<t_2\leq T}$ and $(DF_{t_1, t_2})_{0\leq t_1<t_2\leq T}$ be the invertible stochastic flow maps and tangent maps associated to  ${\rm y}$ (cf. (\ref{Horizontal-u})). The following is true. 

\begin{lem}\label{lem-fore-back}Let $\beta$, ${\rm X}, {\rm Y}$ be introduced as above. Then for almost all $\beta$, we have 
\begin{equation}\label{fore-back}
{\rm Y}\big(D\overrightarrow{F}_{t, T}({\rm u}_t, {w}){\rm X}\big)=D(F_{0, T-t}({\mho}_0, {\rm w}))^{-1}({\rm Y}({\rm X}))=\big[DF_{0, T-t}({\mho}_0, {\rm w})\big]^{-1}({\rm Y}({\rm X})).
\end{equation}
\end{lem}
\begin{proof}By   Corollary \ref{El-Ku-Cor-1},  for almost all ${\rm w}$, the maps $F_{0, T-t}(\cdot, {\rm w})$ are $C^{k-2}$ diffeomorphisms. So for almost all ${\rm w}$, the tangent maps $D(F_{0, T-t}({\mho}_0, {\rm w}))^{-1}$ and $[DF_{0, T-t}({\mho}_0, {\rm w})]^{-1}$ exist and are equal. For (\ref{fore-back}), it suffices to verify the first equality. 

Write $(\theta,  \varpi)_{{\rm u}_t}^{-1}{\rm X}=:({\rm X}^1_t, {\rm X}^2_t)$ and let 
\[
({\rm X}^1_\tau, {\rm X}^2_\tau):=(\theta,  \varpi)_{{\rm u}_{\tau}}^{-1}D\overrightarrow{F}_{t, \tau}({\rm u}_t, {w}){\rm X}, \ \forall \tau\in [t, T]. 
\]
It is true by Lemma \ref{Tangent map-SDE}  that 
\begin{align}
d{\rm X}^1_{\tau}&={\rm X}^2_{\tau}\circ dB_\tau(w),\label{T-SDE-1}\\
d{\rm X}_\tau^{2}&=({\rm u}_{\tau})^{-1}R\left({\rm u}_{\tau}\circ dB_{\tau}(w), {\rm u}_{\tau}{\rm X}^1_{\tau}\right) {\rm u}_{\tau}.\label{T-SDE-2}
\end{align}
Let 
\[({\rm Y}^1_{\tau}, {\rm Y}^2_{\tau}):=(\theta,  \varpi)_{{\mho}_{T-\tau}}^{-1}{\rm Y}\left((\theta,  \varpi)_{{\rm u}_\tau}({\rm X}^1_{\tau}, {\rm X}^2_{\tau})\right).\] Note that $(\mho_{T-\tau})^{-1}{\rm u}_\tau\equiv(\mho_{T})^{-1}{\rm u}_0$. So (\ref{T-SDE-1}) gives 
\begin{align*}
d{\rm Y}^1_{\tau}=(\mho_{T-\tau})^{-1}{\rm u}_\tau d{\rm X}^1_{\tau}&=(\mho_{T-\tau})^{-1}{\rm u}_\tau {\rm X}^2_{\tau}\circ dB_\tau(w)\\
&=-(\mho_{T-\tau})^{-1}{\rm u}_\tau {\rm X}^2_{\tau}\left((\mho_{T-\tau})^{-1}{\rm u}_\tau\right)^{-1}\circ (\mho_{T-\tau})^{-1}{\rm u}_\tau d\overrightarrow{B}_{T-\tau}({\rm w})\\
&=-{\rm Y}^2_{\tau}\circ d\overrightarrow{B}_{T-\tau}({\rm w}),
\end{align*}
where $\circ \overrightarrow{B}_{t}({\rm w})$ denote the backward  Stratonovich integral. 
Similarly, using (\ref{T-SDE-2}), we obtain 
\begin{align*}
d{\rm Y}^2_{\tau}&=((\mho_{T-\tau})^{-1}{\rm u}_\tau)d{\rm X}^2_{\tau}((\mho_{T-\tau})^{-1}{\rm u}_\tau)^{-1}\\
&= (\mho_{T-\tau})^{-1}R\left({\rm u}_{\tau}\circ dB_{\tau}(w), {\rm u}_{\tau}{\rm X}^1_{\tau}\right)\mho_{T-\tau}\\
&= -(\mho_{T-\tau})^{-1}R\left({\mho}_{T-\tau}\circ d\overrightarrow{B}_{T-\tau}({\rm w}), {\mho}_{T-\tau}{\rm Y}^1_{\tau}\right)\mho_{T-\tau}. 
\end{align*}
Altogether, we have 
\begin{align*}
d{\rm Y}^1_{\tau}&=-{\rm Y}^2_{\tau}\circ d\overrightarrow{B}_{T-\tau}({\rm w}),\\
d{\rm Y}^2_{\tau}&=-(\mho_{T-\tau})^{-1}R\left({\mho}_{T-\tau}\circ d\overrightarrow{B}_{T-\tau}({\rm w}), {\mho}_{T-\tau}{\rm Y}^1_{\tau}\right)\mho_{T-\tau}
\end{align*}
and the solution $({\rm Y}^1_{T}, {\rm Y}^2_{T})$ is exactly $(\theta,  \varpi)_{\mho_0}\left(D(F_{0, T-t}({\mho}_0, {\rm w}))^{-1}({\rm Y}({\rm X}))\right)$. 
\end{proof}

 As a corollary of (\ref{diff-in-lam-1}) and Lemma \ref{lem-fore-back}, we have
\begin{cor}\label{Cor-fore-back}Conditioned on ${\rm x}_T=y$, the distribution of $(\lfloor {\rm u}_T\rceil^{\l})^{(1)}_0$  given by (\ref{diff-in-lam-1}) is the same as,  conditioned on ${\rm y}_T=x$, the distribution of 
 \begin{align*}
\notag&\overline{(\lfloor {\rm u}_T\rceil^{\l})^{(1)}_0}:=[DF_{0, T}({\mho}_0, {\rm w})]^{-1}(\lfloor {\rm u}_0\rceil^{\l})^{(1)}_{0} - \int_{0}^{T}[DF_{0, t}({\mho}_0, {\rm w})]^{-1}(H^{\l})^{(1)}_{0}(\mho_{t}, e_i)\circ d\overrightarrow{B}_{t}^i({\rm w}),
\end{align*}
where $\circ d\overrightarrow{B}_{t}({\rm w})$ is the backward Stratonovich infinitesimal. 
\end{cor}

\begin{proof}  Consider the mapping 
\[
(H^{\l})^{(1)}_{0}({\rm u}_t, \cdot)=(H^{\l})^{(1)}_{0}(\mho_{T-t}(\mho_{T})^{-1}{\rm u}_0, \cdot)
\]
from $T_o\Bbb R^m$ to $T_{{\rm u}_t}\mathcal{F}(\M)$.  We have 
\[
(H^{\l})^{(1)}_{0}({\rm u}_t, e_i)\circ dB_t^i({w})=(H^{\l})^{(1)}_{0}({\rm u}_t, \circ dB_t({w}))
\]
and its correspondence at $T_{{\mho}_{T-t}}\mathcal{F}(\M)$ is  $-
(H^{\l})^{(1)}_{0}(\mho_{T-t}, \circ d\overrightarrow{B}_{T-t}({\rm w})).$
So, by Lemma \ref{lem-fore-back},  
\[\left[D \overrightarrow{F}_{t, T}({\rm u}_t, {w})\right](H^{\l})^{(1)}_{0}({\rm u}_t, \circ dB_t({w}))=-[DF_{0, T-t}({\mho}_0, {\rm w})]^{-1}(H^{\l})^{(1)}_{0}(\mho_{T-t}, \circ d\overrightarrow{B}_{T-t}({\rm w}))\]
and the conclusion follows by taking the integral with respect to $t$ on $[0, T]$. \end{proof}

 Let $\alpha^s, {\rm y}^s$ and  ${\mho}^s$ be the processes obtained in Theorem \ref{Main-alpha-x-v-Q}. Let $({F}^s_{t_1, t_2})_{0\leq t_1<t_2\leq T}$ be the parallel transportation stochastic flow of ${\rm y}^s$ and let $[DF_{t_1, t_2}^s({\mho}_{t_1}^s, {\rm w})]$ be the associated tangent maps.  By Proposition  \ref{El-Ku}, $[D{F}^s_{0, t}(\mho^s_t, {\rm w})]$ is invertible for almost all ${\rm w}$. Hence the inverse maps  $[DF_{0, t}^s({\mho}_0^s, {\rm w})]^{-1}$ are well-defined.   Corollary \ref{Cor-fore-back} shows the distribution of $(\lfloor {\rm u}_T\rceil^{\l})^{(1)}_0(w)$ is the same as  $\overline{(\lfloor {\rm u}_T\rceil^{\l})^{(1)}_0}({\rm w})$. We define 
 \[
 (\lfloor {\rm u}_T\rceil^{\l})^{(1)}_0(w)\circ {\bf F}^s:=\overline{(\lfloor {\rm u}_T\rceil^{\l})^{(1)}_0}({\rm w})\circ {\bf F}^s=\overline{(\lfloor {\rm u}_T^s\rceil^{\l})^{(1)}_0}({\rm w}),
 \]
where
\begin{align*}
\notag\overline{(\lfloor {\rm u}_T^s\rceil^{\l})^{(1)}_0}:=[DF_{0, T}^s({\mho}_0^s, {\rm w})]^{-1}(\lfloor {\rm u}_0\rceil^{\l})^{(1)}_{0}- \int_{0}^{T}[DF_{0, t}^s({\mho}_0^s, {\rm w})]^{-1}(H^{\l})^{(1)}_{0}(\mho_{t}^s, e_i)\circ d\overrightarrow\alpha^{s, i}_{t}({\rm w}). 
\end{align*}
So the differentiability of   $(\lfloor {\rm u}_T\rceil^{\l})^{(1)}_0\circ {\bf F}^s$ in $s$
will follow from  the differentiability of $\overline{(\lfloor {\rm u}_T^s\rceil^{\l})^{(1)}_0}$ in $s$,  which is intuitively true by  the differentiability of  (in $s$) of  $\alpha^{s}_{t}$,  $\mho^{s}_t$ and  $[DF_{0, t}^s({\mho}_0^s, \cdot)]^{-1}$.  We will justify this and  formulate $((\lfloor {\rm u}_T\rceil^{\l})^{(1)}_0\circ {\bf F}^s)'_s$ in the remaining part of this subsection. 

\begin{lem}\label{O-mathttg-diff-esti}Let $\alpha^s_t, O^s_t, \mathtt g^s_t$, $\underline{\Upsilon}_{{\rm V}, {\alpha}^\jmath}$ and $\mho_t^s$ be as in Theorem \ref{Main-alpha-x-v-Q}. Fix $T_0>0$. For any $s_0>0$, $q\geq 1$ and $T>T_0$,   there are constants  $\underline{c}_{{\rm A}}$ (which depends on $s_0, m, q,  \mathtt{s}$ and  $\|g^{0}\|_{C^3}$)  and  $c_{{\rm A}}$ (which depends  on $m, q, T, T_0$ and $\|g^{0}\|_{C^3}$) such that 
\begin{equation}\label{A-q-P-y-x-T}
\sup_{s\in [-s_0, s_0]}\E_{\P^*_{y, x, T}}\sup_{t\in [0, T]}\left|{\rm A}\right|^{q}<\underline{c}_{{\rm A}} e^{c_{{\rm A}}(1+d_{\wt{g}^{\l}}(x, y))}, 
\end{equation}
where ${\rm A}=\alpha^s_t,\  (O^s_t)'_s,\  (\mathtt g^s_t)'_s,\  \underline{\Upsilon}_{{\rm V}, {\alpha}^s}, \ (\mho_t^s)'_s $ or $\  (\theta, \varpi)\big( (\mho_t^s)'_s\big).$
\end{lem}
\begin{proof} By our construction,  $\alpha^{s}_{t}=\int_{0}^{t}O_{\tau}^sd B_{\tau}+{\mathtt g}^s_{\tau}\ d\tau$, where $O^s\in \mathcal{O}(\Bbb R^m)$ and 
$|{\mathtt g}^s|\leq {\mathtt{c}}s_0 \sup|{\rm V}|$ for some $\mathtt{c}$ that bounds $\sup_{t\in [0, T]}\{|\mathtt s|, |\mathtt s'|\}$, $\sup\{\|\rm Ric\|\}$.  So,  
\begin{align*}
2^{1-q}\E_{\P^*_{y, x, T}}\sup_{t\in [0, T]}\left|\alpha^{s}_{t}\right|^{q}\leq &\E_{\P^*_{y, x, T}}\sup_{t\in [0, T]}\left\|\int_{0}^{t}O_{\tau}^sd B_{\tau}\right\|^{q}+\underline{c}_0 T_0^{-m}(\mathtt{c}s_0T\sup|{\rm V}|)^q e^{c_0(1+T)}\\
=: & {\rm{(I)}}+\underline{c}_0 T_0^{-m}(\mathtt{c}s_0T\sup|{\rm V}|)^q e^{c_0(1+T)},
\end{align*}
where $\underline{c}_0, c_0$ are from (\ref{p-t-rough}).  Let  $b$ be the Brownian motion in Lemma \ref{Hsu-thm-5.4.4} for  $\P_{y, x, T}$, i.e., 
\begin{equation}\label{B-b-relation-0}
d B_{\tau}=d{b}_{\tau}+2({\mho}_\tau^0)^{-1}\nabla\ln p(T-\tau, {\rm y}_{\tau}^0, x)\ d\tau.
\end{equation} 
Then, \begin{align*}
{\rm{(I)}} =\ & \E_{\P^*_{y, x, T}}\sup_{t\in [0, T]}\left\|\int_{0}^{t}O_{\tau}^s\ d{b}_{\tau}+2O_{\tau}^s({\mho}_\tau^0)^{-1}\nabla\ln p(T-\tau, {\rm y}_{\tau}^0, x)\ d\tau\right\|^{q}\\
\leq\ & 2^{q-1}\E_{\P^*_{y, x, T}}\sup_{t\in [0, T]}\left\|\int_{0}^{t}O_{\tau}^s\ d{b}_{\tau}\right\|^q+2^{2q-1}\E_{\P^*_{y, x, T}}\left|\int_{0}^{T}\left\|\nabla\ln p(T-t, {\rm y}_{t}^0, x)\right\|\ dt\right|^{q}\\
=: &\; {\rm{(I)}_1}+{\rm{(I)}_2}. 
\end{align*}
For ${\rm{(I)}_1}$, by successively using Doob's inequality of submartingale, H\"{o}lder's inequality and Burkholder's inequality,  we obtain 
\begin{align*}
2^{1-q} {\rm{(I)}_1} \leq &\  {\mathtt C}(q)p(T, x, y)\E_{\P_{y, x, T}}\big\|\int_{0}^{T}O_{\tau}^s\ d{b}_{\tau}\big\|^{q}\\
\leq &\ {\mathtt C}(q)p(T, x, y)\left(\E_{\P_{y, x, T}}\big\|\int_{0}^{T}O_{\tau}^s\ d{b}_{\tau}\big\|^{2q}\right)^{\frac{1}{2}}\\
\leq &\ \  \underline{c}_0 T_0^{-m}{\mathtt C}(q){\mathtt C}_1(q)\sqrt{T}^{q}e^{c_0(1+T)}, 
\end{align*}
where ${\mathtt C}(q)=(q/q-1)^{q}$ and ${\mathtt{C}}_1(\cdot)$ is as in Lemma \ref{Ku-lem}. 
For ${\rm{(I)}_2}$,  by  Proposition \ref{cond-nabla-ln-p}, 
\begin{align*}
2^{1-2q}{\rm{(I)}_2} \leq &\ \E_{\P^*_{y, x, T}} \left(e^{q\int_{0}^{T}\|\nabla\ln p(T-t, {\rm y}_{t}^0, x)\|\ dt}\right)< e^{c(1+d(x, y))}, 
\end{align*}
where $c$ is as in (\ref{exp-grad-lnp}). Putting the estimations together, we obtain (\ref{A-q-P-y-x-T}) for ${\rm A}=\alpha^s_t$. 

Next, we consider  (\ref{A-q-P-y-x-T})  for  $(O^s_t)'_s$, $(\mathtt g^s_t)'_s$.  By (\ref{iteration-O}) and (\ref{iteration-g}), we have
\begin{align*}
(O_{t}^s)'_s&=-K_{{\rm V}, \alpha^s}(\tau)O_t^s, \\
(\mathtt g^s_{t})'_s&=-K_{{\rm V}, \alpha^s}(\tau){\mathtt g}_t^s + (\mho_0^s)^{-1}[\mathtt{s}'(t){\rm V}(F^sy)]- {\rm Ric}\left(\mho^s_{t}(\mho^s_{0})^{-1}[\mathtt{s}(t){\rm V}(F^sy)]\right), 
\end{align*}
where 
\begin{align*}K_{{\rm V}, \alpha^{s}}(t)=&\int_{0}^{t}({\mho_{\tau}^{s}})^{-1}R\left({\mho_{\tau}^{s}}d\alpha_{\tau}^{s}, {\mho_{\tau}^{s}}(\mho_{0}^s)^{-1}[\mathtt{s}(\tau){\rm V}(F^s y)]\right){\mho_{\tau}^{s}}\\
&\!+\!\int_{0}^{t}\!({\mho_{\tau}^{s}})^{-1}\!(\nabla ({\mho_{\tau}^{s}} e_i)R)\big({\mho_{\tau}^{s}}e_i, {\mho_{\tau}^{s}}\!(\mho_0^s)^{-1}\![\mathtt{s}(\tau){\rm V}(F^sy)]\big){\mho_{\tau}^{s}}d\tau\notag.
\end{align*}
Since  $O^s\in \mathcal{O}(\Bbb R^m)$,  
$|{\mathtt g}^s|\leq {\mathtt{c}}s_0 \sup|{\rm V}|$,   and all the $|\mathtt s|, |\mathtt s'|$ and  $|{\rm V}|$ are uniformly bounded, it  is clear that  (\ref{A-q-P-y-x-T}) holds for $(O^s_t)'_s, (\mathtt g^s_t)'_s$ if  it holds for $K_{{\rm V}, \alpha^{s}}(t)$.  Using (\ref{B-b-relation-0}) and (\ref{abcq}),  we obtain 
\begin{align*}
&\E_{\P^*_{y, x, T}}\sup_{t\in[0, T]}\left|K_{{\rm V}, \alpha^{s}}(t)\right|^{q}\\
& \ \leq\ 3^{q-1}\E_{\P^*_{y, x, T}}\sup_{t\in[0, T]}\left|\int_{0}^{t}\!({\mho_{\tau}^{s}})^{-1}\!R\left({\mho_{\tau}^{s}}O_{\tau}^{s}db_{\tau}^{s}, {\mho_{\tau}^{s}}(\mho_{0}^s)^{-1}[\mathtt{s}(\tau){\rm V}(F^s y)]\right){\mho_{\tau}^{s}}\right|^q\\
&\ \ +3^{q-1}(2\sup\|R\|\sup|{\rm{V}})^q\E_{{\P^*}_{y, x, T}}\left|\int_{0}^{T}\|\nabla\ln p(T-t, {\rm y}_{t}^0, x)\|\ dt\right|^{q}\\
&\ \ +3^{q-1}(\sup\|\nabla R\|\sup|{\rm{V}}|s_0T)^q,
\end{align*}
which has the same bound type as in (\ref{A-q-P-y-x-T}) by a computation  similar to the one for ${\rm{(I)}}$.  

To  verify (\ref{A-q-P-y-x-T})  for $\underline{\Upsilon}_{{\rm V}, {\alpha}^s}$, it suffices to check it for ${\mathsf{K}}_t:=\int_{0}^{t}\langle K_{{\rm V}, \alpha^s}(\tau), d\alpha_{\tau}^s\rangle$ since 
 \[
\underline{\Upsilon}_{{\rm V}, \alpha^s}(t)=\int_{0}^{t} (\mho_0^s)^{-1}[\mathtt{s}'(\tau){\rm V}(F^sy)]- {\rm Ric}\left(\mho^s_{\tau}(\mho^s_{0})^{-1}[\mathtt{s}(\tau){\rm V}(F^sy)]\right)-{\mathsf{K}}_t. 
\]
By (\ref{B-b-relation-0}) and (\ref{abcq}), 
\begin{align*}
3^{1-q}\E_{\P^*_{y, x, T}}\!\!\sup_{t\in [0, T]}\left|{\mathsf{K}}_t\right|^{q}
 \leq\ &\  \E_{\P^*_{y, x, T}}\sup_{t\in [0, T]}\left|\int_{0}^{t}\langle K_{{\rm V}, \alpha^s}(\tau), O^s_{\tau}db_{\tau}\rangle\right|^{q}\\
&+\E_{\P^*_{y, x, T}}\sup_{t\in [0, T]}\left|\int_{0}^{t}\langle K_{{\rm V}, \alpha^s}(\tau), 2{\mho}_\tau^{-1}\nabla\ln p(T-\tau, {\rm y}_{\tau}^0, x)\ d\tau\rangle\right|^{q}\\
&+\E_{\P^*_{y, x, T}}\sup_{t\in [0, T]}\left|\int_{0}^{t}\langle K_{{\rm V}, \alpha^s}(\tau), {\mathtt g}_{\tau}^s\ d\tau\rangle\right|^{q}\\
=:\ & ({\rm{II}})_1+ ({\rm{II}})_2+({\rm{II}})_3. 
\end{align*}
For  $({\rm{II}})_1$, it is routine to apply successively  H\"{o}lder's inequality,  Doob's inequality of submartingale and Burkholder's  inequality, which gives
\begin{align*}
(({\rm{II}})_1)^2 \leq & p(T, x, y) \E_{\P^*_{y, x, T}}\sup_{t\in [0, T]}\left|\int_{0}^{t}\langle K_{{\rm V}, \alpha^s}(\tau), O^s_{\tau}db_{\tau}\rangle\right|^{2q}\\
\leq &{\mathtt C}(2q) p(T, x, y) \E_{\P_{y, x, T}}\left|\int_{0}^{T}\langle K_{{\rm V}, \alpha^s}(\tau), O^s_{\tau}db_{\tau}\rangle\right|^{2q}\\
\leq & {\mathtt C}(2q){\mathtt{C}}_1(2q) p(T, x, y)\E_{\P_{y, x, T}}\left|\int_{0}^{T}\left\|\langle K_{{\rm V}, \alpha^s}(\tau), O^s_{\tau}\rangle\right\|^2\ d\tau\right|^{q}\\
\leq &{\mathtt C}(2q) {\mathtt{C}}_1(2q)T^{q} \E_{\P^*_{y, x, T}}\sup_{\tau\in[0, T]}\left|K_{{\rm V}, \alpha^{s}}(\tau)\right|^{2q}. 
\end{align*}
For $({\rm{II}})_2$, it is true that 
\begin{align*}
(({\rm{II}})_2)^2 \leq &2^{2q}\E_{\P^*_{y, x, T}}\sup_{\tau\in[0, T]}\left|K_{{\rm V}, \alpha^{s}}(\tau)\right|^{2q}\cdot \E_{{\P}^*_{y, x, T}}\left|\int_{0}^{T}\|\nabla\ln p(T-\tau, {\rm y}_{t}^0, x)\|\ d\tau\right|^{2q}. 
\end{align*}
 For $({\rm{II}})_3$, a routine calculation shows 
\begin{align*}
({\rm{II}})_3\leq  (cs_0 \sup|{\rm V}|)^q T^{q}\E_{\P^*_{y, x, T}}\sup_{\tau\in[0, T]}\left|K_{{\rm V}, \alpha^{s}}(\tau)\right|^{q}. 
\end{align*}
Putting the estimations on $({\rm{II}})_1, ({\rm{II}})_2$ and  $({\rm{II}})_3$ together,   we conclude from  Proposition \ref{cond-nabla-ln-p} and the  estimation for $K_{{\rm V}, \alpha^{s}}$ that   (\ref{A-q-P-y-x-T}) also  holds true for ${\mathsf{K}}_t$. This shows (\ref{A-q-P-y-x-T}) for  $\underline{\Upsilon}_{{\rm V}, {\alpha}^s}$. 

Finally,  to check (\ref{A-q-P-y-x-T})  for $(\mho_t^s)'_s$, $(\theta, \varpi)\big( (\mho_t^s)'_s\big)$, it  suffices to consider the latter, which holds true by  the above conclusion  for $K_{{\rm V}, \alpha^{s}}$ since,   by  (\ref{Y-t-s-diff-ITO}), 
\[
\theta(Y^s_t)=\mathtt{s}(t)({\mho}_0^s)^{-1}V(F^s y), \  \varpi(Y^s_t)=K_{{\rm V}, \alpha^{s}}(t). 
\]
\end{proof}

\begin{lem}\label{cond-DF-norm}Let $\alpha^s$ be as in Theorem \ref{Main-alpha-x-v-Q}.  For $\underline{t}, t$, $0\leq \underline{t}<t\leq T$, we abbreviate 
\begin{align*}
[DF_{\underline{t}, t}^s({\mho}_{\underline{t}}^s, {\rm w})]&:=[DF_{\underline{t}, t}^{\alpha^s}({\mho}_{\underline{t}}^s, {\rm w})],\\
[\wt{(DF_{\underline{t}, t}^s)}(\mho_{0}^s, {\rm w})]&:=(\theta, \omega)_{\mho^s_t}[DF_{\underline{t}, t}^s({\mho}_{\underline{t}}^s, {\rm w})](\theta, \omega)_{\mho^s_{\underline{t}}}^{-1}. 
\end{align*}
Let $T_0>0$. 
For any $s_0>0$, $q\geq 1$ and $T>T_0$,  there are constants  $\underline{c}_{{\rm F}}$ (which depends on $s_0, m, q,  \mathtt{s}$ and $\|g^{0}\|_{C^2}$)  and  $c_{{\rm F}}$ (which depends  on $s_0, m, q,  \mathtt{s}$, $T, T_0$ and  $\|g^{0}\|_{C^3}$) such that
 \begin{align}
\notag&\sup_{s\in [-s_0, s_0]}\E_{\P^*_{y, x, T}}\sup_{0\leq \underline{t}<{t}\leq T}\left\|[DF_{\underline{t}, t}^s({\mho}_{\underline{t}}^s, {\rm w})]^{-1}\right\|^q,\ \sup_{0\leq \underline{t}<{t}\leq T}\left\|[\wt{(DF_{\underline{t}, {t}}^s)}(\mho_{\un{t}}^s, {\rm w})]^{-1}\right\|^q\\
&\ \ \ \ \ \ \ \ \ \ \ \ \  \ \ \ \  \ \ \ \ \ \ \ \ \ \  <\underline{c}_{{\rm F}} e^{c_{{\rm F}}(1+d_{\wt{g}^{\l}}(x, y))}. \label{eq-cond-DF-n}
\end{align}
\end{lem}
\begin{proof}For (\ref{eq-cond-DF-n}), it suffices to consider the second estimation.  Let  $s\in [-s_0, s_0]$ and  $\un{t}, t\in [0, T]$ with $\un{t}<t$. For $(\mathtt{v}_0, {\mathtt Q}_0)\in T_{o}\mathcal{F}(\Bbb R^m)$,  let 
\[
(\mathtt{v}_{{t}-\tau}, {\mathtt Q}_{{t}-\tau}):= [\wt{(DF_{\tau, t}^s)}(\mho_{0}^s, {\rm w})]^{-1}(\mathtt{v}_0, {\mathtt Q}_0), \ \forall \tau\in [\un{t}, {t}]. 
\]
Then  Lemma \ref{Tangent map-SDE} shows that   ${\rm z}_{\tau}:=(\mathtt{v}_\tau, {\mathtt Q}_\tau)$  satisfies  the  It\^{o} form SDE 
\begin{equation*}
d{\rm z}_{t-\tau}({\rm w})=\sum_{j=1}^m\left(-{\bf M}_j(\mho_{\tau}^s){\rm z}_{t-\tau}({\rm w})\ d\overleftarrow\alpha^{s, j}_{\tau}({\rm w})+[{\bf M}_j(\mho_{\tau}^s)]^2{\rm z}_{t-\tau}({\rm w})\ d\tau\right) +{\bf N}(\mho_{\tau}^s){\rm z}_{t-\tau}({\rm w})\ d\tau,  
\end{equation*}
where  ${\bf M}_j,  {\bf N}$ are given in (\ref{matrix-M-j-def}), (\ref{matrix-N-def}).The remaining estimation for (\ref{eq-cond-DF-n}) can be done  by following the proof of Proposition \ref{est-norm-D-j-F-t-cond}. \end{proof}

\begin{lem}\label{cond-DF-diff-norm} Let $\alpha^s$ be as in Theorem \ref{Main-alpha-x-v-Q}. Then $((D{F}^s_{0, t})^{-1})_{t\in [0, T]}, (\wt{(DF_{0, t}^s)^{-1}})_{t\in [0, T]}$ are $C^1$ in the $s$ parameter.  Let $T_0>0$. 
For any $s_0>0$, $q\geq 1$ and $T>T_0$, there exist   $\underline{c}_{{\rm F}}'$ (which depends on $s_0, m, q,  \mathtt{s}$ and  $\|g^{0}\|_{C^3}$)  and  $c_{{\rm F}}'$ (which depends  on $s_0, m, q,  \mathtt{s}, T, T_0$ and  $\|g^{0}\|_{C^3}$) such that
\begin{align}\notag
&\sup_{s\in [-s_0, s_0]}\overline{\E}_{\ov{\P}^*_{y, x, T}}\sup_{t\in [0, T]}\big\|([{(DF_{0, t}^s)}(\mho_{0}^s, {\rm w})]^{-1})'_s\big\|^q, \sup_{t\in [0, T]}\left\|([\wt{(DF_{0, t}^s)}(\mho_{0}^s, {\rm w})]^{-1})'_s\right\|^q\\
&\ \ \ \ \ \ \ \ \ \ \ \ \  \ \ \ \  \ \ \ \ \ \ \ \ \ \  <\underline{c}_{{\rm F}}' e^{c_{{\rm F}}'(1+d_{\wt{g}^{\l}}(x, y))}. \label{DF-s-diff-b}
\end{align}
\end{lem}
\begin{proof}The $C^1$ regularity of  $s\mapsto (D{F}^s_{0, t})^{-1}$ follows from that of  $s\mapsto \wt{(DF_{0, t}^s)^{-1}}$ since 
\[
[DF_{0, t}^s({\mho}_0^s, {\rm w})]^{-1}=(\theta, \omega)_{\mho^s_0}^{-1}[\wt{(DF_{0, t}^s)}(\mho_{0}^s, {\rm w})]^{-1}(\theta, \omega)_{\mho^s_t}
\]
and $s\mapsto (\theta, \omega)_{\mho^s_t}^{-1}$ is $C^1$.  By Theorem \ref{Main-alpha-x-v-Q}, $\wt{(DF_{0, t}^s)}(\mho_{0}^s, {\rm w})$ is $C^1$ in $s$ for almost all ${\rm w}$. Hence  $[\wt{(DF_{0, t}^s)}(\mho_{0}^s, {\rm w})]^{-1}$ is also $C^1$ in $s$ by the identity \begin{equation*}
[\wt{(DF_{0, t}^s)}(\mho_{0}^s, {\rm w})]^{-1}\circ \wt{(DF_{0, t}^s)}(\mho_{0}^s, {\rm w})={\rm Id}.
\end{equation*}

For (\ref{DF-s-diff-b}), it suffices to consider the second estimation.  For  ${\rm z}_0\in T_o\mathcal{F}(\Bbb R^m)$, let   \[
{\rm z}_{t-\tau}^s:= [\wt{(DF_{0, t-\tau}^s)}(\mho_{0}^s, {\rm w})]^{-1}{\rm z}_0, \ \forall \tau\in [0, t], \ \forall  s\in [-s_0, s_0]. 
\] 
It satisfies the SDE \begin{equation*}
d{\rm z}_{t-\tau}^s({\rm w})=\sum_{j=1}^m\left(-{\bf M}_j(\mho_{\tau}^s){\rm z}_{t-\tau}^s({\rm w})\ d\overleftarrow\alpha^{s, j}_{\tau}({\rm w})+[{\bf M}_j(\mho_{\tau}^s)]^2{\rm z}_{t-\tau}^s({\rm w})\ d\tau\right) +{\bf N}(\mho_{\tau}^s){\rm z}_{t-\tau}^s({\rm w})\ d\tau,
\end{equation*}
where  $({\bf M}_j)_{1\leq j\leq m},  {\bf N}$ are given in (\ref{matrix-M-j-def}), (\ref{matrix-N-def}). 
For $({\rm z}_t^s)'_s:= \left(d{\rm z}_t^s/ds)\right|_s$, its  SDE is 
\begin{align*}
&d({\rm z}_{t-\tau}^s)'_s({\rm w})\\
&= \sum_{j=1}^m\left(-{\bf M}_j(\mho_{\tau}^s)({\rm z}_{t-\tau}^s)'_s({\rm w})\ d\overleftarrow\alpha^{s, j}_{\tau}({\rm w})+[{\bf M}_j(\mho_{\tau}^s)]^2({\rm z}_{t-\tau}^s)'_s({\rm w})\ d\tau\right) +{\bf N}(\mho_{\tau}^s)({\rm z}_{t-\tau}^s)'_s({\rm w})\ d\tau\\
&\ \ \ +\!\sum_{j=1}^{m}-\left({\bf M}_j(\mho_{\tau}^s)\ d\overleftarrow\alpha^{s, j}_{\tau}({\rm w})\right)'_s {\rm z}_{t-\tau}^s+\big(\sum_{j=1}^{m}[{\bf M}_j(\mho_{\tau}^s)]^2+{\bf N}(\mho_{\tau}^s)\big)'_s \ {\rm z}_{t-\tau}^s\ d\tau. 
\end{align*}
Let $O^s=((O^s)_{l}^{j})_{j, l\leq m}$, $\mathtt{g}=(\mathtt g^{s, j})_{j\leq m}$. They are differentiable in $s$ by Theorem \ref{Main-alpha-x-v-Q}.  Let 
\begin{align*}
\big({\bf A}^{(1)}_{l}\big)^s_{\tau}&:=\sum_{j=1}^{m}\left(({\bf M}_j(\mho_{\tau}^s))'_s(O^s_{\tau})^{j}_{l}+{\bf M}_j(\mho_{\tau}^s)((O^s_{\tau})^{j}_{l})'_s\right), \forall  l\leq m, \\
\big({\bf A}^{(2)}\big)^s_{\tau}&:=\sum_{j=1}^m\left({\bf M}_j(\mho_{\tau}^s)(\mathtt{g}^{s, j}_{\tau})'_s+ \left([{\bf M}_j(\mho_{\tau}^s)]^2\right)'_s\right) + \big({\bf N}(\mho_{\tau}^s)\big)'_s-2\sum_{l, j=1}^{m}{\bf M}_j(\mho_{\tau}^s)(O_{\tau}^s)^j_{l}\big({\bf A}^{(1)}_{l}\big)^s_{\tau}. 
\end{align*}
By Duhamel's principle,  we have  
\begin{align*}
({\rm z}_t^s)'_s&=\left[\big[\wt{(DF_{0, t}^s)}(\mho_{0}^s, {\rm w})\big]^{-1}\!\!\int_{0}^{t}\big[\wt{(DF_{\tau, t}^s)}(\mho_{\tau}^s, {\rm w})\big] \big({\bf A}^{(1)}_{l}\big)^s_{\tau}({\rm w})\big[\wt{(DF_{\tau, t}^s)}(\mho_{\tau}^s, {\rm w})\big]^{-1}\!\! d\overleftarrow B_{\tau}^l\right] {\rm z}_0 \\
&\  +\left[\big[\wt{(DF_{0, t}^s)}(\mho_{0}^s, {\rm w})\big]^{-1}\int_{0}^{t}\big[\wt{(DF_{\tau, t}^s)}(\mho_{\tau}^s, {\rm w})\big] \big({\bf A}^{(2)}\big)^s_{\tau}({\rm w})\big[\wt{(DF_{\tau, t}^s)}(\mho_{\tau}^s, {\rm w})\big]^{-1}d\tau\right] {\rm z}_0.
\end{align*}
This means
\begin{align*}
\big(\big[\wt{(DF_{0, t}^s)}(\mho_{0}^s, {\rm w})\big]^{-1}\big)'_s
&=\int_{0}^{t}\big[\wt{(DF_{0, \tau}^s)}(\mho_{0}^s, {\rm w})\big]^{-1} \big({\bf A}^{(1)}_{l}\big)^s_{\tau}({\rm w})\big[\wt{(DF_{\tau, t}^s)}(\mho_{\tau}^s, {\rm w})\big]^{-1} d\overleftarrow B_{\tau}^l \\
&\ \ \  +\int_{0}^{t}\big[\wt{(DF_{0, \tau}^s)}(\mho_{0}^s, {\rm w})\big]^{-1} \big({\bf A}^{(2)}\big)^s_{\tau}({\rm w})\big[\wt{(DF_{\tau, t}^s)}(\mho_{\tau}^s, {\rm w})\big]^{-1} d\tau\\
&=: ({\bf I})_{t}^{s} + ({\bf II})_{t}^{s}. 
\end{align*}
For (\ref{DF-s-diff-b}),  it suffices to show the same bound type is valid for 
\[
({\bf I}):=\sup_{s\in [-s_0, s_0]}\overline{\E}_{\ov{\P}^*_{y, x, T}}\sup_{t\in [0, T]}\left\| ({\bf I})_{t}^{s}\right\|^q,\  ({\bf II}):=\sup_{s\in [-s_0, s_0]}\overline{\E}_{\ov{\P}^*_{y, x, T}}\sup_{t\in [0, T]}\left\| ({\bf II})_{t}^{s}\right\|^q. 
\]
This will follow from    Lemma \ref{cond-DF-norm}  and Proposition \ref{cond-nabla-ln-p}.  Clearly, \[
\overline{\E}_{\ov{\P}^*_{y, x, T}}\!\!\sup_{\tau\in [0, T]}\left\|\big({\bf A}^{(i)}_{l}\big)^s_{\tau}({\rm w})\right\|^{q}\!\leq c_{{\bf M}} \overline{\E}_{\ov{\P}^*_{y, x, T}}\!\max\big\{\!\!\sup_{t\in [0, T]}\left\|(\mho_{t}^s)'_s\right\|^{2q}\!, \!\sup_{t\in [0, T]}\left\|(O_{t}^s)'_s\right\|^{2q}\!,  \!\sup_{t\in [0, T]}\left\|(\mathtt{g}_{t}^s)'_s\right\|^{q}\big\},\]
where  $c_{{\bf M}}$ depends on the norm bounds of  $\{{\bf M}_j\}$ and their differentials.  Hence by Lemma \ref{O-mathttg-diff-esti},  there are constants  $\underline{c}_{{\bf A}}$ (which depends on $s_0, m, q,  \mathtt{s}$ and $\|g^{0}\|_{C^3}$)  and  $c_{{\bf A}}$ (which depends  on $m, q$, $T, T_0$ and  $\|g^{0}\|_{C^3}$) such that 
\begin{equation}\label{P-xy-T-A-i}
\overline{\E}_{\ov{\P}^*_{y, x, T}}\!\!\sup_{\tau\in [0, T]}\left\|\big({\bf A}^{(i)}_{l}\big)^s_{\tau}({\rm w})\right\|^{q}\leq \underline{c}_{{\bf A}} e^{c_{{\bf A}}(1+d_{\wt{g}^{\l}}(x, y))}.  
\end{equation}
Let 
\begin{align*}
({\bf III})_t^s& := \int_{0}^{t}\left[\wt{(DF_{0, \tau}^s)}(\mho_{0}^s, {\rm w})\right]^{-1}\!\!\big({\bf A}^{(1)}_{l}\big)^s_{\tau}({\rm w})\left[\wt{(DF_{\tau, t}^s)}(\mho_{\tau}^s, {\rm w})\right]^{-1}\!\! d\overleftarrow{b}_{\tau}^l,\\
({\bf IV})_t^s& :=\int_{0}^{t}\left[\wt{(DF_{0, \tau}^s)}(\mho_{0}^s, {\rm w})\right]^{-1}\!\!\big({\bf A}^{(1)}_{l}\big)^s_{\tau}({\rm w})\left[\wt{(DF_{\tau, t}^s)}(\mho_{\tau}^s, {\rm w})\right]^{-1}\!\! ({\mho}_\tau^{-1}\nabla\ln p(T-\tau, {\rm y}_{\tau}^0, x))^{l}\ d\tau, 
\end{align*}
where $b_{\tau}$ is the Brownian motion in Lemma \ref{Hsu-thm-5.4.4} for  $\P_{y, x, T}$. Then 
\begin{align*}
\overline{\E}_{\ov{\P}^*_{y, x, T}}\sup_{t\in [0, T]}\left\| ({\bf I})_{t}^{s}\right\|^q\leq  2^{q-1}\overline{\E}_{\ov{\P}^*_{y, x, T}}\sup_{t\in [0, T]}\left\| ({\bf III})_{t}^{s}\right\|^q+2^{2q-1}\overline{\E}_{\ov{\P}^*_{y, x, T}}\sup_{t\in [0, T]}\left\| ({\bf IV})_{t}^{s}\right\|^q.
\end{align*}
As usual, we can use H\"{o}lder's inequality and Doob's  maximal  inequality of sub-martingales to deduce that 
\begin{align*}
\big(\overline{\E}_{\ov{\P}^*_{y, x, T}}\sup_{t\in [0, T]}\left\| ({\bf III})_{t}^{s}\right\|^q\big)^2
\leq &p(T, x, y)(\frac{2q}{2q-1})^{2q}\overline{\E}_{\ov{\P}^*_{y, x, T}}\left\|({\bf III})_{T}^{s}\right\|^{2q}. 
\end{align*}
Let ${\mathtt{C}}_1(\cdot)$ be the constant function in Lemma \ref{Ku-lem}. We  continue to compute that 
\begin{align*}
&\overline{\E}_{\ov{\P}^*_{y, x, T}}\left\|({\bf III})_{T}^{s}\right\|^{2q}\\
& \leq {\mathtt{C}}_1(2q)\overline{\E}_{\ov{\P}^*_{y, x, T}}\left|\int_{0}^{T}\left\|\big[\wt{(DF_{0, \tau}^s)}(\mho_{0}^s, {\rm w})\big]^{-1} \big({\bf A}^{(1)}_{l}\big)^s_{\tau}({\rm w})\big[\wt{(DF_{\tau, T}^s)}(\mho_{\tau}^s, {\rm w})\big]^{-1}\right\|^2\ d\tau\right|^{q}\\
& \leq  {\mathtt{C}}_1(2q)T^{q}\left(\overline{\E}_{\ov{\P}^*_{y, x, T}}\sup_{0\leq \underline{t}<{t}\leq T}\left\|[DF_{\underline{t}, t}^s({\mho}_{\underline{t}}^s, {\rm w})]^{-1}\right\|^{8q}\cdot \overline{\E}_{\ov{\P}^*_{y, x, T}}\sup_{\tau\in [0, T]}\left\|\big({\bf A}^{(1)}_{l}\big)^s_{\tau}({\rm w})\right\|^{4q}\right)^{\frac{1}{2}},
\end{align*}
which has the same type of bound as in (\ref{DF-s-diff-b}) by Lemma \ref{cond-DF-norm} and (\ref{P-xy-T-A-i}).  Similarly, \begin{align*}
\left(\overline{\E}_{{\P}^*_{y, x, T}}\!\sup_{t\in [0, T]}\left\|({\bf IV})_{t}^{s}\right\|^{q}\right)^3\!\!\!\leq &\overline{\E}_{{\P}^*_{y, x, T}}\sup_{0\leq \underline{t}<{t}\leq T}\left\|[DF_{\underline{t}, t}^s({\mho}_{\underline{t}}^s, {\rm w})]^{-1}\right\|^{6q}\!\!\cdot \overline{\E}_{{\P}^*_{y, x, T}}\sup_{\tau\in [0, T]}\left\|\big({\bf A}^{(1)}_{l}\big)^s_{\tau}({\rm w})\right\|^{3q}\\
&\!\cdot \overline{\E}_{{\P}^*_{y, x, T}}\left|\int_{0}^{T}\|\nabla\ln p(T-\tau, {\rm y}_{\tau}^0, x)\|\ d\tau\right|^{3q}, 
\end{align*}
which also has the same type of bound as in (\ref{DF-s-diff-b}) by Proposition \ref{cond-nabla-ln-p}, Lemma \ref{cond-DF-norm} and (\ref{P-xy-T-A-i}). Altogether,  the same type of bound as in (\ref{DF-s-diff-b}) is valid for $({\bf I})$.  This is also true for $({\bf II})$ by Lemma \ref{cond-DF-norm} and (\ref{P-xy-T-A-i}) since 
\begin{align*}
\left(\overline{\E}_{{\P}^*_{y, x, T}}\|({\bf II})_{T}^{s}\|^q\right)^2
\!&\leq T^{2q}\overline{\E}_{{\P}^*_{y, x, T}}\sup_{0\leq \underline{t}<{t}\leq T}\left\|[\wt{(DF_{\underline{t}, {t}}^s)}(\mho_{\un{t}}^s, {\rm w})]^{-1}\right\|^{2q} \overline{\E}_{{\P}^*_{y, x, T}}\sup_{\tau\in [0, T]}\left\|\big({\bf A}^{(2)}_{l}\big)^s_{\tau}({\rm w})\right\|^{2q}.
\end{align*}
\end{proof}

With Lemmas \ref{O-mathttg-diff-esti}-\ref{cond-DF-diff-norm}, we can deduce  the differentiability of   $(D\pi(\lfloor {\rm u}_T\rceil^{\l})^{(1)}_0)\circ {\bf F}^s$ in $s$.

\begin{prop}\label{Diff-int-U-s} Fix  $T_0>0$.  For any $q\geq 1$ and $T>T_0$, there are $\underline{c}_{{\bf F}}$ (depending on $s_0, m, q,  \mathtt{s}$, $\|g^{0}\|_{C^2}$ and $\|\XX^{0}\|_{C^1}$)  and  $c_{{\bf F}}$ (depending on  $s_0, m, q,  \mathtt{s}, T, T_0$ and $\|g^{0}\|_{C^3}$) such that 
\begin{align}\label{dpi-UFS}
\sup_{s\in [-s_0, s_0]}\overline{\E}_{\ov{\P}^*_{y, x, T}} \left\|(D\pi(\lfloor {\rm u}_T\rceil^{\l})^{(1)}_0)\circ {\bf F}^s\right\|^q\leq \underline{c}_{{\bf F}} e^{c_{{\bf F}}(1+d_{\wt{g}}(x, y))}.
\end{align}
The one parameter family  of processes $\{(D\pi(\lfloor {\rm u}_T\rceil^{\l})^{(1)}_0)\circ {\bf F}^s\}$ is differentiable in $s$. Let \[
\nabla_{T, {\rm V}, \mathtt{s}}^sD\pi(\lfloor {\rm u}_T\rceil^{\l})^{(1)}_0:=\left((D\pi(\lfloor {\rm u}_T\rceil^{\l})^{(1)}_0)\circ {\bf F}^s\right)'_{s}.
\] 
For any $q\geq 1$ and $T>T_0$, there are $\underline{c}'_{{\bf F}}$ (depending on $s_0, m, q,  \mathtt{s}$, $\|g^{0}\|_{C^3}$ and  $\|\XX^{0}\|_{C^2}$)  and  $c'_{{\bf F}}$ (depending on  $s_0, m, q,  \mathtt{s}, T, T_0$ and $\|g^{0}\|_{C^3}$) such that 
\begin{align}\label{dpi-UFS-diff}
\sup_{s\in [-s_0, s_0]}\overline{\E}_{\ov{\P}^*_{y, x, T}} \left\|\nabla_{T, {\rm V}, \mathtt{s}}^sD\pi(\lfloor {\rm u}_T\rceil^{\l})^{(1)}_0\right\|^q\leq \underline{c}'_{{\bf F}} e^{c'_{{\bf F}}(1+d_{\wt{g}}(x, y))}.
\end{align}
\end{prop}
\begin{proof}
Recall that \[
(D\pi(\lfloor {\rm u}_T\rceil^{\l})^{(1)}_0)\circ {\bf F}^s=D\pi \big((\lfloor {\rm u}_T\rceil^{\l})^{(1)}_0\circ {\bf F}^s\big)=D\pi\big(\overline{(\lfloor {\rm u}_T^s\rceil^{\l})^{(1)}_0}(\rm w)\big),
\]
where 
\begin{align*}
\overline{(\lfloor {\rm u}_T^s\rceil^{\l})^{(1)}_0}({\rm w})=[DF_{0, T}^s({\mho}_0^s, {\rm w})]^{-1}(\lfloor {\rm u}_0\rceil^{\l})^{(1)}_{0}\!-\!\int_{0}^{T}[DF_{0, t}^s({\mho}_0^s, {\rm w})]^{-1}(H^{\l})^{(1)}_{0}(\mho_{t}^s, e_i)\circ d\overleftarrow\alpha^{s, i}_{t}({\rm w}). 
\end{align*}
Let 
\[
\wt{(\lfloor {\rm u}_T^s\rceil^{\l})^{(1)}_0}({\rm w}):=(\theta, \omega)_{\mho_{0}^s}\left(\overline{(\lfloor {\rm u}_T^s\rceil^{\l})^{(1)}_0}({\rm w})\right), \ \ \wt{(\lfloor {\rm u}_0^s\rceil^{\l})^{(1)}_{0}}:=(\theta, \omega)_{\mho_{T}^s}\left((\lfloor {\rm u}_0\rceil^{\l})^{(1)}_{0}\right).
\]
It is easy to obtain the following  It\^{o} form  expression:
\begin{align*}
\wt{(\lfloor {\rm u}_T^s\rceil^{\l})^{(1)}_0}({\rm w})=&[\wt{(DF_{0, T}^s)}(\mho_{0}^s, {\rm w})]^{-1}\wt{(\lfloor {\rm u}_0^s\rceil^{\l})^{(1)}_{0}}\\
&\!-\!\int_{0}^{T}[\wt{(DF_{0, t}^s)}(\mho_{0}^s, {\rm w})]^{-1}\!\!\left(\varpi\big((H^{\l})^{(1)}_0(\mho_{t}^s, e_i)\big)e_i\ dt, \varpi\big((H^{\l})^{(1)}_0(\mho_{t}^s,  d\overleftarrow\alpha^{s}_{t}({\rm w}))\big)\right).
\end{align*}
For Proposition \ref{Diff-int-U-s},  it is equivalent to show the differentiability of $s\mapsto \wt{(\lfloor {\rm u}_T^s\rceil^{\l})^{(1)}_0}({\rm w})$ and estimate  the conditional $L^q$ integrals of its differential process and itself. 

The estimation in (\ref{dpi-UFS}) is valid since \[
\sup_{s\in [-s_0, s_0]}\overline{\E}_{\ov{\P}^*_{y, x, T}} \left\|(D\pi(\lfloor {\rm u}_T\rceil^{\l})^{(1)}_0)\circ {\bf F}^s\right\|^q\leq \sup_{s\in [-s_0, s_0]}\overline{\E}_{\ov{\P}^*_{y, x, T}} \left\|\wt{(\lfloor {\rm u}_T^s\rceil^{\l})^{(1)}_0}({\rm w})\right\|^q,
\]
where the second term has a bound in (\ref{dpi-UFS}) by  following the argument of (\ref{DF-j-lambda-cond}) in Proposition \ref{est-norm-u-t-j} and  using Lemma \ref{O-mathttg-diff-esti} and Lemma \ref{cond-DF-norm}. 

 The processes $\alpha^s, \mho^s$ and  $[\wt{DF_{0, t}^s}(\mho_{0}^s, {\rm w})]^{-1}$ are all differentiable in $s$ by Theorem \ref{Main-alpha-x-v-Q}. Lemmas \ref{O-mathttg-diff-esti}-\ref{cond-DF-diff-norm} show  that $\alpha^s_t,$ $\underline{\Upsilon}_{{\rm V}, {\alpha}^s},$ $(\mho_t^s)'_s$, $[\wt{(DF_{0, t}^s)}(\mho_{0}^s, {\rm w})]^{-1}$ and $([\wt{(DF_{0, t}^s)}(\mho_{0}^s, {\rm w})]^{-1})'_s$ all have bounded sup $L^{q}$ ($q\geq 1$) norm  with respect to $\overline{\P}_{y, x, T}$.  Hence $s\mapsto \wt{(\lfloor {\rm u}_T^s\rceil^{\l})^{(1)}_0}({\rm w})$ is also differentiable in $s$ and the differential is 
\begin{align*}
&\big(\wt{(\lfloor {\rm u}_T^s\rceil^{\l})^{(1)}_0}({\rm w})\big)'_s\\ &\ \ =\big([\wt{(DF_{0, T}^s)}(\mho_{0}^s, {\rm w})]^{-1}\big)'_s\wt{(\lfloor {\rm u}_0^s\rceil^{\l})^{(1)}_{0}}\\
&\ \ \ \ \ \ -\int_{0}^{T}\big([\wt{(DF_{0, t}^s)}(\mho_{0}^s, {\rm w})]^{-1}\big)'_s\left(\varpi\big((H^{\l})^{(1)}_0(\mho_{t}^s, e_i)\big)e_i\ dt,\ \varpi\big((H^{\l})^{(1)}_0(\mho_{t}^s,  d\overleftarrow\alpha^{s}_{t}({\rm w}))\big)\right)\\
&\ \ \ \ \ \ -\int_{0}^{T}[\wt{(DF_{0, t}^s)}(\mho_{0}^s, {\rm w})]^{-1}\left(\varpi\big((H^{\l})^{(1)}_0(\mho_{t}^s, e_i)\big)'_s e_i\ dt,\right.\\
&\ \ \ \ \ \ \ \ \ \ \ \ \ \  \ \ \ \ \ \ \ \  \ \ \ \ \ \ \ \ \  \ \ \ \ \ \ \ \ \ \left. \varpi\big((H^{\l})^{(1)}_0(\mho_{t}^s,  \cdot)\big)'_s d\overleftarrow\alpha^{s}_{t}({\rm w})+\varpi\big((H^{\l})^{(1)}_0(\mho_{t}^s,  d\underline{\Upsilon}_{{\rm V}, {\alpha}^s})\big)\right)\\
&\ \  =:\  {\rm I}(s)+{\rm II}(s)+{\rm III}(s).\end{align*}
This process has a continuous version in $s$ by  Kolmogorov's criterion (or by continuity of $\alpha^s$, $\mho^s$, $\underline{\Upsilon}_{{\rm V}, {\alpha}^s}, (\mho_t^s)'_s$ and  $[\wt{(DF_{0, t}^s)}(\mho_{0}^s, {\rm w})]^{-1}$ in $s$ using Theorem \ref{Main-alpha-x-v-Q}).

For (\ref{dpi-UFS-diff}),  we do the corresponding conditional  estimations for ${\rm I}(s), {\rm II}(s)$ and ${\rm III}(s)$. 
Clearly,  \begin{align*}
\overline{\E}_{\ov{\P}^*_{y, x, T}}\left|{\rm I}(s)\right|^q\leq &\  \overline{\E}_{\ov{\P}^*_{y, x, T}} \sup_{t\in [0, T]}\left\|([\wt{(DF_{0, t}^s)}(\mho_{0}^s, {\rm w})]^{-1})'_s\right\|^q \cdot \big\|(\lfloor {\rm u}_0\rceil^{\l})^{(1)}_{0}\big\|^q,\end{align*}
which,  by (\ref{DF-s-diff-b}), has a bound as in (\ref{dpi-UFS-diff}). Put \begin{align*}
{\rm II}_1(s):= &  -\int_{0}^{T}\big([\wt{(DF_{0, t}^s)}(\mho_{0}^s, {\rm w})]^{-1}\big)'_s\left(\varpi\big((H^{\l})^{(1)}_0(\mho_{t}^s, e_i)\big)e_i\ dt, \varpi\big((H^{\l})^{(1)}_0(\mho_{t}^s,  \mathtt{g}^{s}_{t}({\rm w})dt)\big)\right), \\
{\rm II}_2(s):= &  -\int_{0}^{T}\big([\wt{(DF_{0, t}^s)}(\mho_{0}^s, {\rm w})]^{-1}\big)'_s\left(0, \varpi\big((H^{\l})^{(1)}_0(\mho_{t}^s, {O}^{s}_{t}d\overleftarrow{B}_t(\rm w))\big)\right). 
\end{align*}
For ${\rm II}(s)$, we have 
\[
\overline{\E}_{\ov{\P}^*_{y, x, T}}\left|{\rm II}(s)\right|^q\leq 2^{q-1}\left(\overline{\E}_{\ov{\P}^*_{y, x, T}}\left|{\rm II}_1(s)\right|^q+\overline{\E}_{\ov{\P}^*_{y, x, T}}\left|{\rm II}_2(s)\right|^q\right). 
\]
As before, we can use   H\"{o}lder's inequality,  Doob's inequality of submartingales and Burkholder's  inequality to obtain   some $C(q, T)$ depending on $s_0, m, q,  \mathtt{s},  T$, $\|g^{0}\|_{C^3}$ and  $\|\XX^{0}\|_{C^2}$ such that 
\begin{align*}
\overline{\E}_{\ov{\P}^*_{y, x, T}} \left|{\rm II}(s)\right|^q\leq & C(q, T)T_0^{-m}\left(\overline{\E}_{\ov{\P}^*_{y, x, T}} \sup_{t\in [0, T]}\left\|([\wt{(DF_{0, t}^s)}(\mho_{0}^s, {\rm w})]^{-1})'_s\right\|^q+\right.\\
&\left.\big(\overline{\E}_{\ov{\P}^*_{y, x, T}} \sup_{t\in [0, T]}\left\|([\wt{(DF_{0, t}^s)}(\mho_{0}^s, {\rm w})]^{-1})'_s\right\|^{2q} \overline{\E}_{\ov{\P}^*_{y, x, T}} e^{\{2q\int_{0}^{T}\|\nabla\ln p(T-\tau, {\rm y}_{\tau}^0, x)\|\ d\tau\}}\big)^{\frac{1}{2}}\!\right),
\end{align*}
which has a bound as in (\ref{dpi-UFS-diff}) by  Lemma \ref{cond-DF-diff-norm} and Proposition \ref{cond-nabla-ln-p}.  The same argument applies to  ${\rm III}(s)$ and we obtain  some $C'(q, T)$ depending on $s_0, m, q,  \mathtt{s},  T$, $\|g^{0}\|_{C^3}$ and  $\|\XX^{0}\|_{C^2}$ such that  \begin{align*}
\overline{\E}_{\ov{\P}^*_{y, x, T}} \left|{\rm III}(s)\right|^q\leq & C'(q, T)T_0^{-m}\big(\overline{\E}_{\ov{\P}^*_{y, x, T}} \sup_{t\in [0, T]}\left\|[\wt{(DF_{0, t}^s)}(\mho_{0}^s, {\rm w})]^{-1}\right\|^{2q}\big)^{\frac{1}{2}}\cdot\\
&\left\{1+\big(\overline{\E}_{\ov{\P}^*_{y, x, T}} e^{\{2q\int_{0}^{T}\|\nabla\ln p(T-\tau, {\rm y}_{\tau}^0, x)\|\ d\tau\}}\big)^{\frac{1}{2}}\cdot \right.\\
& \left(\big(\overline{\E}_{\ov{\P}^*_{y, x, T}} \int_{0}^{T}\left|\big(\varpi\big((H^{\l})^{(1)}_0(\mho_{t}^s, e_i)\big)e_i\big)'_s\right|^{2q} dt \big)^{\frac{1}{2}}+\big(\overline{\E}_{\ov{\P}^*_{y, x, T}}  \!\int_{0}^{T}\left|K_{{\rm V}, {\alpha}^s}\right|^{2q}  dt\big)^{\frac{1}{2}}\right.\\
&\left.\left.+\big(\overline{\E}_{\ov{\P}^*_{y, x, T}}\int_{0}^{T}\!\left|\big(\varpi(H^{\l})^{(1)}_0(\mho_{t}^s, \cdot)\big)'_s{\mathtt g}_t^s\right|^{2q} dt\big)^{\frac{1}{2}}\!\right)\right\}, 
\end{align*}
which also has a bound as in (\ref{dpi-UFS-diff}) by  Lemma \ref{O-mathttg-diff-esti}, Lemma \ref{cond-DF-norm} and  Proposition \ref{cond-nabla-ln-p}. 
\end{proof}

 We can define $\big(D\pi(\lfloor {\rm u}_T\rceil^{\l})^{(1)}_\l\big)\circ \lf {\bf F}^s\rc^{\l}$ for all $\l$.  Let ${\rm V}, $ $F^s$ and   ${\mathtt s}$ be  as in Section \ref{flow-F-S-y}. For $y\in \M$, let $(\lf {\rm y}_t\rc^{\l}({\rm w}), \lf\mho_t\rc^{\l}({\rm w}))_{t\in [0, T]}$  be the stochastic  pair in $(\M, \mathcal{O}^{\wt{g}^{\l}}(\M))$ which defines the $\wt{g}^{\l}$-Brownian motion on $\M$ starting from $y$. Following Theorem \ref{Main-alpha-x-v-Q}, we  can extend the  map $F^s$  on $y$ to be a  map $\lf {\bf F}^s_y\rc^{\l}$ on  paths $(\lf {\rm y}_t\rc^{\l}({\rm w}))_{t\in [0, T]}$ so that 
 \[\lf {\rm y}^s_t\rc^{\l}({\rm w}):=\left(\lf {\bf F}^s_y\rc^{\l}(\lf{\rm y}\rc^{\l}_{[0, T]}({\rm w}))\right)(t), \ \forall  t\in [0, T], \] 
 and its horizontal lift $\big(\lf\mho_{t}^s\rc^{\l}({\rm w})\big)_{t\in [0, T]}$ with $(\lf\mho_{0}^s\rc^{\l})'_s=0$ 
are such that 
\[
\frac{d}{ds}\big( \lf{\rm y}^s_t\rc^{\l}({\rm w})\big)=\Upsilon_{{\rm V}, \lf {\rm y}^s\rc^{\l}}(t)=\mathtt{s}(t)\lf\mho_{t}^s\rc^{\l}(\lf\mho_0^s\rc^{\l})^{-1}{\rm V}(F^s(y)). 
\]
Accordingly, we denote by $\lf\alpha_t^s\rc^{\l}$ the anti-development of $\lf {\rm y}_t^s\rc^{\l}$ and let  $\big(\lf F_{\underline{t}, \overline{t}}^s\rc^{\l}\big)_{0<\underline{t}<\overline{t}<T}$ be the stochastic flow map corresponding to  the SDE
\[
d\overline{\beta}_t=H^{\l}(\overline{\beta}_t, \circ d\lf \alpha_{t}^s\rc^{\l}({\rm w}))
\]
with tangent maps  $(D\lf F_{\underline{t}, \overline{t}}^s\rc^{\l})_{0<\underline{t}<\overline{t}<T}$. We will omit the upper-script at $s=0$.  For $x\in \M$, we define $\lf {\bf F}^s\rc^{\l}$ on  $C_x([0, T], \M)$ conditioned on the value of $\beta_T$, i.e., 
\[
\lf {\bf F}^s\rc^{\l}(\beta):=\lf{\bf F}^{s}_{\beta_T}\rc^{\l}(\beta), \ \forall \beta\in C_x([0, T], \M). 
\]
Let $(\lfloor {\rm x}_t \rceil^{\l}(w),  \lfloor {\rm u}_t \rceil^{\l}(w))_{t\in [0, T]}$  be the stochastic  pair in $(\M, \mathcal{O}^{\wt{g}^{\l}}(\M))$ which defines the $\wt{g}^{\l}$-Brownian motion on $\M$ starting from $x$.  The  correspondence rule in Corollary \ref{Cor-fore-back} shows that conditioned on $\lf{\rm x}_T\rc^{\l}=y$, the distribution of $(\lfloor {\rm u}_T\rceil^{\l})^{(1)}_{\l}(w)$  is the same as,  conditioned on $\lf {\rm y}_T\rc^{\l}=x$, the distribution of 
 \begin{align*}
\overline{(\lfloor {\rm u}_T\rceil^{\l})^{(1)}_\l}({\rm w}):=&[D\lf F_{0, T}\rc^{\l}(\lf{\mho}_0\rc^{\l}, {\rm w})]^{-1}(\lfloor {\rm u}_0\rceil^{\l})^{(1)}_{\l} \\
&- \int_{0}^{T}[D\lf F_{0, t}\rc^{\l}(\lf {\mho}_0\rc^{\l}, {\rm w})]^{-1}(H^{\l})^{(1)}_{\l}\big(\lf\mho_{t}\rc^{\l}, \circ d{B}_{t}({\rm w})\big),
\end{align*}
where  $\circ d\overrightarrow{B}_{t}({\rm w})$ is the backward Stratonovich infinitesimal. Then  we define \[
(D\pi(\lfloor {\rm u}_T\rceil^{\l})^{(1)}_\l)\circ \lf{\bf F}^s\rc^{\l}:=D\pi \big((\lfloor {\rm u}_T\rceil^{\l})^{(1)}_{\l}\circ \lf {\bf F}^s\rc^{\l}\big)=D\pi\big(\overline{(\lfloor {\rm u}_T^s\rceil^{\l})^{(1)}_\l}(\rm w)\big),
\]
where 
\begin{align}
\overline{(\lfloor {\rm u}_T^s\rceil^{\l})^{(1)}_{\l}}({\rm w})=&\; [D\lf F_{0, T}^s\rc^{\l}(\lf {\mho}_0^s\rc^{\l}, {\rm w})]^{-1}(\lfloor {\rm u}_0\rceil^{\l})^{(1)}_{\l}\notag\\
&-\int_{0}^{T}[D\lf F_{0, t}^s\rc^{\l}(\lf {\mho}_0^s\rc^{\l}, {\rm w})]^{-1}(H^{\l})^{(1)}_{\l}\big(\lf \mho_{t}^s\rc^{\l}, \circ d\lf \alpha^{s}_{t}\rc^{\l}({\rm w})\big). \label{u-T-s-la-1}
\end{align}
The proof of  Proposition \ref{Diff-int-U-s} works for $\lf{\bf F}^s\rc^{\l}$, which gives the following. 

\begin{prop}\label{Diff-int-U-s-gen} For each $\l$,  the one parameter family  of processes $\{(D\pi(\lfloor {\rm u}_T\rceil^{\l})^{(1)}_\l)\circ  \lf{\bf F}^s\rc^{\l}\}$ is differentiable in $s$. Moreover, $(D\pi(\lfloor {\rm u}_T\rceil^{\l})^{(1)}_\l)\circ  \lf{\bf F}^s\rc^{\l}$ and the differential stochastic process
\[
\nabla_{T, {\rm V}, \mathtt{s}}^{s, \l}D\pi(\lfloor {\rm u}_T\rceil^{\l})^{(1)}_{\l}:=\left((D\pi(\lfloor {\rm u}_T\rceil^{\l})^{(1)}_{\l})\circ  \lf{\bf F}^s\rc^{\l}\right)'_{s}
\]
conditioned on $\lf {\rm x}_T\rc^{\l}=y$ are  $L^q$ ($q\geq 1$) integrable, locally uniformly in the $s$ parameter. 
\end{prop}

For later use, we list and reformulate some differentials related to $\nabla_{T, {\rm V}, \mathtt{s}}D\pi(\lfloor {\rm u}_T\rceil^{\l})^{(1)}_{\l}$.  The upper-scripts $\l$ in  $\nabla^{\l}$, $R^{\l}$, ${\rm Ric}^{\l}$, $\theta^{\l}, \varpi^{\l}$ and  $(\theta, \varpi)^{\l}$ are to  indicate the metric $\wt{g}^{\l}$ used.  

\begin{lem}\label{all-diff-at-s-0}
We have the following  for almost all ${\rm w}$ and  for all $t\in [0, T]$. 
\begin{itemize}
\item[i)] 
\begin{align*}
&(\lf \alpha_t^s\rc^{\l})'_0=\underline{\Upsilon}_{{\rm V}, B}^{\l}(t)\\
& :=\int_{0}^{t}\left((\lf \mho_0\rc^{\l})^{-1}({\mathtt s}'(\tau){\rm V}(y))-{\rm Ric}^{\l}\big(\Upsilon_{{\rm V}, \lf {\rm y}\rc^{\l}}(\tau)\big)\right)\ d\tau-\int_{0}^{t}\langle K_{{\rm V}, B}^{\l}(\tau),  dB_\tau\rangle,
\end{align*}
where  $\Upsilon_{{\rm V}, \lf {\rm y}\rc^{\l}}(\tau):={\mathtt s}(\tau)\lf \mho_{\tau}\rc^{\l}(\lf \mho_0\rc^{\l})^{-1}{\rm V}(y)$ and  
\begin{align*}K_{{\rm V}, B}^{\l}(\tau)=&\int_{0}^{\tau}(\lf \mho_{\wt{\tau}}\rc^{\l})^{-1}R^{\l}\left(\lf {\mho_{\wt{\tau}}}\rc^{\l}dB_{\wt{\tau}}, \Upsilon_{{\rm V}, \lf {\rm y}\rc^{\l}}(\wt{\tau})\right)\lf{\mho_{\wt{\tau}}}\rc^{\l}\\
&+\int_{0}^{\tau} (\lf \mho_{\wt{\tau}}\rc)^{-1}(\nabla^{\l} (\lf \mho_{\wt{\tau}} \rc^{\l}e_i)R^{\l})\left(\lf \mho_{\wt{\tau}}\rc^{\l}e_i, \Upsilon_{{\rm V}, \lf {\rm y}\rc^{\l}}(\wt{\tau})\right)\lf \mho_{\wt{\tau}}\rc^{\l}\ d\wt{\tau}\notag.
\end{align*}
\item[ii)]
\[
(\theta, \varpi)^{\l}_{\lf \mho_t\rc^{\l}}(\lf \mho_t^s\rc^{\l})'_0=\big({\mathtt s}(t)(\lf \mho_0\rc^{\l})^{-1}{\rm V}(y), K_{{\rm V}, B}^{\l}(t)\big).
\]
\item[iii)]For $s\mapsto {\bf\mathsf v}_t^s\in HT_{\lf \mho_t^s\rc^{\l}}\mathcal{F}(\M)$,  let  $
{\bf\mathsf v}_\tau^s:=[D\lf F_{\tau, t}^s\rc^{\l} (\lf \mho_{\tau}^s\rc^{\l}, {\rm w})]^{-1}{\bf\mathsf v}_t^s, \tau\in [0, t]$.  
Then \begin{align*}
\big({\bf\mathsf v}_0^s\big)'_0=&\int_{0}^{t}\big[D\lf F_{0, \tau}\rc^{\l}\big(\lf \mho_{0}\rc^{\l}, {\rm w}\big)\big]^{-1}\left(\circledast^{\l}\big({\bf\mathsf v}_\tau, (\lf \mho_\tau^s\rc^{\l})'_0\big)\right),
\end{align*}
where 
\begin{align}
\circledast^{\l}\big({\bf\mathsf v}_\tau, (\lf \mho_\tau^s\rc^{\l})'_0\big)=&
\nabla^{(2), \l}\big({\bf\mathsf v}_\tau, (\lf \mho_\tau^s\rc^{\l})'_0\big)H^{\l}(\lf \mho_\tau\rc^{\l}, \circ d B_\tau)+\nabla^{\l}({\bf\mathsf v}_\tau)H^{\l}\big(\lf {\mho_\tau}\rc^{\l}, \circ d\Upsilon_{{\rm V}, B}^{\l}(\tau)\big)\notag\\
&+R^{\l}\big(H^{\l}(\lf \mho_{\tau}\rc^{\l}, \circ dB_{\tau}), (\lf \mho_\tau^s\rc^{\l})'_0\big){\bf\mathsf v}_\tau. \label{circle-v-mho-s}
\end{align}
The It\^{o} form of $\big({\bf\mathsf v}_0^s\big)'_0$ in $(\theta, \varpi)$-chart is \begin{align*}
(\theta, \varpi)^{\l}_{\lf \mho_0\rc^{\l}}\big({\bf\mathsf v}_0^s\big)'_0=&\int_{0}^{t}\left[\wt{D\lf F_{0, \tau}\rc^{\l}}\big(\lf \mho_{0}\rc^{\l}, {\rm w}\big)\right]^{-1}\left\{(\theta, \varpi)^{\l}\circledast_{{\rm I}}^{\l}({\bf\mathsf v}_\tau, (\lf \mho_\tau^s\rc^{\l})'_0)\right.\\
&\ \ \ \ \ \ \ \ \ \ \ \ \ \ \ \ \ \ \ \ \ \ \ \ \ \ \ \ \ \ \ \ \ \left.+\big(\wt{\circledast}_{{\rm A}}^{\theta, \l}({\bf\mathsf v}_\tau, (\lf \mho_\tau^s\rc^{\l})'_0), \wt{\circledast}_{{\rm A}}^{\varpi, \l}({\bf\mathsf v}_\tau, (\lf \mho_\tau^s\rc^{\l})'_0)\big)\right\},
\end{align*}
where $\circledast_{{\rm I}}^{\l}({\bf\mathsf v}_\tau, (\lf \mho_\tau^s\rc^{\l})'_0)$ is $\circledast^{\l}({\bf\mathsf v}_\tau, (\lf \mho_\tau^s\rc^{\l})'_0)$ with the Stratonovich infinitesimals  $\circ dB_{\tau}$, $\circ d\Upsilon_{{\rm V}, B}^{\l}(\tau)$ replaced by the It\^{o} infinitesimals  $dB_{\tau}, d\Upsilon_{{\rm V}, B}^{\l}(\tau)$, 
\begin{align*}
&\wt{\circledast}_{{\rm A}}^{\theta, \l}\big({\bf\mathsf v}_\tau, (\mho_\tau^s)'_0\big)\!=\varpi^{\l}\!\left(\circledast\big({\bf\mathsf v}_\tau, (\lf \mho_\tau^s\rc^{\l})'_0, e_i\big)\right)e_i \ d\tau\!+\theta^{\l}\left(\left[H(\mho_\tau, e_i), \circledast^{\l}\big({\bf\mathsf v}_\tau, (\lf\mho_\tau^s\rc^{\l})'_0, e_i\big)\right]\right) d\tau,\\
&\wt{\circledast}_{{\rm A}}^{\varpi, \l}\big({\bf\mathsf v}_\tau, (\lf \mho_\tau^s\rc^{\l})'_0\big)=(\lf \mho_{\tau}\rc^{\l})^{-1} R^{\l} \left(\lf \mho_{\tau}\rc^{\l}e_i, \lf \mho_{\tau}\rc^{\l}\theta^{\l}(\circledast^{\l}({\bf\mathsf v}_\tau, (\lf \mho_\tau^s\rc^{\l})'_0, e_i))\right)\lf \mho_{\tau}\rc^{\l} d\tau\\
&\ \ \ \ \ \ \ \ \ \ \ \ \ \ \ \ \ \ \ \ \ \ \ \ \ \ +\varpi^{\l}\left(\left[H^{\l}(\lf \mho_\tau\rc^{\l}, e_i), \circledast^{\l}\big({\bf\mathsf v}_\tau, (\lf \mho_{\tau}^s\rc^{\l})'_0, e_i\big)\right]\right)\ d\tau, \ \mbox{and}\\
&\circledast^{\l}\big({\bf\mathsf v}_\tau, (\lf \mho_\tau^s\rc^{\l})'_0, e_i\big)=\nabla^{(2), \l}\big({\bf\mathsf v}_\tau, (\lf \mho_{\tau}^s\rc^{\l})'_0\big)H^{\l}(\lf \mho_{\tau}\rc^{\l}, e_i)+R^{\l}\big(H^{\l}(\lf \mho_{\tau}\rc^{\l}, e_i), (\lf \mho_{\tau}^s\rc^{\l})'_0\big){\bf\mathsf v}_{\tau}\\
&\ \ \ \ \ \ \ \ \ \ \ \ \ \ \ \ \  \ \ \ \ \ \ \ \ \ \ \ +\nabla^{\l}({\bf\mathsf v}_{\tau})H^{\l}\big(\lf {\mho_{\tau}}\rc^{\l}, K_{{\rm V}, B}^{\l}e_i\big).
\end{align*}
\end{itemize}
\end{lem}
\begin{proof}Without loss of generality, we can consider the case $\l=0$. The {\rm i)}, {\rm ii)} are straight forward consequences of Theorem \ref{Main-alpha-x-v-Q} reporting $\alpha^0=B$ in the formulas in Lemma  \ref{cor-mho-t-s-d} and Corollary \ref{gamma-V-apha}.  For {\rm iii)},  a comparison of the  SDEs (\ref{mathsfv-t-s-d-stoch}), (\ref{mathsfv-t-s-d-Ito}) in Lemma \ref{diff-D-F-s-crit} with that of the tangent maps $[DF_{0, \tau}]^{-1}$, $[\wt{DF_{0, \tau}}]^{-1}$ shows that we can use Duhamel's principle to formulate $({\bf\mathsf v}_0^s)'_0$, $(\theta, \varpi)_{\mho_0}({\bf\mathsf v}_0^s)'_0$ as above. \end{proof}

\begin{prop}\label{u-s-t-l-1} With all the notations as above,  then 
\begin{align}
\notag\left(\overline{(\lfloor {\rm u}_T^s\rceil^{\l})^{(1)}_{\l}}({\rm w})\right)'_0\! =&\int_{0}^{T}[D\lf F_{0, \tau}\rc^{\l}(\lf \mho_0\rc^{\l}, {\rm w})]^{-1}\left\{\circledast^{\l}\big((\lf \mho_{\tau}\rc^{\l})^{(1)}_{\l}, (\lf \mho_\tau^s\rc^{\l})'_0\big)\right.\\
&\left.-\big(\nabla^{\l}((\lf \mho_\tau^s\rc^{\l})'_0)(H^{\l})^{(1)}_{\l}(\lf \mho_{\tau}\rc^{\l}, \cdot)\big)\circ dB_{\tau}-\! (H^{\l})^{(1)}_{\l}\big(\lf \mho_{\tau}\rc^{\l}, \circ d\underline{\Upsilon}_{{\rm V}, B}^{\l}(\tau)\big)\right\}, \label{the U-T-s-1}
\end{align}
where $\circledast^{\l}\big((\lf \mho_{\tau}\rc^{\l})^{(1)}_{\l}, (\mho_\tau^s)'_0\big)$ is as in (\ref{circle-v-mho-s}) replacing ${\bf\mathsf v}_\tau$ by $(\lf \mho_{\tau}\rc^{\l})^{(1)}_{\l}$ and 
\begin{align*}
(\lf \mho_{\tau}\rc^{\l})^{(1)}_{\l}({\rm w}):=&\big[D\lf F_{\tau, T}\rc^{\l}(\lf {\mho}_0\rc^{\l}, {\rm w})\big]^{-1}(\lfloor {\rm u}_0\rceil^{\l})^{(1)}_{\l}\\
&-\int_{\tau}^{T}\big[D\lf F_{\tau, t}\rc^{\l}(\lf {\mho}_\tau\rc^{\l}, {\rm w})\big]^{-1}(H^{\l})^{(1)}_{\l}\big(\lf \mho_{t}\rc^{\l},\circ d B_{t}({\rm w})\big).
\end{align*}
In  $(\theta, \varpi)^{\l}$-chart, we have the It\^{o} integral expression  
\begin{align*}
&\left((\theta, \omega)^{\l}_{\mho_{0}^s}\big(\overline{(\lfloor {\rm u}_T^s\rceil^{\l})^{(1)}_0}({\rm w})\big)\right)'_0\\
&\  =\int_{0}^{T}\left[\wt{D\lf F_{0, \tau}}\rc^{\l}\big(\lf \mho_{0}\rc^{\l}, {\rm w}\big)\right]^{-1}\!\left\{(\theta, \varpi)^{\l}\circledast_{{\rm I}}^{\l}\big((\lf \mho_{\tau}\rc^{\l})^{(1)}_{\l}, (\lf \mho_\tau^s\rc^{\l})'_0\big)\right.\\
&\ \ \ \ +\left(\wt{\circledast}_{{\rm A}}^{\theta, \l}((\lf \mho_{\tau}\rc^{\l})^{(1)}_{\l}, (\lf \mho_\tau^s\rc^{\l})'_0), \wt{\circledast}_{{\rm A}}^{\varpi, \l}((\lf \mho_{\tau}\rc^{\l})^{(1)}_{\l}, (\lf \mho_\tau^s\rc^{\l})'_0)\right)\\
&\ \ \ \ +\left(\varpi^{\l}\big(\nabla^{\l}((\lf \mho_\tau^s\rc^{\l})'_0)(H^{\l})^{(1)}_{\l}(\lf \mho_{\tau}\rc^{\l}, e_i)\big)e_i\ d\tau, \right.\\
&\ \ \ \ \ \ \ \ \left.\left. \varpi^{\l}\big(\nabla^{\l}((\lf \mho_\tau^s\rc^{\l})'_0)(H^{\l})^{(1)}_{\l}(\lf \mho_{\tau}\rc^{\l}, dB_\tau)\big)+ \varpi^{\l}\big((H^{\l})^{(1)}_{\l}(\lf \mho_{\tau}\rc^{\l},  d\underline{\Upsilon}_{{\rm V}, B}^{\l}(\tau))\big)\right)\right\}.
\end{align*}
\end{prop}
\begin{proof}Differentiating (\ref{u-T-s-la-1}), we obtain 
\begin{align*}
\left(\overline{(\lfloor {\rm u}_T^s\rceil^{\l})^{(1)}_{\l}}({\rm w})\right)'_0=&\ \big([D\lf F_{0, T}^s\rc^{\l}(\lf {\mho}_0^s\rc^{\l}, {\rm w})]^{-1}\big)'_0(\lfloor {\rm u}_0\rceil^{\l})^{(1)}_{\l}\notag\\
&\ -\int_{0}^{T}\big([D\lf F_{0, t}^s\rc^{\l}(\lf {\mho}_0^s\rc^{\l}, {\rm w})]^{-1}\big)'_0(H^{\l})^{(1)}_{\l}\big(\lf \mho_{t}\rc^{\l}, \circ d B_{t}({\rm w})\big)\\
&\ -\int_{0}^{T}[D\lf F_{0, t}\rc^{\l}(\lf \mho_0\rc^{\l}, {\rm w})]^{-1}\big(\nabla^{\l}((\lf \mho_t^s\rc^{\l})'_0)(H^{\l})^{(1)}_{\l}(\lf \mho_{t}\rc^{\l}, \cdot)\big)\circ dB_{t}({\rm w})\\
&\ -\int_{0}^{T}[D\lf F_{0, t}\rc^{\l}(\lf \mho_0\rc^{\l}, {\rm w})]^{-1}(H^{\l})^{(1)}_{\l}\big(\lf \mho_{t}\rc^{\l}, \circ d\Upsilon_{{\rm V}, B}^{\l}(t, {\rm w})\big)\\
=:&\  {(\rm I)}+{(\rm II)}+{(\rm III)}+{(\rm IV)}. 
\end{align*}
By {\rm iii)} of Lemma \ref{all-diff-at-s-0}, we have
\begin{align*}
{(\rm I)}=&\int_{0}^{T}\big[D\lf F_{0, \tau}\rc^{\l}(\lf \mho_0\rc^{\l}, {\rm w})\big]^{-1}\circledast^{\l}\left(\big[D\lf F_{\tau, T}\rc^{\l}(\lf {\mho}_0\rc^{\l}, {\rm w})\big]^{-1}(\lfloor {\rm u}_0\rceil^{\l})^{(1)}_{\l}, (\lf \mho_\tau^s\rc^{\l})'_0\right).
\end{align*}
For $0\leq\tau<t\leq T$, put
\begin{align*}
{\bf\mathsf v}_{\tau, t}:=\big[D\lf F_{\tau, t}\rc^{\l}(\lf {\mho}_\tau\rc^{\l}, {\rm w})\big]^{-1}(H^{\l})^{(1)}_{\l}\big(\lf \mho_{t}\rc^{\l},\circ d B_{t}({\rm w})\big). 
\end{align*}
We continue to compute that 
\begin{align*}
{(\rm II)}=&-\int_{0}^{T}\int_{0}^{t}\big[D\lf F_{0, \tau}\rc^{\l}(\lf \mho_0\rc^{\l}, {\rm w})\big]^{-1}\circledast^{\l}\big({\bf\mathsf v}_{\tau, t},  (\lf \mho_\tau^s\rc^{\l})'_0\big)\\
=&-\int_{0}^{T}\big[D\lf F_{0, \tau}\rc^{\l}(\lf \mho_0\rc^{\l}, {\rm w})\big]^{-1}\circledast^{\l}\big(\int_{\tau}^{T} {\bf\mathsf v}_{\tau, t},  (\lf \mho_\tau^s\rc^{\l})'_0\big).
\end{align*}
Altogether, we obtain 
\[
{(\rm I)}+{(\rm II)}=\int_{0}^{T}[D\lf F_{0, \tau}\rc^{\l}(\lf \mho_0\rc^{\l}, {\rm w})]^{-1}\circledast^{\l}\big((\lf \mho_{\tau}\rc^{\l})^{(1)}_{\l}, (\lf \mho_\tau^s\rc^{\l})'_0\big).
\]
Hence (\ref{the U-T-s-1}) holds true. 
The It\^{o} form integral expression   of $\big((\theta, \omega)_{\mho_{0}^s}\big(\overline{(\lfloor {\rm u}_T^s\rceil^{\l})^{(1)}_0}({\rm w})\big)\big)'_0$ can be obtained using the It\^{o} form in {\rm iii)} of Lemma \ref{all-diff-at-s-0}. 
\end{proof}

As a corollary of Proposition  \ref{u-s-t-l-1}, we  can further express the differential 
\[
((\lfloor {\rm u}_T\rceil^{\l})^{(1)}_{\l}\circ \lf {\bf F}^s\rc^{\l})'_0=:\left((\lfloor {\rm u}_T^s\rceil^{\l})^{(1)}_{\l}(w)\right)'_0
\]
using   $(\lf {\rm u}_t\rc^{\l}(w))_{t\in [0, T]}$ and the tangent maps $\{[D\lf \overrightarrow{F}_{\underline{t}, \overline{t}}\rc^{\l}(\lf{\rm u}_{\underline{t}}\rc^{\l}, w)]\}_{0\leq \underline{t}<\overline{t}\leq T}$ of the flow maps $\{\lf \overrightarrow{F}_{\underline{t}, \overline{t}}\rc^{\l}(\lf{\rm u}_{\underline{t}}\rc^{\l}, w)\}_{0\leq \underline{t}<\overline{t}\leq T}$ associated to  (\ref{FM-BM-SDE-11}).  We only give the  Stratonovich form.  Let 
\begin{align*}&K_{{\rm V}, B}^{\l, {\rm u}}(t, w):=\\
&\  \int_{0}^{t}(\lf {\rm u}_{T-\wt{\tau}}\rc^{\l})^{-1}R^{\l}\left((\lf {\rm u}_{T-\wt{\tau}}\rc^{\l})^{-1}dB_{T-\wt{\tau}}(w), \lf {\rm u}_{T-\wt{\tau}}\rc^{\l}(\lf {\rm u}_{T}\rc^{\l})^{-1}[\mathtt{s}(\wt{\tau}){\rm V}(y)]\right)\lf {\rm u}_{T-\wt{\tau}}\rc^{\l}\\
&\  +\int_{0}^{t} (\lf {\rm u}_{T-\wt{\tau}}\rc^{\l})^{-1}(\nabla^{\l} (\lf {\rm u}_{T-\wt{\tau}} \rc^{\l}e_i)R^{\l})\left(\lf {\rm u}_{T-\wt{\tau}}\rc^{\l}e_i, \lf {\rm u}_{T-\wt{\tau}}\rc^{\l}(\lf {\rm u}_T\rc^{\l})^{-1}[\mathtt{s}(\wt{\tau}){\rm V}(y)]\right)\lf {\rm u}_{T-\wt{\tau}}\rc^{\l}\ d\wt{\tau}\notag, \\
&\underline{\Upsilon}_{{\rm V}, B}^{\l, {\rm u}}(\tau):=\\
&\int_{0}^{T-\tau}\!\!\!\!\big((\lf {\rm u}_T\rc^{\l})^{-1}({\mathtt s}'(t){\rm V}(y))-{\rm Ric}^{\l}\big(\lf {\rm u}_{T-t}\rc^{\l}(\lf {\rm u}_0\rc^{\l})^{-1}[{\mathtt s}(t){\rm V}(y)]\big)\big) dt-\int_{0}^{T-\tau}\!\!\langle K_{{\rm V}, B}^{\l, {\rm u}}(t, w),  dB_t\rangle. 
\end{align*}
Then $\underline{\Upsilon}_{{\rm V}, B}^{\l, {\rm u}}(\tau)$ corresponds to $\underline{\Upsilon}_{{\rm V}, B}^{\l}(T-\tau)$ and they have the same distribution. Put 
\begin{align*}
(\lf {\rm u}_{\tau, T}^s\rc^{\l})'_0:=(\theta, \varpi)_{\lf {\rm u}_\tau\rc^{\l}}^{-1}\left({\mathtt s}(T-\tau)(\lf {\rm u}_T\rc^{\l})^{-1}{\rm V}(y), K_{{\rm V}, B}^{\l, {\rm u}}(T-\tau)\right). 
\end{align*}
\begin{cor}\label{cor-u-s-t-l-1}With all the notations as above, \begin{align*}
&\left((\lfloor {\rm u}_T^s\rceil^{\l})^{(1)}_{\l}(w)\right)'_0\\
&\ \ \ =\int_{0}^{T}\big[D\lf \overrightarrow{F}_{\tau, T}\rc^{\l}(\lf {\rm u}_{\tau}\rc^{\l}, {w})\big]\left\{\circledast^{\l, {\rm u}}\big((\lf {\rm u}_{\tau}\rc^{\l})^{(1)}_{\l}, (\lf {\rm u}_{\tau, T}^s\rc^{\l})'_0\big)\right.\\
&\ \ \ \ \ \ \ \left.+\big(\nabla^{\l}((\lf {\rm u}_{\tau, T}^s\rc^{\l})'_0)(H^{\l})^{(1)}_{\l}(\lf {\rm u}_{\tau}\rc^{\l}, \cdot)\big)\circ dB_{\tau}(w)+(H^{\l})^{(1)}_{\l}\big(\lf {\rm u}_{\tau}\rc^{\l}, \circ d\Upsilon_{{\rm V}, B}^{\l, {\rm u}}(\tau)\big)\right\},
\end{align*}
where 
\begin{align*}
\circledast^{\l, {\rm u}}\big((\lf {\rm u}_{\tau}\rc^{\l})^{(1)}_{\l}, (\lf {\rm u}_{\tau, T}^s\rc^{\l})'_0\big)=&-\nabla^{\l, (2)}\big((\lf {\rm u}_{\tau}\rc^{\l})^{(1)}_{\l}, (\lf {\rm u}_{\tau, T}^s\rc^{\l})'_0\big)H^{\l}\big(\lf {\rm u}_\tau\rc^{\l}, \circ d B_\tau(w)\big)\\
&-\nabla^{\l}((\lf {\rm u}_{\tau}\rc^{\l})^{(1)}_{\l})H^{\l}\big(\lf {{\rm u}_\tau}\rc^{\l}, \circ d\underline{\Upsilon}_{{\rm V}, B}^{\l, {\rm u}}(\tau)\big)\\
&-R^{\l}\big(H^{\l}(\lf {\rm u}_{\tau}\rc^{\l}, \circ dB_{\tau}({w})), (\lf {\rm u}_{\tau, T}^s\rc^{\l})'_0\big)(\lf {\rm u}_{\tau}\rc^{\l})^{(1)}_{\l}.
\end{align*}
In  $(\theta, \varpi)^{\l}$-chart, we have the It\^{o} integral expression  
\begin{align*}
&\left((\theta, \varpi)(\lfloor {\rm u}_T^s\rceil^{\l})^{(1)}_{\l}(w)\right)'_0\\
&\  =\int_{0}^{T}\big[D\wt{\lf \overrightarrow{F}_{\tau, T}\rc^{\l}}(\lf {\rm u}_{\tau}\rc^{\l}, {w})\big]\!\left\{(\theta, \varpi)^{\l}\circledast^{\l, {\rm u}}_{{\rm I}}\big((\lf {\rm u}_{\tau}\rc^{\l})^{(1)}_{\l}, (\lf {\rm u}_{\tau, T}^s\rc^{\l})'_0\big)\right.\\
&\ \ \ \ +\left(\wt{\circledast}_{{\rm A}}^{\theta, \l}\big((\lf {\rm u}_{\tau}\rc^{\l})^{(1)}_{\l}, (\lf {\rm u}_{\tau, T}^s\rc^{\l})'_0\big), \wt{\circledast}_{{\rm A}}^{\varpi, \l}\big((\lf {\rm u}_{\tau}\rc^{\l})^{(1)}_{\l}, (\lf {\rm u}_{\tau, T}^s\rc^{\l})'_0\big)\right)\\
&\ \ \ \ +\left(\varpi^{\l}\big(\nabla^{\l}((\lf {\rm u}_{\tau, T}^s\rc^{\l})'_0)(H^{\l})^{(1)}_{\l}(\lf {\rm u}_{\tau}\rc^{\l}, e_i)\big)e_i\ d\tau, \right.\\
&\ \ \ \ \ \ \ \ \left.\left. \varpi^{\l}\big(\nabla^{\l}((\lf {\rm u}_{\tau, T}^s\rc^{\l})'_0)(H^{\l})^{(1)}_{\l}(\lf {\rm u}_{\tau}\rc^{\l}, dB_\tau)\big)+ \varpi^{\l}\big((H^{\l})^{(1)}_{\l}\big(\lf {\rm u}_{\tau}\rc^{\l},  d\underline{\Upsilon}_{{\rm V}, B}^{\l, {\rm u}}(\tau)\big)\big)\right)\right\}.
\end{align*}
\end{cor}
\begin{proof}Note that $\big((\lfloor {\rm u}_T^s\rceil^{\l})^{(1)}_{\l}(w)\big)'_0$ conditioned on $\lf {\rm x}_T\rc^{\l}=y$  is the same as  $\big(\overline{(\lfloor {\rm u}_T^s\rceil^{\l})^{(1)}_{\l}}({\rm w})\big)'_0$ conditioned on $\lf {\rm y}_T\rc^{\l}=x$.  The formulas follow  by Proposition \ref{u-s-t-l-1} using the  correspondence between
\[
\big[D\lf \overrightarrow{F}_{\underline{t}, \overline{t}}\rc^{\l}(\lf{\rm u}_{\underline{t}}\rc^{\l}, w)\big]\ \mbox{and}\  \big[D\lf F_{T-\overline{t}, T-\underline{t}}\rc^{\l}(\lf{\mho}_{T-\overline{t}}\rc^{\l}, {\rm w})\big]^{-1}.
\]
\end{proof}

\subsection{The differential of $\l\mapsto p^{\l}(T, x, \cdot)$}\label{TDOp-lam} We will show  Theorem  \ref{regu-p-1st} in  two steps, namely, the $k=3$ and  $k>3$ cases. We begin with the $k=3$ case.  As we sketched in Section \ref{Obs-Stra},   the strategy is to show ${\rm z}_T^{\l, 1}$ defined in (\ref{z-T-lam-1}) is a  $C^1$ vector field,  then derive  a conditional path-wise formula of ${\rm Div}^{\l}{\rm z}_T^{\l, 1}(y)$ and use it to give   the estimation in  (\ref{grad-lnp-lam-1}).  

\begin{lem}\label{C-1-fun-phi-Y} Let $\lambda\in
(-1, 1)\mapsto g^{\lambda}\in \mathcal{M}^3(M)$ be a $C^3$ curve. Let $x\in \M$, $T\in \Bbb R_+$. The map $\overline{\Phi}_{\l}^1:\ Y\mapsto \overline{\Phi}_{\l}^1(Y)$ defined in (\ref{Phi-lam-1-Y}) is a locally bounded $C^1$ functional  on $C^k$ bounded vector fields $Y$ on $\M$.  Consequently, $\{{\rm z}_T^{\l, 1}(y)\}$  is a $C^1$ vector field on $\M$. 
\end{lem}

\begin{proof}Recall that \begin{align*}\overline{\Phi}_{\l}^1(Y)(y)&=\overline{\E}\left(\left.\big\langle Y(\lfloor {\rm x}_T \rceil^{\l}(w)), D\pi({\rm u}_T\rceil^{\l})^{(1)}_{\l}(w)\big\rangle_{\l}\right| \lfloor {\rm x}_T \rceil^{\l}(w)=y\right)=\big\langle Y(y), {\rm z}_T^{\l, 1}(y)\big\rangle. 
\end{align*}
Hence, 
\[
\big\|\overline{\Phi}_{\l}^1(\cdot)(y)\big\|\leq \big\|{\rm z}_T^{\l, 1}(y)\big\|\leq \overline{\E}_{{\P}^{\l}_{x, y, T}}\big\|(\lfloor {\rm u}_T\rceil^{\l})^{(1)}_{\l}(w)\big\|=\frac{1}{p(T, x, y)}\overline{\E}_{{\P}^{\l, *}_{x, y, T}}\big\|(\lfloor {\rm u}_T\rceil^{\l})^{(1)}_{\l}(w)\big\|.  
\]
By Proposition \ref{est-norm-u-t-j},  there are  $\underline{c}_{\Phi}^{\l}$ (depending on $\|g^{\l}\|_{C^2}$ and $\|\XX^{\l}\|_{C^1}$) and ${c}_{\Phi}^{\l}$  (depending on $T, T_0$ and $\|g^{\l}\|_{C^3}$) such that 
\begin{equation*}
\big\|\overline{\Phi}_{\l}^1(\cdot)(y)\big\|\leq \frac{1}{p(T, x, y)}\overline{\E}_{\ov{\P}^{\l, *}_{y, x, T}}\big\|\overline{(\lfloor {\rm u}_T\rceil^{\l})^{(1)}_{\l}}({\rm w})\big\| \leq \frac{1}{p(T, x, y)}\underline{c}_{\Phi}^{\l} e^{{c}_{\Phi}^{\l}(1+d_{\wt{g}^{\l}}(x, y))},
\end{equation*}
where the last term is locally uniformly bounded  in the $y$-coordinate. This shows  the map $Y\mapsto\overline{\Phi}_{\l}^1(Y)$ is locally bounded.  

To show $\overline{\Phi}_{\l}^1$ is $C^1$,  it suffices to show for any flow $F^s$ generated by a smooth bounded vector field ${\rm V}$ on $\M$,  $s\mapsto \overline{\Phi}_{\l}^1(Y)(F^sy), y\in \M$,  is differentiable at $s=0$ and the differential 
\[
(\overline{\Phi}_{\l}^1(Y)(F^sy))'_0:=\left.\frac{d}{ds}\overline{\Phi}_{\l}^1(Y)(F^sy)\right|_{s=0}
\]
varies continuously in $y$. Let  $\lf {\bf F}^s\rc^{\l}$ be  introduced  in Section \ref{the flow F-S}, which extends  $F^s$ to Brownian paths starting from $x$ up to time $T$ using the auxiliary  function $\mathtt s$.  By Proposition \ref{Quasi-IP-2},   $\overline{\P}_x^{\l}\circ \lf {\bf F}^s\rc^{\l}$ is absolutely continuous with respect to $\overline{\P}_x^{\l}$. So the change of variable comparison in Section \ref{Obs-Stra}  works, which gives  (\ref{Ch-Var-compare}), i.e., 
\begin{align*}
\overline{\Phi}_{\l}^1(Y)(F^sy)&=\overline{\E}_{\ov{\P}^{\l}_{x, y, T}}\left(\Phi_{\l}^1(Y, w)\circ \lf {\bf F}^s\rc^{\l} \cdot \frac{d\overline{\P}_x^{\l}\circ \lf {\bf F}^s\rc^{\l}}{d\overline{\P}_x^{\l}}\right)\!\frac{p^{\l}(T, x, y)}{p^{\l}(t, x, F^s y)}\frac{d{\rm Vol}^{\l}}{d{\rm Vol}^{\l}\circ F^s}(y),
\end{align*}
where 
\[\Phi_{\l}^1(Y, {w})=\big\langle Y(\lfloor {\rm x}_T \rceil^{\l}(w)), D\pi(\lfloor {\rm{u}}_T\rceil^{\l})^{(1)}_{\l}(w)\big\rangle_{\l}.\]  By Proposition \ref{Diff-int-U-s-gen}, the process $\Phi_{\l}^1(Y, w)\circ \lf {\bf F}^s\rc^{\l}$ is differentiable in $s$ with 
\begin{equation}\label{PHI-Y-F-s}
(\Phi_{\l}^1\circ \lf {\bf F}^s\rc^{\l})'_s=\big\langle \nabla_{{\rm V}(\lfloor {\rm x}_{T}^s \rceil^{\l})}Y(\lfloor {\rm x}_{T}^s \rceil^{\l}),  D\pi(\lfloor {\rm u}_T^s\rceil^{\l})^{(1)}_{\l}\big\rangle_{\l}\! + \! \big\langle Y(\lfloor {\rm x}_{T}^s \rceil^{\l}),  \nabla_{T, {\rm V}, \mathtt{s}}^{s, \l} D\pi(\lfloor {\rm u}_T\rceil^{\l})^{(1)}_{\l}\big\rangle_{\l}
\end{equation}
and this differential is  $L^q$ integrable  conditioned on ${\rm x}_T=y$ for every $q\geq 1$,  locally  uniformly in the  $s$ parameter. By Lemma \ref{diff-P-F-S}, for $\beta\in C_{x, y}([0, T], \M)$, 
\[
\frac{d\overline{\P}_x^{\l}\circ \lf {\bf F}^s\rc^{\l} }{d\overline{\P}_x^{\l}}(\beta)=\frac{d\ov{\P}_{F^sy, x, T}^{\l}\circ \lf {\bf F}^s\rc^{\l} }{d\ov{\P}_{y, x, T}^{\l}}(\beta)\cdot \frac{p^{\l}(T, x, F^sy)}{p^{\l}(T, x, y)}\cdot\frac{d{\rm Vol}^{\l}\circ F^s}{d{\rm Vol}^{\l}}(y). 
\]
By  Proposition \ref{Quasi-IP-1},  ${d\ov{\P}_{F^sy, x, T}^{\l}\circ \lf {\bf F}^s\rc^{\l}}/{d\ov{\P}_{y, x, T}^{\l}}$ is  differentiable in $s$ with 
\[
\left({d\ov{\P}_{F^sy, x, T}^{\l}\circ \lf {\bf F}^s\rc^{\l}}/{d\ov{\P}_{y, x, T}^{\l}}\right)'_s=\left({d\ov{\P}_{F^sy, x, T}^{\l}\circ \lf {\bf F}^s\rc^{\l}}/{d\ov{\P}_{y, x, T}^{\l}}\right)\cdot \ov{\mathcal{E}}^s_{T}, 
\]
where \[\overline{\mathcal{E}}^{s}_T=-\frac{1}{2}\int_{0}^{T}\big\langle (\ov{{\mathtt g}}^s_t)'_s({\rm w}),  dB_t({\rm w})\big\rangle+\frac{1}{2}\int_0^T \big\langle (\ov{{\mathtt g}}^s_t)'_s({\rm w}), \ov{{\mathtt g}}^s_t({\rm w})\big\rangle\ dt,\]
and both $\big({d\ov{\P}_{F^sy, x, T}^{\l}\circ \lf {\bf F}^s\rc^{\l}}/{d\ov{\P}_{y, x, T}^{\l}}\big)$ and $\ov{\mathcal{E}}^s_{T}$ are  $L^q$ integrable  for all $q\geq 1$, locally uniformly in the $s$ parameter. Using H\"{o}lder's inequality, we conclude that 
\[
\left(\Phi_{\l}^1\circ \lf {\bf F}^s\rc^{\l} \cdot \big({d\overline{\P}_x^{\l}\circ \lf {\bf F}^s\rc^{\l}}/{d\overline{\P}_x^{\l}}\big)\right)'_s
\]
is also  $L^q$ integrable  for every $q\geq 1$, locally uniformly   in the $s$ parameter.  This allows us to take the differential in $s$ under the expectation sign of the expression of $\overline{\Phi}_{\l}^1(Y)(F^s y)$. In particular, this shows $s\mapsto \overline{\Phi}_{\l}^1(Y)(F^sy)$ is differentiable at $s=0$. 

 Let us  derive  a formula for $(\overline{\Phi}_{\l}^1(Y)(F^sy))'_0$.  Note that $\lf\ov{{\mathtt g}}^0\rc^{\l}\equiv 0$ and 
\[
(\lf \ov{{\mathtt g}}^s_t\rc^{\l})'_0({\rm w})=(\lfloor {\mho}_{0}\rceil^{\l})^{-1}\left[\mathtt{s}'(t){\rm V}(y)\right]-{\rm Ric}^{\l}\left(\lfloor{\mho}_{t}\rceil^{\l}(\lfloor{\mho}_0\rceil^{\l})^{-1}[\mathtt{s}(t){\rm V}(y)]\right).
\]
 Using the correspondence between $\lfloor \mho_{t}\rceil^{\l}({\rm w})$ conditioned on ${\rm y}^{\l}_T=x$ and $\lfloor {\rm u}_{T-t}\rceil^{\l}(w)$ conditioned on ${\rm x}^{\l}_T=y$, we have the distribution of $\overline{\mathcal{E}}^{0}_T$ under $\ov{\P}^{\l}_{y, x, T}$ is the same as \[
\oa{\mathcal{E}}_{T, {\rm V}, {\mathtt s}}=-\frac{1}{2}\int_{0}^{T}\big\langle \mathtt{s}'(T\!-\!t)(\lfloor{\rm u}_T\rceil^{\l})^{-1}{\rm V}(\lfloor{\rm x}_T\rceil^{\l})-{\rm Ric}(\lfloor{\rm u}_{t}\rceil^{\l}(\lfloor{\rm u}_T\rceil^{\l})^{-1}\mathtt{s}(T\!-\!t){\rm V}(\lfloor{\rm x}_T\rceil^{\l})), d\oa{B}_{t}\big\rangle.
\]
under $\ov{\P}^{\l}_{x, y, T}$, where $\mathtt s$ is given in Section \ref{flow-F-S-y}.  So,  by (\ref{Ch-Var-compare}) and  (\ref{PHI-Y-F-s}),  we have
\begin{align}
(\overline{\Phi}_{\l}^1(Y)(F^sy))'_0&=\overline{\E}_{\ov{\P}^{\l}_{x, y, T}}\left(\big\langle \nabla_{{\rm V}(\lfloor {\rm x}_{T} \rceil^{\l})}Y, D\pi(\lfloor {\rm u}_T\rceil^{\l})^{(1)}_{\l}\big\rangle_{\l} +\big\langle Y(\lfloor {\rm x}_{T} \rceil^{\l}), \nabla_{T, {\rm V}, \mathtt{s}}^{\l}D\pi(\lfloor {\rm u}_T\rceil^{\l})^{(1)}_{\l}\big\rangle_{\l}  \right.\notag\\
 \notag &\ \ \ \ \  \ \ \ \ \ \ \ \  \  +\left.\big\langle Y(\lfloor {\rm x}_{T} \rceil^{\l}), D\pi(\lfloor {\rm u}_T\rceil^{\l})^{(1)}_{\l}\big\rangle_{\l}\oa{\mathcal{E}}_{T, {\rm V}, {\mathtt s}} \right)\\
&=:\overline{\E}_{\ov{\P}^{\l}_{x, y, T}}\left({\Psi}_{\l}^1(Y, {\rm V})(w)\right),  \label{expression-di-phi}
\end{align}
where we omit the upper-script $0$ of ${\rm x}, {\rm u}$ and $\nabla_{T, {\rm V}, \mathtt{s}}$ at $s=0$ for simplicity.  

To show  $(\overline{\Phi}_{\l}^1(Y)(F^sy))'_0$ is continuous in $y$, we compare it with  its value at  nearby points.  Choose another smooth bounded vector field ${\rm W}$ on $\M$ and let $\leftidx^{r}F$ be the flow it generates, where we use the left upper script  to indicate the parameter associated  with  ${\rm W}$.  As before,  we can extend $\leftidx^{r}F$ to be a one parameter family  of maps 
$\lf {\leftidx^{r}{\bf{F}}}\rc^{\l}=\{ \lf {\leftidx^{r} {{\bf{F}}_y}}\rc^{\l}\}$ on $\wt{g}^{\l}$-Brownian paths starting from $x$ up to time $T$.  Let $ \lfloor\leftidx^{r}{\alpha}\rceil^{\l}, \lfloor\leftidx^{r}{O}\rceil^{\l}, \lfloor\leftidx^{r}{\mathtt g}\rceil^{\l}, \lfloor\leftidx^{r}{\ov{\mathtt g}}\rceil^{\l}, \lfloor\leftidx^{r}{\rm{y}}\rceil^{\l}$, $\lfloor\leftidx^{r}{\mho}\rceil^{\l}$ and   $(\lfloor\leftidx^{r}{\mho}\rceil^{\l})'_r$ denote the corresponding stochastic  processes of $\lf \leftidx^{r}{\bf{F}}\rc^{\l}$ in Theorem \ref{Main-alpha-x-v-Q}. Then a change of variable argument  for (\ref{expression-di-phi})  with $\lf \leftidx^{r}{\bf{F}}\rc^{\l}$ shows that  for $z=\leftidx^{r}F (y)$, 
\[
(\overline{\Phi}_{\l}^1(Y)(F^sz))'_0=\overline{\E}_{\ov{\P}^{\l}_{x, y, T}}\left(\Psi_{\l}^1(Y, {\rm V})\circ \lf {\leftidx^{r}{\bf{F}}}\rc^{\l} \cdot {\frac{d\overline{\P}_x^{\l}\circ \lf {\leftidx^{r}{{\bf F}}}\rc^{\l}}{d\overline{\P}_x^{\l}}}\right)\frac{p^{\l}(T, x, y)}{p^{\l}(T, x, z)}\frac{d{\rm Vol}^{\l}}{d{\rm Vol}^{\l}\circ{\leftidx^{r}{F}}}(y). 
\]
Since  $p^{\l}$ and ${\rm Vol}^{\l}$ are continuous in $y$, for continuity of $(\overline{\Phi}_{\l}^1(Y)(F^sy))'_0$ in $y$, it remains  to show the conditional expectation of the following difference tends to 0 as $r$ goes to 0:
\begin{align*}
&\Psi_{\l}^1(Y, {\rm V})\circ \lf {\leftidx^{r}{{\bf F}}}\rc^{\l} \cdot {\frac{d\overline{\P}_x^{\l}\circ \lf {\leftidx^{r}{{\bf F}}}\rc^{\l}}{d\overline{\P}_x^{\l}}}-\Psi_{\l}^1(Y, {\rm V})\\
&=\left(\Psi_{\l}^1(Y, {\rm V})\circ \lf {\leftidx^{r}{{\bf F}}}\rc^{\l} -\Psi_{\l}^1(Y, {\rm V})\right)\frac{d\overline{\P}_x^{\l}\circ \lf {\leftidx^{r}{{\bf F}}}\rc^{\l}}{d\overline{\P}_x^{\l}} +\left(\frac{d\overline{\P}_x^{\l}\circ \lf {\leftidx^{r}{{\bf F}}}\rc^{\l}}{d\overline{\P}_x^{\l}}-1\right) \Psi_{\l}^1(Y, {\rm V})\\
&=: {\leftidx^{r}{(\rm I)}_1}\cdot  {\leftidx^{r}{(\rm I)}_2}+{\leftidx^{r}{(\rm II)}_1}\cdot {{(\rm II)}_2}.
\end{align*}
For this, it suffices to show 
\begin{align}\label{zero-zero-pf}
\lim\limits_{r\to 0}\overline{\E}_{\ov{\P}^{\l}_{x, y, T}}\big|{\leftidx^{r}{(\rm I)}_1}\big|^2=0 \ \mbox{and}\ \lim\limits_{r\to 0}\overline{\E}_{\ov{\P}^{\l}_{x, y, T}}\big|{\leftidx^{r}{(\rm II)}_1}\big|^2=0
\end{align}
since $\overline{\E}_{\ov{\P}^{\l}_{x, y, T}}|{\leftidx^{r}{(\rm I)}_2}|^2$ is locally uniformly bounded  in $r$ by  Proposition \ref{Quasi-IP-2} and $\overline{\E}_{\ov{\P}^{\l}_{x, y, T}}|{{(\rm II)}_2}|^2$  is bounded  by using Proposition \ref{Quasi-IP-1} and Proposition \ref{Diff-int-U-s}.

Note that $|Y|, |\nabla_{{\rm V}}Y|$ are locally  bounded at $y$ and the difference between $Y(z)$ and $Y(y)$, $\nabla_{{\rm V}}Y(z)$ and $\nabla_{{\rm V}}Y(y)$ under parallel transportation along  $(\imath\mapsto \leftidx^{\imath}F(y))_{\imath\in [0, r]}$ is bounded by a multiple of $r$. Using this,  (\ref{expression-di-phi}) and  a standard  split argument  by H\"{o}lder's inequality, we see that   to conclude the first property in  (\ref{zero-zero-pf}), it suffices to show  \begin{align*}
\begin{array}{ll} {\leftidx^{r}{(\rm III)}}:=& \overline{\E}_{\ov{\P}^{\l}_{x, y, T}}\left(\left|\oa{\mathcal{E}}_{T, {\rm V}, {\mathtt s}}\circ\lf {\leftidx^{r}{{\bf F}}}\rc^{\l}-\oa{\mathcal{E}}_{T, {\rm V}, {\mathtt s}}\right|^4\right)\to 0, \ r\to 0, \\ 
{\leftidx^{r}{(\rm IV)}}:=&\overline{\E}_{\ov{\P}^{\l}_{x, y, T}}\left(\left|D\pi(\lfloor {\rm u}_T\rceil^{\l})^{(1)}_{\l}\circ \lf {\leftidx^{r}{{\bf F}}}\rc^{\l}-D\pi(\lfloor {\rm u}_T\rceil^{\l})^{(1)}_{\l}\right|^4\right)\to 0,  \ r\to 0,\\ 
{\leftidx^{r}{(\rm V)}}:=&  \overline{\E}_{\ov{\P}^{\l}_{x, y, T}}\left(\left|\nabla_{T, {\rm V}, \mathtt{s}}D\pi(\lfloor {\rm u}_T\rceil^{\l})^{(1)}_{\l}\circ \lf {\leftidx^{r}{{\bf F}}}\rc^{\l}-\nabla_{T, {\rm V}, \mathtt{s}}^{\l}D\pi(\lfloor {\rm u}_T\rceil^{\l})^{(1)}_{\l}\right|^2\right)\to 0,  \ r\to 0. 
\end{array}
\end{align*}
Let \begin{align*}
\overline{\mathcal{E}}^{0}_T\circ \lf {\leftidx^{r}{{\bf F}}}\rc^{\l}&=-\frac{1}{2}\int_{0}^{T}\!\!\big\langle (\lfloor \leftidx^{r}{\mho}_{0}\rceil^{\l})^{-1}\!\left[\mathtt{s}'(t){\rm V}(y)\right]-{\rm Ric}^{\l}\big(\lfloor\leftidx^{r}{\mho}_{t}\rceil^{\l}(\lfloor\leftidx^{r}{\mho}_0\rceil^{\l})^{-1}[\mathtt{s}(t){\rm V}(y)]\big), \ d\overleftarrow{\lfloor\leftidx^{r}\alpha_t\rceil^{\l}} \big\rangle\\
&=:-\frac{1}{2}\int_{0}^{T}\!\!\big\langle\leftidx^{r}(\lfloor\ov{{\mathtt g}}^s_t\rceil^{\l})'_0({\rm w}), \ d\overleftarrow{\lfloor\leftidx^{r}\alpha_t\rceil^{\l}}) \big\rangle.
\end{align*}
For ${\leftidx^{r}{(\rm III)}}$,  we have 
\begin{align*}
\begin{array}{ll}
2\cdot {\leftidx^{r}{(\rm III)}}&= \ 2\cdot \ov{\E}_{\overline{\P}^{\l}_{y, x, T}}\left(\big|\overline{\mathcal{E}}^{0}_T\circ \lf {\leftidx^{r}{{\bf F}}}\rc^{\l}-\overline{\mathcal{E}}^{0}_T\big|^4\right)\\
 &\leq\  \ov{\E}_{\overline{\P}^{\l}_{y, x, T}}\left|\int_{0}^{T}\!\!\big\langle\leftidx^{r}(\lfloor\ov{{\mathtt g}}^s_t\rceil^{\l})'_0({\rm w}), \ \big(\overleftarrow{\lfloor\leftidx^{r}O_t\rceil^{\l}}-{\rm{Id}}\big)dB_{t}({\rm{w}})+ \overleftarrow{\lfloor\leftidx^{r}\mathtt g_t\rceil^{\l}}\ dt\big\rangle\right|^4\\
&\ \ \ \ + \ov{\E}_{\overline{\P}^{\l}_{y, x, T}}\left|\int_{0}^{T}\!\!\big\langle\leftidx^{r}(\lfloor\ov{{\mathtt g}}^s_t\rceil^{\l})'_0({\rm w})-\leftidx^{0}(\lfloor\ov{{\mathtt g}}^s_t\rceil^{\l})'_0({\rm w}), \ dB_t({\rm w}) \big\rangle\right|^4\\
&=:\    {\leftidx^{r}{(\rm III)}}_1+{\leftidx^{r}{(\rm III)}}_2.
\end{array}
\end{align*}
For ${\leftidx^{r}{(\rm III)}}_1$, the usual argument  using   Lemma \ref{Hsu-thm-5.4.4} and Burkholder's inequality shows \begin{align*}
3^{-3}\cdot {\leftidx^{r}{(\rm III)}}_1
 \leq &\; \ov{\E}_{\overline{\P}^{\l}_{y, x, T}}\left|\int_{0}^{T}\big\|\overleftarrow{\lfloor\leftidx^{r}O_t\rceil^{\l}}-{\rm{Id}}\big\|^2\big\|\leftidx^{r}(\lfloor\ov{{\mathtt g}}^s_t\rceil^{\l})'_0({\rm w})\big\|^2\ dt\right|^2\\
&\; +\ov{\E}_{\overline{\P}^{\l}_{y, x, T}}\left|\int_{0}^{T}2\big\|\leftidx^{r}(\lfloor\ov{{\mathtt g}}^s_t\rceil^{\l})'_0({\rm w})\big\|\big\|\overleftarrow{\lfloor\leftidx^{r}O_t\rceil^{\l}}-{\rm{Id}}\big\|\big\|\nabla^{\l}\ln p^{\l}(T-t, \lfloor{\rm y}_{t}\rceil^{\l}({\rm w}), x) \big\|\ dt\right|^4\\
&\; +\ov{\E}_{\overline{\P}^{\l}_{y, x, T}}\left|\int_{0}^{T}\big\|\leftidx^{r}(\lfloor\ov{{\mathtt g}}^s_t\rceil^{\l})'_0({\rm w})\big\|\big\|\overleftarrow{\lfloor\leftidx^{r}\mathtt g_t\rceil^{\l}} \big\|\ dt\right|^4. 
\end{align*}
Note that  there is some constant $C$ which depends on $\|g^{\l}\|_{C}, \mathtt{s}$ and $\sup\{\|\rm{V}\|\}$ such that  
\[\big|\overleftarrow{\lfloor\leftidx^{r}\mathtt g_t\rceil^{\l}}({\rm w})\big|,\ \big|\leftidx^{r}(\lfloor\ov{{\mathtt g}}^s_t\rceil^{\l})'_0({\rm w})\big|\leq Cr.\]
Hence 
\begin{align*}
3^{-3}\cdot {\leftidx^{r}{(\rm III)}}_1\leq &\  (Cr)^8T^4+(Cr)^4T^2 \ov{\E}_{\overline{\P}^{\l}_{y, x, T}}\sup_{t\in[0, T]}\big\|\overleftarrow{\lfloor\leftidx^{r}O_t\rceil^{\l}}-{\rm{Id}}\big\|^4\\
&+ (2Cr)^{4}\!\left(\ov{\E}_{\overline{\P}^{\l}_{y, x, T}}\!\sup_{t\in[0, T]}\big\|\overleftarrow{\lfloor\leftidx^{r}O_t\rceil^{\l}}-{\rm{Id}}\big\|^8\cdot \ov{\E}_{\overline{\P}^{\l}_{y, x, T}}e^{8\int_{0}^{T}\big\|\nabla^{\l}\ln p^{\l}(T-t, \lfloor{\rm y}_{t}\rceil^{\l}, x) \big\|\ dt}\right)^{\frac{1}{2}}. 
\end{align*}
By Lemma \ref{O-mathttg-diff-esti} and  Lemma \ref{cond-DF-diff-norm},  for any $q\geq 1$, there is some $C_1(q, T)$  such that 
\[
\ov{\E}_{\ov\P^{\l}_{y, x, T}}\sup_{t\in [0, T]}\  \left\|\lfloor\leftidx^{r}\alpha_t\rceil^{\l}-\lfloor\alpha_t\rceil^{\l} \right\|^{q}\leq C_1(q, T) r^q.\]
Using this and  (\ref{exp-grad-lnp}), we  conclude that ${\leftidx^{r}{(\rm III)}}_1\to 0$ as $r\to 0$. Similarly, using (\ref{B-bridge}),  
 Burkholder's inequality and (\ref{exp-grad-lnp}), we obtain  some  $C_2$ depending on $T, d(x, y)$  such that 
\begin{align*}
 {\leftidx^{r}{(\rm III)}}_2&\leq \ov{\E}_{\overline{\P}^{\l}_{y, x, T}}\left(\int_{0}^{T}\big\|\leftidx^{r}(\lfloor\ov{{\mathtt g}}^s_t\rceil^{\l})'_0-(\lfloor\ov{{\mathtt g}}^s_t\rceil^{\l})'_0\big\|^2\ dt\right)^2\\
&\ \ \ +\ov{\E}_{\overline{\P}^{\l}_{y, x, T}}\left|\int_{0}^{T}\!\left\langle\!\left(\leftidx^{r}(\lfloor\ov{{\mathtt g}}^s_t\rceil^{\l})'_0-(\lfloor\ov{{\mathtt g}}^s_t\rceil^{\l})'_0\!\right), 2(\lfloor{\mho}_t\rceil^{
\l})^{-1}\nabla^{\l}\ln p^{\l}(T-t, \lfloor{\rm y}_{t}\rceil^{\l}, x)\right\rangle_{\l}dt\right|^4\\
&\leq C_2 \ov{\E}_{\overline{\P}^{\l}_{y, x, T}}\left(\sup\limits_{t\in [0, T]}\big\|\leftidx^{r}(\lfloor\ov{{\mathtt g}}^s_t\rceil^{\l})'_0-(\lfloor\ov{{\mathtt g}}^s_t\rceil^{\l})'_0\big\|^4+\sup\limits_{t\in [0, T]}\big\|\leftidx^{r}(\lfloor\ov{{\mathtt g}}^s_t\rceil^{\l})'_0-(\lfloor\ov{{\mathtt g}}^s_t\rceil^{\l})'_0\big\|^8\right). 
\end{align*}
The argument in Lemma  \ref{O-mathttg-diff-esti} shows  there is some $C_3$ depending  on $\|g^{\l}\|_{C^3}$ such that   \begin{align*}\ov{\E}_{\overline{\P}^{\l}_{y, x, T}}\sup\limits_{t\in [0, T]}\big\|\leftidx^{r}(\lfloor\ov{{\mathtt g}}^s_t\rceil^{\l})'_0-(\lfloor\ov{{\mathtt g}}^s_t\rceil^{\l})'_0\big\|^q&\leq C_3\cdot \ov{\E}_{\overline{\P}^{\l}_{y, x, T}}\sup\limits_{t\in [0, T]}\big\|\lfloor\leftidx^{r}{\mho}_{t}\rceil^{\l}-\lfloor{\mho}_{t}\rceil^{\l}\big\|^q\\
 &\leq C_3 r^q \sup_{\imath\in [-r_0, r_0]}\ov{\E}_{\overline{\P}^{\l}_{y, x, T}}\sup_{t\in [0, T]}\big\|(\lfloor\leftidx^{r}{\mho}_{t}\rceil^{\l})'_\imath\big\|^{q}.
 \end{align*}
This immediately implies that $\lim_{r\to 0} {\leftidx^{r}{(\rm III)}}_2=0$.  For ${\leftidx^{r}{(\rm IV)}}$,   we have
\[
{\leftidx^{r}{(\rm IV)}}\leq {\rm Const.}\cdot \ov{\E}_{\overline{\P}^{\l}_{y, x, T}}\left(\left|\overline{(\lfloor {\leftidx^{r}{\rm{u}}}_T\rceil^{\l})^{(1)}_0}-\overline{(\lfloor {\rm u}_T\rceil^{\l})^{(1)}_0}\right|^4\right).
\]Note that 
\begin{align*}
\!\!(\theta, \omega)_{\lfloor{\leftidx^{r}\mho_{0}\rceil^{\l}}}^{\l}\overline{(\lfloor {\leftidx^{r}{\rm{u}}}_T\rceil^{\l})^{(1)}_{\l}}= &\ [D\wt{\lf \leftidx^{r}F_{0, T}\rc^{\l}}(\lfloor \leftidx^{r}\mho_0\rceil^{\l}, {\rm w})]^{-1}\wt{(\lfloor {\rm u}_0\rceil^{\l})^{(1)}_{\l}}+\int_{0}^{T}\![D\wt{\lf \leftidx^{r}F_{0, t}\rc^{\l}}(\lfloor \leftidx^{r}\mho_0\rceil^{\l}, {\rm w})]^{-1}\\
&\ \ \ \ \ \ \ \left(\varpi\big((H^{\l})^{(1)}_{\l}(\lfloor \leftidx^{r}\mho_t\rceil^{\l}, e_i)\big)e_i\ d\tau, \varpi\big((H^{\l})^{(1)}_{\l}(\lfloor \leftidx^{r}\mho_t\rceil^{\l},  d\overleftarrow{\lfloor\leftidx^{r}\alpha\rceil^{\l}_t})\big)\right). 
\end{align*}
By Lemma \ref{O-mathttg-diff-esti} and Lemma \ref{cond-DF-norm},  for any $q\geq 1$, 
\[
\sup_{r\in [-r_0, r_0]}\ov{\E}_{\ov\P^{\l}_{y, x, T}}\sup_{t\in [0, T]}\ \  \left\|\lfloor\leftidx^{r}\alpha\rceil^{\l}_t \right\|^{q}, \ \left\|[\wt{D \lf \leftidx^{r}F_{0, t}\rc^{\l}}(\lfloor \leftidx^{r}\mho_0\rceil^{\l}, {\rm w})]^{-1}\right\|^q<\leftidx^{r}\underline{c} e^{\leftidx^{r}c(1+d_{\wt{g}^{\l}}(x, y))},
\]
where  $\leftidx^{r}\underline{c}$  depends on $r_0, m, q,  \mathtt{s}$ and $\|g^{\l}\|_{C^2}$,  and  $\leftidx^{r}c$  depends  on $r_0, m, q,  \mathtt{s}$, $T, T_0$  and $\|g^{\l}\|_{C^3}$. Moreover, by Lemma \ref{O-mathttg-diff-esti} and  Lemma \ref{cond-DF-diff-norm}, 
\begin{align*}
&\ov{\E}_{\ov\P^{\l}_{y, x, T}}\sup_{t\in [0, T]}\  \left\|\lfloor\leftidx^{r}\alpha_t\rceil^{\l}-\lfloor\alpha_t\rceil^{\l} \right\|^{q}\leq C_4 r^q, \\
&\ov{\E}_{\ov\P^{\l}_{y, x, T}}\sup_{t\in [0, T]}\  \left\|[\wt{D\lf \leftidx^{r}F_{0, t}\rc^{\l}}(\lfloor \leftidx^{r}\mho_0\rceil^{\l}, {\rm w})]^{-1}-[\wt{D\lf F_{0, t}\rc^{\l}}(\lfloor \mho_0\rceil^{\l}, {\rm w})]^{-1}\right\|^{q}\leq C_5 r^q,
\end{align*}
where the constants  $C_4, C_5$ depend on $\|g^{\l}\|_{C^2}$. 
Again,  a standard split argument using  these estimations and H\"{o}lder's inequality gives   $\lim_{r\to 0} {\leftidx^{r}{(\rm IV)}}=0$.  To conclude that $\lim_{r\to 0} {\leftidx^{r}{(\rm V)}}=0$, we see from (\ref{the U-T-s-1}) that it suffices to show  for any $q>1$, 
\begin{equation*}\label{A-F-r-arg}
\ov{\E}_{\ov\P^{\l}_{y, x, T}}\sup_{t\in [0, T]}\left\|{\bf A}_{t}\circ \lf {\leftidx^{r}{\bf{F}}\rc^{\l} }-{\bf A}_{t}\right\|^q\leq C_{\bf A} r^q 
\end{equation*}
for some $C_{\bf A}$ depending on $\|g^{\l}\|_{C^3}, \|\XX\|_{C^2}$, $T$ and $d(x, y)$, 
where ${\bf A}_{t}=(\lfloor {\mathtt g}_{t}^s\rceil^{\l})'_0,$   $(\lfloor O_{t}^s\rceil^{\l})'_0,$  $(\lfloor \mho_{t}^s\rceil^{\l})'_0,$  $(\lfloor \mho_{t}\rceil^{\l})^{(1)}_{\l}$ or   $[\wt{(DF_{0, t})}(\lfloor\mho_{0}\rceil^{\l}, {\rm w})]^{-1}$. Using  Lemma \ref{Cor-O-g-formula}, this can be 
be reduced to  the cases  that ${\bf A}_t=\lfloor\alpha\rceil^{\l}_{t}$,  $\lfloor \mho_{t}\rceil^{\l}$ or $[\wt{D\lf F_{0, t}\rc^{\l}}(\lfloor\mho_{0}\rceil^{\l}, {\rm w})]^{-1}$,  which were shown as above.

Let $C'$ be a bound of $|{d{\rm Vol}^{\l}\circ {\leftidx^{r}F}}/{d{\rm Vol}^{\l}}(y)|$ for $r\in [-r_0, r_0]$. By using (\ref{equ-P-F-S}), we obtain 
\begin{align*}
\ov{\E}_{\ov\P_{y, x, T}^{\l}}\!\left|\frac{d\overline{\P}_x^{\l}\circ\lf {\leftidx^{r}{{\bf F}}}\rc^{\l} }{d\overline{\P}_x^{\l}} -1\right|^2&\!\!\leq 2C'\!\!\left(\ov{\E}_{\ov\P^{\l}_{y, x, T}}\left|\frac{d\P_{z}^{\l}\circ \lf {\leftidx^{r}{{\bf F}}}\rc^{\l} }{d\P_{y}^{\l}}-1\right|^2+\ov{\E}_{\ov\P^{\l}_{y, x, T}}\left|\frac{d{\rm Vol}^{\l}\circ {\leftidx^{r}F}}{d{\rm Vol}^{\l}}(y)-1\right|^{2}\right)\\
&=:C'\left({\leftidx^{r}{(\rm VI)}}+{\leftidx^{r}{(\rm VII)}}\right).
\end{align*}
Clearly, ${\leftidx^{r}{(\rm VII)}}\to 0$ as $r\to 0$. 
For the second property in (\ref{zero-zero-pf}), it remains to show ${\leftidx^{r}{(\rm VI)}}\to 0$ as $r\to 0$.  Following the proof of Proposition \ref{Quasi-IP-1}, we obtain 
\[
\frac{d\P_{z}^{\l}\circ \lf {\leftidx^{r}{{\bf F}}}\rc^{\l} }{d\P_{y}^{\l}}=e^{\left\{-\frac{1}{2}\int_{0}^{T}\langle\lf {\leftidx^{r}{\ov{{\mathtt g}}}}_{\tau}\rc^{\l}({\rm w}),  \  dB_\tau({\rm w})\rangle+\frac{1}{4}\int_{0}^{T}|\lf {\leftidx^{r}{\ov{{\mathtt g}}}}_{\tau}\rc^{\l}({\rm w})|^2\ d\tau\right\}}=:{\leftidx^{r}\mathcal{E}_T}({\rm w})
\]
and 
\begin{align*}
({\leftidx^{r}\mathcal{E}_T}({\rm w}))'_r &=\ {\leftidx^{r}\mathcal{E}_T}({\rm w})\cdot \left(-\frac{1}{2}\int_{0}^{T}\langle  (\lf {\leftidx^{r}{\ov{{\mathtt g}}}}_{t}\rc^{\l})'_r({\rm w}),  dB_t({\rm w})\rangle+\frac{1}{2}\int_0^T \langle (\lf {\leftidx^{r}{\ov{{\mathtt g}}}}_{t}\rc^{\l})'_r({\rm w}), \lf {\leftidx^{r}{\ov{{\mathtt g}}}}_{t}\rc^{\l}({\rm w})\rangle\ dt \right)\\
&=:\  {\leftidx^{r}\mathcal{E}_T}({\rm w})\cdot \overline{\leftidx^{r}\mathcal{E}_T}({\rm w}).
\end{align*}
The   usual argument  using   Lemma \ref{Hsu-thm-5.4.4} and Burkholder's inequality shows that for every  $q\geq 1$,  $
\ov{\E}_{\ov\P_{y, x, T}^{\l}}\left|\overline{\leftidx^{r}\mathcal{E}_T}({\rm w})\right|^q$ 
is locally uniformly  bounded in $r$. Hence ${\leftidx^{r}{(\rm VI)}}\to 0$ as $r\to 0$.

Altogether, we have shown the map $Y\mapsto \overline{\Phi}_{\l}^1(Y)$ is a $C^1$ locally bounded functional on $C^{k}$ vector fields $Y$ on $\M$. Hence there  exists  some  $C^1$ vector field $\wt{{\rm z}_T^{\l, 1}}$ on $\M$ such that 
\[
\overline{\Phi}_{\l}^1(Y)(y)=\big\langle Y(y),  \wt{{\rm z}_T^{\l, 1}}(y)\big\rangle_{\l}. 
\]
This  shows   $ \wt{{\rm z}_t^{\l, 1}}(y)\equiv {\rm z}_T^{\l, 1}(y)$. Thus  $\{{\rm z}_T^{\l, 1}(y)\}$ forms a $C^1$ vector field on $\M$ as claimed. 
\end{proof}

\begin{lem}\label{reg-Z-lam-1}
Let $\lambda\in
(-1, 1)\mapsto g^{\lambda}\in \mathcal{M}^3(M)$ be a $C^3$ curve. Let $x\in \M$, $T\in \Bbb R_+$.  For any smooth bounded vector field ${\rm V}$ on $\M$,  let ${\mathtt s}$, $\nabla^{\l}_{T, {\rm V}, \mathtt{s}}$, $\oa{\mathcal{E}}_{T, {\rm V}, {\mathtt s}}$ be as above, then 
\begin{equation}\label{Cov-der-Z-lam}
\nabla^{\l}_{{\rm V}(y)}{\rm z}_T^{\l, 1}(y)=\ov{\E}\left(\left.\nabla^{\l}_{T, {\rm V}, \mathtt{s}}D\pi(\lfloor {\rm u}_T\rceil^{\l})^{(1)}_{\l}(w)+D\pi(\lfloor {\rm u}_T\rceil^{\l})^{(1)}_{\l}(w)\oa{\mathcal{E}}_{T, {\rm V}, {\mathtt s}}(w)\right|  \lfloor {\rm x}_T \rceil^{\l}({w})=y\right). 
\end{equation}
As a consequence, 
\begin{align*}
&{\rm Div}^{\l}{\rm z}_T^{\l, 1}(y)\\
&\ =\ov{\E}\left({\rm tr}\left({\rm V}\mapsto \nabla^{\l}_{T, {\rm V}, \mathtt{s}}D\pi(\lfloor {\rm u}_T\rceil^{\l})^{(1)}_{\l}\right)\!-\!\big\langle D\pi(\lfloor {\rm u}_T\rceil^{\l})^{(1)}_{\l}, \ \frac{1}{2}\lfloor{\rm u}_T\rceil^{\l}\!\int_{0}^{T}\!\mathtt{s}'(T\!-\!\tau)d\oa{B}_{\tau}\big\rangle_{\l}
\right.\\
&\ \ \ \ \  \ \ \ \left.\left.+\big\langle D\pi(\lfloor {\rm u}_T\rceil^{\l})^{(1)}_{\l}, \frac{1}{2}\!\int_{0}^{T}\!\mathtt{s}(T\!-\!\tau) \lfloor{\rm u}_{T}\rceil^{\l}(\lfloor{\rm u}_\tau\rceil^{\l})^{-1}({\rm Ric}_{\lfloor{\rm u}_{\tau}\rceil^{\l}}^{\l})^{-1}d\oa{B}_{\tau}\big\rangle_{\l} \right|  \lfloor {\rm x}_T \rceil^{\l}({w})=y\right).
\end{align*}
\end{lem}
\begin{proof}Let $Y$ be a $C^k$ bounded  vector field on $\M$. By Lemma \ref{C-1-fun-phi-Y}, 
\begin{equation}\label{PHI-Z-T-L}
\overline{\Phi}_{\l}^1(Y)(y)=\big\langle Y(y),  {\rm z}_T^{\l, 1}(y)\big\rangle_{\l},
\end{equation}
where all the variables $\overline{\Phi}_{\l}^1(Y), Y$ and ${\rm z}_T^{\l, 1}$ are $C^1$ in $y$. Hence
\begin{align}\label{Phi-1-Fs-y-com1}
\nabla^{\l}_{{\rm V}}(\overline{\Phi}_{\l}^1(Y))(y)=\big\langle \nabla^{\l}_{{\rm V}}Y(y),  {\rm z}_T^{\l, 1}(y)\big\rangle_{\l}+ \big\langle Y(y),  \nabla^{\l}_{{\rm V}}{\rm z}_T^{\l, 1}(y)\big\rangle_{\l}. 
\end{align}
Let  $\{F^s\}_{s\in \Bbb R}$ be the flow generated by a smooth vector field ${\rm V}$. Then 
\[
\nabla^{\l}_{{\rm V}}(\overline{\Phi}_{\l}^1(Y))(y)=\left(\overline{\Phi}_{\l}^1(Y)(F^sy)\right)'_0. 
\]
It was  shown in Lemma \ref{C-1-fun-phi-Y} that 
\begin{align}
\notag\left(\overline{\Phi}_{\l}^1(Y)(F^sy)\right)'_0&=\overline{\E}\left(\big\langle \nabla^{\l}_{{\rm V}(\lfloor {\rm x}_{T} \rceil^{\l})}Y, D\pi(\lfloor {\rm u}_T\rceil^{\l})^{(1)}_{\l}\big\rangle_{\l} +\big\langle Y(\lfloor {\rm x}_{T} \rceil^{\l}), \nabla^{\l}_{T, {\rm V}, \mathtt{s}}D\pi(\lfloor {\rm u}_T\rceil^{\l})^{(1)}_{\l}\big\rangle_{\l}  \right.\\
\label{Phi-1-Fs-y-com2}&\ \ \ \ \ \ \ \ \ \ \ \ \ \ \ \ \ \ \ \ \ \ \ \ \ \  \left.\left.+\big\langle Y(\lfloor {\rm x}_{T} \rceil^{\l}), D\pi(\lfloor {\rm u}_T\rceil^{\l})^{(1)}_{\l}\big\rangle_{\l}\oa{\mathcal{E}}_{T, {\rm V}, {\mathtt s}} \right|  \lfloor {\rm x}_T \rceil^{\l}(w)=y\right). 
\end{align}
Applying (\ref{PHI-Z-T-L}) for the $C^{k-1}$ vector field $\nabla^{\l} Y$ (instead of $Y$) gives 
\begin{align*}
\overline{\E}\left(\left.\big\langle \nabla^{\l}_{{\rm V}(\lfloor {\rm x}_{T} \rceil^{\l})}Y, D\pi(\lfloor {\rm u}_T\rceil^{\l})^{(1)}_{\l}\big\rangle_{\l} \right|  \lfloor {\rm x}_T \rceil^{\l}(w)=y\right)=\big\langle \nabla^{\l}_{{\rm V}}Y(y),  {\rm z}_T^{\l, 1}(y)\big\rangle_{\l}. 
\end{align*}
Report this in (\ref{Phi-1-Fs-y-com2}) and then compare it with (\ref{Phi-1-Fs-y-com1}). We obtain  
\begin{align*}
&\big\langle Y(y),  \nabla^{\l}_{{\rm V}}{\rm z}_t^{\l, 1}(y)\big\rangle_{\l}\\
&=\overline{\E}\left(\big\langle Y(\lfloor {\rm x}_{T} \rceil^{\l}), \nabla^{\l}_{T, {\rm V}, \mathtt{s}}D\pi(\lfloor {\rm u}_T\rceil^{\l})^{(1)}_{\l}\big\rangle_{\l}\!+\!\left.\big\langle Y(\lfloor {\rm x}_{T} \rceil^{\l}), D\pi(\lfloor {\rm u}_T\rceil^{\l})^{(1)}_{\l}\big\rangle_{\l}\oa{\mathcal{E}}_{T, {\rm V}, {\mathtt s}} \right|  \lfloor {\rm x}_T \rceil^{\l}(w)=y\right)\\
&=\left\langle Y(y), \ \ov{\E}\left(\left.\nabla^{\l}_{T, {\rm V}, \mathtt{s}}D\pi(\lfloor {\rm u}_T\rceil^{\l})^{(1)}_{\l}(w)+D\pi(\lfloor {\rm u}_T\rceil^{\l})^{(1)}_{\l}(w)\oa{\mathcal{E}}_{T, {\rm V}, {\mathtt s}}(w)\right|  \lfloor {\rm x}_T \rceil^{\l}({w})=y\right)\right\rangle_{\l}. 
\end{align*}
This implies  (\ref{Cov-der-Z-lam}) since $Y$ is arbitrary.       

 The divergence $({\rm Div}^{\l}{\rm z}_T^{\l, 1}(y))$ is just the trace of the mapping ${\rm V}(y)\mapsto \nabla^{\l}_{{\rm V}(y)}{\rm z}_T^{\l, 1}(y)$. Put
\begin{align*}
\oa{\mathcal{E}}_{T, {\rm V}, {\mathtt s}}^1
&=-\frac{1}{2}\int_{0}^{T}\big\langle \mathtt{s}'(T\!-\!\tau)(\lfloor{\rm u}_T\rceil^{\l})^{-1}{\rm V}(\lfloor{\rm x}_T\rceil^{\l}), d\oa{B}_{\tau}\big\rangle,\\
\oa{\mathcal{E}}_{T, {\rm V}, {\mathtt s}}^2&=\frac{1}{2}\int_{0}^{T}\big\langle {\rm Ric}^{\l}_{\lfloor{\rm u}_{\tau}\rceil^{\l}}(\lfloor{\rm u}_{\tau}\rceil^{\l}(\lfloor{\rm u}_T\rceil^{\l})^{-1}\mathtt{s}(T\!-\!\tau){\rm V}(\lfloor{\rm x}_T\rceil^{\l})), d\oa{B}_{\tau}\big\rangle. 
\end{align*} Then 
\begin{align*}
({\rm Div}^{\l}{\rm z}_T^{\l, 1}(y))&=\ov{\E}\left({\rm tr}\left({\rm V}\mapsto \nabla^{\l}_{T, {\rm V}, \mathtt{s}}D\pi(\lfloor {\rm u}_T\rceil^{\l})^{(1)}_{\l}(w)\right)+
\right.\\
&\ \ \ \ \ \ \ \left.\left.\sum_{i=1}^{2}{\rm tr}\left({\rm V}\mapsto D\pi(\lfloor {\rm u}_T\rceil^{\l})^{(1)}_{\l}(w)\oa{\mathcal{E}}_{T, {\rm V}, {\mathtt s}}^i(w)\right)\right|  \lfloor {\rm x}_T \rceil^{\l}({w})=y\right).
\end{align*}
 Take ${\rm V}_1, \cdots, {\rm V}_m$ to be orthogonal at $y$ in the metric $\wt{g}^{\l}$. We obtain 
\begin{align*}{\rm tr}\left({\rm V}\mapsto D\pi(\lfloor {\rm u}_T\rceil^{\l})^{(1)}_{\l}(w)\oa{\mathcal{E}}_{T, {\rm V}, {\mathtt s}}^1(w)\right)&=\sum_{j=1}^{m}\big\langle D\pi(\lfloor {\rm u}_T\rceil^{\l})^{(1)}_{\l}(w), {\rm V}_i\big\rangle_{\l} \cdot \oa{\mathcal{E}}_{T, {\rm V}_i, {\mathtt s}}^1(w) \\
&=\big\langle  D\pi(\lfloor {\rm u}_T\rceil^{\l})^{(1)}_{\l}(w), -\frac{1}{2}{\rm u}_T\int_{0}^{T}\!\mathtt{s}'(T\!-\!\tau)d\oa{B}_{\tau} \big\rangle_{\l}, 
\end{align*}
Note that   $\lfloor{\rm u}_{\tau}\rceil^{\l}(\lfloor{\rm u}_T\rceil^{\l})^{-1}$ is the backward parallel transportation along $\lf {\rm x}\rc^{\l}_{[\tau, T]}(w)$ which preserves the inner-product.  Using (\ref{Ric-def}), we obtain 
\[
\oa{\mathcal{E}}_{T, {\rm V}, {\mathtt s}}^2=\frac{1}{2}\int_{0}^{T}\big\langle {\rm V}(\lfloor{\rm x}_T\rceil^{\l}), \mathtt{s}(T\!-\!\tau) \lfloor{\rm u}_{T}\rceil^{\l}(\lfloor{\rm u}_\tau\rceil^{\l})^{-1}{\rm Ric}_{\lfloor{\rm u}_{\tau}\rceil^{\l}}^{-1}d\oa{B}_{\tau}\big\rangle_{\l} 
\]
and 
\begin{align*}
\begin{array}{ll}
\ \ \ &{\rm tr}\left({\rm V}\mapsto D\pi(\lfloor {\rm u}_T\rceil^{\l})^{(1)}_{\l}(w)\oa{\mathcal{E}}_{T, {\rm V}, {\mathtt s}}^2(w)\right)\\ \\
&\ \ \ \ \ \  =\ \ \sum_{j=1}^{m}\langle D\pi(\lfloor {\rm u}_T\rceil^{\l})^{(1)}_{\l}(w), {\rm V}_i\rangle \cdot \oa{\mathcal{E}}_{T, {\rm V}_i, {\mathtt s}}^2(w) \\ \\
&\ \ \ \ \ \  =\big\langle  D\pi(\lfloor {\rm u}_T\rceil^{\l})^{(1)}_{\l}(w), \frac{1}{2}\int_{0}^{T}\mathtt{s}(T\!-\!\tau) \lfloor{\rm u}_{T}\rceil^{\l}(\lfloor{\rm u}_\tau\rceil^{\l})^{-1}{\rm Ric}_{\lfloor{\rm u}_{\tau}\rceil^{\l}}^{-1}d\oa{B}_{\tau}\big\rangle_{\l}.
\end{array}
\end{align*}
\end{proof}

\begin{proof}[Proof of Theorem  \ref{regu-p-1st} ($k=3$)] Let $x\in \M$ and $T\in \Bbb R_+$.  Let $(\lfloor {\rm x}_t\rceil^{\l}, \lfloor {\rm u}_t\rceil^{\l})_{t\in \Bbb R_+}$ be the stochastic process  pair which defines the Brownian motion on $(\M, \wt{g}^{\l})$ starting from $x$. By   Lemma \ref{u-lambda-differential-1} and Proposition \ref{est-norm-u-t-j}, it is true   that  for  any  $f\in C_c^{\infty}(\M)$, 
\begin{equation*}
\left(\int_{\M} f(y)p^{\l}(T, x, y)\ d{\rm Vol}^{\l}(y)\right)_{\l}^{(1)} =\ \int_{\M} \big\langle \nabla^{\l}_y f(y), {\rm z}_T^{\l, 1}(y)\cdot p^{\l}(T, x, y)\big\rangle_{\l}\ d{\rm Vol}^{\l}(y).  
\end{equation*}
Since $\{{\rm z}_T^{\l, 1}(y)\}$  is a $C^1$ vector field on $\M$ by Lemma \ref{C-1-fun-phi-Y}, the classical integration by parts argument in Section \ref{Obs-Stra}  shows that 
\begin{align*}
&\left(\int_{\M} f(y)p^{\l}(T, x, y)\ d{\rm Vol}^{\l}(y)\right)_{\l}^{(1)}\\
&\  =\ -\int_{\M} f(y)\left(({\rm Div}^{\l}{\rm z}_T^{\l, 1}(y))\cdot p^{\l}(T, x, y)+\big\langle {\rm z}_T^{\l, 1}(y), \nabla^{\l}p^{\l}(T, x, y) \big\rangle_{\l}\right)\ d{\rm Vol}^{\l}(y)\\
&\  =\ \int_{\M}f(y)\phi_{\l}^{1}(T, x, y)p^{\l}(T, x, y)\ d{\rm Vol}^{\l}(y). 
\end{align*}
The function $\phi_{\l}^{1}(T, x, y)$ is continuous in $y$, uniformly in $\l$ (see Lemma \ref{C-1-fun-phi-Y}). Hence its continuity in $\l$  follows from the continuity in $\l$ of $\left(\int_{\M} f(y)p^{\l}(T, x, y)\ d{\rm Vol}^{\l}(y)\right)_{\l}^{(1)}$ for any $f\in C_c^{\infty}(\M)$, which is true  by (\ref{indentity-diff-expect}) and the convergence in $\l$ of  $\lfloor {\rm x}_T\rceil^{\l}(w)$ and  $\lfloor {\rm u}_T\rceil^{\l}(w)$ in the  $L^q$-norm for every $q\geq 1$.  So  the first part argument in the proof of Lemma \ref{weak-reg-p-1} works, which  shows that  $\l\mapsto p^{\l}(T, x, \cdot)$ is $C^{1}$, the differential $(p^{\l})^{(1)}_{\l}(T, x, y)$ is continuous in $y$ and 
\begin{equation*}
(p^{\l})^{(1)}_{\l}(T, x, y)\cdot \rho^{\l}(y)+p^{\l}(T, x, y)\cdot (\rho^{\l})^{(1)}_{\l}(y)=\phi_{\l}^{1}(T, x, y) p^{\l}(T, x, y)\rho^{\l}(y). 
\end{equation*}
This  gives   (\ref{Thm1.3-step 1}) since  $\rho^{\l}$  is  non-zero for $\mathcal{V}_g$ small. 

Next, we show (\ref{grad-lnp-lam-1}) with $l=0$. For this,  it suffices to show the same type of bound holds for  the $L^q$-norm of $\phi_{\l}^{1}(T, x, y)$.  Note that,  by  Lemma  \ref{C-1-fun-phi-Y},  ${\rm z}_T^{\l, 1}(y)$ is such that 
\begin{align*}
\big\langle {\rm z}_T^{\l, 1}(y), \nabla^{\l}\!\ln p^{\l}(T, x, y) \big\rangle_{\l}=\overline{\E}\left(\!\left.\big\langle D\pi({\rm u}_T\rceil^{\l})^{(1)}_{\l}\!(w), \nabla^{\l}\ln p^{\l}(T, x, \lfloor {\rm x}_T \rceil^{\l}(w))\big\rangle_{\l}\right| \lfloor {\rm x}_T \rceil^{\l}=y\right).
\end{align*}
Using this and  the  formula of ${\rm Div}^{\l}{\rm z}_T^{\l, 1}$ in Lemma \ref{reg-Z-lam-1},  we obtain 
\begin{equation}\label{wt-phi-00}\phi_{\l}^{1}(T, x, y)=\ov{\E}\left(\left.\wt{\phi}_{\l}^{1}(T, x, w)\right|\lfloor {\rm x}_T \rceil^{\l}(w)=y \right),\end{equation}
where 
\begin{align}
{\wt{\phi}}_{\l}^{1}(T, x, w)&=-{\rm tr}\big({\rm V}\mapsto \nabla^{\l}_{T, {\rm V}, \mathtt{s}}D\pi(\lfloor {\rm u}_T\rceil^{\l})^{(1)}_{\l}\big)+\langle D\pi(\lfloor {\rm u}_T\rceil^{\l})^{(1)}_{\l},  \frac{1}{2}\lfloor{\rm u}_T\rceil^{\l}\!\int_{0}^{T}\!\mathtt{s}'(T\!-\!\tau)d\oa{B}_{\tau}\big\rangle_{\l}\notag\\
&\ \ \ \ \ \ \ \ \ \ \ \ \ \ \ \ \ \ \ \ \ \  -\big\langle D\pi(\lfloor {\rm u}_T\rceil^{\l})^{(1)}_{\l}, \frac{1}{2}\!\int_{0}^{T}\!\mathtt{s}(T\!-\!\tau) \lfloor{\rm u}_{T}\rceil^{\l}(\lfloor{\rm u}_\tau\rceil^{\l})^{-1}{\rm Ric}_{\lfloor{\rm u}_{\tau}\rceil^{\l}}^{-1}d\oa{B}_{\tau}\big\rangle_{\l}\label{wt-phi-1}\\
&\ \ \ \ \ \ \ \ \ \ \ \ \ \ \ \ \ \ \ \ \ \  -\big\langle D\pi(\lfloor {\rm u}_T\rceil^{\l})^{(1)}_{\l}, \nabla^{\l}\ln p^{\l}(T, x, \lfloor {\rm x}_T \rceil^{\l})\big\rangle_{\l}\ \ \ \ \ \ \ \ \ \ \ \ \ \ \ \ \ \ \ \ \ \ \ \ \ \ \notag\\
&=: ({\rm I})(T, x, w)+({\rm II})(T, x, w)+({\rm III})(T, x, w)+({\rm IV})(T, x, w). \notag
\end{align}
So, 
\begin{align*}
\left\|\phi_{\l}^{1}(T, x, \cdot)\right\|_{L^q}^q = &\ \int_{\M}\left|\ov{\E}\big(\wt{\phi}_{\l}^{1}(T, x, w)\big|\lfloor {\rm x}_T \rceil^{\l}(w)=y \big)\right|^q p^{\l}(T, x, y)\ d{{\rm Vol}}^{\l}(y)\\
\leq & \ \int_{\M}\ov{\E}\big(\|\wt{\phi}_{\l}^{1}(T, x, w)\|^q\big|\lfloor {\rm x}_T \rceil^{\l}(w)=y \big) p^{\l}(T, x, y)\ d{{\rm Vol}}^{\l}(y)\\
\leq &\ 4^{q-1}\left(\ov{\E} \|({\rm I})\|^q+ \ov{\E} \|({\rm II})\|^q+\ov{\E} \|({\rm III})\|^q+\ov{\E} \|({\rm IV})\|^q\right). 
\end{align*}
Hence we will obtain (\ref{grad-lnp-lam-1}) with $l=0$ if  $({\rm I})$, $({\rm II})$, $({\rm III})$ and $({\rm IV})$ all have the same type of $L^q$ bounds. This actually follows from  Proposition \ref{est-norm-D-j-F-t} and Proposition \ref{est-norm-u-t-j}.  For $({\rm IV})$,  it is true by (\ref{DF-j-lambda}) and (\ref{equ-Hs2-ST}) since 
\begin{align*}
\left(\overline{\E}\left\|({\rm IV})\right\|^q\right)^2\leq \ \overline{\E} \left\|(\lfloor {\rm u}_T\rceil^{\l})^{(1)}_{\l}\right\|^{2q}\cdot \overline{\E}\left\|\nabla^{\l}\ln p^{\l}(T, x, \lfloor {\rm x}_T \rceil^{\l})\right\|^{2q}.
\end{align*}
Using  H\"{o}lder's inequality, we obtain  \[\left(\overline{\E}\left\|({\rm III})\right\|^q\right)^2\leq\ \overline{\E} \left\|(\lfloor {\rm u}_T\rceil^{\l})^{(1)}_{\l}\right\|^{2q}\cdot \overline{\E} \left\|\frac{1}{2}\!\int_{0}^{T}\!\mathtt{s}(T\!-\!\tau) \lfloor{\rm u}_{T}\rceil^{\l}(\lfloor{\rm u}_\tau\rceil^{\l})^{-1}{\rm Ric}_{\lfloor{\rm u}_{\tau}\rceil^{\l}}^{-1}d\oa{B}_{\tau}\right\|^{2q}.\]
Using  Proposition \ref{est-norm-u-t-j} and Lemma \ref{Ku-lem}, it is easy to show that $\overline{\E}\left\|({\rm III})\right\|^q$ has 
the same bound type in (\ref{grad-lnp-lam-1}) with $l=0$.  The term $({\rm II})$ can be handled in the same way.  For  $({\rm I})$, it suffices to  estimate the $L^q$-norm of $\big((\lfloor {\rm u}_T^s\rceil^{\l})^{(1)}_{\l}(w)\big)'_0$ for $\rm{V}$ with norm 1. Split the It\^{o} integral of $\big((\theta, \varpi)(\lfloor {\rm u}_T^s\rceil^{\l})^{(1)}_{\l}(w)\big)'_0$ in Corollary \ref{cor-u-s-t-l-1} with infinitesimal increments $dB_{\tau}$ and $d\tau$, respectively,  as
 \[
\overline{({\rm I})}:=\left((\theta, \varpi)(\lfloor {\rm u}_T^s\rceil^{\l})^{(1)}_{\l}(w)\right)'_0=\int_{0}^{T}\big[D\wt{\lf \overrightarrow{F}_{\tau, T}\rc^{\l}}(\lf {\rm u}_{\tau}\rc^{\l}, {w})\big]\left(({\rm I})_1(\tau, w)dB_{\tau}+({\rm I})_2(\tau, w)d\tau\right). 
 \]
Then it is standard to use Burkholder's inequality and H\"{o}lder's inequality to deduce that 
\begin{align*}
2^{1-p}\E\left\|\overline{({\rm I})}\right\|^p\leq &\ \left( \overline{\E}\left|\int_{0}^{T}\left\|\big[D\wt{\lf \overrightarrow{F}_{\tau, T}\rc^{\l}}(\lf {\rm u}_{\tau}\rc^{\l}, {w})\big]({\rm I})_1(\tau, w)\right\|^2d\tau\right|^q\right)^{\frac{1}{2}}\\
&\ +\E\left\|\int_{0}^{T}\big[D\wt{\lf \overrightarrow{F}_{\tau, T}\rc^{\l}}(\lf {\rm u}_{\tau}\rc^{\l}, {w})\big]({\rm I})_2(\tau, w)\ d\tau\right\|^q.\end{align*}
Using Corollary \ref{cor-u-s-t-l-1} and (\ref{DF-j-lambda}), we can continue to estimate $\overline{\E} \left\|({\rm I})_1\right\|^{4q}$, $\overline{\E}\left\|({\rm I})_2\right\|^{2q}$ as in Proposition \ref{Diff-int-U-s} and show  they have same bound type in (\ref{grad-lnp-lam-1}) with $l=0$. 

To  complete  the proof of i), we apply Lemma \ref{weak-reg-p-1}.   It remains to show $(p^{\l})^{(1)}_{\l}(T, x, y)$ is continuous in the $(T, y)$-coordinate,  locally uniformly in $\l$,  which is true if we have\\
\indent 1) the continuity of $y\mapsto (p^{\l})^{(1)}_{\l}(T, x, y)$, locally  uniformly  in $T$ and $\l$, and  \\
\indent 2) the continuity of $T\mapsto (p^{\l})^{(1)}_{\l}(T, x, y)$ for every  $x, y$ fixed, locally uniformly in $\l$. 
\par 
For 1), it holds if the continuity of $y\mapsto (\ln p^{\l})^{(1)}_{\l}(T, x, y)$ is locally uniform in $T$ and $\l$, where the latter  is true if $y\mapsto (\overline{\Phi}_{\l}^1(Y)(F^sy))'_0$ is continuous, locally uniformly in $T$ and $\l$. Since all the bounds in Lemmas \ref{O-mathttg-diff-esti}-\ref {cond-DF-diff-norm}  are locally uniform  in $(y, T)$ and $\l$,  the limits for continuity of $y\mapsto (\overline{\Phi}_{\l}^1(Y)(F^sy))'_0$  in  proof  of Lemma  \ref{C-1-fun-phi-Y}   are all locally  uniform in $T$ and $\l$. 

We proceed to show 2). Simply denote by $({\rm x}^{\l}, {\rm u}^{\l})$ the stochastic pair which defines the Brownian motion starting from $x$. Then for any smooth function $f$ on $\M$ with support contained in a small neighborhood of $y$,
\[
f({\rm x}^{\l}_T)=f(x)+\int_{0}^{T}\Delta^{\l} f({\rm x}^{\l}_t)\ dt+ \int_{0}^{T}H^{\l}({\rm u}_t^{\l}, e_i)\big(\wt{f}({\rm u}_t^{\l})\big)\ dB^i_t. 
\]
Taking expectations on both sides shows 
\[
\E \big(f({\rm x}^{\l}_T)\big)=\int_{0}^{T}\E\big(\Delta^{\l} f({\rm x}^{\l}_t)\big)\ dt. 
\]
Hence for $T'>T$, 
\[
\E \big(f({\rm x}^{\l}_{T'})\big)-\E \big(f({\rm x}^{\l}_T)\big)=\int_{T}^{T'}\E\big(\Delta^{\l} f({\rm x}^{\l}_t)\big)\ dt. 
\]
Differentiating  both sides in $\l$ gives
\begin{align*}
&\int_{\M} f(z)\left((p^{\l})^{(1)}_{\l}(T', x, z)-(p^{\l})^{(1)}_{\l}(T, x, z)\right)\ d{\rm Vol}^{\l}(z)\\
&=- \int_{\M} f(z)\left(p^{\l}(T', x, z)-p^{\l}(T, x, z)\right)(\ln \rho^{\l})^{(1)}_{\l}(z)\ d{\rm Vol}^{\l}(z)+\int_{T}^{T'}\E\left(\big(\Delta^{\l} f\big)^{(1)}_{\l}\!({\rm x}^{\l}_t)\right)\ dt\\
&\ \ \ \ +\int_{T}^{T'} \int_{\M}\big(\Delta^{\l}f\big)(z)\phi^1_{\l}(t, x, z)p^{\l}(t, x, z)\ d{\rm Vol}^{\l}(z),
\end{align*}
where, as $T'\to T$,  the first term tends to zero since $p^{\l}(T, x, z)$ is continuous at $T>0$, locally uniformly in $z$, the second term tends to zero since $\E\big(\big(\Delta^{\l} f\big)^{(1)}_{\l}\!({\rm x}^{\l}_t)\big)$ is uniformly bounded 
for $t$ in a small neighborhood of $T$ and the last term goes to zero as well by using   that  the bound in  (\ref{grad-lnp-lam-1}) with $l=0$ is locally uniform in $t$. In summary, we have 
\[
\lim_{T'\to T}\int_{\M} f(z)\left((p^{\l})^{(1)}_{\l}(T', x, z)-(p^{\l})^{(1)}_{\l}(T, x, z)\right)\ d{\rm Vol}^{\l}(z)=0. 
\]
Since $z\mapsto (p^{\l})^{(1)}_{\l}(T, x, z)$ is continuous, locally  uniformly  in $T$ and $\l$,  and $f$ is arbitrary, we must have   $\lim_{T'\to T}(p^{\l})^{(1)}_{\l}(T', x, y)=(p^{\l})^{(1)}_{\l}(T, x, y)$, locally uniformly in $\l$. This shows 2).

Finally, we show  iii). By symmetry, the mapping $x\mapsto (p^{\l})^{(1)}_{\l}(T, x, y)$ is continuous for all $T,y$, locally uniformly in $y$. Therefore iii) holds for any bounded function with compact support. Fix $q \geq 1$.   Any uniformly continuous and bounded  $\wt{f}\in C(\M)$ can be approximated by  a sequence  $\{\wt{f}_n\}_{n\in \Bbb N}$  of continuous functions on $\M$ with compact support in such a way that
\begin{equation}\label{uniform-f-f-n}\lim_{n\to \infty}\left\|\wt{f}(y)-\wt{f}_n(y)\right\|_q=0,
\end{equation}
locally uniformly in $x$. 
Property iii) follows by using (\ref{uniform-f-f-n}) and (\ref{grad-lnp-lam-1}) with $l=0$.
 \end{proof}

\begin{proof}[Proof of Theorem  \ref{regu-p-1st} ($k>3$)]  
By Theorem \ref{regu-p-1st}  i) of the  $k=3$ case  and Lemma \ref{weak-reg-p-1}, we deduce Theorem  \ref{regu-p-1st} i).  Hence  $\nabla^{(l)}(\ln p^{\l})^{(1)}_{\l}(T, x, \cdot)$, $l\leq k-3$, are well-defined.  By taking the gradients of  the  identity (\ref{Thm1.3-step 1}),  we obtain that  $\nabla^{(l)}\phi_{\l}^1(T, x, \cdot)$, $l\leq k-3$, exist as well.   For (\ref{grad-lnp-lam-1}),  it suffices to show the same type of $L^q$-norm  bounds hold for 
$\nabla^{(l)}\phi_{\l}^1(T, x, \cdot)$,  $l\leq k-3$. 

 The $l=0$ case was treated in the previous  proof  of Theorem  \ref{regu-p-1st} with $k=3$. We proceed to consider the $l=1$ case.  Let ${\rm W}$ be  a smooth bounded vector field  on $\M$ and let $\{\leftidx^{r}F\}_{r\in \Bbb R}$ be the flow it generates. Then 
\[
\nabla^{\l}_{W(y)} \phi_{\l}^1(T, x, \cdot)=\left.\frac{d}{dr}\right|_{r=0}\left(\phi_{\l}^1(T, x, \leftidx^{r}F (y))\right).
\] 
We  will look for a conditional expectation expression of $\nabla^{\l}_{W(y)} \phi_{\l}^1(T, x, \cdot)$ and use it to estimate  $|\nabla\phi_{\l}^1(T, x, \cdot)|$.  For this, we  adopt the idea  we used in analyzing  the regularity  of $\overline{\Phi}_{\l}^1(Y)$  (see Section \ref{Obs-Stra}).  Let $f$  be an arbitrary bounded measurable function  on $\M$. By the definition of the conditional expectation and the change of variable formula under  $\leftidx^{r}F$,  
\begin{align*}
\overline{\E}\left({\wt{\phi}}_{\l}^{1}(T, x, w)f( \lfloor {\rm x}_T \rceil^{\l}(w))\right)=&\ \overline{\E}\left(\overline{\E}\big(\left.{\wt{\phi}}_{\l}^{1}(T, x, w)\right| \lfloor {\rm x}_T \rceil^{\l}(w)=y\big)f(y)\right)\\
=&\int \phi_{\l}^1(T, x, y) f(y) p^{\l}(T, x, y)d{\rm Vol}^{\l}(y)\\
=&\int \phi_{\l}^1(T, x, \leftidx^{r}F (y)) f(\leftidx^{r}F (y)) p^{\l}(T, x, \leftidx^{r}F (y))d{\rm Vol}^{\l}(\leftidx^{r}F (y)).
\end{align*}
Let $\lf\leftidx^{r}{\bf F}\rc^{\l}$ be the extension of $\leftidx^{r} F$ to $C_x([0, T], \M)$ constructed in the  previous subsections. By Proposition \ref{Quasi-IP-2}, all probabilities  $\overline{\P}_x^{\l}\circ \lf\leftidx^{r}{\bf F}\rc^{\l}$  are  absolutely continuous with respect to $\overline{\P}_x^{\l}$. Hence, using the change of variable formula under  $\lf\leftidx^{r}{\bf F}\rc^{\l}$, we obtain 
\begin{align*}
\overline{\E}\left({\wt{\phi}}_{\l}^{1}\cdot f( \lfloor {\rm x}_T \rceil^{\l})\!\right)=&\; \overline{\E}\left({\wt{\phi}}_{\l}^{1} \circ {\lf\leftidx^{r}{\bf{F}}\rc^{\l}} \cdot f \circ {\lf\leftidx^{r}{\bf{F}}\rc^{\l}} \cdot \frac{d\overline{\P}_x^{\l}\circ  {\lf\leftidx^{r}{\bf{F}}\rc^{\l}} }{d\overline{\P}_x^{\l}}\right)\\
=&\;\int \overline{\E}\!\left(\!\left.{\wt{\phi}}_{\l}^{1} \circ  {\lf\leftidx^{r}{\bf{F}}\rc^{\l}} \cdot \frac{d\overline{\P}_x^{\l}\circ  {\lf\leftidx^{r}{\bf{F}}\rc^{\l}} }{d\overline{\P}_x^{\l}}\right|\lfloor {\rm x}_T \rceil^{\l}=y\!\right)\! f(\leftidx^{r} F (y)) p^{\l}(T, x, y)d{\rm Vol}^{\l}(y).
\end{align*}
Since $f$ is arbitrary, a comparison of the  two expressions of $\overline{\E}\big({\wt{\phi}}_{\l}^{1}\cdot f( \lfloor {\rm x}_T \rceil^{\l})\big)$  shows 
\begin{align*}
\phi_{\l}^1(T, x, \leftidx^{r}F (y))=\overline{\E}\!\left(\!\left.{\wt{\phi}}_{\l}^{1} \circ  {\lf\leftidx^{r}{\bf{F}}\rc^{\l}} \cdot \frac{d\overline{\P}_x^{\l}\circ  {\lf\leftidx^{r}{\bf{F}}\rc^{\l}} }{d\overline{\P}_x^{\l}}\right|\lfloor {\rm x}_T \rceil^{\l}=y\!\right)\cdot\frac{p^{\l}(T, x, y)}{p^{\l}(T, x, \leftidx^{r}F (y))}\cdot\frac{d{\rm Vol}^{\l}}{d{\rm Vol}^{\l}\circ \leftidx^{r}F}(y). 
\end{align*}
So differentiating both sides in $r$ at $r=0$ gives 
\begin{align*}
\nabla^{\l}_{W(y)} \phi_{\l}^1(T, x, \cdot)=&\ \phi_{\l}^1(T, x, y)\left(\big(\ln p^{\l}(T, x,  \leftidx^{r}F (y))\big)'_{0}+\big(\ln \rho^{\l}(\leftidx^{r}F (y))\big)'_{0}\right)\\
&+\left(\overline{\E}\!\left(\!\left.{\wt{\phi}}_{\l}^{1} \circ  {\lf\leftidx^{r}{\bf{F}}\rc^{\l}} \cdot \frac{d\overline{\P}_x^{\l}\circ  {\lf\leftidx^{r}{\bf{F}}\rc^{\l}} }{d\overline{\P}_x^{\l}}\right|\lfloor {\rm x}_T \rceil^{\l}=y\!\right)\right)'_0. 
\end{align*}
It was shown in  Proposition \ref{Quasi-IP-2}    that $d\overline{\P}_x^{\l}\circ  {\lf\leftidx^{r}{\bf{F}}\rc^{\l}}/d\overline{\P}_x^{\l}$ is differentiable in $r$ with 
\[
\big(d\overline{\P}_x^{\l}\circ  {\lf\leftidx^{r}{\bf{F}}\rc^{\l}}/d\overline{\P}_x^{\l}\big)'_r=\leftidx^{r}\overline{\mathcal{E}_t}\cdot \big(d\overline{\P}_x^{\l}\circ  {\lf\leftidx^{r}{\bf{F}}\rc^{\l}}/d\overline{\P}_x^{\l}\big) 
\]
and both $\leftidx^{r}\overline{\mathcal{E}_t}$ and  $\big(d\overline{\P}_x^{\l}\circ  {\lf\leftidx^{r}{\bf{F}}\rc^{\l}}/d\overline{\P}_x^{\l}\big)$  conditioned on ${\rm x}_T=y$  are  $L^q$ $(q\geq 1)$ integrable, locally uniformly  in the $r$ parameter.  Using H\"{o}lder's inequality,  if we can further show 
\begin{itemize}
\item[$\bigstar$)] ${\wt{\phi}}_{\l}^{1} \circ  {\lf\leftidx^{r}{\bf{F}}\rc^{\l}}$  is also differentiable in $r$ with both ${\wt{\phi}}_{\l}^{1} \circ  {\lf\leftidx^{r}{\bf{F}}\rc^{\l}}$ and $({\wt{\phi}}_{\l}^{1}\circ  {\lf\leftidx^{r}{\bf{F}}\rc^{\l}})'_r$ conditioned on ${\rm x}_T=y$ are $L^2$  integrable, locally uniformly in the $r$ parameter,  
\end{itemize}
we are allowed to take the differentiation under the expectation sign: 
\begin{align*}
\left(\!\overline{\E}\!\left(\!\left. {\wt{\phi}}_{\l}^{1} \circ  {\lf\leftidx^{r}{\bf{F}}\rc^{\l}} \cdot \frac{d\overline{\P}_x^{\l}\circ  {\lf\leftidx^{r}{\bf{F}}\rc^{\l}} }{d\overline{\P}_x^{\l}}\right|\lfloor {\rm x}_T \rceil^{\l}=y\!\right)\!\right)'_0=&\ \overline{\E}\!\left(\!\left.\left({\wt{\phi}}_{\l}^{1} \circ  {\lf\leftidx^{r}{\bf{F}}\rc^{\l}} \cdot \frac{d\overline{\P}_x^{\l}\circ  {\lf\leftidx^{r}{\bf{F}}\rc^{\l}} }{d\overline{\P}_x^{\l}}\right)'_0\right|\lfloor {\rm x}_T \rceil^{\l}=y\!\right)\\
=&\ \overline{\E}\left(\left. \big({\wt{\phi}}_{\l}^{1} \circ  {\lf\leftidx^{r}{\bf{F}}\rc^{\l}}\big)'_0+ {\wt{\phi}}_{\l}^{1} \cdot \leftidx^{0}\overline{\mathcal{E}_T} \right|\lfloor {\rm x}_T \rceil^{\l}=y \right).
\end{align*}
Altogether, we will have 
\begin{align}
\nabla^{\l}_{W(y)} \phi_{\l}^1(T, x, \cdot)=&\ \phi_{\l}^1(T, x, y)\left(\nabla_{W(y)}^{\l} (\ln p^{\l}(T, x, \cdot))+\nabla_{W(y)}^{\l} (\ln \rho^{\l})\right)\notag\\
&\ + \overline{\E}\left(\left. \big({\wt{\phi}}_{\l}^{1} \circ  {\lf\leftidx^{r}{\bf{F}}\rc^{\l}}\big)'_0+ {\wt{\phi}}_{\l}^{1} \cdot \leftidx^{0}\overline{\mathcal{E}_T} \right|\lfloor {\rm x}_T \rceil^{\l}=y \right) \label{nab-ln-p-1-diff}
\end{align}
and we can use it  to show that  a $L^q$-norm bound  as in (\ref{grad-lnp-lam-1}) is valid for  $\nabla^{\l}\phi_{\l}^1(T, x, \cdot)$.

We show $\bigstar)$ first. 
Consider the  processes  \[
 {\lf \leftidx^{r}{\mho_{\tau}}\rc^{\l}}:=  {\lf {\mho_{\tau}}\rc^{\l}}\circ {\lf \leftidx^{r}{\bf F}\rc^{\l}},\   \big[D{\lf \leftidx^{r} F_{\tau, t}\rc^{\l}} ( {\lf \leftidx^{r}{\mho_{\tau}}\rc^{\l}}, {\rm w})\big]^{-1}:=\big[D\lf F_{\tau, t}\rc^{\l}(\lf {\mho}_\tau\rc^{\l}, {\rm w})\big]^{-1}\circ {\lf\leftidx^r {\bf{F}}\rc^{\l}}. 
\]
They are  well-defined by Theorem \ref{Main-alpha-x-v-Q} and the corresponding estimations in  Lemmas \ref{O-mathttg-diff-esti}-\ref{cond-DF-diff-norm} (for  ${\lf\leftidx^{r}{\bf{F}}\rc^{\l}}$) are valid. Note that $D\pi(\lfloor {\rm u}_T\rceil^{\l})^{(1)}_{\l}$,  $\nabla^{\l}_{T, {\rm V}, \mathtt{s}}D\pi(\lfloor {\rm u}_T\rceil^{\l})^{(1)}_{\l}$ can be expressed by stochastic integrals using $ {\lf {\mho_{\tau}}\rc^{\l}}$ and ${\lf {\bf{F}}\rc^{\l}}$ (see Proposition \ref{Diff-int-U-s} and Proposition \ref{u-s-t-l-1}). Their images under  ${\lf \leftidx^{r}{\bf F}\rc^{\l}}$  can be defined by applying ${\lf \leftidx^{r}{\bf F}\rc^{\l}}$ to each components in the integrals.  So ${\wt{\phi}}_{\l}^{1} \circ  {\lf\leftidx^{r}{\bf{F}}\rc^{\l}}$ is well-defined. By using Lemma \ref{Hs2-ST}, Proposition \ref{est-norm-u-t-j} ii) and  (\ref{wt-phi-1}),  it is  easy to  obtain 
\begin{align*}
\overline{\E}_{\overline{\P}^{\l, *}_{x, y, T}}\big|{\wt{\phi}}_{\l}^{1} \circ  {\lf\leftidx^{r}{\bf{F}}\rc^{\l}}\big|^2\leq \underline{\mathtt c}\left(\big(\frac{1}{T}d_{\wt{g}^{\l}}(x, y)+\frac{1}{\sqrt{T}}\big)^2+1\right) e^{\mathtt c(1+d_{\wt{g}^{\l}}(x, y))} 
\end{align*}
for some constants $\underline{\mathtt c}$ (depending on $\mathtt{s}, r_0$, $\|g^{\l}\|_{C^3}$ and  $\|\XX^{\l}\|_{C^2}$) and  $\mathtt c$  (depending on $T, T_0$ and  $\|g^{\l}\|_{C^3}$).  By Propositions \ref{cond-nabla-ln-p}, \ref{Quasi-IP-1} and \ref{Quasi-IP-2},   we may also  assume $\underline{\mathtt c}, \mathtt c$ are such that 
 \[
\big(\overline{\E}_{\overline{\P}^{\l, *}_{x, y, T}}\big|\leftidx^{0}\overline{\mathcal{E}_T}\big|^2\big)^{\frac{1}{2}}\leq \|W(y)\| \underline{\mathtt c} e^{\mathtt c(1+d_{\wt{g}^{\l}}(x, y))}. 
 \]
To justify (\ref{nab-ln-p-1-diff}), it remains to check  the differentiability  of $
r\mapsto {\rm{A}}\circ {\lf\leftidx^r {\bf{F}}\rc^{\l}}$, for ${\rm{A}}=({\rm I}), ({\rm II}), ({\rm III}), ({\rm IV})$ in (\ref{wt-phi-1}) and show the differentials $\big({\rm{A}}\circ {\lf \leftidx^{r} {\bf{F}}\rc^{\l}} \big)'_r$ are $L^2$ integrable, uniformly in the $r$ parameter.  We begin with ${\rm{A}}=({\rm IV})$.  By Proposition \ref{Diff-int-U-s-gen},  $(D\pi(\lfloor {\rm u}_T\rceil^{\l})^{(1)}_{\l})\circ \lf \leftidx^{r}{\bf F}\rc^{\l}$ is differentiable in $r$.  Let $r\in [-r_0, r_0]$. As usual,  we write 
\[
{\lf \leftidx^{r}{\rm{x}}\rc^{\l}}:=  {\lf {\rm{x}}\rc^{\l}}\circ {\lf \leftidx^{r}{\bf F}\rc^{\l}},\ \  {\lf \leftidx^{r}{\rm{u}}\rc^{\l}}:=  {\lf {\rm{u}}\rc^{\l}}\circ {\lf \leftidx^{r}{\bf F}\rc^{\l}}, \ \nabla_{T, {\rm W}, \mathtt{s}}^{r, \l}D\pi(\lfloor {\rm u}_T\rceil^{\l})^{(1)}_{\l}:=\left((D\pi(\lfloor {\leftidx^{r}{\rm u}}_T\rceil^{\l})^{(1)}_\l)\right)'_{r}. 
\]
Then $({\rm IV})\circ \lf \leftidx^{r}{\bf F}\rc^{\l}$ is differentiable in $r$ with differential 
\begin{align*}
\big(({\rm{IV}})\circ {\lf \leftidx^{r} {\bf{F}}\rc^{\l}} \big)'_r=&-\big\langle \nabla_{T, {\rm W}, \mathtt{s}}^{r, \l}D\pi({\lfloor {\rm u}_T\rceil^{\l})^{(1)}_{\l}},\  \nabla^{\l}\ln p^{\l}(T, x, {\lfloor \leftidx^r {{\rm{x}}}_T \rceil^{\l}})\big\rangle_{\l}\\
&-\big\langle D\pi({\lfloor {\leftidx^r {\rm{u}}}_T\rceil^{\l})^{(1)}_{\l}}, \  \nabla^{\l}_{{{\rm{W}}({\lfloor {\leftidx^r {\rm{x}}}_T \rceil^{\l}} )}}\nabla^{\l}\ln p^{\l}(T, x, {\lfloor {\leftidx^r {\rm{x}}}_T \rceil^{\l}} )\big\rangle_{\l}.
\end{align*}
By Proposition \ref{est-norm-u-t-j}, we can obtain some  $\underline{\mathtt c}'$ (depending on $\mathtt{s}, r_0$, $\|g^{\l}\|_{C^3}$ and  $\|\XX^{\l}\|_{C^1}$) and  $\mathtt c'$  (depending on $T, T_0$ and  $\|g^{\l}\|_{C^3}$) such that \[
\left(\overline{\E}_{\overline{\P}^{\l, *}_{x, y, T}}\big\|D\pi({\lfloor  {\leftidx^r {\rm u}}_T\rceil^{\l})^{(1)}_{\l}} \big\|^2\right)^{\frac{1}{2}}\leq \|W(y)\| \underline{\mathtt c}'e^{\mathtt c'(1+d_{\wt{g}^{\l}}(x, y))}. 
\]
By Proposition \ref{Diff-int-U-s}, we can obtain some $\underline{\mathtt c}''$ (depending on $\mathtt{s}, r_0$, $\|g^{\l}\|_{C^3}$ and $\|\XX^{\l}\|_{C^2}$) and  $\mathtt c''$  (depending on $T, T_0$ and  $\|g^{\l}\|_{C^3}$) such that 
\[
\left(\overline{\E}_{\overline{\P}^{\l, *}_{x, y, T}}\big\|\nabla_{T, {\rm W}, \mathtt{s}}^{r, \l}D\pi(\lfloor {\rm u}_T\rceil^{\l})^{(1)}_{\l}\big\|^2\right)^{\frac{1}{2}}\leq \|W(y)\|\underline{\mathtt c}'' e^{\mathtt c''(1+d_{\wt{g}^{\l}}(x, y))}.\]
Using  (\ref{equ-Hs2-ST}), we  further  obtain   
\begin{align*}
&\!\left(\overline{\E}_{\overline{\P}^{\l, *}_{x, y, T}}\big|\big(({\rm{IV}})\circ {\lf \leftidx^{r} {\bf{F}}\rc^{\l}} \big)'_r\big|^2\right)^{\frac{1}{2}}\!\!\\
&\leq \left(\overline{\E}_{\overline{\P}^{\l, *}_{x, y, T}}\big\|\nabla_{T, {\rm W}, \mathtt{s}}^{r, \l}D\pi(\lfloor {\rm u}_T\rceil^{\l})^{(1)}_{\l}\big\|^2\right)^{\frac{1}{2}} \big\| \nabla^{\l}\ln p^{\l}(T, x,  \lf {\leftidx^r F}\rc^{\l}(y))\big\|\\
&\ \ \ \  +\left(\overline{\E}_{\overline{\P}^{\l, *}_{x, y, T}}\big\|D\pi({\lfloor  \leftidx^r {\rm u}_T\rceil^{\l})^{(1)}_{\l}} \big\|^2\right)^{\frac{1}{2}} \big\| \nabla^{\l}_{W( \lf {\leftidx^r F}\rc^{\l}(y) )}\!\nabla^{\l}\ln p^{\l}(T, x,  \lf {\leftidx^r F}\rc^{\l}(y))\big\|\\
 &\leq \ \|W(y)\| \underline{\mathtt c}'''e^{\mathtt c'''(1+d_{\wt{g}^{\l}}(x, y))}\sum_{i=1}^2\big(\frac{1}{T}d_{\wt{g}^{\l}}(x, \lf {\leftidx^r F}\rc^{\l}(y))+\frac{1}{\sqrt{T}}\big)^i,
\end{align*}
where $\underline{\mathtt c}'''$ (depending on $\mathtt{s}, r_0$, $\|g^{\l}\|_{C^3}$ and  $\|\XX^{\l}\|_{C^2}$) and  $\mathtt c'''$  (depending on $T, T_0$ and  $\|g^{\l}\|_{C^3}$) 
and this  bound is finite and is uniform in $r$.  For ${\rm A}=({\rm II}), ({\rm III})$, the same argument shows  the $C^1$ regularity of  $r\mapsto {\rm{A}}\circ {\lf\leftidx^r {\bf{F}}\rc^{\l}}$ and 
\[
\left(\overline{\E}_{\overline{\P}^{\l, *}_{x, y, T}}\big|\big({\rm{A}}\circ {\lf \leftidx^{r} {\bf{F}}\rc^{\l}} \big)'_r\big|^2\right)^{\frac{1}{2}}\leq \|W(y)\| \underline{c}_{{\rm A}}e^{c_{{\rm A}}(1+d_{\wt{g}^{\l}}(x, y))}
\]
for some $\underline{c}_{{\rm A}}$ (depending on $\mathtt{s}, r_0$, $\|g^{\l}\|_{C^3}$ and  $\|\XX^{\l}\|_{C^1}$) and  $c_{{\rm A}}$  (depending on $T, T_0$ and  $\|g^{\l}\|_{C^3}$). 
It remains to analyze $({\rm I})\circ {\lf\leftidx^r {\bf{F}}\rc^{\l}}$.  Recall that for any smooth bounded vector field ${\rm V}$ on $\M$, 
\[
\nabla^{\l}_{T, {\rm V}, \mathtt{s}}D\pi(\lfloor {\rm u}_T\rceil^{\l})^{(1)}_{\l}=D\pi \big(\overline{(\lfloor {\rm u}_T^s\rceil^{\l})^{(1)}_{\l}}({\rm w})\big)'_0,
\]
where $\big(\overline{(\lfloor {\rm u}_T^s\rceil^{\l})^{(1)}_{\l}}({\rm w})\big)'_0$ was formulated in (\ref{the U-T-s-1}). Hence the regularity of $r\mapsto {(\rm{I})}\circ {\lf\leftidx^r {\bf{F}}\rc^{\l}}$ can be reduced to the regularity of each component of (\ref{the U-T-s-1}) under  ${\lf\leftidx^r {\bf{F}}\rc^{\l}}$. 
Applying Theorem \ref{Main-alpha-x-v-Q} to $ {\lf\leftidx^r {\bf{F}}\rc^{\l}}$ shows  $r\mapsto  {\lf \leftidx^{r}{\mho_{\tau}}\rc^{\l}}, \big[D{\lf \leftidx^{r} F_{\tau, t}\rc^{\l}} ( {\lf \leftidx^{r}{\mho_{\tau}}\rc^{\l}}, {\rm w})\big]^{-1}$ are $C^1$.  Lemmas \ref{O-mathttg-diff-esti}-\ref{cond-DF-diff-norm} also hold true for ${\lf\leftidx^r {\bf{F}}\rc^{\l}}$.  Using these properties and  the fact that $\l\mapsto g^{\lambda}$ is  $C^k$ in  $\mathcal{M}^k(M)$ with $k\geq 4$, we can deduce the regularity of the components of (\ref{the U-T-s-1}) under  ${\lf\leftidx^r {\bf{F}}\rc^{\l}}$. Moreover, by a routine computation using Lemmas \ref{O-mathttg-diff-esti}-\ref{cond-DF-diff-norm}, we can obtain some $\underline{c}_{{\rm I}}$ (depending on $\mathtt{s}, r_0$, $\|g^{\l}\|_{C^4}$ and  $\|\XX^{\l}\|_{C^3}$) and $c_{{\rm I}}$ (depending on $T, T_0$ and  $\|g^{\l}\|_{C^3}$) such that 
\[
\left(\overline{\E}_{\overline{\P}^{\l, *}_{x, y, T}}\left|\big({(\rm{I})}\circ {\lf \leftidx^{r} {\bf{F}}\rc^{\l}} \big)'_r\right|^2\right)^{\frac{1}{2}}\leq \|W(y)\| \underline{c}_{{\rm I}} e^{c_{{\rm I}}(1+d_{\wt{g}^{\l}}(x, y))}. 
\]
Altogether, we have the differentiability of $\l\mapsto {\wt{\phi}}_{\l}^{1} \circ  {\lf\leftidx^{r}{\bf{F}}\rc^{\l}}$ and also  obtain some   $\underline{c}$ (depending on $\mathtt{s}, r_0$, $\|g^{\l}\|_{C^4}$ and $\|\XX^{\l}\|_{C^3}$) and $c$ (depending on $T, T_0$ and  $\|g^{\l}\|_{C^3}$) such that 
 \begin{equation*}
\! \left(\overline{\E}_{\overline{\P}_{x, y, T}^{\l, *}}\!\left|\big({\wt{\phi}}_{\l}^{1} \circ  {\lf\leftidx^{r}{\bf{F}}\rc^{\l}}\big)'_0\right|^2\right)^{\frac{1}{2}} \leq \|W(y)\| \underline{c}e^{c(1+d_{\wt{g}^{\l}}(x, y))}\left( (\frac{1}{T}d_{\wt{g}^{\l}}(x, y)+\frac{1}{\sqrt{T}})^2+1 \right).
 \end{equation*}

Now (\ref{nab-ln-p-1-diff}) holds true. Using H\"{o}lder's inequality, it is easy to deduce
\begin{align*}
4^{1-q}\left\|\nabla^{\l}_{W(y)} \phi_{\l}^1(T, x, \cdot)\right\|_{L^q}^q\leq & \left\|\phi_{\l}^1(T, x, \cdot)\right\|_{L^{2q}}^{q}\!\left(\left\|\nabla_{W(y)}^{\l} (\ln p^{\l}(T, x, \cdot))\right\|_{L^{2q}}^{q}\!\!+\left\|\nabla_{W(y)}^{\l} (\ln \rho^{\l})\right\|_{L^{2q}}^{q}\right)\\
&+\big(\E\big\|{\wt{\phi}}_{\l}^{1} \big\|^{2q}\big)^{\frac{1}{2}}\big(\E\left\| \leftidx^{0}\overline{\mathcal{E}_T}\right\|^{2q}\big)^{\frac{1}{2}}+ \overline{\E}\big\|\big({\wt{\phi}}_{\l}^{1} \circ  {\lf\leftidx^{r}{\bf{F}}\rc^{\l}}\big)'_0\big\|^q\\
=: &\  {\rm{D}}_1(q) +  {\rm{D}}_2(q) +  {\rm{D}}_3(q).
\end{align*}
By (\ref{equ-Hs2-ST}), (\ref{S-p-t-upper-b}), we see that,  for the $i$-th covariant derivative $\nabla^{\l, (i)}\ln p^{\l}(t, x, \cdot)$, $i\leq k-2$,  there is $\underline{c}(i)$ (depending on $m, q, T_0$ and $\|g^{\l}\|_{C^{i+2}}$) and $c(i)$ (depending on $\|g^{\l}\|_{C^3}$) such that 
\begin{equation}\label{L-q-grad-lnp}
\left\|\nabla^{\l, (i)}\ln p^{\l}(T, x, \cdot)\right\|_{L^q}^q\leq \underline{c}(i)e^{c(i)(1+T)}.
\end{equation}
Using this and the $L^q$ estimation of  $\phi_{\l}^1(T, x, \cdot)$ in the proof of Theorem \ref{regu-p-1st} for the $k=3$ case,  we obtain  $
{\rm{D}}_1(q)\leq \underline{c}_{\l}(q)$, 
where $\underline{c}_{\l}(q)$ depends  on $m, q$, $T_0, T$,  $\|g^{\l}\|_{C^3}$ and  $\|\XX^{\l}\|_{C^2}$. With the $L^{2q}$ estimations of $\wt{\phi}_{\l}^1(T, x, \cdot)$ and $\leftidx^{0}\overline{\mathcal{E}_T}$ for Theorem \ref{regu-p-1st} with  $k=3$, we can also conclude that ${\rm{D}}_2(q)$ has the same type of bound as $ {\rm{D}}_1(q)$.  For $ {\rm{D}}_3(q)$, we check  the $L^q$-norms of $\big({\rm{A}}\circ {\lf \leftidx^{r} {\bf{F}}\rc^{\l}} \big)'_0$  for ${\rm{A}}=({\rm I}), ({\rm II}), ({\rm III})$ or $ ({\rm IV})$ in (\ref{wt-phi-1}), respectively.  Using  H\"{o}lder's inequality, (\ref{L-q-grad-lnp}) and Proposition \ref{est-norm-u-t-j}, it suffices to estimate the $L^q$-norms of 
\[
\left(({\lfloor  \leftidx^r {\rm u}_T\rceil^{\l})^{(1)}_{\l}}\right)'_{r=0}, \ \left(\big(\overline{(\lfloor {\rm u}_T^s\rceil^{\l})^{(1)}_{\l}}({\rm w})\big)'_0\circ  {\lf\leftidx^{r}{\bf{F}}\rc^{\l}}\right)'_{r=0}. 
\]
This, by using Lemma \ref{all-diff-at-s-0} and  Proposition \ref{u-s-t-l-1},  can be eventually reduced to a multiple of a constant depending on $m, q$, $T_0$, $T$,   $\|g^{\l}\|_{C^4}$ and $\|\XX^{\l}\|_{C^3}$ with a combination of some   $L^{q'}$  norm estimations (with $q'\geq 1$ depending on $q$) of  $\sup_{0\leq t\leq T}\|(\lf {\rm{u}}_{t}\rc^{\l})^{(1)}_{\l}\|$ and  $\sup_{0\leq \underline{t}<\overline{t}\leq T}\big\|\big[ D\wt{\lf \overrightarrow{F}_{\underline{t}, \overline{t}}\rc^{\l}}(\lf {\rm{u}}_{\underline{t}}\rc^{\l}, {w})\big]\big\|$. Hence, by  Proposition \ref{est-norm-u-t-j}, we conclude that  $ {\rm{D}}_3(q)$ has the same type of bound as ${\rm{D}}_1(q)$ with $\underline{c}_{\l}(q)$  depending on $m, q$, $T_0$, $T$,  $\|g^{\l}\|_{C^4}$ and $\|\XX^{\l}\|_{C^3}$. 
 
 For the $L^q$-norm estimation of $\nabla^{(2)} \phi_{\l}^1(T, x, \cdot)$, we continue to differentiate (\ref{nab-ln-p-1-diff}). Let $W_2$ be another smooth bounded vector field on $\M$. Then 
 \begin{align*}
\nabla^{\l}_{W_2(y)} \nabla^{\l}_{W(y)}\phi_{\l}^1(T, x, \cdot)=&\phi_{\l}^1(T, x, \cdot)\left(\nabla^{\l}_{W_2(y)}\nabla_{W(y)}^{\l}(\ln p^{\l}(T, x, \cdot))+\nabla^{\l}_{W_2(y)}\nabla_{W(y)}^{\l} (\ln \rho^{\l})\right)\\
&+\nabla^{\l}_{W_2(y)}\phi_{\l}^1(T, x, \cdot) \left(\nabla_{W(y)}^{\l}(\ln p^{\l}(T, x, \cdot))+\nabla_{W(y)}^{\l} (\ln \rho^{\l})\right)\\
&+\nabla^{\l}_{W_2(y)}\left(\overline{\E}\left(\left. \big({\wt{\phi}}_{\l}^{1} \circ  {\lf\leftidx^{r}{\bf{F}}\rc^{\l}}\big)'_0+ {\wt{\phi}}_{\l}^{1} \cdot \leftidx^{0}\overline{\mathcal{E}_T} \right|\lfloor {\rm x}_T \rceil^{\l}=y \right)\right)\\
=:&({\bf a})_y+({\bf b})_y+({\bf c})_y. 
 \end{align*}
Using the previous estimations of $\phi^1_{\l}$,  $\nabla^{\l}_{W}\phi^1_{\l}$, (\ref{L-q-grad-lnp}) and H\"{o}lder's inequality,  we obtain
 \[
 \big\|({\bf a})_y\big\|_{L^q}, \big\|({\bf b})_y\big\|_{L^q}\leq \|W_2(y)\|\|W(y)\|\underline{c}_{{\bf a, b}}e^{c_{{\bf a, b}}T},
 \]
 where $\underline{c}_{{\bf a, b}}$ depends on $m, q, \|g^{\l}\|_{C^4}, \|\XX^{\l}\|_{C^3}$ and $c_{{\bf a, b}}$ depends on $m, q, T, T_0$ and $\|g\|_{C^3}$.  For $({\bf c})_y$, we can follow the above argument for $\nabla^{\l}_{W(y)}\phi_{\l}^1(T, x, \cdot)$ to `exchange' the differentiation $\nabla^{\l}_{W_2(y)}$ with the conditional expectation sign and obtain 
 \begin{align*}
 ({\bf c})_y=&\overline{\E}\left(\left. \left.\frac{d}{da}\right|_{a=0}\left(\big( \big({\wt{\phi}}_{\l}^{1} \circ  {\lf\leftidx^{r}{\bf{F}}\rc^{\l}}\big)'_0+ {\wt{\phi}}_{\l}^{1} \cdot \leftidx^{0}\overline{\mathcal{E}_T}\big)\circ {\lf\leftidx^{a}{{\bf{F}}}^{W_2}\rc^{\l}}\right) \right|\lfloor {\rm x}_T \rceil^{\l}=y \right)\\
 &+\overline{\E}\left(\left. \big( \big({\wt{\phi}}_{\l}^{1} \circ  {\lf\leftidx^{r}{\bf{F}}\rc^{\l}}\big)'_0+ {\wt{\phi}}_{\l}^{1} \cdot \leftidx^{0}\overline{\mathcal{E}_T}\big)\leftidx^{0}\overline{\mathcal{E}_{T}^{W_2}}  \right|\lfloor {\rm x}_T \rceil^{\l}=y \right),
 \end{align*}
 where  ${\lf\leftidx^{a}{{\bf{F}}}^{W_2}\rc^{\l}}, \leftidx^{0}\overline{\mathcal{E}_{T}^{W_2}}$  are the corresponding objects  ${\lf\leftidx^{a}{{\bf{F}}}\rc^{\l}}, \leftidx^{0}\overline{\mathcal{E}_{T}}$ for $W_2$.  In addition to the terms involving a single differentiation of ${\lf\leftidx^{a}{{\bf{F}}}^{W_2}\rc^{\l}}$ or ${\lf\leftidx^{r}{{\bf{F}}}^{W_1}\rc^{\l}}$,  we have  the  differentiation of $\big({\wt{\phi}}_{\l}^{1} \circ  {\lf\leftidx^{r}{\bf{F}}\rc^{\l}}\big)'_0$ under ${\lf\leftidx^{a}{{\bf{F}}}^{W_2}\rc^{\l}}$, which  involves $\nabla^{\l}_{W_2(y)}\nabla^{\l}_{W(y)} \nabla^{\l}\ln p^{\l}(T, x, \cdot)$ and multi-stochastic integrals using  the tangent maps $\big[D{\lf F_{\tau, t}\rc^{\l}} ( {\lf {\mho_{\tau}}\rc^{\l}}, {\rm w})\big]$ and geometric terms with bounds determined by  $\|g^{\l}\|_{C^5}$ and $\|\XX^{\l}\|_{C^4}$. So, a routine calculation as above using  Proposition \ref{est-norm-u-t-j} gives
 \[
 \big\|({\bf c})_y\big\|_{L^q}\leq \|W_2(y)\|\|W(y)\|\underline{c}_{{\bf c}}e^{c_{{\bf c}}T},\]
 where $\underline{c}_{{\bf c}}$ depends on $m, q, \|g^{\l}\|_{C^5}$ and  $\|\XX^{\l}\|_{C^4}$,  and $c_{{\bf c}}$ depends on $m, q, T, T_0$  and  $\|g\|_{C^3}$.
  
Continuing  this argument, we can  obtain the estimations in (\ref{grad-lnp-lam-1}) for all $l\leq k-3$. We stop at $l=k-3$  step since  $\nabla^{(l)}\phi_{\l}^1(T, x, \cdot)$ involves $\nabla^{(l+1)}(\ln p^{\l})(T, x, \cdot)$ and the bound estimation in  (\ref{equ-Hs2-ST}) is only valid  for $\nabla^{(l)}(\ln p^{\l})(T, x, \cdot)$, $l\leq k-2$,  in general.  
\end{proof}

In proving (\ref{esti-p-lam-der-k}), we also  obtain the following coarse estimation,  which will be used in the inductive argument in the next section.  
\begin{cor}\label{di-p-lam-1} For all $l$, $0\leq l\leq k-3$, there is $\underline{c}_{\l, (l, 1)}$,  depending  on $m$, $\|g^{\l}\|_{C^{l+3}}$ and $\|\XX^{\l}\|_{C^{l+2}}$,   and $c^{\l, (l, 1)}$, depending  on $l$, $m, q$, $T, T_0$ and  $\|g^{\l}\|_{C^3}$,  such that 
\begin{align*}
\ \ \ \ \ \ \  &\left|\nabla^{(l)}(\ln p^{\l})^{(1)}_{\l}(T, x, y)\right|, \left|\nabla^{(l)}\phi_{\l}^{1}(T, x, y)\right|\notag\\
&\ \ \ \ \ \leq (p^{\l}(T, x, y))^{-1}\underline{c}_{\l, (l, 1)}\left(\big(\frac{1}{T}d_{\wt{g}^{\l}}(x, y)+\frac{1}{\sqrt{T}}\big)^{l+1}+1\right)\cdot  e^{c^{\l, (l, 1)}(1+d_{\wt{g}^{\l}}(x, y))}.
\end{align*}
\end{cor}

\section{Higher order regularity of the heat kernels in metrics}
To conclude Theorem \ref{diff-HK-estimations-gen} for all $i$, $2\leq i\leq k-2$, we use an  inductive argument  based on  the proof of  Theorem  \ref{regu-p-1st} to  identify  the differentials $(p^{\l})^{(i)}_{\l}(T, x, \cdot)$, $2\leq i\leq k-2$, using the SDE theory in Section 4.  The estimations in (\ref{esti-p-lam-der-k}) and (\ref{p-lam-i-p-equ}) will be obtained using the conditional  stochastic expressions of $\{(\ln p^{\l})^{(i)}_{\l}(T, x, \cdot)\}$.   In the following, we  first pick out the  properties  of  $(p^{\l})^{(i)}_{\l}(T, x, \cdot)$ necessary  for an  inductive argument, then  verify these properties for the  $i=2$ case  and the  $i>2$ case, respectively.

\subsection{A sketch of the  proof for Theorem 1.3 with  $i\geq 2$}\label{skectch-6.1} 
\begin{lem}\label{weak-reg-p-geq1}The {\rm i)} of Theorem \ref{diff-HK-estimations-gen} holds true if  there are locally  absolutely  integrable functions $\{\phi_{\l}^{i}(T, x, y)\}_{x\in \M, T\in \Bbb R_+, i\leq k-2}$  on $\M$, which  are continuous in the $\l$-parameter  and are continuous in  the  $(T, y)$-parameter, locally uniformly in $\l$,  such that for any  $f\in C_c^{\infty}(\M)$,
\begin{equation}\label{p-diff-induction-i}
\left(\int_{\M}f(y)p^{\l}(T, x, y)\ d{\rm{Vol}}^{\l}(y)\right)^{(i)}_{\l}=\int_{\M}f(y) \phi_{\l}^{i}(T, x, y)p^{\l}(t, x, y)\ d{\rm{Vol}}^{\l}(y).
\end{equation}
\end{lem}

\begin{proof}Assume  (\ref{p-diff-induction-i}) holds true.  We show the differentials $(p^{\l})^{(i)}_{\l}(T, x, \cdot)$, $i=1, \cdots, k-2$, exist as continuous functions on $\M$ and satisfy
\begin{equation}\label{heat-induction-1}
\sum_{i=0}^{j}\left(\begin{matrix}j\\ i\end{matrix}\right)(p^{\l})^{(i)}_{\l}(T, x, y)(\rho^{\l})^{(j-i)}(y)=\phi_{\l}^{j}(T, x, y)p^{\l}(T, x, y)\rho^{\l}(y),\  j=1, \cdots, k-2.
\end{equation}
The $j=1$ case was handled in Lemma \ref{weak-reg-p-1} and we know that   $(p^{\l})^{(1)}_{\l}(T, x, \cdot)$ is a continuous function on $\Bbb R_+\times \M$. Assume $(p^{\l})^{(i)}_{\l}(T, x, \cdot)$, $i\leq j_0<k-2$, exist, are continuous, and satisfy (\ref{heat-induction-1}) for $j\leq j_0$. Using this,  a comparison of (\ref{p-diff-induction-i}) for $i=j_0$ and $j_0+1$ gives
\begin{align*}
&\int_{\M}\left(\phi_{\l}^{j_0}(T, x, y)p^{{\l}}(T, x, y)\rho^{\l}(y)-\phi_{0}^{j_0}(T, x, y)p^{0}(T, x, y)\rho^0(y)\right)f(y) \ d{\rm Vol}^0(y)\\
&\ =\int_{\M}\left(\int_{0}^{\l}\phi_{\wt{\l}}^{j_0+1}(T, x, y)p^{\wt{\l}}(T, x, y)\rho^{\wt{\l}}(y)d\wt{\l} \right) f(y)\ d{\rm Vol}^0(y), \ \forall  f\in C_c^{\infty}(\M). 
\end{align*}
Since both sides are continuous functions in $y$-variable, we must have
\begin{align*}&\phi_{\l}^{j_0}(T, x, y)p^{{\l}}(T, x, y)\rho^{\l}(y)-\phi_{0}^{j_0}(T, x, y)p^{0}(T, x, y)\rho^0(y)\\
& \ \ \ \ \ \ \ \ \ \ \ \ \ \ \ \ \ \ \ \ \ \ \ \ \ \ \ \ \ \ \ \ \ \ \ \ \ \ \ \ \ \ \ \ \ \ \ \ \ \ \ \ \ \ \ \ \ =\int_{0}^{\l}\phi_{\wt{\l}}^{j_0+1}(T, x, y)p^{\wt{\l}}(T, x, y)\rho^{\wt{\l}}(y)d\wt{\l}. 
\end{align*}
Consequently, 
\[
\left(\sum_{i=0}^{j_0}\left(\begin{matrix}j_0\\ i\end{matrix}\right)(p^{\l})^{(i)}_{\l}(T, x, y)(\rho^{\l})^{(j_0-i)}_{\l}(y)\right)^{(1)}_{\l}=\phi_{\l}^{j_0+1}(T, x, y)p^{{\l}}(T, x, y)\rho^{\l}(y),
\]
which implies that $(p^{\l})^{(j_0+1)}_{\l}(T, x, y)$ exists for every $y$ and satisfies  (\ref{heat-induction-1}). Then we can conclude from this  and the inductive assumption on the continuity of $(p^{\l})^{(i)}_{\l}(T, x, \cdot)$, $i=1, \cdots, j_0$, that $(p^{\l})^{(j_0+1)}_{\l}(T, x, \cdot)$ is  also a continuous function on $\Bbb R_+\times \M$.

Now,  the differentials $(p^{\l})^{(i)}_{\l}(T, x, \cdot)$, $i=1, \cdots, k-2$, exist as continuous functions on $\Bbb R_+\times \M$ and hence their weak derivatives  in $(T, y)$ of any order are well-defined.  Taking the differential of   the heat equations ${\rm L}^{\l}p^{\l}=0$ in $\l$  gives  the following identities in distribution:
\begin{equation*}
{\rm L}^{\l,{\rm w}}(p^{\l})^{(i)}_{\l}(T, x, \cdot)+\sum_{j=1}^{i}\left(\begin{matrix}i\\ j\end{matrix}\right)({\rm L}^{\l})^{(j), {\rm w}}(p^{\l})^{(i-j)}_{\l}(T, x, \cdot)=0, \  \  i=1, \cdots, k-2,
 \end{equation*}
where $({\rm L}^{\l})^{(j), {\rm w}}$ is the weak derivative of the $j$-th differential operator $(L^{\l})^{(j)}_{\l}$.  
We can use Lemma \ref{weak-strong-same} and  Lemma \ref{Friedman-lem} inductively to improve the regularity of $(p^{\l})^{(i)}_{\l}(T, x, \cdot)$.  Shrinking the neighborhood $\mathcal{V}_g$ of $g$ if necessary, we may assume  there is $\iota>0$ such that $p^{\l}(T, x, \cdot)\in C^{k, \iota}(\M)$ for all $\l$.  Since it  is a local problem, for $(T, y)\in \Bbb R_+\times \M$, we can also restrict ourselves to a bounded domain $\mathcal{D}$ containing $(T, y)$. By Lemma \ref{weak-reg-p-1}, there is some domain $\mathcal{D}_1\subset \mathcal{D}$ such that  $p^{\l}(T, x, \cdot), (p^{\l})^{(1)}_{\l}(T, x, \cdot)\in C^{k, \iota}(\mathcal{D}_1)$. Assume  for all $i\leq j_0<k-2$ there are domains $\mathcal{D}_i$ containing $(T, y)$ such that $|(p^{\l})^{(i)}_{\l}(T, x, \cdot)|_{0, 2+\iota}<\infty$ on $\mathcal{D}_i$ and  $(p^{\l})^{(i)}_{\l}(T, x, \cdot)\in C^{k, \iota}(\mathcal{D}_i)$. Then 
\begin{align}
{\rm L}^{\l,{\rm w}}(p^{\l})^{(j_0+1)}_{\l}(T, x, \cdot)= & -\sum_{j=1}^{j_0+1}\left(\begin{matrix}j_0+1\\ j\end{matrix}\right)({\rm L}^{\l})^{(j), {\rm w}}(p^{\l})^{(j_0+1-j)}_{\l}(T, x, \cdot)\notag\\
= & -\sum_{j=1}^{j_0+1}\left(\begin{matrix}j_0+1\\ j\end{matrix}\right)({\rm L}^{\l})^{(j)}_{\l}(p^{\l})^{(j_0+1-j)}_{\l}(T, x, \cdot). \label{induction-p-i-0} 
\end{align}
Shrinking $\mathcal{D}_{j_0}$ to $\mathcal{D}_{j_0+1}$ if necessary,  we can deduce from $|(p^{\l})^{(i)}_{\l}(T, x, \cdot)|_{0, 2+\iota}<\infty$ on $\mathcal{D}_i$ that  $|({\rm L}^{\l})^{(j)}(p^{\l})^{(j_0+1-j)}_{\l}(T, x, \cdot)|_{2, \iota}$ is finite for all $j\leq j_0+1$ on $\mathcal{D}_{j_0+1}$. Since $(p^{\l})^{(j_0+1)}_{\l}(T, x, \cdot)$ is continuous,  Lemma  \ref{weak-strong-same} shows that (\ref{induction-p-i-0}) holds  in the usual sense. Then we can apply Lemma \ref{Friedman-lem-2} to conclude that  $|(p^{\l})^{(j_{0}+1)}_{\l}(T, x, \cdot)|_{0, 2+\iota}<\infty$ on $\mathcal{D}_{j_0+1}$ and apply Lemma \ref{Friedman-lem} to conclude $(p^{\l})^{(j_0+1)}_{\l}(T, x, \cdot)\in C^{k, \iota}(\mathcal{D}_{j_0+1})$. Accordingly, the continuity of $\l\mapsto (p^{\l})^{(1)}_{\l}(T, x, \cdot)$ in $C(\M)$ can be improved to be the continuity in $C^{k, \iota}(\M)$ by using the parabolic differential equation (\ref{induction-p-i-0}), Lemma \ref{Friedman-lem} and Lemma \ref{Friedman-lem-2}.  
\end{proof}

The  $\phi_{\l}^{1}$ satisfying (\ref{p-diff-induction-i}) was identified  in Theorem \ref{regu-p-1st}.  We continue to pick up a candidate  $\phi_{\l}^{2}$  for (\ref{p-diff-induction-i}).  Let $\wt{\phi}_{\l}^{1}(T, x, \cdot)$ be as in (\ref{wt-phi-1}) such that  (\ref{wt-phi-00}) holds.  Then, for any  $f\in C_c^{\infty}(\M)$,
\begin{align}\label{ind-wt-phi-1}
\left(\int_{\M} f(y)p^{\l}(T, x, y)\ d{\rm Vol}^{\l}(y)\right)^{(1)}_{\l}= \overline{\E}\left(f( \lfloor {\rm x}_T \rceil^{\l}(w))\wt{\phi}_{\l}^{1}(T, x, w)\right). 
\end{align}
If we can show  $\l\mapsto \wt{\phi}_{\l}^{1}$ is differentiable,  and both $\wt{\phi}_{\l}^{1}$ and the  differential $(\wt{\phi}_{\l}^{1})^{(1)}_{\l}$  are  $L^q$ integrable for some $q\geq 1$, we are allowed  to differentiate under the expectation sign of the right hand side term of (\ref{ind-wt-phi-1}). 
This will give 
\begin{align}
\notag&\left(\int_{\M} f(y)p^{\l}(T, x, y)\ d{\rm Vol}^{\l}(y)\right)^{(2)}_{\l}\\
\notag& \ \ \  \ \  \ \ \ \ \ \ \ =\int_{\M}f(y)\overline{\E}\left((\wt{\phi}_{\l}^{1})^{(1)}_{\l}(T, x, w)\big| {\lfloor {\rm x}_T \rceil^{\l}(w)}=y\right)p^{\l}(T, x, y)\ d{\rm Vol}^{\l}(y)\\
\label{ex-term-pl2}& \ \ \  \ \  \  \ \ \  \ \ \ \ \ \ \  +\overline{\E}\left(\big\langle \nabla^{\l}_{\lfloor {\rm x}_T \rceil^{\l}(w)} f({\lfloor {\rm x}_T \rceil^{\l}(w)}), \ \wt{\phi}_{\l}^{1}(T, x, w)\cdot D\pi(\lfloor {\rm{u}}_T\rceil^{\l})^{(1)}_{\l}(w) \big\rangle_{\l}\right). 
\end{align}
We can deal with the last  expectation term in (\ref{ex-term-pl2})  as  we did  for   $\phi_{\l}^{1}$ in  Section 5.   Define 
\begin{equation*}
\Phi_{\l}^2(Y, w):=\big\langle Y(\lfloor {\rm x}_T\rceil^{\l}(w)),\   \wt{\phi}_{\l}^{1}(T, x, w)\cdot D\pi(\lfloor {\rm{u}}_T\rceil^{\l})^{(1)}_{\l}(w) \big\rangle_{\l}, 
\end{equation*}
where $Y$ is any  $C^{k}$ bounded vector field  on $\M$,  and consider the linear functional 
\begin{equation*}
\overline{\Phi}_{\l}^2:\ Y\mapsto \overline{\E}\left(\left.\Phi_{\l}^2(Y, w)\right| \lfloor {\rm x}_T \rceil^{\l}(w)=y\right). 
\end{equation*}
If we can show $\overline{\Phi}_{\l}^2$ is such that $\overline{\Phi}_{\l}^2(Y)$ is $C^1$ in $y$ variable, we can conclude that 
\begin{equation}\label{z-T-lam-2}
\overline{\E}\left(\left.\wt{\phi}_{\l}^{1}(T, x, w)\cdot D\pi(\lfloor {\rm{u}}_T\rceil^{\l})^{(1)}_{\l}(w)\right| \lfloor {\rm x}_T \rceil^{\l}(w)=y\right)=:{\rm z}_T^{\l, 2}(y)
\end{equation}
 is a $C^1$ vector field on $\M$ and  satisfies 
\begin{equation*}
\overline{\Phi}_{\l}^2(\nabla f)(y)=\big\langle \nabla_y^{\l}f(y),  {\rm z}_T^{\l, 2}(y)\big\rangle_{\l}. 
\end{equation*}
Using the classical  integration by parts formula, we obtain 
\begin{align*}
&\overline{\E}\left(\big\langle \nabla^{\l}_{\lfloor {\rm x}_T \rceil^{\l}(w)} f({\lfloor {\rm x}_T \rceil^{\l}(w)}), \wt{\phi}_{\l}^{1}(T, x, w)\cdot D\pi(\lfloor {\rm{u}}_T\rceil^{\l})^{(1)}_{\l}(w) \big\rangle_{\l}\right)\\
&=\int_{\M}\big\langle\nabla_y^{\l} f(y),  p^{\l}(T, x, y){\rm z}_T^{\l, 2}(y)\big\rangle_{\l}\ d{\rm Vol}^{\l}(y)\\
&=-\int_{\M} f(y)\left({\rm Div}^{\l}{\rm z}_T^{\l, 2}(y)+\big\langle {\rm z}_T^{\l, 2}(y), \nabla^{\l}\ln p^{\l}(T, x, y) \big\rangle_{\l}\right)p^{\l}(T, x, y)\ d{\rm Vol}^{\l}(y). 
\end{align*}
Therefore, a candidate of $\phi_{\l}^{2}(T, x, \cdot)$  for  (\ref{p-diff-induction-i})  is 
\begin{align}
\notag\phi_{\l}^{2}(T, x, y):=&\  \overline{\E}\left((\wt{\phi}_{\l}^{1})^{(1)}_{\l}(T, x, w)\big| {\lfloor {\rm x}_T \rceil^{\l}(w)}=y\right)\\
&\  -\left({\rm Div}^{\l}{\rm z}_T^{\l, 2}(y)+\big\langle {\rm z}_T^{\l, 2}(y), \nabla^{\l}\ln p^{\l}(T, x, y) \big\rangle_{\l}\right). \label{phi-2-1-rel}
\end{align}
Once we show   $\phi_{\l}^{2}(T, x, y)$ fulfills the continuity requirement of  Lemma \ref{weak-reg-p-geq1}, we can conclude the second order differentiability of $\l\mapsto p^{\l}(T, x, )$ in $C^{k, \iota}$ for some $\iota>0$.   It follows that 
\begin{align*}
(\ln p^{\l})^{(2)}_{\l}(T, x, y)=\phi_{\l}^{2}(T, x, y)-(\phi_{\l}^1)^2(T, x, y)-(\ln \rho^{\l})^{(2)}_{\l}(y). 
\end{align*}
Note that  the  gradients estimations of $\phi_{\l}^1$ were already  handled in Theorem \ref{regu-p-1st}. Hence the gradients estimations of $(\ln p^{\l})^{(2)}_{\l}$ can be reduced to that of $\phi_{\l}^{2}$, which can be analyzed  following the proof of Theorem \ref{regu-p-1st}, if we can find some controllable $\wt{\phi}_{\l}^{2}(T, x, \cdot)$ such that 
 \[\phi_{\l}^{2}(T, x, y)=\ov{\E}\left(\left.\wt{\phi}_{\l}^{2}(T, x, w)\right|\lfloor {\rm x}_T \rceil^{\l}(w)=y \right). \]

We will follow this  line  of discussion to find all the candidates $\phi_{\l}^{i}$ for (\ref{p-diff-induction-i}).   Put $\wt{\phi}^0_{\l}(T, x, w)\equiv 1$ and let $\wt{\phi}^1_{\l}(T, x, w)$ be as in  (\ref{wt-phi-1}). For $i,$ $2\leq i\leq k-2$,  define 
\begin{align}\label{wt-phi-induction-i}
{\wt{\phi}}_{\l}^{i}(T, x, w)\ := \ &\big({\wt{\phi}}_{\l}^{i-1}\!(T, x, w)\big)^{(1)}_{\l}\!-\!\big\langle \nabla^{\l}_{T, \mathtt{s}}\wt{\phi}_{\l}^{i-1}\!(T, x, w), D\pi(\lfloor {\rm{u}}_T\rceil^{\l})^{(1)}_{\l}(w)\big\rangle\\
&+{\wt{\phi}}_{\l}^{i-1}(T, x, w){\wt{\phi}}_{\l}^{1}(T, x, w),\notag
\end{align}
where the `path-wise gradient' $\nabla^{\l}_{T, \mathtt{s}}\wt{\phi}_{\l}^{i-1}\!(T, x, w)$ will be specified later. We will show each \begin{align}\label{-phi-induction-i}
\phi_{\l}^{i}(T, x, y):=\ov{\E}\left(\left.\wt{\phi}_{\l}^{i}(T, x, w)\right|\lfloor {\rm x}_T \rceil^{\l}(w)=y \right)
\end{align}
fulfills all the requirements in  Lemma \ref{weak-reg-p-geq1}. The stochastic  expression (\ref{-phi-induction-i}) will be used for two purposes:  one is for the gradient estimations  of $\phi_{\l}^{i}(T, x, \cdot)$ and $(\ln p^{\l})^{(i)}_{\l}(T, x, \cdot)$; the other is  for obtaining  $\phi_{\l}^{i+1}(T, x, y)$  as we exposed above for $\phi_{\l}^{2}(T, x, y)$ (see (\ref{phi-2-1-rel})).

Let us highlight  the necessary steps to undergo an inductive argument for Theorem \ref{diff-HK-estimations-gen}.  Assume for all $i<j\leq k-2$, the 
$\phi_{\l}^{i}$ defined in (\ref{-phi-induction-i}) are such that Lemma \ref{weak-reg-p-geq1} holds true,  $(p^{\l})^{(i)}_{\l}(T, x, \cdot)\in C^{k, \iota}(\M)$, is continuous in $\l$  and (\ref{esti-p-lam-der-k})  holds for $(\ln p^{\l})^{(i)}_{\l}(T, x, \cdot)$. We also assume the following coarse pointwise estimation holds true for all $i<j$. 
\begin{itemize}
\item[{\bf 0)}] For all $l$, $0\leq l\leq k-2-i$, there is $\underline{c}_{\l, (l, i)}$ depending  on $\|g^{\l}\|_{C^{l+i+2}}$, $\|\XX^{\l}\|_{C^{l+i+1}}$  and $c^{\l, (l, i)}$ depending  on $(l, i)$, $m, q$, $T, T_0$ and  $\|g^{\l}\|_{C^3}$ such that 
\begin{align}
\ \ \ \ \ \ \  &\left|\nabla^{(l)}(\ln p^{\l})^{(i)}_{\l}(T, x, y)\right|, \left|\nabla^{(l)}\phi_{\l}^{i}(T, x, y)\right|\notag\\
&\ \ \ \ \ \leq (p^{\l}(T, x, y))^{-i}\underline{c}_{\l, (l, i)}\left(\big(\frac{1}{T}d_{\wt{g}^{\l}}(x, y)+\frac{1}{\sqrt{T}}\big)^{l+i}+1\right)\cdot  e^{c^{\l, (l, i)}(1+d_{\wt{g}^{\l}}(x, y))}.\label{esti-p-lam-der-2}
\end{align}
\end{itemize}

For the existence of  $(p^{\l})^{(j)}_{\l}(T, x, \cdot)$, the very first step is find  some measurable candidate satisfying (\ref{p-diff-induction-i}), which can be done once we show the following. 
\begin{itemize}
\item[{\bf i)}] The function $\wt{\phi}_{\l}^{j-1}(T, x, w)$ is differentiable in $\l$ for almost all $w\in \Theta_+$ and both $\wt{\phi}_{\l}^{j-1}(T, x, w)$  and $\big({\wt{\phi}}_{\l}^{j-1}\big)^{(1)}_{\l}(T, x, w)$ are  $L^q$ integrable in $w$ for all $q\geq 1$.
\item[{\bf ii)}] For any $C^{k}$ bounded  vector field $Y$ on $\M$, let 
\begin{equation*}
\Phi_{\l}^j(Y, w):=\big\langle Y(\lfloor {\rm x}_T\rceil^{\l}(w)),  \wt{\phi}_{\l}^{j-1}(T, x, w) D\pi(\lfloor {\rm{u}}_T\rceil^{\l})^{(1)}_{\l}(w) \big\rangle_{\l}. 
\end{equation*}
Then the linear  functional \begin{equation*}
\overline{\Phi}_{\l}^j:\ Y\mapsto \overline{\E}\left(\left.\Phi_{\l}^j(Y, w)\right| \lfloor {\rm x}_T \rceil^{\l}(w)=y\right)
\end{equation*}
is bounded with $\overline{\Phi}_{\l}^j(Y)(y)$ varying $C^1$ in  the $y$-coordinate.
\end{itemize}
\begin{cla}\label{cla-1-11-ov-phi}Assume {\bf i), ii)} are true. Then (\ref{p-diff-induction-i}) hold with some $\ov{\phi}_{\l}(T, x, \cdot)$ for $i=j$. 
\end{cla}
\begin{proof}
By the inductive assumption, \begin{align}\label{ch6-DUES}
\left(\int_{\M} f(y)p^{\l}(T, x, y)\ d{\rm Vol}^{\l}(y)\right)^{(j-1)}_{\l}\!=\ \ov{\E}\left (f(\lfloor {\rm x}_{T}\rceil^{\l})\wt{\phi}_{\l}^{j-1}(T, x, w)\right). 
\end{align}
If  {\bf i)} is true, we can differentiate under the expectation sign of (\ref{ch6-DUES}). This gives 
\begin{align*}
\left(\int_{\M} f(y)p^{\l}(T, x, y)\ d{\rm Vol}^{\l}(y)\right)^{(j)}_{\l}=\; \overline{\E}\left(f( \lfloor {\rm x}_T \rceil^{\l}(w))(\wt{\phi}_{\l}^{j-1})^{(1)}_{\l}(T, x, w)\right)+\overline{\E}\left(\Phi_{\l}^j(\nabla f, w)\right).
\end{align*}
The property {\bf ii)}  implies 
\begin{equation}\label{z-T-lam-j}
{\rm z}_T^{\l, j}(y) :=\; \overline{\E}\left(\left.\wt{\phi}_{\l}^{j-1}(T, x, w)\cdot D\pi(\lfloor {\rm{u}}_T\rceil^{\l})^{(1)}_{\l}(w)\right| \lfloor {\rm x}_T \rceil^{\l}(w)=y\right)
\end{equation}
is a $C^1$ vector field on $\M$ such that 
\[
\overline{\Phi}_{\l}^j(Y)(y)=\big\langle Y(y), {\rm z}_T^{\l, j}(y)\big\rangle_{\l}. 
\]
In particular, we have
\begin{align*}
\overline{\E}\left(\Phi_{\l}^j(\nabla f, w)\right)
&=\int_{\M}\big\langle\nabla_y^{\l} f(y),  p^{\l}(T, x, y){\rm z}_T^{\l, j}(y)\big\rangle_{\l}\ d{\rm Vol}^{\l}(y)\\
&=\!-\int_{\M} f(y)\!\left({\rm Div}^{\l}{\rm z}_T^{\l, j}(y)\!+\!\big\langle {\rm z}_T^{\l, j}(y), \nabla^{\l}\!\ln p^{\l}(T, x, y)\big\rangle_{\l}\right)\! p^{\l}(T, x, y)\ d{\rm Vol}^{\l}(y). 
\end{align*}
This means a measurable candidate $\phi^{j}_{\l}$ for  (\ref{p-diff-induction-i}) at $i=j$ is 
\begin{align}
\notag\ov{\phi}_{\l}(T, x, y):=&\; \overline{\E}\left(\left. (\wt{\phi}_{\l}^{j-1})^{(1)}_{\l}(T, x, w)\right| {\lfloor {\rm x}_T \rceil^{\l}(w)}=y\right)\\
\label{ov-phi}&-\left({\rm Div}^{\l}{\rm z}_T^{\l, j}(y)+\big\langle {\rm z}_T^{\l, j}(y), \nabla^{\l}\ln p^{\l}(T, x, y) \big\rangle_{\l}\right).
\end{align}
\end{proof}

We will show  $\ov{\phi}_{\l}(T, x, \cdot)$ given in Claim \ref{cla-1-11-ov-phi}  coincides with $\phi_{\l}^{j}(T, x, \cdot)$ defined by (\ref{-phi-induction-i}).  For any smooth bounded vector field ${\rm V}$ on $\M$, let  $\{F^s\}_{s\in \Bbb R}$ be the flow it generates.  As we did for $\overline{\Phi}_{\l}^1$ in Section 5 (see Lemma \ref{C-1-fun-phi-Y}), we will  prove the following  in verifying {\bf ii)}.  
\begin{itemize}
\item[{\bf iii)}] For any $y\in \M$,   $s\mapsto \overline{\Phi}_{\l}^j(Y)(F^sy)$  is differentiable at $s=0$ and the differential $(\overline{\Phi}_{\l}^j(Y)(F^sy))'_0$
varies continuously in $y$.  Moreover, \begin{align*}
\ \ \ \ \ \ \ \ \ \ &(\overline{\Phi}_{\l}^j(Y)(F^sy))'_0\\ 
\ \ \ \ \ \ \ \ &\  =\ \overline{\Phi}_{\l}^j(\nabla_{\rm V}^{\l} Y)(y)+\overline{\E}\left(\big\langle Y(\lfloor {\rm x}_{T} \rceil^{\l}(w)), \nabla^{\l}_{T, {\rm V}, \mathtt{s}}\big(\wt{\phi}_{\l}^{j-1}(T, x, w) D\pi(\lfloor {\rm u}_T\rceil^{\l})^{(1)}_{\l}(w)\big)\big\rangle_{\l} \right.\\
\ \ \ \ \ \ \ \ & \ \ \ \ \left.\left.+\ \big\langle Y(\lfloor {\rm x}_{T} \rceil^{\l}(w)), \big(\wt{\phi}_{\l}^{j-1}(T, x, w) D\pi(\lfloor {\rm u}_T\rceil^{\l})^{(1)}_{\l}(w)\big)\big\rangle_{\l}\oa{\mathcal{E}}_{T, {\rm V}, {\mathtt s}}(w) \right|  \lfloor {\rm x}_T \rceil^{\l}(w)=y\right),  
\end{align*}
where the path-wise differential $\nabla^{\l}_{T, {\rm V}, \mathtt{s}}\big(\wt{\phi}_{\l}^{i-1}\!(T, x, w)D\pi(\lfloor {\rm u}_T\rceil^{\l})^{(1)}_{\l}(w)\big)$ will be clarified later and it satisfies 
\begin{align*}
\ \ \ \  \ \ \ \ \ \nabla^{\l}_{T, {\rm V}, \mathtt{s}}\big(\wt{\phi}_{\l}^{i-1}(T, x, w)D\pi(\lfloor {\rm u}_T\rceil^{\l})^{(1)}_{\l}(w)\big)=&\ \big(\nabla^{\l}_{T, {\rm V}, \mathtt{s}}\wt{\phi}_{\l}^{i-1}(T, x, w)\big)\cdot D\pi(\lfloor {\rm u}_T\rceil^{\l})^{(1)}_{\l}(w)\\
&+\wt{\phi}_{\l}^{i-1} (T, x, w)\cdot \big(\nabla^{\l}_{T, {\rm V}, \mathtt{s}}D\pi(\lfloor {\rm u}_T\rceil^{\l})^{(1)}_{\l}(w)\big). 
\end{align*}
\end{itemize}

\begin{cla}\label{Cla-i-ii-} Assume {\bf i)-iii)} are true. Then $\ov{\phi}_{\l}(T, x, \cdot)=\phi_{\l}^j(T, x, \cdot)$. 
\end{cla}
\begin{proof}By {\bf ii)}, both $\overline{\Phi}_{\l}^j(Y)(y)$ and ${\rm z}_T^{\l, j}(y)$ vary  $C^1$ in $y$. Hence 
\[
(\overline{\Phi}_{\l}^j(Y)(F^sy))'_0=\nabla^{\l}_{{\rm V}}(\overline{\Phi}_{\l}^j(Y))(y)=\big\langle \nabla^{\l}_{{\rm V}}Y(y),  {\rm z}_T^{\l, j}(y)\big\rangle_{\l}+ \big\langle Y(y),  \nabla^{\l}_{{\rm V}}{\rm z}_T^{\l, j}(y)\big\rangle_{\l}. 
\] 
Comparing this with the expression of $\big(\overline{\Phi}_{\l}^j(Y)(F^sy)\big)'_0$ in {\bf iii)}  gives  \begin{align}
&\nabla^{\l}_{{\rm V}(y)}{\rm z}_T^{\l, j}(y)=\ov{\E}\left(\nabla^{\l}_{T, {\rm V}, \mathtt{s}}\big(\wt{\phi}_{\l}^{j-1}(T, x, w)\cdot D\pi(\lfloor {\rm u}_T\rceil^{\l})^{(1)}_{\l}(w)\big)\right.\notag\\
&\left.\left.\ \ \ \ \ \ \ \ \ \ \ \ \ \ \ \ \ \ \ \  \ \ \ +\big(\wt{\phi}_{\l}^{j-1}(T, x, w)\cdot D\pi(\lfloor {\rm u}_T\rceil^{\l})^{(1)}_{\l}(w)\big)\oa{\mathcal{E}}_{T, {\rm V}, {\mathtt s}}(w)\right|  \lfloor {\rm x}_T \rceil^{\l}({w})=y\right).\label{DIV-Z-j}
\end{align}
Following the  argument in the proof of Lemma  \ref{reg-Z-lam-1} and then using (\ref{wt-phi-1}), we obtain
\begin{align*}
&{\rm Div}^{\l}{\rm z}_T^{\l, j}(y)\\
&=\ov{\E}\left({\rm tr}\big({\rm V}\mapsto \nabla^{\l}_{T, {\rm V}, \mathtt{s}}\big(\wt{\phi}_{\l}^{j-1}(T, x, w)\cdot D\pi(\lfloor {\rm u}_T\rceil^{\l})^{(1)}_{\l}(w)\big)\big)\right.\\
&\ \ \ \ \ \ \ \ +\wt{\phi}_{\l}^{j-1}(T, x, w)\cdot\big\langle D\pi(\lfloor {\rm u}_T\rceil^{\l})^{(1)}_{\l}(w), \frac{1}{2}\!\int_{0}^{T}\!\mathtt{s}(T\!-\!\tau) \lfloor{\rm u}_{T}\rceil^{\l}(\lfloor{\rm u}_\tau\rceil^{\l})^{-1}{\rm Ric}_{\lfloor{\rm u}_{\tau}\rceil^{\l}}^{-1}d\oa{B}_{\tau}\big\rangle_{\l}
\\
&\ \ \ \ \ \ \ \ +\left.\left. \wt{\phi}_{\l}^{j-1}(T, x, w)\cdot\big\langle D\pi(\lfloor {\rm u}_T\rceil^{\l})^{(1)}_{\l}(w),  -\frac{1}{2}\lfloor{\rm u}_T\rceil^{\l}\!\int_{0}^{T}\!\mathtt{s}'(T\!-\!\tau)d\oa{B}_{\tau}\big\rangle_{\l}\ \ \right|  \lfloor {\rm x}_T \rceil^{\l}({w})=y\right)\\
&=\ov{\E}\left(\left.\big\langle \nabla^{\l}_{T, \mathtt{s}}\wt{\phi}_{\l}^{j-1}(T, x, w), D\pi(\lfloor {\rm{u}}_T\rceil^{\l})^{(1)}_{\l}(w)\big\rangle \right| \lfloor {\rm x}_T \rceil^{\l}(w)=y\right)\\
&-\ov{\E}\left(\left.{\wt{\phi}}_{\l}^{j-1}\!(T, x, w)\big({\wt{\phi}}_{\l}^{1}(T, x, w)\!+\!\big\langle D\pi(\lfloor {\rm u}_T\rceil^{\l})^{(1)}_{\l}, \nabla^{\l}\ln p^{\l}(T, x, \lfloor {\rm x}_T \rceil^{\l})\big\rangle_{\l}\big)\right|  \lfloor {\rm x}_T \rceil^{\l}(w)=y\right).\end{align*}
Note that 
\begin{align*}
&\langle {\rm z}_T^{\l, j}(y), \nabla^{\l}\ln p^{\l}(T, x, y) \rangle_{\l}\\
&\ \ \ \ \ \ \ =\ov{\E}\left(\left.{\wt{\phi}}_{\l}^{j-1}(T, x, w)\big\langle D\pi(\lfloor {\rm u}_T\rceil^{\l})^{(1)}_{\l}(w), \nabla^{\l}\ln p^{\l}(T, x, \lfloor {\rm x}_T \rceil^{\l}(w))\big\rangle_{\l}\right| \lfloor {\rm x}_T \rceil^{\l}(w)=y\right).
\end{align*}
Reporting these two expressions  in (\ref{ov-phi}), we obtain  
\begin{align*}
\ov{\phi}_{\l}(T, x, y)=\ &\ov{\E}\left(\big({\wt{\phi}}_{\l}^{j-1}\!(T, x, w)\big)^{(1)}_{\l}\!-\!\big\langle \nabla^{\l}_{T, \mathtt{s}}\wt{\phi}_{\l}^{j-1}\!(T, x, w), D\pi(\lfloor {\rm{u}}_T\rceil^{\l})^{(1)}_{\l}(w)\big\rangle\right.\\
&\ \ \ \  \ \ \ \ \ \ \ \ \  \ \ \ \ \ \ \ \ \ \ \ \ \ \left.\left.+{\wt{\phi}}_{\l}^{j-1}(T, x, w){\wt{\phi}}_{\l}^{1}(T, x, w)\right| \lfloor {\rm x}_T \rceil^{\l}(w)=y\right)\\
=\ &\phi_{\l}^j(T, x, y).
\end{align*}
\end{proof}

To study the  continuity of $\phi^j_{\l}(T, x, y)$ in $(T, y)$,  we first show the following. 
\begin{itemize}
\item[{\bf iv)}]For all $x\in \M$, $T\in \Bbb R_+$, 
\[
y\mapsto \overline{\E}\left(\left. (\wt{\phi}_{\l}^{j-1})^{(1)}_{\l}(T, x, w)\right| {\lfloor {\rm x}_T \rceil^{\l}(w)}=y\right)
\]
is continuous, locally  uniformly  in $T$ and $\l$. Moreover, there exist  $\underline{c}_{\l, j}$ (depending on $\|g^{\l}\|_{C^{j+2}}, \|\XX^{\l}\|_{C^{j+1}}$) and $\wt{c}^{\l, j}$  (depending on $j$,  $T, T_0$ and  $\|g^{\l}\|_{C^3}$)  such that
\begin{align*}
\ \ \ &\left|\overline{\E}_{\ov{\P}^{\l}_{x, y, T}}\left((\wt{\phi}_{\l}^{j-1})^{(1)}_{\l}(T, x, w)\right)\right|\\
&\ \ \ \ \ \leq (p^{\l}(T, x, y))^{-j}\underline{c}_{\l, j}\left(\big(\frac{1}{T}d_{\wt{g}^{\l}}(x, y)+\frac{1}{\sqrt{T}}\big)^j+1\right)\cdot  e^{\wt{c}^{\l, j}(1+d_{\wt{g}^{\l}}(x, y))}.
\end{align*}
\item[{\bf v)}] For all  $x\in \M$, $T\in \Bbb R_+$, the mappings   $y\mapsto {\rm z}_T^{\l, j}(y), y\mapsto {\rm Div}^{\l}{\rm z}_T^{\l, j}(y)$ are continuous, locally  uniformly in $T$ and $\l$.  Moreover,  there are $\underline{c}'_{\l, j}$ (depending on $\|g^{\l}\|_{C^{j+2}}, \|\XX^{\l}\|_{C^{j+1}}$) and $\ov{c}^{\l, j}$ depending on $j$, $T, T_0$ and  $\|g^{\l}\|_{C^3}$  such that   \begin{align*}
\ \ \ \ \ \ \ \ \ &\big|{\rm z}_T^{\l, j}(y)\big|, \ \big|{\rm Div}^{\l}{\rm z}_T^{\l, j}(y)\big|\\ 
\ \ \ \ \ \ \ \ \ &\ \ \ \ \ \ \ \leq (p(T, x, y))^{-j}\underline{c}'_{\l, j} \left(\big(\frac{1}{T}d_{\wt{g}^{\l}}(x, y)+\frac{1}{\sqrt{T}}\big)^j+1\right)\cdot  e^{\ov{c}^{\l, j}(1+d_{\wt{g}^{\l}}(x, y))}.
 \end{align*}
 \item[{\bf vi)}] There is $\underline{c}_{\l, (0, j)}(q)$ depending  on $j$, $m, q$, $T, T_0$, $\|g^{\l}\|_{C^{j+2}}$ and  $\|\XX^{\l}\|_{C^{j+1}}$ such that  \begin{align}\label{norm-phi-j-lam-0}
\left\|\phi_{\l}^j(T, x, \cdot)\right\|_{L^q}\leq \underline{c}_{\l, (0, j)}(q), \ \forall q\geq 1. \end{align}
\end{itemize} 

\begin{cla}\label{cont-phi-j-T-y} Assume {\bf i)-vi)}. For every  $x\in \M$, $(T, y)\mapsto \phi^j_{\l}(T, x, y)$ is continuous, locally uniformly in $\l$. \end{cla}
\begin{proof}To conclude the continuity of $\phi^j_{\l}(T, x, y)$ in $(T, y)$, we verify the following.  \\
\indent 1) For all $x\in \M$, $T\in \Bbb R_+$, $y\mapsto \phi^j_{\l}(T, x, y)$ is continuous, locally  uniformly  in $T$ and $\l$.\\
\indent 2) For each $x, y$ fixed,   $T\mapsto \phi^j_{\l}(T, x, y)$ is continuous, locally uniformly in $\l$.

By the inductive assumption,  all $(p^{\l})^{(i)}(T, x, y)$, $i<j$, exist, are continuous in $(T, y)$  and satisfy the bound estimation (\ref{esti-p-lam-der-k}).  By  {\bf i)-iii), }
\begin{align}
\phi^j_{\l}(T, x, y)=&\; \overline{\E}\left(\left.(\wt{\phi}_{\l}^{j-1})^{(1)}_{\l}(T, x, w)\right| {\lfloor {\rm x}_T \rceil^{\l}(w)}=y\right)\notag\\
\label{phi-i-iden}&-\left({\rm Div}^{\l}{\rm z}_T^{\l, j}(y)+\big\langle {\rm z}_T^{\l, j}(y), \nabla^{\l}\ln p^{\l}(T, x, y) \big\rangle_{\l}\right)
\end{align}
satisfies (\ref{p-diff-induction-i}). By {\bf iv)} and  {\bf v)}, for each $x\in \M$,  we have that the mapping $y\mapsto \phi_{\l}^j(T, x, y)$ is continuous, locally  uniformly   in $T$. 

For 2), we follow  the proof of Theorem \ref{regu-p-1st} for the $k=3$ case.  Simply denote by $({\rm x}^{\l}, {\rm u}^{\l})$ the stochastic pair which defines the Brownian motion on $(\M, \wt{g}^{\l})$ starting from $x$. Then, for any  $f\in C^{\infty}_c(\M)$ with support contained in a small neighborhood of $y$ and  $T'>T$, 
\[
\E \big(f({\rm x}^{\l}_{T'})\big)-\E \big(f({\rm x}^{\l}_T)\big)=\int_{T}^{T'}\E\big(\Delta^{\l} f({\rm x}^{\l}_t)\big)\ dt. 
\]
Take the $j$-th differential in $\l$ of both sides and use (\ref{p-diff-induction-i}). We obtain 
\begin{align*}
&\int_{\M} f(z)\left(\phi^{j}_{\l}(T', x, z)-\phi^{j}_{\l}(T, x, z)\right)p^{\l}(T', x, z)\ d{\rm Vol}^{\l}(z)\\
& \ +\int_{\M} f(z)\phi^{j}_{\l}(T, x, z)\left(p^{\l}(T', x, z)-p^{\l}(T, x, z)\right)\ d{\rm Vol}^{\l}(z)\\
&=\int_{T}^{T'} \int_{\M}\sum_{i=0}^{j}\left(\begin{matrix}j\\ i\end{matrix}\right)(\Delta^{\l}f(z))^{(i)}_{\l}\left(p^{\l}(t, x, z)\rho^{\l}(z)\right)^{(j-i)}_{\l}\ d{\rm Vol}^{0}(z)\ dt. 
\end{align*}
Using (\ref{norm-phi-j-lam-0}),  we deduce  that 
\[
\lim_{T'\to T}\int_{\M} f(z)\left(\phi^{j}_{\l}(T', x, z)-\phi^{j}_{\l}(T, x, z)\right)p^{\l}(T', x, z)\ d{\rm Vol}^{\l}(z)=0. 
\]
Since  $\phi^{j}_{\l}(T, x, z)$ is continuous in $z$ and $\l$, locally  uniformly  in $T$,   and $f$ is arbitrary, we must have
 $\lim_{T'\to T}\phi^{j}_{\l}(T', x, y)=\phi^{j}_{\l}(T, x, y)$, locally uniformly in $\l$. This shows 2) and finishes the proof of Claim \ref{cont-phi-j-T-y}. 
\end{proof}

\begin{cla}
Assume {\bf i)-vi)}. Then for any  $x\in \M$, $T\in \Bbb R_+$,   $\l\mapsto p^{\l}(T, x, \cdot)$ is $C^j$ in $C^{k, \iota}(\M)$ for some $\iota>0$.  The differential $(p^{\l})^{(j)}_{\l}(T, x, y)$ satisfies the equation
\begin{equation}\label{p-j-phi-j}
(p^{\l})^{(j)}_{\l}(T, x, y)=\phi_{\l}^{j}(T, x, y)p^{{\l}}(T, x, y)-(\rho^{\l}(y))^{-1}\sum_{i=0}^{j-1}\left(\begin{matrix}j\\ i\end{matrix}\right)(p^{\l})^{(i)}_{\l}(T, x, y)(\rho^{\l})^{(j-i)}_{\l}(y). 
\end{equation} 
Consequently,  $\phi^{j}_{\l}(T, x, \cdot)\in C^{k, \iota}(\M)$ as well.
\end{cla}
\begin{proof} The function $\phi_{\l}^{j}(T, x, y)$ is continuous in $y$, uniformly in $\l$ by using {\bf iv), v)} and (\ref{phi-i-iden}). So, it is continuous  in $\l$ if for any $f\in C_c^{\infty}(\M)$, we have the continuity of 
\begin{equation*}
\l\mapsto \left(\int_{\M} f(y)p^{\l}(T, x, y)\ d{\rm Vol}^{\l}(y)\right)_{\l}^{(j)}=: A_j(\l, T, x).\end{equation*}
Note that 
\[
A_1(\l, T, x)=\E\left(\left\langle \nabla^{\l}_{\lfloor {\rm x}_T \rceil^{\l}(w)} (f\circ \pi)(\lfloor {\rm{u}}_T\rceil^{\l}(w)), (\lfloor {\rm{u}}_T\rceil^{\l})^{(1)}_{\l}(w) \right\rangle_{\l}\right). 
\]
Differentiating $A_1(\l, T, x)$ in $\l$ for $j-1$ times, we get a similar expression  $A_j(\l, T, x)$ of  a  combination of inner products  involving  $\{\nabla^{(i)}f\}_{i\leq j}$, $\{(\lfloor {\rm x}_T \rceil^{\l})^{(i)}\}_{i\leq j}$ and $\{(\lfloor {\rm{u}}_T\rceil^{\l})^{(i)}_{\l}\}_{i\leq j}$. Following Proposition \ref{est-norm-u-t-j} i), we can derive the $L^q$ $(q\geq 1)$ convergence of $(\lfloor {\rm x}_T \rceil^{\l})^{(i)}$ and $(\lfloor {\rm{u}}_T\rceil^{\l})^{(i)}_{\l}$ in $\l$. As a consequence, we obtain the continuity of $\l\mapsto  A_j(\l, T, x)$. 

Now, by {\bf vi)}, the continuity of $\phi_{\l}^{j}(T, x, y)$ in $\l$ and $(T, y)$ and the induction assumption, we can apply Lemma \ref{weak-reg-p-geq1} to conclude  that  $\l\mapsto p^{\l}(T, x, \cdot)$ is $C^j$ in $C^{k, \iota}(\M)$ for some $\iota>0$. The equation (\ref{p-j-phi-j}) holds by comparison and hence $\phi^{j}_{\l}(T, x, \cdot)\in C^{k, \iota}(\M)$. 
\end{proof}

With {\bf i)-vi)},  the gradients $\{\nabla^{(l)}(\ln p^{\l})^{(j)}_{\l}(T, x, y)\}_{1\leq l\leq k-2-j}$  are well-defined.  To conclude  Theorem \ref{diff-HK-estimations-gen} ii)  by induction, it remains to show (\ref{esti-p-lam-der-k}) for $i=j$. With the identity (\ref{p-j-phi-j}) and {\bf vi)}, it remains to show the following. 
\begin{itemize}
\item[{\bf vii)}] For all $l$, $1\leq l\leq k-2-j$, $q\geq 1$,  there is $\underline{c}_{\l, (l, j)}(q)$ which depends on $(l, j)$, $m, q$, $T, T_0$, $\|g^{\l}\|_{C^{l+j+2}}$ and  $\|\XX^{\l}\|_{C^{l+j+1}}$ such that  \begin{align}\label{norm-phi-j-lam}
\left\|\nabla^{(l)}\phi_{\l}^j(T, x, \cdot)\right\|_{L^q}\leq \underline{c}_{\l, (l, j)}(q). \end{align}
\end{itemize} 
For (\ref{norm-phi-j-lam}), we will  use (\ref{-phi-induction-i}) to formulate  $\nabla^{(l)}\nabla_{W_1, W_2, \cdots, W_j}\phi_{\l}^1$ (for any smooth bounded vector fields $W_1, W_2, \cdots, W_j$) as some conditional expectation and  use it for  evaluations as in the proof of Theorem \ref{regu-p-1st}. For this, we need the bounds control on $\phi_{\l}^j(T, x, y)$ from {\bf iv), v)}. Note that  in showing {\bf iv), v)},  we need a bound control of $\{|\nabla^{(l)}(\ln p^{\l})^{(i)}_{\l}(T, x, \cdot)|\}_{1\leq l\leq j-i}$. So, to continue the inductive argument,  we also  need to verify {\bf 0)} at $i=j$, which can be obtained in showing {\bf vi)} and {\bf vii)}. 

Theorem \ref{diff-HK-estimations-gen} iii) will follow from ii). Indeed, for $i=1$,  (\ref{p-lam-i-p-equ}) is true by Theorem \ref{regu-p-1st} for the $k=3$ case.  
For $i\geq 2$, by (\ref{p-j-phi-j}), \begin{align*}
\frac{(p^{\l})^{(i)}(T, x, y)}{p^{\l}(T, x, y)}=\phi_{\l}^{i}(T, x, y)-(\rho^{\l}(y))^{-1}\sum_{j=0}^{i-1}\left(\begin{matrix}i\\ j\end{matrix}\right)\frac{(p^{\l})^{(j)}(T, x, y)}{p^{\l}(T, x, y)}(\rho^{\l})^{(i-j)}(y). 
\end{align*}
So an inductive argument using (\ref{norm-phi-j-lam})  and (\ref{esti-p-lam-der-k}) will  conclude (\ref{p-lam-i-p-equ}) for all $i\leq k-2$.

Finally,  consider  Theorem \ref{diff-HK-estimations-gen} iv). By symmetry, the mapping $x \mapsto (p^{\l})^{(i)}(T, x, y)$ is continuous for all $T,y$, locally uniformly in $y$.  We conclude using (\ref{uniform-f-f-n}) and (\ref{p-lam-i-p-equ})  as in the proof of Theorem  \ref{regu-p-1st} iii).

In summary, to carry out the above inductive argument for Theorem \ref{diff-HK-estimations-gen}, all we need to do is to verify the  properties {\bf i)-vii)}  at each step.   We first consider {\bf i)}, followed by {\bf iv)} and then check {\bf ii), iii), v), vi)} and  {\bf vii)}. 
 The ideas to show these properties at each step are similar. So we only check them  for the $j=2$ case in details and indicate the necessary modifications to make them work  for the general case.

\subsection{Proofs of the properties concerning $\phi^j_{\l}$}\label{sec6.2} Let  $\lambda\in
(-1, 1)\mapsto g^{\lambda}\in \mathcal{M}^k(M)$ be a $C^k$ curve $(k\geq 4)$.   Assume  all the properties  {\bf i)}-{\bf vii)} in Section \ref{diff-HK-estimations-gen} hold true  for $\wt{\phi}_{\l}^{i}, \phi_{\l}^{i}$  and $(p^{\l})^{(i)}_{\l}$,  $i< j\leq k-2$. We continue to verify the conclusions for $\wt{\phi}_{\l}^{j}, \phi_{\l}^{j}$ and $(p^{\l})^{(j)}_{\l}$.

\begin{proof}[Proof of properties {\bf i)} and {\bf iv)} in Section \ref{skectch-6.1}] We first show {\bf i)} and the estimation in {\bf iv)}.  We begin with the case $j=2$.  For {\bf i)}, it suffices to consider the differentiability and $L^q$ integrability of each term  in (\ref{wt-phi-1}). We add an upper-script $\l$ to  $({\rm{I}}), ({\rm{II}}), ({\rm{III}})$ and $({\rm{IV}})$ in   (\ref{wt-phi-1}) to indicate their dependence on $\l$. 

 For $({\rm{IV}})^{\l}$,  it is differentiable in $\l$  by Lemma \ref{u-lambda-differential-1} and Theorem \ref{regu-p-1st} for the $k=3$ case. Denote the differential by  $\big(({\rm{IV}})^{\l}\big)^{(1)}_{\l}$. Then 
\begin{align*}
\big(({\rm{IV}})^{\l}\big)^{(1)}_{\l}=&\  -\big\langle D\pi(\lfloor {\rm u}_T\rceil^{\l})^{(2)}_{\l}, \nabla^{\l}\ln p^{\l}(T, x, \lfloor {\rm x}_T \rceil^{\l})\big\rangle_{\l}\\
&\ -\big\langle D\pi(\lfloor {\rm u}_T\rceil^{\l})^{(1)}, \nabla^{\l}(\ln p^{\l})^{(1)}_{\l}(T, x, \lfloor {\rm x}_T \rceil^{\l})\big\rangle_{\l} \\
&\ -\big\langle D\pi(\lfloor {\rm u}_T\rceil^{\l})^{(1)}, \nabla^{\l}_{(\lfloor {\rm x}_T \rceil^{\l})^{(1)}_{\l}}\nabla^{\l}\ln p^{\l}(T, x, \lfloor {\rm x}_T \rceil^{\l})\big\rangle_{\l} \\
=:&\ ({\rm{IV}})^{\l}_1+({\rm{IV}})^{\l}_2+({\rm{IV}})^{\l}_3.
\end{align*}
 By an abuse of notation, we  use $\underline{c}(q)$ to denote a constant depending on $T, T_0, m, q, \|g^{\l}\|_{C^4}$ and $\|\XX\|_{C^3}$, which  may vary from line to line. 
Using i) of Proposition \ref{est-norm-u-t-j} and Lemma \ref{Hs2-ST}, we obtain $\underline{c}(q)$  such that 
\begin{align*}
\big(\overline{\E}\big| ({\rm{IV}})^{\l}_1\big|^q\big)^2\leq \ov{\E}\big\|(\lfloor {\rm u}_T\rceil^{\l})^{(2)}_{\l}\big\|^{2q}\cdot \ov{\E}\big\|\nabla^{\l}\ln p^{\l}(T, x, \lfloor {\rm x}_T \rceil^{\l})\big\|^{2q}\leq \underline{c}(q). 
\end{align*}
Similarly, by i) of  Proposition \ref{est-norm-u-t-j} and Lemma \ref{Hs2-ST}, we can derive  that 
\begin{align*}
\ \   \big(\overline{\E}\big| ({\rm{IV}})^{\l}_3\big|^q\big)^3\leq \ov{\E}\big\|(\lfloor {\rm u}_T\rceil^{\l})^{(1)}_{\l}\big\|^{4q}\cdot \ov{\E}\big\|\nabla^{\l, (2)}\ln p^{\l}(T, x, \lfloor {\rm x}_T \rceil^{\l})\big\|^{q}\leq \underline{c}(q).
\end{align*}
Using i) of Proposition \ref{est-norm-u-t-j}  and (\ref{grad-lnp-lam-1}), we obtain 
\begin{align*}
\ \ \ \ \big(\overline{\E}\big| ({\rm{IV}})^{\l}_2\big|^q\big)^2\leq \ov{\E}\big\|(\lfloor {\rm u}_T\rceil^{\l})^{(1)}_{\l}\big\|^{2q}\cdot \ov{\E}\big\|\nabla^{\l}(\ln p^{\l})^{(1)}_{\l}(T, x, \lfloor {\rm x}_T \rceil^{\l})\big\|^{2q}\leq\underline{c}(q).
\end{align*}
As for  the conditional  expectations,  by  ii) of  Proposition \ref{est-norm-u-t-j} and Lemma \ref{Hs2-ST}, we obtain 
\begin{align*}
\overline{\E}_{\P^{\l, *}_{x, y, T}}\left|({\rm{IV}})^{\l}_1\right|+\overline{\E}_{\P^{\l, *}_{x, y, T}}\left|({\rm{IV}})^{\l}_3\right|\leq &\ \overline{\E}_{\P^{\l, *}_{x, y, T}}\big\|(\lfloor {\rm u}_T\rceil^{\l})^{(2)}_{\l}\big\|\cdot\big\|\nabla^{\l}\ln p^{\l}(T, x, y)\big\|\\
& +\overline{\E}_{\P^{\l, *}_{x, y, T}}\big\|(\lfloor {\rm u}_T\rceil^{\l})^{(1)}_{\l}\big\|\cdot \big\|\nabla^{\l, (2)}\ln p^{\l}(T, x, y)\big\|\\
\leq &\; \underline{c}\left(\big(\frac{1}{T}d_{\wt{g}^{\l}}(x, y)+\frac{1}{\sqrt{T}}\big)^2+1\right)e^{{c}(1+d_{\wt{g}^{\l}}(x, y))}.
\end{align*}
By Corollary \ref{di-p-lam-1},  for some different $\underline{c}, c$, 
\[
p^{\l}(T, x, y)\big\|\nabla^{\l}(\ln p^{\l})^{(1)}_{\l}(T, x, y)\big\|\leq \underline{c}\left(\big(\frac{1}{T}d_{\wt{g}^{\l}}(x, y)+\frac{1}{\sqrt{T}}\big)^2+1\right)e^{{c}(1+d_{\wt{g}^{\l}}(x, y))}.
\]
Using this  and ii) of Proposition \ref{est-norm-u-t-j}, we conclude that  the same type of bound is valid for  the term $p^{\l}(T, x, y)\overline{\E}_{\P^{\l, *}_{x, y, T}}\left|({\rm{IV}})^{\l}_2\right|$.

By Lemma \ref{u-lambda-differential-1},  $\l\mapsto ({\rm{III}})^{\l}$ is also differentiable in $\l$. Its   differential is  given by 
\begin{align*}
\big(({\rm{III}})^{\l}\big)^{(1)}_{\l}=&\ \big\langle D\pi(\lfloor {\rm u}_T\rceil^{\l})^{(2)}_{\l},\  \frac{1}{2}\!\int_{0}^{T}\!\mathtt{s}(T\!-\!\tau) \lfloor{\rm u}_{T}\rceil^{\l}(\lfloor{\rm u}_\tau\rceil^{\l})^{-1}{\rm Ric}_{\lfloor{\rm u}_{\tau}\rceil^{\l}}^{-1}d\oa{B}_{\tau}\big\rangle_{\l}\\
&\ + \big\langle D\pi(\lfloor {\rm u}_T\rceil^{\l})^{(1)}_{\l},\  \frac{1}{2}\!\int_{0}^{T}\!\mathtt{s}(T\!-\!\tau)\big( \lfloor{\rm u}_{T}\rceil^{\l}(\lfloor{\rm u}_\tau\rceil^{\l})^{-1}{\rm Ric}_{\lfloor{\rm u}_{\tau}\rceil^{\l}}^{-1}\big)^{(1)}_{\l}d\oa{B}_{\tau}\big\rangle_{\l}\\
&\ +\big\langle D\pi(\lfloor {\rm u}_T\rceil^{\l})^{(1)}_{\l},\  \frac{1}{2}\!\int_{0}^{T}\!\mathtt{s}(T\!-\!\tau) \lfloor{\rm u}_{T}\rceil^{\l}(\lfloor{\rm u}_\tau\rceil^{\l})^{-1}{\rm Ric}_{\lfloor{\rm u}_{\tau}\rceil^{\l}}^{-1}d\oa{B}_{\tau}\big\rangle_{\l}^{(1)}\\
=:&\ ({\rm{III}})^{\l}_1+({\rm{III}})^{\l}_2+({\rm{III}})^{\l}_3,
\end{align*}
where the last term denotes the differential of the inner product. Then it is standard to estimate the expectation of $\big(({\rm{III}})^{\l}\big)^{(1)}_{\l}$ using  H\"{o}lder's inequality, Burkholder's inequality and  i) of Proposition \ref{est-norm-u-t-j}, which gives
\begin{equation*}
\left(\ov{\E}\big|\big(({\rm{III}})^{\l}\big)^{(1)}_{\l}\big|^q\right)^2\!\leq  \underline{c}(T^q+T^{2q})\!\left(\ov{\E}\big\|(\lfloor {\rm u}_T\rceil^{\l})^{(2)}_{\l}\big\|^{2q}+\ov{\E}\big\|(\lfloor {\rm u}_T\rceil^{\l})^{(1)}_{\l}\big\|^{2q}+ \ov{\E}\big\|(\lfloor {\rm u}_T\rceil^{\l})^{(1)}_{\l}\big\|^{4q}\right)
\!\leq  \underline{c}(q). 
\end{equation*}
For the corresponding  conditional expectation estimation, we use (\ref{B-b-relation-x}), H\"{o}lder's inequality and Burkholder's inequality as before.  It is easy to deduce that 
\begin{align*}
\left(\overline{\E}_{\P^{\l}_{x, y, T}}\big|\big(({\rm{III}})^{\l}\big)^{(1)}_{\l}\big|\right)^2\leq & \left(\overline{\E}_{\P^{\l}_{x, y, T}}\big\|(\lfloor {\rm u}_T\rceil^{\l})^{(2)}_{\l}\big\|^2+\overline{\E}_{\P^{\l}_{x, y, T}}\big\|(\lfloor {\rm u}_T\rceil^{\l})^{(1)}_{\l}\big\|^2+\overline{\E}_{\P^{\l}_{x, y, T}}\big\|(\lfloor {\rm u}_T\rceil^{\l})^{(1)}_{\l}\big\|^4\right)\\
&\cdot \underline{c}\left(T+\overline{\E}_{\P^{\l}_{x, y, T}}\big|\int_{0}^{T}\big\|2\nabla^{\l}\ln p(T-\tau, \lf {\rm{x}}_{\tau}\rc^{\l}, y)\big\|\ d\tau\big|^{2}\right).
\end{align*}
So by ii) of Proposition \ref{est-norm-u-t-j} and Proposition \ref{cond-nabla-ln-p}, we have 
\[
\overline{\E}_{\P^{\l, *}_{x, y, T}}\big|\big(({\rm{III}})^{\l}\big)^{(1)}_{\l}\big|\leq \underline{c}e^{{c}(1+d_{\wt{g}^{\l}}(x, y))}.
\]
For $({\rm{II}})^{\l}$,  the same argument gives 
\[
\ov{\E}\big|\big(({\rm{II}})^{\l}\big)^{(1)}_{\l}\big|^q\leq \underline{c}(q), \ \overline{\E}_{\P^{\l, *}_{x, y, T}}\big|\big(({\rm{II}})^{\l}\big)^{(1)}_{\l}\big|\leq \underline{c}e^{{c}(1+d_{\wt{g}^{\l}}(x, y))}.
\]
For $({\rm I})^{\l}$, we can check the differentiability of  $\nabla^{\l}_{T, {\rm V}, \mathtt{s}}D\pi(\lfloor {\rm u}_T\rceil^{\l})^{(1)}_{\l}$  term-by-term  using its expression in Corollary \ref{cor-u-s-t-l-1}. The estimation can be done  as above  using Proposition \ref{est-norm-u-t-j}.

By the inductive construction,  $\wt{\phi}_{\l}^{j-1}(T, x, w)$ involves the mixed differentials of order $j$ in $\l$ and in $\nabla^{\l}_{T, \mathtt{s}}$ of $\lfloor u_T\rceil^{\l}$  and can be expressed by  a multi-stochastic integral involving a mixture of differential processes $\{\big(\lf {\rm u}_t\rc^{\l}\big)^{(j')}_{\l}\}_{j'\leq j-1}$, $\{\big[ D^{(1)}\lf F_{\underline{t}, \overline{t}}\rc^{\l}({\lf {\rm{u}}_{\underline{t}}\rc^{\l}}, w)\big] ^{(i)}_{\l}\}_{i\leq j-2}$ and  $\{\nabla^{\l, (l)}(\ln p^{\l})^{(i)}_{\l}(T, x, \lf {\rm x}_T\rc^{\l})\}_{l+i\leq j-1, i\leq j-2}$. So, by Lemma \ref{u-lambda-differential-1} and Proposition \ref{est-norm-u-t-j}, we have the differentiability of $\l\mapsto \wt{\phi}_{\l}^{j-1}(T, x, w)$ and the derivative $(\wt{\phi}_{\l}^{j-1})^{(1)}_{\l}(T, x, w)$ involves     $\{\big(\lf {\rm u}_t\rc^{\l}\big)^{(j')}_{\l}\}_{j'\leq j}$, $\{\big[ D^{(1)}\lf F_{\underline{t}, \overline{t}}\rc^{\l}({\lf {\rm{u}}_{\underline{t}}\rc^{\l}}, w)\big] ^{(i)}_{\l}\}_{i\leq j-1}$ and  $\{\nabla^{\l, (l)}(\ln p^{\l})^{(i)}_{\l}(T, x, \lf {\rm x}_T\rc^{\l})\}_{l+i\leq j, i\leq j-1}$. The estimations in {\bf i)} and {\bf iv)} will follow from a repeated application of  Proposition \ref{est-norm-u-t-j} to the multiple stochastic integral as in the $j=2$ case. The  bound estimation in  {\bf iv)} contains $\big(T^{-1}d_{\wt{g}^{\l}}(x, y)+{(\sqrt{T})}^{-1}\big)^j$ since the formula of  $(\wt{\phi}_{\l}^{j-1})^{(1)}_{\l}(T, x, w)$ contains  the terms  $\nabla^{\l, (j)}\ln p^{\l}(T, x, \lf {\rm x}_T\rc^{\l})$, $\nabla^{\l, (j-1)}(\ln p^{\l})^{(1)}_{\l}(T, x, \lf {\rm x}_T\rc^{\l})$. 

As to the continuity and its uniformity in $T$ and $\l$ of the map \[
y\mapsto  \varPsi^{j}(y):=\overline{\E}\left(\left. (\wt{\phi}_{\l}^{j-1})^{(1)}_{\l}(T, x, w)\right| {\lfloor {\rm x}_T \rceil^{\l}(w)}=y\right),
\]
we compare $\varPsi^{j}(y)$ with $\varPsi^{j}(\leftidx^r F(y))$, where $\{\leftidx^r F\}_{r\in \Bbb R}$ is the flow map generated by a bounded smooth vector field ${\rm W}$ on $\M$.  Let ${\lf\leftidx^{r}{\bf{F}}\rc^{\l}}$ be as in Section \ref{sec5} which extends $\leftidx^r F$ to the $\wt{g}^{\l}$-Brownian paths. Then,  as in the proof of Theorem \ref{regu-p-1st}, we obtain \[
\varPsi^{j}(\leftidx^r F(y))=\overline{\E}\!\left(\!\left.(\wt{\phi}_{\l}^{j-1})^{(1)}_{\l} \circ  {\lf\leftidx^{r}{\bf{F}}\rc^{\l}} \cdot \frac{d\overline{\P}_x^{\l}\circ  {\lf\leftidx^{r}{\bf{F}}\rc^{\l}} }{d\overline{\P}_x^{\l}}\right|\lfloor {\rm x}_T \rceil^{\l}=y\!\right)\cdot\frac{p^{\l}(T, x, y)}{p^{\l}(T, x, \leftidx^{r}F (y))}\cdot\frac{d{\rm Vol}^{\l}}{d{\rm Vol}^{\l}\circ \leftidx^{r}F}(y).\]
In the proof of Lemma \ref{C-1-fun-phi-Y}, we  obtained the local uniform  boundedness in $(T, y)$  and $\l$ of 
\[
\overline{\E}_{\ov{\P}^{\l}_{x, y, T}}\big\|{d\overline{\P}_x^{\l}\circ  {\lf\leftidx^{r}{\bf{F}}\rc^{\l}} }/{d\overline{\P}_x^{\l}}\big\|^q
\]
and the local uniformity  in $(T, y)$ and $\l$ of  the convergence of 
\[
\overline{\E}_{\ov{\P}^{\l}_{x, y, T}}\left\|{d\overline{\P}_x^{\l}\circ  {\lf\leftidx^{r}{\bf{F}}\rc^{\l}} }/{d\overline{\P}_x^{\l}}-1\right\|^2\to 0, \ {as}\ r\to 0.
\]
Following the estimation for {\bf{iv})}, we  obtain  the local uniform  boundedness in $(T, y)$ and $\l$ of 
\[
\overline{\E}_{\ov{\P}^{\l}_{x, y, T}}\big\|(\wt{\phi}_{\l}^{j-1})^{(1)}_{\l}(T, x, w)\circ {\lf\leftidx^{r} {\bf{F}}\rc^{\l}}\big\|^2. 
\]
So for {\bf iv)}, it remains to show the local uniform  convergence  in $(T, y)$ and $\l$ of 
\[
\overline{\E}_{\ov{\P}^{\l}_{x, y, T}}\left\|(\wt{\phi}_{\l}^{j-1})^{(1)}_{\l}(T, x, w)\circ {\lf\leftidx^{r} {\bf{F}}\rc^{\l}}-(\wt{\phi}_{\l}^{j-1})^{(1)}_{\l}(T, x, w)\right\|^2\to 0, \ {\rm{as}}\ r\to 0. 
\]
This, by using ii) of Proposition \ref{est-norm-u-t-j},  can be reduced to showing the local uniformity  in $(T, y)$ and $\l$ of  the convergence of  
\[
\overline{\E}_{\ov{\P}^{\l}_{x, y, T}}\left\|{\bf A}\circ {\lf\leftidx^{r} {\bf{F}}\rc^{\l}}-{\bf A}\right\|^2\to 0, \ {\rm{as}}\ r\to 0,\]  for elements   $\lf {\rm u}_t\rc^{\l}$, $\{\big(\lf {\rm u}_t\rc^{\l}\big)^{(j')}_{\l}\}_{j'\leq j}$ and  $\{\big[ D^{(1)}\lf F_{\underline{t}, \overline{t}}\rc^{\l}({\lf {\rm{u}}_{\underline{t}}\rc^{\l}}, w)\big] ^{(i)}_{\l}\}_{i\leq j-1}$ that appear in the expression of $(\wt{\phi}_{\l}^{j-1})^{(1)}_{\l}$, which is true since they can be further reduced to the ${\bf A}$  appearing in Lemma \ref{C-1-fun-phi-Y}  by the construction of ${\lf\leftidx^{r}{\bf{F}}\rc^{\l}}$.  \end{proof}

\begin{proof}[Proof of properties {\bf ii)}, {\bf iii)} and {\bf v)} in Section \ref{skectch-6.1}] Using (\ref{DF-j-lambda-cond}) and the inductive assumption on the boundedness of  $
\ov{\E}_{\ov{\P}^{\l}_{x, y, T}}\big\|\wt{\phi}_{\l}^{j-1}(T, x, w)\big\|^q$, 
we deduce that  $\overline{\Phi}_{\l}^j:$ $Y\mapsto \overline{\Phi}_{\l}^j(Y)$, where 
\begin{align*}\overline{\Phi}_{\l}^j(Y)(y)&:=\overline{\E}\left(\left.\big\langle Y(\lfloor {\rm x}_T\rceil^{\l}(w)),  \wt{\phi}_{\l}^{j-1}(T, x, w)\cdot D\pi(\lfloor {\rm u}_T\rceil^{\l})^{(1)}_{\l}(w) \big\rangle_{\l}\right| \lfloor {\rm x}_T \rceil^{\l}(w)=y\right)
\end{align*}
is a locally bounded  functional  on $C^k$ bounded vector fields $Y$ on $\M$. 

 To show $\overline{\Phi}_{\l}^j(Y)$ is $C^1$, we follow the argument in the proof of Lemma \ref{C-1-fun-phi-Y}. Let $\{F^s\}_{s\in \Bbb R}$ be the flow generated by a smooth bounded vector field ${\rm V}$ on $\M$.  Let  $\lf {\bf F}^s\rc^{\l}$ be  constructed as in Section \ref{the flow F-S}, which extends  $F^s$ to Brownian paths starting from $x$ up to time $T$ using the auxiliary  function $\mathtt s$. Then the change of variable comparison in Section \ref{Obs-Stra}  gives 
\begin{align*}
\overline{\Phi}_{\l}^j(Y)(F^sy)&=\overline{\E}_{\ov{\P}^{\l}_{x, y, T}}\left(\Phi_{\l}^j(Y, w)\circ \lf {\bf F}^s\rc^{\l} \cdot \frac{d\overline{\P}_x^{\l}\circ \lf {\bf F}^s\rc^{\l}}{d\overline{\P}_x^{\l}}\right)\!\frac{p^{\l}(T, x, y)}{p^{\l}(t, x, F^s y)}\frac{d{\rm Vol}^{\l}}{d{\rm Vol}^{\l}\circ F^s}(y),
\end{align*}
where 
\[\Phi_{\l}^j(Y, {w})=\big\langle Y(\lfloor {\rm x}_T\rceil^{\l}(w)),  \wt{\phi}_{\l}^{j-1}(T, x, w)\cdot D\pi(\lfloor {\rm{u}}_T\rceil^{\l})^{(1)}_{\l}(w) \big\rangle_{\l}.\] 
The process $\Phi_{\l}^j(Y, w)\circ \lf {\bf F}^s\rc^{\l}$ is differentiable in $s$ with 
\begin{align*}
(\Phi_{\l}^j\circ \lf {\bf F}^s\rc^{\l})'_s=&\ \big\langle \nabla_{{\rm V}(\lfloor {\rm x}_{T}^s \rceil^{\l})}Y(\lfloor {\rm x}_{T}^s \rceil^{\l}),  \wt{\phi}_{\l}^{j-1}\circ  \lf {\bf F}^s\rc^{\l}  D\pi(\lfloor {\rm u}_T^s\rceil^{\l})^{(1)}_{\l}\big\rangle_{\l}\\
& + \big\langle Y(\lfloor {\rm x}_{T}^s \rceil^{\l}),  \big(\nabla_{T, {\rm V}, \mathtt{s}}^{s, \l} \wt{\phi}_{\l}^{j-1}\big) \circ\lf {\bf F}^s\rc^{\l} D\pi(\lfloor {\rm u}_T\rceil^{\l})^{(1)}_{\l}\big\rangle_{\l}\\
&+ \big\langle Y(\lfloor {\rm x}_{T}^s \rceil^{\l}),  \wt{\phi}_{\l}^{j-1}\circ  \lf {\bf F}^s\rc^{\l} \nabla_{T, {\rm V}, \mathtt{s}}^{s, \l} D\pi(\lfloor {\rm u}_T\rceil^{\l})^{(1)}_{\l}\big\rangle_{\l}
\end{align*}
and  this differential is  $L^q$ integrable  conditioned on ${\rm x}_T=y$, uniformly in $s$,  for all $q\geq 1$.  Using this and  Proposition \ref{Quasi-IP-1}, we can
conclude that $\overline{\Phi}_{\l}^j(Y)(F^sy)$ is differentiable in $s$.  Following  the proof of Lemma \ref{C-1-fun-phi-Y} (see (\ref{expression-di-phi})), we obtain \begin{align}
(\overline{\Phi}_{\l}^j(Y)(F^sy))'_0&=\overline{\E}_{\ov{\P}^{\l}_{x, y, T}}\left(\big\langle \nabla_{{\rm V}(\lfloor {\rm x}_{T} \rceil^{\l})}Y,   \wt{\phi}_{\l}^{j-1}\cdot D\pi(\lfloor {\rm u}_T\rceil^{\l})^{(1)}_{\l}(w) \big\rangle_{\l} \right.\notag \\
&\ \ \ \ \  \ \ \ \ \ \ \ \  \  +\big\langle Y(\lfloor {\rm x}_{T} \rceil^{\l}), \nabla_{T, {\rm V}, \mathtt{s}}^{\l}\big(  \wt{\phi}_{\l}^{j-1}\cdot D\pi(\lfloor {\rm u}_T\rceil^{\l})^{(1)}_{\l}(w) \big)\big\rangle_{\l}\notag \\
&\ \ \ \ \  \ \ \ \ \ \ \ \  \  +\left.\big\langle Y(\lfloor {\rm x}_{T} \rceil^{\l}),  \wt{\phi}_{\l}^{j-1}\cdot D\pi(\lfloor {\rm u}_T\rceil^{\l})^{(1)}_{\l}(w) \big\rangle_{\l}\oa{\mathcal{E}}_{T, {\rm V}, {\mathtt s}} \right)\label{expression-di-phi-j}\\
&=:\overline{\E}_{\ov{\P}^{\l}_{x, y, T}}\left( {\Psi}_{\l}^j(Y, {\rm V})(w)\right).\notag
\end{align}
To show  $y\mapsto (\overline{\Phi}_{\l}^j(Y)(F^sy))'_0$ is continuous, we compare  (\ref{expression-di-phi-j})  with  its value at  nearby points.  Choose another smooth bounded vector field ${\rm W}$ on $\M$ and let $\{\leftidx^{r}F\}_{r\in \Bbb R}$ be the flow it generates and let  $\lf \leftidx^{r}{\bf F}\rc^{\l}$  be its extension to  $\wt{g}^{\l}$-Brownian paths starting from $x$ up to time $T$.  A change of variable argument in  Section \ref{Obs-Stra} for $\lf \leftidx^{r}{\bf{F}}\rc^{\l}$ shows that  for $z=\leftidx^{r}F (y)$, 
\[
(\overline{\Phi}_{\l}^j(Y)(F^sz))'_0=\overline{\E}_{\ov{\P}^{\l}_{x, y, T}}\left(\Psi_{\l}^j(Y, {\rm V})\circ \lf {\leftidx^{r}{\bf{F}}}\rc^{\l} \cdot {\frac{d\overline{\P}_x^{\l}\circ \lf {\leftidx^{r}{{\bf F}}}\rc^{\l}}{d\overline{\P}_x^{\l}}}\right)\frac{p^{\l}(T, x, y)}{p^{\l}(T, x, z)}\frac{d{\rm Vol}^{\l}}{d{\rm Vol}^{\l}\circ{\leftidx^{r}{F}}}(y). 
\]
We can show the local uniform convergence (in $(y, T)$ and $\l$)  of $(\overline{\Phi}_{\l}^j(Y)(F^sz))'_0$ to $(\overline{\Phi}_{\l}^j(Y)(F^sy))'_0$ as $r\to 0$ exactly as we did in the previous proofs of properties {\bf i)} and {\bf iv)}.

As to the estimations in {\bf v)}, $\big|{\rm z}_T^{\l, j}(y)\big|$ can be estimated using  the conditional $L^q$ expectations of $\wt{\phi}_{\l}^{j-1}(T, x, w),(\lfloor {\rm{u}}_T\rceil^{\l})^{(1)}_{\l}(w)$, respectively.  By (\ref{expression-di-phi-j}), we have the  formula in {\bf iii)}. By Claim \ref{Cla-i-ii-}, we obtain the formula of $\nabla^{\l}_{{\rm V}}{\rm z}_t^{\l, j}(y)$ in (\ref{DIV-Z-j}). We can use them  and  Proposition \ref{est-norm-u-t-j} to give  the desired estimation of ${\rm{Div}}^{\l}{\rm z}_t^{\l, j}(y)$. 
\end{proof}

\begin{proof}[Proof of properties {\bf vi)} and {\bf vii)} in Section \ref{skectch-6.1}] The $j=1$ case was considered in Theorem  \ref{regu-p-1st}.
When $j=2$, since we have  (\ref{p-diff-induction-i}) for $i=1, 2$, so,  for all $f\in C_{c}^{\infty}(\M)$, 
\begin{equation*}
\int_{\M}f(y) \phi_{\l}^{2}(T, x, y)p^{\l}(T, x, y)\ d{\rm{Vol}}^{\l}(y)=\left(\int_{\M}f(y) \phi_{\l}^{1}(T, x, y)p^{\l}(T, x, y)\ d{\rm{Vol}}^{\l}(y)\right)^{(1)}_{\l}. 
\end{equation*}
This  implies
\begin{align*}
\phi_{\l}^{2}(T, x, y)&=(\phi_{\l}^{1})^{(1)}_{\l}(T, x, y)+\phi_{\l}^{1}(T, x, y)\cdot \left(\ln(p^{\l}(T, x, y)\rho^{\l}(y))\right)^{(1)}_{\l}\\
&= (\phi_{\l}^{1})^{(1)}_{\l}(T, x, y)+\left(\phi_{\l}^{1}(T, x, y)\right)^2.
\end{align*}
Hence, 
\begin{align*}
(\ln p^{\l})^{(2)}_{\l}(T, x, y)=\phi_{\l}^{2}(T, x, y)-(\phi_{\l}^1)^2(T, x, y)-(\ln \rho^{\l})^{(2)}_{\l}(y)
\end{align*}
and 
\begin{align*}
\nabla (\ln p^{\l})^{(2)}_{\l}(T, x, y)=\nabla \phi_{\l}^{2}(T, x, y)-2\phi_{\l}^1(T, x, y)\nabla \phi_{\l}^{1}(T, x, y)-\nabla (\ln \rho^{\l})^{(2)}_{\l}(y). 
\end{align*}
This, together with {\bf vi)}, {\bf vii)} in the $j=1$ case,  shows that the estimation for the term   $|\nabla^{(l)}(\ln p^{\l})^{(2)}_{\l}(T, x, y)|$ in (\ref{esti-p-lam-der-2}) holds true if  the same type of  estimation is valid for  $|\nabla^{(l)}\phi_{\l}^{2}(T, x, y)|$.  By {\bf i)}-{\bf v)},  Claims \ref{cla-1-11-ov-phi}-\ref{Cla-i-ii-} apply.  We have 
\begin{equation}\label{phi-l-2-ex}
\phi_{\l}^{2}(T, x, y)=\ov{\E}\left(\left.\wt{\phi}_{\l}^{2}(T, x, w)\right|\lfloor {\rm x}_T \rceil^{\l}(w)=y \right),
\end{equation}
where 
\begin{align*}{\wt{\phi}}_{\l}^{2}(T, x, w)=&\big({\wt{\phi}}_{\l}^{1}\big)^{(1)}_{\l}\!(T, x, w)\!-\!\big\langle \nabla^{\l}_{T, \mathtt{s}}\wt{\phi}_{\l}^{1}\!(T, x, w), D\pi(\lfloor {\rm{u}}_T\rceil^{\l})^{(1)}_{\l}(w)\big\rangle+{\wt{\phi}}_{\l}^{1}(T, x, w){\wt{\phi}}_{\l}^{1}(T, x, w).
\end{align*}
We can use (\ref{phi-l-2-ex}) to derive the conditional expectation expressions of $\nabla^{(l)}\phi_{\l}^{2}(T, x, y)$ as in the proof of Theorem  \ref{regu-p-1st}. Using this and Proposition \ref{est-norm-u-t-j}, we can derive the desired estimations of   $\nabla^{(l)}\phi_{\l}^{2}(T, x, y)$ and its $L^q$-norm. 

For $j\geq 3$, with {\bf i)}-{\bf v)}, we  have the identity
\begin{align}
\notag&\nabla^{(l)} (\ln p^{\l})^{(j)}_{\l}(T, x, y)\\
\label{lnp-j-phi-j}&\ \ \ =\nabla^{(l)}\phi_{\l}^{j}(T, x, y)-\!\sum_{i=1}^{j-1}\nabla^{(l)}\big(\phi_{\l}^{i}\cdot \phi_{\l}^{1}\big)^{(j-i-1)}_{\l}\!(T, x, y)-\!\nabla^{(l)} \big(\ln \rho^{\l}\big)^{(j)}_{\l}(y).  
\end{align}
By Theorem \ref{regu-p-1st} i), 
\[
\phi_{\l}^{1}(T, x, y)=\left(\ln(p^{\l}(T, x, y)\rho^{\l}(T, x, y))\right)^{(1)}_{\l}=(\ln p^{\l})^{(1)}_{\l}(T, x, y)+ (\ln \rho^{\l})^{(1)}_{\l}(y). 
\]
By (\ref{p-diff-induction-i}), for all $i$,  $i\leq j$,   and  all $f\in C_c^{\infty}(\M)$, 
\begin{align*}
\int_{\M}f(y) \phi_{\l}^{i}(T, x, y)p^{\l}(T, x, y)\ d{\rm{Vol}}^{\l}(y)=\left(\int_{\M}f(y) \phi_{\l}^{i-1}(T, x, y)p^{\l}(T, x, y)\ d{\rm{Vol}}^{\l}(y)\right)^{(1)}_{\l}.
\end{align*}
Since $f$ is arbitrary, we must have  
\begin{align*}
\phi_{\l}^{i}(T, x, y)&=(\phi_{\l}^{i-1})^{(1)}_{\l}(T, x, y)+\phi_{\l}^{i-1}(T, x, y)\cdot \left(\ln(p^{\l}(T, x, y)\rho^{\l}(T, x, y))\right)^{(1)}_{\l}\\
&= (\phi_{\l}^{i-1})^{(1)}_{\l}(T, x, y)+\phi_{\l}^{i-1}(T, x, y)\cdot \phi_{\l}^{1}(T, x, y).
\end{align*}
Using this relationship inductively, we obtain 
\[
\phi_{\l}^{j}(T, x, y)=(\phi_{\l}^1(T, x, y))^{(j-1)}_{\l}+\sum_{i=1}^{j-1}\big(\phi_{\l}^{i}\cdot \phi_{\l}^{1}\big)^{(j-i-1)}_{\l}(T, x, y),
\]
which implies \[
(\ln p^{\l})^{(j)}_{\l}(T, x, y)=\phi_{\l}^{j}(T, x, y)-\sum_{i=1}^{j-1}\big(\phi_{\l}^{i}\cdot \phi_{\l}^{1}\big)^{(j-i-1)}_{\l}(T, x, y)-\big(\ln \rho^{\l}\big)^{(j)}_{\l}(y). 
\]
A differentiation of this equation gives (\ref{lnp-j-phi-j}).  By induction, we see that  the differentials $(\phi_{\l}^{i})^{(r)}_{\l}(T, x, y)$ for $r\leq j-i-1$ only  consist  of $(\ln p^{\l})^{(s)}_{\l}(T, x, y), \big(\ln \rho^{\l}\big)^{(s)}_{\l}(y)$ up to order $s=i+r\leq j-1$. By the inductive assumption on the  gradient estimations of $(\ln p^{\l})^{(s)}_{\l}(T, x, y)$, $s\leq j-1$, to obtain {\bf vii)} for $\big|\nabla^{(l)}(\ln p^{\l})^{(j)}_{\l}(T, x, y)\big|$, it suffices to give the estimation for $|\nabla^{(l)}\phi_{\l}^{j}(T, x, y)|$.  By {\bf i)}-{\bf v)},  Claims \ref{cla-1-11-ov-phi}-\ref{Cla-i-ii-} apply and we have 
\[
\phi_{\l}^{j}(T, x, y):=\ov{\E}\left(\left.\wt{\phi}_{\l}^{j}(T, x, w)\right|\lfloor {\rm x}_T \rceil^{\l}(w)=y \right),
\]
where $\wt{\phi}_{\l}^{j}(T, x, w)$ is defined inductively using (\ref{wt-phi-induction-i}).  As in the proof of  Theorem \ref{regu-p-1st}, we can further  obtain $\wt{\phi}_{\l}^{j, (l)}(T, x, w)$ such that \[
\nabla^{(l)}\phi_{\l}^{j}(T, x, y)=\ov{\E}\left(\left.\wt{\phi}_{\l}^{j, (l)}(T, x, w)\right|\lfloor {\rm x}_T \rceil^{\l}(w)=y \right).
\]
The term $\wt{\phi}_{\l}^{j, (l)}(T, x, w)$ involves the derivatives of $\wt{\phi}_{\l}^{j}(T, x, w)$ under  $\leftidx^{r}{\bf F}$ up to the $j$-th order and can be formulated as a stochastic integral using $\mathtt s, \mathtt s'$, $\lf {\rm{u}}\rc^{\l}$, $\{\big(\lf {\rm u}_t\rc^{\l}\big)^{(j')}_{\l}\}_{j'\leq j}$, $\{\big[ D^{(1)}\lf F_{\underline{t}, \overline{t}}\rc^{\l}({\lf {\rm{u}}_{\underline{t}}\rc^{\l}}, w)\big] ^{(i)}_{\l}\}_{i\leq j-1}$ and  $\{\nabla^{\l, (l)}(\ln p^{\l})^{(i)}_{\l}(T, x, \lf {\rm x}_T\rc^{\l})\}_{l+i\leq j, i\leq j-1}$. So we can  use this and ii) of Proposition \ref{est-norm-u-t-j} to  derive the desired bound of  $|\nabla^{(l)}\phi_{\l}^{j}(T, x, y)|$ as in  Theorem  \ref{regu-p-1st}.  As to {\bf vi)}, it is equivalent to estimate  $\big(\ov{\E}(\|\wt{\phi}_{\l}^{j, (l)}(T, x, w)\|^{q})\big)^{\frac{1}{q}}$, which can be handled using i) of Proposition \ref{est-norm-u-t-j} and the inductive assumption on (\ref{esti-p-lam-der-k}) for $i<j$. 
\end{proof}

\section{Regularity of the entropy}

The analog of formula (\ref{ell-lambda}) for the entropy involves the Martin kernel of the Brownian motion on $(\M, \wt{g}^{\l})$ for ${g}^{\l}\in \Re^k(M)$. Recall that the Green function on $(\M, \wt{g}^{\l})$ is given by 
\[
{\bf G}^{\l}(x, y):=\int_{0}^{\infty}p^{\l}(t, x, y)\ dt, \ \mbox{for}\ x, y\in \M, 
\]
and it can be  associated with a ``Green metric''  on $\M$ (\cite{LS2}) by letting
\begin{align*}
d_{{\bf G}^{\l}}(x, y):=\left\{ \begin{array}{ll}
-\ln ({\mathtt{c}_0}{\bf G}^{\l}(x, y)), &\mbox{if}\ d_{\wt{g}^{\l}}(x, y)>1,\\
-\ln {\mathtt{c}_0}, &\mbox{otherwise},
\end{array}
 \right.
\end{align*}
where $\mathtt{c}_0$ can be chosen to be independent of $\l$ for $g^{\l}$ in a small neighborhood of $g^0$.
By Anderson and  Schoen \cite{AS} (see also \cite{Anc}), the Martin kernel for $(\M, \wt{g}^{\l})$  is defined by 
\begin{equation}\label{Martin-kernel}
k^{\l}(x, y, \xi):=\lim_{z\to \xi}k^{\l}(x, y, z), \ \mbox{where}\  
k^{\l}(x, y, z):=\frac{{\bf G}^{\l}(y, z)}{{\bf G}^{\l}(x, z)}. 
\end{equation}
Hence, the logarithm of the Martin kernel is an  analog of   the  Busemann function using the Green metric since 
\[
\ln k^{\l}(x, y, \xi):=\lim\limits_{z\to \xi} \big(d_{{\bf G}^{\l}}(x, z)-d_{{\bf G}^{\l}}(y, z)\big), \ \mbox{for}\ x, y\in \M, \ \xi\in \partial \M. 
\]
Moreover,  it is known that the entropy satisfies the following formula (\cite{K1})
\begin{equation*}\label{h-lambda}
h^{\l}=-\int \left.\Delta^\l_y \ln k^\l(x, y, \xi)\right|_{y=x} \ d{\bf{m}}^\l(x, \xi). 
\end{equation*}

For  $x, y\in \M$ fixed, the function $k^{\l}(x, y, \xi)$ is a continuous version of the Radon-Nikodyn derivative $(d\wt{{\bf m}}_y/d\wt{{\bf m}}_x)(\xi)$; the gradient 
\[\overline{Z}(x, \xi):=\nabla^{\l}_y  k^{\l}(x, y, \xi)|_{y=x}\]
is a $G$-equivariant stable vector field  that depends H\"{o}lder continuously on $\xi$ with the H\"{o}lder  exponent uniformly in $\l$ for $g^{\l}$ in a small  neighborhood of $g^0$ (\cite{AS}, see also \cite{Ha}). Furthermore,  we have the following.  
\begin{lem}\label{Holder-2nd-k}
 Let  $M$ be a closed connected smooth manifold.  For each $g\in \Re^{k}(M)$ ($k\geq 3$), there  exist a neighborhood $\mathcal{V}''_g$ of $g$ in $\Re^{k}(M)$ and ${\mathtt{b}}''$, ${\mathtt{b}}''>0$,   such that for any $C^k$ curve $\lambda\in
(-1, 1)\mapsto g^{\lambda}\in \mathcal{V}''_g$  with $g^0=g$,  the second order differentials of $k^{\l}(x, y, \xi)$ in $y$ at $y=x$ are  H\"{o}lder continuous in $\xi$ with exponent ${\mathtt{b}}''$;  for  ${\mathtt{b}}<{\mathtt{b}}''\varkappa$, where $\varkappa$ is as in (\ref{distance-bd}), we have 
\begin{equation}\label{Delta-lnk-2}
\Delta^{\l}_y \ln k^0(x, y, \xi)\big|_{y=x},\  (\Delta^{\l}_y)'_{\l} \ln k^0(x, y, \xi)\big|_{y=x}\in \mathcal{H}_{{\mathtt{b}}}^0. 
\end{equation}
\end{lem}
\begin{proof} The second part follows from \cite[Theorem 6.2]{AS} and the first part. We show  the first part  by following the proof of   \cite[Lemma 3.2]{Ha}.

Let $x\in \M$ and let $B(x, \delta)$ be a small neighborhood around $x$ with a positive  radius $\delta$. For ${\bf v}=(x', \zeta)\in S\M_{\wt{g}}$ with $x'\in B(x, 2\delta)\backslash B(x, \delta)$  and $\rho, 0 < \rho \leq \pi/2$, let 
\[
{\rm C^\l} ({\bf v}, \rho):=\left\{z\in \M:\  \angle_{x'}^{\wt{g}^{\l}} ({\bf v}, \dot{\gamma}_{x', z}^{\wt{g}^{\l}}(0))<\rho \right\}, \quad {\rm C} ({\bf v}, \rho) := {\rm C^0} ({\bf v}, \rho)
\]
be the open  cone of vertex $x'$, axis ${\bf v}$ and angle $\rho$, 
where  $\angle_{x'}^{\wt{g}^{\l}}(\cdot, \cdot)$ is the $\wt{g}^{\l}$ angle function in $S_{x'}\M$ and $\dot{\gamma}_{x', z}^{\wt{g}^{\l}}(0)$ is the initial tangent vector  of the ${\wt{g}^{\l}}$ unit speed geodesic $\gamma_{x', z}^{\wt{g}^{\l}}$ from $x'$ to $z$.  

 There  exists a neighborhood $\mathcal{V}_g$ of $g$ in $\Re^{k}(M)$ such that  if $g^\l  \in  \mathcal{V}_g,$ then for  all ${\bf v},  {\rm C} ({\bf v}, \pi /6)\subset {\rm C^\l} ({\bf v}, \pi /4) \subset {\rm C} ({\bf v}, \pi/3)$ and for all $x \in \M, B(x, \d/4)\subset B_{g^{\l}} (x,\d/2) \subset B(x, \d).$\footnote{There is a neighborhood $\mathcal {V}_g$ of $g$ in $\mathfrak {R}^3(M)$ and a number $r$ such that for $\wt{g}, \wt{g}' \in \mathcal {V}_g, \tau \geq r,$ \[  \angle_{x'}^{\wt {g}'} (\dot{\gamma}^{\wt {g}'}_{x', \gamma ^{\wt {g}'}_{x,\xi (\tau)}}, \dot{\gamma}^{\wt {g}}_{x', \gamma _{x,\xi (\tau)}^{\wt {g}}})<\angle_{x'}^{\wt {g}'} (\dot{\gamma}^{\wt {g}'}_{x', \gamma ^{\wt {g}'}_{x,\xi (r)}}, \dot{\gamma}^{\wt {g}}_{x', \gamma^{\wt {g}}_{x,\xi (r)}} )+ \pi / 100.\] It suffices then to control the angles on $B(x, r+ 2\d)$. } Let $\{x_{s,t}\}_{|s|, |t| <1}$, with $x_{0,0} = x$,  be any $C^2$ two parameter family of points in $B(x, \delta/4)$.   For ${\rm C} ({\bf v}, \pi/2)$ apart from $B(x, \delta)$ and $z\in {\rm C}  ({\bf v}, \pi/2)$, let 
\begin{equation*}
\varphi_{s, t}(z):=\frac{1}{st}\left(k^{\l} (x,x_{s,t},z) -   k^{\l} (x,x_{0,t},z) -  k^{\l} (x,x_{s,0},z) +  k^{\l} (x,x_{0,0},z)\right). 
\end{equation*}
To conclude the first part of Lemma \ref{Holder-2nd-k}, it suffices to show for $\mathcal{V}_g$ small, there is some  $C$, $C>0$,  independent of $s, t, x, z$ and $g^{\l}$ such that
\begin{equation}\label{bound of quotient}
|\varphi_{s, t}(z)|\leq C.
\end{equation}
This is because (\ref{bound of quotient}) implies that, for $z \in {\rm C^\l} ({\bf v}, \pi /4)$,
\[
\varphi_{s, t}(z)+C=\frac{\frac{1}{st}\left({\bf G}^{\l} (x_{s,t},z) -  {\bf G}^{\l}(x_{0,t},z) -  {\bf G}^{\l}(x_{s,0},z) +  {\bf G}^{\l}(x_{0,0},z)\right)+C{\bf G}^{\l}(x, z)}{{\bf G}^{\l}(x, z)}
\]
is the quotient of two positive harmonic functions in  ${\rm C^\l} ({\bf v}, \pi/4)$ which vanish at the infinity boundary $\partial \M$.  Hence, by using \cite[Theorem 6.2]{AS}, for $\mathcal{V}_g$ small,  we obtain  two  positive numbers $C', {\mathtt{b}}''$,  independent of $s, t, x$ and $g^{\l}$,  such that   
\begin{equation}\label{holder-z-z'}\left|(\varphi_{s, t}(z)+C)-(\varphi_{s, t}(z')+C)\right|\leq C'e^{-{\mathtt{b}}''\varkappa(z|z')_x^{\l}}, \ \forall z, z'\in {\rm C^{\l}} ({\bf v},\pi /4).
\end{equation}
Let  $\xi, \eta\in \partial \M$ be points lying in the closure of  ${\rm C} ({\bf v}, \pi/6)$.  Letting $z\to \xi, z'\to \eta$ and then  letting $s, t\to 0$ in  (\ref{holder-z-z'}), the first part statement of Lemma \ref{Holder-2nd-k} follows by using (\ref{Martin-kernel}).

It remains to show (\ref{bound of quotient}), or, equivalently,  
\begin{equation}\label{gradient-Green}
\left|\frac{1}{st}\left({\bf G}^{\l} (x_{s,t},z) -  {\bf G}^{\l}(x_{0,t},z) -  {\bf G}^{\l}(x_{s,0},z) +  {\bf G}^{\l}(x_{0,0},z)\right)\right|\leq C{\bf G}^{\l}(x, z).
\end{equation}
Since ${\bf G}^{\l}(\cdot, z)$ is harmonic in $B_{g^{\l}}(x, \delta /2)$, by the Harnack inequality (\cite{AS}) and 
the infinitesimal Harnack inequality of Cheng-Yau (\cite{CY}), for $\mathcal{V}_g$ small, there is some  $C''$, $C''>0$,  independent of $s, t, x, z$ and $g^{\l}$ such that
\begin{align*}
{\bf G}^{\l}(y, z),\  \big\|\nabla^{\l}_{y}{\bf G}^{\l}(y, z)\big\|_{\wt{g}^{\l}}\leq C''{\bf G}^{\l}(x, z), \ \forall y\in B_{g^{\l}}(x, \delta /2), \ z\in {\rm C^{\l}} ({\bf v}, \pi /4).
\end{align*}
To continue, we consider  $L_{W}\big|_y{\bf G}^{\l}(y, z)$, where $W$ is any smooth  bounded vector field on $\M$, $L_{W}\big|_y$ is the  Lie derivative in $W$ evaluated at $y$. Then, in the distribution sense, 
\begin{align*}
\Delta^{\l}L_{W}\big|_{y}{\bf G}^{\l}(y, z)=  L_{W}\big|_y\Delta^{\l}{\bf G}^{\l}(y, z)+\big[L_W, \Delta^{\l}\big]{\bf G}^{\l}(y, z)=\big[L_W, \Delta^{\l}\big]{\bf G}^{\l}(y, z),
\end{align*}
where the last commutator term  is a linear combination of the contractions  $R^{\l}*\nabla^{\l}{\bf G}^{\l}(\cdot, z)$, $\nabla^{\l} R^{\l}* {\bf G}^{\l}(\cdot, z)$ evaluated at $W\in T_y\M$. Since $L_{W}\big|_{y}{\bf G}^{\l}(y, z)$ is $C^1$ in $y$, it must be  a  real solution function $f$ to the equation  
\[
\Delta^{\l}f(y)=\big[L_W, \Delta^{\l}\big]{\bf G}^{\l}(y, z). 
\]
Hence the classical estimation property of elliptic equation (cf. \cite{Fr}) shows that there is some positive $C'''$ depending on the geometry, which can be chosen to be independent of $x, z, g^{\l}$ for $\mathcal{V}_g$ small, such that 
\begin{align*}
\big\|\nabla^{\l}_y L_{W}\big|_{y}{\bf G}^{\l}(y, z)\big\|_{\wt{g}^{\l}} 
\leq &\; C'''\big(\sup_{y\in B(x, \delta)}\big\|\nabla^{\l}_{y}{\bf G}^{\l}(y, z)\big\|_{\wt{g}^{\l}}+\sup_{y\in B(x, \delta)}{\bf G}^{\l}(y, z)\big)\\
\leq &\; 2C''' C''{\bf G}^{\l}(x, z). 
\end{align*}
This shows  (\ref{gradient-Green}) since $W$ can be arbitrary. 
\end{proof}

\begin{proof}[Proof of Theorem \ref{main-h}] Let $\mathcal{V}''_g$, ${\mathtt{b}}''$ be as in Lemma \ref{Holder-2nd-k}. Let  ${\mathtt{b}}<{\mathtt{b}}''\varkappa$, $\mathcal{V}_g$ and $\mathcal{H}^0_{\mathtt{b}}$  be such that  Theorem \ref{regularity-harmonic measure} holds true. Let $\lambda\in
(-1, 1)\mapsto g^{\lambda}\in \mathcal{V}''_g\cap \mathcal{V}_g$  with $g^0=g$ be a $C^3$ curve. We omit the index $0$ for $h^0, k^0$, $p^{0}$,  $\overline{Z}^0$,  $\Delta^0$, ${\rm{Div}}^0$, $\nabla^0$,  $\langle \cdot, \cdot\rangle_0$, ${\rm{Vol}}^0$ and ${\bf m}^0$ at $g^0$. 

  We study the differentiability of $h^{\l}$ by writing, as in \cite{LS2},
\begin{align*}
\frac{1}{\l}(h^{\l}-h)=\frac{1}{\l}(h^{\l}-h^{\l, 0})+\frac{1}{\l}(h^{\l, 0}-h)=: {\rm{(I)}}_{\l}+ {\rm{(II)}}_{\l},
\end{align*}
where
\begin{align*}
h^{\l, 0}
=& -\int \left.\Delta^{\l}_y \ln k(x, y, \xi)\right|_{y=x} \ d{\bf{m}}^{\l}(x, \xi).
\end{align*}
Then, by (\ref{Delta-lnk-2}) and Theorem \ref{regularity-harmonic measure}, 
\begin{align*}
\lim\limits_{\l\to 0}{\rm{(II)}}_{\l}=-\int (\Delta^{\l}_y)'_0 \ln k(x, y, \xi)\big|_{y=x} \ d{\bf{m}}(x, \xi)-\int \Delta_y \ln k(x, y, \xi)\big|_{y=x} \ d({\bf{m}}^\l)'_{0}(x, \xi). 
\end{align*}
Using $u_1$ for the function such that 
\[
\Delta u_1=-\Delta_y \ln k(x, y, \xi)\big|_{y=x}-h, \ (\mbox{see \cite[(5.7)]{LS2}}),
\]
we obtain, as in Section \ref{sec-3.1}, 
\[
\mathcal{K}:=\lim\limits_{\l\to 0}{\rm{(II)}}_{\l}=\int \left(-\frac{1}{2}\langle \nabla {\rm{trace}}\XX, \overline{Z}+\nabla u_1 \rangle +{\rm{Div}}\big(\XX (\overline{Z}+\nabla u_1)\big)\right)\ d{\bf m}. 
\]
Clearly, $\mathcal{K}$ is linear  on  $\XX\in C^k(S^2T^*)$. When $g$ is locally symmetric, $u_1\equiv 0$ and $\overline{Z}=\ell {\overline{X}}$. Hence, 
\[
\mathcal{K}=\ell \int \left(-\frac{1}{2}\langle \nabla {\rm{trace}}\XX, \overline{X}\rangle -\ell \XX (\overline{X}, \overline{X})\right)\ d{\bf m},
\]
which vanishes (see Remark \ref{cri-remark}).

We will now show $\lim_{\l\to 0}{\rm{(I)}}_{\l}=0$, which will complete the proof of Theorem \ref{main-h}. Following \cite[Proposition 2.4]{LS2}, we obtain  $
h^{\l, 0}=\inf_{s>0}\{\overline{h}^{\l, 0}(s)\}$, 
where 
\begin{align*}
\overline{h}^{\l, 0}(s) :=& \lim\limits_{t\to \infty}-\frac{1}{t}\int (\ln p(st, x, y)) p^{\l}(t, x, y)\ d{\rm{Vol}}^{\l}(y). 
\end{align*}
Then, for all $\l>0$, 
\begin{align*}
({\rm{I}})_{\l}=&\; \frac{1}{\l}\sup\limits_{s>0}\lim\limits_{t\to \infty} -\frac{1}{t}\int \ln \frac{p^{\l}(t, x, y)}{p(st, x, y)}p^{\l}(t, x, y)\ d{\rm{Vol}}^{\l}(y)\\
=&\; \frac{1}{\l}\sup\limits_{s>0}\lim\limits_{t\to \infty} -\frac{1}{t}\int \ln \frac{p^{\l}(t, x, y)\rho^{\l}(y)}{p(st, x, y)}p^{\l}(t, x, y)\ d{\rm{Vol}}^{\l}(y)\\
\leq  &\; \frac{1}{\l}\sup\limits_{s>0}\lim\limits_{t\to \infty} \frac{1}{t}\left(1-\int p(st, x, y)\ d{\rm{Vol}}(y)\right)\leq 0,
\end{align*}
where the third inequality holds since $-\ln a\leq  a^{-1}-1$ for all $a>0$. On the other hand, 
\begin{align*}
({\rm{I}})_{\l}
\geq &\; \frac{1}{\l}\left(h^{\l}-\overline{h}^{\l, 0}(1)\right)\\
=&\; \frac{1}{\l}\lim\limits_{t\to \infty} -\frac{1}{t}\int \ln \frac{p^{\l}(t, x, y)}{p(t, x, y)} p^{\l}(t, x, y)\ d{\rm{Vol}}^{\l}(y)\\
=&\; \frac{1}{\l}\lim\limits_{n\in \Bbb N, n\to \infty} \frac{1}{n}\int \ln \frac{p(n, x, y)}{p^{\l}(n, x, y)\rho^{\l}(y)} p^{\l}(n, x, y)\ d{\rm{Vol}}^{\l}(y). 
\end{align*}
To estimate $({\rm I})_{\l}$, consider the stationary Markov chain on the space $\underline{\Omega}=\M^{\Bbb N\cup \{0\}}$ with transition probability $p^{\l}(1, x, y)\ d{\rm{Vol}}^{\l}(y)$ and the process ${Y}_0(\underline{w})=1$ and, for $n\geq 1$, 
\[
{Y}_n(\underline{w}):=\frac{p(1,\underline{w}_0, \underline{w}_1)}{p^{\l}(1, \underline{w}_0, \underline{w}_1)\rho^{\l}(\underline{w}_1)}\cdot \frac{p(1,\underline{w}_1, \underline{w}_2)}{p^{\l}(1, \underline{w}_1, \underline{w}_2)\rho^{\l}(\underline{w}_2)}\cdots  \frac{p(1,\underline{w}_{n-1}, \underline{w}_n)}{p^{\l}(1, \underline{w}_{n-1}, \underline{w}_n)\rho^{\l}(\underline{w}_{n})}.
\]
Observe that 
\[
\frac{p(n, x, y)}{p^{\l}(n, x, y)\rho^{\l}(y)}=\E_{\P_{x}^{\l}}\left(Y_n(\underline{w})\; \big| \; \underline{w}_n=y\right).
\]
So we may write 
\begin{align*}
{\rm{(I)}}_{\l}\;\geq &\; \frac{1}{\l}\lim\limits_{n\in \Bbb N, n\to \infty} \frac{1}{n}\E_{\P_x^{\l}}\left(\ln  \E_{\P_{x}^{\l}}\left(Y_n(\underline{w})\; \big| \; \underline{w}_n=y\right)\right)\\
\geq &\; \frac{1}{\l}\lim\limits_{n\in \Bbb N, n\to \infty} \frac{1}{n}\E_{\P_x^{\l}}\left(\ln  Y_n(\underline{w})\right)\\
=&\; \frac{1}{\l}\lim\limits_{n\in \Bbb N, n\to \infty} \frac{1}{n}\sum\limits_{i=0}^{n-1}\E_{\P_x^{\l}}\left(\ln  \frac{Y_{i+1}(\underline{w})}{Y_{i}(\underline{w})}\right)\\
=&\; \frac{1}{\l}\lim\limits_{n\in \Bbb N, n\to \infty} \frac{1}{n}\sum\limits_{i=0}^{n-1}\E_{\P_x^{\l}}\left(\E_{\P_{\underline{w}_i}^{\l}}\left(\ln  \frac{Y_{i+1}(\underline{w})}{Y_{i}(\underline{w})}\right)\right).
\end{align*}
Set $\underline{w}_{i}=y$ and let $(\lf{\rm y}_t\rc^{\l}, \lf \mho_t\rc^{\l})_{t\geq 0}$ be the stochastic pair in $\M\times \mathcal{O}^{\wt{g}^{\l}}(\M)$ that defines the $\wt{g}^{\l}$ Brownian motion on $\M$ starting from $y$. 
Then, 
\[
\E_{\P_{\underline{w}_i}^{\l}}\left(\ln  \frac{Y_{i+1}(\underline{w})}{Y_{i}(\underline{w})}\right)=\E_{{\rm Q}} \left(\ln \frac{p(1, y, \lf{\rm y}_1\rc^{\l}({\rm w}))}{p^{\l}(1, y, \lf{\rm y}_1\rc^{\l}({\rm w}))\rho^{\l}(\lf{\rm y}_1\rc^{\l}({\rm w}))}\right)=: \E_{{\rm Q}}\left(({\rm{III}})_{\l, y}\right),
\]
which is  $L^1$ integrable in  $y$. Hence the ergodic theorem applied to the $g^{\l}$ Brownian motion on $M$ (see Proposition \ref{mixing}) gives that
\[
\lim\limits_{n\in \Bbb N, n\to \infty} \frac{1}{n}\sum\limits_{i=0}^{n-1}\E_{\P_x^{\l}}\left(\E_{\P_{\underline{w}_i}^{\l}}\left(\ln  \frac{Y_{i+1}(\underline{w})}{Y_{i}(\underline{w})}\right)\right) = \E_{\P^{\l}}\E_{{\rm Q}} \left(({\rm{III}})_{\l, \underline{w}_0}\right).
\]
Since  
\begin{align*}\E_{{\rm Q}}\left(({\rm{III}})_{\l, y}\right)'_{\l}\; =\; &\E_{{\rm Q}}\left(-(\ln p^{\l})'_{\l}(1, y, \lf {\rm{y}}_1\rc^{\l}({\rm w}))-(\ln \rho^{\l})'_{\l}(\lf {\rm{y}}_1\rc^{\l}({\rm w}))\right)\\
& +\E_{{\rm Q}}\left(\big\langle \left.\nabla_z \ln \frac{p(1, y, z)}{p^{\l}(1, y, z)\rho^{\l}(z)}\right|_{z=\lf{\rm y}_1\rc^{\l}({\rm w})}, D\pi \big(\lf \mho_1\rc^{\l}\big)^{(1)}_{\l}({\rm w}) \big\rangle \right),
\end{align*}
we conclude that $\left(({\rm{III}})_{\l, y}\right)'_{\l}$ is $L^1$ integrable, uniformly in $\l$ and $y$,  by using  Theorem \ref{diff-HK-estimations-gen} ii), Lemma \ref{Hs2-ST} and Proposition \ref{est-norm-u-t-j} i) for $(\lf{\rm y}_1\rc^{\l}, \lf \mho_1\rc^{\l})$. Moreover,
\[
\left(\E_{{\rm Q}} \left(({\rm{III}})_{\l, y}\right)\right)'_0= \E_{{\rm Q}}\left(-(\ln p^{\l})'_{0}(1, y, \lf {\rm{y}}_1\rc^{0})-(\ln \rho^{\l})'_{0}(\lf {\rm{y}}_1\rc^{0})\right)=0
\]
by taking the differential in $\l$ of   $\int p^{\l}(1, y, z)\ d{\rm{Vol}}^{\l}(z)\equiv 1$. 

Hence, 
\[
\lim\limits_{\l\to 0+0}\frac{1}{\l} \E_{\P^{\l}}\E_{{\rm Q}} \left(({\rm{III}})_{\l, \underline{w}_0}\right)=\E_{\P}\left(\E_{{\rm Q}} \left(({\rm{III}})_{\l, \underline{w}_0}\right)\right)'_0 = 0,\]
where the first equality holds since we are integrating a function that depends  only on $\underline{w}_0$.  Consequently,  we obtain  $\lim_{\l\to 0+0}{\rm{(I)}}_{\l}=0$.  In the same way, we show $\lim_{\l\to 0-0}{\rm{(I)}}_{\l}=0$. Thus, $\lim_{\l\to 0}{\rm{(I)}}_{\l}=0$. 
\end{proof}

\begin{remark}Note that for all $\l$, $h^{\l}\leq (\upsilon^{\l})^2$ by \cite{Gu} and \cite{K1}. As in Corollary  \ref{prop-thm1.4},  we can also  use \cite{BCG}, \cite{KKPW} and the $C^1$ differentiability of $\l\mapsto h^{\l}$ for any $C^{3}$ curve $\lambda\mapsto
g^{\lambda}\in \Re^{3}(M)$ to conclude that $(h^{\l})'_0=0$  at locally symmetric  $g^0$.
\end{remark}
  In proving Theorem \ref{main-h}, we obtain the following formula. 

\begin{theo}\label{deriv-entr-for}Let  $M$ be a closed connected smooth manifold and let $g\in \Re^3(M)$.   For any $C^3$ curve $\lambda\in (-1, 1)\mapsto g^{\lambda}\in \Re^3(M)$ with $g^0=g$ and constant volume,
\[
(h^{\l})'_0=\int \left(-\frac{1}{2}\langle \nabla {\rm{trace}}\XX, \overline{Z}+\nabla u_1 \rangle +{\rm{Div}}\big(\XX (\overline{Z}+\nabla u_1)\big)\right)\ d{\bf m}.
\]
\end{theo}

\vspace*{1cm} 
{{\bf{Acknowledgments}} --- We thank MSRI, ICERM and IML for their partial support. The second author would  like to also thank  LPMA and the  Department of Mathematics of the University of Notre Dame for hospitality during her stays.}

\end{document}